%% file: Etaglob.tex
\documentclass[10pt]{amsart}
\usepackage{amssymb}
\usepackage[pdftex]{graphicx}
\usepackage{bm}
\usepackage[colorlinks=true]{hyperref}
\makeindex
\newcommand{\Nto}{N^{\LXN}_{-\vartheta}}
\newcommand{\LXN}{\Lambda\ac\left(T^{*}X \oplus N^{*}\right)}
\newcommand{\bt}{b,\vartheta}
\newcommand{\ct}{\cos\left(\vartheta\right)}

\newcommand{\Aut}{\mathrm{Aut}}

\newcommand{\TsX}{T^{*}X}
\newcommand{\Ad}{\mathrm{Ad}}

\newcommand{\kpg}{\mathfrak k^{\perp}\left(\gamma\right)}
\newcommand{\mpg}{\mathfrak p^{\perp}\left(\gamma\right)}
\newcommand{\Yok}{Y_{0}^{\mathfrak k}}
\newcommand{\ad}{\mathrm{ad}}
\newcommand{\pa}{\partial}

\newcommand{\ac}{^{\cdot}}
\newcommand{\we}{\wedge}
\newcommand{\ho}{\widehat{\otimes}}
\newcommand{\n}{\nabla}
\newcommand{\N}{\mathbf{N}}
\newcommand{ \Z}{\mathbf{Z}}
\newcommand{ \R}{\mathbf{R}}
\newcommand{ \C}{\mathbf{C}}

\def \Trs{\mathrm {Tr_{s}}}
\def \Tr{\mathrm{Tr}}

\def \ch{\mathrm {ch}}
\def \End{\mathrm {End}}

\def \even{\mathrm{even}}
\def \odd{\mathrm{odd}}

\def \Re{\mathrm{Re}}
\def\deg{\mathrm{deg}}
\def\Im{\mathrm{Im}}

\newcommand{\ei}{e^{i}}
\newcommand{\ej}{e^{j}}
\theoremstyle{plain}
\newtheorem{lem}{Lemma}[section]
\newtheorem{thm}[lem]{Theorem}
\newtheorem{prop}[lem]{Proposition}

\theoremstyle{definition}
\newtheorem{defin}[lem]{Definition}
\newtheorem{exa}[lem]{Example}

\theoremstyle{remark}
\newtheorem{remk}[lem]{Remark}
\newtheorem{remark}[lem]{Remark}

\numberwithin{equation}{section}
\numberwithin{figure}{section}
\begin{document} 
 
 \bibliographystyle{alpha}
 
\title{Eta invariants and the hypoelliptic Laplacian}


\author{Jean-Michel Bismut}
\address{D\'epartement de Math\'ematique \\ Universit\'e Paris-Sud
\\ B\^atiment 425 \\ 91405 Orsay \\ France}
\thanks{The research leading to the results contained in this paper has received  funding from
the European Research Council 
(E.R.C.) under European Union's Seventh Framework Program 
(FP7/2007-2013)/ ERC grant  agreement No. 291060. The author is  
indebted to a referee for reading the paper  carefully and making very
helpful suggestions.}
\curraddr{}
\email{Jean-Michel.Bismut@math.u-psud.fr}

\subjclass[2010]{11F72, 35H10, 58J20, 58J28, 58J65}
\keywords{Spectral theory, Selberg trace formula; Hypoelliptic 
equations; Index theory and related fixed 
point theorems; Eta-invariants, Chern-Simons invariants; Diffusion processes and stochastic analysis on manifolds}
\date{}


\dedicatory{This article is dedicated to H. Moscovici and R.J. Stanton}

\begin{abstract}
  The purpose of this paper is to give a new proof of results of 
  Moscovici and Stanton on the orbital 
  integrals associated with eta invariants on compact locally symmetric 
  spaces. Moscovici and Stanton used 
  methods of
  harmonic analysis on reductive groups. Here, we combine our approach to orbital integrals 
   that uses the hypoelliptic Laplacian, with the 
  introduction
  of a rotation on certain 
  Clifford algebras. Probabilistic methods play an important role in 
  establishing key estimates. In particular, we construct the 
  proper It\^{o} calculus associated with certain hypoelliptic 
  diffusions.
\end{abstract}
\maketitle
\tableofcontents
\include{Eta0}
\include{Eta1}
\include{Eta2}
\include{Eta3}

\include{Eta4}
\include{Eta5}

\include{Eta6}

\include{Eta7}
\include{Eta8}

\include{Eta9}

\include{Eta10}

\include{Eta11}

\printindex
\bibliography{Bismut,Others}
\end{document}

%% file: Eta0.tex
\section{Introduction}
The purpose of this paper is to give a new approach to the evaluation 
by Moscovici-Stanton \cite{MoscoviciStanton89} of the orbital 
integrals that appear in  the eta invariant of a 
Dirac operator  on a  
compact odd dimensional locally symmetric space. Moscovici and Stanton obtained their 
results by an appropriate use of Selberg trace formula.   They were extending earlier results by Millson \cite{Millson78} 
on compact manifolds of constant negative curvature. 
In a later paper 
  \cite{MoscoviciStanton91}, Moscovici and Stanton used similar 
  methods to evaluate the integrand of the Ray-Singer analytic 
  torsion for compact locally symmetric spaces.  In 
  both papers \cite{MoscoviciStanton89,MoscoviciStanton91}, Moscovici and Stanton built up on the evaluation of 
   orbital integrals associated with heat kernels to give an explicit 
   formula for the eta invariant and analytic torsion. Here, we will 
   recover the results of Moscovici-Stanton  on the orbital integrals  in the case of eta invariants.
   
   Our method uses the theory of the hypoelliptic Laplacian on 
   symmetric spaces \cite{Bismut08b}, which we briefly review. The 
   hypoelliptic Laplacian is a general construction that is valid on arbitrary 
   Riemannian manifolds \cite{Bismut05a,Bismut06d}, in which a family 
   of hypoelliptic operators acting 
   on the total space of the tangent bundle, or of a larger vector 
   bundle, is produced, that 
   interpolates in the proper sense between a generalized Laplacian, 
   and the generator of the geodesic flow. In 
   \cite{Bismut08b},  a version 
   of the hypoelliptic Laplacian was shown to exist on symmetric 
   spaces, such that the semisimple orbital 
   integrals associated with the heat kernel for the Casimir operator 
   are preserved by the hypoelliptic deformation.  By deforming all the way to 
   the geodesic flow,  a general 
   explicit formula was given in \cite{Bismut08b} for these 
   orbital integrals. 
   
   In \cite[chapter 7]{Bismut08b}, applications 
   were given to the evaluation of the orbital integrals 
   appearing in the evaluation of the Ray-Singer analytic torsion. 
   We recovered this way some of the results of Moscovici-Stanton 
   \cite{MoscoviciStanton91}. This is because, in the context of 
   analytic torsion, the square of the relevant Dirac operator is 
   just the Casimir operator.
   The same is true for the other index theoretic calculations of 
   \cite[chapter 7]{Bismut08b} for even dimensional  locally 
   symmetric spaces, since the square of the Dirac operator differs 
   from the Casimir operator by a covariantly constant matrix 
   operator. Let us also point out the work by  
   Shen \cite{Shen16,Shen16b} who was able to complete the proof 
   by Moscovici-Stanton \cite{MoscoviciStanton91} of the Fried 
   conjecture for the analytic torsion  \cite{Fried86a} in the case 
   of compact locally symmetric spaces. Shen used the explicit 
   formulas established in \cite{Bismut08b} for the orbital integrals 
   for the heat kernel associated with the Casimir operator, and 
   filled a gap in Moscovici-Stanton's proof 
   \cite{MoscoviciStanton91}, also using arguments from 
   representation theory.
   
   The eta invariant is a spectral function of the Dirac 
   operator, and not just of its square, and so the 
   analysis in \cite{Bismut08b}  breaks down. Before we explain 
   how to deal with this difficulty, let us describe  the 
   ideas of \cite{Bismut08b} in more detail.
   
   Let $G$ be a connected
    reductive Lie group, let $\theta$ be its Cartan involution,  let $\mathfrak g$ be its Lie 
   algebra, and let $K$ 
   be the associated maximal compact subgroup of $G$ with Lie algebra $\mathfrak k 
   \subset \mathfrak g$. Let $\mathfrak g=\mathfrak p \oplus 
   \mathfrak k$ be the Cartan splitting of $\mathfrak g$.  Let 
   $X=G/K$ be the corresponding symmetric space.  Let $B$ be    a   bilinear 
 symmetric nondegenerate  form  on $\mathfrak g$, which is    
 $G$-invariant and $\theta$-invariant,  positive on $\mathfrak 
   p$, and negative on $\mathfrak k$.
   
   Let $C^{\mathfrak g}$ be the Casimir operator of $G$ associated 
   with $B$, and let $F$ 
   be the complex vector bundle on $X$ associated with an 
   irreducible 
   representation of $K$. Then 
   $C^{\mathfrak g}$ 
   descends to a second order elliptic operator $C^{\mathfrak g,X}$ acting on smooth 
   sections of $F$ on $X$. Set 
   $\mathcal{L}^{X}=\frac{1}{2}C^{\mathfrak g,X}+c$, where $c$ is an 
   explicit constant.
   
   Note that $TX$ is the vector bundle 
   corresponding to the adjoint action of $K$ on $\mathfrak p$. Let $N$ be 
   the vector bundle on $X$ associated with the adjoint action of $K$ 
   on $\mathfrak k$. Let $\widehat{\pi}:\widehat{\mathcal{X}}\to X$ be the total space 
   of $TX \oplus N$.  The hypoelliptic 
   deformation $\mathcal{L}_{b}^{X}\vert_{b>0}$ of 
   $\mathcal{L}^{X}$  is obtained 
   via a corresponding family of generalized Dirac operators 
   $\mathfrak D_{b}^{X}\vert_{b>0}$. The 
   operators $\mathfrak D^{X}_{b},\mathcal{L}^{X}_{b}$  act on 
   the smooth sections over $\widehat{\mathcal{X}}$ of 
   $\widehat{\pi}^{*}\left(\Lambda\ac\left(T^{*}X \oplus N^{*}\right)
   \otimes_{\R}F\right)$.  Up to lower order terms, the operator 
   $\mathcal{L}^{X}_{b}$ is a scaled sum of the harmonic oscillator 
   along $TX \oplus N$, of the generator of the geodesic flow,
   and of a nonnegative scalar term of degree $4$ in the variables in the fibre 
   $TX \oplus N$. 
   
   In \cite{Bismut08b}, the 
   fact that the semisimple orbital integrals for the heat kernel 
   of $\mathcal{L}^{X}$ are
   invariant 
    by  the hypoelliptic deformation is shown to  be a
   version of the McKean-Singer formula \cite{McKeanSinger} for the 
   Lefschetz supertrace associated with a classical Dirac operator. By making $b\to + \infty $, which forces the 
   hypoelliptic Laplacian $\mathcal{L}^{X}_{b}$ to converge in the 
   proper sense to the generator of the geodesic flow, we obtain an 
   explicit geometric formula for the orbital integrals of the heat 
   kernel of $\mathcal{L}^{X}$.
   
     Let 
   $c\left(\mathfrak g\right),\widehat{c}\left(\mathfrak g\right)$ be 
   the Clifford algebras associated with $\left(\mathfrak g,B\right),
  \left(\mathfrak g,-B\right)$, and let $U\left(\mathfrak g\right)$ 
  be the enveloping algebra of $\mathfrak g$. In the construction of 
  $\mathfrak D^{X}_{b}$, one key idea in
 \cite{Bismut08b} is to express   the Casimir operator $C^{\mathfrak 
 g}$ up to a constant as minus
 the  square of the Kostant Dirac 
   operator $\widehat{D}^{ \mathfrak g}\in \widehat{c}\left(\mathfrak 
   g\right)\otimes U\left(\mathfrak g\right)$ 
   \cite{Kostant76,Kostant97}.  Ultimately, both Clifford algebras 
   $c\left(\mathfrak g\right), \widehat{c}\left(\mathfrak g\right)$ 
   are used in the constructions of \cite{Bismut08b}.

  We assume $G$ to be simply connected. Let $\overline{TX}$ be another copy of $TX$, and let $S^{\overline{TX}}$ denote 
  the vector bundle of $\overline{TX}$ spinors. The classical Dirac 
  operator $D^{X}$ on $X$ acts on $C^{ \infty }\left(X,S^{\overline{TX}} 
  \otimes F\right)$. Let 
   $\overline{\mathfrak p}$ be another copy of $\mathfrak p$. The 
   Clifford algebra $c\left(\overline{\mathfrak p}\right)$ descends to 
   the Clifford algebra $c\left(\overline{TX}\right)$,  that is used   in the construction of  the Dirac operator $D^{X}$.  The operator
   $D^{X,2}$ differs from $C^{\mathfrak g,X}$ by a constant tensor 
   in which  the Clifford algebra $c\left(\overline{\mathfrak p}\right)$ does 
   not appear.   This is why in \cite[chapter 7]{Bismut08b}, in which 
   the index theory of Dirac operators is considered in relation with 
   the trace formula, the odd part of the Clifford algebra $c\left(\overline{\mathfrak 
   p}\right)$  plays  no role. These considerations suggest 
   that to apply the methods of \cite{Bismut08b} to the Dirac operator 
   $D^{X}$ itself, we have to consider the three Clifford algebras 
   $c\left(\mathfrak g\right),\widehat{c}\left(\mathfrak 
   g\right),c\left(\overline{\mathfrak p}\right)$ together.
    In a formally similar context, three Clifford algebras were 
    already considered in
   \cite{Bismut06d} to construct a hypoelliptic version of the Dirac 
   operator associated with an arbitrary Riemannian manifold.
   
   The complex structure of 
   $\mathfrak p \oplus \overline{\mathfrak p}$ will play an important role 
   in our method to recover the results of  Moscovici-Stanton  \cite{MoscoviciStanton89}.
   What replaces  the conservation of Lefschetz supertraces is a 
   conservation principle for another class of orbital integrals. More precisely,  we apply the odd 
   superconnection  
   formalism of Quillen \cite{Quillen85a} to a two parameter family 
   of generalized Dirac operators $\mathfrak 
   D^{X}_{b,\vartheta}\vert_{b>0,\vartheta\in \left[0,\frac{\pi}{2}\right[}$ 
   over $\widehat{\mathcal{X}}$.  The corresponding family of hypoelliptic 
   Laplacians is denoted 
   $\mathcal{L}^{X}_{b,\vartheta}\vert_{b>0,\vartheta\in 
   \left[0,\frac{\pi}{2}\right[}$. If $\vartheta=0$, we recover the 
   families $\mathfrak 
   D^{X}_{b}\vert_{b>0},\mathcal{L}^{X}_{b}\vert_{b>0}$. If $\gamma\in G$ 
   is semisimple and nonelliptic, we 
    show that the integral of a $1$-form $\mathsf b$ constructed via 
   orbital integrals associated with $\gamma$ on 
    $\left[0,\frac{\pi}{2}\right[$ does not depend on $b>0$.  By making $b\to 0$, 
    this quantity is shown to be an explicit multiple of the orbital integral
   $\Tr^{\left[\gamma\right]}\left[\frac{D^{X}}{\sqrt{2}}\exp\left(-sD^{X,2}/2\right)\right]*
   \frac{1}{\sqrt{s}}\left(t\right)$, where $*$ is  the 
   convolution of functions on $\R_{+}$. 
   By making $b\to + \infty $, we express this quantity in geometric 
   terms, by a localization procedure which is essentially taken from 
   our previous work
   \cite{Bismut08b}.  If $\gamma=e^{a}k^{-1}, a\in \mathfrak p,k\in 
   K, \Ad\left(k\right)a=a$,  like in \cite{Bismut08b}, the geometric 
   expression involves an integral on $\mathfrak 
   k\left(\gamma\right)$, the $\mathfrak k$ part of the Lie algebra 
   $\mathfrak z\left(\gamma\right)$ of the centralizer of $\gamma$. 
   Ultimately, we recover 
   the vanishing results of 
   Moscovici and Stanton  \cite{MoscoviciStanton89}, 
   and also the explicit formulas they obtained in the case where 
   the orbital integrals do not vanish.
   
   In the paper, we tried to clearly separate the algebraic 
   and geometric arguments from the analytic arguments. In particular, in 
   sections \ref{sec:linal}--\ref{sec:applic}, we construct the 
   operators $\mathfrak 
   D^{X}_{b,\vartheta},\mathcal{L}^{X}_{b,\vartheta}$,  key estimates 
   on heat kernels for $\mathcal{L}^{X}_{b,\vartheta}$ on 
   $\widehat{\mathcal{X}}$ are  
   stated without proof, and are used to establish our main results in Theorems \ref{Tvan} and \ref{Tgen}. The 
   proof of these estimates is deferred to sections \ref{sec:extra}--\ref{sec:unilar}.  Most of the estimates are obtained by properly adapting 
   the estimates in \cite{Bismut08b}, except for a uniform estimate 
   on solutions of a linear differential equation, that is established in 
   Theorems \ref{Tfuest} and \ref{Tunif}. This estimate
    eluded us for some time, and largely explains the length of the 
    paper, which should otherwise be a rather straightforward 
    extension of \cite{Bismut08b}. As in \cite{Bismut08b}, and for 
    fundamental reasons, probabilistic techniques play an important
    role in the proofs.
   
   This paper is organized as follows. In section \ref{sec:linal}, we 
   briefly recall simple facts of linear algebra, that include  
    Quillen's superconnection formalism \cite{Quillen85a}.
   
   In section \ref{sec:eta}, we recall the construction of the 
   hypoelliptic Laplacian $\mathcal{L}_{b}^{X}$ of \cite{Bismut08b}, which up to an 
   explicit constant, deforms  
    $C^{\mathfrak g,X}/2$.  The operator $\mathcal{L}_{b}^{X}$ is obtained via the 
   construction of a Dirac like operator $\mathfrak D_{b}^{X}$, which 
   itself deforms the  operator $0$.
   
   In section \ref{sec:defdx}, we construct the families of operators 
   $\mathfrak 
   D^{X}_{b,\vartheta}\vert_{\left(b,\vartheta\right)\in\R^{*}_{+} 
   \times \left[0,\frac{\pi}{2}\right[}$ and 
   $\mathcal{L}^{X}_{b,\vartheta}
   \vert_{\left(b,\vartheta\right)\in\R^{*}_{+}\times 
   \left[0,\frac{\pi}{2}\right[}$.  Given 
   $\vartheta\in\left[0,\frac{\pi}{2}\right[$, the family $\mathfrak 
   D_{b,\vartheta}^{X}\vert_{b>0}$ is a deformation of 
   $\sin\left(\vartheta\right)iD^{X}$. For $\vartheta>0$, the 
   construction of $\mathfrak D^{X}_{b,\vartheta}$ now involves 
   the Clifford algebra 
   $c\left(\overline{TX}\right)$ explicitly.
   
   In section \ref{sec:defhyp}, given a semisimple element $\gamma\in G$,  we construct a closed superconnection $1$-form 
   $\mathsf{b}$ on $\R_{+}^{*}\times \left[0,\frac{\pi}{2}\right[$, 
   involving orbital integrals associated with the heat kernel for 
   $\mathcal{L}^{X}_{b,\vartheta}$, in which the operator $\mathfrak 
   D^{X}_{\bt}$ also appears. 
   
   In section \ref{sec:pres}, we show that the integral of 
   $\mathsf{b}$ on $\left[0,\frac{\pi}{2}\right[$ does not depend on 
   $b>0$, and coincides with the integral on 
   $\left[0,\frac{\pi}{2}\right[$ of a $1$-form $\mathsf{a}$ 
   involving the orbital integrals 
  $\Tr^{\left[\gamma\right]}\left[D^{X}\exp\left(-sD^{X,2}/2\right)\right]$.
   
  In section \ref{sec:oddorb}, by making $b\to + \infty $, we give an 
  explicit geometric formula for the  integral of  $\mathsf{a}$. The 
  geometric computations are closely related to our earlier work 
  \cite{Bismut08b}.
   
  In section \ref{sec:applic}, by working in more detail on the 
  geometric
  formulas of section \ref{sec:oddorb}, we obtain our main results, 
  i.e., we rederive the formulas of Moscovici-Stanton 
  \cite{MoscoviciStanton89} for the orbital integrals 
  $\Tr^{\left[\gamma\right]}\left[D^{X}\exp\left(-tD^{X,2}/2\right)\right]$. 
  
  In the sections that follow, we establish the estimates that are 
  needed in the proofs of the previous results. The analysis of the 
  operators $\mathcal{L}^{X}_{b,\vartheta}$ involves an essentially 
  different step from the analysis in \cite{Bismut08b} when $b>0$ is 
  bounded, because we need to obtain uniform estimates in 
  $\vartheta\in\left[0,\frac{\pi}{2}\right[$ as $b\to 0$. 
  When $b\to + \infty $, the analysis is essentially the same as in 
  \cite{Bismut08b}. 
  
  In section \ref{sec:extra}, we improve on the uniform estimates we had 
  obtained in \cite[chapter 12--14]{Bismut08b} for the smooth kernels 
  $\exp\left(-t\mathcal{L}^{X}_{b}\right)$ , when $b>0,t>0$ remain 
  uniformly bounded. In \cite{Bismut08b}, only the case where $t$ 
  remains bounded  away from $0$ was considered. The reason for doing 
  this is that some of 
  our estimates for the  kernels for 
  $\exp\left(-\mathcal{L}^{X}_{b,\vartheta}\right)$ are equivalent to 
  the just described estimates when $t>0$ is small. Also we establish uniform 
  estimates on the rate of escape from an open ball for the 
  hypoelliptic diffusion associated with 
  $\mathcal{L}^{X}_{b}\vert_{0<b\le 1}$. These estimates were not stated 
  explicitly in \cite{Bismut08b},  and turn out to be useful in the 
  sections that follow. 
  
  In section \ref{sec:unisca}, we obtain uniform estimates for the 
  heat kernels for a scalar version of 
  $\mathcal{L}^{X}_{b,\vartheta}$ for $b>0$ bounded and $\vartheta\in 
  \left[0,\frac{\pi}{2}\right[$. These estimates are easy 
  consequences of the results of section \ref{sec:extra}.
  
  In section \ref{sec:fin}, we obtain the required estimates for the 
  heat kernels $\exp\left(-\mathcal{L}^{X}_{b,\vartheta}\right)$ when 
  $b>0$ remains bounded and $\vartheta\in 
  \left[0,\frac{\pi}{2}\right[$. This is the technically most 
  difficult section of the paper. Indeed, passing from the estimates 
  for the scalar version of $\mathcal{L}^{X}_{b,\vartheta}$ to the 
  full operator $\mathcal{L}^{X}_{b,\vartheta}$ introduces new 
  difficulties that did not appear in \cite{Bismut08b}, essentially 
  because the exterior algebra $\Lambda\ac\left(T^{*}X \oplus 
  N^{*}\right)$ and the Clifford algebra 
  $c\left(\overline{TX}\right)$ are coupled in a nontrivial way for 
  $\theta>0$.  Probabilistic techniques are especially useful there. 
  While we described before the main new difficulty with respect to 
  \cite{Bismut08b} to be the proof of 
  uniform bounds on the solution of a family of differential 
  equations, the way this control is obtained is via geometric 
  considerations on the projection of this solution on the proper 
  symmetric space.
  
  Finally, in section \ref{sec:unilar}, we obtain the required uniform
  estimates for the smooth kernels 
  $\exp\left(-\mathcal{L}^{X}_{b,\vartheta}\right)$ for $b\ge 1$ and 
  $\vartheta\in\left[0,\frac{\pi}{2}\right[$.
 
  Let us also point out that in the very short subsection 
  \ref{subsec:geito}, we establish a generalized It\^o formula for 
  the hypoelliptic diffusion associated with the operator 
  $\mathcal{L}^{X}_{b}$.  This subsection can be read independently 
  of the remainder of the paper. It makes clear that the proper It\^o 
  calculus for our hypoelliptic diffusion is deduced from the 
  classical It\^o calculus by a simple convolution. This new form of 
  the It\^o calculus plays a key role in the proof of our estimates, 
  and should be of independent interest. For an introduction to 
  Brownian motion and the stochastic calculus, we refer to 
  Ikeda-Watanabe \cite{IkedaWatanabe89} and Le 
  Gall \cite{LeGall16}.

   In the whole paper, if $\mathcal{A}$ is a $\Z_{2}$-graded algebra, 
   if $a,a'\in\mathcal{A}$, $\left[a,a'\right]$ denotes the 
   supercommutator of $a,a'$, so that
   \begin{equation}\label{eq:suprcm}
\left[a,a'\right]=aa'-\left(-1\right)^{\deg a\deg a'}a'a.
\end{equation}
    If $\mathcal{A},\mathcal{A'}$ are 
   $\Z_{2}$-graded algebras,  
    $\mathcal{A}\ho \mathcal{A}'$ denotes the tensor product 
    $\mathcal{A} \otimes \mathcal{A}'$ 
   equipped with the induced $\Z_{2}$-graded structure.
   
   Also, in our estimates, the  constants $c>0,C>0$ may vary from 
   line to line.

%% file: Eta1.tex
\section{Linear algebra}%
\label{sec:linal}
The purpose of this section is to develop an algebraic formalism in a finite dimensional 
context  that 
will be used in the next sections in infinite dimensions, to properly 
handle eta invariants on locally symmetric spaces. 

This section is organized as follows. In subsection 
\ref{subsec:elop}, if $E$ is a vector space, and $\mathcal{A}=\End\left(E\right)$, we introduce associated superconnections and 
$\sigma$-traces   in the sense of
 Quillen \cite{Quillen85a}. Of special 
importance is the construction of a function on $\mathcal{A}$ that 
transgresses a Quillen form of degree $1$.

In subsection \ref{subsec:cen},  we specialize the 
constructions of subsection \ref{subsec:elop} to the case where 
$\mathcal{A}'$ is a subalgebra of $\mathcal{A}$ and where $C$ is a 
central element in $\mathcal{A}'$. 
\subsection{Superconnections and $\sigma$-traces}%
\label{subsec:elop}
First, we describe Quillen's odd superconnection formalism \cite[\S
5]{Quillen85a}. Let $\sigma$ be the odd generator of the Clifford algebra 
$c\left(\R\right)$, so that $\sigma^{2}=1$.

Let $E$ be a finite dimensional complex vector space.
Put 
\begin{equation}\label{eq:clx1}
\mathcal{A}=\End\left(E\right).
\end{equation}
Set
\begin{equation}\label{eq:cla1}
\mathcal{A}_{\sigma}=\mathcal{A} \otimes_{\R} c\left(\R\right).
\end{equation}
Then $\mathcal{A}_{\sigma}$ is a $\Z_{2}$-graded algebra. The 
splitting of $\mathcal{A}_{\sigma}$ into its even and odd 
parts is given by
\begin{equation}\label{eq:cla2}
\mathcal{A}_{\sigma}=\mathcal{A} \oplus \sigma \mathcal{A}.
\end{equation}
If $\alpha\in \mathcal{A}_{\sigma}$, then $\alpha=a+\sigma b, a,b\in 
\mathcal{A}$. Put
\index{Trs@$\Tr_{\sigma}$}%
\begin{equation}\label{eq:cla2x}
\Tr_{\sigma}\left[\alpha\right]=\Tr\left[b\right].
\end{equation}
Then $\Tr_{\sigma}$ vanishes on the even part of 
$\mathcal{A}_{\sigma}$, and also on supercommutators in 
$\mathcal{A}_{\sigma}$.

Let $\mathcal{A}^{*}$ denote the  dual vector space  to $\mathcal{A}$. Put
\begin{equation}\label{eq:cla3}
\mathcal{B}=\Lambda\ac\left(\mathcal{A}^{*}\right) 
\ho\mathcal{A}_{\sigma}.
\end{equation}
Then $\mathcal{B}$ is also a $\Z_{2}$-graded algebra.
We extend $\Tr_{\sigma}$ to a map from $\mathcal{B}$ into 
$\Lambda\ac\left(\mathcal{A}^{*}\right)$, with the convention that if $\omega\in 
\Lambda\ac\left(\mathcal{A}^{*}\right), 
\alpha\in\mathcal{A}_{\sigma}$, 
\begin{equation}\label{eq:cla4}
\Tr_{\sigma}\left[\omega 
\alpha\right]=\omega\Tr_{\sigma}\left[\alpha\right].
\end{equation}
Then $\Tr_{\sigma}$ still vanishes on supercommutators in 
$\mathcal{B}$.

We now view $\mathcal{A}$ as a trivial vector bundle on the vector 
space $\mathcal{A}$. Let $D$ denote the tautological section of 
$\mathcal{A}$ over the vector space $\mathcal{A}$. 
Let 
$d$ be the de Rham operator on $\mathcal{A}$. Let $A$ be the 
 superconnection over $\mathcal{A}$, 
\begin{equation}\label{eq:bob1}
A=d+D\sigma.
\end{equation}
Its curvature $A^{2}\in \mathcal{B}^{\even}$ is given by
\begin{equation}\label{eq:bob2}
A^{2}=D^{2}+\left(dD\right) \sigma.
\end{equation}
It verifies the Bianchi identity
\begin{equation}\label{eq:bob2x1}
\left[A,A^{2}\right]=0.
\end{equation}

    Let $\varphi$ be a holomorphic  function from $\C$ into itself. 
Then $\varphi$ extends to an analytic function from 
$\mathcal{B}^{\even}$ 
into itself. In particular, $\varphi\left(A^{2}\right)$ lies in 
$\mathcal{B}^{\even}$.

Set
\begin{equation}\label{eq:bob3}
\Phi=\Tr_{\sigma}\left[\varphi\left(A^{2}\right)\right].
\end{equation}
Then $\Phi$ is an odd form on $\mathcal{A}$. A result of 
Quillen \cite{Quillen85a} says that $\Phi$ is a closed form. Indeed, 
using the vanishing of $\Tr_{\sigma}$ on supercommutators and 
Bianchi's identity (\ref{eq:bob2x1}), we get
\begin{equation}\label{eq:bob3x0}
d\Phi=\Tr_{\sigma}\left[\left[d,\varphi\left(A^{2}\right)\right]\right]=
\Tr_{\sigma}\left[\left[A,\varphi\left(A^{2}\right)\right]\right]=0.
\end{equation}

We will be especially 
interested in the component $\Phi^{(1)}$ of degree $1$ of $\Phi$, 
that is given by
\begin{equation}\label{eq:bob5x-1}
\Phi^{(1)}=\Tr\left[\varphi'\left(D^{2}\right)dD\right].
\end{equation}
Let $F:\mathcal{A}\to\C$ be the smooth function that is given by
\begin{equation}\label{eq:bob4}
F\left(D\right)=\int_{0}^{1}\Tr\left[
D\varphi'\left(u^{2}D^{2}\right)\right]du,
\end{equation}
Since $\Phi^{(1)}$ is closed,  $F$ is the unique smooth function on $\mathcal{A}$ vanishing at 
$0$ such that
\begin{equation}\label{eq:bib4x1}
dF=\Phi^{(1)}.
\end{equation}

For $t\ge 0$, set
\begin{equation}\label{eq:bob4x-1}
\varphi_{t}\left(z\right)=\varphi\left(tz\right).
\end{equation}
The associated function $F_{t}$ is given by
\begin{equation}\label{eq:bob4x2}
F_{t}\left(D\right)=\sqrt{t}F\left(\sqrt{t}D\right).
\end{equation}
Equivalently,
\begin{equation}\label{eq:bob4x2a}
F_{t}\left(D\right)=\sqrt{t}\int_{0}^{t}\frac{1}{2\sqrt{u}}\Tr\left[D\varphi'\left(uD^{2}\right)\right]
du.
\end{equation}
The function $t\in \R_{+}\to F_{t}\left(D\right)\in\C$ is smooth, and moreover,
\begin{equation}\label{eq:bob4x3}
\frac{d}{dt}F_{t}\left(D\right)\vert_{t=0}= \varphi'\left(0\right)\Tr\left[D\right].
\end{equation}
\begin{exa}\label{exa1}
Let $\log$ be the  logarithm defined on 
$\left\{z\in \C^{*}, \mathrm{\Re}\ z\ge 0\right\}$, with
polar angle lying in $[-\frac{\pi}{2},\frac{\pi}{2}]$. For $z\in \R, 
\Re z\ge 0$, set
\begin{equation}\label{eq:bob3ba}
\varphi\left(z\right)=\log\left(1+z\right).
\end{equation}

Here, we assume $E$ to be a Hermitian vector space. 
Let $\mathcal{A}_{\mathrm{sa}}, \mathcal{A}_{\ge 0}$ be the vectors 
subspaces of $\mathcal{A}$ of 
self-adjoint elements, and self-adjoint nonnegative elements.
 If $B\in \mathcal{A}_{\ge 0}$, then
\begin{equation}\label{eq:bob3y-1}
\Tr\left[\varphi\left(B\right)\right]=\log\det\left(1+B\right).
\end{equation}

Here, we restrict our construction to elements $D\in 
\mathcal{A}_{\mathrm{sa}}$. We define the form $\Phi$ on 
$\mathcal{A}_{\mathrm{sa}}$ as in (\ref{eq:bob3}). Then
\begin{equation}\label{eq:bob4u}
    \Phi^{(1)}=\Tr\left[\left(1+D^{2}\right)^{-1}dD\right].
\end{equation}
As in Moscovici-Stanton \cite[\S\ 2]{MoscoviciStanton89}, we introduce the Cayley 
transform
\begin{equation}\label{eq:bob4x1}
C\left(D\right)=\frac{1+iD}{1-iD}.
\end{equation}

Set
\begin{equation}\label{eq:bob5}
F\left(D\right)=\frac{1}{2i}\log\det\left[C\left(D\right)\right]=\Im\log\det\left(1+iD\right).
\end{equation}
Then $F$ is the real function on $\mathcal{A}_{\mathrm{sa}}$  such 
that  $F\left(0\right)=0, dF=\Phi^{(1)}$.

We denote by $\Tr'\left[\frac{D}{\left\vert  D\right\vert}\right]$ the 
trace of $\frac{D}{\left\vert  D\right\vert}$ in which the possible 
zero eigenvalues of $D$ have been eliminated.  
When $t\to + \infty $, then
\begin{equation}\label{eq:bob5x1}
\frac{F_{t}\left(D\right)}{\sqrt{t}}\to \frac{\pi}{2}\Tr'\left[\frac{D}{\left\vert  
D\right\vert}\right].
\end{equation}

Another definition of $F\left(D\right)$ is by a formula like in (\ref{eq:bob4}), 
\begin{equation}\label{eq:bob6}
F\left(D\right)=\int_{0}^{1}\Tr\left[\left(1+u^{2}D^{2}\right)^{-1}D\right]du.
\end{equation}
One can derive (\ref{eq:bob5x1}) from (\ref{eq:bob6}).
\end{exa}
\begin{exa}\label{exa2}
Set
\begin{equation}\label{eq:bob7}
\varphi\left(z\right)=\exp\left(-z\right).
\end{equation}
Then
\begin{equation}\label{eq:bob8}
\Phi^{(1)}=-\Tr\left[dD\exp\left(-D^{2}\right)\right].
\end{equation}
Also
\begin{equation}\label{eq:bob9}
F_{t}\left(D\right)=-\sqrt{t}\int_{0}^{t}\frac{1}{2\sqrt{u}}
\Tr\left[D\exp\left(-uD^{2}\right)\right]du.
\end{equation}
If $D\in \mathcal{A}_{\mathrm{sa}}$, when $t\to + \infty $, 
\begin{equation}\label{eq:bob9x1}
\frac{F_{t}\left(D\right)}{\sqrt{t}}\to -\frac{\sqrt{\pi}}{2}\Tr'\left[\frac{D}{\left\vert  
D\right\vert}\right].
\end{equation}

For $s\in \C, \mathrm{Re}\,s>0$, if $B\in\mathcal{A}_{\ge 0}$,
then
\begin{equation}\label{eq:bob9x2}
\frac{1}{\Gamma\left(s\right)}\int_{0}^{+ \infty 
}u^{s-1}e^{-u}\varphi\left(uB\right)du=\left(1+B\right)^{-s}.
\end{equation}
By (\ref{eq:bob9x2}), we get
\begin{equation}\label{eq:bob9x3}
\frac{\pa}{\pa s}\left[\frac{1}{\Gamma\left(s\right)}\int_{0}^{+ \infty 
}u^{s-1}e^{-u}\varphi\left(uB\right)du\right]\left(0\right)=-\log\left(1+B\right).
\end{equation}
Note that (\ref{eq:bob5x1}) also follows from (\ref{eq:bob9x1}), 
(\ref{eq:bob9x3}).
\end{exa}
\subsection{A central element}%
\label{subsec:cen}
Let $\mathcal{A}'$ be a subalgebra of $\mathcal{A}$, let $C$ be a 
central element in $\mathcal{A}'$. We will  restrict ourselves to 
forms on $\mathcal{A}'$. In particular $D$ now denotes the generic 
element of $\mathcal{A}'$, and $d$ is the de Rham operator on 
$\mathcal{A}'$.  We still define the superconnection $A$ on 
$\mathcal{A}'$ by equation (\ref{eq:bob1}). Then 
\begin{equation}\label{eq:cla5}
\left[A,C\right]=0.
\end{equation}
By (\ref{eq:bob2x1}), (\ref{eq:cla5}), we obtain the Bianchi identity 
\begin{equation}\label{eq:cla6}
\left[A,A^{2}-C\right]=0.
\end{equation}

Let $\varphi$ be a holomorphic function as in 
subsection \ref{subsec:elop}. Set
\begin{equation}\label{eq:bob10}
\Phi_{C}=\Tr_{\sigma}\left[\varphi\left(A^{2}-C\right)\right].
\end{equation}
By proceeding as in (\ref{eq:bob3x0}) and using (\ref{eq:cla6}), we 
get
\begin{equation}\label{eq:bob10x1}
d\Phi_{C}=\Tr_{\sigma}\left[\left[d,\varphi\left(A^{2}-C\right)\right]\right]=
\Tr_{\sigma}\left[A,\varphi\left(A^{2}-C\right)\right]=0,
\end{equation}
i.e.,  $\Phi_{C}$ is  a closed odd
form on $\mathcal{A}'$. Also
\begin{equation}\label{eq:bob11}
\Phi_{C}^{(1)}=\Tr\left[\varphi'\left(D^{2}-C\right)dD\right].
\end{equation}
As in (\ref{eq:bob4}), for $D\in \mathcal{A}'$, set
\begin{equation}\label{eq:cla7}
F_{C}\left(D\right)=\int_{0}^{1}\Tr\left[D\varphi'\left(u^{2}D^{2}-C\right)\right]du.
\end{equation}
Then $F_{C}$ is the unique function on $\mathcal{A}'$ that vanishes 
at $0$  and is such that
\begin{equation}\label{eq:cla8}
dF_{C}=\Phi^{(1)}.
\end{equation}
More generally, the considerations of subsection \ref{subsec:elop} 
also apply to $F_{C}$. For $t\ge 0$, we denote by 
\index{FCt@$F_{C,t}$}%
$F_{C,t}$ the 
function associated with $\varphi_{t}$.
\begin{exa}\label{exa3}
    Set
 \begin{equation}\label{eq:bob7x0}
\varphi\left(z\right)=\exp\left(z\right).
\end{equation}
Let $\Phi_{C,t}$ be associated with $\varphi_{t}$.
Then
\begin{equation}\label{eq:bob12x1}
\Phi_{C,t}^{(1)}=t\Tr\left[dD\exp\left(t\left(D^{2}-C\right)\right) \right].
\end{equation}
Also
\begin{equation}\label{eq:bob13}
F_{C,t}\left(D\right)=\sqrt{t}\int_{0}^{t}\frac{1}{2\sqrt{u}}\Tr\left[D\exp\left(uD^{2}-tC  
\right)\right]du.
\end{equation}
The function $F_{C,t}\left(D\right)$ is  a smooth function of $t\in \R_{+}$. 
Moreover,
\begin{equation}\label{eq:bob13x1}
\frac{d}{dt}F_{C,t}\left(D\right)\vert_{t=0}=\Tr\left[D\right].
\end{equation}

We will now restrict the above forms to the $D\in \mathcal{A}'$ such 
that $D^{2}=C$. Since $C$ is central, this set is invariant by 
conjugation by invertible elements in $\mathcal{A}'$. Then (\ref{eq:bob13}) takes the form
\begin{equation}\label{eq:bob14}
F_{C,t}\left(D\right)=\sqrt{t}\int_{0}^{t}\frac{1}{2\sqrt{u}}\Tr\left[D\exp\left(-\left(t-u\right)D^{2}\right)\right]du.
\end{equation}
We can rewrite (\ref{eq:bob14}) in the form
\begin{equation}\label{eq:bob15}
F_{C,t}\left(D\right)=t\int_{0}^{1}\frac{1}{2\sqrt{u}}\Tr\left[D\exp\left(-t
\left(1-u\right)D^{2}\right)\right]du.
\end{equation}
Let $*$ denote the convolution of distributions with support in 
$\R_{+}$. Then (\ref{eq:bob14}) can be  written in the form
\begin{equation}\label{eq:bob15x1}
F_{C,t}\left(D\right)=\frac{\sqrt{t}}{2}
\left( \Tr\left[D\exp\left(-sD^{2}\right)\right]*\frac{1}{\sqrt{s}} \right) \left(t\right).
\end{equation}

We have the identity of distributions on $\R_{+}$,
\begin{equation}\label{eq:conk2}
\frac{1}{\sqrt{s}}*\frac{1}{\sqrt{s}}=\pi.
\end{equation}
Since $F_{C,t}\left(D\right)$ is a smooth function of $t\in \R_{+}$, 
$\frac{F_{C,s}\left(D\right)}{\sqrt{s}}*\frac{1}{\sqrt{s}}\left(t\right)$ is a 
smooth function of $t\in\R_{+}^{*}$.
\begin{prop}\label{Psim}
For $t> 0$, the following identity holds:
\begin{equation}\label{eq:bob15x2}
\Tr\left[D\exp\left(-tD^{2}\right)\right]=\frac{2}{\pi}\frac{d}{dt}\left[\frac{F_{C,s}\left(D\right)}{\sqrt{s}}*\frac{1}{\sqrt{s}}\right]
\left(t\right).
\end{equation}
\end{prop}
\begin{proof}
This follows from (\ref{eq:bob15x1}), (\ref{eq:conk2}).
\end{proof}

Assume now that  $D$ is 
self-adjoint, and $C=D^{2}$ is central.  We define $\Tr'\left[\frac{D}{\left\vert  D\right\vert^{2\alpha}}\right]$
by taking the 
trace of $\frac{D}{\left\vert  D\right\vert^{2\alpha}}$ in which the possible 
zero eigenvalues of $D$ have been eliminated. 
By (\ref{eq:bob15}), for $0<\alpha<1$, then
\begin{equation}\label{eq:bob15x3}
\frac{1}{\Gamma\left(\alpha\right)}\int_{0}^{+ \infty }t^{\alpha-2}
F_{C,t}\left(D\right)dt=\frac{\Gamma\left(1/2\right)\Gamma\left(1-\alpha\right)}{
\Gamma\left(\frac{3}{2}-\alpha\right)}\frac{1}{2}\Tr'\left[\frac{D}{\left\vert  D\right\vert^{2\alpha}}\right].
\end{equation}
For $\alpha=1/2$, we get
\begin{equation}\label{eq:bob16}
\frac{1}{\sqrt{\pi}}\int_{0}^{+ \infty }t^{-3/2}F_{C,t}\left(D\right)dt=\frac{\pi}{2}
\Tr'\left[\frac{D}{\left\vert  D\right\vert}\right].
\end{equation}
\end{exa}
\begin{defin}\label{Dsym}
The self-adjoint operator $D$ is said to be fully symmetric if its 
spectrum, counted with  multiplicity, is invariant under the 
 map $\lambda\to -\lambda$.
\end{defin}
\begin{prop}\label{psym}
The self-adjoint operator $D$ is fully symmetric if and only if for any 
$t>0$, $F_{C,t}\left(D\right)=0$.
\end{prop}
\begin{proof}
The operator $D$ is fully symmetric if and only if for any $s\ge 0$, 
\begin{equation}\label{eq:conk1}
\Tr\left[D\exp\left(-sD^{2}\right)\right]=0.
\end{equation}
By (\ref{eq:bob15x1}), (\ref{eq:bob15x2}), we get our proposition.
\end{proof}

%% file: Eta2.tex
\section{The hypoelliptic Laplacian on a symmetric space}%
\label{sec:eta}
Let $G$ be a connected and simply connected reductive group, and let $X=G/K$ be the symmetric space 
associated with $G$. 
In this section, we  recall the construction in
\cite{Bismut08b} of the hypoelliptic Laplacian, a  deformation 
$\mathcal{L}^{X}_{b}\vert_{b>0}$ of 
the  operator $C^{\mathfrak g,X}/2$ acting on $C^{\infty}\left(X,S^{\overline{TX}} 
\otimes F\right)$, where $S^{\overline{TX}}$ is the bundle of 
$\overline{TX}$  spinors, and $F$ is a homogeneous vector bundle on $X$. In 
section \ref{sec:defdx}, we will add an extra deformation parameter $\vartheta\in 
\left[0,\frac{\pi}{2}\right[$, the constructions  in the 
present section corresponding to the case $\vartheta=0$.

This section is organized as follows. In subsection 
\ref{subsec:redgp}, we introduce the connected reductive group $G$, its Lie algebra 
$\mathfrak g$, and the symmetric space $X$.

In subsection \ref{subsec:actcart}, we lift the Cartan involution to 
homogeneous vector bundles on $X$.

In subsection \ref{subsec:env}, we construct the Casimir operator 
$C^{ \mathfrak g}$.

In subsection \ref{subsec:clifgtx}, we introduce the  exterior 
algebra 
$\Lambda\ac\left(\mathfrak g^{*}\right)$, and the Clifford algebras 
$c\left(\mathfrak g\right),\widehat{c}\left(\mathfrak g\right)$ 
associated with an invariant bilinear form $B$ on $\mathfrak g$.

In subsection \ref{subsec:symalg}, we describe a few properties of  the symmetric algebra 
$S\ac\left(\mathfrak g^{*}\right)$, and the Bargmann isomorphism, 
that identifies a completion of $S\ac\left(\mathfrak g^{*}\right)$ 
with $L_{2}\left(\mathfrak g\right)$.

In subsection \ref{subsec:spin}, when $G$ is simply connected, if 
$\overline{TX}$ is another copy of $TX$, we 
construct the bundle of spinors $S^{\overline{TX}}$ associated with 
$\overline{TX}$.

In subsection \ref{subsec:twi}, we define the Dirac operator  
$\widehat{D}^{X}$ acting on $C^{ \infty }\left(X,S^{\overline{TX}} 
\otimes F\right)$, and the elliptic operator $\mathcal{L}^{X}_{0}$ acting 
on the same vector space, that differs by a constant from the action 
of $\frac{1}{2}C^{ \mathfrak g}$ on $C^{ \infty }\left(X,S^{\overline{TX}} 
\otimes F\right)$,  and by another constant 
from $-\frac{1}{2}\widehat{D}^{X,2}$.

In subsection \ref{subsec:opdb}, following \cite{Bismut08b},
we use the Dirac operator of Kostant 
$\widehat{D}^{\mathfrak g}$ to define
an operator $\mathfrak D_{b}$ acting on $C^{ \infty }\left(G\times 
\mathfrak g,\Lambda\ac\left(\mathfrak g^{*}\right)\right)$.

In subsection \ref{subsec:comp}, we establish a simple compression 
property of the operator $\mathfrak D_{b}$. 

In subsection \ref{subsec:formDb2}, we state a formula of 
\cite{Bismut08b} for $\mathfrak D_{b}^{2}$.

Let $\pi:\widehat{\mathcal{X}}\to X$ be the total space of the vector bundle 
 $TX \oplus N=G\times _{K} 
\mathfrak g$ on $X$.
In subsection \ref{subsec:actqu}, by quotienting the above 
constructions by $K$, we descend 
the operator $\mathfrak D_{b}$ to an operator $\mathfrak D_{b}^{X}$ 
acting on $C^{\infty }\left(\widehat{\mathcal{X}},\pi^{*}
\left(\Lambda\ac\left(T^{*}X \oplus N^{*}\right)\otimes  S^{\overline{TX}} 
\otimes F\right)\right)$. Also we construct the hypoelliptic Laplacian 
$\mathcal{L}^{X}_{b}$ that   acts on the same vector space.

Finally, in subsection \ref{subsec:formre}, we give an important  formula 
established in \cite{Bismut08b} that explains why the family 
$\mathcal{L}^{X}_{b}\vert_{b>0}$
 deforms $\mathcal{L}^{X}_{0}$ as $b\to 0$.

In the constructions of the present section, the Dirac operator $\widehat{D}^{X}$ only 
appears through its square $\widehat{D}^{X,2}$ (that coincides, up 
to a constant, with the action of $-\frac{1}{2}C^{ \mathfrak g}$), the family of operators  
$\mathfrak D^{X}_{b}\vert_{b=0}$ being a deformation of the $0$ 
operator. In section \ref{sec:defdx}, we will resurrect the operator $\widehat{D}^{X}$ 
through another family of operators, by introducing an extra 
parameter $\vartheta\in\left[0,\frac{\pi}{2}\right[$.
\subsection{A connected reductive group}%
\label{subsec:redgp}
Let 
\index{G@$G$}%
$G$ be a real connected reductive group, and let 
\index{th@$\theta$}%
$\theta$ be its 
Cartan involution. Let 
\index{K@$K$}%
$K \subset G$ be the subgroup of $G$ fixed by 
$\theta$, so that $K$ is a maximal compact subgroup.  Then $K$ is 
also connected.

Let 
\index{g@$\mathfrak g$}%
\index{k@$\mathfrak k$}%
$\mathfrak g,\mathfrak k$ be the Lie algebras of $G,K$. Then 
$\mathfrak  k$ is the $+1$ eigenspace of $\theta$ in $\mathfrak g$. Let
\index{p@$\mathfrak p$}%
$\mathfrak p$ be the $-1$ eigenspace of $\theta$ so that
\begin{equation}\label{eq:Lie1}
\mathfrak g= \mathfrak p \oplus \mathfrak k.
\end{equation}
Set
\index{m@$m$}%
\index{n@$n$}%
\begin{align}\label{eq:Lie1x1}
&m=\dim \mathfrak p, &n=\dim \mathfrak k.
\end{align}

Let 
\index{tg@$\theta^{\mathfrak g}$}%
$\theta^{\mathfrak g}$ be the $\mathfrak g$-valued canonical 
left-invariant $1$-form on $\mathfrak g$, and let $\theta^{\mathfrak 
g}=\theta^{\mathfrak p} + \theta^{ \mathfrak k}$ denote its splitting with respect to 
(\ref{eq:Lie1}).

Let 
\index{B@$B$}%
$B$ 
be a bilinear symmetric nondegenerate  form on $\mathfrak g$ which is 
$G$-invariant, and also invariant under $\theta$, let 
\index{f@$\varphi$}%
$\varphi:\mathfrak g\to \mathfrak g^{*}$ denote the canonical 
isomorphism induced by $B$. Then 
(\ref{eq:Lie1}) is an orthogonal splitting with respect to $B$. We 
assume $B$ to be positive on $\mathfrak p$, and negative on 
$\mathfrak k$. Let $\left\langle  
\,\right\rangle=-B\left(\cdot,\theta\cdot\right)$ be the induced 
scalar product on $\mathfrak g$. Then (\ref{eq:Lie1}) is still an 
orthogonal splitting with respect to $\left\langle  \,\right\rangle$. 
If $a\in \mathfrak g$, $\ad\left(a\right)\in \End\left(\mathfrak g\right)$ is antisymmetric with 
respect to $B$. If $a\in \mathfrak k$, $\ad\left(a\right)$ preserves $\mathfrak p$ and 
$\mathfrak k$,  and it is antisymmetric with respect to $\left\langle  
\,\right\rangle$. If $a\in \mathfrak p$, $\ad\left(a\right)$ 
exchanges $\mathfrak p$ and $\mathfrak k$, and it is symmetric with 
respect to $\left\langle  \,\right\rangle$.

Let 
\index{X@$X$}%
$X=G/K$ be the  symmetric space associated with $\left(G,K\right)$, and let 
\index{p@$p$}%
$p:G\to X$ 
be the corresponding  projection.  Then $\theta^{\mathfrak k}$ defines a 
connection form on the $K$-bundle $p:G\to X$.  The tangent bundle $TX$ to $X$ is given by
\begin{equation}\label{eq:Lie2}
TX=G\times _{K}\mathfrak p.
\end{equation}
Then $TX$ comes 
equipped with a metric $g^{TX}$ induced by $B$, and with a Euclidean connection 
\index{nTX@$\n^{TX}$}%
$\n^{TX}$,  which coincides with the Levi-Civita connection of $TX$. 
Let
\index{d@$d$}%
$d$ denote the Riemannian distance on $X$.

Set
\index{N@$N$}%
\begin{equation}\label{eq:Lie3}
N=G\times _{K} \mathfrak k.
\end{equation}
The vector bundle  $N$ is also equipped with a metric $g^{N}$ and with a Euclidean 
connection 
\index{nN@$\n^{N}$}%
$\n^{N}$.  Let $\n^{TX \oplus N}$ be the connection on $TX 
\oplus N$ that is induced by $\n^{TX}, \n^{N}$. 

Clearly,
\begin{equation}\label{eq:Lie4}
TX \oplus N=G\times_{K} \mathfrak g.
\end{equation}
The bilinear form $B$  descends to $TX \oplus N$. As explained in 
\cite[section 2.2]{Bismut08b}, since $G$ also acts on $\mathfrak g$, 
the map $\left(g,a\right)\in G\times \mathfrak 
g\to\Ad\left(g\right)a\in \mathfrak g$ identifies $TX \oplus N$ with 
the trivial vector bundle $\mathfrak g$ on $X$. Let 
\index{nTXN@$\n^{TX \oplus N,f}$}%
$\n^{TX \oplus 
N,f}$ denote the corresponding flat connection on the bundle of Lie 
algebras $TX \oplus N$.  By \cite[eq. (2.2.2)]{Bismut08b}, we get
\begin{equation}\label{eq:Lie5}
\n^{TX \oplus N,f}=\n^{TX \oplus N}+\ad\left(\cdot\right).
\end{equation}
Let 
\index{nTXN@$\n^{TX \oplus N f*}$}%
$\n^{TX \oplus N f*}$ be the flat connection on $TX \oplus N$, 
\begin{equation}\label{eq:Lie4x1}
\n^{TX \oplus N,f*}=\n^{TX \oplus N}-\ad\left(\cdot\right).
\end{equation}
 The above flat connections  preserve $B$.
 
 Let  $E$ be a finite dimensional real or complex vector space, let 
$\rho^{E}:K\to \Aut\left(E\right)$ be a representation of $K$ on $E$, 
that preserves a scalar or Hermitian product.
Let $F$ the vector bundle on $X$,
\begin{equation}\label{eq:cart1}
F=G\times_{K}E.
\end{equation}
Then $F$ is a Euclidean or Hermitian vector bundle, that is 
canonically equipped with a metric preserving connection 
\index{nF@$\n^{F}$}%
$\n^{F}$.
 If $a\in N$, let $\rho^{F}\left(a\right)\in 
\End\left(F\right)$ correspond to the action of $a$ via the 
representation $\rho^{E}$.

Let 
\index{RF@$R^{F}$}%
$R^{F}$ be the curvature of $\n^{F}$. If 
$e,f\in TX$, then
\begin{equation}\label{eq:Lie13y1}
R^{F}\left(e,f\right)=-\rho^{F}\left(\left[e,f\right]\right).
\end{equation}
\subsection{The action of the Cartan involution}%
\label{subsec:actcart}
The Cartan 
involution $\theta$ acts isometrically on $X$. It  preserves the orientation of $X$ if $m$ is 
even, reverses the orientation if $m$ is odd.

The action of $\theta$ on $X$ lifts to 
two possible isometric involutions $\theta_{\pm}$ of $F$ that preserve $\n^{F}$.  Namely, if
$\left(g,f\right)\in G\times_{K}E$, then
\begin{equation}\label{eq:cart2}
\theta_{\pm}\left(g,f\right)=\left(\theta g,\pm f\right).
\end{equation}
Then $\theta_{-}=-\theta_{+}$. If $E = \mathfrak p$, then $F=TX$, and $\theta_{-}$ 
corresponds to the derivative of the action of $\theta$ on $X$. 

More 
generally, if $\left(g,h\right)\in G\times_{K} \mathfrak g=TX \oplus N 
$, one can also define a lift $\Theta$  of $\theta$ to $TX \oplus N$ by the 
formula
\begin{equation}\label{eq:cart3}
\Theta\left(g,h\right)=\left(\theta g,\theta h\right).
\end{equation}

\subsection{The Casimir operator}%
\label{subsec:env}
 In the sequel, we identify $\mathfrak g$ to the vector space of 
    left-invariant vector fields on $G$. The
    enveloping algebra 
    \index{Ug@$U\left(\mathfrak g\right)$}%
    $U\left(\mathfrak g\right)$ will be identified with the 
     algebra of left-invariant differential operators on $G$.
 
     Let 
     \index{Cg@$C^{\mathfrak g}$}%
     $C^{\mathfrak g}\in U\left( \mathfrak g\right)$ be the
  Casimir element of $G$. If  $e_{1},\ldots,e_{m+n}$ is a basis  of 
  $\mathfrak g$ and if $e_{1}^{*},\ldots,e^{*}_{m+n}$ is the dual 
  basis of $\mathfrak g$ with respect to  $B$, then
  \begin{equation}
     C^{\mathfrak g}=-\sum_{i=1}^{m+n}e^{*}_{i}e_{i}.
      \label{eq:kos16-1}
  \end{equation}
  Then $C^{\mathfrak g} $ lies in the centre of $ U\left( \mathfrak g\right)$.
  
  If we assume that $e_{1},\ldots,e_{m}$ is an orthonormal basis of 
  $\mathfrak p$ and $e_{m+1},\ldots,e_{m+n}$ is an orthonormal basis 
  of $\mathfrak k$, then
  \begin{align}\label{eq:toplitz0}
e^{*}_{i}=&\,e_{i}\ \mathrm{for}\ 1\le i\le m,\\
&\,-e_{i}\,\mathrm{for}\ m+1\le i\le m+n. \nonumber 
\end{align}
In particular,
  \begin{equation}
    C^{\mathfrak g}=-\sum_{i=1}^{m}e_{i}^{2}+\sum_{i=m+1}^{m+n}e_{i}^{2}.
      \label{eq:kos16-2}
  \end{equation}

  The Casimir operator 
  \index{Ck@$C^{\mathfrak k}$}%
  $C^{\mathfrak k}$ of $K$ will be
  calculated with respect to  the bilinear form induced by $B$ on  
  $\mathfrak k$, i.e.,
  \begin{equation}
      C^{\mathfrak k}=\sum_{i=m+1}^{m+n}e_{i}^{2}.
      \label{eq:kos16-3}
  \end{equation}
  
  In the sequel, we use the notation
  \index{CgH@$C^{\mathfrak g,H}$}%
  \begin{equation}\label{eq:kos16-3a1}
C^{\mathfrak g,H}=-\sum_{i=1}^{m}e_{i}^{2}.
\end{equation}
By (\ref{eq:kos16-2})--(\ref{eq:kos16-3a1}), we get
\begin{equation}\label{eq:kos16-3a2}
C^{\mathfrak g}=C^{\mathfrak g,H}+C^{\mathfrak k}.
\end{equation}
Moreover, 
\begin{equation}\label{eq:comm1}
\left[C^{\mathfrak g,H},C^{\mathfrak k}\right]=0.
\end{equation}

  Let $E,F$ be taken as in (\ref{eq:cart1}).  Let
  \index{CkE@$C^{ \mathfrak k,E}$}%
  $C^{ \mathfrak 
  k,E}\in \End\left(E\right)$ be the associated Casimir operator 
  acting on $E$, 
  \index{CkE@$ C^{ \mathfrak k,E}$}%
  \begin{equation}
      C^{ \mathfrak k,E}=\sum_{i=m+1}^{m+n}\rho^{E,2}\left(e_{i}\right).
      \label{eq:kos16-6a}
  \end{equation}
  Then $C^{\mathfrak k,E}$ commutes with the $\rho^{E}\left(k\right), 
  k\in K$. If $\rho^{E}$ is irreducible, $C^{\mathfrak k,E}$ is a 
  constant. 
  We denote by  
  \index{CkF@$C^{\mathfrak k,F}$}%
  $C^{\mathfrak k,F}$ the self-adjoint parallel 
  endomorphism of $F$ that corresponds to $C^{\mathfrak k,E}$.
  
  In particular $C^{ \mathfrak k, \mathfrak k}\in \End\left(\mathfrak k\right),C^{ \mathfrak 
  k,\mathfrak p}\in \End\left( \mathfrak p\right)$ are the Casimir 
  operators  associated with the  actions of $K$ on $\mathfrak 
  k,\mathfrak p$.
  
  Let 
  \index{SX@$S^{X}$}%
  $S^{X}$ be the scalar curvature of $X$. By \cite[eq. 
  (2.6.8)]{Bismut08b}, we get
  \begin{equation}\label{eq:qsic-7}
S^{X}=\Tr^{\mathfrak p}\left[C^{\mathfrak k,\mathfrak p}\right].
\end{equation}
\subsection{The algebras $\Lambda\ac\left(\mathfrak g^{*}\right)$ and 
$c\left(\mathfrak g\right),\widehat{c}\left(\mathfrak g\right)$}%
     \label{subsec:clifgtx}
     Let $\Lambda\ac\left(\mathfrak g^{*}\right)$ be the exterior 
     algebra of $\mathfrak g^{*}$.  Let
     \index{NLg@$N^{\Lambda\ac\left( \mathfrak g^{*}\right)}$}%
     $N^{\Lambda\ac\left( \mathfrak g^{*}\right)}$ denote the number operator of 
     $\Lambda\ac\left(\mathfrak g^{*}\right)$. Let 
	 \index{Bg@$B^{*}$}%
	 $B^{*}$ be the 
	 bilinear symmetric form 
     on $\Lambda\ac\left(\mathfrak g^{*}\right)$ that is induced by 
     $B$.  
     
     Let
     \index{kg@$\kappa^{\mathfrak g}$}%
     $\kappa^{\mathfrak g}\in \Lambda^{3}\left(\mathfrak 
     g^{*}\right)$ be such that if $a,b,c\in \mathfrak g$,
     \begin{equation}\label{eq:form1}
\kappa^{\mathfrak 
g}\left(a,b,c\right)=B\left(\left[a,b\right],c\right).
\end{equation}
We denote by $\kappa^{\mathfrak k}\in \Lambda^{3}\left(\mathfrak k^{*}\right)$ 
the corresponding form associated with the restriction of $B$ to 
$\mathfrak k$. By \cite[eq. (2.6.7) and Proposition 2.6.1]{Bismut08b},
\begin{align}\label{eq:qsic-10}
&B^{*}\left(\kappa^{\mathfrak g},\kappa^{\mathfrak 
g}\right)=\frac{1}{2}\Tr^{\mathfrak p}\left[C^{\mathfrak k, 
\mathfrak p}\right]+\frac{1}{6}\Tr^{\mathfrak k}\left[C^{\mathfrak k, 
\mathfrak k}\right],
&B^{*}\left(\kappa^{\mathfrak k},\kappa^{\mathfrak 
k}\right)=\frac{1}{6}\Tr^{\mathfrak k}\left[C^{\mathfrak k, \mathfrak 
k}\right].
\end{align}

  We follow \cite[sections 1.1 and  2.3]{Bismut08b}.  Let 
  \index{cg@$c\left(\mathfrak g\right)$}%
  \index{cg@$\widehat{c}\left(\mathfrak g\right)$}%
  $c\left(\mathfrak g\right), 
     \widehat{c}\left(\mathfrak g\right)$ denote the Clifford 
     algebras associated with $\left( \mathfrak g,B\right), 
     \left(\mathfrak g,-B\right)$. Then $c\left(\mathfrak g\right), 
     \widehat{c}\left(\mathfrak g\right)$ are the algebras generated 
     by $1\in \R, e\in \mathfrak g$, and the commutation relations for 
     $e,f\in \mathfrak g$  given by $ef+fe=-2B\left(e,f\right)$ and 
     $ef+fe=2B\left(e,f\right)$.

    Recall that 
	\index{f@$\varphi$}%
$\varphi: \mathfrak g\to \mathfrak g^{*}$ is the canonical 
identification induced by $B$. If $a\in \mathfrak g$, let 
 \index{ca@$c\left(a\right)$}%
                     \index{ca@$\widehat{c}\left(a\right)$}%
 $c\left(a\right),\widehat{c}\left(a\right)\in 
 \End\left(\Lambda\ac\left(\mathfrak g^{*}\right)\right)$ be given by
\begin{align}
    &c\left(a\right)=\varphi a\we -i_{a}, 
    &\widehat{c}\left(a\right)=\varphi a\we+i_{a}.
    \label{eq:kos6}
\end{align}
Then  $c\left(a\right)$ and 
$\widehat{c}\left(a\right)$ are odd operators, which are  respectively antisymmetric and 
symmetric  with respect to $B^{*}$. If $a,b\in \mathfrak g$, then
\begin{align}
    &\left[c\left(a\right),c\left(b\right)\right]=-2B\left(a,b\right),
    &\left[\widehat{c}\left(a\right),\widehat{c}\left(b\right)\right]=
    2B\left(a,b\right),\qquad
    \left[c\left(a\right),\widehat{c}\left(b\right)\right]=0.
    \label{eq:kos7}
\end{align}
By (\ref{eq:kos7}),  $\Lambda\ac\left(\mathfrak g^{*}\right)$ is a 
$c\left(\mathfrak g\right)$ and a $\widehat{c}\left(\mathfrak 
g\right)$ Clifford module. As explained in \cite[section 
1.1]{Bismut08b}, from the above we get canonical isomorphisms of 
$\Z_{2}$-graded vector spaces,
\begin{align}\label{eq:Lie7}
&c\left(\mathfrak g\right) \simeq \Lambda\ac\left(\mathfrak 
g^{*}\right),
&\widehat{c}\left(\mathfrak g\right) \simeq \Lambda\ac\left(\mathfrak 
g^{*}\right).
\end{align}
The action of $c\left(\mathfrak g\right)$ on 
$\Lambda\ac\left(\mathfrak g^{*}\right)$ corresponds to  left 
multiplication on $c\left(\mathfrak g\right)$, and the action of 
$\widehat{c}\left(\mathfrak g^{*}\right)$ to right multiplication 
on $c\left(\mathfrak g\right)$ multiplied  by
$\left(-1\right)^{N^{\Lambda\ac\left(\mathfrak g^{*}\right)}}$. 
By \cite[eq. (1.1.15)]{Bismut08b}, we get
\begin{equation}\label{eq:ors0}
N^{\Lambda\ac\left(\mathfrak 
g^{*}\right)}=\frac{1}{2}c\left(e^{*}_{i}\right)\widehat{c}\left(e_{i}\right)+\frac{1}{2}\left(m+n\right).
\end{equation}

 If $a\in \mathfrak g$, we will  denote by 
     $c\left(a\right),\widehat{c}\left(a\right)$  the 
     corresponding elements in  $c\left(\mathfrak 
     g\right),\widehat{c}\left(\mathfrak g\right)$. There will be no 
     risk of confusion with the above definitions of 
     $c\left(a\right), \widehat{c}\left(a\right)$.
     
     Let
     \index{Ag@$\mathcal{A}\left(\mathfrak g\right)$}%
     $\mathcal{A}\left(\mathfrak g\right)$ be the Lie algebra of  
endomorphisms of $\mathfrak g$ that are antisymmetric with respect to 
$B$. Then $\mathcal{A}\left(\mathfrak g\right)$ embeds as a Lie subalgebra of 
$c\left(\mathfrak g\right)$ and $\widehat{c}\left(\mathfrak 
g\right)$. Namely, let $e_{1},\ldots,e_{m+n}$ be a basis of $\mathfrak g$, let 
    $e^{*}_{1},\ldots,e^{*}_{m+n}$ be the corresponding dual basis of 
    $\mathfrak g$ with respect to $B$. If $A\in 
    \mathcal{A}\left(\mathfrak g\right)$, as in  \cite[eqs. (1.1.9) 
    and (1.1.11)]{Bismut08b}, we define 
    \index{cA@$c\left(A\right)$}%
    \index{cA@$\widehat{c}\left(A\right)$}%
    $c\left(A\right)\in 
    c\left(\mathfrak g\right), \widehat{c}\left(A\right)\in 
    \widehat{c}\left(\mathfrak g\right)$ by the formulas
    \begin{align}\label{eq:ors1}
&c\left(A\right)=\frac{1}{4}B\left(Ae^{*}_{i},e^{*}_{j}\right)c\left(e_{i}\right)c\left(e_{j}\right),
&\widehat{c}\left(A\right)=-\frac{1}{4}B\left(Ae^{*}_{i},e^{*}_{j}\right)\widehat{c}\left(e_{i}\right)\widehat{c}\left(e_{j}\right).
\end{align}
Then if $e\in \mathfrak g$, 
\begin{align}\label{eq:ors2}
&\left[c\left(A\right),c\left(e\right)\right]=c\left(Ae\right),&\left[\widehat{c}\left(A\right),\widehat{c}\left(e\right)\right]=
\widehat{c}\left(Ae\right).
\end{align}

   Let 
   \index{ckg@$c\left(\kappa^{\mathfrak g}\right)$}%
   \index{ckg@$\widehat{c}\left(-\kappa^{\mathfrak g}\right)$}%
   $c\left(\kappa^{\mathfrak 
    g}\right)\in c\left(\mathfrak g \right),\widehat{c}\left(-\kappa^{\mathfrak g}\right)
    \in \widehat{c}\left(\mathfrak g\right)$ be given by
    \begin{align}\label{eq:kos18}
       &  c\left(\kappa^{\mathfrak g}\right)=\frac{1}{6}\kappa^{\mathfrak g}\left(e^{*}_{i},e^{*}_{j},e^{*}_{k}\right)c\left(e_{i}\right)
	 c\left(e_{j}\right)c\left(e_{k}\right),\\
	 &\widehat{c}\left(-\kappa^{\mathfrak g}\right)=-\frac{1}{6}\kappa^{\mathfrak g}\left(e^{*}_{i},e^{*}_{j},e^{*}_{k}\right)\widehat{c}\left(e_{i}\right)
	 \widehat{c}\left(e_{j}\right)\widehat{c}\left(e_{k}\right). 
	 \nonumber 
	 \end{align}
	Then $c\left(\kappa^{\mathfrak 
    g}\right),\widehat{c}\left(-\kappa^{\mathfrak g}\right)$ 
    correspond to $\kappa^{\mathfrak g},-\kappa^{\mathfrak g}$ by the 
    above canonical isomorphisms.

     By (\ref{eq:Lie1}), we get
     \begin{equation}\label{eq:Lie6}
\Lambda\ac\left(\mathfrak g^{*}\right)=\Lambda\ac\left(\mathfrak 
p^{*}\right)  \ho\Lambda\ac\left( \mathfrak k^{*}\right).
\end{equation}
     By restricting $B$ to $ \mathfrak p, \mathfrak k$, we obtain the 
     Clifford algebras $c\left( \mathfrak p\right),\widehat{c}\left( 
     \mathfrak p\right),c\left( \mathfrak k\right),\widehat{c}\left( 
     \mathfrak k\right)$. Since the splitting $\mathfrak g= \mathfrak 
     p \oplus \mathfrak k$ is orthogonal with respect to $B$, we get
     \begin{align}\label{eq:sumex6}
&c\left( \mathfrak g\right)=c\left( \mathfrak p\right)\ho c\left( 
\mathfrak k\right),
&\widehat{c}\left( \mathfrak g\right)=\widehat{c}\left( \mathfrak 
p\right)\ho \widehat{c}\left( \mathfrak k\right).
\end{align}
The  algebras 
     $\Lambda\ac\left(\mathfrak 
     p^{*}\right),\Lambda\ac\left(\mathfrak k^{*}\right)$ descend to 
    the bundles of algebras
     $\Lambda\ac\left(\TsX\right),\Lambda\ac\left(N^{*}\right)$, and 
     $\Lambda\ac\left(\mathfrak g^{*}\right)$ descends to 
     $\Lambda\ac\left(\TsX\oplus N^{*}\right)$. The vector bundle $TX \oplus N$ is also equipped with 
     the bilinear form $B$. Let $c\left(TX \oplus N\right), 
     \widehat{c}\left(TX \oplus N\right)$ be the associated bundles of 
     Clifford algebras.   Let 
     $c\left(TX\right),c\left(N\right),\widehat{c}\left(TX\right),\widehat{c}\left(N\right)$ denote
     the  bundles of algebras  associated with the 
     restriction of $B$ to $TX,N$.
     As in (\ref{eq:sumex6}), we get
     \begin{align}\label{eq:sumex7}
    &c\left( TX \oplus N\right)=c\left( TX\right)\ho c\left( 
    N\right),
    &\widehat{c}\left( TX \oplus N\right)=\widehat{c}\left( TX
    \right)\ho \widehat{c}\left( N\right).
    \end{align}
\subsection{The  symmetric algebra  $S\ac\left(\mathfrak g^{*}\right)$}%
    \label{subsec:symalg}
    Let 
    \index{Sg@$S\ac\left(\mathfrak g^{*}\right)$}%
    $S\ac\left(\mathfrak g^{*}\right)$ denote the symmetric 
    algebra of $\mathfrak g^{*}$. Equivalently $S\ac\left(\mathfrak 
    g^{*}\right)$ is the polynomial algebra on $\mathfrak g$. Let 
    $N^{S\ac\left(\mathfrak g^{*}\right)}$ be the number operator of 
    $S\ac\left(\mathfrak g^{*}\right)$.
    By 
    (\ref{eq:Lie1}), as in (\ref{eq:Lie6}), we get
    \begin{equation}\label{eq:Lie8}
S\ac\left(\mathfrak g^{*}\right) =S\ac\left(\mathfrak p^{*}\right) 
\otimes S\ac\left(\mathfrak k^{*}\right).
\end{equation}

For the moment, we view $\mathfrak g= \mathfrak p \oplus \mathfrak k$ 
as a Euclidean vector space. Then $\left\langle  \,\right\rangle$ 
induces a scalar product on $S\ac\left(\mathfrak g^{*}\right)$. Let 
$\overline{S}\ac\left(\mathfrak g^{*}\right)$ denote the Hilbert 
completion of $S\ac\left(\mathfrak g^{*}\right)$.

Let 
\index{Dg@$\Delta^{\mathfrak g}$ }%
$\Delta^{\mathfrak g}$ be the Laplacian on the Euclidean vector 
space $\mathfrak g$. Let 
\index{Hg@$H^{\mathfrak g}$}%
$H^{\mathfrak g}$ be the harmonic oscillator on $\mathfrak 
g$, so that if $Y$ is the generic element of $\mathfrak g$,
\begin{equation}\label{eq:Lie8x1}
H^{\mathfrak g}=\frac{1}{2}\left(-\Delta^{\mathfrak g}+\left\vert  
Y\right\vert^{2}-m-n\right).
\end{equation}

Let 
\index{L2g@$L_{2}\left(\mathfrak g\right)$}%
$L_{2}\left(\mathfrak g\right)$ denote the Hilbert space of square 
integrable functions on $\mathfrak g$. As explained in \cite[section 
1.4]{Bismut08b}, we have the classical Bargmann isomorphism 
$\overline{S}\ac\left(\mathfrak g^{*}\right) \simeq 
L_{2}\left(\mathfrak g\right)$. Under this isomorphism, 
$N^{S\ac\left(\mathfrak g^{*}\right)}$ corresponds to the harmonic 
oscillator $H^{\mathfrak g}$.

Then $S\ac\left(\mathfrak p^{*}\right),S\ac\left(\mathfrak 
k^{*}\right)$ descends to 
$S\ac\left(\TsX\right),S\ac\left(N^{*}\right)$. Also $H^{\mathfrak 
g}$ descend to the harmonic oscillator $H^{TX \oplus N}$ along the 
fibres of $TX \oplus N$. The Bargmann isomorphism identifies  
$\overline{S}\ac\left(T^{*}X \oplus N^{*}\right)$ with $L_{2}\left(TX 
\oplus N\right)$, and it maps $N^{S\ac\left(\TsX \oplus 
N^{*}\right)}$ to $H^{TX \oplus N}$.
\subsection{The spinors of $\overline{TX}$}%
\label{subsec:spin}
We fix once and for all an orientation of $\mathfrak p$, which in 
turn defines an orientation of $X$.

Let $\overline{\mathfrak p}$ denote another copy of $\mathfrak p$, 
which we equip 
 with the orientation corresponding to the 
orientation of $\mathfrak p$. We 
equip $\overline{\mathfrak p}$ with the scalar product corresponding 
to the scalar product of $\mathfrak p$. If 
$e\in \mathfrak p$, let $\overline{e}$ denote the corresponding 
element of $\overline{\mathfrak p}$. 

Let $c\left(\overline{\mathfrak p}\right)$ be the Clifford algebra of 
$\overline{\mathfrak p}$. If $e\in 
\mathfrak p$, let $c\left(\overline{e}\right)$ be the corresponding 
element in $c\left(\overline{\mathfrak p}\right)$. 

Let 
$S^{ \overline{\mathfrak p}}$ denote the Hermitian vector space of spinors 
that is associated with the Euclidean vector space $\overline{\mathfrak p}$. Then $S^{\overline{\mathfrak p}}$ 
is a $c\left(\overline{\mathfrak p}\right)$ Clifford module. If $e\in \mathfrak p$, 
$c\left(\overline{e}\right)$ acts as a skew-adjoint 
operator on $S^{\overline{\mathfrak p}}$. 

Everything we did for $c\left(\overline{\mathfrak p}\right)$ remains 
valid for $\widehat{c}\left(\overline{\mathfrak p}\right)$. We make 
the convention that if $e\in \mathfrak p$, 
$\widehat{c}\left(\overline{e}\right)$ acts on $S^{\overline{\mathfrak p}}$ 
like $ic\left(\overline{e}\right)$. In particular 
$\widehat{c}\left(\overline{e}\right)$ acts on $S^{ \overline{\mathfrak p}}$ 
as a self-adjoint operator. Using the same conventions as in 
(\ref{eq:ors1}), if $A\in \End\left(\overline{\mathfrak p}\right)$ is 
antisymmetric, 
\begin{equation}\label{eq:bub1x-1}
c\left(A\right)=\widehat{c}\left(A\right)
\end{equation}
acts on $S^{\overline{\mathfrak p}}$.
In this specific case, using the second notation instead of the first 
one will be a matter of convenience.

If 
$m$ is even,   $S^{\overline{\mathfrak p}}$ splits as $S^{\overline{\mathfrak p}}=S^{\overline{\mathfrak p}}_{+} \oplus S^{\overline{\mathfrak p}}_{-}$, $S^{\overline{\mathfrak p}}_{\pm}$ being of 
dimension $2^{m/2-1}$. Moreover,  we have an identification of 
$\Z_{2}$-graded algebras,
\begin{equation}\label{eq:Lie9}
c\left(\overline{\mathfrak p}\right)\otimes _{\R}\C \simeq 
\End\left(S^{\overline{\mathfrak p}}\right).
\end{equation}
 
If $m$ is odd, then $S^{\overline{\mathfrak p}}$ has dimension 
$2^{\left(m-1\right)/2}$. Moreover, we have the identification of 
algebras
\begin{equation}\label{eq:Lie10}
c\left(\overline{\mathfrak p}\right)\otimes _{\R}\C \simeq 
\End\left(S^{\overline{\mathfrak p}}\right) \oplus \End\left(S^{\overline{\mathfrak p}}\right). 
\end{equation}
If we give its canonical orientation to $\overline{\mathfrak p} \oplus 
\R$, then
\begin{equation}\label{eq:Lie11}
S^{\overline{\mathfrak p}}=S^{ \overline{\mathfrak p} \oplus \R}_{+}.
\end{equation}

The Lie group $\mathrm{Spin}\left(\overline{\mathfrak p}\right)$  embeds in 
$c^{\even}\left(\overline{\mathfrak p}\right)$,  and acts unitarily on 
$S^{ \overline{\mathfrak p}}$. 
 In the sequel, we will assume  that $G$ is simply connected.
Equivalently, we suppose that $K$ is simply connected. Then the representation 
$k\in K\to \Ad\left(k\right)\in \mathrm{SO}\left(\overline{\mathfrak p}\right)$ 
lifts to a group homomorphism $k\in K\to \sigma\left(k\right)\in 
\mathrm{Spin}\left(\overline{\mathfrak p}\right)$. If $f\in \mathfrak 
k$, then
\begin{equation}\label{eq:batt1}
\sigma\left(f\right)=\widehat{c}\left(\ad\left(f\right)\vert_{\overline{\mathfrak p} }\right).
\end{equation}
By (\ref{eq:batt1}), we get
\begin{equation}\label{eq:batt2}
C^{\mathfrak k,S^{\overline{\mathfrak 
p}}}=\sum_{i=m+1}^{m+n}\widehat{c}\left(\ad\left(e_{i}\right)\vert_{\overline{\mathfrak p}}\right)^{2}.
\end{equation}
By \cite[eq. (7.8.6)]{Bismut08b}, we have the identity
\begin{equation}\label{eq:Lie11x1}
C^{\mathfrak k, S^{\overline{\mathfrak p}}}=\frac{1}{8}\Tr^{\mathfrak 
p}\left[C^{\mathfrak k, \mathfrak p}\right].
\end{equation}

Ultimately, $k\in K\to 
\sigma\left(k\right)\in \mathrm{Spin}\left(\overline{\mathfrak p}\right)$ 
defines a unitary representation of $K$ into $U\left(S^{\overline{ \mathfrak 
p}}\right)$. If $m$ is even, this representation preserves 
$S^{\overline{\mathfrak p}}_{\pm}$.

Set
\begin{equation}\label{eq:Lie12}
S^{\overline{TX}}=G\times_{K}S^{\overline{\mathfrak p}}.
\end{equation}
Then $S^{\overline{TX}}$ is a Hermitian vector bundle with connection on $X$. 
This bundle is just the  bundle of spinors of $TX$ equipped with 
the connection $\n^{S^{\overline{TX}}}$ induced by $\n^{TX}$.

Let $\n^{S^{\overline{TX}} \otimes F}$ be the connection on $S^{\overline{TX}} \otimes F$ 
that is induced by $\n^{S^{\overline{TX}}},\n^{F}$.

As we saw in subsection \ref{subsec:actcart}, $\theta_{\pm}$ acts on 
$S^{\overline{TX}} \otimes F$. The induced action of $\theta_{\pm}$ on 
$C^{ \infty }\left(X,S^{\overline{TX}}\otimes F\right)$ is given by
\begin{equation}\label{eq:Lie12z1}
\theta_{\pm}s\left(x\right)=\theta_{\pm}s\left(\theta^{-1}x\right).
\end{equation}

\subsection{The elliptic Dirac operator $\widehat{D}^{X}$}%
\label{subsec:twi}
In the sequel,  $e_{1},\ldots,e_{m}$ is an orthonormal basis of 
$\mathfrak p$, and $e_{m+1},\ldots,e_{m+n}$ is an orthonormal basis 
of $\mathfrak k$. We will use the same notation for corresponding 
orthonormal bases of $TX$ and $N$.
\begin{defin}\label{DellD}
Let $D\in c\left(\overline{\mathfrak p}\right)\otimes 
U\left(\mathfrak g\right),\widehat{D}\in 
\widehat{c}\left(\overline{\mathfrak p}\right)\otimes 
U\left(\mathfrak g\right)$ be given by
\begin{align}\label{eq:Lie14}
&D=\sum_{i=1}^{m}c\left(\overline{e}_{i}\right)e_{i},
&\widehat{D}=\sum_{i=1}^{m}\widehat{c}\left(\overline{e}_{i}\right)e_{i}.
\end{align}
Then $D,\widehat{D}$ are $K$-invariant, so that $D,\widehat{D}$ 
descend to  Dirac operators  
\index{DX@$D^{X}$}%
\index{DX@$\widehat{D}^{X}$}%
$D^{X},\widehat{D}^{X}$ acting on $C^{ \infty 
}\left(X,S^{\overline{TX}}\otimes F\right)$. These operators are 
given by
\begin{align}\label{eq:Lie15}
&D^{X}=\sum_{i=1}^{m}c\left(\overline{e}_{i}\right)\n^{S^{\overline{TX}} \otimes 
F}_{e_{i}},
&\widehat{D}^{X}=\sum_{i=1}^{m}\widehat{c}\left(\overline{e}_{i}\right)\n^{S^{\overline{TX}} \otimes 
F}_{e_{i}}.
\end{align}
Then $D^{X}$ is formally self-adjoint, and $\widehat{D}^{X}$ is 
formally skew-adjoint. 
Because of the conventions we made before, we have the identities, 
\begin{align}\label{eq:Lie5y2}
&\widehat{D}=iD,&\widehat{D}^{X}=iD^{X}.
\end{align}
Also $D^{X}$ is a classical Dirac operator.
\end{defin}
\begin{prop}\label{Pinv}
The following identities hold:
\begin{align}\label{eq:Lie16}
&\theta_{\pm} D^{X}\theta_{\pm}^{-1}=-D^{X},
&\theta_{\pm} \widehat{D}^{X}\theta_{\pm}^{-1}=-\widehat{D}^{X}.
\end{align}
\end{prop}
\begin{proof}
If   $e\in TX$, using the considerations after 
(\ref{eq:cart2}), we have the identities,
\begin{align}\label{eq:16z1}
&\theta_{\pm}c\left(\overline{e}\right)\theta_{\pm}^{-1}=c\left(\overline{e}\right),
    &\theta_{\pm}\n^{S^{\overline{TX} \otimes 
F}}_{e}\theta_{\pm}^{-1}=\n^{S^{\overline{TX} \otimes F}}_{-e}.
\end{align}
By (\ref{eq:Lie15}), (\ref{eq:16z1}), we get (\ref{eq:Lie16}).
\end{proof}

Let 
\index{CgX@$C^{ \mathfrak g,X}$}%
\index{CgHX@$C^{\mathfrak g,H,X}$}%
$C^{ \mathfrak g,X},C^{\mathfrak g,H,X}$ denote the action of 
$C^{ \mathfrak g},C^{\mathfrak g,H}$ on 
$C^{ \infty }\left(X,S^{\overline{TX}} \otimes F\right)$. By 
\cite[eqs. (2.12.17) and (7.2.6)]{Bismut08b} or by (\ref{eq:kos16-3a2}), $C^{\mathfrak g,X}$ 
splits as
\begin{equation}\label{eq:japo1}
C^{\mathfrak g,X}=C^{\mathfrak g,H,X}+C^{\mathfrak 
k,S^{\overline{TX}} \otimes F}.
\end{equation}
If 
\index{DHX@$\Delta^{X,H}$}%
$\Delta^{X,H}$ denotes the Bochner Laplacian, 
then
\begin{equation}\label{eq:japo2}
C^{\mathfrak g,H,X}=-\Delta^{X,H}.
\end{equation}

Let 
\index{LX@$\mathcal{L}^{X}_{0}$ }%
$\mathcal{L}^{X}_{0}$ \footnote{In \cite{Bismut08b}, the operator 
$\mathcal{L}^{X}_{0}$ was instead denoted $\mathcal{L}^{X}$.}
denote the operator that was defined in \cite[eq. 
(7.2.8)]{Bismut08b}, i.e.,
\begin{equation}\label{eq:Lie17}
\mathcal{L}^{X}_{0}=\frac{1}{2}C^{ \mathfrak g,X}+\frac{1}{8}B^{*}\left(\kappa^{ 
\mathfrak g},\kappa^{\mathfrak g}\right).
\end{equation}
Recall that by (\ref{eq:bub1x-1}), if $f\in \mathfrak k$, 
$\widehat{c}\left(\ad\left(f\right)\vert _{\overline{\mathfrak p}}\right)$ acts 
on $S^{\overline{\mathfrak p}}$, and if $f\in N$, 
$c\left(\ad\left(f\right)\vert_{\overline{TX}}\right)$ acts on 
$S^{\overline{TX}}$.
\begin{prop}\label{PLX}
The following identity holds:
\begin{multline}\label{eq:qsic-9a1}
\mathcal{L}^{X}_{0}=\frac{1}{2} 
\left( -\Delta^{X,H}
+\frac{1}{4}\Tr^{\mathfrak 
p}\left[C^{\mathfrak k,\mathfrak p}\right] +
2\sum_{i=m+1}^{m+n}
\widehat{c}\left(\ad\left(e_{i}\right)\vert_{\overline{TX}}\right)
\rho^{F}\left(e_{i}\right) \right) \\
+\frac{1}{48}\Tr^{\mathfrak 
k}\left[C^{\mathfrak k,\mathfrak 
k}\right]+\frac{1}{2}
C^{\mathfrak k,F}.
\end{multline}
\end{prop}
\begin{proof}
    By (\ref{eq:qsic-10}),  (\ref{eq:japo1})--(\ref{eq:Lie17}),   we get
\begin{equation}\label{eq:gzinc2c1}
\mathcal{L}^{X}_{0}=\frac{1}{2}
\left( -\Delta^{X,H}+C^{\mathfrak k,S^{\overline{TX}} \otimes F}
+\frac{1}{8}\Tr^{\mathfrak 
p}\left[C^{\mathfrak k,\mathfrak p}\right] \right) 
+\frac{1}{48}\Tr^{\mathfrak 
k}\left[C^{\mathfrak k,\mathfrak 
k}\right].
\end{equation}
Moreover, using (\ref{eq:Lie11x1}), we obtain
\begin{equation}\label{eq:japo3}
C^{\mathfrak k,S^{\overline{TX}} \otimes 
F}=\frac{1}{8}\Tr^{\mathfrak p}\left[C^{\mathfrak k, \mathfrak 
p}\right]+
C^{\mathfrak 
k,F}+2\sum_{i=m+1}^{m+n}\widehat{c}\left(\ad\left(e_{i}\right)
\vert_{\overline{TX}}\right)
\rho^{F}\left(e_{i}\right).
\end{equation}
By (\ref{eq:gzinc2c1}), (\ref{eq:japo3}), we get (\ref{eq:qsic-9a1}).
\end{proof}

By \cite[Theorem 7.2.1]{Bismut08b}, we get
\begin{equation}\label{eq:Lie17x1}
   \frac{1}{2}\widehat{D} ^{X,2}=-\mathcal{L}^{X}_{0}+\frac{1}{8}B^{*}
    \left(\kappa^{ \mathfrak k},\kappa^{ \mathfrak 
    k}\right)+\frac{1}{2}C^{ \mathfrak k,F}.
\end{equation}
Also Lichnerowicz's formula asserts that
\begin{equation}\label{eq:qsic-8}
\widehat{D}^{X,2}=\Delta^{X,H}
-\frac{S^{X}}{4}+\frac{1}{2}\sum_{1\le i,j\le 
m}^{}\widehat{c}\left(\overline{e}_{i}\right)\widehat{c}\left(\overline{e}_{j}\right)
R^{F}\left(e_{i},e_{j}\right).
\end{equation}
As was explained in \cite[section 7.2]{Bismut08b}, equations 
(\ref{eq:Lie17x1}) and (\ref{eq:qsic-8}) are equivalent. By (\ref{eq:Lie13y1}), (\ref{eq:qsic-7}),
(\ref{eq:ors1}), 
 and (\ref{eq:qsic-8}), we get
\begin{equation}\label{eq:qsic-9}
\frac{1}{2}\widehat{D}^{X,2}=\frac{1}{2}\Delta^{X,H}
-\frac{1}{8}\Tr^{\mathfrak p}\left[C^{\mathfrak k, \mathfrak 
p}\right]-\sum_{i=m+1}^{m+n}\widehat{c}\left(\ad\left(e_{i}\right)
\vert_{\overline{TX}}\right)
\rho^{F}\left(e_{i}\right).
\end{equation}

\subsection{The operator $ \mathfrak D_{b}$}%
\label{subsec:opdb}
   Set
   \index{Ag@$\mathcal{A}^{\mathfrak g}$}
   \index{Ag@$\widehat{\mathcal{A}}^{\mathfrak g}$}%
   \begin{align}
       &\mathcal{A}^{\mathfrak g}=c\left(\mathfrak g\right) \otimes U\left( \mathfrak 
       g\right),
       &\widehat{\mathcal{A}}^{\mathfrak g}=\widehat{c}\left(\mathfrak g\right) \otimes U\left(\mathfrak 
       g\right).
       \label{eq:kos19}
   \end{align}
   Then $\mathcal{A}^{ \mathfrak g}, 
   \widehat{\mathcal{A}}^{ \mathfrak g}$ are $\Z_{2}$-graded 
   algebras. Moreover, $G$ acts on $\mathcal{A}^{ \mathfrak 
   g},\widehat{\mathcal{A}}^{ \mathfrak g}$. Let 
   $e_{1},\ldots,e_{m+n}$ be a basis of $\mathfrak g$, and let 
   $e^{*}_{1},\ldots,e^{*}_{m+n}$ denote the dual basis of $\mathfrak g$ 
    with respect to $B$.
    
    Now, we introduce the Dirac operators of Kostant 
    \cite{Kostant76,Kostant97}.
  \begin{defin}\label{DKos1}
Let 
\index{Dg@$D^{ \mathfrak g}$}
\index{Dg@$\widehat{D}^{\mathfrak g}$}%
$D^{ \mathfrak g}\in \mathcal{A}^{ \mathfrak g},
\widehat{D}^{\mathfrak g}\in \widehat{\mathcal{A}}^{\mathfrak g}$ be 
the Dirac operators
\begin{align}
   &D^{\mathfrak g}=\sum_{i=1}^{m+n}c\left(e_{i}^{*}\right)e_{i}+
   \frac{1}{2}c\left(\kappa^{\mathfrak g}\right),
    &\widehat{D}^{\mathfrak 
    g}=\sum_{i=1}^{m+n}\widehat{c}\left(e^{*}_{i}\right)
    e_{i}+\frac{1}{2}\widehat{c}\left(-\kappa^{\mathfrak g}\right).
    \label{eq:kos20}
\end{align}
\end{defin}

Now we recall a   result of Kostant 
\cite{Kostant76,Kostant97}, \cite[Theorem 2.7.2]{Bismut08b}.
\begin{thm}\label{TKos1}
The following identities hold:
\begin{align}\label{eq:kos22}
&D^{\mathfrak g,2}=C^{\mathfrak g}+\frac{1}{4}B^{*}\left(\kappa^{ \mathfrak 
g},\kappa^{\mathfrak g}\right),
&\widehat{D}^{\mathfrak g,2}=-C^{\mathfrak g}-\frac{1}{4}B^{*}\left(\kappa^{ \mathfrak 
g},\kappa^{\mathfrak g}\right). 
\end{align}
\end{thm}

By (\ref{eq:Lie17}), (\ref{eq:kos22}), $\mathcal{L}^{X}_{0}$ is just the 
action of $\frac{1}{2}D^{ \mathfrak g,2}$ on $C^{ \infty }\left(X,S^{\overline{TX}} \otimes 
F\right)$.

Let $d^{\mathfrak p},d^{\mathfrak k}$ be the de Rham operators on 
 $\mathfrak p,\mathfrak k$, and let $d^{\mathfrak 
p*},d^{\mathfrak k*}$ denote their formal $L_{2}$ adjoints with respect to 
the scalar products on $\mathfrak p,\mathfrak k$. Let 
\index{Yp@$Y^{\mathfrak p}$}%
\index{Yk@$Y^{\mathfrak k}$}%
$Y^{\mathfrak 
p}, Y ^{\mathfrak k}$ be the tautological sections of $\mathfrak p, 
\mathfrak k$ on $\mathfrak p, \mathfrak k$. We identify $Y^{\mathfrak 
p}, Y^{\mathfrak k}$ to the corresponding $1$-forms via the scalar 
products of $\mathfrak p, \mathfrak k$.

If $f\in \mathfrak g$, let $\n_{e}$ denote the corresponding 
differentiation on $\mathfrak g$.  Let $e_{1},\ldots,e_{m}$ be an orthonormal basis of $\mathfrak p$, 
let $e_{m+1},\ldots,e_{m+n}$ be an orthonormal basis of $\mathfrak 
k$. We use the notation in (\ref{eq:kos6}).
Set
\index{Dp@$\mathcal{D}^{ \mathfrak p}$}%
\index{Ep@$\mathcal{E}^{ \mathfrak p}$}%
\index{Dk@$\mathcal{D}^{ \mathfrak k}$}%
\index{Ek@$\mathcal{E}^{\mathfrak k}$}%
\begin{align}\label{eq:brav12}
&\mathcal{D}^{ \mathfrak p}=\sum_{i=1}^{m}c\left(e_{i}\right)\n_{e_{i}},
&\mathcal{E}^{ \mathfrak p}=\widehat{c}\left(Y^{ 
\mathfrak p}\right),\\
&\mathcal{D} ^{\mathfrak 
k  }=\sum_{i=m+1}^{m+n}c\left(e_{i}^{*}\right)\n_{e_{i}},
&\mathcal{E}^{ \mathfrak k }=\widehat{c}\left(Y^{ \mathfrak 
k}\right). \nonumber 
\end{align}
A trivial computation \cite[eqs. (2.8.6), (2.8.11)]{Bismut08b} shows 
that
\begin{align}\label{eq:Lie18}
&\mathcal{D}^{ \mathfrak p}+\mathcal{E}^{ \mathfrak p}=d^{\mathfrak 
p}+Y^{\mathfrak p}\we+d^{\mathfrak p*}+i_{Y^{\mathfrak p}},
&\mathcal{D}^{\mathfrak k}-\mathcal{E}^{ \mathfrak k}=d^{\mathfrak 
k}+Y^{ \mathfrak k}\we-d^{\mathfrak k*}-i_{Y^{ \mathfrak k}}.
\end{align}
Let 
\index{Dp@$\Delta^{\mathfrak p}$}%
\index{Dk@$\Delta^{\mathfrak k}$}%
$\Delta^{\mathfrak p}, \Delta^{\mathfrak k}$ denote the 
Laplacians on $\mathfrak p, \mathfrak k$. By \cite[eqs. (2.8.8), (2.8.13)]{Bismut08b}, 
\begin{align}\label{eq:Lie19}
&\frac{1}{2}\left(\mathcal{D}^{ \mathfrak p}+\mathcal{E}^{ 
    \mathfrak p}\right)^{2}=\frac{1}{2}\left(-\Delta^{ \mathfrak 
    p}+\left\vert  Y^{ \mathfrak 
    p}\right\vert^{2}-m\right)+N^{\Lambda\ac\left( \mathfrak 
    p^{*}\right)},\\
    & \frac{1}{2}\left(-i\mathcal{D}^{ \mathfrak k}+i\mathcal{E}^{ 
    \mathfrak k}\right)^{2}=\frac{1}{2}\left(-\Delta^{ \mathfrak 
    k}+\left\vert  Y^{ \mathfrak 
    k}\right\vert^{2}-n\right)+N^{\Lambda\ac\left( \mathfrak 
    k^{*}\right)}. \nonumber 
\end{align}

If $k\in K$, the action of $k$   on $C^{ \infty }\left(G,\Lambda\ac\left( \mathfrak 
g^{*}\right) \otimes S\ac\left( \mathfrak g^{*}\right)\right)$ is 
given by
\begin{equation}
    ks\left(g\right)=\rho^{\Lambda\ac\left( \mathfrak g^*\right) 
    \otimes S\ac\left( \mathfrak g^{*}\right)} \left(k\right)s\left(gk\right).
    \label{eq:trubl1}
\end{equation}
Also we have the identification
\begin{equation}
    C^{ \infty }\left(G,\Lambda\ac\left( \mathfrak g^{*}\right) 
    \otimes C^{ \infty }\left( \mathfrak g\right)\right)=C^{ \infty 
    }\left(G\times \mathfrak g,\Lambda\ac\left( \mathfrak g^{*}\right)\right).
    \label{eq:semex17}
\end{equation}
The  action of $K$ on (\ref{eq:semex17}) that corresponds to 
(\ref{eq:trubl1}) is given by 
\begin{equation}
    ks\left(g,Y\right)=\rho^{\Lambda\ac\left( \mathfrak g^{*}\right)}
    \left(k\right)s\left(gk, \Ad\left(k^{-1}\right)Y\right).
    \label{eq:sumex18}
\end{equation}

Now we define the operator
\index{Db@$\mathfrak D_{b}$}%
$\mathfrak D_{b}$ as in \cite[Definition 
2.9.1]{Bismut08b}.
\begin{defin}\label{Dfanta}
For $b>0$, let $\mathfrak D_{b}\in \End \left( 
C^{ \infty 
    }\left(G\times \mathfrak g,\Lambda\ac\left( \mathfrak g^{*}\right)\right) \right) $ be given by
\begin{equation}
    \mathfrak D_{b}=\widehat{D}^{ \mathfrak g}+ic\left(\left[
    Y^{ \mathfrak k},Y^{ \mathfrak p}\right]\right)+\frac{1}{b}
    \left(\mathcal{D}^{ \mathfrak p}+\mathcal{E}^{ \mathfrak p}-
    i\mathcal{D}^{ \mathfrak k }+
    i\mathcal{E} ^{ \mathfrak k }\right).
    \label{eq:rio1}
\end{equation}
Then $\mathfrak D_{b}$ commutes with $K$. 
\end{defin}
\subsection{The compression of the operator $\mathfrak D_{b}$}%
\label{subsec:comp}
Let $Y$ be the tautological section of $\mathfrak g$ on $\mathfrak 
g$, so that $Y=Y^{ \mathfrak p}+Y^{ \mathfrak k}\in \mathfrak 
g$. We have the  identity 
\begin{equation}
    \left\vert  Y\right\vert^{2}=\left\vert  Y^{ \mathfrak p}\right\vert^{2}
    +\left\vert  Y^{ \mathfrak k}\right\vert^{2}.
    \label{eq:rio2}
\end{equation}

By \cite[section 1.6]{Bismut08b}, the kernel 
\index{H@$H$}%
$H\subset 
\Lambda\ac\left( \mathfrak g^{*}\right)\otimes L_{2}\left( \mathfrak g\right)$ of the operator 
$D^{ \mathfrak p}+\mathcal{E}^{ \mathfrak p}-i\mathcal{D}^{ \mathfrak 
k}+i\mathcal{E}^{ \mathfrak k}$ is $1$-dimensional and spanned by 
 $\exp\left(-\left\vert  Y\right\vert^{2}/2\right)$. Let 
 \index{P@$P$}%
 $P$ denote the 
orthogonal projection operator on $H$. Of course $P$  acts on 
$C^{ \infty }\left(G,\Lambda\ac\left( \mathfrak 
g^{*}\right)\otimes L_{2}\left( \mathfrak  g\right) \right) $. We 
identify $C^{\infty }\left(G,\R\right)$ with a vector subspace of 
$C^{ \infty }\left(G\times \mathfrak g,\Lambda\ac\left(\mathfrak 
g^{*}\right)\right)$ via the embedding $s\to 
s\exp\left(-\left\vert  Y\right\vert^{2}/2\right)/\pi^{\left(m+n\right)/4}$.

Then 
we have the result in \cite[Proposition 2.10.1]{Bismut08b}.
\begin{prop}\label{Pcomp1}
The following identity holds:
\begin{equation}
    P \left( \widehat{D}^{ \mathfrak g}+ic\left(\left[Y^{ \mathfrak 
    k},Y^{ \mathfrak p}\right]\right) \right) 
  P=0.
    \label{eq:comp1}
\end{equation}
\end{prop}
\subsection{A formula for $\mathfrak D^{2}_{b}$}%
\label{subsec:formDb2}
Now, we denote by 
\index{Dpk@$\Delta^{ \mathfrak p \oplus \mathfrak k}$}%
$\Delta^{ \mathfrak p \oplus \mathfrak k}$ the standard Euclidean 
Laplacian on the Euclidean vector space $ \mathfrak g = 
\mathfrak p \oplus \mathfrak k$.

We make the same assumptions on the basis $e_{1},\ldots,e_{m+n}$ of 
$\mathfrak g$ as in (\ref{eq:toplitz0}). If $V\in \mathfrak k$, $\ad\left(V\right)\vert_{ 
\mathfrak p}$ acts as an antisymmetric endomorphism of $ \mathfrak 
p$, so that by (\ref{eq:ors1}),
\begin{equation}
    c\left(\ad\left(V\right)\vert_{ \mathfrak 
    p}\right)=\frac{1}{4}\sum_{1\le i,j\le m}^{}\left\langle  
    \left[V,e_{i}\right],e_{j}\right\rangle 
    c\left(e_{i}\right)c\left(e_{j}\right).
    \label{eq:rus1a}
\end{equation}
Also, if $W\in \mathfrak p$, $\ad\left(W\right)$ exchanges 
$ \mathfrak k$ and $\mathfrak p$ and is antisymmetric with respect to 
$B$, i.e., it is symmetric with respect to  the scalar product on 
$\mathfrak g$. Moreover, by (\ref{eq:ors1}), 
\begin{equation}
    c\left(\ad\left(W\right)\right)=\frac{1}{2}\sum_{\substack{m+1\le i\le m+n
    \\ 1\le j\le m}}^{}\left\langle  \left[W,e_{i}^{*}\right],e_{j}\right\rangle
    c\left(e_{i}\right)c\left(e_{j}\right).
    \label{eq:rus1b}
\end{equation}

If $a\in \mathfrak g$, let
\index{nVa@$\n_{a}^{V}$}%
$\n_{a}^{V}$ be the corresponding 
differentiation operator along $\mathfrak g$.
In particular $\n^{V}_{\left[Y^{ \mathfrak 
     k},Y^{ \mathfrak p}\right]}$ denotes the differentiation 
     operator  in the direction $\left[Y^{ 
     \mathfrak k},Y^{ \mathfrak p}\right]\in \mathfrak p$.
If $Y\in \mathfrak g$, we denote by $\underline{Y}^{ \mathfrak p}+i\underline{Y} ^{ \mathfrak 
     k}$ the section  of $U\left( \mathfrak g\right) \otimes _{\R}\C$ 
     associated to $Y^{ \mathfrak p}+iY^{ \mathfrak k}\in  \mathfrak 
     g \otimes _{\R}\C$. Recall that 
     $N^{\Lambda\ac\left( \mathfrak g^{*}\right)}$ is the number 
     operator of  $\Lambda\ac\left( \mathfrak g^{*}\right)$.
     
     The following identity was established in \cite[Theorem 
     2.11.1]{Bismut08b}.
\begin{thm}\label{Tcarop}
The following identity holds:
\begin{multline}
    \frac{\mathfrak D^{2}_{b}}{2}=\frac{\widehat{D}^{ \mathfrak 
	 g,2}}{2}+\frac{1}{2}\left\vert  \left[Y^{ \mathfrak k},Y^{ \mathfrak 
     p}\right]\right\vert^{2}
     +\frac{1}{2b^{2}} \left( -\Delta^{ \mathfrak 
	 p \oplus \mathfrak k}+\left\vert  Y\right\vert^{2}-m-n\right)+\frac{N^{\Lambda\ac\left( 
	 \mathfrak g^{*}\right)}}{b^{2}} \\
     +\frac{1}{b}\Biggl(\underline{Y}^{ \mathfrak p}+i\underline{Y} ^{ \mathfrak 
     k}-i\n^{V}_{\left[Y^{ \mathfrak k},
     Y^{ \mathfrak p}\right]}+
     \widehat{c}\left(\ad\left(Y^{ \mathfrak p}+iY^{ 
		    \mathfrak k}\right)
     \right)\\
     +2ic\left(\ad\left(Y^{ \mathfrak k}\right)\vert _{ 
     \mathfrak p} \right)
     -c\left(\ad\left(Y^{ \mathfrak 
     p}\right)\right)\Biggr).
      \label{eq:rio2a}
\end{multline}
\end{thm}
\subsection{The operator $\mathcal{L}^{X}_{b}$}%
\label{subsec:actqu}
To make our notation simpler,  from now on, and in the whole paper, if $V$ is a real vector space and if $W$ is a complex 
vector space, we will use the notation $V \otimes W$ to denote the 
tensor product $V \otimes _{\R}W$. 

Observe that $K$  acts on $$C^{ \infty }\left(G,\Lambda\ac \left( \mathfrak 
g^{*}\right) \otimes S\ac\left(\mathfrak g^{*}\right) \otimes 
S^{\overline{ \mathfrak p}} \otimes E\right) $$ by a formula 
similar to (\ref{eq:trubl1}).

Let
\index{X@$\widehat{\mathcal{X}}$}%
$\widehat{\mathcal{X}}$ be the total space of $TX \oplus N$ over 
$X$, and let 
\index{pi@$\widehat{\pi}$}%
$\widehat{\pi}:\widehat{\mathcal{X}}\to X$ be the corresponding  
projection.  By \cite[eq. (2.12.7)]{Bismut08b}, since $TX \oplus N$ 
can be identified with the trivial vector bundle $\mathfrak g$, then
\begin{equation}
    \widehat{\mathcal{X}}=X\times \mathfrak g,
    \label{eq:prod1}
\end{equation}
and $\widehat{\pi}:\widehat{\mathcal{X}}\to X$ is the  projection $X\times 
\mathfrak g\to X$. Let 
\index{YTX@$Y^{TX}$}
\index{YN@$Y^{N}$}%
$Y=Y^{TX}+Y^{N}, Y^{TX}\in TX,Y^{N}\in N$ be 
the canonical section  of $\widehat{\pi}^{*}\left(TX \oplus N\right)$ over 
$\widehat{\mathcal{X}}$. Then
\begin{equation}
    \left\vert  Y\right\vert^{2}=\left\vert  Y^{TX}\right\vert^{2}
    +\left\vert  Y^{N}\right\vert^{2}.
    \label{eq:sumex18ax1}
\end{equation}
\begin{defin}\label{Dspace}
    Let 
    \index{nC@$\n^{C^{ \infty }\left(TX \oplus 
     N,\widehat{\pi}^{*}\left(\Lambda\ac\left(T^{*}X \oplus 
     N^{*}\right)\otimes  S^{\overline{TX}} \otimes 
     F\right) \right) }$}%
    $\n^{C^{ \infty }\left(TX \oplus 
     N,\widehat{\pi}^{*}\left(\Lambda\ac\left(T^{*}X \oplus 
     N^{*}\right)\otimes S^{\overline{TX}} \otimes 
     F\right) \right) }$ be the connection on the vector bundle  $C^{ \infty }\left(TX \oplus 
     N,\widehat{\pi}^{*}\left(\Lambda\ac\left(T^{*}X \oplus 
     N^{*}\right)\otimes  S^{\overline{TX}} \otimes 
     F\right) \right) $ on $X$ that is induced by the connection 
      $\theta^{\mathfrak k}$ on the $K$-bundle $p:G\to X=G/K$.

      Let 
\index{H@$\mathcal{H}$}%
   $\mathcal{H}$ be the vector space 
of smooth sections 
over $X$ of the vector bundle $C^{ \infty }\left(TX \oplus 
 N, \widehat{\pi}^{*}\left( \Lambda\ac\left(T^{*}X \oplus 
 N^{*}\right)\otimes S^{\overline{TX}} \otimes  F \right) \right) $.  Then 
 \begin{equation}\label{eq:Lie20}
\mathcal{H}=C^{ \infty 
}\left(\widehat{\mathcal{X}},\widehat{\pi}^{*}\left(\Lambda\ac\left(\TsX \oplus N^{*}\right)
\otimes S^{\overline{TX}} \otimes F\right)\right).
\end{equation}
\end{defin}

Since 
$\widehat{D}^{\mathfrak g}, \mathfrak 
D_{b}$ commute with the action of  $K$, they descend to  operators
\index{DgX@$\widehat{D}^{\mathfrak g,X}$}%
\index{Dgb@$\mathfrak D^{X}_{b}$}%
$\widehat{D}^{\mathfrak g,X},\mathfrak D^{X}_{b}$ acting on 
$\mathcal{H}$. The operators $\mathcal{D}^{\mathfrak 
p},\mathcal{E}^{\mathfrak p},\mathcal{D}^{\mathfrak 
k},\mathcal{E}^{\mathfrak k}$ descend to operators 
\index{DTX@$\mathcal{D}^{TX}$}%
\index{ETX@$\mathcal{E}^{TX}$}%
\index{DN@$\mathcal{D}^{N}$}%
\index{EN@$\mathcal{E}^{N}$}%
$\mathcal{D}^{TX},\mathcal{E}^{TX},\mathcal{D}^{N},\mathcal{E}^{N}$ 
acting along the fibres of $TX \oplus N$. Also the operator 
$C^{\mathfrak g}$ descends to an operator $C^{\mathfrak g,X}$. 
By 
(\ref{eq:kos22}), we get
\begin{equation}\label{eq:dep0}
\widehat{D}^{\mathfrak g,X,2}=-C^{\mathfrak g,X}-\frac{1}{4}B^{*}\left(
\kappa^{\mathfrak g},\kappa^{\mathfrak g}\right).
\end{equation}

Put
\index{LbX@$\mathcal L_{b}^{X}$}%
\begin{equation}
 \mathcal L_{b}^{X}=-\frac{1}{2}\widehat{D}^{ \mathfrak 
    g,X,2}+\frac{1}{2}\mathfrak D ^{X,2}_{b}.
    \label{eq:tra1}
\end{equation}

Let 
\index{DTXN@$\Delta^{TX \oplus N}$}%
$\Delta^{TX \oplus N}$ be the Laplacian acting along the fibres 
of the Euclidean vector bundle $TX \oplus N$. 

We will give an important formula that was established in \cite[Theorem 
2.12.5 and eq. (2.13.5)]{Bismut08b}. Note that
$\widehat{c}\left(\ad\left(Y^{N}\right)\vert_{\overline{TX}}\right)$ 
is a section of the Clifford algebra $c\left(\overline{TX}\right)$ 
and acts on $S^{\overline{TX}}$, while
$c\left(\ad\left(Y^{TX}\right)\right),c\left(\theta\ad\left(Y^{N}\right)\right),
\widehat{c}\left(\ad\left(Y^{TX}\right)\right)$ 
 lie in $c\left(TX \oplus N\right), \widehat{c}\left(TX \oplus 
N\right)$, and  act on $\Lambda\ac\left(T^{*}X \oplus 
N^{*}\right)$.
\begin{thm}\label{Timpfo}
The following identities hold:
\begin{align}\label{eq:tra2}
    &\mathfrak D_{b}^{X}=\widehat{D}^{ \mathfrak g,X}+ic\left(\left[Y^{ 
N},Y^{TX}\right]\right)+\frac{1}{b}\left(\mathcal{D}^{ TX}+\mathcal{E}^{ TX}-i\mathcal{D}^{ 
N}+i\mathcal{E}^{N}\right), \nonumber \\
&\mathcal L_{b}^{X}=\frac{1}{2}\left\vert  \left[Y^{N},Y^{ TX}\right]\right\vert^{2}
     +\frac{1}{2b^{2}} \left( -\Delta^{ TX \oplus N
	 }+\left\vert  Y\right\vert^{2}-m-n\right)
	 +\frac{N^{\Lambda\ac\left( 
	 T^{*}X \oplus N^{*}\right)}}{b^{2}}\\
     &+\frac{1}{b}\Biggl(\n_{Y^{ TX}}^{C^{ \infty }\left(TX \oplus 
     N,\widehat{\pi}^{*}\left(\Lambda\ac\left(T^{*}X \oplus N^{*}\right)\otimes 
      S^{\overline{TX}} \otimes  F\right) \right) } +
     \widehat{c}\left(\ad\left(Y^{TX}\right)
     \right) \nonumber 
     \\
    &- c\left(\ad\left(Y^{TX} \right) +i\theta\ad\left(Y^{N}\right)\right)  
   -i\widehat{c}\left(\ad\left(Y^{N}\right)\vert_{\overline{TX}}\right) -i\rho^{F}\left(Y^{ N}\right)\Biggr). \nonumber 
\end{align}
\end{thm}

In \cite[Theorem 2.13.2]{Bismut08b}, it was shown that 
$\frac{\pa}{\pa t}+\mathcal{L}^{X}_{b}$ is hypoelliptic, as a 
consequence of a result of H\"{o}rmander \cite{Hormander67}. According to 
the terminology of \cite{Bismut08b}, $\mathcal{L}^{X}_{b}$ is called 
a hypoelliptic Laplacian.

\subsection{A formula relating $\mathcal{L}_{b}^{X}$ to 
$\mathcal{L}^{X}_{0}$}%
\label{subsec:formre}
By (\ref{eq:tra2}), $\mathcal{L}^{X}_{b}$ can be written in the form
\begin{equation}
    \mathcal{L}^{X}_{b}=\frac{\alpha}{b^{2}}+\frac{\beta}{b}+\gamma.
    \label{eq:bugr1}
\end{equation}
We  still denote by 
\index{H@$H$}%
$H$ the kernel of $\alpha$. Then $H$ is the module over $C^{\infty 
}\left(X,\C\right)$ given by 
\begin{equation}
    H=\left\{\exp\left(-\left\vert  Y\right\vert^{2}/2\right)\right\} 
    \otimes S^{\overline{TX}} \otimes F.
    \label{eq:bugr2}
\end{equation}
Let 
\index{Hp@$H^{\perp}$}%
$H^{\perp}$ be the orthogonal  space to $H$ in $L_{2}\left(\widehat{\mathcal{X}},
\widehat{\pi}^{*}\left(\Lambda\ac\left(T^{*}X \oplus N^{*}\right)
\otimes  S^{\overline{TX}} \otimes  F\right)\right)$. We still denote by 
\index{P@$P$}%
$P$  
the orthogonal projection on $H$. Let
 \index{P@$P^{\perp}$}%
$P^{\perp}$ be  the orthogonal projection from $
L_{2}\left(\widehat{\mathcal{X}},\widehat{\pi}^{*}\left(\Lambda\ac\left(T^{*}X 
\oplus N^{*}\right)\otimes  S^{\overline{TX}}
\otimes F\right)\right)$ on $H^{\perp}$. We embed 
$L_{2}\left(X,S^{\overline{TX}} \otimes F\right)$ 
into $L_{2}\left(\widehat{\mathcal{X}},\widehat{\pi}^{*}\left(\Lambda\ac\left(T^{*}X \oplus N^{*}\right)
\otimes  S^{\overline{TX}} \otimes  F\right) \right) $ via the isometric embedding $s \to 
\widehat{\pi}^{*}s\exp\left(-\left\vert  
Y\right\vert^{2}/2\right)/\pi^{\left(m+n\right)/4}$. 

Note that $\beta$ maps $H$ into 
$H^{\perp}$. Let 
\index{a-1@$\alpha^{-1}$}%
$\alpha^{-1}$ be the inverse of $\alpha$ 
restricted to $H^{\perp}$.

The following result was established in \cite[Theorem 
2.16.1]{Bismut08b}.
\begin{thm}\label{Tfundid}
The following identity holds:
\begin{equation}
    P\left(\gamma-\beta\alpha^{-1}\beta\right)P=\mathcal{L}^{X}_{0}.
    \label{eq:bugr3}
\end{equation}
\end{thm}
\begin{remk}\label{Rodd}
    In \cite{Bismut08b}, two proofs were given of (\ref{eq:bugr3}). A 
    first proof relies on Proposition \ref{Pcomp1}. A second proof is 
    based on explicit computations. In subsections 
    \ref{subsec:formrebis} and \ref{subsec:anopr}, both 
    arguments will be extended to the deformation 
    $\mathcal{L}^{X}_{\bt}$ of $\mathcal{L}^{X}_{b}$.
    
In the above constructions, only the even part of 
$c\left(\overline{\mathfrak p}\right)$ or of $c\left(\overline{TX}\right)$ is involved. If $m$ is even, our 
operators preserve the splitting $S^{\overline{TX}}=S^{\overline{TX}}_{+} \oplus 
S^{\overline{TX}}_{-}$. There is no way that by the above method, odd elements 
in $c\left(\overline{\mathfrak p}\right)$ or 
$c\left(\overline{TX}\right)$ would appear. Incidentally,  equations 
(\ref{eq:rio1}),
(\ref{eq:comp1}), and (\ref{eq:tra2}) suggest that $\mathfrak D^{X}_{b}\vert_{b>0}$ is a deformation of the  operator $0$.
\end{remk}

%% file: Eta3.tex
\section{The hypoelliptic operators  $\mathcal{L}^{X}_{b,\vartheta}$}%
\label{sec:defdx}
In this section, we introduce a family of hypoelliptic operators 
$\mathcal{L}^{X}_{b,\vartheta}\vert_{\left( b,\vartheta\right)\in 
\R_{+}^{*}\times \left[0,\frac{\pi}{2}\right[}$  acting on 
$C^{\infty }\left(\widehat{\mathcal{X}}, 
\pi^{*}\left(\Lambda\ac\left(T^{*}X \oplus N^{*}\right)\right)\otimes 
S^{\overline{TX}} \otimes F\right)$, that coincides with the family
$\mathcal{L}^{X}_{b}\vert_{b>0} $ for $\vartheta=0$. The construction of this new family is done 
through a new family of operators $\mathfrak D^{X}_{b,\vartheta}\vert_{\left(b,\vartheta\right)\in 
\R_{+}^{*}\times \left[0,\frac{\pi}{2}\right[}$, that coincides with 
the family $\mathfrak D^{X}_{b}\vert_{b>0}$ for $\vartheta=0$.   
While in Remark \ref{Rodd}, one could argue that the family $\mathfrak 
D_{b}^{X}\vert_{b>0}$ 
deforms the  operator $0$ as $b\to 0$, here, given $\vartheta\in 
\left[0,\frac{\pi}{2}\right[$, $\mathfrak 
D_{b,\vartheta}^{X}\vert_{b>0}$ 
deforms the classical Dirac operator 
$\sin\left(\vartheta\right)\widehat{D}^{X}$ as $b\to 0$. It is in this way  
that the Dirac operator $\widehat{D}^{X}$ enters the picture, and not 
only through its square as in section \ref{sec:eta}.  

In this section, superconnections associated with the 
family  $\mathfrak D^{X}_{b,\vartheta}\vert_{\left( b,\vartheta\right)\in 
\R_{+}^{*}\times \left[0,\frac{\pi}{2}\right[}$ do appear. Also we 
introduce the conjugate families $\overline{\mathfrak D}^{X}_{\bt}, 
\mathfrak D^{X \prime }_{\bt}$, the first family  being more suitable when 
studying the limit $b\to 0$, and the second will be shown later to be 
more convenient when considering the limit $b\to + \infty $.

The organization 
of the section is  closely related to the organization of 
section \ref{sec:eta}. In subsection \ref{subsec:dedx}, we introduce 
the family of operators 
$\widehat{D}^{X}_{\vartheta}=\sin\left(\vartheta\right)\widehat{D}^{X}, 
\vartheta\in \left[0,\frac{\pi}{2}\right[$,  we construct a
corresponding superconnection $A^{X}$, and a version $T^{X}$ of its 
curvature.

In subsection \ref{subsec:grpso}, if $\overline{\mathfrak p}$ is 
another copy of $\mathfrak p$, we make $\mathrm{SO}\left(2\right)$ 
act on $\mathfrak p \oplus \overline{\mathfrak p}$.

In subsection \ref{subsec:defdb}, using the above action of 
$\mathrm{SO}\left(2\right)$, we construct a deformation $\mathfrak 
D_{b,\vartheta}$ of $\mathfrak D_{b}$, and also the conjugate 
operators 
$\overline{\mathfrak D}_{\bt}, \mathfrak D'_{b,\vartheta}$.

In subsection \ref{subsec:compdb}, we give a formula for a 
compression of $\overline{\mathfrak D}_{b,\vartheta}$.

In subsection \ref{subsec:dbct}, we give  formulas for $\mathfrak 
D_{b,\vartheta}^{2}, \mathfrak D^{\prime 2}_{b,\vartheta}$.

In subsection \ref{subsec:scbs}, we introduce superconnections 
$B,\overline{B},B'$ 
on $\R_{+}^{*}\times \left[0,\frac{\pi}{2}\right[$ that are 
associated with the above families,  and we prove a compression identity.

In subsection \ref{subsec:des}, we descend the above  
operators to 
$X$. We obtain this way new operators $\mathfrak D^{X}_{b,\vartheta}, 
\overline{\mathfrak D}^{X}_{\bt},
\mathfrak D^{X \prime }_{b,\vartheta}$ and new  hypoelliptic 
Laplacians 
$\mathcal{L}^{X}_{b,\vartheta},\overline{\mathcal{L}}^{X}_{\bt},\mathcal{L}^{X \prime}_{b,\vartheta}$ 
that act on the same space as $\mathfrak D^{X}_{b},\mathcal{L}^{X}_{b}$ and coincide with 
$\mathfrak D^{X}_{b},\mathcal{L}^{X}_{b}$ for $\vartheta=0$. 

In subsection \ref{subsec:formrebis}, we give a formula relating 
$\overline{\mathcal{L}}^{X}_{b,\vartheta}$ to $\mathcal{L}^{X}_{0,\vartheta}$, 
that ultimately explains why $\mathcal{L}^{X}_{b,\vartheta}$ deforms 
$\mathcal{L}^{X}_{0,\vartheta}$. As in \cite[Theorem 2.16.1]{Bismut08b}, the proof uses 
the compression identity of subsection \ref{subsec:compdb}. 

In subsection \ref{subsec:newsup}, we descend the superconnections 
$B,\overline{B}, B'$ to superconnections $B^{X},\overline{B}^{X},B^{X \prime }$ over $\R^{*}_{+}\times 
\left[0,\frac{\pi}{2}\right[$, and we introduce their proper curvatures 
$L^{X}, \overline{L}^{X},L^{X \prime }$. We establish a corresponding 
compression identity, and we give a formula relating $\overline{L}^{X}$ to 
$T^{X}$, that extend in higher degree what we did in subsection \ref{subsec:formrebis}.

In subsection \ref{subsec:ide}, we establish a quadratic compression 
identity on linear maps.

In subsection \ref{subsec:anopr}, as in \cite[section 
2.16]{Bismut08b}, using the compression identity of subsection 
\ref{subsec:ide}, we give a direct 
computational proof of the formulas of subsection 
\ref{subsec:formrebis} and \ref{subsec:newsup}. One reason for giving this direct proof is 
that as in \cite{Bismut08b}, in section \ref{sec:fin}, and more 
specifically in subsection \ref{subsec:limu},  we will use  some of the intermediate 
identities established in this direct proof to 
study the behaviour of our hypoelliptic orbital integrals as $b\to 
0$. 

Finally, in subsection \ref{subsec:scal}, when the bilinear form $B$ 
on $\mathfrak g$ is replaced by $B/t$, we give a formula expressing 
the above operators  associated with $B/t$ in terms of the original 
operators.

In this section, we make the same assumptions as in section 
\ref{sec:eta}, and we use the corresponding notation. In particular 
the connected reductive Lie group 
$G$ is assumed to be simply connected.
\subsection{A deformation of $\widehat{D}^{X}$}%
\label{subsec:dedx}
Here, as in \cite[section 2]{BismutFreed86b}, we will adopt the formalism of subsection \ref{subsec:elop} 
while omitting the Clifford variable $\sigma$ for simplicity. We will 
instead take into account the grading in the Clifford algebras 
$\widehat{c}\left(\overline{\mathfrak p}\right)$. In the sequel, we 
form the $\Z_{2}$-graded tensor 
product 
$\Lambda\ac\left(\R^{*}\right)\ho\widehat{c}\left(\overline{\mathfrak 
p}\right) \otimes U\left(\mathfrak g\right)$. In particular 
$c^{\odd}\left(\overline{\mathfrak p}\right)$ anticommutes with 
$d\vartheta\in \Lambda^{1}\left(\R^{*}\right)$. 

Recall that the operator $\widehat{D}$ was defined in 
(\ref{eq:Lie14}).   For $\vartheta\in \left[0,\frac{\pi}{2}\right[$, set
\index{Dt@$\widehat{D}_{\vartheta}$}%
\begin{equation}\label{eq:defa1}
\widehat{D}_{\vartheta}=\sin\left(\vartheta\right)\widehat{D}.
\end{equation}
Then  
$\widehat{D}_{\vartheta}\in \widehat{c}^{\odd}\left(\overline{\mathfrak p}\right) 
\otimes U\left(\mathfrak g\right)$.

Let 
\index{A@$A$}%
$A$ be the superconnection  over 
$\left[0,\frac{\pi}{2}\right[$,
\begin{equation}\label{eq:defa2x}
A=d\vartheta\frac{\pa}{\pa 
\vartheta}+\frac{\widehat{D}_{\vartheta}}{\sqrt{2}}.
\end{equation}
The curvature $A^{2}$ of $A$ is given by
\begin{equation}\label{eq:defa3x}
A^{2}=\frac{1}{2}\sin^{2}\left(\vartheta\right)\widehat{D}^{2}+\frac{d\vartheta}{\sqrt{2}}\cos\left(\vartheta\right)\widehat{D}.
\end{equation}

Also $\widehat{D}_{\vartheta}$ descends to an 
operator 
\index{DXt@$\widehat{D}^{X}_{\vartheta}$}%
$\widehat{D}^{X}_{\vartheta}$ acting on $C^{ \infty 
}\left(X,S^{\overline{TX}} \otimes F\right)$ given by
\begin{equation}\label{eq:defa1x0}
\widehat{D}^{X}_{\vartheta}=\sin\left(\vartheta\right)\widehat{D}^{X}.
\end{equation}
The superconnection  $A$ descends to a superconnection  
\index{AX@$A^{X}$}%
$A^{X}$ on the trivial vector bundle $C^{\infty }\left(X,S^{\overline{TX}} \otimes F\right)$
on $\left[0,\frac{\pi}{2}\right[$ given by
\begin{equation}\label{eq:defa2}
A^{X}=d\vartheta\frac{\pa}{\pa \vartheta}+\frac{\widehat{D}^{X}_{\vartheta}}{\sqrt{2}}.
\end{equation}
The curvature $A^{X,2}$ of $A^{X}$ is given by
\begin{equation}\label{eq:defa3}
A^{X,2}=\frac{1}{2}\sin^{2}\left(\vartheta\right)\widehat{D}^{X,2}+\frac{d\vartheta}{
\sqrt{2}}\cos\left(\vartheta\right)\widehat{D}^{X}.
\end{equation}

Recall that the operator $\mathcal{L}_{0}^{X}$ was defined in 
(\ref{eq:Lie17}). 
\begin{defin}\label{DLXt}
Set
\index{LXt@$\mathcal{L}^{X}_{0,\vartheta}$}%
\begin{equation}\label{eq:gzinc1}
\mathcal{L}^{X}_{0,\vartheta}=\frac{1}{2}\sin^{2}\left(\vartheta\right)\widehat{D}^{X,2}+\mathcal{L}_{0}^{X}.
\end{equation}
\end{defin}

Let $e_{m+1},\ldots,e_{m+n}$ be an orthonormal basis of $N$.
\begin{prop}\label{PLXt}
The following identities hold:
\begin{align}\label{eq:japo4}
    &\mathcal{L}^{X}_{0,\vartheta}=-\frac{1}{2}\cos^{2}\left(\vartheta\right)
    \widehat{D}^{X,2}+\frac{1}{48}\Tr^{\mathfrak k}\left[C^{\mathfrak 
    k, \mathfrak k}\right]+\frac{1}{2}C^{\mathfrak k,F}, \nonumber \\
    &\mathcal{L}^{X}_{0,\vartheta}=\cos^{2}\left(\vartheta\right)\mathcal{L}_{0}^{X}+
\sin^{2}\left(\vartheta\right)\left(\frac{1}{8}
B^{*}\left(\kappa^{\mathfrak k},\kappa^{\mathfrak 
k}\right)+\frac{1}{2}C^{\mathfrak k,F}\right),\\
&\mathcal{L}^{X}_{0,\vartheta}=\frac{1}{2}\cos^{2}\left(\vartheta\right)
\left( -\Delta^{X,H}
+\frac{1}{4}\Tr^{\mathfrak 
p}\left[C^{\mathfrak k,\mathfrak p}\right] +
2\sum_{i=m+1}^{m+n}
\widehat{c}\left(\ad\left(e_{i}\right)\vert_{\overline{TX}}\right)
\rho^{F}\left(e_{i}\right) \right)  \nonumber \\
&\qquad \qquad +\frac{1}{48}\Tr^{\mathfrak 
k}\left[C^{\mathfrak k,\mathfrak 
k}\right]+\frac{1}{2}
C^{\mathfrak k,F}. \nonumber 
\end{align}
\end{prop}
\begin{proof}
By (\ref{eq:Lie17x1}), (\ref{eq:gzinc1}), we get
\begin{multline}\label{eq:gzinc2}
\mathcal{L}^{X}_{0,\vartheta}=-\frac{1}{2}\cos^{2}\left(\vartheta\right)\widehat{D}^{X,2}+\frac{1}{8}
B^{*}\left(\kappa^{\mathfrak k},\kappa^{\mathfrak 
k}\right)+\frac{1}{2}C^{\mathfrak k,F}\\
=\cos^{2}\left(\vartheta\right)\mathcal{L}_{0}^{X}+
\sin^{2}\left(\vartheta\right)\left(\frac{1}{8}
B^{*}\left(\kappa^{\mathfrak k},\kappa^{\mathfrak 
k}\right)+\frac{1}{2}C^{\mathfrak k,F}\right).
\end{multline}
By  (\ref{eq:qsic-10}), (\ref{eq:qsic-9a1}), and (\ref{eq:gzinc2}), we get (\ref{eq:japo4}). 
\end{proof}
\begin{defin}\label{DTX}
Set
\index{TX@$T^{X}$}%
\begin{equation}\label{eq:defa4}
T^{X}=A^{X,2}+\mathcal{L}_{0}^{X}.
\end{equation}
By  (\ref{eq:defa3}), (\ref{eq:gzinc1}), and (\ref{eq:defa4}), we get
\begin{equation}\label{eq:defa5x1b}
T^{X}=\mathcal{L}^{X}_{0,\vartheta}+\frac{d\vartheta}{
\sqrt{2}}\cos\left(\vartheta\right)\widehat{D}^{X}.
\end{equation}
\end{defin}

Now we establish a version of Bianchi's identity, similar to 
(\ref{eq:cla6}).
\begin{prop}\label{PBi}
The following identity holds:
\begin{equation}\label{eq:bia1}
\left[A^{X},T^{X}\right]=0.
\end{equation}
\end{prop}
\begin{proof}
As in (\ref{eq:bob2x1}), we have the classical Bianchi identity,
\begin{equation}\label{eq:bia2}
\left[A^{X},A^{X,2}\right]=0.
\end{equation}
Moreover,  $C^{\mathfrak g}$ lies in the centre of 
$U\left(\mathfrak g\right)$. By (\ref{eq:Lie17}),  we  get an analogue of (\ref{eq:cla5}),
\begin{equation}\label{eq:bia3}
\left[A^{X},\mathcal{L}_{0}^{X}\right]=0.
\end{equation}
By (\ref{eq:defa4}), (\ref{eq:bia2}), and (\ref{eq:bia3}), we get 
(\ref{eq:bia1}).
\end{proof}

Let $e_{1},\ldots,e_{m}$ be an orthonormal basis of $TX$. At each 
$x\in X$, we identify this frame to the orthonormal frame in $TX$ 
defined on a neighbourhood of $x$ obtained by parallel 
transport along geodesics centred at $x$ with respect to the 
Levi-Civita connection. We have the identity
\begin{equation}\label{eq:supp1}
\Delta^{X,H}=\sum_{i=1}^{m}\n^{S^{TX} \otimes F,2}_{e_{i}}.
\end{equation}

We
have a version of an  identity established in 
\cite[Proposition 2.1]{BismutFreed86b}.
\begin{prop}\label{PBide}
The following identity holds:
\begin{multline}\label{eq:defa5x1c}
T^{X}=-\frac{1}{2}\sum_{i=1}^{m}\left(
\cos\left(\vartheta\right)\n^{S^{\overline{TX}} \otimes F}_{e_{i}}-\frac{d\vartheta}{
\sqrt{2}}\widehat{c}\left(\overline{e}_{i}\right)\right)^{2}\\
+\frac{1}{2}\cos^{2}\left(\vartheta\right)
\left( 
\frac{1}{4}\Tr^{\mathfrak 
p}\left[C^{\mathfrak k,\mathfrak p}\right] +
2\sum_{i=m+1}^{m+n}
\widehat{c}\left(\ad\left(e_{i}\right)\vert_{\overline{TX}}\right)
\rho^{F}\left(e_{i}\right) \right)  \\
\qquad \qquad +\frac{1}{48}\Tr^{\mathfrak 
k}\left[C^{\mathfrak k,\mathfrak 
k}\right]+\frac{1}{2}
C^{\mathfrak k,F}. 
\end{multline}
\end{prop}
\begin{proof}
This follows from the third identity in (\ref{eq:japo4}) and from 
(\ref{eq:defa5x1b}).
\end{proof}
\subsection{The action of  $\mathrm{SO}\left(2\right)$ on 
$\mathfrak p \oplus \overline{\mathfrak p}$}%
\label{subsec:grpso}
Let 
\index{J@$J$}%
$J\in\mathrm{so}\left(2,\R\right)$ be 
given by
\begin{equation}\label{eq:co2x1}
J=
\begin{bmatrix}
    0 & -1 \\
    1 & 0
\end{bmatrix}.
\end{equation}
For $\vartheta\in \R$, set
\index{Rt@$R_{\vartheta}$}%
\begin{equation}\label{eq:co2x2}
R_{\vartheta}=\exp\left(\vartheta J\right).
\end{equation}
Then $R_{\vartheta}$ is the rotation of angle $\vartheta$, i.e.,
\begin{equation}\label{eq:co3}
R_{\vartheta}=
\begin{bmatrix}
     \cos\left(\vartheta\right)&-\sin\left(\vartheta\right) \\
   \sin\left(\vartheta\right) & \cos\left(\vartheta\right)
\end{bmatrix}.
\end{equation}
We equip $\mathfrak p \oplus \overline{\mathfrak p}$ with the direct sum of the 
obvious scalar products. Then $J$ acts as an antisymmetric 
endomorphism of $\mathfrak p \oplus \overline{\mathfrak p}$, and $R_{\vartheta}$ acts as an 
isometry of $\mathfrak p \oplus \overline{\mathfrak p}$.

Note that
\begin{equation}\label{eq:co3x1}
\mathfrak g \oplus\overline{ \mathfrak p}= \mathfrak p \oplus \mathfrak  k \oplus \overline{\mathfrak 
p}.
\end{equation}
Let $\beta$ be the bilinear symmetric form on $\mathfrak 
g \oplus \overline{\mathfrak p}$ which coincides with $B$ on $\mathfrak g$ and 
with the scalar product  on $\overline{\mathfrak p}$, and is such 
that the splitting in (\ref{eq:co3x1}) is  orthogonal. We 
extend the action of $J$ to $\mathfrak g \oplus 
\overline{\mathfrak p}$, by making $J$ act like $0$ on $\mathfrak k$. 
Then the action of $R_{\vartheta}$ extends to $\mathfrak g \oplus 
\overline{\mathfrak p}$. Also $R_{\vartheta}$ preserves $\beta$.

Let 
$\widehat{c}\left(\mathfrak g \oplus \overline{\mathfrak p}\right)$ be the Clifford 
algebra of $\left(\mathfrak g \oplus \overline{\mathfrak p},-\beta\right)$. Then
\begin{equation}\label{eq:co2}
\widehat{c}\left(\mathfrak g \oplus \overline{\mathfrak 
p}\right)=\widehat{c}\left(\mathfrak g\right) 
\ho\widehat{c}\left(\overline{\mathfrak p}\right).
\end{equation}
If $f\in \mathfrak g \oplus \overline{\mathfrak p}$, let  
$\widehat{c}\left(f\right)$ denote the corresponding element in 
$\widehat{c}\left(\mathfrak g \oplus \overline{\mathfrak p}\right)$. 
By (\ref{eq:co2}),  $\widehat{c}\left(\mathfrak g \oplus \overline{\mathfrak p}\right)$ acts 
naturally on $\Lambda\ac\left(\mathfrak g^{*}\right)\otimes S^{\overline{\mathfrak p}}$. 

Let $\widehat{c}\left(J\right)\in \widehat{c}\left(\mathfrak g \oplus 
\overline{\mathfrak p}\right)$ correspond to $J$. 
If $e_{1},\ldots,e_{m}$ is an orthonormal basis of $\mathfrak p$, 
then
\begin{equation}\label{eq:co4}
\widehat{c}\left(J\right)=-\frac{1}{2}\sum_{i=1}^{m}\widehat{c}\left(e_{i}\right)\widehat{c}\left(\overline{e}_{i}\right).
\end{equation}
For $\vartheta\in\R$, set
\index{Rt@$\widehat{R}_{\vartheta}$}%
\begin{equation}\label{eq:co5}
\widehat{R}_{\vartheta}=\exp\left(\vartheta\widehat{c}\left(J\right)\right).
\end{equation}
Then
\begin{equation}\label{eq:co6}
\widehat{R}_{\vartheta}=\prod_{i=1}^{m}\left(\cos\left(\vartheta/2\right)
-\sin\left(\vartheta/2\right)\widehat{c}\left(e_{i}\right)
\widehat{c}\left(\overline{e}_{i}\right)\right).
\end{equation}

For $f\in \mathfrak g \oplus \overline{\mathfrak p}$, set
\begin{equation}\label{eq:co6x1}
\widehat{c}_{\vartheta}\left(f\right)=\widehat{c}\left(R_{\vartheta}f\right).
\end{equation}
Then
\begin{equation}\label{eq:co7}
\widehat{c}_{\vartheta}\left(f\right)=\widehat{R}_{\vartheta}\widehat{c}\left(f\right)\widehat{R}_{\vartheta}^{-1}.
\end{equation}

If $A\in \End\left(\mathfrak g \oplus \overline{\mathfrak 
p}\right)$ is antisymmetric with respect to $\beta$, we denote by 
$\widehat{c}\left(A\right)$ the corresponding element in 
$\widehat{c}\left(\mathfrak g \oplus \overline{\mathfrak p}\right)$. 
Put
\begin{equation}\label{eq:hero1}
\widehat{c}_{\vartheta}\left(A\right)=\widehat{R}_{\vartheta}\widehat{c}\left(A\right)\widehat{R}_{\vartheta}^{-1}.
\end{equation}
Then
\begin{equation}\label{eq:hero1a}
\widehat{c}_{\vartheta}\left(A\right)=\widehat{c}\left(R_{\vartheta}AR_{\vartheta}^{-1}\right).
\end{equation}

Recall that if $e\in \mathfrak g$, 
$\widehat{c}\left(\ad\left(e\right)\right)\in 
\widehat{c}\left(\mathfrak g\right)$ is given by (\ref{eq:ors1}). If $e\in \mathfrak g $, we denote by 
$\ad\left(e\right)\vert_{ \mathfrak g}$ the endomorphism of $\mathfrak g 
\oplus \overline{\mathfrak p}$, which is $\ad\left(e\right)$ on $\mathfrak g$, 
and $0$ on   $\overline{\mathfrak p}$. Then $\ad\left(e\right)\vert_{ 
\mathfrak g}$ is antisymmetric with respect to $\beta$, and 
$\widehat{c}\left(\ad\left(e\right)\vert_{\mathfrak 
g}\right)=\widehat{c}\left(\ad\left(e\right)\right)$.

Set
\index{Nlp@$N^{\Lambda\ac\left(\mathfrak p^{*}\right) \prime}_{\vartheta}$}%
\begin{equation}\label{eq:co16x1}
N^{\Lambda\ac\left(\mathfrak p^{*}\right) \prime}_{\vartheta}
=\widehat{R}_{\vartheta}N^{\Lambda\ac\left(\mathfrak p^{*}\right)}\widehat{R}_{\vartheta}^{-1}.
\end{equation}
By (\ref{eq:ors0}), we get
\begin{equation}\label{eq:co16x2}
N^{\Lambda\ac\left(\mathfrak 
p^{*}\right)\prime}_{\vartheta}-\frac{m}{2}=\frac{1}{2}\sum_{i=1}^{m}c\left(e_{i}\right)\widehat{c}_{\vartheta}
\left(e_{i}\right).
\end{equation}
By (\ref{eq:co6x1}), (\ref{eq:co16x2}), we obtain
\begin{equation}\label{eq:co16x3}
N^{\Lambda\ac\left(\mathfrak 
p^{*}\right) \prime}_{\vartheta}=\cos\left(\vartheta\right)
N^{\Lambda\ac\left(\mathfrak p^{*}\right)}
+\frac{1}{2}\sin\left(\vartheta\right)\sum_{i=1}^{m}c\left(e_{i}\right)\widehat{c}\left(\overline{e}_{i}\right)
+\frac{m}{2}\left(1-\cos\left(\vartheta\right)\right).
\end{equation}
\subsection{The deformation $\mathfrak D_{b,\vartheta}$ of  $\mathfrak D_{b}$}%
\label{subsec:defdb}
\begin{defin}\label{Dkosd}
For $\vartheta\in \left[0,\frac{\pi}{2}\right[$, set
\index{Dgt@$\widehat{D}^{\mathfrak g}_{\vartheta}$}%
\index{Ept@$\mathcal{E}^{\mathfrak p}_{\vartheta}$}%
\begin{align}\label{eq:co8}
&\widehat{D}^{\mathfrak g}_{\vartheta}=\widehat{R}_{\vartheta}\widehat{D}^{\mathfrak g}\widehat{R}^{-1}_{\vartheta},
&\mathcal{E}^{\mathfrak 
p}_{\vartheta}=\widehat{R}_{\vartheta}\mathcal{E}^{\mathfrak 
p}\widehat{R}_{\vartheta}^{-1}.
\end{align}
\end{defin}

By (\ref{eq:brav12}), (\ref{eq:co6x1}), we get
\begin{equation}\label{eq:co9}
\mathcal{E}^{\mathfrak 
p}_{\vartheta}=\widehat{c}_{\vartheta}\left(Y^{ \mathfrak 
p}\right)=\cos\left(\vartheta\right)\widehat{c}\left(Y^{ \mathfrak 
p}\right)+\sin
\left(\vartheta\right)\widehat{c}\left(\overline{Y}^{ \mathfrak p}\right).
\end{equation}
By (\ref{eq:kos22}),  (\ref{eq:brav12}), and (\ref{eq:co8}), we obtain
\begin{align}\label{eq:co10}
	&\widehat{D}^{\mathfrak g,2}_{\vartheta}=-C^{\mathfrak 
	g}-\frac{1}{4}B^{*}\left(\kappa^{\mathfrak g},\kappa^{\mathfrak 
	g}\right),
&\mathcal{E}_{\vartheta}^{\mathfrak p,2}=\left\vert  Y^{\mathfrak 
p}\right\vert^{2}.
\end{align}
It is crucial to observe that $\widehat{D}^{\mathfrak 
g,2}_{\vartheta}$ does not depend on $\vartheta$, and lies in the 
centre of $U\left(\mathfrak g\right)$.
\begin{defin}\label{Defdb}
Set
\index{Dbt@$\mathfrak D_{b,\vartheta}$}%
\begin{align}\label{eq:co11}
&\mathfrak D_{b,\vartheta}=\widehat{D}^{\mathfrak 
g}_{\vartheta}+\cos\left(\vartheta\right)ic\left(\left[
    Y^{ \mathfrak k},Y^{ \mathfrak p}\right]\right)+\frac{1}{b}
    \left(\mathcal{D}^{ \mathfrak p}+\mathcal{E}^{ \mathfrak p}-
    i\mathcal{D}^{ \mathfrak k }+
    \cos\left(\vartheta\right)i\mathcal{E} ^{ \mathfrak k }\right),\\
    &\mathfrak D_{b,\vartheta}'=\widehat{D}^{\mathfrak g}
    +\cos\left(\vartheta\right)ic\left(\left[
    Y^{ \mathfrak k},Y^{ \mathfrak p}\right]\right)+\frac{1}{b}
    \left(\mathcal{D}^{ \mathfrak p}+\mathcal{E}^{ \mathfrak p}_{-\vartheta}-
    i\mathcal{D}^{ \mathfrak k }+
    \cos\left(\vartheta\right)i\mathcal{E} ^{ \mathfrak k }\right). \nonumber 
\end{align}
\end{defin}

The operators $\mathfrak D_{b,\vartheta},\mathfrak D'_{b,\vartheta}$ 
lie in $\End\left(C^{ \infty  }\left(G\times \mathfrak 
g,\Lambda\ac\left(\mathfrak g^{*}\right)\right)\right) \ho 
\widehat{c}\left(\overline{\mathfrak p}\right)$
\footnote{If $m$ is even, we could have written instead $\End \left( C^{\infty 
}\left(G\times \mathfrak g,\Lambda\ac\left(\mathfrak g^{*}\right)\otimes 
S^{\overline{\mathfrak p}}\right)\right) $. If $m$ is odd, one should 
use instead the identification in (\ref{eq:Lie9}). This discrepancy is 
relevant since when $m$ is odd, the $\Z_{2}$-grading of 
$Ìc\left(\overline{\mathfrak p}\right)$ does not come from 
$S^{\overline{\mathfrak p}}$.}.
Clearly,
\begin{equation}\label{eq:co12}
\mathfrak D'_{b,\vartheta}=\widehat{R}^{-1}_{\vartheta}\mathfrak 
D_{b,\vartheta}\widehat{R}_{\vartheta}.
\end{equation}
Comparing with (\ref{eq:rio1}), we get
\begin{equation}\label{eq:co13}
\mathfrak D_{b,0}=\mathfrak D'_{b,0}=\mathfrak D_{b}.
\end{equation}

For $a>0$, set
\index{Kka@$K^{\mathfrak k}_{a}$}%
\begin{equation}\label{eq:co13a1}
K^{\mathfrak k}_{a}s\left(g,Y^{\mathfrak p},Y^{\mathfrak 
k}\right)=a^{n/2}s\left(g,Y^{\mathfrak p},aY^{\mathfrak k}\right).
\end{equation}
The factor $a^{n/2}$ in (\ref{eq:co13a1}) is introduced only to make 
$K_{a}$ a $L_{2}$ isometry.
\begin{defin}\label{DEt}
For $\vartheta\in\left[0,\frac{\pi}{2}\right[$, set
\begin{equation}\label{eq:co13a2}
\overline{\mathfrak D}_{b,\vartheta}=K^{\mathfrak k}_{1/\cos^{1/2}\left(\vartheta\right)} 
\mathfrak D_{b,\vartheta}K^{\mathfrak 
k}_{\cos^{1/2}\left(\vartheta\right)}.
\end{equation}
\end{defin}

By (\ref{eq:co11}), (\ref{eq:co13a2}), we get
\begin{equation}\label{eq:co13a4}
\overline{\mathfrak D}_{b,\vartheta}=\widehat{D}^{\mathfrak 
g}_{\vartheta}+\cos^{1/2}\left(\vartheta\right)ic\left(\left[
    Y^{ \mathfrak k},Y^{ \mathfrak p}\right]\right)+\frac{1}{b}
    \left(\mathcal{D}^{ \mathfrak p}+\mathcal{E}^{ \mathfrak p} 
    \right) +\frac{\cos^{1/2}\left(\vartheta\right)}{b} \left( -
    i\mathcal{D}^{ \mathfrak k }+
    i\mathcal{E} ^{ \mathfrak k }\right).
\end{equation}
\begin{remk}\label{Rcons}
In our definition of $\mathfrak D_{b,\vartheta}$, another possibility 
is to still use equation (\ref{eq:co11}), with 
$\cos\left(\vartheta\right)$ replaced by $1$. This  
simplifies the algebraic computations which follow. However, this 
change would make the analysis  more difficult. 
\end{remk}
\subsection{The compression of $\overline{\mathfrak D}_{b,\vartheta}$}%
\label{subsec:compdb}
For $\vartheta\in \left[0,\frac{\pi}{2}\right[$,  by the results of 
section \ref{subsec:comp}, the kernel 
\index{H@$H$}%
$H$ of the operator 
$D^{\mathfrak p}+\mathcal{E}^{ \mathfrak p}+\frac{\cos^{1/2}\left(\vartheta\right)}{b}\left( 
-i\mathcal{D}^{\mathfrak k}+
i\mathcal{E}^{\mathfrak k} \right) $ is 
$1$-dimensional and is just $\left\{\exp\left(-\left\vert  
Y\right\vert^{2}/2 \right) 
\right\}\otimes 
S^{\overline{\mathfrak p}}$. We still denote by
\index{P@$P$}%
$P$  
the orthogonal projection on $H$. Let 
$H^{\perp}$ denote the orthogonal vector space to $H$.

The operator $\widehat{D}_{\vartheta}$ defined in (\ref{eq:defa1}) acts on 
$C^{ \infty }\left(G,S^{ 
\overline{\mathfrak p}} \right)$. As in section \ref{subsec:comp},  
the vector space 
  $C^{ \infty }\left(G,S^{\overline{\mathfrak p}}\right)$ 
can be identified with $H$, i.e., with a vector
subspace of $C^{ \infty }\left(G\times \mathfrak 
g,\Lambda\ac\left(\mathfrak g^{*}\right)\otimes S^{\overline{\mathfrak 
p}}\right)$ via the embedding $s\to s\exp\left(-\left\vert  
Y\right\vert^{2}/2\right)/\pi^{\left(m+n\right)/4}$.

Now, we  extend \cite[Proposition 2.10.1]{Bismut08b},  that was 
stated before as Proposition \ref{Pcomp1}.
\begin{prop}\label{Pcomp}
For $\vartheta\in \left[0,\frac{\pi}{2}\right[$, the following identity holds:
\begin{equation}\label{eq:co14}
P\left( \widehat{D}^{\mathfrak 
g}_{\vartheta}+\cos^{1/2}\left(\vartheta\right)ic\left(\left[Y^{ 
\mathfrak k},Y^{\mathfrak p}\right]\right) \right) 
P=\widehat{D}_{\vartheta}.
\end{equation}
\end{prop}
\begin{proof}
In the proof, we will assume that $e_{1},\ldots,e_{m}$ is an 
orthonormal basis of $\mathfrak p$, and $e_{m+1},\ldots,e_{m+n}$
is an orthonormal basis of $\mathfrak k$. We proceed as in the proof of \cite[Proposition 2.10.1]{Bismut08b}. 
Recall that the kernel $H$ is concentrated in degree $0$ in 
$\Lambda\ac\left(\mathfrak g^{*}\right)$, and that for $e\in 
\mathfrak g$, $c\left(e\right),\widehat{c}\left(e\right)$ act as odd 
operators on $\Lambda\ac\left(\mathfrak g^{*}\right)$. It follows that
\begin{equation}\label{eq:co15}
P\left( \sum_{1}^{m+n}\widehat{c}\left(R_{\vartheta}e^{*}_{i}\right)e_{i}
\right) P=
\sin\left(\vartheta\right)\widehat{D}.
\end{equation}
By \cite[eq. (2.7.4)]{Bismut08b}, we get
\begin{equation}
	 \widehat{c}\left(- \kappa^{ \mathfrak 
	 g}\right)=-2\sum_{i=m+1}^{m+n}\widehat{c}\left(e_{i}\right)
	 \widehat{c}\left(\ad\left(e_{i}\right)\vert_{ \mathfrak p}\right)+\widehat{c}\left(
	 -\kappa^{ \mathfrak k}\right).
	 \label{eq:bostoc2}
     \end{equation}
By (\ref{eq:bostoc2}), we deduce easily that
\begin{equation}\label{eq:co16}
P\widehat{R}_{\vartheta}\widehat{c}\left(- \kappa^{ \mathfrak 
	 g}\right)\widehat{R}_{\vartheta}^{-1}P=0.
\end{equation}
Similarly,
\begin{equation}\label{eq:co16y1}
Pc\left(\left[Y^{\mathfrak k},Y^{\mathfrak 
p}\right]\right)P=0.
\end{equation}
By (\ref{eq:kos20}), (\ref{eq:co8}), (\ref{eq:co15}),  (\ref{eq:co16}), and (\ref{eq:co16y1}), we get 
(\ref{eq:co14}). The proof of our proposition  is completed. 
\end{proof}
\subsection{A formula for $\mathfrak D^{2}_{b,\vartheta}, \mathfrak 
D^{\prime 2}_{b,\vartheta}$}%
\label{subsec:dbct}
We use the same notation as in subsection \ref{subsec:formDb2}.
We establish an extension of \cite[Theorem 2.11.1]{Bismut08b}, 
stated here as Theorem \ref{Tcarop}.
\begin{thm}\label{Tcaropbis}
The following identities hold:
\begin{align}\label{eq:rio2abis}
    &\frac{\mathfrak D^{2}_{b,\vartheta}}{2}=\frac{\widehat{D}^{ \mathfrak 
	 g,2}_{\vartheta}}{2}+\frac{\cos^{2}\left(\vartheta\right)}{2}\left\vert  \left[Y^{ \mathfrak k},Y^{ \mathfrak 
     p}\right]\right\vert^{2}
     +\frac{1}{2b^{2}} \Biggl( -\Delta^{ \mathfrak 
	 p \oplus \mathfrak k}+\left\vert  Y^{\mathfrak p}\right\vert^{2}
	 +\cos^{2}\left(\vartheta\right)
	 \left\vert  Y^{\mathfrak k}\right\vert^{2}\nonumber \\
	& -m-\cos\left(\vartheta\right)n\Biggr)+\frac{N^{\Lambda\ac\left( 
	 \mathfrak 
	 p^{*}\right)}+\cos\left(\vartheta\right)N^{\Lambda\ac
	 \left(\mathfrak k^{*}\right)}}{b^{2}}  \nonumber \\
     &+\frac{\cos\left(\vartheta\right)}{b}\Biggl(\underline{Y}^{ \mathfrak p}
     +i\underline{Y} ^{ \mathfrak 
     k}-i\n^{V}_{\left[Y^{ \mathfrak k},
     Y^{ \mathfrak p}\right]}+
     \widehat{c}_{\vartheta}\left(\ad\left(Y^{ \mathfrak p}+iY^{ 
		    \mathfrak k}\right)\vert_{\mathfrak g}
     \right) \nonumber \\
    & +2ic\left(\ad\left(Y^{ \mathfrak k}\right)\vert _{ 
     \mathfrak p} \right)
     -c\left(\ad\left(Y^{ \mathfrak 
     p}\right)\right)\Biggr),\\
  &\frac{\mathfrak D^{\prime 2}_{b,\vartheta}}{2}=\frac{\widehat{D}^{ \mathfrak 
	 g,2}}{2}+\frac{\cos^{2}\left(\vartheta\right)}{2}\left\vert  \left[Y^{ \mathfrak k},Y^{ \mathfrak 
     p}\right]\right\vert^{2}
      +\frac{1}{2b^{2}} \Biggl( -\Delta^{ \mathfrak 
	 p \oplus \mathfrak k}+\left\vert  Y^{\mathfrak p}\right\vert^{2}
	 +\cos^{2}\left(\vartheta\right)
	 \left\vert  Y^{\mathfrak k}\right\vert^{2}\nonumber \\
	& -m-\cos\left(\vartheta\right)n\Biggr)+\frac{N^{\Lambda\ac\left( 
	 \mathfrak 
	 p^{*}\right) \prime }_{-\vartheta}+\cos\left(\vartheta\right)N^{\Lambda\ac
	 \left(\mathfrak k^{*}\right)}}{b^{2}}  \nonumber \\
     &+\frac{\cos\left(\vartheta\right)}{b}\Biggl(\underline{Y}^{ \mathfrak p}+i\underline{Y} ^{ \mathfrak 
     k}-i\n^{V}_{\left[Y^{ \mathfrak k},
     Y^{ \mathfrak p}\right]}+
     \widehat{c}\left(\ad\left(Y^{ \mathfrak p}+iY^{ 
		    \mathfrak k}\right)\vert_{\mathfrak g}
     \right) \nonumber \\
    & +2ic\left(\ad\left(Y^{ \mathfrak k}\right)\vert _{ 
     \mathfrak p} \right)
     -c\left(\ad\left(Y^{ \mathfrak 
     p}\right)\right)\Biggr). \nonumber    
      \end{align}
\end{thm}
\begin{proof}
The proof of the first identity in (\ref{eq:rio2a}) is the same the 
proof of \cite[Theorem 2.11.1]{Bismut08b}. Combining this identity 
with  (\ref{eq:co12}), we get the second identity. The proof of our theorem is completed. 
\end{proof}
\begin{remk}\label{Rcons1}
A most important aspect of equation (\ref{eq:rio2abis}) is that the linear 
terms in the variable $Y$ all have the same weight 
$\frac{\cos\left(\vartheta\right)}{b}$. This would not be the case 
with the modification of $\mathfrak D^{X}_{b,\vartheta}$ that was 
suggested in Remark \ref{Rcons}. This fact will play a most important 
role in the analysis.
\end{remk}
\subsection{The superconnections $B,\overline{B},B'$}%
\label{subsec:scbs}
For $a\in \R$, set
\index{Ka@$K_{a}$}%
\begin{equation}\label{eq:ors-1}
K_{a}f\left(Y\right)=a^{\left(m+n\right)/2}f\left(aY\right).
\end{equation}

Let $\left(b,\vartheta\right)$ be the generic element of $\R^{*}_{+}\times 
\left[0,\frac{\pi}{2}\right[$. Let $d^{\R^{*}_{+}\times 
\left[0,\frac{\pi}{2}\right[}$ denote the de Rham operator on $\R^{*}_{+}\times 
\left[0,\frac{\pi}{2}\right[$.
\begin{defin}\label{Dconec}
 Over $\R^{*}_{+}\times 
   \left[0,\frac{\pi}{2}\right[$, let 
   $\n^{C^{ \infty }\left(G\times \mathfrak 
   g,\Lambda\ac\left(\mathfrak g^{*}\right) \otimes S^{\overline{\mathfrak p}}\right)},\n^{C^{ \infty }\left(G\times \mathfrak 
   g,\Lambda\ac\left(\mathfrak g^{*}\right) \otimes 
   S^{\overline{\mathfrak p}}\right)'}$ denote the flat connections on 
   $$C^{ \infty }\left(G\times \mathfrak 
   g,\Lambda\ac\left(\mathfrak g^{*}\right) \otimes 
   S^{\overline{\mathfrak p}}\right)$$
   that are given by
   \begin{align}\label{eq:co30x-3}
&\n^{C^{ \infty }\left(G\times \mathfrak 
   g,\Lambda\ac\left(\mathfrak g^{*}\right) \otimes S^{\overline{\mathfrak p}}\right)}
   =K_{b}\widehat{R}_{\vartheta}d^{\R^{*}_{+}\times 
\left[0,\frac{\pi}{2}\right[}\widehat{R}_{\vartheta}
^{-1}K_{b}^{-1}, \nonumber \\
&\overline{\n}^{C^{ \infty }\left(G\times \mathfrak 
   g,\Lambda\ac\left(\mathfrak g^{*}\right) \otimes 
   S^{\overline{\mathfrak p}}\right)}=K^{\mathfrak k}_{1/\cos^{1/2}\left(\vartheta\right)}
   \n^{C^{ \infty }\left(G\times \mathfrak 
   g,\Lambda\ac\left(\mathfrak g^{*}\right) \otimes 
   S^{\overline{\mathfrak p}}\right)}K^{\mathfrak 
   k}_{\cos^{1/2}\left(\vartheta\right)}, \\
&\n^{C^{ \infty }\left(G\times \mathfrak 
   g,\Lambda\ac\left(\mathfrak g^{*}\right) \otimes S^{\overline{\mathfrak p}}\right) \prime 
}=K_{b\cos\left(\vartheta\right)}d^{\R^{*}_{+}\times 
\left[0,\frac{\pi}{2}\right[}
K_{b\cos\left(\vartheta\right)}^{-1}. \nonumber 
\end{align}
\end{defin}

Then
 \begin{align}\label{eq:co30x-1}
&\n^{C^{ \infty }\left(G\times \mathfrak 
   g,\Lambda\ac\left(\mathfrak g^{*}\right) \otimes S^{\overline{\mathfrak p}}\right)}=db\left(\frac{\pa}{\pa 
   b}-\frac{1}{b}\n^{V}_{Y}-\frac{m+n}{2b}\right)+d\vartheta \left( \frac{\pa}{\pa 
\vartheta}-\widehat{c}\left(J\right)\right), \nonumber \\
&\overline{\n}^{C^{ \infty }\left(G\times \mathfrak 
   g,\Lambda\ac\left(\mathfrak g^{*}\right) \otimes S^{\overline{\mathfrak p}}\right)}=db\left(\frac{\pa}{\pa 
b}-\frac{1}{b}\n^{V}_{Y}-\frac{m+n}{2b}\right) \nonumber \\
&\qquad \qquad+d\vartheta \left( \frac{\pa}{\pa 
\vartheta}-\widehat{c}\left(J\right)-\frac{1}{2}\tan\left(\vartheta\right)\n^{V}
_{Y^{\mathfrak k}}-\frac{1}{4}\tan\left(\vartheta\right)n\right),\\
&\n^{C^{ \infty }\left(G\times \mathfrak 
   g,\Lambda\ac\left(\mathfrak g^{*}\right) \otimes S^{\overline{\mathfrak p}}\right)\prime }=db\left(\frac{\pa}{\pa 
   b}-\frac{1}{b}\n^{V}_{Y}-\frac{m+n}{2b}\right) \nonumber \\
   &\qquad \qquad +d\vartheta \left( \frac{\pa}{\pa \vartheta}
   +\tan\left(\vartheta\right)\n^{V}_{Y}+\tan\left(\vartheta\right)\frac{m+n}{2}\right) . \nonumber 
\nonumber 
\end{align}
\begin{prop}\label{Pderop}
    The following identities hold:
 \begin{align}\label{eq:co29x1y}
    &\n^{C^{ \infty }\left(G\times \mathfrak 
   g,\Lambda\ac\left(\mathfrak g^{*}\right) \otimes S^{\overline{\mathfrak p}}\right)}\mathfrak D _{b,\vartheta}= -\frac{2db}{b^{2}}\left(b
    \cos\left(\vartheta\right)ic\left(\left[Y^{\mathfrak k},Y^{\mathfrak p}\right]\right)
+\mathcal{E}^{\mathfrak p}+\cos\left(\vartheta\right)i\mathcal{E}^{\mathfrak k}\right) \nonumber \\
&-\frac{d\vartheta}{b}\left( b\sin\left(\vartheta\right)
ic\left(\left[Y^{\mathfrak k},Y^{\mathfrak p}\right]\right) 
+\widehat{c}\left(\overline{Y}^{\mathfrak p}\right)+\sin\left(\vartheta\right)i\mathcal{E}^{\mathfrak k}
\right) ,  \nonumber   \\
&\overline{ \n}^{C^{ \infty }\left(G\times \mathfrak 
   g,\Lambda\ac\left(\mathfrak g^{*}\right) \otimes S^{\overline{\mathfrak p}}\right)}\overline{\mathfrak D} _{b,\vartheta}= -\frac{2db}{b^{2}}\left(b
    \cos^{1/2}\left(\vartheta\right)ic\left(\left[Y^{\mathfrak k},Y^{\mathfrak p}\right]\right)
+\mathcal{E}^{\mathfrak p}+\cos^{1/2}\left(\vartheta\right)i\mathcal{E}^{\mathfrak k}\right) \nonumber \\
&-\frac{d\vartheta}{b}\left( 
b\frac{\sin\left(\vartheta\right)}{\cos^{1/2}\left(\vartheta\right)}
ic\left(\left[Y^{\mathfrak k},Y^{\mathfrak p}\right]\right) 
+\widehat{c}\left(\overline{Y}^{\mathfrak 
p}\right)+\frac{\sin\left(\vartheta\right)}{\cos^{1/2}\left(\vartheta\right)}i\mathcal{E}^{\mathfrak k}
\right) ,   \\
&\n^{C^{ \infty }\left(G\times \mathfrak 
   g,\Lambda\ac\left(\mathfrak g^{*}\right) \otimes S^{\overline{\mathfrak 
   p}}\right) \prime }\mathfrak D'
_{b,\vartheta}=-\frac{2db}{b^{2}}\left(b
\cos\left(\vartheta\right)ic\left(\left[Y^{\mathfrak k},Y^{\mathfrak p}\right]\right)
+\mathcal{E}^{\mathfrak p}_{-\vartheta}+\cos\left(\vartheta\right)i\mathcal{E}^{\mathfrak k}\right) \nonumber \\
&+\frac{d\vartheta}{b}\left(b\sin\left(\vartheta\right)
ic\left(\left[Y^{\mathfrak k},Y^{\mathfrak p}\right]\right)-
\tan \left(\vartheta\right)\left(\mathcal{D}^{\mathfrak p}-i\mathcal{D}^{\mathfrak k}\right)
-\frac{1}{\cos\left(\vartheta\right)}\widehat{c}\left(\overline{Y}^{\mathfrak p}\right)\right). \nonumber 
\end{align}
\end{prop}
\begin{proof}
Our proposition follows from (\ref{eq:co11}), (\ref{eq:co13a2}), 
(\ref{eq:co30x-3}), and (\ref{eq:co30x-1}).
\end{proof}
\begin{defin}\label{Dsupnew}
Let 
\index{B@$B$}%
\index{B@$\overline{B}$}%
\index{B@$B'$}%
$B,\overline{B},B'$ be the superconnections, 
   \begin{align}\label{eq:co31-x}
&B=\n^{C^{ \infty }\left(G\times \mathfrak 
   g,\Lambda\ac\left(\mathfrak g^{*}\right) \otimes S^{\overline{\mathfrak 
   p}}\right)}+\frac{\mathfrak D_{b,\vartheta}}{\sqrt{2}}, \nonumber 
   \\
   &\overline{B}=\overline{\n}^{C^{ \infty }\left(G\times \mathfrak 
   g,\Lambda\ac\left(\mathfrak g^{*}\right) \otimes S^{\overline{\mathfrak 
   p}}\right)}+\frac{\overline{\mathfrak 
   D}_{b,\vartheta}}{\sqrt{2}},\\
   &B'=\n^{C^{ \infty }\left(G\times \mathfrak 
   g,\Lambda\ac\left(\mathfrak g^{*}\right) \otimes S^{\overline{\mathfrak 
   p}}\right) \prime }+\frac{\mathfrak D'_{b,\vartheta}}{\sqrt{2}} . 
   \nonumber 
\end{align}
\end{defin}

By (\ref{eq:co13a2}), (\ref{eq:co30x-3}), we get
\begin{equation}\label{eq:co31-xa1}
\overline{B}=K^{\mathfrak 
k}_{1/\cos^{1/2}\left(\vartheta\right)}BK^{\mathfrak k}
_{\cos^{1/2}\left(\vartheta\right)}.
\end{equation}

Recall  that the superconnection 
\index{A@$A$}%
$A$ was defined in (\ref{eq:defa2x}).
We will extend Proposition  \ref{Pcomp}. We make temporarily $db=0$.
\begin{prop}\label{Pcompbi}
The following identity of superconnections on $C^{\infty 
}\left(G,S^{\overline{\mathfrak p}}\right)$ over 
$\left[0,\frac{\pi}{2}\right[$ holds:
\begin{equation}\label{eq:co14x}
P\left(\overline{\n}^{C^{ \infty }\left(G\times \mathfrak 
g,\Lambda\ac\left(\mathfrak g^{*}\right)\otimes S^{\overline{\mathfrak 
p}}\right)}\vert_{db=0}+\frac{1}{\sqrt{2}} \left( \widehat{D}^{\mathfrak 
g}_{\vartheta}+\cos^{1/2}\left(\vartheta\right)ic\left(\left[Y^{ 
\mathfrak k},Y^{\mathfrak p}\right]\right) \right) \right) P=A.
\end{equation}
\end{prop}
\begin{proof}
In degree $0$ in the variable $d\vartheta$, equation (\ref{eq:co14x}) was already established in 
Proposition  \ref{Pcomp}. To establish (\ref{eq:co14x}), we only need to 
show that we have the identity of connections on $C^{ \infty 
}\left(G, S^{\overline{\mathfrak p}}\right)$,
\begin{equation}\label{eq:gign1}
Pd\vartheta\left(\frac{\pa}{\pa 
\vartheta}-\widehat{c}\left(J\right)-\frac{1}{2}\tan\left(\vartheta\right)\n_{Y^{\mathfrak k}}
-\frac{1}{4}\tanh\left(\vartheta\right)n\right)P=d\vartheta\frac{\pa}{\pa \vartheta}.
\end{equation}
By (\ref{eq:co4}), we get
\begin{equation}\label{eq:gign4}
P\widehat{c}\left(J\right)P=0.
\end{equation}
Also by \cite[eq. (2.16.23)]{Bismut08b}, if $u\in \mathfrak k$, we 
have the elementary identity
\begin{equation}\label{eq:ele1}
P\left\langle  u,Y^{\mathfrak k}\right\rangle^{2}P=\frac{1}{2}\left\vert  
u\right\vert^{2}.
\end{equation}
From (\ref{eq:ele1}), it is easy to deduce that
\begin{equation}\label{eq:ele2}
P\left(\n_{Y^{\mathfrak k}}+\frac{n}{2}\right)P=0.
\end{equation}
By (\ref{eq:gign4}), (\ref{eq:ele2}), we get (\ref{eq:gign1}).  The 
proof of our proposition  is completed. 
\end{proof}
\subsection{The operators $\mathcal{L}^{X}_{b,\vartheta}, 
\overline{\mathcal{L}}^{X}_{\bt},
\mathcal{L}^{X \prime }_{b,\vartheta}$}%
\label{subsec:des}
We use the same notation as in section \ref{sec:eta}.  Then $J$ and 
$R_{\vartheta}$ still act on $TX \oplus \overline{TX}$, so that 
$\widehat{c}\left(J\right)$ is  a section of $\widehat{c}\left(TX 
\oplus N \oplus \overline{TX}\right)$. 

  We define 
$\widehat{R}_{\vartheta}$ as in (\ref{eq:co5}). Recall that the vector space 
\index{H@$\mathcal{H}$}%
$\mathcal{H}$ was defined in Definition 
\ref{Dspace}. As in 
(\ref{eq:co13a1}), if $s\in \mathcal{H}$, set
\begin{equation}\label{eq:ele3}
K^{N}_{a}s\left(x,Y^{TX},Y^{N}\right)=a^{n/2}s\left(x,Y^{TX},aY^{N}\right).
\end{equation}
As in section \ref{subsec:actqu},  
$\widehat{D}^{\mathfrak g}_{\vartheta}, \mathfrak D
_{b,\vartheta}, \mathfrak D'_{\bt},\overline{\mathfrak 
D}_{b,\vartheta}$ 
 descend to  operators  
 \index{DgXt@$\widehat{D}^{\mathfrak g,X}_{\vartheta}$}%
 \index{DXbt@$\mathfrak D^{X}_{b,\vartheta}$}%
 \index{DXbt@$\mathfrak D^{X \prime }_{b,\vartheta}$}%
 \index{DXbt@$\overline{\mathfrak D}^{X}_{\bt}$}%
 $\widehat{D}^{\mathfrak g,X}_{\vartheta}, 
 \mathfrak D^{X}_{b,\vartheta},\mathfrak 
D^{X \prime }_{b,\vartheta},\overline{\mathfrak D}^{X}_{\bt}$ acting on $\mathcal{H}$.  By 
(\ref{eq:co12}), (\ref{eq:co13a2}),  we get
\begin{align}\label{eq:co2bu}
&\mathfrak  D^{X \prime}_{b,\vartheta}=\widehat{R}^{-1}_{\vartheta} 
\mathfrak D^{X}_{b,\vartheta}\widehat{R}_{\vartheta},
&\overline{\mathfrak 
D}^{X}_{\bt}=K^{N}_{1/\cos^{1/2}\left(\vartheta\right)} \mathfrak 
D^{X}_{\bt}K^{N}_{\cos^{1/2}\left(\vartheta\right)}.
\end{align}
\begin{defin}\label{defnew}
Put
\index{LXbt@$\mathcal{L}_{b,\vartheta}^{X}$}%
\index{LXbt@$\mathcal{L}_{b,\vartheta}^{X \prime}$}%
\begin{align}\label{eq:co17}
&\mathcal{L}_{b,\vartheta}^{X}=
-\frac{1}{2}\widehat{D}^{\mathfrak g,X,2}_{\vartheta}+
\frac{1}{2} \mathfrak D^{X,2}_{b,\vartheta}, \nonumber \\
&\overline{\mathcal{L}}_{b,\vartheta}^{X}=
-\frac{1}{2}\widehat{D}^{\mathfrak g,X,2}_{\vartheta}+
\frac{1}{2} \overline{\mathfrak D}^{X,2}_{b,\vartheta}, \\
&\mathcal{L}_{b,\vartheta}^{X \prime}=-\frac{1}{2}\widehat{D}^{\mathfrak g,X,2}+\frac{1}{2}\mathfrak D^{X \prime 
,2}_{b,\vartheta}. \nonumber 
\end{align}
\end{defin}

As observed after (\ref{eq:co10}), we can rewrite the first equation 
in (\ref{eq:co17}) in the form
\begin{equation}\label{eq:co17ra1}
\mathcal{L}_{b,\vartheta}^{X}=
-\frac{1}{2}\widehat{D}^{\mathfrak g,X,2}+
\frac{1}{2} \mathfrak D^{X,2}_{b,\vartheta}.
\end{equation}

By (\ref{eq:co8}), (\ref{eq:co2bu}), and (\ref{eq:co17}), we get
\begin{align}\label{eq:co18}
&\mathcal{L}_{b,\vartheta}^{X \prime 
}=\widehat{R}^{-1}_{\vartheta}\mathcal{L}_{b,\vartheta}^{X 
}\widehat{R}_{\vartheta},
&\overline{\mathcal{L}}^{X}_{\bt}=K^{N}_{1/\cos^{1/2}\left(\vartheta\right)}\mathcal{L}^{X}_{\bt}
K^{N}_{\cos^{1/2}\left(\vartheta\right)}.
\end{align}

 As before, we extend $\ad\left(Y^{TX}\right)$ to an endomorphism 
$\ad\left(Y^{TX}\right)_{TX \oplus N}$ of $TX \oplus N \oplus 
\overline{TX}$ 
that coincides with $\ad\left(Y^{TX}\right)$ on $TX \oplus N$, and 
vanishes on $\overline{TX}$. Similarly, we denote by 
$\ad\left(Y^{N}\vert_{\overline{TX}}\right)$ the morphism of $TX \oplus N 
\oplus \overline{TX}$ that 
extends the action of $\ad\left(Y^{N}\right)$ on  
$\overline{TX}$ by $0$ on $TX \oplus N$. As in (\ref{eq:co16x1}),  (\ref{eq:co8}), put
\index{NLTXt@$N_{\vartheta}^{\Lambda\ac\left(T^{*}X\right) \prime}$}%
\index{ETXt@$\mathcal{E}^{TX}_{\vartheta}$}%
\begin{align}\label{eq:co18z1}
    &N_{\vartheta}^{\Lambda\ac\left(T^{*}X\right) \prime}=\widehat{R}_{\vartheta}
    N^{\Lambda\ac\left(T^{*}X\right)}\widehat{R}_{\vartheta}^{-1},
&\mathcal{E}^{TX}_{\vartheta}=\widehat{R}_{\vartheta}\mathcal{E}^{TX}\widehat{R}
_{\vartheta}^{-1}.
\end{align}
Set
\index{NLTXN@$N_{\vartheta}^{\Lambda\ac\left(T^{*}X \oplus N^{*}\right)}$}%
\index{NLTXN@$N_{\vartheta}^{\Lambda\ac\left(T^{*}X \oplus 
N^{*}\right) \prime}$}%
\begin{align}\label{eq:gip1}
   &N_{\vartheta}^{\Lambda\ac\left(T^{*}X \oplus 
   N^{*}\right)}=N^{\Lambda\ac\left(\TsX\right)}+\cos\left(\vartheta\right)N^{\Lambda\ac\left(N^{*}\right)},\\
&N^{\Lambda\ac\left(T^{*}X \oplus 
N^{*}\right) \prime 
}_{\vartheta}=N_{\vartheta}^{\Lambda\ac\left(\TsX\right) \prime }+
\cos\left(\vartheta\right)N^{\Lambda\ac\left(N^{*}\right)}.\nonumber 
\end{align}
Then
\begin{equation}\label{eq:gip2}
N^{\Lambda\ac\left(T^{*}X \oplus 
N^{*}\right) \prime }_{\vartheta}=\widehat{R}_{\vartheta}
N^{\Lambda\ac\left(T^{*}X \oplus N^{*}\right)}_{\vartheta}\widehat{R}_{-\vartheta}.
\end{equation}

By  (\ref{eq:co16x3}),  (\ref{eq:gip1}), we obtain
\begin{multline}\label{eq:comp1z1}
N_{-\vartheta}^{\Lambda\ac\left(T^{*}X \oplus N^{*}\right) \prime }=\cos\left(\vartheta\right)N^{\Lambda\ac\left(T^{*}X
\oplus 
N^{*}\right)} \\
-\frac{1}{2}\sin\left(\vartheta\right)\sum_{i=1}^{m}c\left(e_{i}\right)
\widehat{c}\left(\overline{e}_{i}\right)
+\frac{m}{2}\left(1-\cos\left(\vartheta\right)\right).
\end{multline}

Let 
\index{DTX@$\Delta^{TX}$}%
\index{DN@$\Delta^{N}$}%
$\Delta^{TX},\Delta^{N}$ denote the Laplacians acting along the 
fibres of the Euclidean vector bundles $TX,N$. Now we extend \cite[Theorem 2.12.5 and eq. (2.13.5)]{Bismut08b}, that was stated 
before as Theorem \ref{Timpfo}.
\begin{thm}\label{Tform}
The following identities hold:
\begin{align}\label{eq:vb4a}
&\mathfrak D_{b,\vartheta}^{X}=\widehat{D}^{ \mathfrak g,X}_{\vartheta}+
\cos\left(\vartheta\right)ic\left(\left[Y^{ 
N},Y^{TX}\right]\right)+\frac{1}{b}\left(\mathcal{D}^{ TX}+\mathcal{E}^{ TX}-i\mathcal{D}^{ 
N}+\cos\left(\vartheta\right)i\mathcal{E}^{N}\right),   \nonumber \\
&\overline{\mathfrak D}^{X}_{\bt}=\widehat{D}^{ \mathfrak g,X}_{\vartheta}+\cos^{1/2}\left(\vartheta\right)ic\left(\left[Y^{N},Y^{TX}\right]\right)
+\frac{1}{b}\left(\mathcal{D}^{TX}+\mathcal{E}^{TX}\right) \\
&\qquad \qquad+\frac{\cos^{1/2}\left(\vartheta\right)}{b}
\left(-i\mathcal{D}^{N}+i\mathcal{E}^{N}\right), \nonumber \\
&\mathfrak D_{b,\vartheta}^{ X \prime}=\widehat{D}^{ \mathfrak 
g,X}+\cos\left(\vartheta\right)ic\left(\left[Y^{ 
N},Y^{TX}\right]\right)+\frac{1}{b}\left(\mathcal{D}^{ 
TX}+\mathcal{E}^{TX}_{-\vartheta}
-i\mathcal{D}^{ 
N}+\cos\left(\vartheta\right)i\mathcal{E}^{N}\right). \nonumber 
\end{align}
Moreover,
\begin{align}\label{eq:co19x-1}
    &\mathcal{L}^{X }_{b,\vartheta}=\frac{\cos^{2}\left(\vartheta\right)}{2}\left\vert  \left[Y^{ N},Y^{TX}\right]\right\vert^{2}
     +\frac{1}{2b^{2}} \Biggl( -\Delta^{TX\oplus N}+\left\vert  
     Y^{TX}\right\vert^{2}+\cos^{2}\left(\vartheta\right)
     \left\vert  Y^{N}\right\vert^{2}-m \nonumber \\
     &-\cos\left(\vartheta\right)n\Biggr) 
    +\frac{N_{-\vartheta}^{\Lambda\ac\left(T^{*}X \oplus N^{*}\right) }}{b^{2}} 
   +\frac{\cos\left(\vartheta\right)}{b}\Biggl(\n_{Y^{ TX}}^{C^{ \infty }\left(TX \oplus 
 N,\widehat{\pi}^{*} \left( \Lambda\ac\left(T^{*}X \oplus N^{*}\right)\otimes S^{\overline{TX}} \otimes F 
 \right) \right)} \nonumber \\
   &  +
	 \widehat{c}_{\vartheta}\left(\ad\left(Y^{TX}\right)
	 \vert_{TX \oplus N} \right) 
 -c\left(\ad\left(Y^{ 
	  TX}\right)
   +i\theta\ad\left(Y^{N}\right)\right) \nonumber \\ 
  & -i\widehat{c}_{\vartheta}\left(\ad\left(Y^{N}\right)\vert_{\overline{TX}}\right)
   -i\rho^{F}\left(Y^{N}\right)\Biggr) , \nonumber \\
   &\overline{\mathcal{L}}^{X }_{b,\vartheta}=\frac{\cos\left(\vartheta\right)}{2}\left\vert  \left[Y^{ N},Y^{TX}\right]\right\vert^{2}
     +\frac{1}{2b^{2}} \left( -\Delta^{TX}+\left\vert  
     Y^{TX}\right\vert^{2}
     -m\right)  \nonumber \\
    & +\frac{\ct}{2b^{2}}
     \left(-\Delta^{N}+\left\vert  Y^{N}\right\vert^{2}-n\right)
    +\frac{N^{\Lambda\ac\left(T^{*}X \oplus 
    N^{*}\right)}_{-\vartheta}}{b^{2}} \nonumber \\ 
     &+\frac{\cos\left(\vartheta\right)}{b}\Biggl(\n_{Y^{ TX}}^{C^{ \infty }\left(TX \oplus 
 N,\widehat{\pi}^{*} \left( \Lambda\ac\left(T^{*}X \oplus N^{*}\right)\otimes 
 S^{\overline{TX}} \otimes F 
 \right) \right)}
     +
	 \widehat{c}_{\vartheta}\left(\ad\left(Y^{TX}\right)
	 \vert_{TX \oplus N} \right)   \\ 
 & -c\left(\ad\left(Y^{ TX}\right) 
   \right) \Biggr) 
   -i\frac{\cos^{1/2}\left(\vartheta\right)}{b}
   \left(c\left(\theta\ad\left(Y^{N}\right)\right)+
   \widehat{c}_{\vartheta}\left(\ad\left(Y^{N}\right)\vert_{\overline{TX}}\right)
   +\rho^{F}\left(Y^{N}\right)\right), \nonumber \\ 
   &\mathcal{L}^{X \prime }_{b,\vartheta}=\frac{\cos^{2}\left(\vartheta\right)}{2}\left\vert  \left[Y^{ N},Y^{TX}\right]\right\vert^{2}
     +\frac{1}{2b^{2}} \Biggl( -\Delta^{TX\oplus N}+\left\vert  
     Y^{TX}\right\vert^{2}+\cos^{2}\left(\vartheta\right)
     \left\vert  Y^{N}\right\vert^{2} -m \nonumber \\
     &-\cos\left(\vartheta\right)n\Biggr) 
    +\frac{N^{\Lambda\ac\left(T^{*}X \oplus N^{*}\right) \prime}_{-\vartheta}}{b^{2}} 
     +\frac{\cos\left(\vartheta\right)}{b}\Biggl(\n_{Y^{ TX}}^{C^{ \infty }\left(TX \oplus 
 N,\widehat{\pi}^{*} \left( \Lambda\ac\left(T^{*}X \oplus N^{*}\right)\otimes S^{\overline{TX}} \otimes F 
 \right) \right)} \nonumber \\
    & +
	  \widehat{c}\left(\ad\left(Y^{TX}\right)
	  \right)  
  -c\left(\ad\left(Y^{ 
	  TX}\right)
   +i\theta\ad\left(Y^{N}\right)\right)-i\widehat{c}\left(\ad\left(Y^{N}\right)\vert_{\overline{TX}}\right)
   -i\rho^{F}\left(Y^{N}\right)\Biggr). \nonumber 
\end{align}
\end{thm}
\begin{proof}
The identities in (\ref{eq:vb4a}) follow from
(\ref{eq:co11}), (\ref{eq:co13a4}).
By proceeding as in the proof of \cite[Theorem 2.12.5]{Bismut08b}, 
and using the second equation in (\ref{eq:rio2abis}), we get the  
third 
equation in (\ref{eq:co19x-1}). By conjugating this equation by 
$\widehat{R}_{\vartheta}$, we get the first equation. By 
(\ref{eq:co18}) and by the first identity in (\ref{eq:co19x-1}), we 
get the second identity. The proof of our theorem is completed. 
\end{proof}
\begin{remk}\label{Rcons2}
The considerations of Remark \ref{Rcons1} still apply here. Namely, 
in $\mathfrak L^{X}_{\bt}, \mathfrak L^{X \prime }_{\bt}$, 
the linear terms in $Y^{TX},Y^{N}$ all have the same factor 
$\frac{\cos\left(\vartheta\right)}{b}$.
\end{remk}

Now we follow \cite[Definition 2.4.1]{Bismut08b}.
\begin{defin}\label{Dflatbis}
     Let 
     \index{nLTX@$\n^{\Lambda\ac\left(T^{*}X \oplus N^{*}\right), f*,\hat{f}}$}%
     $\n^{\Lambda\ac\left(T^{*}X \oplus N^{*}\right), f*,\hat{f}}$ be the connection 
     on $\Lambda\ac\left(T^{*}X \oplus N^{*}\right)$,
     \begin{equation}
	 \n^{\Lambda\ac\left(T^{*}X \oplus N^{*}\right), 
	 f*,\hat{f}}_{\cdot}=\n^{\Lambda\ac\left(T^{*}X \oplus 
	 N^{*}\right)}_{\cdot}-c\left(\ad\left(\cdot\right)\right)+\widehat{c}\left(\ad\left(\cdot\right)\right).
	 \label{eq:glub1}
     \end{equation}
     By \cite[Proposition 2.4.2]{Bismut08b}, $\n^{\Lambda\ac\left(T^{*}X \oplus N^{*}\right), f*,\hat{f}}$
     is a flat connection. Let
     \index{nCTX@$\n^{C^{\infty }\left( TX \oplus N,\widehat{\pi}^{*} \left( \Lambda\ac\left(T^{*}X 
     \oplus N^{*}\right)\otimes S^{\overline{TX}} \otimes F \right) 
     \right) , f*,\hat{f}}$}%
     $$\n^{C^{\infty }
     \left( TX \oplus N,\widehat{\pi}^{*} \left( \Lambda\ac\left(T^{*}X 
     \oplus N^{*}\right)\otimes S^{\overline{TX}} \otimes F \right) 
     \right) , f*,\hat{f}}$$
     be the 
     connection on $C^{\infty }\left( TX \oplus N,
     \widehat{\pi}^{*} \left( \Lambda\ac\left(T^{*}X 
     \oplus N^{*}\right)\otimes S^{\overline{TX}} \otimes F \right)  \right) $ that is induced by 
     $\n^{\Lambda\ac\left(T^{*}X \oplus N^{*}\right), 
     f*,\hat{f}},\n^{S^{\overline{TX}}},\n^{F}$.
     \end{defin}
     
     We can rewrite the third equation in (\ref{eq:co19x-1}) in the 
     form
     \begin{multline}\label{eq:glub2}
\mathcal{L}^{X \prime }_{b,\vartheta}=\frac{\cos^{2}\left(\vartheta\right)}{2}\left\vert  \left[Y^{ N},Y^{TX}\right]\right\vert^{2}
     +\frac{1}{2b^{2}} \Biggl( -\Delta^{TX\oplus N}+\left\vert  
     Y^{TX}\right\vert^{2}+\cos^{2}\left(\vartheta\right)
     \left\vert  Y^{N}\right\vert^{2} -m \\
     -\cos\left(\vartheta\right)n\Biggr) 
    +\frac{N_{-\vartheta}^{\Lambda\ac\left(\TsX \oplus 
    N^{*}\right) \prime }}{b^{2}} 
     +\frac{\cos\left(\vartheta\right)}{b}\Biggl(\n_{Y^{ TX}}^{C^{ \infty }\left(TX \oplus 
 N,\widehat{\pi}^{*} \left( \Lambda\ac\left(T^{*}X \oplus 
 N^{*}\right)\otimes S^{\overline{TX}} \otimes F
 \right) \right),f*,\widehat{f} }
	    \\
  -ic\left(
   \theta\ad\left(Y^{N}\right) \right) 
   -i\widehat{c}\left(\ad\left(Y^{N}\right)\vert_{\overline{TX}}\right)
   -i\rho^{F}\left(Y^{N}\right)\Biggr)  . 
\end{multline}
\subsection{A formula relating $\overline{\mathcal{L}}^{X}_{b,\vartheta}$ to 
$\mathcal{L}^{X}_{0,\vartheta}$}%
\label{subsec:formrebis}
We proceed as in subsection \ref{subsec:formre}. By (\ref{eq:co19x-1}), 
we can write $\overline{\mathcal{L}}^{X}_{b,\vartheta}$ in 
the form
\begin{equation}\label{eq:co18x1}
\overline{\mathcal{L}}^{X}_{b,\vartheta}=\frac{\alpha_{\vartheta}}{b^{2}}
+\frac{\beta_{\vartheta}}{b}+\gamma_{\vartheta}.
\end{equation}

For $\vartheta\in \left[0,\frac{\pi}{2}\right[$, the 
operator $\alpha_{\vartheta}$ acting fibrewise has discrete spectrum. Its  
kernel  is the vector space 
\index{H@$H$}%
$H$ in (\ref{eq:bugr2}). We still define
\index{Hp@$H^{\perp}$}%
$H^{\perp}$ as in subsection \ref{subsec:formre}.

Note that 
$\beta_{\vartheta}$ maps $H$ into $H^{\perp}$.
Let $\alpha_{\vartheta}^{-1}$ be the inverse 
of $\alpha_{\vartheta}$ restricted to $H^{\perp}$.

Recall that the operator 
\index{LXt@$\mathcal{L}^{X}_{0,\vartheta}$}%
$\mathcal{L}^{X}_{0,\vartheta}$ was defined in 
Definition \ref{DLXt}. Now we establish an extension of  
\cite[Theorem 2.16.1]{Bismut08b}, which was stated before as Theorem 
\ref{Tfundid}.
\begin{thm}\label{Thfub}
For $\vartheta\in\left[0,\frac{\pi}{2}\right[$, the following identity holds:
\begin{equation}\label{eq:co21}
P\left(\gamma_{\vartheta}-\beta_{\vartheta}\alpha_{\vartheta}^{-1}\beta_{\vartheta}\right)P=
\mathcal{L}^{X}_{0,\vartheta}.
\end{equation}
\end{thm}
\begin{proof}
We proceed as in \cite[Theorem 2.16.1]{Bismut08b}. By 
(\ref{eq:vb4a}), we can write $\frac{1}{\sqrt{2}} \overline{\mathfrak D}^{X}_{b,\vartheta}$ in the form
\begin{equation}\label{eq:co22}
\frac{1}{\sqrt{2}}\overline{\mathfrak D}^{X}_{b,\vartheta}=E_{\vartheta}+\frac{F_{\vartheta}}{b}.
\end{equation}
Using (\ref{eq:co17}),  and comparing 
(\ref{eq:co18x1}) and (\ref{eq:co22}), we obtain
\begin{align}\label{eq:co23}
&\alpha_{\vartheta}=F_{\vartheta}^{2},\qquad
\beta_{\vartheta}=\left[E_{\vartheta},F_{\vartheta}\right],
&\gamma_{\vartheta}=E_{\vartheta}^{2}-\frac{1}{2}
\widehat{D}^{\mathfrak g,X,2}_{\vartheta}.
\end{align}
By (\ref{eq:co23}), we obtain
\begin{equation}\label{eq:co24}
P\left(\gamma_{\vartheta}-\beta_{\vartheta}
\alpha_{\vartheta}^{-1}\beta_{\vartheta}\right)P=
P\left(E^{2}_{\vartheta}-E_{\vartheta}P^{\perp}E_{\vartheta}-
\frac{1}{2}\widehat{D}^{\mathfrak g,X,2}_{\vartheta}\right)P.
\end{equation}
By equation (\ref{eq:co14}) in Proposition \ref{Pcomp}, we get
\begin{equation}\label{eq:co25}
P\left(E^{2}_{\vartheta}-E_{\vartheta}P^{\perp}_{\vartheta}
E_{\vartheta}\right)P=
\left(PE_{\vartheta}P\right)^{2}=\frac{1}{2}
\sin^{2}\left(\vartheta\right)\widehat{D}^{X,2}.
\end{equation}
By (\ref{eq:Lie17}), (\ref{eq:kos22}), (\ref{eq:co24}), and 
(\ref{eq:co25}), we obtain
\begin{equation}\label{eq:group0}
P\left(\gamma_{\vartheta}-\beta_{\vartheta}
\alpha_{\vartheta}^{-1}\beta_{\vartheta}\right)P=\frac{1}{2}\sin^{2}\left(\vartheta\right)
\widehat{D}^{X,2}+\mathcal{L}_{0}^{X}.
\end{equation}
By  (\ref{eq:gzinc1}),  (\ref{eq:group0}),  we get (\ref{eq:co21}).  
The proof of our theorem is completed. 
\end{proof}
\subsection{The superconnections $B^{X},\overline{B}^{X},B^{X \prime}$}%
\label{subsec:newsup}
The  flat connections 
$$\n^{C^{ \infty }\left(G\times \mathfrak 
g,\Lambda\ac\left(\mathfrak g^{*} \right) \otimes S^{\overline{\mathfrak p}} \right)},
\overline{\n}^{C^{\infty }\left(G\times \mathfrak  
g, \Lambda\ac\left(\mathfrak g^{*}\right) \otimes  
S^{\overline{\mathfrak p}}\right)},
\n^{
 C^{ \infty }\left(G\times \mathfrak 
g, \Lambda\ac\left(\mathfrak g^{*}\right)  \otimes S^{\overline{\mathfrak p}} \right)\prime }$$
induce flat connections 
\index{nH@$\n^{\mathcal{H}}$}%
\index{nH@$\overline{\n}^{\mathcal{H}}$}%
\index{nH@$\n^{\mathcal{H} \prime}$}%
$\n^{\mathcal{H}},\overline{\n}^{\mathcal{H}},\n^{\mathcal{H} \prime }$ on the vector bundle $\mathcal{H}$ over 
 $\R_{+}^{*}\times\left[0,\frac{\pi}{2}\right[$.
Let $d^{\R^{*}_{+}\times 
\left[0,\frac{\pi}{2}\right[}$ still denote the de Rham operator on $\R^{*}_{+}\times 
\left[0,\frac{\pi}{2}\right[$.

For $a\in\R$, we still define 
\index{Ka@$K_{a}$}%
$K_{a}$ as in (\ref{eq:ors-1}), i.e., 
if $s\in \mathcal{H}$, then
\begin{equation}\label{eq:ors-1bis}
K_{a}s\left(x,Y\right)=a^{\left(m+n\right)/2}s\left(x,aY\right).
\end{equation}
By (\ref{eq:co30x-3}), we get
  \begin{align}\label{eq:co30x-3bis}
&\n^{\mathcal{H}}
   =K_{b}\widehat{R}_{\vartheta}d^{\R^{*}_{+}\times 
\left[0,\frac{\pi}{2}\right[}\widehat{R}_{\vartheta}
^{-1}K_{b}^{-1}, \nonumber \\
&\overline{\n}^{\mathcal{H}}=K^{N}_{1/\cos^{1/2}\left(\vartheta\right)}
   \n^{\mathcal{H}}K^{N}_{\cos^{1/2}\left(\vartheta\right)}, \\
&\n^{\mathcal{H} \prime }=K_{b\cos\left(\vartheta\right)}d^{\R^{*}_{+}\times 
\left[0,\frac{\pi}{2}\right[}
K_{b\cos\left(\vartheta\right)}^{-1}. \nonumber 
\end{align}

By (\ref{eq:co30x-1}), we obtain
   \begin{align}\label{eq:co30x-1a}
&\n^{\mathcal{H}}=db\left(\frac{\pa}{\pa 
b}-\frac{1}{b}\n^{V}_{Y}-\frac{m+n}{2b}\right)+d\vartheta \left( \frac{\pa}{\pa 
\vartheta}-\widehat{c}\left(J\right)\right),\nonumber \\
&\overline{\n}^{\mathcal{H}}=db\left(\frac{\pa}{\pa 
b}-\frac{1}{b}\n^{V}_{Y}-\frac{m+n}{2b}\right) \nonumber \\
&\qquad \qquad+d\vartheta \left( \frac{\pa}{\pa 
\vartheta}-\widehat{c}\left(J\right)-\frac{1}{2}\tan\left(\vartheta\right)\n^{V}
_{Y^{N}}-\frac{1}{4}\tan\left(\vartheta\right)n\right),\\
&\n^{\mathcal{H} \prime }=db\left(\frac{\pa}{\pa 
b}-\frac{1}{b}\n^{V}_{Y}-\frac{m+n}{2b}\right) \nonumber \\
&\qquad\qquad +d\vartheta \left( \frac{\pa}{\pa \vartheta}
+\tan\left(\vartheta\right)\n^{V}_{Y}+\tan\left(\vartheta\right)\frac{m+n}{2}\right) .  \nonumber 
\nonumber 
\end{align}

By (\ref{eq:co29x1y}),  we get
\begin{align}\label{eq:co29x1}
    &\n^{\mathcal{H}}\mathfrak D^{X}_{b,\vartheta}= -\frac{2db}{b^{2}}\left(b
    \cos\left(\vartheta\right)ic\left(\left[Y^{N},Y^{TX}\right]\right)
+\mathcal{E}^{TX}+\cos\left(\vartheta\right)i\mathcal{E}^{N}\right) \nonumber \\
&-\frac{d\vartheta}{b}\left( b\sin\left(\vartheta\right)
ic\left(\left[Y^{N},Y^{TX}\right]\right) +
\widehat{c}\left(\overline{Y}^{TX}\right)+\sin\left(\vartheta\right)i\mathcal{E}^{N}
\right) ,  \nonumber  \\
&\overline{ \n}^{\mathcal{H}}\overline{\mathfrak D}^{X} _{b,\vartheta}= -\frac{2db}{b^{2}}\left(b
    \cos^{1/2}\left(\vartheta\right)ic\left(\left[Y^{N},Y^{TX}\right]\right)
+\mathcal{E}^{TX}+\cos^{1/2}\left(\vartheta\right)i\mathcal{E}^{N}\right) \nonumber \\
&-\frac{d\vartheta}{b}\left( 
b\frac{\sin\left(\vartheta\right)}{\cos^{1/2}\left(\vartheta\right)}
ic\left(\left[Y^{N},Y^{TX}\right]\right) 
+\widehat{c}\left(\overline{Y}^{TX
}\right)+\frac{\sin\left(\vartheta\right)}{\cos^{1/2}\left(\vartheta\right)}i\mathcal{E}^{N}
\right) ,   \\
&\n^{\mathcal{H} \prime }\mathfrak D^{X 
\prime }_{b,\vartheta}=-\frac{2db}{b^{2}}\left(b
\cos\left(\vartheta\right)ic\left(\left[Y^{N},Y^{TX}\right]\right)
+\mathcal{E}^{TX}_{-\vartheta}+\cos\left(\vartheta\right)i\mathcal{E}^{N}\right) \nonumber \\
&+\frac{d\vartheta}{b}\left(b\sin\left(\vartheta\right)
ic\left(\left[Y^{N},Y^{TX}\right]\right)-
\tan \left(\vartheta\right)\left(\mathcal{D}^{TX}-i\mathcal{D}^{N}\right)
-\frac{1}{\cos\left(\vartheta\right)}\widehat{c}\left(\overline{Y}^{TX}\right)\right). \nonumber 
\end{align}

Note that the superconnections $B,\overline{B},B'$  in (\ref{eq:co31-x}) descend to 
superconnections 
\index{BX@$B^{X}$}%
\index{BX@$\overline{B}^{X}$}%
\index{BX@$B^{X \prime}$}%
$B^{X},\overline{B}^{X},B^{X \prime }$ on $\mathcal{H}$ given by
\begin{align}\label{eq:co31}
&B^{X}=\n^{\mathcal{H}}+\frac{\mathfrak D^{X 
}_{b,\vartheta}}{\sqrt{2}}, \nonumber \\
&\overline{B}^{X}=\overline{\n}^{\mathcal{H}}+\frac{\overline{\mathfrak 
D}^{X}_{\bt}}{\sqrt{2}},\\
&B^{X \prime}=\n^{\mathcal{H} \prime }+\frac{\mathfrak 
D^{X \prime  }_{b,\vartheta}}{\sqrt{2}}. \nonumber 
\end{align}

By (\ref{eq:co31-xa1}), we get
\begin{equation}\label{eq:co31z1}
\overline{B}^{X}=K^{N}_{1/\cos^{1/2}\left(\vartheta\right)}B^{X}K^{N}_{\cos^{1/2}\left(\vartheta\right)}.
\end{equation}
Moreover, we have the identities
\begin{align}\label{eq:co33}
&B^{X,2}=\frac{1}{2}\mathfrak D^{X,2}_{b,\vartheta}+\frac{1}{\sqrt{2}}\n^{\mathcal{H}}\mathfrak 
D^{X}_{b,\vartheta}, \nonumber \\
&\overline{B}^{X,2}=\frac{1}{2}\overline{\mathfrak 
D}^{X,2}_{\bt}+\frac{1}{\sqrt{2}}\overline{\n}^{\mathcal{H}}\overline{\mathfrak D}^{X}_{\bt},\\
&B^{ X\prime, 2}=\frac{1}{2}\mathfrak D^{X\prime, 
2}_{b,\vartheta}+\frac{1}{\sqrt{2}}\n^{\mathcal{H} \prime }\mathfrak 
D^{X \prime }_{b,\vartheta}. \nonumber 
\end{align}

\begin{defin}\label{Dmx}
Put
\index{LX@$L^{X}$}%
\index{LX@$\overline{L}^{X}$}%
\index{LX@$L^{X \prime}$}%
\begin{align}\label{eq:co35x-1}
&L^{X}=-\frac{1}{2}\widehat{D}^{\mathfrak 
g,X,2}_{\vartheta}+B^{X,2}, \nonumber \\
&\overline{L}^{X}=-\frac{1}{2}\widehat{D}^{\mathfrak 
g,X,2}_{\vartheta}+\overline{B}^{X,2},\\
&L^{X \prime}=-\frac{1}{2}\widehat{D}^{\mathfrak 
g,X,2}+B^{X \prime ,2}. \nonumber 
\end{align}
\end{defin}

By (\ref{eq:co10}), (\ref{eq:co31z1}), and   (\ref{eq:co35x-1}), we obtain
\begin{equation}\label{eq:co35x-1a}
\overline{L}^{X}=K^{N}_{1/\cos^{1/2}\left(\vartheta\right)}L^{X}K^{N}_{\cos^{1/2}\left(\vartheta\right)}.
\end{equation}

By  (\ref{eq:co10}), (\ref{eq:co35x-1}), we get
\begin{align}\label{eq:co35}
&L^{X}=B^{X,2}+\frac{1}{2}C^{\mathfrak 
g,X}+\frac{1}{8}B^{*}\left(\kappa^{\mathfrak g},\kappa^{\mathfrak 
g}\right), \nonumber \\
&\overline{L}^{X}=\overline{B}^{X,2}+\frac{1}{2}C^{\mathfrak 
g,X}+\frac{1}{8}B^{*}\left(\kappa^{\mathfrak g},\kappa^{\mathfrak 
g}\right),\\
&L^{X \prime }=B^{X \prime, 2}+\frac{1}{2}C^{\mathfrak 
g,X}+\frac{1}{8}B^{*}\left(\kappa^{\mathfrak g},\kappa^{\mathfrak 
g}\right). \nonumber 
\end{align}
By (\ref{eq:co17}), (\ref{eq:co33}), and (\ref{eq:co35x-1}), we obtain
\begin{align}\label{eq:co36}
&L^{X}=\mathcal{L}^{X}_{b,\vartheta}+\frac{1}{\sqrt{2}}\n^{\mathcal{H}}\mathfrak D^{X}_{b,\vartheta}, \nonumber \\
&\overline{L}^{X}=\overline{\mathcal{L}}^{X}_{b,\vartheta}+\frac{1}{\sqrt{2}}\overline{\n}^{\mathcal{H}}\overline{\mathfrak D}^{X}_{b,\vartheta}, \\
&L^{X \prime }=\mathcal{L}^{X \prime}_{b,\vartheta}+\frac{1}{\sqrt{2}}
\n^{\mathcal{H} \prime }\mathfrak 
D^{X \prime }_{b,\vartheta}. \nonumber 
\end{align}

We will establish an analogue of the Bianchi identities in 
(\ref{eq:cla6}) and in Proposition \ref{PBi}.
\begin{prop}\label{PBibis}
The following identities hold:
\begin{align}\label{eq:co37}
&\left[B^{X},L^{X}\right]=0,
&\left[\overline{B}^{X},\overline{L}^{X}\right]=0, \qquad
\left[B^{X \prime },L^{X \prime }\right]=0.
\end{align}
\end{prop}
\begin{proof}
   The classical Bianchi identity asserts that
\begin{equation}\label{eq:co34}
\left[B^{X},B^{X,2}\right]=0.
\end{equation}
Since $C^{\mathfrak g}$ lies in the centre of $U\left(\mathfrak 
g\right)$, by (\ref{eq:co10}), we get
\begin{equation}\label{eq:bia4}
\left[B^{X},C^{\mathfrak g,X}\right]=0.
\end{equation}
By (\ref{eq:co35}), (\ref{eq:co34}), and (\ref{eq:bia4}), we get  
the first identity in 
(\ref{eq:co37}). The proof of the other identities is similar.
\end{proof}

 Recall that the superconnection
 \index{AX@$A^{X}$}%
   $A^{X}$ was defined in equation (\ref{eq:defa2}).
   \begin{prop}\label{Pnewco}
The following identity holds:
 \begin{equation}\label{eq:cogign5}
P\left(\overline{\n}^{\mathcal{H}}\vert_{db=0}+\frac{1}{\sqrt{2}} \left( \widehat{D}^{\mathfrak 
g,X}_{\vartheta}+\cos^{1/2}\left(\vartheta\right)ic\left(\left[Y^{N},Y^{TX}\right]\right) \right) \right)P
=A^{X}.
\end{equation}
\end{prop}
\begin{proof}
Our proposition follows from Proposition  \ref{Pcompbi}.
\end{proof}

By (\ref{eq:co19x-1}), (\ref{eq:co18x1}),  (\ref{eq:co29x1}), and 
(\ref{eq:co36}), we 
can write $\overline{L}^{X}\vert_{db=0}$ in the form
\begin{equation}\label{eq:gign6}
\overline{L}^{X}\vert_{db=0}=\frac{\alpha_{\vartheta}}{b^{2}}+\frac{\underline{\beta}_{\vartheta}}{b}+\underline{\gamma}_{\vartheta},
\end{equation}
and moreover,
\begin{align}\label{eq:gign7}
&\underline{\beta}_{\vartheta}=\beta_{\vartheta}-\frac{d\vartheta}{\sqrt{2}}
\left( 
\widehat{c}\left(\overline{Y}^{TX}\right)+\frac{\sin\left(\vartheta\right)}{\cos^{1/2}\left(\vartheta\right)}i
\mathcal{E}^{N} \right) ,\\
&\underline{\gamma}_{\vartheta}=\gamma_{\vartheta}-\frac{d\vartheta 
}{\sqrt{2}} \frac{\sin\left(\vartheta\right)}{\cos^{1/2}\left(\vartheta\right)}
ic\left(\left[Y^{N},Y^{TX}\right]\right). \nonumber 
\end{align}
Again, $\underline{\beta}_{\vartheta}$ maps $H$ in 
$H^{\perp}$.

Recall that 
\index{TX@$T^{X}$}%
$T^{X}$ was defined in Definition \ref{DTX} and is given 
by (\ref{eq:defa5x1b}). Now, we extend Theorem \ref{Thfub}.
\begin{thm}\label{Thfubter}
For $\vartheta\in \left[0,\frac{\pi}{2}\right[$, the following 
identity holds:
\begin{equation}\label{eq:co51}
P\left(\underline{\gamma}_{\vartheta}-\underline{\beta}_{\vartheta}
\alpha_{\vartheta}^{-1}\underline{\beta}_{\vartheta} \right) P
=T^{X}.
\end{equation}
\end{thm}
\begin{proof}
    By (\ref{eq:co22}), (\ref{eq:co30x-1a}), and  (\ref{eq:co31}), we get
\begin{equation}\label{eq:48x2}
\overline{B}^{X}\vert_{db=0}=\underline{E}_{\vartheta}+\frac{F_{\vartheta}}{b},
\end{equation}
with
\begin{equation}\label{eq:48x2b}
\underline{E}_{\vartheta}=\overline{\n}^{\mathcal{H}}\vert_{db=0}+E_{\vartheta}.
\end{equation}
Using (\ref{eq:co35x-1}), and comparing (\ref{eq:gign6}) and 
(\ref{eq:48x2}), we get
\begin{align}\label{eq:48x2c}
&\alpha_{\vartheta}=F^{2}_{\vartheta},
\qquad 
\underline{\beta}_{\vartheta}=\left[\underline{E}_{\vartheta},F_{\vartheta}\right],
&\underline{\gamma}_{\vartheta}=\underline{E}_{\vartheta}^{2}-\frac{1}{2}
\widehat{D}^{\mathfrak g,X,2}_{\vartheta}.
\end{align}
By (\ref{eq:48x2c}), we get
\begin{equation}\label{eq:co24b1}
P\left(\underline{\gamma}_{\vartheta}-\underline{\beta}_{\vartheta}\alpha_{\vartheta}^{-1}
\underline{\beta}_{\vartheta}\right)P=P\left(\underline{E}_{\vartheta}^{2}
-\underline{E}_{\vartheta}P^{\perp}_{\vartheta}\underline{E}_{\vartheta}
-\frac{1}{2}\widehat{D}^{\mathfrak g,X,2}_{\vartheta}\right)
P.
\end{equation}
By equation (\ref{eq:cogign5}) in Proposition \ref{Pnewco}, we obtain
\begin{equation}\label{eq:co24b2}
P_{\vartheta}\left(\underline{E}^{2}_{\vartheta}-\underline{E}_{\vartheta}P^{\perp}_{\vartheta}\underline{E}_{\vartheta}\right)
P_{\vartheta}=\left(P_{\vartheta}\underline{E}_{\vartheta}P_{\vartheta}\right)^{2}=A^{X,2}.
\end{equation}
By (\ref{eq:Lie17}), (\ref{eq:kos22}), (\ref{eq:defa4}), 
(\ref{eq:co24b1}), and (\ref{eq:co24b2}), we get (\ref{eq:co51}).
\end{proof}
\subsection{A compression identity on linear maps}%
\label{subsec:ide}
\begin{defin}\label{RtY}
Set
\index{RtY@$R_{\vartheta}\left(Y\right)$}%
\begin{multline}\label{eq:qsic4}
R_{\vartheta}\left(Y\right)=\cos\left(\vartheta\right)\Bigl(
	 \widehat{c}_{\vartheta}\left(\ad\left(Y^{TX}\right)
	 \vert_{TX \oplus N} \right)
  -c\left(\ad\left(Y^{ 
	  TX}\right)
   \right) \Bigr) \\
   -i\cos^{1/2}\left(\vartheta\right)
   \Bigl(c\left(\theta\ad\left(Y^{N}\right)\right)+
   \widehat{c}_{\vartheta}\left(\ad\left(Y^{N}\right)\vert_{\overline{TX}}\right)
   +\rho^{F}\left(Y^{N}\right)\Bigr). 
 \end{multline}
Then $R_{\vartheta}\left(Y\right)$ splits as
\index{RtY@$R_{\vartheta}\left(Y^{TX}\right)$}%
\index{RtY@$R_{\vartheta}\left(Y^{N}\right)$}%
\begin{equation}\label{eq:qsic5}
R_{\vartheta}\left(Y\right)=R_{\vartheta}\left(Y^{TX}\right)+R_{\vartheta}\left(Y^{N}\right).
\end{equation}
Let 
\index{P@$\mathbf{P}$}%
$\mathbf{P}$ be the orthogonal projection from 
$\Lambda\ac\left(\TsX \oplus N^{*}\right) \otimes 
S^{\overline{TX}}\otimes F$ on $S^{\overline{TX}} \otimes F$.
\end{defin}
\begin{prop}\label{Pproj}
The following identity holds:
\begin{align}\label{eq:nea1}
    &\mathbf{P}R_{\vartheta}\left(Y^{TX}\right)\mathbf{P=0},\\
&\mathbf{P}R_{\vartheta}\left(Y^{N}\right)\mathbf{P}=-
\cos^{5/2}\left(\vartheta\right)\widehat{c}\left(i\ad\left(Y^{N}\right)
\vert_{\overline{TX}}\right)-i\cos^{1/2}\left(\vartheta\right)\rho^{F}\left(Y^{N}\right).
\nonumber 
\end{align}
\end{prop}
\begin{proof}
If 
$e_{1},\ldots,e_{m}$ is an orthonormal basis of $TX$, and if 
$e_{m+1},\ldots,e_{m+n}$ is an orthonormal basis of $N$, by 
(\ref{eq:ors1}),  we have the 
identities
\begin{align}\label{eq:nea0}
&c\left(\ad\left(Y^{TX}\right)\right)=-\frac{1}{2}\sum_{\substack{1\le 
i\le m\\ m+1\le j\le m+n}}{}\left\langle  
\left[Y^{TX},e_{i}\right],e_{j}\right\rangle\left(\ei-i_{e_{i}}\right)\left(\ej+i_{e_{j}}\right),\\
&\widehat{c}\left(\ad\left(Y^{TX}\right)\right)=\frac{1}{2}\sum_{\substack{1\le 
i\le m\\ m+1\le j\le m+n}}{}\left\langle  
\left[Y^{TX},e_{i}\right],e_{j}\right\rangle\left(\ei+i_{e_{i}}\right)\left(\ej-i_{e_{j}}\right). \nonumber 
\end{align}
By  (\ref{eq:nea0}), we deduce that
\begin{align}\label{eq:nea-1}
&\mathbf{P}c\left(\ad\left(Y^{TX}\right)\right)\mathbf{P}=0,
&\mathbf{P}\widehat{c}\left(\ad\left(Y^{TX}\right)\right)\mathbf{P}=0.
\end{align}

By (\ref{eq:co6x1}), (\ref{eq:nea0}), we get
\begin{multline}\label{eq:nea0x1}
\widehat{c}_{\vartheta}\left(\ad\left(Y^{TX}\right)\vert_{TX \oplus 
N}\right)=\cos\left(\vartheta\right)\widehat{c}\left(\ad\left(Y^{TX}\right)\right)\\
+\sin\left(\vartheta\right)
\frac{1}{2}\sum_{\substack{1\le 
i\le m\\ m+1\le j\le m+n}}{}\left\langle  
\left[Y^{TX},e_{i}\right],e_{j}\right\rangle\widehat{c}\left(\overline{e}_{i}\right)\left(\ej-i_{e_{j}}\right).
\end{multline}
By (\ref{eq:nea-1}), (\ref{eq:nea0x1}), we get
\begin{equation}\label{eq:nea0x2}
\mathbf{P}\widehat{c}_{\vartheta}\left(\ad\left(Y^{TX}\right)\vert_{TX \oplus 
N}\right)\mathbf{P}=0.
\end{equation}

Moreover, we have the identities
\begin{align}\label{eq:he1}
&c\left(\theta\ad\left(Y^{N}\right)\right)=-\frac{1}{4}\sum_{1\le 
i,j\le m}^{}\left\langle  
\left[Y^{N},e_{i}\right],e_{j}\right\rangle\left(e^{i}-i_{e_{i}}\right)\left(e^{j}-i_{e_{j}}\right) \nonumber \\
&-\frac{1}{4}\sum_{m+1\le i,j\le m+n}^{}\left\langle  
\left[Y^{N},e_{i}\right],e_{j}\right\rangle\left(\ei+i_{e_{i}}\right)\left(\ej+i_{e_{j}}\right),\\
&\widehat{c}\left(\ad\left(Y^{N}\right)\right)=-\frac{1}{4}\sum_{1\le 
i,j\le m}^{}\left\langle  
\left[Y^{N},e_{i}\right],e_{j}\right\rangle\left(\ei+i_{e_{i}}\right)\left(\ej+i_{e_{j}}\right) \nonumber \\
&+\frac{1}{4}\sum_{m+1\le i,j\le m+n}^{}\left\langle  
\left[Y^{N},e_{i}\right],e_{j}\right\rangle\left(\ei-i_{e_{i}}\right)\left(\ej-i_{e_{j}}\right). \nonumber 
\end{align}
Since $\ad\left(Y^{N}\right),\theta\ad\left(Y^{N}\right)$ are antisymmetric, we get easily
\begin{align}\label{eq:nea0x2b}
    &\mathbf{P}c\left(\theta\ad\left(Y^{N}\right)\vert_{TX}\right)\mathbf{P}=0,
&\mathbf{P}\widehat{c}\left(\ad\left(Y^{N}\right)\right)\mathbf{P}=0.
\end{align}
Also we have the identity
\begin{equation}\label{eq:nea0x3}
\widehat{c}\left(\ad\left(Y^{N}\right)\vert_{\overline{TX}}\right)=-\frac{1}{4}\sum_{1\le i,j\le m}^{}
\left\langle  
\left[Y^{N},e_{i}\right],e_{j}\right\rangle_{TX}\widehat{c}\left(\overline{e}_{i}\right)\widehat{c}\left(\overline{e}_{j}\right).
\end{equation}
By (\ref{eq:nea0x3}), we deduce that
\begin{multline}\label{eq:nea0x4}
\widehat{c}_{\vartheta}\left(\ad\left(Y^{N}\right)\vert_{\overline{TX}}\right)=\cos^{2}\left(\vartheta\right)
\widehat{c}\left(\ad\left(Y^{N}\right)\vert_{\overline{TX}}\right)+\sin^{2}\left(\vartheta\right)\widehat{c}\left(\ad\left(Y^{N}\right)
\vert_{TX}\right)\\
+\frac{1}{2}\sin\left(\vartheta\right)\cos\left(\vartheta\right)\sum_{1\le i,j\le m}^{}\left\langle  
\left[Y^{N},e_{i}\right],e_{j}\right\rangle_{TX}\left(\ei+i_{e_{i}}\right)\widehat{c}\left(\overline{e}_{j}\right).
\end{multline}
By (\ref{eq:nea0x2b}), (\ref{eq:nea0x4}), we get
\begin{equation}\label{eq:nea0x5}
\mathbf{P}\widehat{c}_{\vartheta}\left(\ad\left(Y^{N}\right)\vert_{\overline{TX}}\right)\mathbf{P}=
\cos^{2}\left(\vartheta\right)\widehat{c}\left(\ad\left(Y^{N}\right)\vert_{\overline{TX}}\right).
\end{equation}

By 
(\ref{eq:qsic4}), (\ref{eq:nea-1}), (\ref{eq:nea0x2}), 
(\ref{eq:nea0x2b}),  and 
(\ref{eq:nea0x5}),
we get (\ref{eq:nea1}).
\end{proof}

Now we extend \cite[Definition 2.16.2]{Bismut08b}.
\begin{defin}\label{Dsten}
For $0\le \vartheta<\frac{\pi}{2}$, set
\index{StY@$S_{\vartheta}\left(Y\right)$}%
\begin{multline}\label{eq:qsic6}
S_{\vartheta}\left(Y\right)=\mathbf{P}\Biggl[R_{\vartheta}\left(Y^{TX}\right)
\left(1+N_{-\vartheta}^{\Lambda\ac\left(\TsX 
 \oplus N^{*}\right)}\right)^{-1}R_{\vartheta}\left(Y^{TX}\right)\\
+R_{\vartheta}\left(Y^{N}\right)\left(\cos\left(\vartheta\right)
+N_{-\vartheta}^{\Lambda\ac\left(\TsX 
\oplus N^{*}\right)}\right)
^{-1}R_{\vartheta}\left(Y^{N}\right)\Biggr]\mathbf{P}.
\end{multline}
Then $S_{\vartheta}\left(Y\right)$ lies in 
$\End\left(S^{\overline{TX}}\otimes F\right)$.
In the sequel, we use the notation
\begin{equation}\label{eq:qsic7x-1}
x=\cos\left(\vartheta\right).
\end{equation}
\end{defin}

We give an extension of \cite[Proposition 2.16.3]{Bismut08b}.
\begin{prop}\label{Pidnew}
The following identity holds:
\begin{multline}\label{eq:qsic7}
S_{\vartheta}\left(Y\right)=
\frac{x^{2}}{4}\frac{x+3}{x+2}\Tr\left[\ad^{2}\left(Y^{TX}\right)\vert_{N}\right]\mathbf{P}\\
+\frac{1}{4}\left(\frac{x}{2}\frac{\left(x^{2}-2\right)^{2}}{x+2}-
x^{3}\left(x-1\right)\right)\Tr
\left[-\ad^{2}\left(Y^{N}\right)\vert_{TX}\right]\mathbf{P}\\
+\frac{1}{24}\Tr
\left[-\ad^{2}\left(Y^{N}\right)\vert_{N}\right]\mathbf{P}
-\left(\rho^{F}\left(Y^{N}\right)+x^{2}
\widehat{c}\left(\ad\left(Y^{N}\right)\vert_{\overline{TX}}\right)\right)^{2}\mathbf{P}.
\end{multline}
\end{prop}
\begin{proof}
By (\ref{eq:gip1}), 
(\ref{eq:qsic4}), (\ref{eq:nea0}), and (\ref{eq:nea0x1}), we get
\begin{multline}\label{eq:qsic8}
\mathbf{P}R_{\vartheta}\left(Y^{TX}\right)\left(1+
N_{-\vartheta}^{\Lambda\ac\left(\TsX 
 \oplus N^{*}\right)}
\right)^{-1}
R_{\vartheta}\left(Y^{TX}\right)\mathbf{P}\\
=\frac{x^{2}}{4}\sum_{\substack{1\le i,i'\le m\\
m+1\le j,j'\le m+n}}^{}\left\langle  
\left[Y^{TX},e_{i}\right],e_{j}\right\rangle
\left\langle  \left[Y^{TX},e_{i'}\right],e_{j'}\right\rangle\\
\mathbf{P}\left(-\frac{\left(x+1\right)^{2}}{x+2}i_{e_{i'}}i_{e_{j'}}
\ei\ej-\frac{-x^{2}+1}{x+1}\widehat{c}\left(\overline{e}_{i'}\right)
i_{e_{j'}}\widehat{c}\left(\overline{e}_{i}\right)\ej\right)\mathbf{P}.
\end{multline}
Equivalently,
\begin{multline}\label{eq:qsic9}
\mathbf{P}R_{\vartheta}\left(Y^{TX}\right)
\left(1+
N_{-\vartheta}^{\Lambda\ac\left(\TsX 
 \oplus N^{*}\right)}
\right)^{-1}
R_{\vartheta}\left(Y^{TX}\right)\mathbf{P}\\
=\frac{x^{2}}{4}\frac{x+3}{x+2}\Tr\left[\ad^{2}\left(Y^{TX}\right)\vert_{N}\right]\mathbf{P}.
\end{multline}
The variables 
$ 
\widehat{c}\left(\overline{e}_{i'}\right),\widehat{c}\left(\overline{e}_{i}\right)$ in the right-hand side of (\ref{eq:qsic8}) 
disappear in (\ref{eq:qsic9}) because in (\ref{eq:qsic8}), a nonzero contribution is only possible 
if $j=j'$. Ultimately only the case where $i=i'$ contributes to the 
right-hand side of (\ref{eq:qsic9}), which explains why no Clifford 
variable $\widehat{c}$ appears.

The same arguments as before combined with (\ref{eq:he1}), 
(\ref{eq:nea0x4})  show that
\begin{multline}\label{eq:qsic9x1}
\mathbf{P}R_{\vartheta}\left(Y^{N}\right)\left(\cos\left(\vartheta\right)
+N_{-\vartheta}^{\Lambda\ac\left(\TsX 
 \oplus N^{*}\right)}\right)^{-1}
R_{\vartheta}\left(Y^{N}\right)\mathbf{P}\\
=\frac{1}{4}\left( 
\frac{x}{2}\frac{\left(x^{2}-2\right)^{2}}{x+2}-
\frac{x^{3}\left(x^{2}-1\right)}{x+1} \right) \sum_{1\le i,j\le m}^{}
\left\langle  
\left[Y^{N},e_{i}\right],e_{j}\right\rangle^{2}\mathbf{P}\\
+\frac{1}{24}\sum_{m+1\le i, j\le m+n}^{}
\left\langle  \left[Y^{N},e_{i}\right],e_{j}\right\rangle^{2}-
\left(\rho^{F}\left(Y^{N}\right)+x^{2}\widehat{c}\left(\ad\left(Y^{N}\right)\vert_{\overline{TX}}\right)\right)^{2}\mathbf{P}.
\end{multline}
Equivalently,
\begin{multline}\label{eq:qsic9x2}
\mathbf{P}R_{\vartheta}\left(Y^{N}\right)\left(\cos\left(\vartheta\right)
+N^{\Lambda\ac\left(T^{*}X \oplus N^{*}\right)}_{-\vartheta}\right)^{-1}
R_{\vartheta}\left(Y^{N}\right)\mathbf{P}\\
=\frac{1}{4}\left( 
\frac{x}{2}\frac{\left(x^{2}-2\right)^{2}}{x+2}-x^{3}\left(x-1\right) \right) \Tr
\left[-\ad^{2}\left(Y^{N}\right)\vert_{TX}\right]\mathbf{P}\\
+\frac{1}{24}\Tr\left[-\ad^{2}\left(Y^{N}\right)\vert_{N}\right]-
\left(\rho^{F}\left(Y^{N}\right)+x^{2}\widehat{c}\left(\ad\left(Y^{N}\right)\vert_{\overline{TX}}\right)
\right)^{2}\mathbf{P}.
\end{multline}
 By (\ref{eq:qsic9}), (\ref{eq:qsic9x2}), we get (\ref{eq:qsic7}). 
 The proof of our proposition  is completed. 
\end{proof}
\begin{remk}\label{Rsing}
By Proposition \ref{Pidnew}, $S_{\vartheta}\left(Y\right)$ extends by 
continuity at $\vartheta=\frac{\pi}{2}$. 
\end{remk}

Recall that the projector 
\index{P@$P$}%
$P$ was defined in subsections 
\ref{subsec:comp} and \ref{subsec:formre}.

\begin{defin}\label{Defdel}
Put
\index{dt@$\delta_{\vartheta}$}%
\begin{equation}\label{eq:qsic10}
\delta_{\vartheta}=PS_{\vartheta}\left(Y\right)P.
\end{equation}
Then $\delta_{\vartheta}$ lies in $\End\left(S^{\overline{TX}} 
\otimes F\right)$.
\end{defin}

Now we extend \cite[Proposition 2.16.5]{Bismut08b}. The 
absence of $S^{\overline{TX}}$ in \cite{Bismut08b} explains 
the fact that the results are different.
\begin{prop}\label{Pdelta}
The following identity holds:
\begin{multline}\label{eq:qsic11}
\delta_{\vartheta}=-\frac{1}{8}x\left(x+1\right)\Tr^{\mathfrak 
p}\left[C^{\mathfrak k, \mathfrak 
p}\right]-\frac{1}{48}\Tr^{\mathfrak k}\left[C^{\mathfrak k,\mathfrak 
k}\right]-\frac{1}{2}C^{\mathfrak 
k,E} \\
-x^{2}\sum_{i=m+1}^{m+n}\widehat{c}\left(\ad\left(e_{i}\right)\vert 
_{\overline{TX}}\right)
\rho^{F}\left(e_{i}\right).
\end{multline}
\end{prop}
\begin{proof}
By \cite[eq. (2.16.23)]{Bismut08b} or by (\ref{eq:ele1}),  if $u\in \mathfrak g$, we get
\begin{equation}\label{eq:qsic12}
P\left\langle  u,Y\right\rangle^{2}P=\frac{1}{2}\left\vert  
u\right\vert^{2}.
\end{equation}
By (\ref{eq:qsic12}), we deduce that
\begin{align}\label{eq:qsic13}
&P\Tr\left[\ad^{2}\left(Y^{TX}\right)\vert_{N}\right]P=-\frac{1}{2}\Tr^{\mathfrak p}
\left[C^{\mathfrak k, \mathfrak p}\right], \nonumber \\
&P\Tr\left[-\ad^{2}\left(Y^{N}\right)\vert_{TX}\right]P=-
\frac{1}{2}\Tr^{\mathfrak p}\left[C^{\mathfrak k,\mathfrak 
p}\right],\\
&P\Tr\left[-\ad^{2}\left(Y^{N}\right)\vert_{N}\right]P=-\frac{1}{2}\Tr^{\mathfrak k}\left[C^{\mathfrak k, 
\mathfrak k}\right],\nonumber \\
&P\left(-\left(\rho^{F}\left(Y^{N}\right)+x^{2}
\widehat{c}\left(\ad\left(Y^{N}\right)\vert_{\overline{TX}}\right)\right)^{2}\right)P 
\nonumber \\= 
&-\frac{1}{2}C^{\mathfrak k,E}-\frac{x^{4}}{2}C^{\mathfrak 
k,S^{\overline{\mathfrak p}}}-x^{2}\sum_{i=m+1}^{m+n}
\widehat{c}\left(\ad\left(e_{i}\right)\vert_{\overline{TX}}\right)\rho^{F}\left(e_{i}\right). 
\nonumber 
\end{align}
By (\ref{eq:qsic7}), (\ref{eq:qsic10}), and (\ref{eq:qsic13}), we get
\begin{multline}\label{eq:qsic15}
\delta_{\vartheta}=-\frac{1}{8}\left(\frac{2x^{2}\left(x+3\right)+x\left(x^{2}-2\right)
^{2}}{2\left(x+2\right)}-x^{3}\left(x-1\right)\right)\Tr^{\mathfrak p}
\left[C^{\mathfrak k,\mathfrak p}\right]-\frac{1}{48}\Tr^{\mathfrak 
k}\left[C^{\mathfrak k,\mathfrak k}\right]\\
-\left( \frac{1}{2}C^{\mathfrak k,E}+\frac{x^{4}}{2}C^{\mathfrak 
k,S^{\overline{\mathfrak 
p}}}+x^{2}\sum_{i=m+1}^{m+n}\widehat{c}\left(\ad\left(e_{i}\right)\vert_{\overline{TX}}\right)
\rho^{F}\left(e_{i}\right) \right) .
\end{multline}
Also
\begin{equation}\label{eq:qsic16}
\frac{2x^{2}\left(x+3\right)+x\left(x^{2}-2\right)
^{2}}{2\left(x+2\right)}-x^{3}\left(x-1\right)=\frac{x}{2}\left(-x^{3}+2x+2\right).
\end{equation}
By (\ref{eq:Lie11x1}),  (\ref{eq:qsic15}), and (\ref{eq:qsic16}), we get (\ref{eq:qsic11}).
\end{proof}
\begin{remk}\label{Rcons2a}
If we make the modifications suggested in Remarks \ref{Rcons} and 
\ref{Rcons1},  in equation (\ref{eq:qsic11}), $x\left(x+1\right)$ 
should be replaced by $x^{2}+1$.
\end{remk}

\begin{defin}\label{Dst}
Let 
\index{R0Y@$R^{0}_{\vartheta}\left(Y\right)$}%
\index{S0Y@$S^{0}_{\vartheta}\left(Y\right)$}%
$R^{0}_{\vartheta}\left(Y\right), 
S^{0}_{\vartheta}\left(Y\right)$ be 
$R_{\vartheta}\left(Y\right),S_{\vartheta}\left(Y\right)$  when 
$\rho^{E}$ is the trivial representation. 
For $0\le\vartheta<\frac{\pi}{2}$, set
\index{T0Y@$T^{0}_{\vartheta}\left(Y\right)$}%
\begin{multline}\label{eq:psic27}
T^{0}_{\vartheta}\left(Y\right)=\mathbf{P}R^{0}_{\vartheta}\left(Y\right)
\Biggl[\left(1+N_{-\vartheta}^{\Lambda\ac\left(T^{*}X \oplus 
N^{*}\right)}\right)^{-1}R^{0}_{\vartheta}\left(Y^{TX}\right)\\
+\left(\cos\left(\vartheta\right)+N_{-\vartheta}^{\Lambda\ac\left(T^{*}X \oplus N^{*}\right)}\right)
^{-1}R^{0}_{\vartheta}\left(Y^{N}\right)\Biggr]
\mathbf{P}.
\end{multline}
\end{defin}

Now we will establish an auxiliary identity that will be needed in 
section \ref{sec:fin}, in the proof of Theorem \ref{TlimV}.
\begin{prop}\label{Peqel}
The following identity holds:
\begin{equation}\label{eq:psica1}
T^{0}_{\vartheta}\left(Y\right)=S^{0}_{\vartheta}\left(Y\right).
\end{equation}
\end{prop}
\begin{proof}
By (\ref{eq:qsic6}),  it is enough to show that in 
(\ref{eq:psic27}), the contribution of the bilinear terms in 
$Y^{TX},Y^{N}$ vanishes identically. By (\ref{eq:qsic4}), 
(\ref{eq:nea0}), (\ref{eq:nea0x1}), (\ref{eq:he1}), and 
(\ref{eq:nea0x4}),  we get
\begin{align}\label{eq:nid1}
&R_{\vartheta}^{0}\left(Y^{TX}\right)\mathbf{P}=\frac{\cos\left(\vartheta\right)}{2}\sum_{\substack{1\le i\le m \\
m+1\le j\le m+n}}^{}\left\langle  
\left[Y^{TX},e_{i}\right],e_{j}\right\rangle \nonumber \\
&\left( 
\left(\cos\left(\vartheta\right)\ei\we+\sin\left(\vartheta\right)\widehat{c}\left(\overline{e}_{i}\right)\right)\ej
+\ei\we\ej  \right) \mathbf{P},\\
&\mathbf{P}R_{\vartheta}^{0}\left(Y^{N}\right)=i\mathbf{P}\frac{\cos^{1/2}\left(\vartheta\right)}{4}
\Biggl(\sum_{1\le i,j\le m}^{}\left\langle  
\left[Y^{N},e_{i}\right],e_{j}\right\rangle 
i_{e_{i}}i_{e_{j}}  \nonumber \\
&+\sum_{m+1\le i,j\le 
m+n}^{}\left\langle  
\left[Y^{N},e_{i}\right],e_{j}\right\rangle 
i_{e_{i}}i_{e_{j}} +\sum_{1\le i,j\le m}^{}\left\langle  \left[Y^{N},e_{i}\right],e_{j}\right\rangle \nonumber \\
&
\Bigl(-\sin\left(\vartheta\right)i_{e_{i}}+\cos\left(\vartheta\right)\widehat{c}\left(\overline{e}_{i}\right)\Bigr)
\Bigl(-\sin\left(\vartheta\right)i_{e_{j}}+\cos\left(\vartheta\right)\widehat{c}\left(\overline{e}_{j}\right)\Bigr)
\Biggr). \nonumber 
\end{align}
By (\ref{eq:nid1}), we deduce easily that
\begin{equation}\label{eq:nid2}
\mathbf{P}R^{0}_{\vartheta}\left(Y^{N}\right)\left(1+N^{\Lambda\ac\left(T^{*}X \oplus N^{*}\right)}
_{-\vartheta}\right)^{-1}R^{0}_{\vartheta}\left(Y^{TX}\right)
\mathbf{P}=0.
\end{equation}
Similarly, we have the identities
\begin{align}\label{eq:nid3}
&\mathbf{P}R^{0}_{\vartheta}\left(Y^{TX}\right)=-\mathbf{P}\frac{\cos\left(\vartheta\right)}{2}
\sum_{\substack{1\le i\le m\\
m+1\le j\le m+n}}^{}\Biggl( \left\langle  
\left[Y^{TX},e_{i}\right],e_{j}\right\rangle  \nonumber \\
&\left( \left(\cos\left(\vartheta\right)i_{e_{i}}+\sin\left(\vartheta\right)\widehat{c}\left(\overline{e}_{i}\right)
\right)i_{e_{j}}+i_{e_{i}}i_{e_{j}} \right) , \nonumber \\
&R^{0}_{\vartheta}\left(Y^{N}\right)\mathbf{P}=-i\frac{\cos^{1/2}\left(\vartheta\right)}{4}
\Biggl(-\sum_{1\le i,j\le m}^{}\left\langle  
\left[Y^{N},e_{j}\right],e_{j}\right\rangle\ei\we\ej \\ 
&-\sum_{m+1\le i,j\le 
m+n}^{}\left\langle  \left[Y^{N},e_{i}\right],e_{j}\right\rangle\ei\we 
\ej  -\sum_{1\le i,j\le m}^{}\left\langle  \left[Y^{N},e_{i}\right],e_{j}\right\rangle\nonumber \\
&
\left(-\sin\left(\vartheta\right)\ei+\cos\left(\vartheta\right)\widehat{c}\left(\overline{e}_{i}\right)\right)
\left(-\sin\left(\vartheta\right)\ej+\cos\left(\vartheta\right)\widehat{c}\left(\overline{e}_{j}\right)\right)
\Biggr)\mathbf{P}. \nonumber 
\end{align}
By (\ref{eq:nid3}), we also deduce that
\begin{equation}\label{eq:nid4}
\mathbf{P}R^{0}_{\vartheta}\left(Y^{TX}\right)\left(\cos\left(\vartheta\right)+N^{\Lambda\ac\left(T^{*}X \oplus N^{*}\right)
}_{-\vartheta}\right)^{-1}R^{0}_{\vartheta}
\left(Y^{N}\right)\mathbf{P}=0.
\end{equation}
By (\ref{eq:nid2}), (\ref{eq:nid4}), we get (\ref{eq:psica1}).
\end{proof}
\begin{defin}\label{DSbis}
	Put
\index{StY@$\overline{S}^{0}_{\vartheta}\left(Y\right)$}%
\begin{equation}\label{eq:psic27y1}
\overline{S}^{0}_{\vartheta}\left(Y\right)=S^{0}_{\vartheta}\left(Y\right)+
\cos^{4}\left(\vartheta\right)\widehat{c}\left(\ad\left(Y^{N}\right)\vert_{\overline{TX}}
\right)^{2}.
\end{equation}
\end{defin}
By equation (\ref{eq:qsic7}) in Proposition \ref{Pidnew}, 
$\overline{S}^{0}_{\vartheta}\left(Y\right)$ is a scalar operator, 
that depends quadratically on $Y$.

\begin{defin}\label{Dbon1}
	Let 
\index{dt@$\delta^{0}_{\vartheta}$}%
$\delta^{0}_{\vartheta}$ be the constant 
\index{dt@$\delta_{\vartheta}$}%
$\delta_{\vartheta}$ in (\ref{eq:qsic11}) 
with $E$ the trivial representation. 
Set
\index{d0t@$\pmb \delta^{0}_{\vartheta}$}%
\begin{equation}\label{eq:mir-1}
\pmb \delta^{0}_{\vartheta}=\delta^{0}_{\vartheta}+\frac{\cos^{4}\left(\vartheta\right)}{16}
\Tr^{\mathfrak p}\left[C^{\mathfrak k, \mathfrak p}\right].
\end{equation}
\end{defin}

By (\ref{eq:qsic11}), (\ref{eq:mir-1}), if $x=\ct$, we get
\begin{equation}\label{eq:rob3}
\pmb\delta_{\vartheta}^{0}=\frac{1}{16}\left(x^{4}-2x^{2}-2x\right)\Tr^{\mathfrak p}\left[C^{\mathfrak k, \mathfrak p}\right]
-\frac{1}{48}\Tr^{\mathfrak k}\left[C^{\mathfrak k, \mathfrak 
k}\right].
\end{equation}
\begin{prop}\label{Pidnez}
	The following identity holds:
	\begin{equation}\label{eq:mir9a}
P\overline{S}^{0}_{\vartheta}\left(Y\right)P=\pmb 
\delta^{0}_{\vartheta}.
\end{equation}
\end{prop}
\begin{proof}
	Using (\ref{eq:Lie11x1}),   Proposition \ref{Pdelta}, the last 
equation in (\ref{eq:qsic13}), (\ref{eq:psic27y1}), and 
(\ref{eq:mir-1}), we get (\ref{eq:mir9a}).
\end{proof}

\subsection{A computational proof of Theorems \ref{Thfub} and \ref{Thfubter}}%
\label{subsec:anopr}
First,  we give another proof of  Theorem \ref{Thfub}.

By \cite[eq. (2.16.25)]{Bismut08b}, we get
\begin{equation}\label{eq:qsic19}
P\frac{1}{2}\left\vert  \left[Y^{N},Y^{TX}\right]\right\vert^{2}P=-
\frac{1}{8}\Tr^{\mathfrak p}\left[C^{\mathfrak k,\mathfrak p}\right].
\end{equation}
By the second equation in (\ref{eq:co19x-1}), by (\ref{eq:co18x1}), 
   (\ref{eq:qsic6}), 
(\ref{eq:qsic10}), (\ref{eq:qsic12}),  and (\ref{eq:qsic19}), we obtain
\begin{equation}\label{eq:qsic20}
P\left(\gamma_{\vartheta}-\beta_{\vartheta}\alpha_{\vartheta}^{-1}
\beta_{\vartheta}\right)P=
-\frac{x^{2}}{2}\Delta^{X,H}
-\delta_{\vartheta}-\frac{x}{8}\Tr^{\mathfrak p}\left[C^{\mathfrak k,\mathfrak p}\right].
\end{equation}
By (\ref{eq:qsic11}), we get
\begin{multline}\label{eq:qsic21}
-\delta_{\vartheta}-\frac{x}{8}\Tr^{\mathfrak p}\left[C^{\mathfrak k,\mathfrak p}\right]
=\frac{x^{2}}{8}\Tr^{\mathfrak k}\left[C^{\mathfrak k, \mathfrak 
p}\right]+\frac{1}{48}\Tr^{\mathfrak k}\left[C^{\mathfrak k, 
\mathfrak k}\right]\\
+\frac{1}{2}C^{\mathfrak 
k,E}+x^{2}\sum_{i=m+1}^{m+n}
\widehat{c}\left(\ad\left(e_{i}\right)\vert_{\overline{TX}}\right)\rho^{F}\left(e_{i}\right).
\end{multline}
By (\ref{eq:qsic20}), (\ref{eq:qsic21}), we obtain
\begin{multline}\label{eq:qsic22}
P\left(\gamma_{\vartheta}-\beta_{\vartheta}\alpha_{\vartheta}^{-1}
\beta_{\vartheta}\right)P=
\frac{x^{2}}{2}\Biggl( -\Delta^{X,H}
+\frac{1}{4}\Tr^{\mathfrak p}\left[C^{\mathfrak k, \mathfrak p}\right] \\
+2
\sum_{i=m+1}^{m+n}\widehat{c}\left(\ad\left(e_{i}\right)\vert_{\overline{TX}}
\right)\rho^{F}\left(e_{i}\right)\Biggr)+\frac{1}{48}\Tr^{\mathfrak k}
\left[C^{\mathfrak k, \mathfrak k}\right]+\frac{1}{2}C^{\mathfrak 
k,E}.
\end{multline}
By (\ref{eq:japo4}),  (\ref{eq:qsic22}), we 
get (\ref{eq:co21}). This completes the second proof of Theorem 
\ref{Thfub}.

Now, we give another proof of Theorem \ref{Thfubter}.
By Theorem \ref{Thfub}, we only need to show the coincidence of 
the $d\vartheta$ components in (\ref{eq:co51}).  We use 
(\ref{eq:gign7}). We have the trivial identity 
\begin{equation}\label{eq:qsic29}
Pc\left(\left[Y^{N},Y^{TX}\right]\right)P=0.
\end{equation}
By (\ref{eq:defa5x1b}), (\ref{eq:gign7}),
to establish (\ref{eq:co51}), we only to show that
\begin{multline}\label{eq:qsic30}
P\beta_{\vartheta}\alpha^{-1}_{\vartheta}\frac{d\vartheta}{\sqrt{2}}
\left(\widehat{c}\left(\overline{Y}^{TX}\right)+
\frac{\sin\left(\vartheta\right)}{\cos^{1/2}\left(\vartheta\right)}i\mathcal{E}^{N}\right)P\\
+P\frac{d\vartheta}{\sqrt{2}}\left(\widehat{c}\left(\overline{Y}^{TX}\right)+
\frac{\sin\left(\vartheta\right)}{\cos^{1/2}\left(\vartheta\right)}i\mathcal{E}^{N}\right)
\alpha^{-1}_{\vartheta}\beta_{\vartheta}P=
\frac{d\vartheta}{\sqrt{2}}\cos\left(\vartheta\right)\widehat{D}^{X}.
\end{multline}
In 
(\ref{eq:qsic30}),  using (\ref{eq:co19x-1}), (\ref{eq:nid1}), and 
(\ref{eq:nid3}), 
an easy computation shows that we can replace $\beta_{\vartheta}$ by 
$$\cos\left(\vartheta\right)\n_{Y^{ TX}}^{C^{ \infty }\left(TX \oplus 
 N,\widehat{\pi}^{*} \left( \Lambda\ac\left(T^{*}X \oplus N^{*}\right)\otimes 
 S^{\overline{TX}} \otimes F 
 \right) \right)},$$
 which combined with (\ref{eq:qsic12}) gives  (\ref{eq:qsic30}). This 
 completes the second proof of Theorem \ref{Thfubter}.
\subsection{The scaling of the invariant form $B$}%
\label{subsec:scal}
Given $t>0$, we denote with an extra index $t$ the objects considered 
above that are associated with the form $B/t$ over $\mathfrak g$.
We will establish an analogue of the results in \cite[section 
2.14]{Bismut08b}.  By (\ref{eq:vb4a}), (\ref{eq:co30x-1a}), as in 
\cite[eq. (2.14.3)]{Bismut08b}, we get
\begin{align}\label{eq:he-1}
&K_{\sqrt{t}}t^{N^{\Lambda\ac\left(T^{*}X \oplus 
N^{*}\right)}/2}\mathfrak 
D^{X}_{b,\vartheta,t}t^{-N^{\Lambda\ac\left(T^{*}X \oplus 
N^{*}\right)}/2}K^{-1}_{\sqrt{t}}=\sqrt{t}\mathfrak 
D^{X}_{\sqrt{t}b,\vartheta}, \nonumber \\
&K_{\sqrt{t}}t^{N^{\Lambda\ac\left(T^{*}X \oplus 
N^{*}\right)}/2}\mathfrak 
D^{X \prime }_{b,\vartheta,t}t^{-N^{\Lambda\ac\left(T^{*}X \oplus 
N^{*}\right)}/2}K^{-1}_{\sqrt{t}}=\sqrt{t}\mathfrak 
D^{X \prime}_{\sqrt{t}b,\vartheta},\\
&K_{\sqrt{t}}t^{N^{\Lambda\ac\left(T^{*}X \oplus N^{*}\right)}/2}
\n^{\mathcal{H}}
   t^{-N^{\Lambda\ac\left(T^{*}X \oplus 
   N^{*}\right)}/2}K^{-1}_{\sqrt{t}}=\n^{\mathcal{H}}, \nonumber \\
   &K_{\sqrt{t}}t^{N^{\Lambda\ac\left(T^{*}X \oplus N^{*}\right)}/2}
\n^{\mathcal{H}\prime }
   t^{-N^{\Lambda\ac\left(T^{*}X \oplus 
   N^{*}\right)}/2}K^{-1}_{\sqrt{t}}=\n^{\mathcal{H} \prime}. \nonumber 
\end{align}
By (\ref{eq:he-1}), we deduce that
\begin{align}\label{eq:he-2}
&K_{\sqrt{t}}t^{N^{\Lambda\ac\left(T^{*}X \oplus N^{*}\right)}/2}
\mathcal{L}^{X}_{b,\vartheta,t}
   t^{-N^{\Lambda\ac\left(T^{*}X \oplus 
   N^{*}\right)}/2}K^{-1}_{\sqrt{t}}=t\mathcal{L}^{X}_{\sqrt{t}b,\vartheta},\\
&K_{\sqrt{t}}t^{N^{\Lambda\ac\left(T^{*}X \oplus N^{*}\right)}/2}
\mathcal{L}^{X \prime}_{b,\vartheta,t}
   t^{-N^{\Lambda\ac\left(T^{*}X \oplus 
   N^{*}\right)}/2}K^{-1}_{\sqrt{t}}=t\mathcal{L}^{X}_{\sqrt{t}b,\vartheta}. \nonumber    
\end{align}
The above identities remain valid when replacing $\mathfrak 
D^{X}_{\bt},\n^{\mathcal{H}}$ by $\overline{\mathfrak D}^{X}_{\bt}, 
\overline{\n}^{\mathcal{H}}$.

%% file: Eta4.tex
\section{A closed $1$-form on $\R_{+}^{*}\times 
\left[0,\frac{\pi}{2}\right[$}%
\label{sec:defhyp}
In this section, when $m=\dim \mathfrak p$ is odd, using hypoelliptic orbital integrals, we define a 
closed $1$-form $\mathsf{b}$ on $\R^{*}_{+}\times 
\left[0,\frac{\pi}{2}\right[$. This $1$-form will play an important 
role in establishing our main result.

This section is organized as follows. In subsection 
\ref{subsec:trasup}, we define in our context the proper  traces and 
supertraces. 

In subsection \ref{subsec:oddorb},  if $\gamma\in G$ is semisimple, we introduce 
 orbital integrals  associated with the family of   elliptic operators 
$\mathcal{L}^{X}_{\vartheta}\vert_{\vartheta\in\left[0,\frac{\pi}{2}\right[}$ considered in subsection \ref{subsec:dedx}.
We show that we can 
limit ourselves to the  case where $\gamma$ is nonelliptic.
 Also we introduce a $1$-form $\mathsf{a}$ on 
$\left[0,\frac{\pi}{2}\right[$.

Finally, in subsection \ref{subsec:onfob}, if $\gamma\in G$ is semisimple, we introduce  
 orbital integrals associated with the family of Dirac 
operators $\mathcal{L}^{X}_{\bt}\vert_{\left(\bt\right)\in 
\R_{+}^{*}\times \left[0,\frac{\pi}{2}\right[}$ considered in 
subsection \ref{subsec:des}, and we construct the closed $1$-form 
$\mathsf{b}$ on $\R^{*}_{+}\times \left[0,\frac{\pi}{2}\right[$.

In this section, we make the same assumptions as in sections \ref{sec:eta}
and \ref{sec:defdx}, and we use the corresponding notation. In 
particular, $G$ is still assumed to be simply connected.
\subsection{Traces and supertraces}%
\label{subsec:trasup}
If $m=\dim \mathfrak p$ is even, $S^{\overline{\mathfrak p}}$ is $\Z_{2}$-graded, and 
$\widehat{c}\left(\overline{\mathfrak p}\right)\otimes 
_{\R}\C=\End\left(S^{\overline{\mathfrak p}}\right)$ is  equipped with 
a  corresponding supertrace. When combined with the usual trace on 
$\End\left(E\right)$, we get a supertrace 
\index{Trs@$\Trs$}%
$\Trs:\Lambda\ac\left(\R^{2* 
}\right)\ho\widehat{c}\left(\overline{\mathfrak p}\right) \otimes 
\End\left(E\right)\to \Lambda\ac\left(\R^{2*}\right) \otimes _{\R}\C$, 
with the  convention  that if 
$\eta\in\Lambda\ac\left(\R^{2*}\right),a\in \Lambda\ac\left(\R^{2* 
}\right)\ho\widehat{c}\left(\overline{\mathfrak p}\right) \otimes 
\End\left(E\right)$, 
\begin{equation}\label{eq:sixt1}
\Trs\left[\eta a\right]=\eta\Trs\left[a\right].
\end{equation}
Since $\Lambda\ac\left(\mathfrak g^{*}\right)$ is $\Z_{2}$-graded, 
$\End\left(\Lambda\ac\left(\mathfrak 
g^{*}\right)\right)=c\left(\mathfrak 
g\right)\ho\widehat{c}\left(\mathfrak g\right)$ is equipped with a 
supertrace $\Trs$, so that
 $\Lambda\ac\left(\R^{2*}\right)\ho 
\End\left(\Lambda\ac\left(\mathfrak g^{*}\right)\right) 
\ho\widehat{c}\left(\overline{\mathfrak p}\right)\otimes 
_{\R}\End\left(E\right)$ is equipped with a supertrace $\Trs$ with 
values in $\Lambda\ac\left(\R^{2*}\right) \otimes _{\R}\C$, that 
 vanishes on supercommutators.
 When quotienting by $K$, it descends to a supertrace $\Trs$ from 
 $$\Lambda\ac\left(\R^{2*}\right)\ho\End\left(\Lambda\ac\left(T^{*}X 
 \oplus N^{*}\right)\right)\ho\widehat{c}\left(\overline{\mathfrak p}\right) \otimes 
 \End\left(F\right)$$
 into 
 $\Lambda\ac\left(\R^{2*}\right) \otimes _{\R}\C$.

If $m=\dim \mathfrak p$ is odd, then $S^{\overline{\mathfrak p}}$ is not 
$\Z_{2}$-graded, and by (\ref{eq:Lie10}),  $\widehat{c}\left(\overline{\mathfrak 
p}\right)\otimes _{\R}\C=\End\left(S^{\overline{\mathfrak p}}\right) 
\oplus \End\left(S^{\overline{\mathfrak p}}\right)$. Instead of adopting the $\Tr_{\sigma}$ formalism of 
Quillen that was described in subsection \ref{subsec:elop}, we will 
instead exploit the $\Z_{2}$-grading of 
$\widehat{c}\left(\overline{\mathfrak p}\right)$. 
We denote by 
\index{Tro@$\Tr^{\odd}$}%
$\Tr^{\odd}$ the linear 
map that vanishes on 
$\widehat{c}^{\even}\left(\overline{\mathfrak p}\right)$ 
and coincides with $\Tr^{S^{\overline{\mathfrak p}} }$ on $\widehat{c}^{\odd}\left(
\overline{\mathfrak p}\right) $. Then $\Tr^{\odd}$ vanishes 
on supercommutators in $\widehat{c}\left(\overline{\mathfrak 
p}\right)$. We extend $\Tr^{\odd}$ to a map from 
$\Lambda\ac\left(\R^{2*}\right) \ho
\widehat{c}\left(\overline{\mathfrak p}\right) \otimes 
\End\left(E\right)$ into $\Lambda\ac\left(\R^{2*}\right)\otimes 
_{\R}\C$, with 
the convention that if $\eta\in \Lambda\ac\left(\R^{2*}\right),a\in \widehat{c}\left(\overline{\mathfrak p}\right) \otimes 
_{\R}\End\left(E\right)$, then
\begin{equation}\label{eq:sixt2}
\Tr^{\odd}\left[\eta a\right]=\eta\Tr^{\odd}\left[a\right].
\end{equation}
 When combining $\Tr^{\odd}$ with the supertrace on 
$\End\left(\Lambda\ac\left(\mathfrak g^{*}\right)\right)$, 
we get a linear map
\index{Trso@$\Trs^{\odd}$}%
 $\Trs^{\odd}: \Lambda\ac\left(\R^{2*}\right)\ho 
 \End\left(\Lambda\ac\left(\mathfrak g^{*}\right)\right)\ho\widehat{c}\left(\overline{\mathfrak p}\right) \otimes 
\End\left(E\right)\to \Lambda\ac\left(\R^{2*}\right) \otimes 
_{\R}\C$.
 It still vanishes on supercommutators. It
vanishes on the even part of $ \End\left(\Lambda\ac\left(\mathfrak 
g^{*}\right)\right)\ho\widehat{c}\left(\overline{\mathfrak p}\right)\otimes 
\End\left(E\right)$ and is an ordinary supertrace on the odd 
part.
If 
$$\beta\in \left( \Lambda\ac\left(\R^{2*}\right)\ho 
\End\left(\Lambda\ac\left(\mathfrak g^{*}\right)\right)\ho\widehat{c}\left(\overline{\mathfrak p}\right)\otimes 
\End\left(E\right)\right) ^{\even},$$
then 
$\Trs^{\odd}\left[\beta\right]$ is a $1$-form.
When quotienting by $K$, this map descend to a map $\Trs^{\odd}$ from 
$$\Lambda\ac\left(\R^{2*}\right)\ho \End\left(\Lambda\ac
\left(T^{*}X \oplus N^{*}\right)\right)\ho\widehat{c}\left(\overline{TX}\right)\otimes 
\End\left(F\right)$$
with 
values in 
$\Lambda\ac\left(\R^{2*}\right) \otimes _{\R}\C$. 

From now on, and \textbf{in the whole paper},  we assume that \textit{$m$ is odd}.
\subsection{Elliptic orbital integrals and the $1$-form $\mathsf{a}$}%
\label{subsec:oddorb}
Let $\gamma\in G$ . Let 
   \index{dg@$d_{\gamma}$}%
   $d_{\gamma}:X\to \R_{+}$  be the displacement function 
associated with $\gamma$, i.e., if $x\in X$,
\begin{equation}\label{eq:boul1}
d_{\gamma}\left(x\right)=d\left(x,\gamma x\right).
\end{equation}
Following \cite[2.19.21]{Eberlein96},  $\gamma$ is said to be
semisimple if $d_{\gamma}$ reaches its 
minimum value  on $X$. Semisimplicity is a property of the conjugacy 
class of $\gamma$ in $G$. 

By \cite[Theorem 3.1.2]{Bismut08b}, $\gamma$ is semisimple if and only 
if after conjugation, we
can write $\gamma$  in the form
\begin{align}\label{eq:dist9}
    &\gamma=e^{a}k^{-1},&a\in \mathfrak p,\\
    &k\in K,\qquad &\Ad\left(k\right)a=a, \nonumber 
    \end{align}
    the factorization in (\ref{eq:dist9}) being unique.
    By \cite[Theorem 3.1.2]{Bismut08b}, 
\begin{equation}\label{eq:boul2}
\inf_{x\in X}d_{\gamma}\left(x\right)=\left\vert  a\right\vert.
\end{equation}

Recall that the vector space 
\index{H@$\mathcal{H}$}%
$\mathcal{H}$ was defined in Definition \ref{Dspace}.
The left action of $\gamma$ on $G$ descends to an  action 
   on $\mathcal{H}$.
   
     Using the same 
   arguments as in \cite[section 4.4]{Bismut08b}, for $t>0$, we can define the 
   elliptic 
   orbital integrals
   \index{Trg@$\Tr^{\left[\gamma\right]}$}%
   $\Tr^{\left[\gamma\right]}\left[D^{X}\exp\left(-tD^{X,2}/2\right)
   \right]$. These integrals only depend on the conjugacy class of 
   $\gamma$ in $G$.
   
   Recall that since $\gamma$ is semisimple, $\theta\gamma$ is also 
   semisimple.
   \begin{prop}\label{Pzero}
       For $t>0$, the following identity holds:
       \begin{equation}\label{eq:idf1}
\Tr^{\left[\theta\gamma\right]}\left[D^{X}\exp\left(-tD^{X,2}/2\right)\right]=
-\Tr^{\left[\gamma\right]}\left[D^{X}\exp\left(-tD^{X,2}/2\right)\right].
\end{equation}
In particular, if $\gamma$ is elliptic, i.e., if $a=0$, for $t>0$, then
\begin{equation}\label{eq:co40x1}
\Tr^{\left[\gamma\right]}\left[D^{X}\exp\left(-t
D^{X,2}/2\right)\right]=0.
\end{equation}
\end{prop}
\begin{proof}
    As we saw in subsection \ref{subsec:twi},  $\theta_{\pm}$ acts on $C^{ \infty 
}\left(X,S^{\overline{TX}} \otimes F\right)$. Moreover, we have the 
identity of morphisms of $C^{\infty }\left(X,S^{\overline{TX} \otimes 
F}\right)$,
\begin{equation}\label{eq:brun1}
\theta_{\pm} \gamma\theta_{\pm}^{-1}=\theta\left(\gamma\right).
\end{equation}
Combining (\ref{eq:Lie16}) and (\ref{eq:brun1}), we get
    \begin{equation}\label{eq:Lie16bis}
\theta _{\pm}\left[\gamma 
D^{X}\exp\left(-tD^{X,2}/2\right)\right]\theta_{\pm}^{-1}=-\theta\left(\gamma\right) D^{X}\exp\left(-tD^{X,2}/2\right).
\end{equation}
By (\ref{eq:Lie16bis}), we get (\ref{eq:idf1}). If $\gamma$ is 
elliptic, after conjugation, we may  assume that $\gamma\in K$, so 
that $\theta\gamma=\gamma$, and so from 
(\ref{eq:idf1}), we get
(\ref{eq:co40x1}). The proof of our proposition  is completed. 
\end{proof}

In the sequel, we may and we will  assume that $\gamma$ is \textbf{nonelliptic}, i.e., $a\neq 
0$. 

Recall that $T^{X}$ was defined in (\ref{eq:defa4}). A formula for 
$T^{X}$ is also given in (\ref{eq:defa5x1b}). 
\begin{defin}\label{Dforma}
Let 
\index{a@$\mathsf{a}$}%
$\mathsf{a}$ be the $1$-form on  $\left[0,\frac{\pi}{2}\right[$,
\begin{equation}\label{eq:co47}
\mathsf{a}=\Tr^{\left[\gamma\right],\odd}\left[\exp\left(-T^{X}\right) \right].
\end{equation}
\end{defin}

By proceeding as in 
\cite[section 4.4]{Bismut08b}, the orbital integral in 
(\ref{eq:co47}) is well defined.
Like any other $1$-form on $\left[0,\frac{\pi}{2}\right[$,  the 
$1$-form $\mathsf{a}$ is closed. 
\begin{prop}\label{Psimp}
The following identity holds:
\begin{equation}\label{eq:co47x1}
\mathsf{a}=
-\Tr^{\left[\gamma\right]}\left[\frac{\cos\left(\vartheta\right)}{\sqrt{2}}\widehat{D}^{X}\exp\left(-
\mathcal{L}^{X}_{0,\vartheta}\right) \right]d\vartheta.
\end{equation}
\end{prop}
\begin{proof}
    To establish (\ref{eq:co47x1}), we use (\ref{eq:defa5x1b}), 
    (\ref{eq:co47}),   the fact that $\Tr^{\odd}$ 
vanishes on supercommutators, and also the arguments of \cite[chapter 
4]{Bismut08b}, and especially \cite[Theorem 4.3.4]{Bismut08b}. The easy details are left to the 
reader.
\end{proof}
\subsection{Hypoelliptic  orbital integrals and the $1$-form $\mathsf{b}$}%
\label{subsec:onfob}
  Note that $L^{X},\overline{L}^{X},L^{X \prime }$ were defined in  
  Definition \ref{Dmx}, and are given by (\ref{eq:co36}). They are both 
  even in the proper algebra. It is now crucial to use the formalism 
  of the second half of subsection \ref{subsec:trasup}. By \cite[chapter 4]{Bismut08b}, the  hypoelliptic orbital integrals 
   $$\Trs^{\left[\gamma\right],\odd}\left[\exp\left(-L^{X}\right)\right],\Trs^{\left[\gamma\right],\odd}\left[\exp\left(-\overline{L}^{X}\right)\right],
   \Trs^{\left[\gamma\right],\odd}\left[\exp\left(-L
  ^{X \prime } \right)\right]$$
   are  well-defined. These are $1$-forms
   on $\R^{*}_{+}\times \left[0,\frac{\pi}{2}\right[$. Using the same 
   arguments as in \cite[Theorem 4.3.4]{Bismut08b}, and also 
   (\ref{eq:co36}), we get
   \begin{align}\label{eq:co38}
&\Trs^{\left[\gamma\right],\odd}\left[\exp\left(-L^{X}\right)\right]=
-\Trs^{\left[\gamma\right]}\left[\frac{1}{\sqrt{2}}\n^{\mathcal{H}} \mathfrak D^{X}_{b,\vartheta}
\exp\left(-\mathcal{L}^{X}_{b,\vartheta}\right)\right],\nonumber \\
&\Trs^{\left[\gamma\right],\odd}\left[\exp\left(-\overline{L}^{X}\right)\right]=
-\Trs^{\left[\gamma\right]}\left[\frac{1}{\sqrt{2}}\overline{\n}^{\mathcal{H}} \overline{\mathfrak D}^{X}_{b,\vartheta}
\exp\left(-\overline{\mathcal{L}}^{X}_{b,\vartheta}\right)\right],\\
&\Trs^{\left[\gamma\right],\odd}\left[\exp\left(-L^{X \prime 
}\right)\right]=-\Trs^{\left[\gamma\right]}\left[\frac{1}{\sqrt{2}}\n^{\mathcal{H} \prime} \mathfrak D^{X \prime }_{b,\vartheta}
\exp\left(-\mathcal{L}^{X \prime }_{b,\vartheta}\right)\right]. \nonumber 
\end{align}
In the right-hand side of (\ref{eq:co38}), we have eliminated the 
mention \textit{odd}, because the morphism that appears inside is 
indeed odd. Equivalently, the supertrace that appears in the 
right-hand side is an ordinary supertrace associated with the 
$\Z_{2}$-grading of $\Lambda\ac\left(T^{*}X \oplus N^{*}\right)$.
\begin{thm}\label{Tclo}
We have the identity of closed $1$-forms on 
$\R^{*}_{+}\times\left[0,\frac{\pi}{2}\right[$,
\begin{equation}\label{eq:co39}
\Trs^{\left[\gamma\right],\odd}\left[\exp\left(-L^{X}\right)\right]=
\Trs^{\left[\gamma\right],\odd}\left[\exp\left(-\overline{L}^{X}\right)\right]=
\Trs^{\left[\gamma\right],\odd}\left[\exp\left(-L^{X \prime }
\right)\right].
\end{equation}
\end{thm}
\begin{proof}
    Using equation (\ref{eq:co37}) in Proposition \ref{PBibis}, and proceeding as in \cite[proof of 
    Theorem 4.6.1]{Bismut08b} and in (\ref{eq:bob10x1}), we get
    \begin{equation}\label{eq:dep1}
d^{\R^{*}_{+}\times 
\left[0,\frac{\pi}{2}\right[}\Trs^{\left[\gamma\right],\odd}\left[\exp\left(-L^{X}\right)\right]
=\Trs^{\left[\gamma\right],\odd}\left[\left[B^{X},\exp\left(-L^{X}\right)\right]\right]
=0,
\end{equation}
i.e., the form 
$\Trs^{\left[\gamma\right],\odd}\left[\exp\left(-L^{X}\right)\right]$ 
is closed. The same argument shows that 
 the other forms in (\ref{eq:co39}) are also closed. By (\ref{eq:co35x-1a}), the 
first two forms in (\ref{eq:co39}) are equal.

Set
\begin{align}\label{eq:co39x1}
&\widetilde {\n}^{\mathcal{H}}=\widehat{R}^{-1}_{\vartheta}\n^{\mathcal{H}}\widehat{R}_{\vartheta},
&\widetilde{B}^{X}=\widetilde {\n}^{\mathcal{H}}+\frac{\mathfrak D^{X \prime }_{b,\vartheta}}{\sqrt{2}}.
\end{align}
By  (\ref{eq:co2bu}),  (\ref{eq:co31}), and (\ref{eq:co39x1}),   we get
\begin{equation}\label{eq:co39xx2}
\widetilde{B}^{X}=\widehat{R}_{\vartheta}^{-1}B^{X}\widehat{R}_{\vartheta}.
\end{equation}
Set
\begin{equation}\label{eq:co39xx2y}
\widetilde{M}^{X}=-\frac{1}{2}\widehat{D}^{\mathfrak g,X,2}+\widetilde{B}^{X,2}.
\end{equation}
By  (\ref{eq:co35x-1}), (\ref{eq:co39xx2}),  and (\ref{eq:co39xx2y}), we get
\begin{equation}\label{eq:co39xx2z}
\widetilde{M}^{X}=\widehat{R}_{\vartheta}^{-1}L^{X}\widehat{R}_{\vartheta}.
\end{equation}

By (\ref{eq:co39xx2z}), we get the identity of closed $1$-forms
\begin{equation}\label{eq:co39xx3}
\Trs^{\left[\gamma\right],\odd}\left[\exp\left(-L^{X}\right)\right]=
\Trs^{\left[\gamma\right],\odd}\left[\exp\left(-\widetilde{M}^{X}\right)\right].
\end{equation}
To establish  the last identity in (\ref{eq:co39}), we need to show that in 
(\ref{eq:co39xx3}), we can as well replace $\widetilde{\n}^{\mathcal{H}}$ by 
$\n^{\mathcal{H} \prime }$ without changing the corresponding 
$1$-form. We
introduce an extra interpolation parameter $\ell\in\left[0,1\right]$ 
to  interpolate linearly between $\widetilde{\n}^{\mathcal{H}}$ and 
$\n^{\mathcal{H} \prime}$ through connections 
$\n^{\mathcal{H}}_{\ell}$ on $\mathcal{H}$.  Let $\underline
{\n}^{\mathcal{H}}=d\ell\frac{\pa}{\pa \ell}+\n^{\mathcal{H}}_{\ell}$ 
denote the corresponding  connection over $\R_{+}^{*}\times 
\left[0,\frac{\pi}{2}\right[\times \left[0,1\right] $. By the same 
construction as before, we get an  odd closed 
form $\beta$ on $\R_{+}^{*}\times \left[0,\frac{\pi}{2}\right[\times 
\left[0,1\right]$. Since $\mathfrak 
D^{X \prime}_{b,\vartheta}$ does not depend on $\ell$, the component 
$\beta^{(1)}$
of  total degree $1$ of $\beta$ does not
contain
$d\ell $. Since $\beta^{(1)}$ is closed,   as a $1$-form on 
$\R_{+}^{*}\times \left[0,\frac{\pi}{2}\right[$, $\beta^{(1)}$ does 
not depend on $\ell$.  Combining (\ref{eq:co39xx3}) with this result, we 
get the last identity in (\ref{eq:co39}). The proof of our proposition  is completed. 
\end{proof}

In the sequel, we denote by 
\index{b@$\mathsf{b}$}%
$\mathsf{b}$ the $1$-form in 
(\ref{eq:co39}).

%% file: Eta5.tex
\section{A conserved quantity}%
\label{sec:pres}
In this section, we show that the integral of $\mathsf b$ on 
$\left[0,\frac{\pi}{2}\right[$ does not depend of $b>0$ and coincides 
with the integral of $\mathsf{a}$.

This section is organized as follows.
In subsection \ref{subsec:uniell}, we recall known estimates on the 
elliptic heat kernel over $X$.

In subsection \ref{subsec:bsm}, 
we give without proof uniform estimates on the hypoelliptic heat kernel for 
$\overline{L}^{X}\vert _{db=0}$ when $b>0$ remains uniformly bounded, 
and also a convergence result of the hypoelliptic heat kernels to their 
elliptic counterpart when $b\to 0$.  The proof 
of these  results is deferred to section 
\ref{sec:fin}.
 
In subsection 
\ref{subsec:evia}, we give a formula expressing  $\int_{0\le \vartheta\le 
\frac{\pi}{2}}^{}\mathsf{a}$ in terms of the orbital integrals 
$\Trs^{\left[\gamma\right]}\left[D^{X}\exp\left(-sD^{X,2}/2\right)\right]$. 

In subsection \ref{subsec:pres}, we state the conservation result 
that was mentioned before.

Finally, in subsection \ref{subsec:cont}, we prove the identity of subsection 
\ref{subsec:pres} by integrating the closed $1$-form $\mathsf b$ on a 
suitable contour, and by using the estimates of subsection 
\ref{subsec:bsm}.

We make the same assumptions as in section 
\ref{sec:defhyp}. In particular, we assume $m$ to be 
odd, and also that $\gamma$ is nonelliptic, i.e., $a\neq 0$.
\subsection{Uniform estimates on the elliptic heat kernel}%
\label{subsec:uniell}
Recall that the  operator 
\index{TX@$T^{X}$}%
$T^{X}$ was defined in Definition 
\ref{DTX}. By Proposition \ref{PLXt} and by (\ref{eq:defa5x1b}), for 
$\vartheta\in\left[0,\frac{\pi}{2}\right[$, $T^{X}$ is a second order 
elliptic operator.
\begin{defin}\label{DptX}
For $\vartheta\in \left[0,\frac{\pi}{2}\right[, t>0$, let 
\index{pXtt@$p^{X}_{\vartheta,t}\left(x,x'\right)$}%
$p^{X}_{\vartheta,t}\left(x,x'\right)$ be the smooth kernel 
associated with the operator 
$\exp\left(-t T^{X}\right)$ with respect to the volume $dx'$. We use the notation
\index{pXt@$p_{\vartheta}^{X}$}%
    \begin{equation}\label{eq:laus1}
p_{\vartheta}^{X}=p^{X}_{\vartheta,1}.
\end{equation}
\end{defin}

Recall that 
\index{d@$d$}%
$d$ is the Riemannian distance on $X$.
By  (\ref{eq:japo4}),   (\ref{eq:gzinc2}), 
and (\ref{eq:defa5x1b}), classical estimates on elliptic heat kernels 
show that 
given $0<\epsilon\le M < + \infty $, there exist $C>0,C'>0$ such that 
for $\vartheta\in\left[0,\frac{\pi}{2}\right[, \epsilon\le t\le M, x,x'\in X$, then
\begin{equation}\label{eq:gon1}
\left\vert  p_{\vartheta,t}^{X}\left(x,x'\right)\right\vert\le 
C\cos^{-m-1}\left(\vartheta\right)\exp\left(-C'\frac{d^{2}\left(x,x'\right)}{\cos^{2}\left(\vartheta\right)}\right).
\end{equation}
The uniformity of the estimate (\ref{eq:gon1}) on $X$ comes from the 
fact that $X$ is a symmetric space.
\subsection{Uniform estimates on the hypoelliptic heat kernel for $b$ small}%
\label{subsec:bsm}
Recall  that the projector
\index{P@$\mathbf{P}$}%
$\mathbf{P}$ was defined in Definition 
\ref{RtY}. Also 
\index{LX@$\overline{L}^{X}$}%
$\overline{L}^{X}$ was defined in (\ref{eq:co35x-1a}), and is given by 
(\ref{eq:co36}).  
\begin{defin}\label{Dplim}
    For $b>0,\vartheta\in \left[0,\frac{\pi}{2}\right[,t>0$,
let 
\index{qXbt@$\overline{q}^{X }_{b,\vartheta,t}\left(\left(x,Y\right),\left(x',Y'\right)\right)$}%
$\overline{q}^{X }_{b,\vartheta,t}\left(\left(x,Y\right),\left(x',Y'\right)\right)$ be 
the smooth kernel associated with the operator 
$\exp\left(-t\overline{L}^{X}\vert_{db=0}\right)$ with respect to the volume 
$dx'dY'$. We use the notation
\index{qXbt@$\overline{q}_{b,\vartheta}^{X}$}%
    \begin{equation}\label{eq:not1}
\overline{q}_{b,\vartheta}^{X}=\overline{q}^{X}_{b,\vartheta,1}.
\end{equation}

Set
\index{qX0t@$\overline{q}_{0,\vartheta,t}^{X}\left(\left(x,Y\right),\left(x',Y'\right)\right)$}%
\begin{equation}\label{eq:reo-1y1}
\overline{q}_{0,\vartheta,t}^{X}\left(\left(x,Y\right),\left(x',Y'\right)\right)=\mathbf{P}p^{X}_{\vartheta,t}\left(x,x'\right)\pi^{-\left(m+n\right)/2}
\exp\left(-\frac{1}{2}\left( \left\vert  Y\right\vert^{2}+\left\vert  
Y'\right\vert^{2} \right) \right)\mathbf{P}.
\end{equation}
\end{defin}

For the proper functional analytic setting showing that the heat 
kernels $\overline{q}^{X}_{b,\vartheta,t}$ are well defined, we refer 
to \cite[chapter 11]{Bismut08b}.

Now we state an extension of \cite[Theorem 4.5.2]{Bismut08b}.
\begin{thm}\label{Test}
Given $0< \epsilon\le M<+ \infty $, there exist $C>0,C'>0,k\in\N$ such
that for $0<b\le M,\vartheta\in\left[0,\frac{\pi}{2}\right[,\epsilon\le t\le M$,
$\left(x,Y\right),\left(x',Y'\right)\in \widehat{\mathcal{X}}$, then
\begin{multline}
    \left\vert  
   \overline{q}_{b,\vartheta,t}^{X}\left(\left(x,Y\right),\left(x',Y'\right)\right)\right\vert\le 
   C\cos^{-k}\left(\vartheta\right)\\
   \exp \Biggl( 
    -C'\Biggl( 
    \frac{d^{2}\left(x,x'\right)}{\cos^{2}\left(\vartheta\right)}+
    \left\vert  Y^{TX}\right\vert^{2}+\left\vert  
    Y^{TX \prime}\right\vert^{2} 
    +\cos\left(\vartheta\right)\left( \left\vert  Y^{N}\right\vert^{2}+\left\vert  
    Y^{N \prime }\right\vert^{2}\right) \Biggr)\Biggr).
    \label{EQ:BERN0}
\end{multline}
Moreover, as $b\to 0$, 
\begin{equation}
    \overline{q}_{b,\vartheta,t}^{X}\left(\left(x,Y\right),\left(x',Y'\right)\right)\to
    \overline{q}_{0,\vartheta,t}^{X}\left(\left(x,Y\right),\left(x',Y'\right)\right).
    \label{eq:sumex32bis}
\end{equation}
\end{thm}
\begin{proof}
Our theorem will be established in sections \ref{sec:unisca} and  \ref{sec:fin}.
\end{proof}
\subsection{Evaluation of the integral $\int_{0\le \vartheta\le 
\frac{\pi}{2}}^{}\mathsf{a}$}%
\label{subsec:evia}
Recall that the displacement 
    function 
   \index{dg@$d_{\gamma}$}%
    $d_{\gamma}\left(x\right)$ was introduced in 
    (\ref{eq:boul1}).

    By  (\ref{eq:boul2}), (\ref{eq:gon1}),  there exist $C>0,C'>0$ such that
    \begin{equation}\label{eq:gon1y1}
\left\vert  p^{X}_{\vartheta}\left(x,\gamma x\right)\right\vert\le 
C\cos^{-m-1}\left(\vartheta\right)\exp\left(-C'\left(\frac{\left\vert  a\right\vert^{2}}{\cos^{2}\left(\vartheta\right)}+
d^{2}_{\gamma}\left(x\right)\right)\right).
\end{equation}

The $1$-form
\index{a@$\mathsf{a}$}%
$\mathsf{a}$ on $\left[0,\frac{\pi}{2}\right[$ was defined 
in Definition \ref{Dforma}. By proceeding as in \cite[Theorem 4.2.1]{Bismut08b}, using 
(\ref{eq:co47}),  (\ref{eq:gon1y1}) and the fact that $\left\vert  
a\right\vert>0$, for 
$\vartheta\in\left[0,\frac{\pi}{2}\right[$,  we get
\begin{equation}\label{eq:gon1y2}
\left\vert  \mathsf{a}\right\vert\le C\exp\left(-C'\frac{\left\vert  
a\right\vert^{2}}{\cos^{2}\left(\vartheta\right)}\right).
\end{equation}
By (\ref{eq:gon1y2}), the integral 
$\int_{0\le\vartheta\le\frac{\pi}{2}}^{}\mathsf{a}$ is well defined. 

By the same arguments as before, for $s\in \left]0,1\right]$, we get
\begin{equation}\label{eq:gon1y3}
\left\vert  \Tr^{\left[\gamma\right]}\left[\frac{D^{X}}{\sqrt{2}}\exp\left(
-s\frac{D^{X,2}}{2}\right)\right]\right\vert\le 
C\exp\left(-C'\frac{\left\vert  a\right\vert^{2}}{s}\right).
\end{equation}
Therefore the integral $\int_{0}^{1}\frac{1}{2}\Tr^{\left[\gamma\right]}\left[\frac{D^{X}}{\sqrt{2}}\exp\left(
-s\frac{D^{X,2}}{2}\right)\right]
\frac{1}{\sqrt{1-s}}ds$ is well defined.
\begin{prop}\label{Pelem}
The following identity holds:
\begin{multline}\label{eq:echo1}
    \int_{0\le\vartheta\le\frac{\pi}{2}}^{}\mathsf{a}=
-i\exp\left(-\frac{1}{48}\Tr^{\mathfrak k}\left[C^{\mathfrak k, 
\mathfrak k}\right]-\frac{1}{2}C^{\mathfrak k,E}\right)\\
\int_{0}^{1}\frac{1}{2}\Tr^{\left[\gamma\right]}\left[\frac{D^{X}}{\sqrt{2}}
\exp\left(-
sD^{X,2}/2\right)\right]
\frac{1}{\sqrt{1-s}}ds.
\end{multline}
\end{prop}
\begin{proof}
Using (\ref{eq:Lie5y2}), the first identity in (\ref{eq:japo4}), and (\ref{eq:co47x1}),   and  making the change of variables 
$s=\cos^{2}\left(\vartheta\right)$ in the integral of $\mathsf a$,  we get 
(\ref{eq:echo1}).
\end{proof}
\begin{remk}\label{Rfindi}
The right-hand side of equation (\ref{eq:echo1}) should be compared 
with equation (\ref{eq:bob14}) for
\index{FCt@$F_{C,t}\left(D\right)$}
$F_{C,t}\left(D\right)$.
\end{remk}

    By (\ref{eq:boul2}) and by equation (\ref{EQ:BERN0}) in Theorem 
    \ref{Test}, given $M>0$, there exist $C>0,C'>0$ such that  if $0<b\le M,\vartheta\in 
    \left[0,\frac{\pi}{2}\right[,\left(x,Y\right)\in 
   \widehat{\mathcal{X}} $, we get
    \begin{multline}\label{eq:boul3}
\left\vert  \overline{q}_{b,\vartheta}^{X 
}\left(\left(x,Y\right),\gamma\left(x,Y\right)\right) \right\vert\\
\le C\cos^{-k}\left(\vartheta\right)\exp\left(-C'\left(\frac{\left\vert  a\right\vert^{2}}{
\cos^{2}\left(\vartheta\right)}+d^{2}_{\gamma}\left(x\right)
+\left\vert  
Y^{TX}\right\vert^{2}+\cos\left(\vartheta\right)\left\vert  
Y^{N}\right\vert^{2}
\right)\right).
\end{multline}
Using 
 (\ref{eq:boul3}) and  proceeding as in \cite[section 4.3]{Bismut08b},  for $0<b\le M$, we get
\begin{equation}\label{eq:boul4}
\left\vert  
\Trs^{\left[\gamma\right],\odd}\left[\exp\left(-\overline{L}^{X}\vert_{db=0}\right)\right]\right\vert\le
C\cos^{-k-n/2}\left(\vartheta\right)
\exp\left(-C'\frac{\left\vert  a\right\vert^{2}}{\cos^{2}\left(\vartheta\right)}\right).
\end{equation}
By (\ref{eq:boul4}), we deduce that for  $B>0$, the integral 
$\int_{\substack{b=B \\
0\le\vartheta\le\frac{\pi}{2}}}^{}\mathsf{b}$ is well defined.
\subsection{A preserved quantity}%
\label{subsec:pres}
Now we state the following key result, that replaces for us the 
conservation result of 
\cite[Theorem 4.6.1]{Bismut08b} in the context of more classical orbital 
integrals not involving the Dirac operator $D^{X}$.
\begin{thm}\label{Ttempid}
For any $B>0$, the following identity holds:
\begin{equation}\label{eq:co60x4}
\int_{0\le\vartheta\le\frac{\pi}{2}}^{}\mathsf{a}=
\int_{\substack{b=B \\
0\le\vartheta\le\frac{\pi}{2}}}^{}\mathsf{b}.
\end{equation}
\end{thm}
\begin{proof}
Our theorem will be established in subsection \ref{subsec:cont}.
\end{proof}
\subsection{A proof of Theorem \ref{Ttempid}}%
\label{subsec:cont}
Take $B_{0},B,\epsilon$ such that 
$0<B_{0}<B,0<\epsilon<\frac{\pi}{2}$.   Let $\Gamma=\Gamma_{B_{0},B,\epsilon}$ be 
the oriented contour in $\R_{+}^{*}\times \left[0,\frac{\pi}{2}\right[$ shown in 
Figure \ref{F1}. The contour $\Gamma$ is made of oriented 
segments $\Gamma_{i},1\le i\le 4$.
\begin{figure}
			\centerline{\includegraphics[width=3.5in,viewport=75 173 536 640]{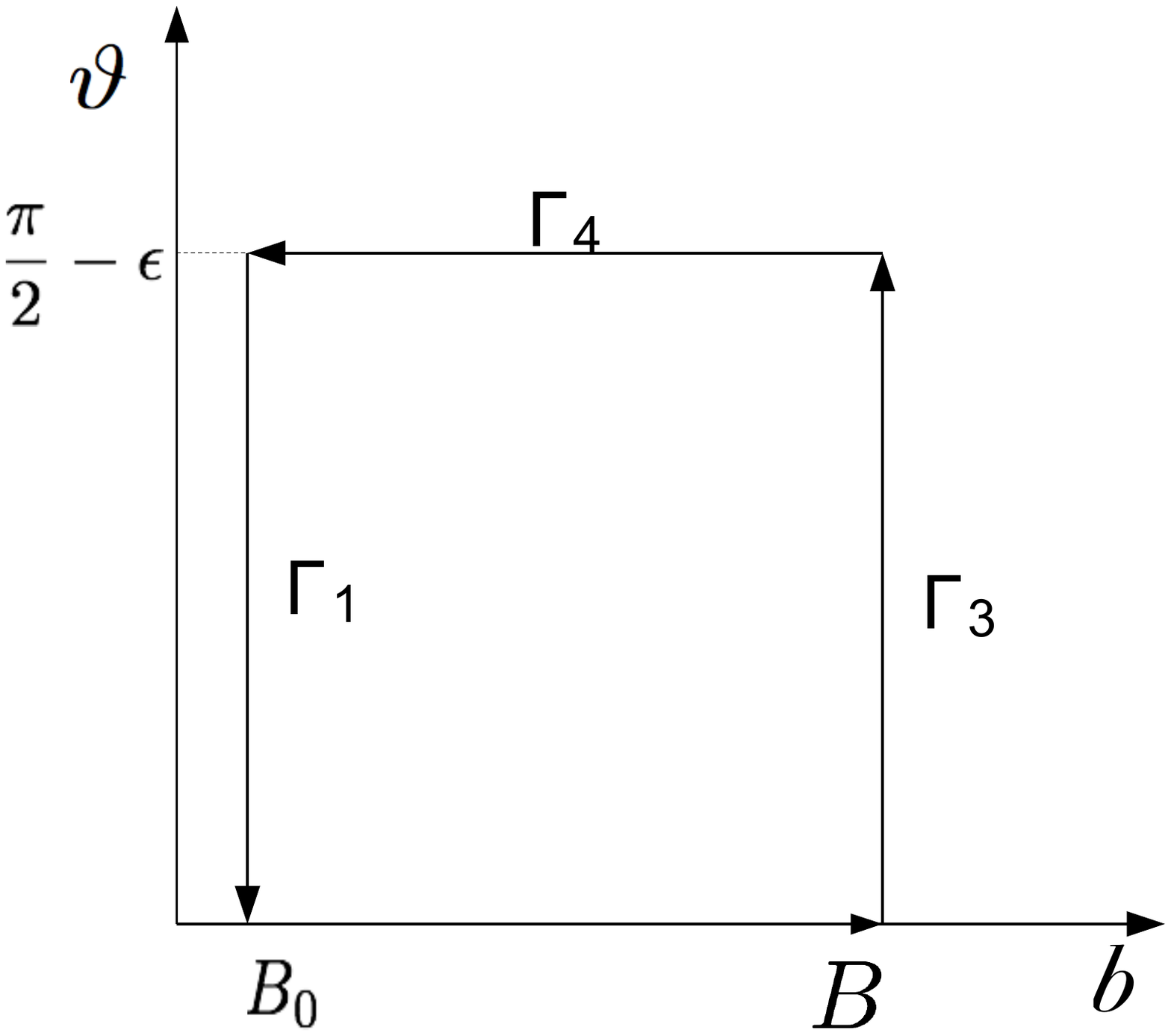}}
			\caption{} \label{F1}	
			\end{figure}
			
By Theorem \ref{Tclo}, the form $\mathsf{b}$ is closed, and so
\begin{equation}\label{eq:co41}
\int_{\Gamma}^{}\mathsf{b}=0.
\end{equation}

For $1\le j\le 4$, set
\begin{equation}\label{eq:co42}
I^{0}_{i}=\int_{\Gamma_{i}}^{}\mathsf{b}.
\end{equation}
By (\ref{eq:co41}), we get
\begin{equation}\label{eq:co43}
\sum_{i=1}^{4}I^{0}_{i}=0.
\end{equation}

We fix $B>0$. We will study the terms $I_{i}^{0},1\le i\le 4$ by making in succession 
$\epsilon\to 0,B_{0}\to 0$.

\noindent \underline{1) The term $I_{1}^{0}$}\\
We have the identity
\begin{equation}\label{eq:co44}
I_{1}^{0}=\int_{\Gamma_{1}}^{}\mathsf{b}.
\end{equation}
We can rewrite (\ref{eq:co44}) 
in the form
\begin{equation}\label{eq:co45}
I^{0}_{1}=-\int_{\substack{b=B_{0} \\
0\le \vartheta\le 
\frac{\pi}{2}-\epsilon}}^{}\Trs^{\left[\gamma\right],\odd}\left[\exp\left(-\overline{L}^{X}\vert_{db=0}\right)\right].
\end{equation}
\begin{enumerate}
    \item \underline{$\epsilon\to 0$}
By (\ref{eq:boul4}), as $\epsilon\to 0$, we get
\begin{equation}\label{eq:boul5}
I_{1}^{0}\to I_{1}^{1}=
-\int_{\substack{b=B_{0} \\
0\le \vartheta\le 
\frac{\pi}{2}}}^{}\Trs^{\left[\gamma\right],\odd}\left[\exp\left(-\overline{L}^{X}\vert_{db=0}\right)\right].
\end{equation}
\item  \underline{$B_{0}\to 0$}
We have the fundamental result.
\begin{thm}\label{Tonvb}
As $B_{0}\to 0$, 
\begin{equation}\label{eq:co48}
I^{1}_{1}\to I_{1}^{2}=-\int_{0\le \vartheta\le\frac{\pi}{2}}\mathsf{a}.
\end{equation}
\end{thm}
\begin{proof}
By   (\ref{eq:sumex32bis}),   (\ref{eq:boul3}), 
and by proceeding as in \cite[proof of Theorem 4.6.1]{Bismut08b}, we 
obtain the pointwise convergence of $1$-forms on 
$\left[0,\frac{\pi}{2}\right[$,
\begin{equation}\label{eq:co48a1}
\Trs^{\left[\gamma\right],\odd}\left[\exp\left(-\overline{L}^{X}\vert_{db=0}\right)\right]\to
\Tr^{\left[\gamma\right],\odd}\left[\exp\left(-T^{X}\vert_{db=0}\right)\right].
\end{equation}
Using  the uniform bounds in 
(\ref{eq:boul4}), (\ref{eq:co48a1}), and dominated convergence, we get (\ref{eq:co48}).
\end{proof}
\end{enumerate}

\noindent\underline{2) The term $I_{2}^{0}$}\\
By (\ref{eq:co29x1}), (\ref{eq:co38}), we get
\begin{equation}\label{eq:co57}
I_{2}^{0}=\frac{1}{\sqrt{2}}\int_{B_{0}}^{B}\Trs^{\left[\gamma\right],\odd}\left[\left(
bic\left(\left[Y^{N},Y^{TX}\right]\right)+\mathcal{E}^{TX}+i
\mathcal{E}^{N}\right)
\exp\left(-\mathcal{L}^{X}_{b}\right)\right]\frac{2db}{b^{2}}.
\end{equation}
By equation (\ref{eq:tra2}), the term 
$-\frac{i}{b}\widehat{c}\left(\ad\left(Y^{N}\right)\vert_{\overline{TX}}\right)$ 
is the only term in $\mathcal{L}^{X}_{b}$ that contains Clifford variables in $c\left(\overline{TX}\right)$, 
and it lies in $c^{\even}\left(\overline{TX}\right)$. Also as 
we saw in subsection \ref{subsec:spin}, $K$ maps to
$c^{\even}\left(\overline{\mathfrak p}\right)$ via the 
$\mathrm{spin}$ representation. Therefore, the integrand in (\ref{eq:co57}) 
vanishes identically, so that
\begin{equation}\label{eq:co58}
I_{2}^{0}=0.
\end{equation}
A related argument is that in the right-hand side of (\ref{eq:co57}), 
only odd endomorphisms of $\Lambda\ac\left(T^{*}X \oplus  N^{*}\right) 
\otimes S^{\overline{TX}}\otimes F$ appear, so that the 
corresponding supertrace vanishes.

\noindent\underline{3) The term $I^{0}_{3}$}\\
By definition,
\begin{equation}\label{eq:co59}
I^{0}_{3}=\int_{\Gamma_{3}}^{}\mathsf{b}.
\end{equation}
We can rewrite (\ref{eq:co59}) in the form
\begin{equation}\label{eq:co59x1}
I^{0}_{3}=\int_{\substack{b=B\\
0\le\vartheta\le 
\frac{\pi}{2}-\epsilon}}^{}\Tr^{\left[\gamma\right],\odd}
\left[\exp\left(-\overline{L}^{X}\vert_{db=0}\right)\right].
\end{equation}

\begin{enumerate}
\item \underline{$\epsilon\to 0$}
    By (\ref{eq:boul4}), as $\epsilon\to 0$, 
  \begin{equation}\label{eq:co60}
I_{3}^{0}\to 
   I_{3}^{1}= \int_{\substack{b=B\\
0\le\vartheta\le 
\frac{\pi}{2}}}^{}\Tr^{\left[\gamma\right],\odd}
\left[\exp\left(-\overline{L}^{X}\vert_{db=0}\right)\right].
    \end{equation}
\item As $B_{0}\to 0$, $I_{3}^{1}$ remains constant and equal to 
$I_{3}^{2}$.

    \end{enumerate}
    
   \noindent\underline{4) The term $I_{4}^{0}$}\\
   By definition, we get
   \begin{equation}\label{eq:laus2}
I_{4}^{0}=\int_{\Gamma_{4}}^{}\mathsf{b}.
\end{equation}
   \begin{prop}\label{Plimgeg}
As $\epsilon\to 0$, 
\begin{equation}\label{eq:co60x3}
I_{4}^{0}\to 0.
\end{equation}
\end{prop}
   \begin{proof}
   By (\ref{eq:co29x1}), (\ref{eq:co38}), we get
   \begin{multline}\label{eq:co60x2}
I_{4}^{0}=-\frac{1}{\sqrt{2}}\int_{B_{0}}^{B}\Trs^{\left[\gamma\right]}
\Biggl[\left( b\sin^{1/2}\left(\epsilon\right)ic\left(\left[Y^{N},Y^{TX}\right]\right)+\mathcal{E}
^{TX}
+\sin^{1/2}\left(\epsilon\right)i\mathcal{E}^{N}\right) \\
\exp\left(-\overline{\mathcal{L}}^{X }_{b,\frac{\pi}{2}-\epsilon}\right)
\Biggr]\frac{2db}{b^{2}}.
\end{multline}
By making $d\vartheta=0$ in (\ref{eq:boul3}) and using (\ref{eq:co60x2}),  
we find that given $0<B_{0}<B<+ \infty $,  there exist 
$C>0,C'>0$ such that
\begin{equation}\label{eq:co60x2y}
\left\vert  I_{4}^{0}\right\vert\le C\exp\left(-C'\frac{\left\vert  
a\right\vert^{2}}{\sin^{2}(\epsilon)}\right),
\end{equation}
which gives (\ref{eq:co60x3}).
\end{proof}

We are now ready to prove Theorem \ref{Ttempid}. By taking the limit 
of (\ref{eq:co43}) as $\epsilon\to 0,B_{0}\to 0$, we get
\begin{equation}\label{eq:simp1}
\sum_{i=1}^{4}I_{i}^{2}=0,
\end{equation}
which by (\ref{eq:co48}), (\ref{eq:co58}), (\ref{eq:co60}), and 
(\ref{eq:co60x3}) is just (\ref{eq:co60x4}).

%% file: Eta6.tex
\section{A geometric formula for 
$\int_{0\le\vartheta\le\frac{\pi}{2}}^{}\mathsf{a}$}\label{sec:oddorb}
In this section, by making $B\to + \infty $ in  Theorem \ref{Ttempid}, we give an explicit geometric 
formula  for the elliptic orbital integral 
$\int_{0\le\vartheta\le\frac{\pi}{2}}^{}\mathsf{a}$.
 This formula is in some sense the main result 
of this paper. It will be worked out in more detail in section 
\ref{sec:applic}, in order to make  the proper comparison with the results of 
Moscovici-Stanton \cite{MoscoviciStanton89}. 

The proof of our main result relies on the results of 
\cite{Bismut08b}, and also on uniform
estimates on the hypoelliptic heat kernels for $b>0$ large, that will 
be established in section \ref{sec:unilar}.  

The structure of this 
section is strictly similar to the structure of \cite[chapter 
9]{Bismut08b}, where corresponding results are established for 
standard orbital integrals. The main new ingredients with respect to 
\cite{Bismut08b} are 
the uniform upper 
bounds on certain integrands when
$\vartheta\in\left[0,\frac{\pi}{2}\right[$.

This section is organized as follows. In subsection 
\ref{subsec:minge}, we recall the results of \cite{Bismut08b} that 
identify the minimizing set $X\left(\gamma\right)$ for the 
displacement function $d_{\gamma}$.

In subsection \ref{subsec:lim}, we state our geometric formula for 
$\int_{0\le\vartheta\le\frac{\pi}{2}}^{}\mathsf{a}$  in terms of objects 
previously obtained in \cite{Bismut08b}. Among these objects, there 
is an important function $J_{\gamma}$ that will play an essential 
role in section \ref{sec:applic}. The remainder of this 
section is devoted to the proof of this result.

In subsection \ref{subsec:prth}, we give a formula for the limit as 
$b\to + \infty $ of the $1$-form 
$\Trs^{\left[\gamma\right],\odd}\left[\exp\left(-L^{X \prime 
}\vert_{db=0}\right)\right]$. Our main result in subsection 
\ref{subsec:lim} is a trivial consequence of this convergence result.
The subsections that follow are 
devoted to its proof.

In subsection \ref{subsec:estaw}, we give  estimates  for the 
smooth 
kernel for $\exp \left( -\underline{L}^{X \prime }\vert_{db=0} \right) $ in 
the range $b\ge 1,\vartheta\in\left[0,\frac{\pi}{2}\right[$ away from a submanifold of 
$\widehat{\mathcal{X}}$ that fibres over $X\left(\gamma\right)$. The 
proof of these estimates is deferred to section \ref{sec:unilar}.

In subsection \ref{subsec:rescfY}, if $\underline{L}^{X \prime}$ is a 
rescaled version of $L^{X \prime }$, we show that the orbital integral 
$\Trs^{\left[\gamma\right],\odd}\left[\exp\left(-\underline{L}^{X \prime 
}\vert_{db=0}\right)\right]$ localizes near a submanifold over 
$X\left(\gamma\right)$, and we suitably rescale coordinates on 
$\widehat{\mathcal{X}}$.

In subsection \ref{subsec:getres}, as in \cite[section 
9.5]{Bismut08b}, we introduce a conjugation on certain Clifford 
variables, which is equivalent to a suitable Getzler rescaling on the 
matrix part of the operator $\underline{L}^{X \prime }\vert_{db=0}$. From 
$\underline{L}^{X \prime }\vert_{db=0}$, we obtain an operator 
$\mathfrak L^{X \prime }_{b,\vartheta}$.

In subsection \ref{subsec:limloc}, we give a result on the limit as 
$b\to + \infty $ of a local supertrace of a rescaled heat kernel over a neighbourhood of 
$\widehat{\pi}^{-1}X\left(\gamma\right)$. The proof is deferred to 
subsections \ref{subsec:tras}--\ref{subsec:propi}. 

In subsection \ref{subsec:prlimi}, we complete the proof of the
result that was stated in subsection \ref{subsec:prth}.

In subsection \ref{subsec:tras}, we make the translation $Y^{TX}\to 
a^{TX}+Y^{TX}$ on  the operator $\mathfrak L^{X \prime 
}_{b,\vartheta}$, and we obtain a new operator $\mathcal{O}^{X \prime 
}_{a,b,\vartheta}$.

In subsection \ref{subsec:cotri}, as in \cite[section 
9.9]{Bismut08b}, we choose a coordinate system on  
$\widehat{\mathcal{X}}$ based at $x_{0}=p1$, and we trivialize our 
vector bundles. We obtain this way an operator $\mathcal{P}^{X \prime 
}_{a,b,\vartheta,\Yok}$.

In subsection \ref{subsec:aspy}, we show that as $b\to + \infty $, 
$\mathcal{P}^{X}_{a,b,\vartheta,\Yok}$ converges in the proper sense 
to an operator $P^{X \prime }_{a,\infty ,\vartheta,\Yok}$. The 
wonderful fact is that the dependence of this operator on 
$\vartheta$ is very mild, and that it differs very little from the 
operator $\mathcal{P}^{X \prime }_{a, \infty, 0,\Yok}$ already 
considered in \cite[section 9.10]{Bismut08b}.

Finally, in section \ref{subsec:propi}, we state a result of 
convergence of heat kernels, that implies the convergence results of 
section \ref{subsec:limloc}.  

The techniques used in this chapter are variations on the techniques 
of \cite{Bismut08b} in the range $b\ge 1$. The fact that 
$\vartheta$ may approach $\frac{\pi}{2}$ is handled using the fact 
that $a\neq 0$.

We make the same assumptions and we use the same notation as in section 
\ref{sec:pres}.
\subsection{The geometry of the minimizing set}%
\label{subsec:minge} 
We follow \cite[chapter 3]{Bismut08b}.  Let 
\index{Zg@$Z\left(\gamma\right) $}%
$Z\left(\gamma\right) \subset G$ be the centralizer of $\gamma$, 
and let 
\index{zg@$\mathfrak z\left(\gamma\right)$}%
$\mathfrak z\left(\gamma\right)$ be its Lie algebra. Let
\index{Za@$Z\left(a\right)$}%
$Z\left(a\right) \subset G$ be the stabilizer of $a$, and let 
\index{za@$\mathfrak z\left(a\right)$}%
$\mathfrak z\left(a\right)$ be its Lie algebra. By \cite[Proposition 
3.2.8]{Bismut08b}, we have the identity 
\begin{equation}\label{eq:supr1}
Z\left(e^{a} \right) =Z\left(a\right).
\end{equation}

By \cite[eq. 
(3.1.2)]{Bismut08b}, we have the identity
\begin{equation}\label{eq:grzu-1a1}
\mathfrak z\left(a\right)=\ker\ad\left(a\right).
\end{equation}
By 
\cite[eqs. (3.3.4) and (3.3.6)]{Bismut08b}, we have 
\begin{align}\label{eq:grzu0}
&Z\left(\gamma\right)=Z\left(e^{a}\right)\cap Z\left(k\right),
&\mathfrak z\left(\gamma\right) =\mathfrak z\left(e^{a}\right)\cap 
\mathfrak z\left(k\right).
\end{align}
Since $\theta a=-a,\theta k=k$,  by (\ref{eq:grzu0}), $\theta$ acts on 
$Z\left(\gamma\right), Z\left(e^{a}\right),Z\left(k\right)$. 

As in 
\cite[eq. (3.3.7)]{Bismut08b}, put
\begin{align}\label{eq:grzu1}
&\mathfrak p\left(\gamma\right)= \mathfrak z\left(\gamma\right)\cap 
\mathfrak p,
&\mathfrak k\left(\gamma\right)=\mathfrak z\left(\gamma\right)\cap 
\mathfrak k.
\end{align}
By \cite[eq. (3.3.8)]{Bismut08b}, we have the splitting
\begin{equation}\label{eq:grzu2}
\mathfrak z\left(\gamma\right)= \mathfrak p\left(\gamma\right) \oplus 
\mathfrak k\left(\gamma\right).
\end{equation}
Put
\index{r@$r$}%
\index{p@$p$}%
\index{q@$q$}%
\begin{align}
   & r=\dim \mathfrak z\left(\gamma\right),&p=\dim 
   \mathfrak p\left(\gamma\right),\qquad q=\dim \mathfrak k\left(\gamma\right).
    \label{eq:bonn0}
\end{align}
By (\ref{eq:grzu2}), we get
\begin{equation}
    r=p+q.
    \label{eq:sumex63-a}
\end{equation}

Put
\index{Kg@$K\left(\gamma\right)$}%
\begin{equation}\label{eq:tuz1}
K\left(\gamma\right)=K\cap Z\left(\gamma\right).
\end{equation}
Then $K\left(\gamma\right)$ is a compact Lie subgroup of $Z\left(\gamma\right)$ with Lie
algebra $\mathfrak k\left(\gamma\right)$.

Let 
\index{Z0g@$Z^{0}\left(\gamma \right)$}%
\index{K0g@$K^{0}\left(\gamma\right)$}%
$Z^{0}\left(\gamma \right) ,K^{0}\left(\gamma\right)$ be the 
connected components of the identity in 
$Z\left(\gamma\right),K\left(\gamma\right)$. The Cartan involution $\theta$ acts on 
$Z^{0}\left(\gamma\right)$. Then $Z^{0}\left(\gamma\right)$ is a 
connected reductive group with maximal compact subgroup 
$K^{0}\left(\gamma\right)$.  Also the symmetric space 
$Z^{0}\left(\gamma\right)/K^{0}\left(\gamma\right)$ embeds 
canonically in the symmetric space $X=G/K$. The same considerations still hold when 
replacing $\gamma$ by $e^{a}$ or $k$.

We denote by 
\index{zpg@$\mathfrak z^{\perp}\left(\gamma\right)$}%
$\mathfrak z^{\perp}\left(\gamma\right)$ the orthogonal 
to $\mathfrak z\left(\gamma\right)$ with respect to $B$. By 
(\ref{eq:grzu2}),  
$\mathfrak z^{\perp}\left(\gamma\right)$ splits as
\index{zpg@$\mathfrak z^{\perp}\left(\gamma\right)$}%
\index{ppg@$\mathfrak p^{\perp}\left(\gamma\right)$}%
\index{kpg@$\mathfrak k^{\perp}\left(\gamma\right)$}%
\begin{equation}\label{eq:grzu2x1}
\mathfrak z^{\perp}\left(\gamma\right)=\mathfrak 
p^{\perp}\left(\gamma\right) \oplus \mathfrak 
k^{\perp}\left(\gamma\right).
\end{equation}

In the sequel,  we use the notation
\index{Z0@$Z_{0}$}%
\index{K0@$K_{0}$}%
\index{z0@$\mathfrak z_{0}$}%
\begin{align}\label{eq:grzu3}
&Z_{0}=Z\left(a\right), &K_{0}=Z\left(a\right)\cap K,\qquad\mathfrak z_{0}=\mathfrak z\left(a\right).
\end{align}
We denote by 
\index{Z00@$Z^{0}_{0}$}%
\index{K00@$K^{0}_{0}$}%
$Z_{0}^{0},K_{0}^{0}$ the connected components of the 
identity in $Z_{0},K_{0}$.
Put
\index{p0@$\mathfrak p_{0}$}%
\index{k0@$\mathfrak k_{0}$}%
\begin{align}\label{eq:grzu4}
&\mathfrak p_{0}=\ker\ad\left(a\right)\cap \mathfrak p,
&\mathfrak k_{0}=\ker\ad\left(a\right)\cap \mathfrak k.
\end{align}
By \cite[eq. (3.5.5)]{Bismut08b}, we get
\begin{equation}\label{eq:grzu5}
\mathfrak z_{0}= \mathfrak p_{0} \oplus \mathfrak k_{0}.
\end{equation}
As was explained before,  $Z_{0}^{0}$ equipped with the involution $\theta$ is a connected reductive group with maximal compact 
subgroup $K_{0}^{0}$. Also (\ref{eq:grzu5}) is the Cartan splitting of 
$\mathfrak z_{0}$ associated with $\theta$.
Observe that $\mathfrak p_{0}$ and $\mathfrak p\left(k\right)$ 
intersect orthogonally along $\mathfrak p\left(\gamma\right)$.

Let 
\index{zp0@$\mathfrak z_{0}^{\perp}$}%
$\mathfrak z_{0}^{\perp}$ be the orthogonal vector space  to $\mathfrak z_{0}$ 
in $\mathfrak g$. Then $\mathfrak z_{0}$ splits as
\index{pp0@$\mathfrak p_{0}^{\perp}$}%
\index{kp0@$\mathfrak k_{0}^{\perp}$}%
\begin{equation}\label{eq:grzu6x1}
\mathfrak z_{0}^{\perp}=\mathfrak p^{\perp}_{0} \oplus \mathfrak 
k^{\perp}_{0}.
\end{equation}

Let $\mathfrak z_{0}^{\perp}\left(\gamma\right)$ be the orthogonal vector space 
to $\mathfrak z\left(\gamma\right)$ in $\mathfrak z_{0}$ with respect 
to $B$. Let $\mathfrak p_{0}^{\perp}\left(\gamma\right), \mathfrak 
k_{0}^{\perp}\left(\gamma\right)$ be the orthogonal  vector spaces to $\mathfrak 
p\left(\gamma\right),\mathfrak k\left(\gamma\right)$ in $\mathfrak 
p_{0},\mathfrak k_{0}$. By \cite[eq. (5.3.5)]{Bismut08b}, we have the splitting
\begin{equation}\label{eq:grzu6}
\mathfrak z_{0}^{\perp}\left(\gamma\right)=\mathfrak 
p_{0}^{\perp}\left(\gamma\right) \oplus \mathfrak 
k_{0}^{\perp}\left(\gamma\right).
\end{equation}

Recall that the 
displacement function
\index{dg@$d_{\gamma}$}%
$d_{\gamma}$ on $X$ was defined in 
(\ref{eq:boul1}).
Let 
\index{Xg@$X\left(\gamma\right)$}%
$X\left(\gamma\right) \subset X$ be the minimizing set for
$d_{\gamma}$. By \cite[p. 78]{BalGroSchro}, $X\left(\gamma\right)$
is a closed convex subset of $X$. By \cite[Theorem 3.3.1]{Bismut08b}, 
$X\left(\gamma\right)$ can be canonically identified with the 
symmetric space $Z^{0}\left(\gamma\right)/K^{0}\left(\gamma\right)$. 
The embedding $Z^{0}\left(\gamma\right)/K^{0}\left(\gamma\right) 
\subset G/K$ corresponds to the embedding $X\left(\gamma\right) 
\subset X$. The same considerations apply to 
$X\left(e^{a}\right),X\left(k\right)$.

By 
\cite[Theorem 3.3.1]{Bismut08b}, we have the identity 
\begin{equation}\label{eq:bres2}
X\left(\gamma\right)=X\left(e^{a}\right)\cap X\left(k\right).
\end{equation}
Also the manifolds $X\left(e^{a}\right)$ and $X\left(k\right)$ intersect 
orthogonally along $X\left(\gamma\right)$.
\subsection{A fundamental identity}%
\label{subsec:lim}
Note that if $\Yok\in \mathfrak k\left(\gamma\right)$, then 
$\ad\left(\Yok\right)$ acts on $\mathfrak 
z_{0}^{\perp}\left(\gamma\right)$ and preserves the splitting 
(\ref{eq:grzu6}). 

Set
\index{Ax@$\widehat{A}\left(x\right)$}%
\begin{equation}\label{eq:laus3}
\widehat{A}\left(x\right)=\frac{x/2}{\sinh\left(x/2\right)}.
\end{equation}
We identify $\widehat{A}$ with the corresponding multiplicative genus.
If $V$ is a finite dimensional  Hermitian vector space and if $B\in\End\left(V\right)$ is  
self-adjoint, then $\frac{B/2}{\sinh\left(B/2\right)}$ is a 
self-adjoint positive endomorphism. Set
\begin{equation}\label{eq:laus4}
\widehat{A}\left(B\right)=\det\,^{1/2}\left[\frac{B/2}{\sinh\left(B/2\right)}\right].
\end{equation}
In (\ref{eq:laus4}), the square root is taken to be the positive 
square root.

Now we follow \cite[Theorem 5.5.1]{Bismut08b}.
\begin{defin}\label{DJg}
Let 
\index{JgY@$J_{\gamma}\left(\Yok\right)$}%
$J_{\gamma}\left(\Yok\right)$ be the function defined on 
$\mathfrak k\left(\gamma\right)$ with values in $\C$ given by
\begin{multline}\label{eq:crub3}
J_{\gamma}\left(Y_{0}^{ \mathfrak 
k}\right)=\frac{1}{\left\vert  \det\left(1-\Ad\left(\gamma\right)\right)\vert_{ 
\mathfrak z_{0}^{\perp}}\right\vert^{1/2}}
\frac{\widehat{A}\left(i\ad\left(\Yok\right)\vert _{ \mathfrak p\left(\gamma\right)}\right)}{\widehat{A}\left(i\ad\left(\Yok\right)
\vert_{\mathfrak k\left(\gamma\right)}\right)}\\
\left[\frac{1}{\det\left(1-\Ad\left(k^{-1}\right)\right)\vert_{ \mathfrak z^{\perp}_{0}\left(\gamma\right)}
    }\frac{\det\left(1-\exp\left(-i\ad\left(\Yok\right)
    \right)\Ad\left(k^{-1}\right)\right)\vert_{\mathfrak k^{\perp}_{0}\left(\gamma\right)}}{\det\left(1-\exp\left(-i\ad\left(\Yok\right)\right)
    \Ad\left(k^{-1}\right)\right)\vert_{\mathfrak 
    p_{0}^{\perp}\left(\gamma\right)}}\right]^{1/2}.
\end{multline}
\end{defin}
As explained in \cite[section 5.5]{Bismut08b}, the square root 
appearing in the right-hand side of (\ref{eq:crub3}) is unambiguously 
defined.

We  now state the main result of this section. This result will be 
fully exploited in section \ref{sec:applic}.
\begin{thm}\label{Tkeyres}
    The following identity holds:
\begin{multline}\label{eq:key1}
i\exp\left(-\frac{1}{48}\Tr^{\mathfrak k}\left[C^{\mathfrak k, 
\mathfrak k}\right]-\frac{1}{2}C^{\mathfrak k,E}\right)\\
\int_{0}^{1}\frac{1}{2}\Tr^{\left[\gamma\right]}\left[\frac{D^{X}}{\sqrt{2}}
\exp\left(-
sD^{X,2}/2\right)\right]
\frac{1}{\sqrt{1-s}}ds \\
=\frac{1}{\left(2\pi\right)^{p/2}}
\int_{\mathfrak k\left(\gamma\right)
}^{}\frac{\sqrt{\pi}}{2}
\exp\left(-\frac{1}{2}\left(\left\vert  
a\right\vert^{2}+\left\vert  \Yok\right\vert^{2}\right)\right)\\
J_{\gamma}\left(\Yok\right)
\Tr^{S^{\mathfrak p}}
\left[\frac{\widehat{c}\left(a\right)}{\left\vert  a\right\vert}\Ad\left(k^{-1}\right)
\exp\left(-i\widehat{c}
\left(\ad\left(\Yok\right)\vert_{\mathfrak p}\right)\right)\right]\\
\Tr^{E}\left[\rho^{E}\left(k^{-1}\right)\exp\left(-i\rho^{E}\left(\Yok\right)\right)\right]
\frac{d\Yok}{\left(2\pi\right)^{q/2}}.
\end{multline}
\end{thm}
\begin{proof}
Subsection \ref{subsec:prth}  is devoted to the proof of our theorem.
\end{proof}
    \subsection{The limit of the forms 
    $\Trs^{\left[\gamma\right],\odd}\left[\exp\left(-
    L^{X \prime }\vert_{db=0}\right)\right]$ as $b\to + \infty$}%
\label{subsec:prth}
In this subsection, we  study the limit of the $1$-form $\Trs^{\left[\gamma\right],\odd}
\left[
\exp\left(-L^{X \prime }\vert_{db=0}\right)\right]$ as $b\to + \infty 
$.
\begin{thm}\label{Tlimi}
 As $b\to + \infty $, we have the pointwise convergence of $1$-forms 
 on $\left[0,\frac{\pi}{2}\right[$,
\begin{multline}\label{eq:key2}
\Trs^{\left[\gamma\right],\odd}
\left[
\exp\left(-L^{X \prime }\vert_{db=0}\right)\right]\to
-d\vartheta{\left(2\pi\right)^{-p/2}}\exp\left(-\left\vert  a\right\vert^{2}
    /2\cos^{2}\left(\vartheta\right) \right) \\
    \int_{\mathfrak k\left(\gamma\right)}^{}
   \frac{1}{\sqrt{2}} J_{\gamma}\left(\Yok\right)
    \Tr^{S^{\mathfrak p}}
    \left[\frac{\widehat{c}\left(a\right)}{\cos^{2}\left(\vartheta\right)}
    \Ad\left(k^{-1}\right)\exp\left(-i\widehat{c}\left(\ad\left(\Yok\right)\vert_{\mathfrak p}\right)\right)\right] \\
    \Tr^{E}\left[\rho^{E}\left(k^{-1}\right)
    \exp\left(-i\rho^{E}\left(\Yok\right)\right)\right]
    \exp\left(-\left\vert  \Yok\right\vert^{2}/2\right)
    \frac{d\Yok}{\left(2\pi\right)^{q/2}}. 
\end{multline}
Moreover, there exist $C>0,C'>0$ such that for $b\ge 
1,\vartheta\in\left[0,\frac{\pi}{2}\right[$,
\begin{equation}\label{eq:bub3}
\left\vert \Trs^{\left[\gamma\right]}\left[
\exp\left(-L^{X \prime }\vert_{db=0}\right)\right] \right\vert\le
C\exp\left(-C'\frac{\left\vert  
a\right\vert^{2}}{\cos^{2}\left(\vartheta\right)}\right).
\end{equation}
\end{thm}
\begin{proof}
The proof of our theorem is delayed to subsections 
\ref{subsec:rescfY}--\ref{subsec:prlimi}. 
\end{proof}
\begin{remk}\label{Rdedu}
We claim that Theorem \ref{Tkeyres} follows from Theorem \ref{Tlimi}. 
Indeed, let  $\mathsf{c}$ be the $1$-form in the right-hand side of (\ref{eq:key2}). 
By  equation (\ref{eq:co60x4}) 
in Theorem \ref{Ttempid}, and by 
equations (\ref{eq:key2}), (\ref{eq:bub3}) in Theorem \ref{Tlimi},  
using dominated convergence, we get
\begin{equation}\label{eq:key3}
    \int_{0\le\vartheta\le\frac{\pi}{2}}^{}\mathsf{a}=\int_{0\le\vartheta\le\frac{\pi}{2}}^{}\mathsf{c}.
\end{equation}
By making the change of variables 
$v=\tan\left(\vartheta\right)$, for $x>0$, we get
\begin{equation}\label{eq:laus5}
\int_{0}^{\frac{\pi}{2}
}\cos^{-2}\left(\vartheta\right)\exp\left(-x^{2}/2\cos^{2}\left(\vartheta\right)\right)
d\vartheta=\exp\left(-x^{2}/2\right)\int_{0}^{+ 
\infty}\exp\left(-x^{2}v^{2}/2\right)dv.
\end{equation}
By (\ref{eq:laus5}), we obtain
\begin{equation}\label{eq:key4}
\int_{0}^{\frac{\pi}{2}
}\cos^{-2}\left(\vartheta\right)\exp\left(-x^{2}/2\cos^{2}\left(\vartheta\right)\right)
d\vartheta=\frac{\sqrt{\pi}}{\sqrt{2}x}\exp\left(-x^{2}/2\right).
\end{equation}
By equation (\ref{eq:echo1}) in Proposition \ref{Pelem}, 
and by equations (\ref{eq:key2}),  (\ref{eq:key3}), and 
(\ref{eq:key4}), we get (\ref{eq:key1}), which completes the proof of 
Theorem \ref{Tkeyres}.
\end{remk}
\subsection{Estimates on the  heat kernel  for $ \protect\underline{L}^{X 
\prime }\vert_{db=0}$  away 
from $\widehat{i}_{a}\mathcal{N}\left(k^{-1}\right)$}%
\label{subsec:estaw}
As in \cite[section 3.4]{Bismut08b}, we identify the total space 
of the normal bundle
$N_{X\left(\gamma\right)/X}$ to the symmetric space $X$ via the 
normal geodesic coordinate based at $X\left(\gamma\right)$. 

We proceed as in \cite[section 3.6]{Bismut08b}. Note that 
$\Ad\left(k^{-1}\right)$ acts on $N\vert_{X\left(\gamma\right)}$. Set
\index{Nk@$N\left(k^{-1}\right)$}%
\begin{equation}
    N\left(k^{-1}\right)=\left\{Y^{N}\in 
    N\vert_{X\left(\gamma\right)},\Ad\left(k^{-1}\right)Y^{N}=Y^{N}\right\}.
    \label{eq:bubble1}
\end{equation}
Then $N\left(k^{-1}\right)$ is the vector bundle on $X\left(\gamma\right)$ associated with the eigenspace of 
$\mathfrak k$ corresponding to the eigenvalue $1$ of 
$\Ad\left(k^{-1}\right)$, i.e., with the Lie algebra
\index{kk@$\mathfrak 
k\left(k^{-1}\right)$}%
$\mathfrak 
k\left(k^{-1}\right)$ of the centralizer of $k^{-1}$ in $K$.  Let 
\index{Nk@$\mathcal{N}\left(k^{-1}\right)$}%
$\mathcal{N}\left(k^{-1}\right)$ be the total space of $N\left(k^{-1}\right)$.

Let
\index{ia@$\widehat{i}_{a}$}%
$\widehat{i}_{a}$ be the embedding $\left(x,Y^{N}\right)\in 
\mathcal{N}\left(k^{-1}\right)\to \left(x,a^{TX},Y^{N}\right)\in 
\widehat{\mathcal{X}} $. For the geometric interpretation of the set 
$\widehat{i}_{a}\mathcal{N}\left(k^{-1}\right)$, we refer to 
\cite[Proposition 3.6.1]{Bismut08b}.

For $a>0$, set
\index{Ka@$K_{-a}$}%
\begin{equation}\label{eq:sma1}
K_{-a}s\left(x,Y\right)=a^{\left(m+n\right)/2}s\left(x,-aY\right).
\end{equation}

  For $b>0,\vartheta\in \left[0,\frac{\pi}{2}\right[$, by analogy 
    with \cite[eq. (9.1.2)]{Bismut08b}, where only the case  
    $\vartheta=0$ was considered,  set
    \index{LXbt@$\underline{\mathcal{L}}^{X \prime }_{b,\vartheta}$}%
    \index{LX@$\underline{L}^{X \prime}\vert_{db=0}$}%
    \begin{align}\label{eq:ors9}
&\underline{\mathcal{L}}^{X \prime }_{b,\vartheta}=K_{-b/\cos\left(\vartheta\right)}\mathcal{L}^{X 
\prime }_{b,\vartheta}K_{-b/\cos\left(\vartheta\right)}^{-1},
&\underline{L}^{X \prime}\vert_{db=0}=
K_{-b/\cos\left(\vartheta\right)}L^{X \prime }\vert_{db=0}K_{-b/\cos\left(\vartheta\right)}^{-1}
\end{align}
By (\ref{eq:glub2}), we get
 \begin{multline}\label{eq:glub2bis1}
\underline{\mathcal{L}}^{X \prime }_{b,\vartheta}=\frac{b^{4}}{2\cos^{2}\left(\vartheta\right)}\left\vert  \left[Y^{ N},Y^{TX}\right]\right\vert^{2}
     +\frac{1}{2} \Biggl( 
     -\frac{\cos^{2}\left(\vartheta\right)}{b^{4}}\Delta^{TX\oplus N}+\frac{\left\vert  
     Y^{TX}\right\vert^{2}}{\cos^{2}\left(\vartheta\right)}+\
     \left\vert  Y^{N}\right\vert^{2}  \\
     -\frac{1}{b^{2}}\left(m+\ct n\right)\Biggr) 
    +\frac{N^{\Lambda\ac\left(T^{*}X \oplus N^{*}\right) 
    \prime }_{-\vartheta}}{b^{2}}  \\
     -\Biggl(\n_{Y^{ TX}}^{C^{ \infty }\left(TX \oplus 
 N,\widehat{\pi}^{*} \left( \Lambda\ac\left(T^{*}X \oplus 
 N^{*}\right)\otimes S^{\overline{TX}} \otimes F
 \right) \right),f*,\widehat{f} }
	    \\
  -ic\left(
   \theta\ad\left(Y^{N}\right) \right) 
   -i\widehat{c}\left(\ad\left(Y^{N}\right)\vert_{\overline{TX}}\right)
   -i\rho^{F}\left(Y^{N}\right)\Biggr).
   \end{multline} 
   Also by (\ref{eq:co29x1}), (\ref{eq:co36}), and (\ref{eq:ors9}), we obtain
   \begin{multline}\label{eq:glub2bis1x}
\underline{L}^{X \prime 
}\vert_{db=0}=\underline{\mathcal{L}}^{X \prime }_{b,\vartheta}
+\frac{d\vartheta}{\sqrt{2}b}\Biggl(\frac{b^{3}\sin\left(\vartheta\right)}{\cos^{2}\left(\vartheta\right)}
ic\left(\left[Y^{N},Y^{TX}\right] \right)\\
+\frac{\sin\left(\theta\right)}{b}\left(\mathcal{D}^{TX}-i\mathcal{D}^{N}\right) 
+\frac{b}{\cos^{2}\left(\vartheta\right)}
\widehat{c}\left(\overline{Y}^{TX}\right)\Biggr).
\end{multline}
\begin{prop}\label{simide}
The following identity holds:
\begin{multline}\label{eq:fair1}
\Trs^{\left[\gamma\right],\odd}\left[\exp\left(-\underline{L}^{X 
\prime}\vert_{db=0}\right)\right]=
-\frac{d\vartheta}{\sqrt{2}b}\Trs^{\left[\gamma\right]}\Biggl[\Biggl(\frac{b^{3}\sin\left(\vartheta\right)}{\cos^{2}\left(\vartheta\right)}
ic\left(\left[Y^{N},Y^{TX}\right] \right)\\
+\frac{\sin\left(\theta\right)}{b}\left(\mathcal{D}^{TX}-i\mathcal{D}^{N}\right) 
+\frac{b}{\cos^{2}\left(\vartheta\right)}
\widehat{c}\left(\overline{Y}^{TX}\right)\Biggr)
\exp\left(-\underline{\mathcal{L}}^{X \prime }_{\bt}\right)\Biggr].
\end{multline}
\end{prop}
\begin{proof}
As we saw in subsection \ref{subsec:trasup}, 
$\Trs^{\left[\gamma\right],\odd}$ is  a 
 supertrace, i.e.,  it vanishes on 
supercommutators. Then (\ref{eq:fair1}) follows from 
(\ref{eq:glub2bis1x}).
\end{proof}
\begin{defin}\label{Dlastke}
 Let
  \index{qXbt@$\underline{q}^{X \prime }_{b,\vartheta,t}\left(\left(x,Y\right),\left(x',Y'\right)\right)$}%
  $\underline{q}^{X \prime }_{b,\vartheta,t}\left(\left(x,Y\right),\left(x',Y'\right)\right)$ be the smooth kernel associated with 
  the operator 
  $\exp\left(-t\underline{\mathcal{L}}^{X \prime}_{\bt}\right)$. When $t=1$, we use 
   the notation  
  \index{qXb@$\underline{q}^{X \prime}_{b,\vartheta}\left(\left(x,Y\right),\left(x',Y'\right)\right)$}%
  $\underline{q}^{X \prime 
  }_{b,\vartheta}\left(\left(x,Y\right),\left(x',Y'\right)\right)$.
\end{defin}

 First, we extend the estimates in \cite[Theorem 9.1.1]{Bismut08b} 
 that are valid for $\vartheta=0$.
\begin{thm}\label{TLARGE1QU}
    Given $0<\epsilon\le M <+ \infty$, there exist 
    $C_{\epsilon,M}>0,C'_{\epsilon,M}>0$ such that  for $b\ge 1,0\le 
    \vartheta<\frac{\pi}{2},\epsilon\le t\le M$,
    $\left(x,Y\right),\left(x',Y'\right)\in\widehat{\mathcal{X}}$,
    \begin{multline}
       \left\vert  \underline{q}^{X\prime}_{b,\vartheta,t}\left(\left(x,Y\right),\left(x',Y^{ \prime }\right)\right)\right\vert
       \le 
       C_{\epsilon,M}\left(b/\cos^{1/2}\left(\vartheta\right)\right)
       ^{4m+2n} \\
       \exp\left(-C'_{\epsilon,M}\left(\frac{d^{2}\left(x,x'\right)}{\cos^{2}\left(\vartheta\right)}
       +\left\vert  
	   Y\right\vert^{2}+\left\vert  Y^{\prime 
	   }\right\vert^{2}\right) \right) .
	\label{eq:large5globbis}
    \end{multline}
    
      Given 
       $\beta>0,0<\epsilon\le M<+ \infty$, there exist
       $\eta_{M}>0, C_{\epsilon,M}>0,C'_{\epsilon,M}>0, C''_{\gamma,\beta,M}>0$ such that  for 
       $b\ge 1,0\le\vartheta<\frac{\pi}{2},\epsilon\le t\le M$,
       $\left(x,Y\right)
       \in\widehat{\mathcal{X}}$, if 
       $d\left(x,X\left(\gamma\right)\right)\ge \beta$, 
       \begin{multline}
	   \left\vert 
	   \underline{q}^{X\prime}_{b,\vartheta,t}\left(\left(x,Y\right),
	   \gamma\left(x,Y\right)\right)\right\vert
	      \le C_{\epsilon,M}\left(b/\cos^{1/2}\left(\vartheta\right)\right)
	      ^{4m+2n} \\
	      \exp\left(-C'_{\epsilon,M}\left( \frac{d^{2}_{\gamma}\left(x\right)+
	      \left\vert  
	      a\right\vert^{2}}{\cos^{2}\left(\vartheta\right)}
	      +\left\vert  
	       Y\right\vert^{2}\right)-C''_{\gamma,\beta,M}
	       \exp\left(-2\eta_{M}\left\vert 
	       Y^{TX}\right\vert\right)b^{4}/\cos^{2}\left(\vartheta\right)\right) .
	   \label{eq:large5terqu}
       \end{multline}
  
       There exists $\eta>0$ such that given $\beta>0,\mu>0$, there exist  $C>0,C'>0, 
       C''_{\gamma,\beta,\mu}>0$ such that  for $b\ge 
       1,0\le\vartheta<\frac{\pi}{2},
	   \left(x,Y\right)
	   \in\widehat{\mathcal{X}}$, if 
	   $d\left(x,X\left(\gamma\right)\right)\le \beta,\left\vert  
	   Y^{TX}-a^{TX}\right\vert\ge \mu$,
	   \begin{multline}
	      \left\vert 
	      \underline{q}^{X\prime}_{b,\vartheta}\left(\left(x,Y\right),
	      \gamma\left(x,Y\right)\right)\right\vert
		  \le 
		  C\left(b/\cos^{1/2}\left(\vartheta\right)\right)^{4m+2n}\\
		  \exp\left(-C' \left( \frac{\left\vert  
		  a\right\vert^{2}}{
		  \cos^{2}\left(\vartheta\right)}+\left\vert  
		   Y\right\vert^{2} \right) -C''_{\gamma,\beta,\mu}
		   \exp\left(-2\eta\left\vert Y^{TX}\right\vert\right)b^{4}
		   /\cos^{2}\left(\vartheta\right)\right)  .
	       \label{eq:large5bisa2qu}
	   \end{multline}
	   
				     
There exist $c  >0,C>0,C' _{\gamma} >0$ 
				      such that for 
$b\ge 1, 0\le\vartheta<\frac{\pi}{2}, f\in \mathfrak p^{\perp}
\left(\gamma\right),\left\vert  f\right\vert\le 
				    1,
				    Y\in \left(TX \oplus N\right)_{x},\left\vert  
				      Y^{TX}-a^{TX}\right\vert\le 1$,
				      \begin{multline}
\left\vert\underline{ 
q}_{b,\vartheta}^{X 
\prime}\left(\left(e^{f}p1,Y\right),\gamma\left(e^{f}p1,Y\right)\right)\right\vert\\
\le c 
\left(b/\cos^{1/2}\left(\vartheta\right)\right)^{4m+2n} \exp\Biggl(-C 
\left( \frac{\left\vert  
a\right\vert^{2}}{\cos^{2}\left(\vartheta\right)}+
\left\vert Y^{N}\right\vert^{2} \right) \\
-C'_{\gamma}  \left(\left\vert f\right\vert^{2}+ \left\vert
Y^{TX}-a^{TX}\right\vert^{2} 
\right)b^{4}/\cos^{2}\left(\vartheta\right)\\
-C'_{\gamma}  \left\vert \left(
\Ad\left(k^{-1}\right)-1\right)Y^{N}\right\vert b^{2}/ \ct
-C'_{\gamma} \left\vert
\left[a^{TX},Y^{N}\right]\right\vert b^{2}/\cos\left(\vartheta\right)\Biggr).
					  \label{eq:bull2bisqu}
	\end{multline}
	
	The above inequalities remain valid when replacing 
	$\underline{q}_{\bt,t}^{X \prime }$ 
	or 
	$\underline{q}^{X \prime }_{\bt}$ by 
	$\frac{\ct}{b^{2}}\n^{V}\underline{q}_{\bt,t}^{X \prime }$ or 
	$\frac{\ct}{b^{2}}\n^{V}\underline{q}^{X \prime }_{\bt}$.
   \end{thm}
   \begin{proof}
   The proof of our theorem is deferred to section \ref{sec:unilar}.
\end{proof}
\subsection{A rescaling of the coordinates $\left(f,Y\right)$}%
\label{subsec:rescfY}

In the sequel, we use the notation
\begin{multline}\label{eq:fair2}
\underline{r}^{X 
\prime}_{\bt}\left(\left(x,Y\right),\left(x',Y'\right)\right)=
-\frac{d\vartheta}{\sqrt{2}b}\Biggl(\frac{b^{3}\sin\left(\vartheta\right)}{\cos^{2}\left(\vartheta\right)}
ic\left(\left[Y^{N},Y^{TX}\right] \right)\\
+\frac{\sin\left(\theta\right)}{b}\left(\mathcal{D}^{TX}-i\mathcal{D}^{N}\right) 
+\frac{b}{\cos^{2}\left(\vartheta\right)}
\widehat{c}\left(\overline{Y}^{TX}\right)\Biggr)\underline{q}^{X 
\prime}_{\bt}\left(\left(x,Y\right),\left(x',Y'\right)\right).
\end{multline}

We start proving Theorem \ref{Tlimi}. We proceed as in \cite[section 9.2]{Bismut08b}.
When $f\in 
\mathfrak p^{\perp}\left(\gamma\right)$, we identify $e^{f}\in G$ and 
$e^{f}p1\in X$. By \cite[eq. (4.3.10)]{Bismut08b}, we get
\begin{multline}\label{eq:final1}
\Trs^{\left[\gamma\right],\odd}\left[\exp\left(-L^{X \prime }\vert_{db=0}\right)\right]\\
=
\int_{\widehat{\pi}^{-1}\mathfrak p^{\perp}\left(\gamma\right) 
}^{}\Trs^{\odd}
\left[\gamma 
\underline{r}^{X\prime}_{b,\vartheta}\left(\left(e^{f},Y\right),\gamma\left(e^{f},Y\right)\right)\right]r\left(f\right)dYdf.
\end{multline}
As explained in detail in \cite[section 4.2]{Bismut08b}, in the 
right-hand side of (\ref{eq:final1}), $\gamma$ is viewed  as 
acting on $X$, this action lifting to $\Lambda\ac\left(T^{*}X \oplus 
N^{*}\right) \otimes S^{\overline{TX}} \otimes F$.
Also  $r\left(f\right)>0$ is a Jacobian.  
By \cite[eq. (3.4.36)]{Bismut08b},  we have the estimate
\begin{equation}\label{eq:bub2}
r\left(f\right)\le C\exp\left(C'\left\vert  f\right\vert\right).
\end{equation}

By \cite[eq. (3.4.4)]{Bismut08b}, there exists $C_{\gamma}>0$ such 
that 
\begin{equation}\label{eq:bub1}
d_{\gamma}\left(e^{f}\right)\ge \left\vert  
a\right\vert+C_{\gamma}\left\vert  f\right\vert.
\end{equation}

By \cite[eq. (9.1.6)]{Bismut08b}, given $\eta>0, C'>0,C''>0,C'''>0$,  
there exist $c>0,d>0,e>0$ such that for $b\ge 1, Y^{TX}\in TX$, then
\begin{multline}
    C'\left\vert  
    Y^{TX}\right\vert^{2}+C''\exp\left(-2\eta\left\vert  
    Y^{TX}\right\vert\right)b^{4}-C'''\log\left(b\right)\ge \frac{C'}{2}\left\vert  
    Y^{TX}\right\vert^{2}\\
    +\frac{C''}{2}\exp\left(-2\eta
    \left\vert  Y^{TX}\right\vert\right)b^{4}-c\ge
    \frac{C'}{4}\left\vert  
    Y^{TX}\right\vert^{2}+e\log^{2}\left(b\right)-d.
    \label{eq:lowbd3}
\end{multline}

Take $\beta\in ]0,1]$. By 
equation (\ref{eq:large5terqu}) in Theorem 
\ref{TLARGE1QU}, by the last part of this theorem, and by 
(\ref{eq:fair2})--(\ref{eq:lowbd3}), there 
exist $C>0,C'>0,C''>0$ such that for $b\ge 1,\vartheta
\in\left[0,\frac{\pi}{2}\right[$, 
\begin{multline}\label{eq:bibi1}
\left\vert  \int_{\substack{\left(f,Y\right)\in\widehat{\pi}^{-1}\mathfrak p^{\perp}\left(\gamma\right) 
    \\ \left\vert  f\right\vert >\beta}}^{}\Trs^{\odd}
    \left[\gamma 
    \underline{r}^{X\prime}_{b,\vartheta}\left(\left(e^{f},Y\right),\gamma\left(e^{f},Y\right)\right)\right]r
    \left(f\right)dYdf\right\vert\\
    \le 
    C\exp\left(-C'\left\vert  a\right\vert^{2}/\cos^{2}\left(\vartheta\right)
    -C''\log^{2}\left(b\right)\right).
\end{multline}

By   equation  (\ref{eq:large5bisa2qu}) in Theorem \ref{TLARGE1QU}, 
by the last statement in this theorem, and by (\ref{eq:lowbd3}),  given $\beta>0,\mu>0$, we may as 
well obtain a similar result for the integral of 
\begin{equation}\label{eq:spli1y1}
\Trs^{\odd}
    \left[\gamma 
    \underline{r}^{X\prime}_{b,\vartheta}\left(\left(e^{f},Y\right),
    \gamma\left(e^{f},Y\right)\right)\right]r\left(f\right)
\end{equation}
    over the region considered
    in (\ref{eq:large5bisa2qu}). In particular, as $b\to + \infty $, 
    both integrals tend to $0$, and they can be estimated uniformly 
    by an expression like the right-hand side of (\ref{eq:bub3}).

As in \cite[section 9.2]{Bismut08b}, we trivialize $TX,N$ by parallel transport 
with respect to the connections $\n^{TX}, \n^{N}$ along the 
geodesics $t\in \R\to e^{tf}p1\in X$.  In this trivialization 
$Y^{N}\in N$ splits as
\begin{equation}\label{eq:spli1}
Y^{N}=\Yok+Y^{N,\perp}, \ \Yok\in \mathfrak k\left(\gamma\right), 
Y^{N,\perp}\in \mathfrak k^{\perp}\left(\gamma\right).
\end{equation}

     To control the behaviour of (\ref{eq:final1}) as $b\to + 
   \infty $, given $\beta>0$, we may as well consider the integral
\begin{multline}\label{eq:final3}
    \int_{\substack{\left\vert  f\right\vert\le \beta\\
    \left\vert  Y^{TX}-a^{TX}\right\vert\le \beta}}
    ^{}\Trs^{\odd}
    \left[\gamma 
    \underline{r}^{X\prime}_{b,\vartheta}\left(\left(e^{f},Y\right),\gamma\left(e^{f},Y\right)\right)\right]r\left(f\right)dY^{TX}dY^{N}df\\
    =
    \left(b^{2}/\cos\left(\vartheta\right)\right)^{-2m-n+r}\int_{\substack{\left\vert  f\right\vert\le \beta b^{2}/\cos\left(\vartheta\right)\\
	\left\vert  Y^{TX}\right\vert\le \beta b^{2}/\cos\left(\vartheta\right)}}
	^{}\\
\Trs^{\odd}	\Biggl[\gamma 
	\underline{r}^{X\prime}_{b,\vartheta}\Biggl(\left(e^{\cos\left(\vartheta\right)f/b^{2}},a^{TX}_{e^{\ct f/b^{2}}}
	+\cos\left(\vartheta\right)
	Y^{TX}/b^{2},\Yok+\cos\left(\vartheta\right)Y^{N,\perp}/b^{2}
	\right),\\
	\gamma\left(e^{\cos\left(\vartheta\right)f/b^{2}},a^{TX}_{e^{\ct 
	f/b^{2}}}+\cos\left(\vartheta\right)Y^{TX}/b^{2},\Yok+
	\cos\left(\vartheta\right)Y^{N,\perp}/b^{2}\right)\Biggr)\Biggr]\\
	r\left(\cos\left(\vartheta\right)f/b^{2}\right)dY^{TX}d\Yok dY^{N,\perp}df.
    \end{multline}
    
 Using  equation (\ref{eq:bull2bisqu}) in Theorem \ref{TLARGE1QU} and 
 the last part of this theorem,   and proceeding as in \cite[eqs. 
    (9.2.8)--(9.2.12)]{Bismut08b}, we find that for $\beta>0$ small 
    enough, for $\left\vert  f\right\vert\le \beta 
    b^{2}/\cos\left(\vartheta\right),\left\vert  Y^{TX}\right\vert\le
    \beta b^{2}/\cos\left(\vartheta\right)$, then
\begin{multline}
 \left(b^{2}/\cos\left(\vartheta\right)\right)^{-2m-n} \Biggl\vert
  \underline{q}^{X\prime}_{b,\vartheta}\Biggl(\Bigl(e^{\cos\left(\vartheta\right)
  f/b^{2}},a^{TX}_{e^{\ct f/b^{2}}}
+\cos\left(\vartheta\right)Y^{TX}/b^{2}, \\
  \Yok+\cos\left(\vartheta\right)Y^{N,\perp}/b^{2}
  \Bigr),\\
  \gamma\left(e^{\cos\left(\vartheta\right)f/b^{2}},a^{TX}_{e^{\ct f/b^{2}}}
+\cos\left(\vartheta\right)
  Y^{TX}/b^{2},\Yok+\cos\left(\vartheta\right)Y^{N,\perp}/b^{2}\right)\Biggr)
  \Biggr\vert \\
					 \le C 
					 \exp\Biggl(-C'\left\vert  
					 a^{2}
					 \right\vert/\cos^{2}\left(\vartheta\right)-C' \left\vert 
					 \Yok\right\vert^{2}\\
					 -C' _{\gamma} \left(\left\vert f\right\vert^{2}+ \left\vert
					 Y^{TX}\right\vert^{2} \right) -
					 C'_{\gamma}  \left\vert \left(
					 \Ad\left(k^{-1}\right)-1\right)
					 Y^{N,\perp}\right\vert  -C'_{\gamma} \left\vert
					 \left[a,Y^{N,\perp}\right]\right\vert \Biggr).
   \label{eq:mossbrugg3}
	 \end{multline} 
	 The same argument shows that in (\ref{eq:mossbrugg3}), we 
can replace $\underline{q}^{X \prime }_{\bt}$ by 
$\frac{\ct}{b^{2}}\n^{V}\underline{q}^{X \prime }_{\bt}$.
	By \cite[eq. (9.10.4)]{Bismut08b}, we have
\begin{equation}\label{eq:ronro1}
a^{TX}_{e^{\cos\left(\vartheta\right)f/b^{2}}}=
a+\mathcal{O}\left(\left\vert  f\right\vert^{2}/b^{4}\right).
\end{equation}
By (\ref{eq:ronro1}), we deduce that
\begin{multline}\label{eq:ronro2}
b^{2}\left[\Yok+\ct Y^{N,\perp}/b^{2},a^{TX}_{e^{\ct f/b^{2}}}+\ct Y^{TX}/b^{2}\right]=
\ct\left[\Yok,Y^{TX}\right] \\
+\ct\left[Y^{N,\perp},a^{TX}_{e^{\ct f/b^{2}}}\right]+\frac{\cos^{2}\left(\vartheta\right)}{b^{2}}\left[Y^{N,\perp},
Y^{TX}\right]+\left\vert  \Yok\right\vert\mathcal{O}\left(\left\vert  
f\right\vert^{2}/b^{2}\right).
\end{multline}
By (\ref{eq:fair2}), (\ref{eq:mossbrugg3}), and (\ref{eq:ronro2}), in 
(\ref{eq:mossbrugg3}), we may as well replace $\underline{q}^{X 
\prime }_{\bt}$ by $\underline{r}^{X \prime }_{\bt}$.

	 As in \cite[section 9.2]{Bismut08b}, we are still left with 
	 the diverging term 
	 $\left(b^{2}/\cos\left(\vartheta\right)\right)^{r}$ in 
	 the right-hand side of (\ref{eq:final3}). As in 
	 this reference, this term will be dealt with using local 
	 cancellation techniques.
\subsection{A conjugation of the hypoelliptic Laplacian}%
\label{subsec:getres}
We will proceed exactly as in \cite[sections 9.3--9.5]{Bismut08b}.
\begin{defin}\label{Dnewextalg}
    Let
    \index{zg@$\underline{\mathfrak z}\left(\gamma\right)$}%
    $\underline{\mathfrak z}\left(\gamma\right)$ denote 
    another copy of $ \mathfrak z\left(\gamma\right)$, and let
    $\underline{ \mathfrak z}\left(\gamma\right)^{*}$ denote the 
    corresponding copy of the dual of $\mathfrak 
    z\left(\gamma\right)$. If $\alpha\in \Lambda\ac\left( \mathfrak 
z\left(\gamma\right)^{*}\right)$, let $\underline{\alpha}$ be
the corresponding element in $\Lambda\ac\left( \underline{\mathfrak 
z}\left(\gamma\right)^{*}\right)$.
    Also 
\index{czg@$c\left( \mathfrak 
z\left(\gamma\right)\right)$}%
\index{czg@$\widehat{c}\left( \mathfrak 
z\left(\gamma\right)\right)$}%
$c\left( \mathfrak 
z\left(\gamma\right)\right),\widehat{c}\left( \mathfrak 
z\left(\gamma\right)\right)$ denote the Clifford algebras of 
$\left(\mathfrak z\left(\gamma\right),B\vert_{ \mathfrak 
z\left(\gamma\right)}\right),\left(\mathfrak z\left(\gamma\right),-B\vert_{ \mathfrak z\left(\gamma\right)}\right)$.
\end{defin}

Let $e_{1},\ldots,e_{r}$ be a basis of $ \mathfrak z\left(\gamma\right)$, let 
$e^{1},\ldots,e^{r}$ be the corresponding dual basis of $\mathfrak 
z\left(\gamma\right)^{*}$.

As in \cite[eq. (9.3.4)]{Bismut08b}, put
\index{a@$\alpha$}
\index{a@$\widehat{\alpha}$}%
\begin{align}\label{eq:bruggmoss1}
&\alpha=\sum_{i=1}^{r}c\left(e_{i}\right)\underline{e}^{i},
&\widehat{\alpha}=\sum_{1}^{r}\widehat{c}\left(e_{i}\right)\underline{e}^{i}.
\end{align}
Then
\begin{align}\label{eq:bruggmoss2}
&\alpha\in c\left( \mathfrak z\left(\gamma\right)\right)\ho \Lambda\ac\left( \underline{\mathfrak 
z}\left(\gamma\right)^{*}\right),
&\widehat{\alpha}\in\widehat{c}\left(\mathfrak z\left(\gamma\right)\right)\ho 
\Lambda\ac\left(\underline{ \mathfrak z}\left(\gamma\right)^{*}\right).
\end{align}
Since $\mathfrak z\left(\gamma\right) \subset \mathfrak g$, from 
(\ref{eq:bruggmoss2}), we get
\begin{align}\label{eq:bruggmoss2x1}
&\alpha\in c\left( \mathfrak g\right)\ho \Lambda\ac\left( \underline{\mathfrak 
z}\left(\gamma\right)^{*}\right),
&\widehat{\alpha}\in\widehat{c}\left(\mathfrak g\right)\ho 
\Lambda\ac\left(\underline{ \mathfrak z}\left(\gamma\right)^{*}\right).
\end{align}

We proceed as in \cite[sections 3.10, 9.3, and 9.4]{Bismut08b}. Under the 
identification $TX \oplus N = \mathfrak g$, we denote by $\left(TX 
\oplus N\right)\left(\gamma\right)$ the subvector bundle of $TX 
\oplus N$ corresponding to $\mathfrak z\left(\gamma\right)$. Then 
$\left(TX \oplus N\right)\left(\gamma\right)$ is a flat vector 
subbundle of $TX \oplus N$ with respect to the connection $\n^{TX 
\oplus N,f}$ which is preserved by the left action of $\gamma$. Let 
$\left( \underline{T}X \oplus  \underline{N}\right) \left(\gamma\right)$ be 
the vector bundle on $X$ corresponding to $\underline{\mathfrak 
z}\left(\gamma\right)$. Since $\left(TX \oplus 
N\right)\left(\gamma\right)$ is a subvector bundle of $TX \oplus N$, 
it inherits a scalar product. We equip  $\left( \underline{T}X \oplus  \underline{N}\right) \left(\gamma\right)$
with the corresponding scalar product.

Then $\alpha$ can be viewed as a section of $c\left(TX \oplus 
N\right)\ho\Lambda\ac\left(\left(\underline{T}X \oplus 
\underline{N}\right)\left(\gamma\right)^{*}\right)$,
and $\widehat{\alpha}$ as a section of $\widehat{c}\left(TX \oplus 
N\right)\ho\Lambda\ac\left(\left(\underline{T}X \oplus 
\underline{N}\right)\left(\gamma\right)^{*}\right)$. Moreover, 
$\alpha$ and $\widehat{\alpha}$
 can also be considered as sections of 
$\End \left( \Lambda\ac\left(T^{*}X \oplus N^{*}\right)\right)\ho 
\Lambda\ac \left( \left(\underline{T}X \oplus 
\underline{N}\right)\left(\gamma\right)^{*} \right) $. 

From the flat connection $\n^{\Lambda\ac\left(T^{*}X \oplus 
    N^{*} \right) ,f*,\widehat{f}}$ on $\Lambda\ac\left(T^{*}X 
    \oplus N^{*}\right)$ and from the trivial connection on 
    $\Lambda\ac\left(\underline{ \mathfrak 
    z}\left(\gamma\right)^{*}\right)$,
   we obtain a  flat connection on 
    $$\Lambda\ac\left(T^{*}X \oplus 
    N^{*}\right)\ho\Lambda\ac\left(\underline{ \mathfrak 
    z}\left(\gamma\right)^{*}\right),$$
    which is still denoted 
    $\n^{\Lambda\ac\left(T^{*}X \oplus 
	N^{*} \right) ,f*,\widehat{f}}$.  
	By 
\cite[Proposition 9.4.1]{Bismut08b}, we have
\begin{equation}\label{eq:grzu7}
\n^{\Lambda\ac\left(T^{*}X \oplus 
N^{*}\right),f*,\widehat{f}}\widehat{\alpha}=0.
\end{equation}
By \cite[Proposition 9.3.2]{Bismut08b} and by (\ref{eq:comp1z1}), we get
\begin{equation}\label{eq:grzu8}
\exp\left(-\frac{b^{2}}{\cos\left(\vartheta\right)}\widehat{\alpha}\right)
N_{-\vartheta}^{\Lambda\ac\left(T^{*}X \oplus N^{*}\right) \prime }\exp\left(\frac{b^{2}}{\cos\left(\vartheta\right)}
\widehat{\alpha}\right) 
=
N_{-\vartheta}^{\Lambda\ac\left(T^{*}X \oplus N^{*}\right) \prime }+b^{2}\alpha.
\end{equation}

 As in \cite[Definition 9.5.1]{Bismut08b}, set
 \index{LXbt@$\mathfrak L^{X \prime}_{\bt}$}%
 \index{LX@$\mathfrak L^{X \prime}$}%
\begin{align}\label{eq:co62}
	&\mathfrak L^{X \prime}_{\bt}=\exp\left(-\frac{b^{2}}{\cos\left(\vartheta\right)}
\widehat{\alpha}\right)\underline{\mathcal{L}}^{X \prime }_{\bt}
\exp\left(\frac{b^{2}}{\cos\left(\vartheta\right)}\widehat{\alpha}\right),\\
&\mathfrak L^{X \prime}=\exp\left(-\frac{b^{2}}{\cos\left(\vartheta\right)}
\widehat{\alpha}\right)\underline{L}^{X \prime }\vert_{db=0}
\exp\left(\frac{b^{2}}{\cos\left(\vartheta\right)}\widehat{\alpha}\right).\nonumber 
\end{align}
Using (\ref{eq:glub2bis1}), (\ref{eq:glub2bis1x}), and 
(\ref{eq:grzu7})--(\ref{eq:co62}),   we get
\begin{align}\label{eq:co63}
	&\mathfrak L^{X \prime }_{\bt}=\underline{\mathcal{L}}^{X \prime 
	}_{\bt}+\alpha,\\
&\mathfrak L^{X \prime }=\underline{L}^{X 
\prime }\vert_{db=0}+\alpha.\nonumber 
\end{align}
In (\ref{eq:co63}), we used the fact that in the factor of 
$d\vartheta$ in the right-hand side of (\ref{eq:glub2bis1x}), no term 
in $\widehat{c}\left(\mathfrak g\right)$ appears.

\begin{defin}\label{Dnewkernel}
For $t>0$, let 
\index{qXbtxY@$\mathfrak 
q_{b,\vartheta,t}^{X \prime}\left(\left(x,Y\right),\left(x',Y'\right)\right)$}%
$ \mathfrak 
q^{X \prime}_{b,\vartheta,t}\left(\left(x,Y\right),\left(x',Y'\right)\right)$ denote  the smooth 
kernel associated with $\exp\left(-t\mathfrak{L}^{X \prime}_{b,\vartheta}\right)$. 
Also we use the notation
\index{qXb@$\mathfrak q^{X \prime}_{b,\vartheta}$}%
$\mathfrak q^{X \prime}_{b,\vartheta}$ instead of $\mathfrak 
q^{X \prime}_{b,\vartheta,1}$.

By (\ref{eq:co62}), we get
\begin{equation}
    \mathfrak 
    q_{b,\vartheta}^{X \prime}\left(\left(x,Y\right),\left(x',Y'\right)\right)=
    \exp\left(-\frac{b^{2}}{\cos\left(\vartheta\right)}\widehat{\alpha}\right)
    \underline{q}_{b,\vartheta}^{X \prime}\left(\left(x,Y\right),\left(x',Y'\right)\right)
    \exp\left(\frac{b^{2}}{\cos\left(\vartheta\right)}\widehat{\alpha}\right).
    \label{eq:bonn12a}
\end{equation}
\end{defin}

Put
\index{G@$\mathcal{G}$}%
\begin{equation}
   \mathcal{G}=\End\left(\Lambda\ac\left( \mathfrak 
   g^{*}\right)\right)\ho 
    \Lambda\ac\left(\underline{ \mathfrak 
    z}\left(\gamma\right)^{*}\right).
    \label{eq:bonn9}
\end{equation}

A basis $e_{1},\ldots,e_{m+n}$ of $\mathfrak g$ is said to be 
unimodular if the determinant of $B$ on this basis is equal to 
$\left(-1\right)^{n}$. Let $e_{1},\ldots,e_{m+n}$ be a unimodular basis of $ \mathfrak g$, 
such that $e_{1},\ldots,e_{r}$ is a basis of $\mathfrak 
z\left(\gamma\right)$. Let 
$\underline{e}^{1},\ldots,\underline{e}^{r}$ be the basis of 
$\underline{ \mathfrak z}\left(\gamma\right)^{*}$ that is dual to 
$e_{1},\ldots,e_{r}$.
Let 
\index{Trs@$\widehat{\Trs}$}%
$\widehat{\Trs}$ be the linear map from 
$\mathcal{G}$ into $\R$ that, up to permutation, vanishes on all the 
monomials in the $c\left(e_{i}\right),\widehat{c}\left(e_{i}\right), 
1\le i\le m+n,\underline{e}^{j},1\le j\le r$ 
except on $c\left(e_{1}\right)\underline{e}^{1}\ldots 
c\left(e_{r}\right)\underline{e}^{r}c\left(e_{r+1}\right)\widehat{c}\left(e_{r+1}\right)\ldots
c\left(e_{m+n}\right)\widehat{c}\left(e_{m+n}\right)$, and moreover,
\begin{multline}
    \widehat{\Trs}\left[c\left(e_{1}\right)\underline{e}^{1}\ldots 
c\left(e_{r}\right)\underline{e}^{r}c\left(e_{r+1}\right)\widehat{c}\left(e_{r+1}\right)\ldots
c\left(e_{m+n}\right)\widehat{c}\left(e_{m+n}\right)\right]\\
=\left(-1\right)^{r}\left(-2\right)^{m+n-r}
\left(-1\right)^{n-q}.
    \label{eq:bonn9a}
\end{multline}
If we assume that $e_{r+1},\ldots,e_{m+n}$ is a basis of $\mathfrak 
z^{\perp}\left(\gamma\right)$, and 
that $e^{*}_{r+1},\ldots,e^{*}_{m+n}$ is the dual basis to 
$e_{r+1},\ldots,e_{m+n}$ with respect to $B\vert_{ \mathfrak 
z^{\perp}\left(\gamma\right)}$, then (\ref{eq:bonn9a}) can be replaced by
\begin{multline}
    \widehat{\Trs}\left[c\left(e_{1}\right)\underline{e}^{1}\ldots 
c\left(e_{r}\right)\underline{e}^{r}c\left(e^{*}_{r+1}\right)\widehat{c}\left(e_{r+1}\right)\ldots
c\left(e^{*}_{m+n}\right)\widehat{c}\left(e_{m+n}\right)\right]\\
=\left(-1\right)^{r}\left(
-2\right)^{m+n-r}.
    \label{eq:bonn9ax1}
\end{multline}

We can extend the map $\widehat{\Trs}$ to a map 
$\widehat{\Trs}^{\odd}$ that maps $\mathcal{G} \ho 
c\left(\overline{TX}\right) \otimes \End\left(F\right)$ into $\C$.

Set
\index{rXbt@$\mathfrak r^{X \prime 
}_{\bt}\left(\left(x,Y\right),\left(x',Y'\right)\right)$}%
\begin{multline}\label{eq:fair3}
\mathfrak r^{X \prime 
}_{\bt}\left(\left(x,Y\right),\left(x',Y'\right)\right)=-\frac{d\vartheta}{\sqrt{2}b}\Biggl(\frac{b^{3}\sin\left(\vartheta\right)}{\cos^{2}\left(\vartheta\right)}
ic\left(\left[Y^{N},Y^{TX}\right] 
\right)\\
+\frac{\sin\left(\theta\right)}{b}\left(\mathcal{D}^{TX}-i\mathcal{D}^{N}\right) 
+\frac{b}{\cos^{2}\left(\vartheta\right)}
\widehat{c}\left(\overline{Y}^{TX}\right)\Biggr)\mathfrak q^{X \prime 
}_{\bt}\left(\left(x,Y\right),\left(x',Y'\right)\right).
\end{multline}

We recall a result in \cite[Proposition 9.5.4]{Bismut08b}.
\begin{prop}\label{Pconjubis}
For $b>0$, the following identity holds:
\begin{equation}
    \Trs^{\odd}\left[\gamma\underline{r}^{X \prime}_{b,\vartheta}\left(\left(x,Y\right),\gamma\left(x,Y\right)\right)\right]
    =\left(b^{2}/\cos\left(\vartheta\right)\right)^{-r}\widehat{\Trs}^{\odd}\left[\gamma\mathfrak 
    r_{b,\vartheta}^{X \prime}\left(\left(x,Y\right),\gamma\left(x,Y\right)\right)\right].
    \label{eq:bonn13}
\end{equation}
\end{prop}

In the sequel, the norm of $ \mathfrak 
r_{b,\vartheta}^{X}\left(\left(x,Y\right),\left(x',Y'\right)\right)$ will be 
evaluated with respect to the norms of $\Lambda\ac\left(T^{*}X \oplus 
N^{*}\right),\Lambda\ac\left(\left(\underline{T}X \oplus 
\underline{N}\right)\left(\gamma\right)^{*}\right)$, 
$S^{\overline{TX}}$, and  $F$.

Now, we give an extension of \cite[Theorem 9.5.6]{Bismut08b}.
\begin{thm}\label{TBEAUTY}
Given $\beta>0$, there exist $C>0,C'_{\gamma}>0$ such that for $b\ge 
1,\vartheta\in\left[0,\frac{\pi}{2}\right[,f\in 
\mathfrak p^{\perp}\left(\gamma\right), \left\vert  f\right\vert\le 
\beta b^{2}/\cos\left(\vartheta\right)$, and $\left\vert  
Y^{TX}\right\vert\le\beta b^{2}/\cos\left(\vartheta\right)$,
\begin{multline}
 \left(b^{2}/\cos\left(\vartheta\right)\right)^{-2m-n} \\
 \Biggl\vert
\mathfrak q 
_{b,\vartheta}^{X\prime}\Biggl(\left(e^{\cos\left(\vartheta\right)f/b^{2}},a^{TX}_{e^{\ct f/b^{2}}}+\cos\left(\vartheta\right)Y^{TX}/b^{2},\Yok+\cos\left(\vartheta\right)
Y^{N,\perp}/b^{2}
 \right),\\
 \gamma\left(e^{\cos\left(\vartheta\right)f/b^{2}},a^{TX}_{e^{\ct 
 f/b^{2}}}+\cos\left(\vartheta\right)Y^{TX}/b^{2},\Yok+\cos\left(\vartheta\right)Y^{N,\perp}/b^{2}\right)\Biggr)
 \Biggr\vert \\
					\le C
					\exp\Biggl(-C'\left(\frac{\left\vert  a\right\vert^{2}}{\cos^{2}\left(\vartheta\right)}+ \left\vert 
					\Yok\right\vert^{2} \right) \\
					-C' _{\gamma}\left(\left\vert f\right\vert^{2}+ \left\vert
					Y^{TX}\right\vert^{2} \right) -C' _{\gamma}\left\vert \left(
					\Ad\left(k^{-1}\right)-1\right)Y^{N,\perp}\right\vert  -C'_{\gamma}\left\vert
					\left[a,Y^{N,\perp}\right]\right\vert \Biggr).
  \label{eq:mossbrugg3bis}
	\end{multline}
The above inequality remains valid when replacing $\mathfrak q^{X 
\prime }_{\bt}$ by $\frac{\ct}{b^{2}}\n^{V}\mathfrak q^{X \prime }_{\bt}$.
\end{thm}
\begin{proof}
The proof of our theorem will be given in section \ref{sec:unilar}.
\end{proof}
\subsection{The limit of the rescaled heat kernel}%
\label{subsec:limloc}
By (\ref{eq:final3}),  (\ref{eq:bonn13}), we get an analogue of 
\cite[eq. (9.6.1)]{Bismut08b}
\begin{multline}
    \int_{\substack{\left\vert  f\right\vert\le \beta\\
    \left\vert  Y^{TX}-a^{TX}\right\vert\le \beta}}
    ^{}\Trs^{\odd}
    \left[\gamma 
    \underline{r}^{X\prime}_{b,\vartheta}\left(\left(e^{f},Y\right),\gamma\left(e^{f},Y\right)\right)\right]r\left(f\right)dY^{TX}dY^{N}df\\
    =
    \int_{\substack{\left\vert  f\right\vert\le \beta b^{2}/\cos\left(\vartheta\right)\\
	\left\vert  Y^{TX}\right\vert\le \beta b^{2}/\cos\left(\vartheta\right)}}
	^{} 
	\left(b^{2}/\cos\left(\vartheta\right)\right)^{-2m-n}\widehat{\Trs}^{\odd}
	\Biggl[\gamma
	\mathfrak r 
	_{b,\vartheta}^{X\prime}\Biggl(\Bigl(e^{\cos\left(\vartheta\right)f/b^{2}},\\
	a^{TX}_{e^{\ct f/b^{2}}}+\cos\left(\vartheta\right)Y^{TX}/b^{2},\Yok+\cos\left(\vartheta\right)Y^{N,\perp}/b^{2}
	\Bigr),\\
\gamma\left(e^{\cos\left(\vartheta\right)f/b^{2}},a^{TX}_{e^{\ct 
f/b^{2}}}+\cos\left(\vartheta\right)Y^{TX}/b^{2},\Yok+
\cos\left(\vartheta\right)Y^{N,\perp}/b^{2}\right)\Biggr)\Biggr]\\
	r\left(\cos\left(\vartheta\right)f/b^{2}\right)dY^{TX}d\Yok dY^{N,\perp}df.
    \label{eq:final3bis}
\end{multline}

We consider the vector space 
\begin{equation}\label{eq:vspa1}
\mathfrak p\times \mathfrak g= \mathfrak p\times \left(\mathfrak p 
\oplus \mathfrak k\right).
\end{equation}
Recall that 
\index{Dpk@$\Delta^{ \mathfrak p \oplus \mathfrak k}$}%
$\Delta^{ \mathfrak p \oplus \mathfrak k}$ is the standard 
Laplacian on $\mathfrak p \oplus \mathfrak k$.
	This operator acts along 
    $\mathfrak p \oplus \mathfrak k$, i.e., on the second factor in 
    the right-hand side of (\ref{eq:vspa1}).
    
Let
\index{nH@$\n^{H}$}%
$\n^{H}$  denote differentiation in the  variable $y\in 
\mathfrak p$, and let
\index{nV@$\n^{V}$}%
$\n^{V}$ denote differentiation in the variable $Y^{ 
\mathfrak g}\in \mathfrak g$.

Let $e_{1},\ldots,e_{p}$ be an orthonormal basis of $\mathfrak 
p\left(\gamma\right)$, let $e_{p+1},\ldots,e_{r}$ be an orthonormal 
basis of $\mathfrak k\left(\gamma\right)$. We denote with upper 
scripts the corresponding dual bases. Then 
$\underline{e}^{1},\ldots,\underline{e}^{r}$ is a basis of 
$\underline{ \mathfrak z}\left(\gamma\right)^{*}$.

Put
\index{a@$\alpha$}%
  \begin{equation}
      \alpha=\sum_{i=1}^{r}c\left(e_{i}\right)\underline{e}^{i}.
      \label{eq:pins3}
  \end{equation}
The notation in (\ref{eq:pins3}) is compatible with 
(\ref{eq:bruggmoss1}).
\begin{defin}\label{Dneke}
Given $\Yok\in \mathfrak k\left(\gamma\right)$, let 
\index{PaY@$\mathcal{P}_{a,\Yok}$ }%
$\mathcal{P}_{a,\Yok}$ be the differential operator on $\mathfrak 
p\times \mathfrak g$ that was 
defined in \cite[Definition 5.1.2]{Bismut08b}, i.e., 
\begin{multline}\label{eq:co69x1}
	\mathcal{P}_{a,\Yok}=	\frac{1}{2}\left\vert  
	\left[Y^{ \mathfrak k},
	a\right]+\left[\Yok,Y^{ \mathfrak 
	p}\right]\right\vert^{2}
	-
	\frac{1}{2}\Delta^{\mathfrak p\oplus\mathfrak k}
	 +\alpha
	-\n^{H}_{Y^{ \mathfrak p}}\\
	-\n^{V}_{\left[a
	+\Yok,\left[a,y\right]\right]}
	-\widehat{c}\left(\ad\left(a\right)\right)	+c\left(
	\ad\left(a\right)+i\theta\ad\left(\Yok\right)\right).
\end{multline}
Let 
\index{RYy@$R_{\Yok}\left(\left(y,Y\right),\left(y',Y'\right)\right)$}%
$R_{\Yok}\left(\left(y,Y\right),\left(y',Y'\right)\right)$ be the 
smooth kernel for $\exp\left(-\mathcal{P}_{a,\Yok}\right)$ with 
respect to  $dy'dY'$.
\end{defin}

We have the following convergence result that is an analogue of 
\cite[Theorem 9.6.1]{Bismut08b}.
\begin{thm}\label{Tpia}
For  $\vartheta\in\left[0,\frac{\pi}{2}\right[$, as $b\to + \infty 
$, we have the convergence,
\begin{multline}\label{eq:cosu1}
\left(b^{2}/\cos\left(\vartheta\right)\right)^{-2m-n}
\widehat{\Trs}^{\odd}\Biggl[\gamma \mathfrak r^{X \prime}_{b,\vartheta}
\Biggl(e^{\cos\left(\vartheta\right)/b^{2}}f,\\ 
a^{TX}_{e^{\ct f/b^{2}}}+\cos\left(\vartheta\right)Y^{TX}/b^{2},
\Yok+\ct 
Y^{N,\perp}/b^{2}\Biggr),\\
\gamma\left(e^{\cos\left(\vartheta\right)f/b^{2}}, 
a^{TX}_{e^{\ct f/b^{2}}}+\cos\left(\vartheta\right)Y^{TX}/b^{2}, \Yok+\ct 
Y^{N,\perp}/b^{2}\right)\Biggr]\\
\to
- \frac{d\vartheta }{\sqrt{2}}\exp\left(-\left\vert  a\right\vert^{2}/2\cos^{2}\left(\vartheta\right)
\right)
\\
\widehat{\Trs}\left[
\Ad\left(k^{-1}\right)
    R_{\Yok}\left(\left(f,Y\right),\Ad\left(k^{-1}\right)
    \left(f,Y\right)\right)\right]\\
    \Tr^{S^{\mathfrak 
    p}}\left[\frac{\widehat{c}\left(a\right)}{\cos^{2}\left(\vartheta\right)}
 \Ad\left(k^{-1}\right)\exp\left(-i\widehat{c}\left(\ad\left(\Yok
   \right)\vert_{\mathfrak p}\right)\right)\right]\\
    \Tr^{E}\left[\rho^{E}\left(k^{-1}\right)
    \exp\left(-i\rho^{E}\left(\Yok\right)\right)\right]
    \exp\left(-\left\vert  \Yok\right\vert^{2}/2\right)
\end{multline}
\end{thm}
\begin{proof}
The proof of this result will be given in subsections  
\ref{subsec:tras}--\ref{subsec:propi}.
\end{proof}
\subsection{A proof of Theorem \ref{Tlimi}}%
\label{subsec:prlimi}
By (\ref{eq:bibi1}), and by the considerations following 
(\ref{eq:spli1y1}), 
we may as well study the behaviour of the integral in 
 (\ref{eq:final3bis}). By equation (\ref{eq:mossbrugg3bis}) in 
 Theorem \ref{TBEAUTY} and by the last statement in this theorem, we get the uniform bound (\ref{eq:bub3}). 
Using dominated convergence, we deduce from Theorems \ref{TBEAUTY} and  \ref{Tpia} that as $b\to 0$, 
 \begin{multline}\label{eq:co73}
\Trs^{\left[\gamma\right],\odd}\left[
\exp\left(-L^{X \prime }\vert_{db=0}\right)\right]\\
\to -\frac{d\vartheta}{\sqrt{2}} \exp\left(-\left\vert  a\right\vert^{2}/2\cos^{2}\left(\vartheta\right)\right)
\int_{\substack{\left(y,Y^{ \mathfrak 
g},\Yok\right)\\
\in \mpg\times \left(\mathfrak p \oplus \kpg\right)\times\mathfrak 
k\left(\gamma\right)}}^{}\\
\widehat{\Trs}\left[
\Ad\left(k^{-1}\right)
    R_{\Yok}\left(\left(y,Y \right),\Ad\left(k^{-1}\right)
    \left(y,Y\right)\right)\right]\\
    \Tr^{S^{\mathfrak p}}
    \left[\frac{\widehat{c}\left(a\right)}{\cos^{2}\left(\vartheta\right)}
    \Ad\left(k^{-1}\right)\exp\left(-i\widehat{c}\left(\ad\left(\Yok\right)\vert_{\mathfrak p}\right)\right)\right] \\
    \Tr^{E}\left[\rho^{E}\left(k^{-1}\right)
    \exp\left(-i\rho^{E}\left(\Yok\right)\right)\right]
    \exp\left(-\left\vert  \Yok\right\vert^{2}/2
    \right)dydY^{ 
    \mathfrak g}d\Yok.
\end{multline}
\begin{thm}\label{Teval}
If $\vartheta\in\left[0,\frac{\pi}{2}\right[$, the following identity holds:
\begin{multline}\label{eq:co74}
\exp\left(-\left\vert  a\right\vert^{2}/2\cos^{2}\left(\vartheta\right)\right)\int_{\substack{\left(y,Y^{ \mathfrak 
g},\Yok\right)\\
\in \mpg\times \left(\mathfrak p \oplus \kpg\right)\times\mathfrak 
k\left(\gamma\right)}}^{}\\
\widehat{\Trs}\left[
\Ad\left(k^{-1}\right)
    R_{\Yok}\left(\left(y,Y \right),\Ad\left(k^{-1}\right)
    \left(y,Y  \right)\right)\right]\\
    \Tr^{S^{\mathfrak p}}
    \left[\frac{\widehat{c}\left(a\right)}{\cos^{2}\left(\vartheta\right)}
    \Ad\left(k^{-1}\right)\exp\left(-i\widehat{c}\left(\ad\left(\Yok\right)\vert_{\mathfrak p}\right)\right)\right] \\
    \Tr^{E}\left[\rho^{E}\left(k^{-1}\right)
    \exp\left(-i\rho^{E}\left(\Yok\right)\right)\right]
    \exp\left(-\left\vert  \Yok\right\vert^{2}/2\right)dydY^{ 
    \mathfrak g}d\Yok\\
    =\left(2\pi\right)^{-r/2}\exp\left(-\left\vert  a\right\vert^{2}
    /2\cos^{2}\left(\vartheta\right) \right) 
    \int_{\mathfrak k\left(\gamma\right)}^{}
    J_{\gamma}\left(\Yok\right)\\
    \Tr^{S^{\mathfrak p}}
    \left[\frac{\widehat{c}\left(a\right)}{\cos^{2}\left(\vartheta\right)}
    \Ad\left(k^{-1}\right)\exp\left(-i\widehat{c}\left(\ad\left(\Yok\right)\vert_{\mathfrak p}\right)\right)\right] \\
    \Tr^{E}\left[\rho^{E}\left(k^{-1}\right)
    \exp\left(-i\rho^{E}\left(\Yok\right)\right)\right]
    \exp\left(-\left\vert  \Yok\right\vert^{2}/2\right)
    d\Yok.
\end{multline}
\end{thm}
\begin{proof}
Equation (\ref{eq:co74}) follows from \cite[eq. (5.1.11) and  Theorem 
5.5.1]{Bismut08b}.
\end{proof}

Equation (\ref{eq:key2}) in Theorem \ref{Tlimi} now follows from 
(\ref{eq:co73}), (\ref{eq:co74}). This completes the proof of Theorem 
\ref{Tlimi}.
\subsection{A translation of the variable $Y^{TX}$}%
\label{subsec:tras}
We begin the proof of Theorem \ref{Tpia}. We proceed as in 
\cite[section 9.8]{Bismut08b}. 

Recall that $a\in \mathfrak p$. As explained in \cite[section 
2.17]{Bismut08b}, there are sections $a^{TX},a^{N}$ of 
$TX,N$ that are such that via the identification $TX 
\oplus N= \mathfrak g$, we have the identity $a=a^{TX}+a^{N}$. The 
section $a^{TX}$ is just the vector field associated with $a\in 
\mathfrak g$ via the left action of $G$ on $X$.  By 
\cite[eq. (2.17.10)]{Bismut08b}, if $A\in TX$, then
\begin{align}\label{eq:grzu10}
&\n^{TX}_{A}a^{TX}+\left[A,a^{N}\right]=0,
&\n^{N}_{A}a^{N}+\left[A,a^{TX}\right]=0.
\end{align}

If $ e\in TX \oplus N$, 
\index{nV@$\n^{V}_{e}$}%
$\n^{V}_{e}$ denotes 
differentiation along $e$.
Let $L_{a}$ be the Lie derivative operator associated with $a\in 
\mathfrak p$. Then $L_{a}$ acts on $\mathcal{H}$. By \cite[eq. 
(2.18.1)]{Bismut08b},  this action is 
given by
\begin{equation}
    L_{a}=\n^{C^{ \infty }\left(TX \oplus 
N,\widehat{\pi}^{*}\left(\Lambda\ac\left(T^{*}X \oplus N^{*}\right)\otimes 
S^{\overline{TX}}\otimes 
F\right)\right)}_{a^{TX}}+L^{V}_{\left[a^{N},Y\right]}
-\widehat{c}\left(\ad\left(a^{N}\right)\vert_{\overline{TX}}\right)-\rho^{F}\left(a^{N}\right).
    \label{eq:cupr32}
\end{equation}
In (\ref{eq:cupr32}), the fibrewise Lie derivative operator 
\index{LV@$L^{V}_{\left[a^{N},Y\right]}$}%
$L^{V}_{\left[a^{N},Y\right]}$ is given by  
\begin{equation}
    L^{V}_{\left[a^{N},Y\right]}=\n^{V}_{\left[a^{N},Y\right]}-\left(c+\widehat{c}\right)
    \left(\ad\left(a^{N}\right)\right).
    \label{eq:cupr33}
\end{equation}

Now we follow \cite[Definition 9.8.1]{Bismut08b}.
\begin{defin}\label{Dtrans}
Set
\index{Ta@$T_{a}$}
\begin{equation}
    T_{a}s\left(x,Y^{TX},Y^{N}\right)=s\left(x,a^{TX}+Y^{TX},Y^{N}\right).
    \label{eq:cips1}
\end{equation}
\end{defin}

By \cite[eq. (9.8.7)]{Bismut08b} or by using (\ref{eq:cupr32}), 
(\ref{eq:cupr33}),
we get
\begin{equation}\label{eq:cips1x1}
T_{a}L_{a}T_{a}^{-1}=L_{a}.
\end{equation}
By \cite[Proposition 9.3.2]{Bismut08b}, we get
\begin{equation}\label{eq:cips1z1}
\exp\left(-\frac{b^{2}}{\ct}\widehat{\alpha}\right)L_{a}\exp\left(\frac{b^{2}}{\ct}\widehat{\alpha}\right)=
L_{a}.
\end{equation}

As in \cite[eq. (9.8.2)]{Bismut08b}, set
\index{NXabt@$\mathfrak N^{X \prime}_{a,\bt}$}%
\index{NXa@$\mathfrak N^{X \prime}_{a}$}%
\begin{align}
	&\mathfrak N^{X \prime}_{a,\bt}=T_{a}\mathfrak L^{X 
	\prime}_{\bt}T_{a}^{-1},
   &\mathfrak N^{X \prime}_{a}=T_{a}\mathfrak{L}^{X \prime 
    }T_{a}^{-1}.
    \label{eq:cips2}
\end{align}

 Now we establish an extension of 
\cite[Proposition 9.8.2]{Bismut08b}.
\begin{prop}\label{POX}
The following identity holds:
\begin{multline}\label{eq:cips3}
\mathfrak N^{X \prime}_{a}=\frac{b^{4}}{2\cos^{2}\left(\vartheta\right)}\left\vert  \left[Y^{N},
    a^{TX}+Y^{TX}\right]\right\vert^{2}
     +\frac{1}{2} \Biggl( -\frac{\cos^{2}\left(\vartheta\right)}{b^{4}}\Delta^{ TX \oplus 
     N}\\
     +\frac{1}{\cos^{2}\left(\vartheta\right)}\left\vert 
     a^{TX}+
	  Y^{TX}\right\vert^{2}
	  +\left\vert  
	  Y^{N}\right\vert^{2}-\frac{1}{b^{2}}\left(m+\cos\left(\vartheta\right)n\right)\Biggr)\\
	 +\frac{N_{-\vartheta}^{\Lambda\ac\left(T^{*}X 
	 \oplus N^{*}\right) \prime }}{b^{2}}
    +\alpha
    -\Biggl(\n_{a^{TX}+Y^{ TX}}^{C^{ \infty }\left(TX \oplus 
     N,\widehat{\pi}^{*}\left(\Lambda\ac\left(T^{*}X \oplus N^{*}\right)\otimes 
     S^{\overline{TX}} \otimes  F\right) \right), f*,\widehat{f}}\\
     - ic \left( \theta\ad\left(Y^{N}\right)\right)  
     -i\widehat{c}\left(\ad\left(Y^{N}\right)\vert_{\overline{TX}}\right)
    -i\rho^{F}\left(Y^{ 
    N}\right)\Biggr)
    +\n^{V}_{\left[a^{N},a^{TX}+Y^{TX}\right]}\\
    +\frac{d\vartheta}{\sqrt{2}b}\Biggl(
\frac{b^{3}\sin\left(\vartheta\right)}{\cos^{2}\left(\vartheta\right)}
ic\left(\left[Y^{N},a^{TX}+Y^{TX}\right]\right)\\
+\frac{\sin\left(\vartheta\right)}{b}\left(\mathcal{D}^{TX}-i\mathcal{D}
^{N}\right)+\frac{b}{\cos^{2}\left(\vartheta\right)}\widehat{c}\left(\overline{a}^{TX}+
\overline{Y}^{TX} \right) \Biggr).
\end{multline}
\end{prop}
\begin{proof}
Using (\ref{eq:glub2bis1}), (\ref{eq:glub2bis1x}), (\ref{eq:co63}), and (\ref{eq:grzu10}), we get (\ref{eq:cips3}).
\end{proof}

Put
\index{OXabt@$\mathfrak{O}^{X \prime}_{a,b,\vartheta}$}%
\index{OXa@$\mathfrak O^{X \prime}_{a}$}%
\begin{align}
    &\mathfrak{O}^{X \prime}_{a,b,\vartheta}=\mathfrak{N}^{X \prime 
	}_{a,b,\vartheta}+L_{a},
	&\mathfrak O^{X \prime}_{a}=\mathfrak N^{X \prime }_{a}+L_{a}.
    \label{eq:cips4}
\end{align}

By (\ref{eq:co62}), (\ref{eq:cips1x1})--(\ref{eq:cips2}), and 
(\ref{eq:cips4}), we get the analogue of \cite[eq. 
(9.8.8)]{Bismut08b},
\begin{align}\label{eq:cips4z-1}
& \mathfrak{O}^{X 
 \prime}_{a,\bt}=T_{a}\exp\left(-\frac{b^{2}}{\ct}\widehat{\alpha}\right)
 \left(\underline{\mathcal{L}}^{X \prime }_{\bt}
 +L_{a}\right)\exp\left(\frac{b^{2}}{\ct}\widehat{\alpha}\right)T_{a}^{-1},\\
& \mathfrak{O}^{X 
 \prime}_{a}=T_{a}\exp\left(-\frac{b^{2}}{\ct}\widehat{\alpha}\right)
 \left(\underline{L}^{X \prime }\vert_{db=0}
 +L_{a}\right)\exp\left(\frac{b^{2}}{\ct}\widehat{\alpha}\right)T_{a}^{-1}. \nonumber 
\end{align}

Now we extend \cite[Proposition 9.8.3]{Bismut08b}.
\begin{prop}\label{POX1}
The following identity holds:
\begin{multline}\label{eq:cips5}
    \mathfrak{O}^{X \prime}_{a}=\frac{b^{4}}{2\cos^{2}\left(\vartheta\right)}
    \left\vert  
    \left[Y^{N},a^{TX}+Y^{ 
    TX}\right]\right\vert^{2}
     +\frac{1}{2} \Biggl( -\frac{\cos^{2}\left(\vartheta\right)}{b^{4}}
     \Delta^{ TX \oplus 
     N}\\
     +\frac{1}{\cos^{2}\left(\vartheta\right)}\left\vert  
    a^{TX}+
	  Y^{TX}\right\vert^{2}
	  +\left\vert  
	  Y^{N}\right\vert^{2}-\frac{1}{b^{2}}\left(m+\cos\left(\vartheta\right)n\right)\Biggr)\\
    +\frac{N^{\Lambda\ac\left( 
	 T^{*}X 
	 \oplus N^{*}\right) \prime }_{-\vartheta}}{b^{2}}+
	 \alpha-\Biggl(\n_{Y^{ TX}}^{C^{ \infty }\left(TX \oplus 
     N,\widehat{\pi}^{*}\left(\Lambda\ac\left(T^{*}X \oplus 
	 N^{*}\right)\otimes S^{\overline{TX}} \otimes  
     F\right) \right), f*,\widehat{f}}\\
     - ic \left( \theta\ad\left(Y^{N}\right)\right)  
     -i\widehat{c}\left(\ad\left(Y^{N}\right)\vert_{\overline{TX}}\right)
    -i\rho^{F}\left(Y^{ 
    N}\right)\Biggr)\\
    +c\left(\ad\left(a^{TX}-a^{N}\right)\right)-\widehat{c}\left(\ad\left(a\right)\right)
     -\widehat{c}\left(\ad\left(a^{N}\right)\vert_{\overline{TX}}\right) 
     -\rho^{F}\left(a^{N}\right)
    \\
   +\n^{V}_{\left[a^{N},a^{TX}+2Y^{TX}+Y^{N}\right]}+\frac{d\vartheta}{\sqrt{2}b}\Biggl(
\frac{b^{3}\sin\left(\vartheta\right)}{\cos^{2}\left(\vartheta\right)}
ic\left(\left[Y^{N},a^{TX}+Y^{TX}\right]\right)\\
+\frac{\sin\left(\vartheta\right)}{b}\left(\mathcal{D}^{TX}-i\mathcal{D}
^{N}\right)+\frac{b}{\cos^{2}\left(\vartheta\right)}\widehat{c}\left(\overline{a}^{TX}+
\overline{Y}^{TX} \right) \Biggr).
\end{multline}
\end{prop}
\begin{proof}
By  (\ref{eq:glub1}), (\ref{eq:cupr32}), (\ref{eq:cupr33}),   
(\ref{eq:cips3}), and (\ref{eq:cips4}), we get (\ref{eq:cips5}).
\end{proof}
\begin{remk}\label{Rbo}
Most of the 
terms in equation (\ref{eq:cips5}) for $\mathfrak{O}^{X \prime 
}_{a}\vert_{d\vartheta=0}=\mathfrak O^{X \prime }_{a,\bt}$ can be obtained from the corresponding terms in 
\cite[eq. (9.8.5)]{Bismut08b} for 
$\mathcal{O}^{X}_{a,b}=\mathcal{O}^{X 
\prime}_{a,b,0}\vert_{d\vartheta=0}$ by 
replacing $b$ by $b/\cos^{1/2}\left(\vartheta\right)$. Since 
$b/\cos^{1/2}\left(\vartheta\right)\ge b$, when $b\to + \infty $, the 
presence of $\vartheta$ improves the situation with respect to 
\cite{Bismut08b}. The only term for which this is not the case is 
$\frac{1}{2\cos^{2}\left(\vartheta\right)}\left\vert  
a^{TX}+Y^{TX}\right\vert^{2}$. However, because this term is positive, 
the fact that it is larger than when $\vartheta=0$ will  work again in 
our favour.
\end{remk}
\subsection{A coordinate system and a trivialization of  vector 
bundles near $X\left(\gamma\right)$}%
\label{subsec:cotri}
If $x_{0}=p1\in X\left(\gamma\right)$, we take the same coordinate 
system on $X$
near $x_{0}$ and the same trivialization of the vector bundles  as in \cite[section 
9.9]{Bismut08b}.  In particular, $S^{\overline{TX}}$ is treated 
exactly like $F$. 

We proceed  as in  \cite[Definitions 9.9.1 and 
9.10.1]{Bismut08b}. Let $\n^{H}$ denote differentiation in the 
coordinate $y\in \mathfrak p$, and $\n^{V}$ denote differentiation in 
$Y\in \mathfrak g$. Also $Y\in \mathfrak g$ splits as 
$Y=Y^{\mathfrak p}+Y^{\mathfrak k},Y^{\mathfrak p}\in \mathfrak 
p,Y^{\mathfrak k}\in \mathfrak k$.
\begin{defin}\label{DHYo} 
If $\Yok\in \mathfrak k\left(\gamma\right)$ \footnote{In \cite[Definition 
9.9.1]{Bismut08b}, $\Yok\in \mathfrak k^{\perp}\left(\gamma\right)$ 
should be corrected to $\Yok\in \mathfrak k\left(\gamma\right)$.}, set
\index{HbtY@$H_{b,\vartheta,\Yok}$}%
\begin{equation}
    H_{b,\vartheta,\Yok}s\left(y,Y\right)=
    s\left(\cos\left(\vartheta\right)y/b^{2},\Yok+\cos\left(\vartheta\right)Y/b^{2}\right).
    \label{eq:rips3}
\end{equation}
Put
\index{PXab@$\mathfrak{P}^{X \prime}_{a,b,\vartheta,\Yok}$}%
\index{PXab@$\mathfrak{P}^{X \prime}_{a,\Yok}$}%
    \begin{align}\label{eq:cupr18}
&\mathfrak{P}^{X \prime 
	}_{a,b,\vartheta,\Yok}=H_{b,\vartheta,\Yok}\mathfrak{O}^{X \prime 
	}_{a,b,\vartheta} H_{b,\vartheta,\Yok}^{-1}, \nonumber \\
&\mathfrak{S}^{X \prime}_{a,\bt,\Yok}=H_{b,\vartheta,\Yok}
\frac{d\vartheta}{\sqrt{2}b}\Biggl(
\frac{b^{3}\sin\left(\vartheta\right)}{\cos^{2}\left(\vartheta\right)}
ic\left(\left[Y^{N},a^{TX}+Y^{TX}\right]\right)  \\
&+\frac{\sin\left(\vartheta\right)}{b}\left(\mathcal{D}^{TX}-i\mathcal{D}
^{N}\right)+\frac{b}{\cos^{2}\left(\vartheta\right)}\widehat{c}\left(\overline{a}^{TX}+
\overline{Y}^{TX} \right) \Biggr)H_{b,\vartheta,\Yok}^{-1}, \nonumber 
\\
&\mathfrak{P}^{X \prime 
	}_{a,\Yok}=H_{b,\vartheta,\Yok}\mathfrak{O}^{X \prime 
	}_{a} H_{b,\vartheta,\Yok}^{-1}. \nonumber 
\end{align}
\end{defin}

By (\ref{eq:cips5}), (\ref{eq:cupr18}), we get
\begin{equation}\label{eq:cupr18z1}
\mathfrak P^{X \prime }_{a,\Yok}=\mathfrak P^{X \prime }_{a,\bt,\Yok}+
\mathfrak S^{X \prime }_{a,\bt,\Yok}.
\end{equation}
\begin{defin}\label{Dnewop}
 Let 
 \index{pXabt@$\mathfrak p^{X \prime}_{a,b,\vartheta,\Yok}\left(\left(y,Y\right),\left(y',Y'\right)\right)$}%
 $\mathfrak p^{X \prime}_{a,b,\vartheta,\Yok}\left(\left(y,Y\right),\left(y',Y'\right)\right)$ be the 
smooth kernel for the operator  $\exp\left(-\mathfrak{P}^{X 
\prime}_{a,b,\vartheta,\Yok}\right)$ with respect 
to $dy'dY'$. Set
\begin{equation}\label{eq:fufu8}
\mathfrak r^{X \prime 
}_{a,\bt,\Yok}\left(\left(y,Y\right),\left(y',Y'\right)\right)=-
\mathfrak{S}^{X \prime }_{a,\bt,\Yok}\mathfrak 
p^{X \prime}_{a,b,\vartheta,\Yok}\left(\left(y,Y\right),\left(y',Y'\right)\right).
\end{equation}
\end{defin}

By proceeding as in \cite[eq. (9.9.11)]{Bismut08b}, we get
\begin{multline}\label{eq:orsayx2-1}
    \left(b^{2}/\cos\left(\vartheta\right)\right)^{-2m-n}\gamma\mathfrak r _{b,\vartheta}^{X
    \prime } \Biggl( \\
   \left(e^{\ct f/b^{2}},a^{TX}+\ct 
    Y^{TX}/b^{2}, \Yok+\ct Y^{N,\perp}/b^{2}\right),\\
      \gamma\left(e^{\cos\left(\vartheta\right)f/b^{2}},a^{TX}+
      \cos\left(\vartheta\right)Y^{TX }/b^{2},\Yok+\cos\left(\vartheta\right)Y^{N,\perp 
      }/b^{2}\right)\Biggr)\delta\left(\ct f/b^{2}\right)\\
      =\Ad\left(k^{-1}\right)\vert_{\Lambda\ac\left(\mathfrak 
      g^{*}\right)}\otimes \rho^{S^{\overline{\mathfrak p}}}\left(k^{-1}\right)
      \otimes \rho^{E}\left(k^{-1}\right)
      \mathfrak 
      r^{X \prime 
      }_{a,b,\vartheta,\Yok}\left(\left(f,Y\right),\Ad\left(k^{-1}\right)\left(f,Y\right)\right).
\end{multline}
In (\ref{eq:orsayx2-1}), $\delta$ is the Jacobian of 
the geodesic exponential map based at $x_{0}=p1$. That function was 
 defined in \cite[eq. (4.1.11)]{Bismut08b}
\subsection{The asymptotics of the operator $\mathcal{P}^{X \prime 
}_{a,\Yok}$ as $b\to + \infty $}%
\label{subsec:aspy}
Recall that the operator 
\index{PaY@$\mathcal{P}_{a,\Yok}$ }%
$\mathcal{P}_{a,\Yok}$ was defined in 
Definition \ref{Dneke}.
\begin{defin}\label{Dstiop}
Set
	\index{PXa@$ \mathfrak{P}^{X \prime}_{a,\infty,\vartheta,\Yok}$}%
	\index{SXa@$\mathfrak S_{a,\infty ,\vartheta,\Yok}$}%
\begin{align}\label{eq:co70}
	\mathfrak{P}^{X \prime}_{a,\infty,\vartheta,\Yok}&=
	\mathcal{P}_{a,\Yok}+
	\frac{1}{2}\left(\frac{\left\vert  
	    a\right\vert^{2}}{\cos^{2}\left(\vartheta\right)}+\left\vert  \Yok\right\vert^{2}\right)+
	i\widehat{c}\left(\ad\left(\Yok\right)\vert_{\overline{\mathfrak p}}\right)+i\rho^{E}
	\left(\Yok\right), \nonumber \\
	\mathfrak S_{a,\infty ,\vartheta,\Yok}&=\frac{d\vartheta}{\sqrt{2}}
\Biggl( \tan\left(\vartheta\right)ic\Bigl(\left[Y^{\mathfrak 
k},a\right]
+\left[\Yok
,Y^{\mathfrak p}\right]\Bigr)\\
&\qquad\qquad +\tan\left(\vartheta\right)
\left(\mathcal{D}^{\mathfrak p}-i\mathcal{D}^{\mathfrak k}\right)+
\frac{\widehat{c}\left(\overline{a}\right)}{\cos^{2}\left(\vartheta\right)} \Biggr). \nonumber 
\end{align}
Put
\index{PXa@$\mathfrak P^{X \prime}_{a, \infty, \Yok}$}%
\begin{equation}\label{eq:co70za1}
\mathfrak P^{X \prime}_{a, \infty, \Yok}=\mathfrak{P}^{X \prime}_{a,\infty,\vartheta,\Yok} +\mathfrak S_{a,\infty ,\vartheta,\Yok}.
\end{equation}

Let 
\index{pa@$\mathfrak p^{X \prime}_{a,\infty ,\vartheta,\Yok}\left(\left(y,Y\right),\left(y',Y'\right)\right)$}%
$\mathfrak p^{X \prime}_{a,\infty 
,\vartheta,\Yok}\left(\left(y,Y\right),\left(y',Y'\right)\right)$ be the smooth kernel for  
 $\exp\left(-\mathfrak{P}^{X \prime }_{a, \infty,\vartheta,\Yok 
}\right)$ with respect to $dy'dY'$. Set
\begin{equation}\label{eq:co70x1}
\mathfrak r^{X \prime}_{a, \infty ,\vartheta,\Yok}
\left(\left(y,Y\right),\left(y',Y'\right)\right) 
=
- \mathfrak S_{a, \infty ,\vartheta,\Yok}
\mathfrak p^{X \prime}_{a, \infty, 
\vartheta,\Yok}\left(\left(y,Y\right),\left(y',Y'\right)\right).
\end{equation}
\end{defin}

By (\ref{eq:co70}), we get
\begin{multline}\label{eq:co70z-1}
\mathfrak p^{X \prime}_{a, \infty 
,\vartheta,\Yok}\left(\left(y,Y\right),\left(y',Y'\right)\right)=
\exp\left(-\frac{1}{2}\left(\frac{\left\vert  
a\right\vert^{2}}{\cos^{2}\left(\vartheta\right)}+\left\vert  
\Yok\right\vert^{2}\right)\right) \\
R_{\Yok}\left(\left(y,Y\right),\left(y',Y'\right)\right)
\exp\left(-i\widehat{c}\left(\ad\left(\Yok\right)\right)\vert_{\overline{\mathfrak p}}
\right)\exp\left(-i\rho^{E}\left(\Yok\right)\right).
\end{multline}

For $\vartheta=0$, $\mathfrak{P}^{X 
\prime}_{a,\infty,\vartheta,\Yok}$ coincides with the operator 
$\mathcal{P}^{X}_{a, 0,\infty ,\Yok}$ in \cite[eqs. (9.10.1), 
(9.10.2)]{Bismut08b}.

By \cite[eq. (9.10.4)]{Bismut08b}, we 
have
\begin{align}\label{eq:co68}
&a^{TX}_{y/b^{2}}=a+\mathcal{O}\left(\left\vert  
y\right\vert^{2}/b^{4}\right),
&a^{N}_{y/b^{2}}=\frac{\left[a,y\right]}{b^{2}}+\mathcal{O}\left(\left\vert  y\right\vert^{3}/b^{6}\right).
\end{align}

Using (\ref{eq:co69x1}), (\ref{eq:cips5})--(\ref{eq:cupr18}), (\ref{eq:co68}), and proceeding as in the proof of 
\cite[Theorem 9.10.2]{Bismut08b}, we find that 
 as $b\to + \infty $,
\begin{equation}\label{eq:co65}
\mathfrak{P}^{X \prime 
	}_{a,\Yok}\to \mathfrak{P}^{X 
	\prime}_{a,\infty,\Yok}.
\end{equation}
The convergence in (\ref{eq:co65}) just means that the coefficients 
of the operators together with their derivatives of arbitrary order 
converge uniformly on compact sets.

By Remark \ref{Rbo},  since with respect to \cite{Bismut08b}, for 
most terms in (\ref{eq:cips5}), the estimate of the difference in (\ref{eq:co65}) is 
better than in \cite{Bismut08b}, i.e., where positive powers of $1/b$ 
appear in \cite{Bismut08b}, they are replaced here by positive powers 
of $\cos^{1/2}\left(\vartheta\right)/b$, which are smaller.  The only 
exception comes from the term 
$\frac{1}{\cos^{2}\left(\vartheta\right)}\left\vert  
a^{TX}+Y^{TX}\right\vert^{2}$. However, because of (\ref{eq:co68}), 
proceeding as in \cite[eq. (9.10.6)]{Bismut08b}, we get
\begin{multline}\label{eq:rem1}
\frac{1}{\cos^{2}\left(\vartheta\right)}\left\vert  
a^{TX}+\cos\left(\vartheta\right)Y^{TX}/b^{2}\right\vert_{\cos\left(\vartheta\right)y/b^{2}}^{2}=
\frac{\left\vert  a\right\vert^{2}}{\cos^{2}\left(\vartheta\right)} \\
+\frac{1}{
\cos^{2}\left(\vartheta\right)}\left(\mathcal{O}\left(\cos^{4}\left(\vartheta\right)\left\vert  y\right\vert^{4}/b^{8}
+\cos^{2}\left(\vartheta\right)\left\vert  
Y^{TX}\right\vert^{2}/b^{4}\right)\right) \\
+\frac{1}{\cos^{2}\left(\vartheta\right)}
\mathcal{O}\left(\left\vert  
a\right\vert\left(\cos^{2}\left(\vartheta\right)\left\vert  
y\right\vert^{2}/b^{4}+\cos\left(\vartheta\right)\left\vert  
Y^{TX}\right\vert/b^{2}\right)\right).
\end{multline}
In (\ref{eq:rem1}), among the small terms as $b\to + \infty $, the 
only potentially annoying term, not uniform in $\vartheta$ is given by 
$\frac{1}{\cos^{2}\left(\vartheta\right)}\mathcal{O}\left(\ct\left\vert  
Y^{TX}\right\vert/b^{2}\right)$. But dominated convergence allows us 
to ignore these questions of uniformity.
\subsection{A proof of Theorem \ref{Tpia}}%
\label{subsec:propi}
Now, we establish an analogue of \cite[Theorem 9.11.1]{Bismut08b}.
\begin{thm}\label{Tlififi}
As $b\to + \infty $, 
\begin{equation}\label{eq:fifi1}
\mathfrak r^{X \prime 
}_{a,b,\vartheta,\Yok}\left(\left(y,Y\right),\left(y',Y'\right)\right)\to
\mathfrak r^{X \prime }_{a,\infty 
,\vartheta,\Yok}\left(\left(y,Y\right),\left(y',Y'\right)\right).
\end{equation}
\end{thm}
\begin{proof}
Given $\vartheta\in\left[0,\frac{\pi}{2}\right[$, the structure of 
the operator $\mathfrak{P}^{X \prime }_{a,\Yok}$ is very 
similar to the structure of the operator $\mathcal{P}^{X \prime 
}_{a,0,b,\Yok}$ that was considered in \cite[section 9.10 and 
15.10]{Bismut08b}. In view of (\ref{eq:co65}) and of the 
considerations that follow, the proof of (\ref{eq:fifi1}) is  exactly the 
same as the proof  of \cite[Theorem 9.11.1]{Bismut08b}.
\end{proof}

We are now ready to prove Theorem \ref{Tpia}. Indeed by
(\ref{eq:orsayx2-1})  and by equation (\ref{eq:fifi1}) in Theorem 
\ref{Tlififi}, as $b\to + \infty $, 
\begin{multline}\label{eq:fifi2}
\left(b^{2}/\cos\left(\vartheta\right)\right)^{-2m-n}
\widehat{\Trs}^{\odd}\Biggl[\gamma \mathfrak r^{X \prime}_{b,\vartheta}
\Biggl(e^{\cos\left(\vartheta\right)f/b^{2}}, 
a^{TX}_{e^{\ct f/b^{2}}}+\cos\left(\vartheta\right)Y^{TX}/b^{2},\Yok \\
+\ct 
Y^{N,\perp}/b^{2}\Biggr),
\gamma\Bigl(e^{\cos\left(\vartheta\right)f/b^{2}}, 
a^{TX}_{e^{\ct f/b^{2}}} 
+\cos\left(\vartheta\right)Y^{TX}/b^{2},  \\
\Yok+\ct 
Y^{N,\perp}/b^{2}\Bigr)\Biggr] 
\to
\widehat{\Trs}^{\odd}\Bigg[\Ad\left(k^{-1}\right)\vert_{\Lambda\ac\left(\mathfrak 
      g^{*}\right)}\otimes \rho^{S^{\overline{\mathfrak p}}}\left(k^{-1}\right)
      \otimes \rho^{E}\left(k^{-1}\right)\\
      \mathfrak 
      r^{X \prime 
      }_{a,\infty ,\vartheta,\Yok}\left(\left(f,Y\right),\Ad\left(k^{-1}\right)\left(f,Y\right)\right)
      \Biggr].
\end{multline}
By (\ref{eq:co69x1}), (\ref{eq:co70}),  the operator 
$\mathfrak{P}^{X \prime }_{a, \infty 
,\vartheta,\Yok}$ is even in 
every possible way, including in the Clifford variables in 
$c\left(\overline{\mathfrak p}\right)$. 
By  (\ref{eq:orsayx2-1}), (\ref{eq:co70x1}),  (\ref{eq:co70z-1}), and by
equation (\ref{eq:fifi1}) in Theorem \ref{Tlififi}, we get equation 
(\ref{eq:cosu1}) in Theorem \ref{Tpia}.
In that equation, we replaced 
$S^{\overline{\mathfrak p}}$ by $S^{\mathfrak p}$, since the 
distinction has now become irrelevant. This concludes the proof of 
Theorem \ref{Tpia}.

%% file: Eta7.tex
\section{An explicit formula for the odd orbital integrals}%
\label{sec:applic}
In this section,  as an application of the results of section 
\ref{sec:oddorb}, we give a simple formula for the 
orbital integrals 
$\Tr^{\left[\gamma\right]}\left[D^{X}\exp\left(-tD^{X,2}/2\right)\right]$. 
In particular, we recover all the results by Moscovici-Stanton 
\cite{MoscoviciStanton89} on the explicit evaluation of such orbital 
integrals.

This section is organized as follows. In subsection \ref{subsec:ref}, 
using Theorem \ref{Tkeyres},
we give a formula for 
$\Tr^{\left[\gamma\right]}\left[D^{X}\exp\left(-tD^{X,2}/2\right)\right]$.

In subsection \ref{subsec:van}, we find conditions under which the 
  above orbital integrals 
vanish identically. Our conditions are exactly the ones in Moscovici-Stanton 
\cite{MoscoviciStanton89}. 

In subsection \ref{subsec:siide}, we establish a simple convolution 
identity.

Finally, in subsection \ref{subsec:gen}, when the orbital integrals 
do not vanish, we give an explicit formula for these orbital 
integrals in terms of characteristic forms on 
$X\left(\gamma\right)$.  We recover this way the explicit geometric 
formulas of Moscovici-Stanton \cite{MoscoviciStanton89}.

We make the same assumptions and we use the same notation as in 
sections \ref{sec:pres} and \ref{sec:oddorb}. In particular $m$ is 
still assumed to be odd, and $\gamma$ to be semisimple and 
nonelliptic.
\subsection{A reformulation of Theorem \ref{Tkeyres}}%
\label{subsec:ref}
Let 
\index{Dzg@$\Delta^{ \mathfrak z\left(\gamma\right)}$}%
$\Delta^{ \mathfrak z\left(\gamma\right)}$ be the 
usual Laplacian on $\mathfrak z\left(\gamma\right)=\mathfrak 
p\left(\gamma\right) \oplus \mathfrak k\left(\gamma\right)$ with 
respect to the scalar product induced by the scalar product of 
$\mathfrak g$. For $t>0$, let $\exp\left(t\Delta^{ \mathfrak 
z\left(\gamma\right)}/2\right)$ be the corresponding heat operator, 
and let $\exp\left(t\Delta^{ \mathfrak 
z\left(\gamma\right)}/2\right)\left(\left(y,\Yok\right),\left(y',Y_{0}^{ \mathfrak k \prime }\right)\right)$ be the 
associated Gaussian heat kernel with respect 
to $dy'd\Yok$.

Let $\left(y,\Yok\right)$ be the generic element in $\mathfrak 
z\left(\gamma\right)=\mathfrak p\left(\gamma\right) \oplus \mathfrak 
k\left(\gamma\right)$. Observe that  
\begin{multline}\label{eq:laus7}
J_{\gamma}\left(\Yok\right)\Tr^{S^{\mathfrak p}}
\left[\frac{c\left(a\right)}{\left\vert  a\right\vert}\Ad\left(k^{-1}\right)
\exp\left(-ic
\left(\ad\left(\Yok\right)\vert_{\mathfrak p}\right)\right)\right]\\
\Tr^{E}\left[\rho^{E}\left(k^{-1}\right)\exp\left(-i\rho^{E}\left(\Yok\right)\right)\right]
\delta_{y=a}
\end{multline}
is a distribution on $\mathfrak 
z\left(\gamma\right)$.
For $t>0$, the heat operator  $\exp\left(t\Delta^{ 
\mathfrak z\left(\gamma\right)}/2\right)$ can be applied to this 
distribution and we obtain this way a smooth function on $\mathfrak 
z\left(\gamma\right)$. 
\begin{thm}\label{Tnew}
For any $t>0$, the following identity holds:
\begin{multline}\label{eq:conk4x1}
\Tr^{\left[\gamma\right]}\left[D^{X}\exp\left(-
sD^{X,2}/2\right)\right]*
\frac{1}{\sqrt{s}}\left(t\right)\\
=\sqrt{2\pi}\exp\left(\frac{t}{48}\Tr^{\mathfrak k}\left[C^{\mathfrak 
k, \mathfrak k}\right]+\frac{t}{2}C^{\mathfrak k,E}\right)\\
\frac{1}{\left(2\pi t\right)^{p/2}}
\int_{\mathfrak k\left(\gamma\right)
}^{}
\exp\left(-\frac{1}{2t}\left(\left\vert  
a\right\vert^{2}+\left\vert  \Yok\right\vert^{2}\right)\right)
J_{\gamma}\left(\Yok\right)\\
\Tr^{S^{\mathfrak p}}
\left[\frac{c\left(a\right)}{\left\vert  a\right\vert}\Ad\left(k^{-1}\right)
\exp\left(-ic
\left(\ad\left(\Yok\right)\vert_{\mathfrak p}\right)\right)\right]\\
\Tr^{E}\left[\rho^{E}\left(k^{-1}\right)\exp\left(-i\rho^{E}\left(\Yok\right)\right)\right]
\frac{d\Yok}{\left(2\pi t\right)^{q/2}}.
\end{multline}
Equivalently, for any $t>0$, 
\begin{multline}\label{eq:conk4x2}
    \Tr^{\left[\gamma\right]}\left[D^{X}\exp\left(-sD^{X,2}/2\right)\right]*\frac{1}{\sqrt{s}
}\left(t\right)
=\sqrt{2\pi}\exp\left(\frac{t}{48}\Tr^{\mathfrak k}\left[C^{\mathfrak 
k, \mathfrak k}\right]+\frac{t}{2}C^{\mathfrak k,E}\right)\\
\exp\left(\frac{t}{2}\Delta^{\mathfrak z\left(\gamma\right)}\right)
\Bigg(J_{\gamma}\left(\Yok\right)\Tr^{S^{\mathfrak p}}
\left[\frac{c\left(a\right)}{\left\vert  a\right\vert}\Ad\left(k^{-1}\right)
\exp\left(-ic
\left(\ad\left(\Yok\right)\vert_{\mathfrak p}\right)\right)\right]\\
\Tr^{E}\left[\rho^{E}\left(k^{-1}\right)\exp\left(-i\rho^{E}\left(\Yok\right)\right)\right]\delta_{Y_{0}^{\mathfrak p}=a}
 \Biggr) \left(0\right).
\end{multline}
\end{thm}
\begin{proof}
Recall that if $e\in \mathfrak p$, the action of 
$\widehat{c}\left(e\right)$ on $S^{\mathfrak p}$ is given by 
$\widehat{c}\left(e\right)=ic\left(e\right)$. When $t=1$, by equation 
(\ref{eq:key1}) in Theorem \ref{Tkeyres}, we get 
(\ref{eq:conk4x1}). For an arbitrary $t>0$,  
(\ref{eq:conk4x1}) is just the same equation with $t=1$ when  
 $B$ is replaced by $B/t$. By (\ref{eq:conk4x1}), we 
get (\ref{eq:conk4x2}).
\end{proof}

From now on, we assume that the representation $\rho^{E}:K\to 
\mathrm{U}\left(E\right)$ is irreducible. In particular $C^{\mathfrak 
k,E}$ is  scalar. Let 
\index{T@$T$}%
$T$ be a maximal torus in $K$, and let 
\index{t@$\mathfrak t$}%
$\mathfrak t$ be its Lie 
algebra. Let $R_{+}$ be an associated positive root system, and let $\rho$ be the 
half-sum of the positive roots. Let $\lambda\in \mathfrak t^{*}$ be the 
nonnegative weight that defines the irreducible representation 
$\rho^{E}$.

By \cite[eq. (7.2.15) and Proposition 7.5.2]{Bismut08b}, we get
\begin{equation}\label{eq:conk4y1}
\frac{1}{48}\Tr^{\mathfrak k}\left[C^{\mathfrak k, \mathfrak k}\right]+\frac{1}{2}C^{\mathfrak k,E}=-2\pi^{2}
\left\vert  \rho+\lambda\right\vert^{2}.
\end{equation}

\begin{thm}\label{Tanew}
For any $t>0$, the following identity holds:
\begin{multline}\label{eq:conk4x3}
 \Tr^{\left[\gamma\right]}\left[D^{X}
 \exp\left(-tD^{X,2}/2\right)\right]
=\frac{\sqrt{2}}{\sqrt{\pi}}\frac{d}{dt}\Biggl[\exp\left(
-2\pi^{2}s\left\vert  \rho+\lambda\right\vert^{2}\right)\\
\exp\left(\frac{s}{2}\Delta^{\mathfrak z\left(\gamma\right)}\right)
\Bigg(J_{\gamma}\left(\Yok\right)\Tr^{S^{\mathfrak p}}
\left[\frac{c\left(a\right)}
{\left\vert  
a\right\vert}\Ad\left(k^{-1}\right)\exp\left(-ic
\left(\ad\left(\Yok\right)\vert_{\mathfrak p}\right)\right)\right]\\
\Tr^{E}\left[\rho^{E}\left(k^{-1}\right)
\exp\left(-i\rho^{E}\left(\Yok\right)\right)\right]\delta_{Y_{0}^{\mathfrak p}=a}
\Biggr)\left(0\right)
*\frac{1}{\sqrt{s}}\Biggr]\left(t\right).
\end{multline}
\end{thm}
\begin{proof}
By (\ref{eq:conk2}) and (\ref{eq:conk4x2}), we get 
(\ref{eq:conk4x3}). 
\end{proof}
\subsection{The vanishing of the orbital integrals}%
\label{subsec:van}
Since $\mathfrak p$ is odd dimensional and since $\Ad\left(k\right)$ preserves the orientation of $\mathfrak p$, 
 $\dim \mathfrak p\left(k\right)$ is odd. Since $a\in 
\mathfrak p\left(k\right)$, then $\dim 
\mathfrak p\left(k\right)\ge 1$.

Now we make the same discussion as in \cite[section 7.9]{Bismut08b}.
Set
\index{b@$\mathfrak b$}%
\begin{equation}\label{eq:ptit1}
\mathfrak b=\left\{e\in \mathfrak p, \left[e,\mathfrak 
t\right]=0\right\}.
\end{equation}
Since $\mathfrak p$ is odd dimensional, $\mathfrak b$ is also odd 
dimensional.

 Put
 \index{h@$\mathfrak h$}%
\begin{equation}\label{eq:ptit2}
\mathfrak h=\mathfrak b \oplus \mathfrak t.
\end{equation}
By \cite[p. 129]{Knapp86}, $\mathfrak h$ is a 
Cartan subalgebra of 
$\mathfrak g$. Also $\dim \mathfrak t$ is the
complex rank of $K$, 
and $\dim \mathfrak h$ is the complex rank of $G$.

 By (\ref{eq:grzu0}), we get
\begin{equation}\label{eq:grzu-1}
K\left(\gamma\right) \subset Z\left(a\right)\cap  Z\left(k\right).
\end{equation}

Recall that
\index{K0g@$K^{0}\left(\gamma\right)$}%
$K^{0}\left(\gamma\right) \subset K\left(\gamma\right)$ is the connected component
of the identity.
Let 
\index{Tg@$T\left(\gamma\right)$}%
$T\left(\gamma\right) \subset K^{0}\left(\gamma\right)$ be a 
 maximal torus in $K^{0}\left(\gamma\right)$, and let 
 \index{tg@$\mathfrak t\left(\gamma\right)$}%
 $\mathfrak t \left(\gamma\right)\subset \mathfrak k\left(\gamma\right)$ be its Lie 
algebra. Since $T\left(\gamma\right)$ is 
commutative, and  since by (\ref{eq:grzu-1}),  $k$ commutes with 
$T\left(\gamma\right)$, we may and we will assume that 
$T\left(\gamma\right) \subset T, k\in T$. In particular $\mathfrak 
t\left(\gamma\right) \subset \mathfrak t$.

Set
\index{bg@$\mathfrak b\left(\gamma\right)$}%
\begin{equation}\label{eq:shak12x1}
\mathfrak b\left(\gamma\right)=\left\{e\in \mathfrak p\left(k\right), \left[e, \mathfrak 
t\left(\gamma\right)\right]=0\right\}.
\end{equation}
Since $\mathfrak p\left(k\right)$ is odd dimensional, $\mathfrak 
b\left(\gamma\right)$ is also odd dimensional,
so that
\begin{equation}\label{eq:shak12bx2}
\dim \mathfrak b\left(\gamma\right)\ge 1.
\end{equation}
Also $a\in \mathfrak b\left(\gamma\right)$.

Moreover, since $\mathfrak t\left(\gamma\right) \subset \mathfrak t$, 
and $k\in T$,  
\begin{equation}\label{eq:shak12bx3}
\mathfrak 
b\subset \mathfrak b\left(\gamma\right).
\end{equation}
Here, we will recover a first result of Moscovici-Stanton 
\cite{MoscoviciStanton89}.
\begin{thm}\label{Tvan}
If $\dim \mathfrak b\left(\gamma\right)\ge 3$,    for any $t>0$, 
\begin{equation}\label{eq:conk5}
\Tr^{\left[\gamma\right]}\left[D^{X}\exp\left(-t
D^{X,2}/2\right)\right]=0.
\end{equation} 
In particular, this is the case if $\dim \mathfrak b\ge 3$.
\end{thm}
\begin{proof}
    By (\ref{eq:conk4x3}), to establish (\ref{eq:conk5}), we only 
    need to show that if $\Yok\in \mathfrak k\left(\gamma\right)$, 
    \begin{equation}\label{eq:gong2}
\Tr^{S^{\mathfrak p}}\left[\frac{c\left(a\right)}{\left\vert  
a\right\vert}\Ad\left(k^{-1}\right)\exp\left(-ic\left(\ad\left(\Yok\right)\vert_{\mathfrak p}\right)\right)\right]=0.
\end{equation}
Using the adjoint action of  $K^{0}\left(\gamma\right)$ on $\mathfrak 
k\left(\gamma\right)$, we may as well assume 
    that $\Yok\in \mathfrak t\left(\gamma\right)$.
    Then $\Ad\left(k^{-1}\right)$ acts like the 
    identity on $\mathfrak b\left(\gamma\right)$, and 
    $\ad\left(\Yok\right)$ vanishes on $\mathfrak 
    b\left(\gamma\right)$. Therefore $\mathfrak b\left(\gamma\right)$ 
    lies in the eigenspace for the action on $\mathfrak p$ of 
    $\Ad\left(k^{-1}\right)\exp\left(-\ad\left(\Yok\right)\right)$ 
    associated with
 the eigenvalue $1$. If $\dim \mathfrak b\left(\gamma\right)\ge 
    3$, this eigenspace   is of dimension $\ge 3$.

Let $\left\{a\right\}^{\perp}$ be the orthogonal space to $a$ in $\mathfrak 
p$. By the above,  the eigenspace for 
the action of 
$\Ad\left(k^{-1}\right)\exp\left(-\ad\left(\Yok\right)\right)$ on 
$\left\{a\right\}^{\perp}$  for 
the eigenvalue $1$
is of dimension $\ge 2$. Let $e$ be a unit vector in this eigenspace. We 
have the identity
\begin{equation}\label{eq:laus7a1}
\left[c\left(e\right),\frac{c\left(a\right)}{\left\vert  
a\right\vert}\Ad\left(k^{-1}\right)\exp\left(-c\left( 
\ad\left(\Yok\right)\vert_{\mathfrak p}\right) \right)\right]=0.
\end{equation}
By (\ref{eq:laus7a1}), we get 
\begin{multline}\label{eq:laus7a1b}
\frac{c\left(a\right)}{\left\vert  
a\right\vert}\Ad\left(k^{-1}\right)\exp\left(-c\left( 
\ad\left(\Yok\right)\vert_{\mathfrak p}\right) \right)\\
=\frac{1}{2}\left[c\left(e\right),\frac{c\left(a\right)}{\left\vert  
a\right\vert}\Ad\left(k^{-1}\right)\exp\left(-c\left( 
\ad\left(\Yok\right)\vert_{\mathfrak p}\right) \right)c\left(e\right)\right].
\end{multline}
Since $\Ad\left(k^{-1}\right)\exp\left(-c\left( 
\ad\left(\Yok\right)\vert_{\mathfrak p}\right) \right)\in c^{\even}\left(\mathfrak 
p\right)$, the right-hand side of (\ref{eq:laus7a1b}) is a 
commutator.
Since $\Tr^{S^{\mathfrak p}}$ 
vanishes on commutators, by (\ref{eq:laus7a1b}), we get
\begin{equation}\label{eq:gong1}
\Tr^{S^{\mathfrak p}}\left[\frac{c\left(a\right)}{\left\vert  
a\right\vert}\Ad\left(k^{-1}\right)\exp\left(-c\left( 
\ad\left(\Yok\right)\vert_{\mathfrak p}\right) \right) \right]=0.
\end{equation}
By analyticity, from (\ref{eq:gong1}), we get (\ref{eq:gong2}).
The proof of our theorem is completed. 
\end{proof}
\begin{remk}\label{Rgrox}
Now we reproduce the content of \cite[Remark 7.9.2]{Bismut08b}. For $p,q\in \N$, let $\mathrm{SO}^{0}\left(p,q\right)$ be the 
connected component of the identity in the real group 
$\mathrm{SO}\left(p,q\right)$. By  \cite[Table V p. 518]{Helgason78} 
and \cite[Table C1 p. 713, and Table C2  p. 
714]{Knapp86}, among the noncompact simple connected complex groups such that
$m$ is odd and $\dim \mathfrak b=1$, there is only 
$\mathrm{SL}_{2}\left(\C\right)$, and among the noncompact simple real 
connected groups with the same property,  there are only $\mathrm{SL}_{3}\left(\R\right)$,
$\mathrm{SL}_{4} \left(\R\right)$, $\mathrm{SL}_{2}\left(\mathbb 
H\right)$, and 
$\mathrm{SO}^{0}\left(p,q\right)$ with $pq$ odd $>1$. Also by 
\cite[pp. 519, 520]{Helgason78}, 
$\mathrm{sl}_{2}\left(\C\right)=\mathrm{so}\left(3,1\right)$, 
$\mathrm{sl}_{4}\left(\R\right)=\mathrm{so}\left(3,3\right)$, and 
$\mathrm{sl}_{2}\left(\mathbb H\right)=\mathrm{so}\left(5,1\right)$. 
Therefore the above list can be reduced to 
$\mathrm{SL}_{3}\left(\R\right)$ and $\mathrm{SO}^{0}\left(p,q\right)$ 
with $pq$ odd $>1$.\footnote{I am  indebted to Yves Benoist 
for providing the above information.} This is exactly the list given 
by \cite{MoscoviciStanton89} that implies the vanishing of the odd 
traces in (\ref{eq:conk5}).
\end{remk}
\subsection{A convolution identity}%
\label{subsec:siide}
For $x\ge 0$, set
\index{fx@$\phi\left(x\right)$}%
\begin{equation}\label{eq:nauf16x1}
\phi\left(x\right)=\int_{x}^{+ \infty 
}\exp\left(-\lambda\right)\lambda^{-1/2}d\lambda=
2\int_{\sqrt{x}}^{+ \infty }\exp\left(-\lambda^{2}\right)d\lambda.
\end{equation}
\begin{prop}\label{Poff}
For $x>0,t>0$, 
\begin{equation}\label{eq:nauf18}
\frac{1}{\sqrt{s}}\exp\left(-x/s\right)*\frac{1}{\sqrt{s}}\left(t\right)=
\sqrt{\pi}
\phi\left(x/t\right).
\end{equation}
\end{prop}
\begin{proof}
Clearly,
\begin{multline}\label{eq:nauf15a}
\frac{1}{\sqrt{s}}\exp\left(-x/s\right)*\frac{1}{\sqrt{s}}\left(t\right)=
\int_{1}^{+ \infty 
}\exp\left(-\frac{xs}{t}\right)\frac{1}{s\sqrt{s-1}}ds\\
=2\exp\left(-x/t\right)\int_{0}^{+ \infty }\exp\left(-xs^{2}/t\right)
\left(1+s^{2}\right)^{-1}ds.
\end{multline}
Moreover,
\begin{multline}\label{eq:nauf16a}
2\int_{0}^{+ \infty }\exp\left(-xs^{2}/t\right)
\left(1+s^{2}\right)^{-1}ds\\
=2\int_{\R^{2}_{+}}^{}
\exp\left(-xs^{2}/t-\left(1+s^{2}\right)\lambda\right)dsd\lambda\\
=\sqrt{\pi}\int_{0}^{+ \infty }\exp\left(-\lambda\right)
\left(\lambda+x/t\right)^{-1/2}d\lambda
=\sqrt{\pi}\exp\left(x/t\right)\int_{x/t}^{+ \infty 
}\exp\left(-\lambda\right)\lambda^{-1/2}d\lambda.
\end{multline}
By (\ref{eq:nauf15a}), (\ref{eq:nauf16a}), we get
(\ref{eq:nauf18}).
\end{proof}
\subsection{The case where $\dim \mathfrak b\left(\gamma\right)=1$}%
\label{subsec:gen}
In the sequel, we assume that $\dim \mathfrak 
b\left(\gamma\right)=1$. By (\ref{eq:shak12bx3}),
$\mathfrak b=\mathfrak b\left(\gamma\right)$ is $1$-dimensional and 
generated by $a$. Since $\ad\left(\mathfrak t\right)$ vanishes on 
$a$, and $k\in T$, we get
\begin{align}\label{eq:cung1}
&\mathfrak t\left(\gamma\right)=\mathfrak t,&T\left(\gamma\right)=T.
\end{align}

Note that
$\ad\left(a\right)$ is an invertible endomorphism of $\mathfrak 
z_{0}^{\perp}$ that exchanges $\mathfrak p_{0}^{\perp}$ and 
$\mathfrak k_{0}^{\perp}$ and commutes with $\Ad\left(k^{-1}\right)$.

Let 
\index{ap@$\left\{a\right\}^{\perp} $}%
$\left\{a\right\}^{\perp} \subset \mathfrak p$ be the orthogonal 
vector space 
to $a$ in $\mathfrak p$. We have the orthogonal splitting
\begin{equation}\label{eq:freeb1}
\left\{a\right\}^{\perp}=\left\{a\right\}^{\perp}\cap \mathfrak p_{0} 
\oplus \mathfrak p_{0}^{\perp}.
\end{equation}
Then $\mathfrak t$ preserves $\left\{a\right\}^{\perp}$. Since 
$\mathfrak t \subset \mathfrak k_{0}$, $\mathfrak t$
   preserves the splitting (\ref{eq:freeb1}). Since $\mathfrak b$ is reduced to $\left\{a\right\}$, $\left\{a\right\}^{\perp}\cap
\mathfrak p_{0}$ and $\mathfrak 
p_{0}^{\perp}$ are even dimensional, and preserved by $T$. Since 
$k\in T$, the action of $\Ad\left(k\right)$ on these two vector spaces preserves 
their orientation. In particular, the eigenspaces of $k$ that are 
associated with the eigenvalue $1$ are even dimensional. 

As we saw in subsection \ref{subsec:minge},  $\mathfrak p_{0}$ and $\mathfrak p\left(k\right)$ 
intersect orthogonally along $\mathfrak p\left(\gamma\right)$.

Since $\left\{a\right\}^{\perp}\cap \mathfrak p_{0} $  is even dimensional,  $\mathfrak p_{0}$ is odd 
dimensional. Since $T=T\left(\gamma\right)$, then $T$ preserves 
$\mathfrak p_{0}$. In particular $k\in T$ acts like an oriented 
isomorphism of $\mathfrak p_{0}$. Since $\mathfrak 
p\left(\gamma\right)$ is the part of $\mathfrak p_{0}$ that is fixed 
by $\Ad\left(k\right)$, $\mathfrak p\left(\gamma\right)$ is also odd 
dimensional, and $\mathfrak p_{0}^{\perp}\left(\gamma\right)$ is even 
dimensional. In the same way, $\Ad\left(k\right)$ acts like an 
oriented isomorphism of $\mathfrak p_{0}^{\perp}$. Since $\mathfrak 
p_{0}^{\perp}$ is even dimensional,   $
\mathfrak p_{0}^{\perp}\cap \mathfrak p\left(k\right)$,  which is the vector subspace of 
$\mathfrak p_{0}^{\perp}$ fixed by $\Ad\left(k\right)$,  is even 
dimensional. Since $\mathfrak p_{0}$ and $\mathfrak p\left(k\right)$ 
intersect orthogonally along $\mathfrak p\left(\gamma\right)$, $\mathfrak p_{0}^{\perp}\cap \mathfrak 
p\left(k\right)$ is just the orthogonal space to $\mathfrak 
p\left(\gamma\right)$ in $\mathfrak p\left(k\right)$.

We orient $\left\{a\right\}^{\perp}$ so that 
when completing an oriented basis of $\left\{a\right\}^{\perp}$ by $a$, we obtain an oriented basis 
of $\mathfrak p$. 
We orient the vector spaces in the right-hand side of 
(\ref{eq:freeb1}), so that (\ref{eq:freeb1}) is an identity of 
oriented vector spaces.  The orientation of these two vector spaces 
is noncanonical. Since $\ad\left(a\right)$ induces an 
isomorphism from  
$\mathfrak p_{0}^{\perp}$ into $\mathfrak k_{0}^{\perp}$,
we equip $\mathfrak k_{0}^{\perp}$ with the corresponding 
orientation. Similarly, since $\left\{a\right\}^{\perp}\cap \mathfrak 
p_{0}$ is oriented, we orient $\mathfrak p_{0}$ by the procedure that 
was outlined before.

Since $k$ acts as an oriented isometry of $\mathfrak p$, $\mathfrak 
p\left(k\right)$ is odd dimensional, and $\mathfrak p^{\perp}\left(k\right)$
is even dimensional. Since $k\in T$, there is $t_{0}\in \mathfrak t$ such that 
\begin{equation}\label{eq:laus8}
    k=e^{t_{0}}.
\end{equation}
 In particular 
$\ad\left(t_{0}\right)$ acts as an invertible endomorphism of 
$\mathfrak p^{\perp}\left(k\right)$, so that $\mathfrak 
p^{\perp}\left(k\right)$ is canonically oriented (the orientation 
depending on the choice of $t_{0}$). It follows that the 
orientation line $\mathrm{o}\left(\mathfrak p\right)$ of $\mathfrak p$ is 
just the orientation line $\mathrm{o}\left(\mathfrak p\left(k\right)\right)$
 of 
$\mathfrak p\left(k\right)$, i.e.,
\begin{equation}\label{eq:laus9}
\mathrm{o}\left(\mathfrak p\right)=\mathrm{o}\left(\mathfrak p\left(k\right)\right).
\end{equation}
Other orientations lines will be denoted in the same way. By 
(\ref{eq:laus9}), we get
\begin{equation}\label{eq:laus9x1}
\mathrm{o}\left(\mathfrak p\right)=\mathrm{o}\left(\mathfrak 
p\left(\gamma\right)\right) \otimes \mathrm{o}\left(\mathfrak p_{0}^{\perp}\cap 
\mathfrak p\left(k\right)\right).
\end{equation}

Since the simply connected group $K$ acts on 
$S^{\mathfrak p}$, the action of $k$ on $S^{\mathfrak p}$ is 
unambiguously determined by (\ref{eq:rus1a}), (\ref{eq:laus8}).

 Let $S^{\left\{a\right\}^{\perp}}$ be the spinors associated with 
 the oriented Euclidean vector space $\left\{a\right\}^{\perp}$. 
 Then $S^{\left\{a\right\}^{\perp}}$ is a $\Z_{2}$-graded vector 
space.  Let $\tau=\pm 1$ be the endomorphism defining the
$\Z_{2}$-grading. Then
\begin{equation}\label{eq:free2}
S^{\mathfrak p}=S^{\left\{a\right\}^{\perp}}.
\end{equation}
Moreover,  $\frac{c\left(a\right)}{\left\vert  
a\right\vert}$ acts on  $S^{\left\{a\right\}^{\perp}}$ like $-i\tau$. 
Let $S^{\left\{a\right\}^{\perp}\cap 
\mathfrak p_{0}}, S^{\mathfrak p_{0}^{\perp}}$ be the spinors 
associated with the oriented Euclidean vector spaces $\left\{a\right\}^{\perp}\cap 
\mathfrak p_{0},\mathfrak p_{0}^{\perp}$. Again, these vector spaces 
are $\Z_{2}$-graded.
Moreover, because of the splitting (\ref{eq:freeb1}), we get
\begin{equation}\label{eq:free2x1}
S^{\left\{a\right\}^{\perp}}=S^{\left\{a\right\}^{\perp}\cap 
\mathfrak p_{0}}\ho S^{\mathfrak p_{0}^{\perp}}.
\end{equation}

The simply connected group $K$ acts on $S^{\mathfrak 
p}=S^{\left\{a\right\}^{\perp}}$. Recall that $K_{0}^{0}$ is the 
connected component of the identity in $K_{0}=Z\left(a\right)\cap K$. Since $K^{0}_{0} \subset 
K$, $K^{0}_{0}$ also acts on 
$S^{\left\{a\right\}^{\perp}}$. However, $K^{0}_{0}$ is not 
necessarily simply connected. While $K^{0}_{0}$ preserves 
the splitting (\ref{eq:freeb1}), the action of $K^{0}_{0}$ 
does not necessarily lift to an action of $K^{0}_{0}$ on $S^{\left\{a\right\}^{\perp}\cap 
\mathfrak p_{0}}, S^{\mathfrak p_{0}^{\perp}}$, the possible lift 
having a $\pm 1$ ambiguity. However, because of (\ref{eq:free2x1}), 
the ambiguity is the same when acting on both vector spaces.

Recall that (\ref{eq:supr1}) holds.
By the results of \cite{Bismut08b} that were explained in subsection 
\ref{subsec:minge}, $X\left(e^{a}\right)$ can be identified with the 
symmetric space associated with $Z^{0}_{0}$, so that 
$X\left(e^{a}\right)=Z^{0}_{0}/K^{0}_{0}$. Then $TX\left(e^{a}\right)$ is the 
vector bundle associated with the action of $K^{0}_{0}$ on 
$\mathfrak p_{0}$, and the normal bundle $N_{X\left(e^{a}\right)/X}$ is 
associated with the action of $K^{0}_{0}$ on $\mathfrak 
p_{0}^{\perp}$. In particular $X\left(e^{a}\right)$ is odd dimensional. Moreover, $k^{-1}$ acts on 
$X\left(e^{a}\right)$, and its fixed point set is given by 
$X\left(\gamma\right)$. Also 
$N_{X\left(\gamma\right)/X\left(e^{a}\right)}$ is just the vector bundle 
on $X\left(\gamma\right)$ associated with the action of 
$K^{0}\left(\gamma\right)$ on $\mathfrak 
p_{0}^{\perp}\left(\gamma\right)$.

Then $k^{-1}$ acts naturally on 
$TX\left(e^{a}\right)\vert_{X\left(\gamma\right)}$. Also 
$TX\left(\gamma\right)$ is the eigenbundle of this action associated 
with the eigenvalue $1$. The distinct angles  
$\pm\theta_{1},\ldots,\pm \theta_{s},0<\theta_{i}\le \pi$  of the 
action of $k^{-1}$ on $N_{X\left(\gamma\right)/X\left(e^{a}\right)}$ are exactly 
the nonzero
angles of the action of $\Ad\left(k^{-1}\right)$ on $\mathfrak 
p^{\perp}_{0}\left(\gamma\right)$. Let 
$N_{X\left(\gamma\right)/X\left(e^{a}\right),\theta_{i}},1\le i\le s$ be the part of 
$N_{X\left(\gamma\right)/X\left(e^{a}\right)}$ on which $\Ad\left(k^{-1}\right)$ acts 
by a rotation of angle 
$\theta_{i}$. 

We will consider characteristic 
forms of homogeneous vector bundles on 
$X\left(\gamma\right)$. Since these vector bundles are 
equipped with canonical connections, when noting their 
corresponding characteristic forms, we will not note the connection 
forms explicitly.

If $\theta\in \R\setminus 2\pi\Z$, set
\index{Atx@$ \widehat{A}^{\theta}\left(x\right)$}%
\begin{equation}
    \widehat{A}^{\theta}\left(x\right)=\frac{1}{2\sinh\left(\frac{x+i\theta}{2}\right)}.
    \label{eq:tomsk1}
\end{equation}
Given $\theta$, we identify $\widehat{A}^{\theta}\left(x\right)$ with 
the corresponding multiplicative genus.
We define the following closed form on $X\left(\gamma\right)$,
\index{AkTX@$\widehat{A}^{k^{-1}}\left(TX\left(e^{a}\right)\vert_{X\left(\gamma\right)}\right)$}%
\begin{equation}
    \widehat{A}^{k^{-1}}\left(TX\left(e^{a}\right)\vert_{X\left(\gamma\right)}
    \right)=
    \widehat{A}\left(TX\left(\gamma\right)\right)
    \prod_{i=1}^{s}
    \widehat{A}^{\theta_{i}}\left(N_{X\left(\gamma\right)/X\left(e^{a}\right),\theta_{i}}\right).
    \label{eq:tomsk2}
\end{equation}
As usual in such formulas, there is a $\pm 1$ sign ambiguity in the 
right-hand side of (\ref{eq:tomsk2}). However, 
because the action of $k\in T$ lifts to $S^{\mathfrak p}$, the 
ambiguity disappears when considering instead $\left(\widehat{A}^{k^{-1}\vert_{\mathfrak 
p_{0}^{\perp}}}
\left(0\right)\right)^{-1}\widehat{A}^{k^{-1}}\left(TX\left(e^{a}\right)\right)$.

Similarly, the group $Z\left(a\right)$ acts on $E$ via the 
representation $\rho^{E}$. The corresponding vector bundle on 
$X\left(e^{a}\right)$ is just the restriction of $F$ to 
$X\left(e^{a}\right)$. Also $k^{-1}$ acts on the left on this vector 
bundle. Recall that 
\index{RF@$R^{F}$}%
$R^{F}$ is the curvature of $\n^{F}$.
Let 
\index{chk@$\ch^{k^{-1}}\left(F\right)$}%
$\ch^{k^{-1}}\left(F\right)$ denote the Chern 
character form on 
$X\left(\gamma\right)$ that is given by
\index{chkF@$\ch^{k^{-1}}\left(F\vert_{X\left(\gamma\right)}\right)$}%
\begin{equation}
    \ch^{k^{-1}}\left(F\vert_{X\left(\gamma\right)}\right)
    =\Tr\left[\rho^{F}\left(k^{-1}\right)\exp\left(-\frac{R^{F}\vert_{X\left(\gamma\right)}}{2i\pi}\right)\right].
    \label{eq:tomsk3}
\end{equation}
The closed forms in (\ref{eq:tomsk2}), (\ref{eq:tomsk3}) on $X\left(\gamma\right)$  are exactly the ones that appear 
in the Lefschetz fixed point formula of 
Atiyah-Bott \cite{AtiyahBott67,AtiyahBott68} when considering  the 
action of $k^{-1}$ on $X\left(e^{a}\right)$. Note that there are 
questions of signs  to be taken care of, because of the need 
to distinguish between $\theta_{i}$ and $-\theta_{i}$. We refer to 
the above references for more detail.

Let 
$N_{X\left(e^{a}\right)/X}\left(k\right)$ be the  subvector bundle of  
$N_{X\left(e^{a}\right)/X}\vert_{X\left(\gamma\right)}$ that 
is fixed by $k$.  This vector bundle is associated with the action of 
$K^{0}\left(\gamma\right)$ on $\mathfrak 
p^{\perp}_{0}\cap \mathfrak p\left(k\right)$. Since
$\mathfrak 
p^{\perp}_{0}\cap \mathfrak p\left(k\right)$ is even 
dimensional, $N_{X\left(e^{a}\right)/X}\left(k\right)$ is an even 
dimensional  vector bundle. Since $X\left(\gamma\right)$ and 
$X\left(e^{a}\right)$ intersect orthogonally, 
$N_{X\left(e^{a}\right)/X}\left(k\right)$ is also the normal bundle 
$N_{X\left(\gamma\right)/X\left(k\right)}$.

 Let 
 \index{eNXg@$e\left(N_{X\left(\gamma\right)/X\left(k\right)}\right)$}%
 $e\left(N_{X\left(\gamma\right)/X\left(k\right)}\right)$ denote 
 the  Euler form of $N_{X\left(\gamma\right)/X\left(k\right)}$ on $X\left(\gamma\right)$. The form 
 $e\left(N_{X\left(\gamma\right)/X\left(k\right)}\right)$ is a 
 section of 
$\Lambda^{\dim \mathfrak p_{0}^{\perp}\cap \mathfrak 
p\left(k\right)}\left(T^{*}X\left(\gamma\right)\right) 
\otimes \mathrm{o}\left(N_{X\left(\gamma\right)/X\left(k\right)}\right)$.
   By (\ref{eq:laus9x1}), the form 
 $e\left(N_{X\left(\gamma\right)/X\left(k\right)}\right)$ can be 
 considered as a section of $\Lambda^{\dim \mathfrak p_{0}^{\perp}\cap \mathfrak 
p\left(k\right)}\left(T^{*}X\left(\gamma\right)\right)  \otimes 
\mathrm{o}\left(TX\left(\gamma\right)\right) \otimes 
\mathrm{o}\left(TX\right)$.  Also 
$\mathrm{o}\left(TX\left(\gamma\right)\right) \otimes 
\mathrm{o}\left(TX\right)=
 \mathrm{o}\left(\mathfrak p\left(\gamma\right)\right) \otimes 
 \mathrm{o}\left(\mathfrak p\right)$.

Also $K_{0}^{0}$ 
preserves the splitting $\mathfrak k=\mathfrak k_{0} \oplus \mathfrak 
k_{0}^{\perp}$. Let $N_{0},  N_{0}^{\perp}$ denote 
the corresponding vector bundles on $X\left(e^{a}\right)$.

The Lie group $Z_{0}$ acts  on $\mathfrak 
z\left(a\right)=\mathfrak z_{0}$, and so it acts on $\mathfrak 
z_{0}^{\perp}= \mathfrak p_{0}^{\perp} \oplus \mathfrak 
k_{0}^{\perp}$. By proceeding as in \cite[eq. 
(7.7.5)]{Bismut08b}, we find that on $X\left(e^{a}\right)$, 
$N_{X\left(e^{a}\right)/X}\oplus N_{0}^{\perp}$ is equipped 
with a Euclidean connection preserving the splitting, and also with a flat connection. Also 
$k^{-1}$ acts on the restriction of these vector bundles to 
$X\left(\gamma\right)$. Our 
characteristic forms will  be computed using  the relevant Euclidean 
connections. 
By proceeding as in \cite[eq. (7.7.5)]{Bismut08b}, we get the 
identity of forms on $X\left(\gamma\right)$
\begin{equation}\label{eq:nauf11}
\widehat{A}^{k^{-1}}\left(N_{X\left(e^{a}\right)/X}\right)
\widehat{A}^{k^{-1}}\left(N_{0}^{\perp}\right)=
\widehat{A}^{k^{-1}
\vert _{\mathfrak z_{0}^{\perp}}}\left(0\right).
\end{equation}

Also $\ad\left(a\right)$ is a parallel isomorphism from 
$N_{X\left(e^{a}\right)/X}$ into $N_{0}^{\perp}$ with respect to their 
canonical Euclidean connections.  It follows that on 
$X\left(\gamma\right)$, we have the identity of differential forms
\begin{equation}\label{eq:nauf12}
\widehat{A}^{k^{-1}}\left(N_{X\left(e^{a}\right)/X}\right)=
\widehat{A}^{k^{-1}}\left(N_{0}^{\perp}\right).
\end{equation}
By (\ref{eq:nauf11}), (\ref{eq:nauf12}), we get
\begin{equation}\label{eq:nauf13}
\widehat{A}^{k^{-1}}\left(N_{X\left(e^{a}\right)/X}\right)=\widehat{A}^{k^{-1}\vert
_{\mathfrak p_{0}^{\perp}}
}\left(0\right).
\end{equation}

 Let 
$\eta\vert_{X\left(\gamma\right)}$ be 
the canonical section of norm $1$ in 
$\Lambda^{p}\left(T^{*}X\left(\gamma\right)\right) \otimes 
\mathrm{o}\left(TX\left(\gamma\right)\right)$. Equival{ently, 
$\eta\vert_{X\left(\gamma\right)}$ is the volume form on 
$X\left(\gamma\right)$. If $\alpha\in 
\Lambda\ac\left(T^{*}X\left(\gamma\right)\right) \otimes 
\mathrm{o}\left(TX\left(\gamma\right)\right) \otimes 
\mathrm{o}\left(TX\right)$, and if 
$\alpha^{(p)}$ denotes its component of top degree $p$, let 
$\alpha^{\max}$ be the section of $ \mathrm{o}\left(TX\right)$   given by
\begin{equation}\label{eq:laus10}
\alpha^{\max}=\frac{\alpha^{(p)}}{\eta\vert_{X\left(\gamma\right)}}.
\end{equation}
We will also view $	\alpha^{\max}$ as a section of 
$\mathrm{o}\left(\mathfrak p\right)$.

Let 
\index{a@$a^{*}$}%
$a^{*}$  be the $1$-form on $X\left(\gamma\right) $  which is dual 
to $a^{TX}$. Put
\index{a@$\mathbf{a}$}%
\begin{equation}\label{eq:bres5}
\mathbf{a}=\frac{a^{*}}{\left\vert  a\right\vert}.
\end{equation}

As we saw in subsection \ref{subsec:minge},  $Z^{0}_{0}$ is a 
reductive group. By the above, we know that 
$k\in Z^{0}_{0}$. To the couple 
$\left(Z^{0}\left(a\right),k^{-1}\right)$, we will apply the 
constructions we made before for $\left(G,\gamma\right)$. The analogue of $\mathfrak 
k\left(\gamma\right)$  is still  
equal to $\mathfrak k\left(\gamma\right)$.
Let 
\index{JkY@$J'_{k^{-1}}\left(Y_{0}^{\mathfrak k_{0}}\right)$}%
$J'_{k^{-1}}\left(\Yok\right),\Yok\in \mathfrak k\left(\gamma\right)$ be the function defined in 
\cite[Theorem 5.5.1]{Bismut08b} and in Definition \ref{DJg} which is 
associated with the group $Z^{0}_{0}$ and with $k^{-1}\in 
Z^{0}_{0}$. 
 By \cite[Theorem 5.5.1]{Bismut08b} or by (\ref{eq:crub3}), for $\Yok\in 
\mathfrak k\left(\gamma\right)$,  we get
\begin{equation}\label{eq:free1}
J_{\gamma}\left(\Yok\right)=\frac{J'_{k^{-1}}\left(\Yok\right)}
{\left\vert  
\det\left(1-\Ad\left(\gamma\right)\right)\vert_{ \mathfrak 
z_{0}^{\perp}}\right\vert^{1/2}}.
\end{equation}

We will now recover the explicit formula by Moscovici-Stanton 
\cite[Theorem 5.10]{MoscoviciStanton89} for the orbital integrals 
$\Tr^{\left[\gamma\right]}\left[D^{X}\exp\left(-tD^{X,2}/2\right)\right]$. Since they depend on the choice of an orientation of $\mathfrak p$, they are just sections of $\mathrm{o}\left(\mathfrak p\right)$.
\begin{thm}\label{Tgen}
When $\dim \mathfrak b\left(\gamma\right)=1$, for $t>0$, the 
following identity of sections of $\mathrm{o}\left(\mathfrak p\right)$ holds:
\begin{multline}\label{eq:free9}
\Tr^{\left[\gamma\right]}\left[D^{X}
\exp\left(-tD^{X,2}/2\right)\right]
=-i\frac{\left(-1\right)^{\dim \mathfrak p_{0}^{\perp}/2}\left(\widehat{A}^{k^{-1}\vert_{\mathfrak 
p_{0}^{\perp}}}
\left(0\right)\right)^{-1}}{\left\vert  
\det\left(1-\Ad\left(\gamma\right)\right)\vert_{ \mathfrak 
z_{0}^{\perp}}\right\vert^{1/2}}\\
\left[\widehat{A}^{k^{-1}}\left(TX\left(e^{a}\right)\right)
e\left(N_{X\left(\gamma\right)/X\left(k\right)}\right)\ch
^{k^{-1}}\left(F\right)\frac{\mathbf{a}}{\sqrt{2\pi}}\right]^{\max}\\
\frac{\sqrt{2}}{\sqrt{\pi}}\frac{d}{dt}
\left[\frac{1}{\sqrt{s}}\exp\left(-\frac{\left\vert  
a\right\vert^{2}}{2s}\right)*\frac{1}{\sqrt{s}}\right]\left(t\right)\\
=-i\frac{\left(-1\right)^{\dim \mathfrak p_{0}^{\perp}/2}\left(\widehat{A}^{k^{-1}\vert_{\mathfrak 
p_{0}^{\perp}}}
\left(0\right)\right)^{-1}}{\left\vert  
\det\left(1-\Ad\left(\gamma\right)\right)\vert_{ \mathfrak 
z_{0}^{\perp}}\right\vert^{1/2}}\\
\left[\widehat{A}^{k^{-1}}\left(TX\left(e^{a}\right)\right)e\left(N_{X\left(\gamma\right)/X\left(k\right)}\right)\ch
^{k^{-1}}\left(F\right)\frac{a^{*}}{\sqrt{2\pi}}\right]^{\max}
t^{-3/2}\exp\left(-\frac{\left\vert  
a\right\vert^{2}}{2t}\right).
\end{multline}
\end{thm}
\begin{proof}
 By the considerations we made after equation (\ref{eq:free2}), we get 
\begin{multline}\label{eq:free3}
\Tr^{S^{\mathfrak 
p}}\left[\frac{c\left(a\right)}{\left\vert  
a\right\vert}\Ad\left(k^{-1}\right)\exp\left(-ic\left(\ad\left(\Yok\right)\vert_{\mathfrak p}\right)\right)\right]
= \\
-i\Trs^{S^{\left\{a\right\}^{\perp}}}
\left[\Ad\left(k^{-1}\right)\exp\left(-ic\left(\ad\left(\Yok\right)\vert_{\mathfrak p}
\right)\right)\right].
\end{multline}
In the right-hand side of (\ref{eq:free3}), the supertrace is 
associated with the $\Z_{2}$-grading 
$S^{\left\{a\right\}^{\perp}}=S^{\left\{a\right\}^{\perp}}_{+} \oplus 
S^{\left\{a\right\}^{\perp}}_{-}$.
In general $\ad\left(a\right)$ does  not induce an isometry from 
$\mathfrak p_{0}^{\perp}$ into $\mathfrak k_{0}^{\perp}$. However, it 
intertwines the action of 
$\Ad\left(k^{-1}\right)\exp\left(-i\ad\left(\Yok\right)\right)$  on 
these two vector spaces. Let $S^{\mathfrak k_{0}^{\perp}}$ denote the 
spinors associated with $\mathfrak k_{0}^{\perp}$. By (\ref{eq:free3}), we get
\begin{multline}\label{eq:free3x1}
\Tr^{S^{\mathfrak 
p}}\left[\frac{c\left(a\right)}{\left\vert  
a\right\vert}\Ad\left(k^{-1}\right)\exp\left(-ic\left(\ad\left(\Yok\right)\vert_{\mathfrak p}\right)\right)\right]
= \\
-i\Trs^{S^{\left\{a\right\}^{\perp}\cap \mathfrak p_{0}}}
\left[\Ad\left(k^{-1}\right)\exp\left(-ic\left(\ad\left(\Yok\right)\vert_{\left\{a\right\}^{\perp}\cap \mathfrak p_{0}}
\right)\right)\right]\\
\Trs^{S^{\mathfrak k_{0}^{\perp}}}
\left[\Ad\left(k^{-1}\right)\exp\left(-ic\left(\ad\left(\Yok\right)_{\mathfrak k_{0}^{\perp}}
\right)\right)\right].
\end{multline}

By (\ref{eq:free1}),  (\ref{eq:free3x1}), we get
\begin{multline}\label{eq:free4}
\exp\left(\frac{1}{2}\Delta^{\mathfrak z\left(\gamma\right)}\right)
\Bigg(J_{\gamma}\left(\Yok\right)\Tr^{S^{\mathfrak p} }
\left[\frac{c\left(a\right)}
{\left\vert  a\right\vert}\Ad\left(k^{-1}\right)\exp\left(-ic
\left(\ad\left(\Yok\right)\vert_{\mathfrak p}\right)\right)\right]\\
\Tr^{E}\left[\rho^{E}\left(k^{-1}\right)
\exp\left(-i\rho^{E}\left(\Yok\right)\right)\right]\delta_{Y_{0}^{\mathfrak p}=a}
\Biggr)\left(0\right)\\
=-\frac{i}{\left\vert  
\det\left(1-\Ad\left(\gamma\right)\right)\vert_{ \mathfrak 
z_{0}^{\perp}}\right\vert^{1/2}}\frac{\exp\left(-\left\vert  a\right\vert^{2}/2\right)}{\left(2\pi 
\right)^{p/2}}
\exp\left(\frac{1}{2}\Delta^{\mathfrak 
k\left(\gamma\right)}\right)\\
\Biggl( J'_{k^{-1}}\left(\Yok\right)
\Trs^{S^{\left\{a\right\}^{\perp}\cap \mathfrak p_{0}}}
\left[\Ad\left(k^{-1}\right)\exp\left(-ic\left(\ad\left(\Yok\right)\vert_{\left\{a\right\}^{\perp}\cap \mathfrak p_{0}}
\right)\right)\right]\\
\Trs^{S^{\mathfrak k_{0}^{\perp}}}
\left[\Ad\left(k^{-1}\right)\exp\left(-ic\left(\ad\left(\Yok\right)_{\mathfrak k_{O}^{\perp}}
\right)\right)\right]
\Tr^{E}\left[\rho^{E}\left(k^{-1}\right)
\exp\left(-i\rho^{E}\left(\Yok\right)\right)\right]\Biggr) .
\end{multline}

Let $\Lambda \subset \mathfrak t$ be the coroot lattice associated 
with $K$, so
that $T = \mathfrak t/\Lambda$.
Recall that $R _{+}\subset \mathfrak t^{*}$ is a positive root system associated 
with the group $K$. Let $R_{+} \left(a\right)\subset R_{+}$ be a positive subroot 
system associated with the group $K^{0}_{0}=K^{0}\left(a\right)$. Let $W,W_{0}$ be the Weyl groups of $K,K_{0}^{0}$. Then $W_{0} 
\subset W$.

Let 
$\sigma\left(t\right),\sigma_{0}\left(t\right),
t\in 
T$ be the denominators of Weyl's character formulas for the groups 
$K,K_{0}^{0}$. Then
\begin{align}\label{eq:nauf3}
&\sigma\left(t\right)=\prod_{\alpha\in R_{+}}\left(e^{i\pi
\left\langle  \alpha,t\right\rangle}-e^{-i\pi\left\langle  
\alpha,t\right\rangle}\right),
&\sigma_{0}\left(t\right)=\prod_{\alpha\in R_{+,0}}^{}\left(e^{i\pi
\left\langle  \alpha,t\right\rangle}-e^{-i\pi\left\langle  
\alpha,t\right\rangle}\right).
\end{align}
Since $K$ is simply connected, the function $\sigma\left(t\right)$ is well defined 
on $T$. Since $K^{0}\left(a\right)$ is not necessarily simply 
connected, the function $\sigma_{0}\left(t\right)$ is only defined up 
to sign on $T$. 
Moreover, for $t\neq 0$, we have the identity up to sign
\begin{equation}\label{eq:nauf4}
\Trs^{S^{\mathfrak k_{0}^{\perp}}}
\left[\Ad\left(t^{-1}\right)\right]=\frac{\sigma}{\sigma_{0}}\left(t\right).
\end{equation}
Note here again that  by nature, both sides of (\ref{eq:nauf4}) are only defined up to sign.

  Let $\chi_{\lambda}$ be the character of the 
representation of the irreducible representation $\rho^{E}$ of $K$ with highest 
weight $\lambda$.  By Weyl's character formula, for 
$t\in T$, we get
\begin{equation}\label{eq:nauf5}
\frac{\sigma}{\sigma_{0}}\left(t\right)\chi_{\lambda}\left(t\right)=
\frac{1}{\sigma_{0}\left(t\right)}\sum_{w\in 
W}^{}\epsilon_{w}\exp\left(2i\pi\left\langle \rho+ 
\lambda,wt\right\rangle\right),
\end{equation}
the two sides being well-defined up to sign.
Note that (\ref{eq:nauf5}) is $W_{0}$-invariant. In the sequel, we 
denote by $\chi_{\lambda}^{0}$ the associated  function  on 
$K_{0}^{0}$. By (\ref{eq:nauf4}), (\ref{eq:nauf5}), for $t\in T$, we get
\begin{equation}\label{eq:nauf5x1}
\chi_{\lambda}^{0}\left(t\right)=\left(-1\right)^{\left\vert  
R_{+}\setminus R_{+,0}\right\vert}\Trs^{S^{\mathfrak 
k_{0}^{\perp}}}\left[\Ad\left(t\right)\right]\chi_{\lambda}\left(t\right).
\end{equation}
Again, both sides in (\ref{eq:nauf5x1}) have the same sign ambiguity.

By (\ref{eq:nauf5}), we get
\begin{multline}\label{eq:nauf6}
	\frac{1}{\left(2\pi\right)^{p/2}}\exp\left(\Delta^{\mathfrak 
k\left(\gamma\right)}/2\right)\Biggl( J'_{k^{-1}}\left(\Yok\right)\\
\Trs^{S^{\left\{a\right\}^{\perp}\cap \mathfrak p_{0}}}
\left[\Ad\left(k^{-1}\right)\exp\left(-ic\left(\ad\left(\Yok\right)\vert_{\left\{a\right\}^{\perp}\cap \mathfrak p_{0}}
\right)\right)\right]\\
\Trs^{S^{\mathfrak k_{0}^{\perp}}}
\left[\Ad\left(k^{-1}\right)\exp\left(-ic\left(\ad\left(\Yok\right)\vert_{\mathfrak k_{0}^{\perp}}
\right)\right)\right]
\Tr^{E}\left[\rho^{E}\left(k^{-1}\right)
\exp\left(-i\rho^{E}\left(\Yok\right)\right)\right]\Biggr)\left(0\right)\\
=\frac{\left(-1\right)^{\left\vert  R_{+}\setminus 
R_{+,0}\right\vert}}{\left(2\pi\right)^{p/2}}
\exp\left(\Delta^{\mathfrak k\left(\gamma\right)}/2\right)
\Biggl(J'_{k^{-1}}\left(\Yok\right)\\
\Trs^{S^{\left\{a\right\}^{\perp}\cap \mathfrak p_{0}}}
\left[\Ad\left(k^{-1}\right)\exp\left(-ic\left(\ad\left(\Yok\right)\vert_{\left\{a\right\}^{\perp}\cap \mathfrak p_{0}}
\right)\right)\right] 
\chi_{\lambda}^{0}\left(k^{-1}e^{-i\Yok}\right)\Biggr)\left(0\right).
\end{multline}
Again, the last two terms in the right-hand side of (\ref{eq:nauf6}) 
suffer from a $\pm 1$ ambiguity, but their product is unambiguously 
defined.
We denote by $L$ the expression in  the right-hand side (\ref{eq:nauf6}).

Let $\Omega^{\mathfrak z\left(\gamma\right)}$ be the curvature of the 
canonical connection on the $K^{0}\left(\gamma\right)$ principal 
bundle $Z^{0}\left(\gamma\right)\to X\left(\gamma\right)$. Then 
$\Omega^{\mathfrak z\left(\gamma\right)}$ is a
$\mathfrak k\left(\gamma\right)$-valued $2$-form. By proceeding as in \cite[eqs. (7.7.7)--(7.7.13)]{Bismut08b}, we 
deduce from (\ref{eq:nauf5}), (\ref{eq:nauf6}) that
\begin{multline}\label{eq:nauf7}
L=\left(-1\right)^{\left\vert  R_{+}\setminus 
R_{+,0}\right\vert}\frac{\exp\left(2\pi^{2}\left\vert  
\rho+\lambda\right\vert^{2}\right)}{\left(2\pi\right)^{p/2}}\\
\left[\widehat{A}^{k^{-1}}\left(i\ad\left(\Omega^{\mathfrak 
z\left(\gamma\right)}\right)\vert_{\mathfrak 
p_{0}}\right)\chi_{\lambda}^{0}\left[k^{-1}
\exp\left(i\Omega^{\mathfrak 
z\left(\gamma\right)}\right)\right]\mathbf{a}\right]^{\max}.
\end{multline}
The reason why $\mathbf{a}$  appears is because  when following the methods 
of \cite{Bismut08b}, we should obtain forms of maximal degree on 
the orthogonal space to $a$ in $TX\left(\gamma\right)$. When multiplying by $\mathbf{a}$, we obtain the 
corresponding term of maximal degree on $TX\left(\gamma\right)$.

By the argument that was given after (\ref{eq:free3}), and by (\ref{eq:nauf5x1}), we can rewrite 
(\ref{eq:nauf7}) in the form
\begin{multline}\label{eq:nauf8}
L=\frac{\exp\left(2\pi^{2}\left\vert  
\rho+\lambda\right\vert^{2}\right)}{\left(2\pi\right)^{p/2}}
\Biggl[\widehat{A}^{k^{-1}}\left(i\ad\left(\Omega^{\mathfrak 
z\left(\gamma\right)}\right)\vert_{\mathfrak 
p_{0}}\right)\\
\Trs^{S^{\mathfrak p_{0}^{\perp}}}\left[
\Ad\left(k^{-1}\right)\exp\left(ic\left(\ad\left(\Omega^{\mathfrak 
z\left(\gamma\right)}\right)\vert_{\mathfrak p_{0}^{\perp}}\right)\right)\right]\\
\Tr^{E}\left[ \rho^{E}  \left(k^{-1}\right) 
\exp\left(i\rho^{E}\left(\Omega^{\mathfrak 
z\left(\gamma\right)}\right)\right) \right]\mathbf{a}\Biggr]^{\max}.
\end{multline}

By proceeding as in \cite[eq. (7.7.7)]{Bismut08b}, we get
\begin{multline}\label{eq:nauf9}
\Trs^{S^{\mathfrak p_{0}^{\perp}}}\left[
\Ad\left(k^{-1}\right)\exp\left(ic\left(\ad\left(\Omega^{\mathfrak z\left(\gamma\right)}\right)\right)\right)\right]\\
=\mathrm{Pf}\left[-\ad\left(\Omega^{\mathfrak z\left(\gamma\right)}\right)\vert_{\mathfrak 
p_{0}^{\perp}\cap \mathfrak p\left(k\right)}\right]
\widehat{A}^{-1}\left(i\ad\left(\Omega^{\mathfrak z\left(\gamma\right)}\right)\vert_{\mathfrak 
p_{0}^{\perp}\cap \mathfrak 
p\left(k\right)}\right) \\
\left(\widehat{A}^{ke^{-i\Omega^{\mathfrak 
z\left(\gamma\right)}}\vert_{\mathfrak p_{0}^{\perp}\cap \mathfrak 
p\left(k\right)^{\perp}}}\left(0\right)\right)^{-1}.
\end{multline}
By (\ref{eq:nauf8}), (\ref{eq:nauf9}), we obtain
\begin{multline}\label{eq:nauf10}
L=\left(-1\right)^{\dim \mathfrak p_{0}^{\perp}/2}\exp\left(2\pi^{2}\left\vert  
\rho+\lambda\right\vert^{2}\right)\\
\left[\widehat{A}^{k^{-1}}
\left(TX\left(e^{a}\right)\right)e\left(N_{X\left(\gamma\right)/X\left(k\right)}\right)
\left(\widehat{A}^{k^{-1}}\left( 
N_{X\left(e^{a}\right)/X}\right)\right)^{-1}\ch^{k^{-1}}\left(F\right)\frac{\mathbf{a}}{\sqrt{2\pi}}\right]^{\max}.
\end{multline}

By (\ref{eq:nauf13}), (\ref{eq:nauf10}), we obtain
\begin{multline}\label{eq:nauf14}
L=\left(-1\right)^{\dim \mathfrak p_{0}^{\perp}/2}\exp\left(2\pi^{2}\left\vert  
\rho+\lambda\right\vert^{2}\right)\\
\left( \widehat{A}^{k^{-1}
\vert_{\mathfrak p_{0}^{\perp}}
}\left(0\right) \right) ^{-1}\left[\widehat{A}^{k^{-1}}\left(TX\left(e^{a}\right)\right)
e\left(N_{X\left(\gamma\right)/X\left(k\right)}\right)\ch^{k^{-1}}\left(F\right)
\frac{\mathbf{a}}{\sqrt{2\pi}}\right]^{\max}.
\end{multline}

For $s>0$, let $L_{s}$ be the obvious analogue of $L$, when replacing 
$B$ by $B/s$. We can  use equation (\ref{eq:nauf14}) to evaluate 
$L_{s}$. Recall that to properly use equation 
(\ref{eq:nauf14}), we need to modify our definition of $\alpha^{\max}$. 
Ultimately, we find that
\begin{multline}\label{eq:nauf15}
L_{s}=s^{\left(p-1\right)/2}\left(-1\right)^{\dim \mathfrak p_{0}^{\perp}/2}\exp
\left(2\pi^{2}s\left\vert  
\rho+\lambda\right\vert^{2}\right)\\
\left( \widehat{A}^{k^{-1}\vert
_{\mathfrak p_{0}^{\perp}}
}\left(0\right) \right) ^{-1}\left[\widehat{A}^{k^{-1}}\left(TX\left(e^{a}\right)\right)
e\left(N_{X\left(\gamma\right)/X\left(k\right)}\right)\ch^{k^{-1}}\left(F\right)
\frac{\mathbf{a}}{\sqrt{2\pi}}\right]^{\max}.
\end{multline}

By (\ref{eq:conk4x3}), (\ref{eq:free1}),  (\ref{eq:free4}), 
(\ref{eq:nauf6}),  and 
(\ref{eq:nauf15}), we get the first identity in (\ref{eq:free9}). 
By (\ref{eq:nauf16x1}), (\ref{eq:nauf18}), and the identity just 
proved, we get the second identity in (\ref{eq:free9}). The proof of our theorem is completed. 
\end{proof}
\begin{remk}\label{Rintb}
We will check that the two terms in the right-hand side of 
(\ref{eq:free9}) lie in $\mathrm{o}\left(\mathfrak p\right)$, i.e., 
they behave properly under change of the orientation of $\mathfrak p$.  Indeed note that 
$e\left(N_{X\left(\gamma\right)}/X\left(k\right)\right)$ takes its 
values in $\mathrm{o}\left(N_{X\left(\gamma\right)/X\left(k\right)}\right)$, 
which is modelled on the orientation line $\mathrm{o}\left(
\mathfrak p_{0}^{\perp}\cap \mathfrak p\left(k\right)\right)$. 
Therefore, the right-hand side of (\ref{eq:free9}) takes its values 
in $\mathrm{o}\left(\mathfrak p\left(\gamma\right)\right) \otimes 
\mathrm{o}\left(\mathfrak p_{0}^{\perp}\cap \mathfrak 
p\left(k\right)\right)$. Using (\ref{eq:laus9x1}), we find that 
the right-hand side of (\ref{eq:free9}) is indeed a section of 
$\mathrm{o}	\left(\mathfrak p\right)$.
\end{remk}
 

%% file: Eta8.tex
\section{Scalar hypoelliptic operators and their 
corresponding diffusions}%
\label{sec:extra}
The purpose of this section is to establish uniform estimates on 
 the heat kernel 
$r^{X}_{b,t}\left(\left(x,Y\right),\left(x',Y'\right)\right)$ of a 
scalar version $\mathcal{A}^{X}_{b}$ of the operator $\mathcal{L}^{X}_{b}$. These estimates 
were established in \cite[chapters 12 and 13]{Bismut08b} in the range 
$0<b\le b_{0}, \epsilon\le t\le M$ with $0<\epsilon\le M <+ \infty $. 
For later purposes, we need to extend such estimates in the range 
$0<b, \tau b^{2}\le t\le M$, with $\tau>0,M>0$. 

The techniques  we use to establish our estimates are exactly the 
ones in \cite{Bismut08b}. They combine the Malliavin calculus with 
uniform estimates
on the rate of escape of an open ball of 
the associated hypoelliptic diffusions, a result of independent 
interest, and a slight reinforcement of the results in 
\cite{Bismut08b}.

This section is organized as follows. In subsection 
\ref{subsec:caeu}, if $E$ is a Euclidean vector space, we recall the definition of the harmonic oscillator, 
we give Mehler formula for its heat kernel, and we describe its 
probabilistic interpretation in terms of the Ornstein-Uhlenbeck 
process $Y_{\cdot}$.

In subsection \ref{subsec:unibo}, we obtain a uniform bound on 
$\sup_{0\le s\le t}\left\vert  bY_{s/b^{2}}\right\vert$ that will be 
used in section \ref{sec:fin}.

In subsection \ref{subsec:hypopro}, we define the scalar hypoelliptic 
Laplacian $\mathcal{A}^{E}_{b}$ on $E\times E$,  and we give the 
formula for its heat kernel that was obtained in \cite{Bismut08b}.

In subsection \ref{subsec:bro},  if $X=G/K$, we define the scalar 
elliptic 
heat kernel, which we relate to classical Brownian motion on $X$. 
Also we state It\^{o}'s formula, and we recall a well-known result on 
the rate of escape of Brownian motion from a ball.

In subsection \ref{subsec:scaana}, we introduce the hypoelliptic 
scalar operator $\mathcal{A}^{X}_{b}$ on the total space 
$\mathcal{X}$ of $TX$.

In subsection \ref{subsec:probhea}, we construct the corresponding 
hypoelliptic heat operators, and the associated hypoelliptic 
diffusions
on $\mathcal{X}$.

In subsection \ref{subsec:geito}, we give a generalized formula of 
It\^{o} for the hypoelliptic diffusion. This formula was already 
established in \cite{Bismut08b}. However, we present it here as a 
suitable convolution of the classical It\^{o} formula for the 
Brownian motion on $X$. In the next sections, we will use  
this convolution formula. This approximate It\^{o} formula is of 
independent interest.

In subsection \ref{subsec:esc}, we establish the uniform rate of 
escape of the projection on $X$ of our hypoelliptic diffusions. We 
need to establish these estimates with the proper Gaussian weight on 
the coordinate $Y\in TX$. The techniques we rely on are the ones in 
\cite{Bismut08b}.

In subsection \ref{subsec:copro}, we recall the result established in 
\cite{Bismut08b} stating that as $b\to 0$, the projection on $X$ of 
the hypoelliptic diffusion converges in probability law to Brownian 
motion on $X$.

Finally, in subsection \ref{subsec:heatsce}, we establish the suitable uniform 
upper bound for the hypoelliptic heat kernel 
$r^{X}_{b,t}\left(\left(x,Y\right),\left(x',Y'\right)\right)$ on $\mathcal{X}$.
\subsection{Harmonic oscillator and Mehler formula}%
\label{subsec:caeu}
Let $E$ still be a Euclidean vector space of dimension $m$, and let 
$Y$ be the generic element of $E$. As in \cite[Proposition 
10.3.1]{Bismut08b}, put
\begin{equation}\label{eq:phan1}
H^{*}_{t}\left(Y,Y'\right)=
\frac{1}{2}\left( \tanh\left(t/2\right)
    \left(\left\vert  Y  \right\vert^{2}+\left\vert  
    Y'\right\vert^{2}\right)+\frac{1}{\sinh\left(t\right)}\left\vert  
    Y'-Y  \right\vert^{2} \right).
\end{equation}

Let
\index{DE@$\Delta ^{E}$}%
$\Delta ^{E}$ be the Laplacian 
on $E$. 
Let 
\index{OE@$O^{E}$}%
$O ^{E}$ be the 
harmonic 
oscillator on $E$, 
\begin{equation}
    O ^{E}=\frac{1}{2}\left(-\Delta ^{E}+\left\vert  
    Y\right\vert^{2}-m\right).
    \label{eq:glab14}
\end{equation}
Given $t>0$, let  
\index{hEY@$h_{t} ^{E}\left(Y ,Y'\right)$}%
$h_{t} ^{E}\left(Y ,Y'\right)$ be the smooth kernel 
associated with $\exp\left(-tO^{E}\right)$ with respect to $dY'$.
By 
Mehler's formula \cite{GlimmJaffe87}, \cite[eq. (10.4.2)]{Bismut08b}, we get
\begin{equation}
    h_{t}^{E}\left(Y  ,Y'\right)
    =\left(\frac{e^{t}}{
    2\pi\sinh\left(t\right)}\right)^{m/2}
    \exp\left(-H^{*}_{t}\left(Y,Y'\right)\right).
    \label{eq:glab16a}
\end{equation}
By (\ref{eq:glab16a}), we deduce that
\begin{equation}
    h_{t}^{E}\left(Y  ,Y'\right)
    \le\left(\frac{e^{t}}{
    2\pi\sinh\left(t\right)}\right)^{m/2}
    \exp\left(-\frac{1}{2}\tanh\left(t/2\right)
    \left(\left\vert  Y  \right\vert^{2}+\left\vert  
   Y'\right\vert^{2}\right)\right).
    \label{eq:glab17}
\end{equation}
As $t\to + \infty $, we get
\begin{equation}\label{eq:phan16}
h^{E}_{t}\left(Y,Y'\right)\to 
\left(\frac{1}{\pi}\right)^{m/2}\exp\left(-\frac{1}{2}\left(\left\vert  Y\right\vert^{2}
+\left\vert  Y'\right\vert^{2}\right)\right).
\end{equation}

From (\ref{eq:glab16a}), we deduce the identity in \cite[eq. 
(10.7.12)]{Bismut08b},
\begin{equation}\label{eq:glab17a}
\int_{E}^{}h^{E}_{t}\left(Y,Y'\right)dY'=\left(\frac{e^{t}}{\cosh\left(t\right)}\right)
^{m/2}\exp\left(-\frac{1}{2}\tanh\left(t\right)\left\vert  
Y\right\vert^{2}\right).
\end{equation}

Set
\index{PE@$P^{E}$}%
\begin{equation}\label{eq:phan0}
P^{E}=\exp\left(\left\vert  
Y\right\vert^{2}/2\right)O^{E}\exp\left(-\left\vert  
Y\right\vert^{2}/2\right).
\end{equation}
Let 
\index{nEV@$\n^{E,V}_{Y}$}%
$\n^{E,V}_{Y}$ denote the radial vector field on $E$. 
Then
\begin{equation}\label{eq:phan-1}
P^{E}=\frac{1}{2}\left(-\Delta^{E}+2\n^{E,V}_{Y}\right).
\end{equation}

Let $w^{E}_{\cdot}$ denote the Brownian motion in $E$ with $w^{E}_{0}=0$. 
Given $Y\in E$, consider the stochastic differential equation
\begin{align}\label{eq:phan-2}
&\dot Y=-Y+\dot w^{E},&Y_{0}=Y.
\end{align}
Then $Y_{\cdot}$ is given by
\begin{equation}\label{eq:phan-3}
Y_{t}=e^{-t}Y+\int_{0}^{t}e^{-\left(t-s\right)}dw^{E}_{s}.
\end{equation}
The process $Y_{\cdot}$ is called an Ornstein-Uhlenbeck process.
Let 
\index{Q@$Q$}%
$Q$ be the probability law of $Y_{\cdot}$ on 
$\mathcal{C}\left(\R_{+},E\right)$ in (\ref{eq:phan-3}), and let 
\index{EQ@$E^{Q}$}%
$E^{Q}$ 
be the corresponding expectation operator.
\begin{prop}\label{Psemgr}
For $t>0$, if  $f\in C^{ \infty, c}\left(E,\R\right)$, then 
\begin{align}\label{eq:phan-4}
&\exp\left(-tO^{E}\right)f\left(Y\right)=\exp\left(-\left\vert  
Y\right\vert^{2}/2\right)E^{Q}\left[\exp\left(\left\vert  
Y_{t}\right\vert^{2}/2\right)f\left(Y_{t}\right)\right],\\
&\exp\left(-tQ^{E}\right)f\left(Y\right)=E^{Q}\left[f\left(Y_{t}\right)\right]. \nonumber 
\end{align}
\end{prop}
\begin{proof}
This is a trivial consequence of It\^{o}'s formula.
\end{proof}

Instead of (\ref{eq:phan-2}), we consider the equation
\begin{align}\label{eq:phan-5}
&\dot Y=\dot w^{E},&Y_{0}=Y,
\end{align}
so that 
\begin{equation}\label{eq:phan-6}
Y_{t}=Y+w^{E}_{t}.
\end{equation}
In (\ref{eq:phan-6}), $Y_{\cdot}$ is just a Brownian motion.
We denote by
\index{P@$P$}%
$P$ the probability law of $Y_{\cdot}$ on $\mathcal{C}\left(\R_{+},E\right)$ and by 
\index{EP@$E^{P}$}%
$E^{P}$ the 
corresponding expectation operator.
\begin{prop}\label{Psemgrbi}
For $t>0$, if  $f\in C^{ \infty,c}\left(E,\R\right)$, then 
\begin{equation}\label{eq:phan-7}
\exp\left(-tO^{E}\right)f\left(Y\right)=E^{P}\left[\exp\left(\frac{mt}{2}-\frac{1}{2}\int_{0}^{t}\left\vert  Y
_{s}\right\vert^{2}ds\right)f\left(Y_{t}\right)\right].
\end{equation}
\end{prop}
\begin{proof}
Equation (\ref{eq:phan-7}) follows from (\ref{eq:glab14}) and from 
the formula of Feynman-Kac.
\end{proof}
\begin{remk}\label{RGirs}
The first equation in (\ref{eq:phan-4}) and (\ref{eq:phan-7}) are 
equivalent. Indeed an elementary version of a formula of Girsanov 
shows that if $Q^{t}, P^{t}$ are the probability laws of $Y_{\cdot}$ 
on $\mathcal{C}\left(\left[0,t\right],E\right)$ in (\ref{eq:phan-2}), (\ref{eq:phan-5}), 
 then
\begin{equation}\label{eq:phan-8}
\frac{dQ^{t}}{dP^{t}}=\exp\left(-\int_{0}^{t}\left\langle  
Y_{s},\delta 
w^{E}_{s}\right\rangle-\frac{1}{2}\int_{0}^{t}\left\vert  
Y_{s}\right\vert^{2}ds\right).
\end{equation}
Using  (\ref{eq:phan-5}) and It\^{o}'s formula, we get
\begin{equation}\label{eq:phan-9}
\frac{1}{2}\left\vert  Y_{t}\right\vert^{2}=\frac{1}{2}\left\vert  
Y\right\vert^{2}+\frac{mt}{2}+\int_{0}^{t}\left\langle  
Y_{s},\delta w^{E}_{s}\right\rangle.
\end{equation}
By (\ref{eq:phan-9}), we can rewrite (\ref{eq:phan-8}) in the form
\begin{equation}\label{eq:phan-10}
\frac{dQ^{t}}{dP^{t}}=\exp\left(\frac{mt}{2}-\frac{1}{2}\int_{0}^{t}\left\vert  Y_{s}\right\vert^{2}ds
+\frac{1}{2}\left\vert  Y\right\vert^{2}-\frac{1}{2}\left\vert  
Y_{t}\right\vert^{2}\right).
\end{equation}
By (\ref{eq:phan-10}), the two versions of 
$\exp\left(-tO^{E}\right)f\left(Y\right)$ in (\ref{eq:phan-4}), 
(\ref{eq:phan-7}) are indeed equivalent.
\end{remk}

 Let $Y_{\cdot}$ be as in (\ref{eq:phan-5}).
\begin{prop}\label{Pidpha}
Given $\beta\ge 0$, the following identity holds:
\begin{multline}\label{eq:phan-11b}
E^{P}\left[\exp\left(-\frac{\beta^{2}}{2}\int_{0}^{t}\left\vert  
Y_{s}\right\vert^{2}ds\right)\right]=\\
\left(\frac{1}{\cosh\left(\beta 
t\right)}\right)^{m/2}\exp\left(-\frac{1}{2}\tanh\left(\beta 
t\right)\beta\left\vert  Y\right\vert^{2}\right).
\end{multline}
\end{prop}
\begin{proof}
For $a>0$, set
\begin{equation}\label{eq:phan-11a}
K_{a}s\left(Y\right)=s\left(aY\right).
\end{equation}
For $\beta>0$, we get
\begin{equation}\label{eq:phan-12a}
K_{\sqrt{\beta}}\beta 
O^{E}K_{\sqrt{\beta}}^{-1}=\frac{1}{2}\left(-\Delta^{E}+\beta^{2}
\left\vert  Y\right\vert^{2}-m\beta\right).
\end{equation}
By (\ref{eq:phan-7}), (\ref{eq:phan-12a}),  we 
deduce that
\begin{equation}\label{eq:phan-13}
E^{P}\left[\exp\left(-\frac{\beta^{2}}{2}\int_{0}^{t}\left\vert  
Y_{s}\right\vert^{2}ds\right)\right]=\exp\left(-m\beta 
t/2\right)\int_{E}^{}h^{E}_{\beta t}\left(\sqrt{\beta} Y,Y'\right)dY'.
\end{equation}
By (\ref{eq:glab17a}), (\ref{eq:phan-13}), we get 
(\ref{eq:phan-11b}).
\end{proof}

For $0\le\beta\le1$, let $\rho_{\beta}\ge 0$ be defined by 
\begin{equation}\label{eq:rot6a}
\rho^{2}_{\beta}=1-\beta^{2}.
\end{equation}
Let $Y_{\cdot}$ be as in (\ref{eq:phan-2}).
\begin{prop}\label{PMehla}
The following identity holds:
\begin{multline}\label{eq:rot7a}
   E^{Q}\left[\exp\left(\frac{\beta^{2}}{2}
\int_{0}^{t}\left\vert  
Y _{s}\right\vert^{2}ds\right)\right] \\
=
\left[\frac{\exp\left( t\right)}
{\cosh\left(\rho_{\beta} t\right) 
+\frac{\sinh\left(\rho_{\beta} t\right)}{\rho
_{\beta}}}
\right]^{m/2}
\exp\left(\beta^{2}
\frac{\tanh\left(\rho_{\beta}t\right)}{
\rho_{\beta}+
\tanh\left(\rho_{\beta}t \right)}
\left\vert  Y \right\vert^{2}/2\right). 
\end{multline}
In particular, we have
\begin{equation}\label{eq:rot7ay1}
  E^{Q}\left[\exp\left(\frac{\beta^{2}}{2}
\int_{0}^{t}\left\vert  
Y _{s}\right\vert^{2}ds\right)\right] \le
\exp\left(\frac{\beta^{2}}{2}\left(mt+\left\vert  
Y\right\vert^{2}\right)\right).
\end{equation}
\end{prop}
\begin{proof}
Equation (\ref{eq:rot7a}) was established in \cite[eq. 
(13.2.54)]{Bismut08b}. By (\ref{eq:rot7a}), we get
\begin{equation}\label{eq:rot7ay2}
E^{Q}\left[\exp\left(\frac{\beta^{2}}{2}
\int_{0}^{t}\left\vert  
Y _{s}\right\vert^{2}ds\right)\right] 
\le\exp\left(m\left(1-\rho_{\beta}\right)t/2+\beta^{2}\left\vert  
Y\right\vert^{2}/2\right).
\end{equation}
Since $1-\rho_{\beta}\le 1-\rho_{\beta}^{2}=\beta^{2}$, from 
(\ref{eq:rot7ay2}), we get (\ref{eq:rot7ay1}). The proof of our proposition  is completed. 
\end{proof}
\begin{remk}\label{Rproa}
Equation (\ref{eq:rot7a}) can be extended by analyticity  to
\begin{equation}\label{eq:rot7a1}
\beta^{2}\le 1+\left(\frac{\pi}{2t}\right)^{2}.
\end{equation}
\end{remk}
\begin{prop}\label{Plines}
For $c\ge 0,t\ge 0$, for $0< \beta\le 1$, then
\begin{equation}\label{eq:af1}
E^{Q}\left[\exp\left(c\int_{0}^{t}\left\vert  
Y_{s}\right\vert 
ds\right)\right]\le\exp\left(c^{2}t/2\beta^{2}+\beta^{2}mt/2+c\left(1-e^{-t}\right)\left\vert  Y
\right\vert\right).
\end{equation}
\end{prop}
\begin{proof}
Let $\underline{Y}_{\cdot}$ be $Y_{\cdot}$ in (\ref{eq:phan-3}) with 
$Y_{0}=0$. We can rewrite (\ref{eq:phan-3}) in the form
\begin{equation}\label{eq:af2}
Y_{t}=e^{-t}Y+\underline{Y}_{t}.
\end{equation}
By (\ref{eq:af2}), we get
\begin{equation}\label{eq:af3}
\int_{0}^{t}\left\vert  Y_{s}\right\vert ds\le\int_{0}^{t}\left\vert  
\underline{Y}_{s}\right\vert+\left(1-e^{-t}\right)\left\vert  
Y\right\vert.
\end{equation}
Also
\begin{equation}\label{eq:af4}
c\int_{0}^{t}\left\vert  \underline{Y}_{s}\right \vert ds\le 
\frac{\beta^{2}}{2}\int_{0}^{t}\left\vert  
\underline{Y}_{s}\right\vert^{2}ds+\frac{c^{2}t}{2\beta^{2}}.
\end{equation}
By (\ref{eq:rot7ay1}), (\ref{eq:af3}), (\ref{eq:af4}), we get 
(\ref{eq:af1}).
\end{proof}
\subsection{A uniform estimate on $bY_{\cdot/b^{2}}$}%
\label{subsec:unibo}
We still consider equation (\ref{eq:phan-2}) and the corresponding 
probability law $Q$. 
Given $b>0$, set
\begin{equation}\label{eq:fus1}
Y_{b,\cdot}=Y_{s/b^{2}}.
\end{equation}
There is a Brownian motion, which is still denoted $w^{E}$, such that
\begin{equation}\label{eq:fus2}
dY_{b,\cdot}=-\frac{Y_{b,\cdot}}{b^{2}}+\frac{\dot w^{E}}{b}.
\end{equation}
By (\ref{eq:fus2}), we deduce that
\begin{equation}\label{eq:fus3}
Y_{b,t}=e^{-t/b^{2}}Y+\int_{0}^{t}e^{-\left(t-s\right)/b^{2}}\frac{d w^{E}_{s}}{b}.
\end{equation}
By (\ref{eq:fus3}), we deduce that
\begin{equation}\label{eq:fus4}
Y_{b,t}=e^{-t/b^{2}}Y+\frac{w^{E}_{t}}{b}-\int_{0}^{t}\frac{\exp\left(-\left(t-s\right)/b^{2}\right)}{b^{2}}
\frac{w^{E}_{s}}{b}ds.
\end{equation}
By (\ref{eq:fus4}), we deduce that
\begin{equation}\label{eq:lom7}
\sup_{0\le s\le t}\left\vert  bY_{b,s}\right\vert\le \left\vert  
bY\right\vert+\left(2-e^{-t/b^{2}}\right)\sup_{0\le s\le 
t}\left\vert  w^{E}_{s}\right\vert.
\end{equation}

By \cite[Proposition 14.10.1]{Bismut08b},  for 
$\alpha>1/2$, as $b\to 0$, $\left\vert  
\log\left(b\right)\right\vert^{-\alpha}Y_{b,\cdot}$ converges 
uniformly to $0$   over compact subsets  in probability. Equation 
(\ref{eq:lom7}) is not strong enough to produce this result.
\subsection{The hypoelliptic Laplacian on $E\times E$}%
\label{subsec:hypopro}
The generic element of $E\times E$ will be 
denoted $\left(x,Y\right)$. As in \cite[Proposition 
10.3.2]{Bismut08b}, for $b>0,t>0,x,x',Y,Y'\in E$, set
\begin{multline}\label{eq:stan6a}
H_{b,t}\left(\left(x,Y\right),\left(x',Y'\right)\right)=\frac{b^{2}}{2}\left(\tanh\left(t/2b^{2}\right)\left(\left\vert  Y\right\vert^{2}+\left\vert  Y'\right\vert^{2}
\right) 
+\frac{\left\vert  
Y'-Y\right\vert^{2}}{\sinh\left(t/b^{2}\right)}\right) \\
+\frac{1}{2\left(t-2b^{2}\tanh\left(t/2b^{2}\right)\right)}
\left\vert x'-x-b^{2}\tanh\left(t/2b^{2}\right)
\left(Y+Y'\right)\right\vert^{2}.
\end{multline}
Put
\begin{equation}\label{eq:stan6x1a}
\beta=b/\sqrt{t}.
\end{equation}
By (\ref{eq:stan6a}), we get
\begin{equation}\label{eq:stan6x2a}
H_{b,t}\left(\left(x,Y/b\right),\left(x',Y'/b\right)\right)=H_{\beta,1}\left(\left(
x/\sqrt{t},Y/\beta\right),\left(x'/\sqrt{t},Y'/\beta\right)\right).
\end{equation}
Also
\begin{multline}\label{eq:stan6x3a}
H_{\beta,1}\left(\left(
x/\sqrt{t},Y/\beta\right),\left(x'/\sqrt{t},Y'/\beta\right)\right)\\
=
\frac{1}{2} \left( \tanh\left(1/2\beta^{2}\right)\left( \left\vert  
Y\right\vert^{2}  +\left\vert  Y'\right\vert^{2}\right)+\frac{\left\vert  
Y'-Y\right\vert^{2}}{\sinh\left(1/\beta^{2}\right)}\right)\\
+\frac{1}{2\left(1-2\beta^{2}\tanh\left(1/2\beta^{2}\right)\right)}
\left\vert  \frac{x'-x}{\sqrt{t}}-\beta\tanh\left(1/2\beta^{2}\right)
\left(Y+Y'\right)\right\vert^{2}.
\end{multline}

We give the following extension of \cite[eq. (10.3.52)]{Bismut08b}.
\begin{prop}\label{Pineq}
There exists $C>0$ such that for $b>0,t>0$,
\begin{multline}\label{eq:stan8a}
H_{b,t}\left(\left(x,Y/b\right),\left(x',Y'/b\right)\right) \\
\ge 
C\left(\frac{\left\vert  
x'-x\right\vert^{2}}{t}+\left(1-e^{-t/b^{2}}\right).
\left(\left\vert  Y\right\vert^{2}+\left\vert  
Y'\right\vert^{2}\right)\right).
\end{multline}
In particular  if   $\tau>0$ is such that $t\ge\tau b^{2}$, then 
\begin{equation}\label{eq:stan8b}
H_{b,t}\left(\left(x,Y/b\right),\left(x',Y'/b\right)\right)\ge C
\left(\frac{\left\vert  x'-x\right\vert^{2}}{t}+\left(1-e^{-\tau}\right)\left( \left\vert  
Y\right\vert^{2}+\left\vert  Y'\right\vert^{2}\right) \right) .
\end{equation}
\end{prop}
\begin{proof}
By (\ref{eq:stan6x2a}), (\ref{eq:stan6x3a}), if 
$\beta\tanh\left(1/2\beta^{2}\right)\left(\left\vert  
Y\right\vert+\left\vert  Y'\right\vert\right)\le\left\vert  
x'-x\right\vert/2\sqrt{t}$, (\ref{eq:stan8a}) holds true. If 
$\beta\tanh\left(1/2\beta^{2}\right)\left(\left\vert  
Y\right\vert+\left\vert  Y'\right\vert\right)>\left\vert  
x'-x\right\vert/2\sqrt{t}$, then
\begin{equation}\label{eq:stan8z1a}
\tanh\left(1/2\beta^{2}\right)\left(\left\vert  
Y\right\vert^{2}+\left\vert  Y'\right\vert^{2}\right)\ge 
C\frac{1}{\beta^{2}\tanh\left(1/2\beta^{2}\right)}\frac{\left\vert  
x'-x\right\vert^{2}}{t}.
\end{equation}
Since for $x\ge 0$, $\tanh\left(x\right)\le x$, from 
(\ref{eq:stan8z1a}), we get
\begin{equation}\label{eq:stan8z2a}
\tanh\left(1/2\beta^{2}\right)\left(\left\vert  
Y\right\vert^{2}+\left\vert  Y'\right\vert^{2}\right)\ge 
2C\frac{\left\vert  x'-x\right\vert^{2}}{t},
\end{equation}
so that (\ref{eq:stan8a}) still holds. From (\ref{eq:stan8a}), we get 
(\ref{eq:stan8b}). The proof of our proposition is 
completed is completed. 
\end{proof}

 We will consider hypoelliptic differential operators on $E\times E$. 
 Recall that the generic element of $E\times E$  is denoted $\left(x,Y\right)$. The harmonic oscillator  $O^{E}$ will act 
on the second factor $E$. We denote by $\n^{H}_{Y}$ the vector field 
on $E\times E$ that differentiates along the first copy $E$ in the 
direction $Y$. 
\begin{defin}\label{DAb}
Given $b>0$, let 
\index{AEb@$\mathcal{A}^{E}_{b}$}%
$\mathcal{A}^{E}_{b}$ be the differential operator 
\begin{equation}\label{eq:phan2}
\mathcal{A}^{E}_{b}=\frac{O^{E}}{b^{2}}-\frac{\n^{H}_{Y}}{b}.
\end{equation}
\end{defin}
Then $\frac{\pa}{\pa t}+\mathcal{A}^{E}_{b}$ is a hypoelliptic 
operator on 
$E\times E$.
\begin{defin}\label{Drkereu}
For $b>0,t>0$, let 
\index{rEbt@$r_{b,t}^{E}\left(\left(x,Y\right),\left(x',Y'\right)\right)$}%
$r_{b,t}^{E}\left(\left(x,Y\right),\left(x',Y'\right)\right)$ be the 
smooth kernel for the operator 
$\exp\left(-t\mathcal{A}^{E}_{b}\right)$ with respect to the volume 
$dx'dY'$.
\end{defin}

Here is a result established in \cite[Proposition 10.5.1]{Bismut08b}.
\begin{prop}\label{Pideker}
For $b>0,t>0$, the following identity holds:
\begin{multline}\label{eq:phan3}
r^{E}_{b,t}\left(\left(x,Y\right),\left(x',Y'\right)\right)=
       \left[\frac{e^{t/b^{2}}}{4\pi^{2}\sinh\left(t/b^{2}\right)\left(t-2b^{2}
       \tanh\left(t/2b^{2}\right)\right)}\right]^{m/2}  \\
       \exp\left(-H_{b,t}\left(\left(x,Y/b\right),\left(x',Y'/b\right)\right)\right).
\end{multline}
\end{prop}
\begin{remk}\label{Reuuhea}
By (\ref{eq:stan8a}), since for $y\ge 0$, the function 
$\tanh\left(y\right)/y$ is 
decreasing, given $\tau>0$, there exist
$C_{\tau}>0,C'>0$  such that if $b>0,t\ge \tau b^{2}$, then
\begin{equation}\label{eq:phan4}
r^{E}_{b,t}\left(\left(x,Y\right),\left(x',Y'\right)\right)\le 
\frac{C_{\tau}}{t^{m/2}}\exp\left(-C'\left(\frac{\left\vert  
x'-x\right\vert^{2}}{t}+\left(1-e^{-\tau}\right)\left(\left\vert  
Y\right\vert^{2}+\left\vert  Y'\right\vert^{2}\right)\right) \right) .
\end{equation}
\end{remk}
\subsection{Elliptic heat kernel and Brownian motion}%
\label{subsec:bro}
Let 
\index{DX@$\Delta^{X}$}%
$\Delta^{X}$ be the Laplace-Beltrami operator on $X$. 
For $t>0$, let
\index{pt@$p_{t}\left(x,x'\right)$}%
$p_{t}\left(x,x'\right)$ be the smooth 
kernel associated with 
$\exp\left(t\Delta^{X}/2\right)$ with respect to the volume $dx'$.
Classically, given $M>0$, there exist
$C>0,C'>0$ such that for $0<t\le M,x,x'\in X$, 
\begin{equation}\label{eq:stan0}
\left\vert  p_{t}\left(x,x'\right)\right\vert\le 
Ct^{-m/2}\exp\left(-C'd^{2}\left(x,x'\right)/t\right).
\end{equation}
The uniformity of the constants is a consequence of the fact 
that $X$ is a symmetric space.

Set $x_{0}=p1$. Let $P_{x_{0}}$ be the probability law on 
$\mathcal{C}\left(\R_{+},X\right)$ of the Brownian motion in $X$ starting at 
$x_{0}$.  There is a Brownian motion $w^{TX}_{\cdot}$ with values in 
$T_{x_{0}}X$ such that
\begin{equation}\label{eq:anst-1}
\dot x=\dot w^{TX}.
\end{equation}
In (\ref{eq:anst-1}), $\dot w^{TX}$ denotes the differential of 
$w^{TX}$ in the sense of Stratonovitch
\footnote{Brownian motion is nowhere differentiable. Still, there is 
an efficient calculus along its trajectories, the It\^o calculus. Because it is based on a mean-variance 
description of local variations, it is not coordinate invariant. 
The main advantage of the calculus of Stratonovitch is that it is 
invariant under change of coordinates. The two calculi can be deduced 
from each other. We refer to Ikeda-Watanabe \cite{IkedaWatanabe89} and Le Gall \cite{LeGall16} for more details.}. The precise meaning of (\ref{eq:anst-1}) is that $\dot x$ is the 
parallel transport along $x_{\cdot}$ of $\dot w^{TX}\in T_{x_{0}}X$ 
with respect to the Levi-Civita connection. The theory of 
stochastic differential equations gives an unambiguous 
meaning to (\ref{eq:anst-1}). Let $E$ be the corresponding 
expectation operator.

We denote by 
\index{Cc@$C^{\infty ,c}\left(X,\R\right)$}%
$C^{\infty ,c}\left(X,\R\right)$ the vector space of 
smooth real functions with compact support. The same notation will be 
used when $X$ is replaced by any other space.

If $f\in C^{\infty,c}\left(X,\R\right)$, the basic relation 
between the heat operator and Brownian motion is that
for $t>0$, 
\begin{equation}\label{eq:phan-14}
\exp\left(t\Delta^{X}/2\right)f\left(x_{0}\right)=E\left[f\left(x_{t}\right)\right].
\end{equation}

Let $f:X\to \R$ be a smooth function. Let 
\index{dwTX@$\delta w^{TX}$}%
$\delta w^{TX}$ be the It\^{o} 
differential of $w^{TX}$.  By It\^{o}'s formula, from (\ref{eq:anst-1}), we get
\begin{equation}\label{eq:anfs-2}
f\left(x_{t}\right)=f\left(x_{0}\right)+\int_{0}^{t}\frac{1}{2}\Delta^{X}f\left(x_{s}\right)ds+\int_{0}^{t}\left\langle  \n f\left(x_{s}\right)
,\delta w^{TX}_{s}\right\rangle.
\end{equation}

For later purposes, we first reprove a well-known result on the 
process $x_{\cdot}$.
\begin{prop}\label{Pesbr}
Given $M>0$, there exist $C>0,C'>0$ such that for $0<t\le M, r>0$, 
\begin{equation}\label{eq:anfst1}
P_{x_{0}}\left[\sup_{0\le s\le t}d\left(x_{0},x_{s}\right)\ge 
r\right]\le C\exp\left(-C'r^{2}/t\right).
\end{equation}
\end{prop}
\begin{proof}
Let $k:\R_{+}\to \R_{+}$ be a smooth increasing  function such that 
\begin{align}\label{eq:glab41}
    k\left(u\right)=&\ u^{2}\ \mathrm{for} \ u\le 
1/2,\\
&\ u\ \mathrm{for}\ u\ \ge 1. \nonumber 
\end{align}
We fix $\epsilon>0$. For $z\in X$, set
\begin{equation}
    f\left(z\right)=\epsilon k\left(\frac{d\left(x_{0},z\right)}{\epsilon}\right).
    \label{eq:glab42}
\end{equation}
Then $f$ is a smooth function on $X$. Moreover, 
$d\left(x_{0},z\right)\ge \epsilon$ if and only if 
$f\left(z\right)\ge \epsilon$. 

We use (\ref{eq:anfs-2}) with the previous choice of $f$. By 
\cite[Proposition 13.1.2]{Bismut08b}, we get
\begin{equation}\label{eq:ansf-3}
\Delta^{X}f\le C.
\end{equation}
By (\ref{eq:anfs-2}), (\ref{eq:ansf-3}), for $r\ge\epsilon$, we get
\begin{equation}\label{eq:ansf-4}
P_{x_{0}}\left[\sup_{0\le s\le t}d\left(x_{0},x_{s}\right)\ge 
r\right]\le P_{x_{0}}\left[\sup_{0\le s\le 
t}\int_{0}^{t}\left\langle  \n f\left(x_{s}\right),\delta 
w^{TX}_{s}\right\rangle\ge r-Ct/2\right].
\end{equation}
Since $\left\vert  \n d\left(x_{0},x\right)\right\vert=1$, 
$\n_{\cdot}f$ is uniformly bounded,. By \cite[eq. (2.1) in Theorem 
 4.2.1]{StroockVaradhan79}, we get
 \begin{equation}\label{eq:ansf-5}
P_{x_{0}}\left[\sup_{0\le s\le 
t}\int_{0}^{t}\left\langle  \n f\left(x_{s}\right),\delta 
w^{TX}_{s}\right\rangle\ge r-Ct/2\right]\le 
c\exp\left(-c'\left(r-Ct/2\right)^{2}/t\right).
\end{equation}
Also for $ 0<t\le M$, we get
\begin{equation}\label{eq:ansf-6}
\left(r-Ct/2\right)^{2}/t\ge r^{2}/t-Cr\ge r^{2}/2t -C^{2}t/2\ge 
r^{2}/2t-C^{2}M/2.
\end{equation}
By (\ref{eq:ansf-4})--(\ref{eq:ansf-6}), when $r\ge \epsilon$, we get (\ref{eq:anfst1}). 

Using the exponential map, we can identify a neighbourhood of $x_{0}=p1$ in 
$X$ with a neighbourhood  of $0$ in $\mathfrak p$. To handle the case where 
$r>0$ is small, we can instead take smooth 
functions $f_{1},\ldots,f_{m}$ on $X$ with compact support, which 
coincide with coordinates $x^{1},\ldots,x^{m}$ near $p1 \simeq 0$.  
By proceeding  as before with the functions 
$f_{1},\ldots,f_{m}$ instead of $f$, we obtain (\ref{eq:anfst1}) in the case of a 
small $r>0$.
The proof of our proposition  is completed. 
\end{proof}
\subsection{The scalar analogues of the operator $\mathcal{L}^{X}_{b}$}%
\label{subsec:scaana}
Let $\pi:\mathcal{X}\to X$ be the total space of $TX = G\times _{K} \mathfrak 
p$, and let $Y^{TX}$ be the tautological section of $\pi^{*}TX$ on 
$\mathcal{X}$. We denote  by $\n_{Y^{TX}}$ the generator of the 
geodesic flow on $\mathcal{X}$.  Let $\n^{V}$ denotes differentiation along the 
    fibre $TX$. Then $\n^{V}_{Y^{TX}}$ is the fibrewise radial vector 
    field.  As in \cite[eqs. (11.1.1) and (11.1.2)]{Bismut08b}, let 
\index{AbX@$\mathcal{A}_{b}^{X}$}%
\index{BbX@$\mathcal{B}_{b}^{X}$}%
$\mathcal{A}_{b}^{X},\mathcal{B}^{X}_{b}$ be the scalar differential 
operators  
on $\mathcal{X}$, 
\begin{align}\label{eq:glab5}
    &\mathcal{A}_{b}^{X}=\frac{1}{2b^{2}}\left(-\Delta^{TX}+\left\vert  
    Y^{TX}\right\vert^{2}-m\right)-\frac{1}{b}\n_{Y^{TX}},\\
    &\mathcal{B}_{b}^{X}=\frac{1}{2b^{2}}\left(-\Delta^{TX}+2\n^{V}_{Y^{TX}}\right)-
    \frac{1}{b}\n_{Y^{TX}}. \nonumber 
    \end{align}
Then
\begin{equation}
    \mathcal{B}_{b}^{X}=\exp\left(\left\vert  Y^{TX}\right\vert^{2}/2\right)
    \mathcal{A}_{b}^{X}\exp\left(-\left\vert  Y^{TX}\right\vert^{2}/2\right).
    \label{eq:glab6}
\end{equation}
By H\"{o}rmander \cite{Hormander67}, the operators $\frac{\pa}{\pa 
t}+\mathcal{A}^{X}_{b},\frac{\pa}{\pa t}+\mathcal{B}^{X}_{b}$ are 
hypoelliptic. 

By  \cite[section 11.5]{Bismut08b}, for 
$t>0$, the heat operators 
$\exp\left(-t\mathcal{A}^{X}_{b}\right),\exp\left(-t\mathcal{B}^{X}_{b}\right)$ are unambiguously defined. 
The difficulty in defining them properly is that $X$ is noncompact.
\subsection{Hypoelliptic heat operators and probability}%
\label{subsec:probhea}
Recall that $x_{0}=p1$, and that $TX_{x_{0}} \simeq \mathfrak p$.  
Let 
\index{wp@$w_{\cdot}^{\mathfrak p}$}%
$w_{\cdot}^{\mathfrak p}$ be a Brownian motion with values in 
$\mathfrak p$ such that $w^{\mathfrak p}_{0}=0$,
and let
\index{wTX@$w^{TX}_{\cdot}$}%
$w^{TX}_{\cdot}$ denote the corresponding 
Brownian motion in $T_{x_{0}}X$. Let  
\index{Q@$Q$}%
$Q$ be the probability law of 
$w^{\mathfrak p}_{\cdot}$ on $\mathcal{C}\left(\R_{+},\mathfrak 
p\right)$, and let 
\index{EQ@$E^{Q}$}%
$E^{Q}$ denote the corresponding 
expectation operator.

    Now we follow 
    \cite[section 14.2]{BismutLebeau06a}, and   \cite[section 
    12.2]{Bismut08b}. Fix $Y^{ \mathfrak p}\in  \mathfrak p$.
    Consider the stochastic differential equation on $\mathcal{X}$,
    \begin{align}\label{eq:glab9a6a}
    &\dot x=\frac{Y^{TX}}{b},
    &\dot Y^{TX}=-\frac{Y^{TX}}{b^{2}}+\frac{\dot w^{TX}}{b},\\
    &x_{0}=p1,  &Y^{TX}_{0}=Y^{ \mathfrak p}. \nonumber 
    \end{align}
    In (\ref{eq:glab9a6a}), $\dot Y^{TX}$ denotes the covariant 
    derivative of $Y^{TX}$ with respect to $\n^{TX}$.
    The first line of equation (\ref{eq:glab9a6a}) can  be rewritten in the form 
    \begin{align}\label{eq:glab9a7}
    &b^{2}\ddot x+  \dot x = \dot w^{TX},
    &Y^{TX}=b\dot x.
    \end{align}

    Instead of (\ref{eq:glab9a6a}), as in \cite[aq. 
    (12.2.8)]{Bismut08b}, we may instead consider the horizontal lift 
    $g_{\cdot}$ of $x_{\cdot}$ in $G$, i.e., consider  the 
    system 
    \begin{align}\label{eq:glab9a6ay1}
&\dot g=\frac{Y^{\mathfrak p}}{b},
    &\dot Y^{\mathfrak p}=-\frac{Y^{\mathfrak p}}{b^{2}}+\frac{\dot 
    w^{\mathfrak p}}{b},\\
    &g_{0}=1,  &Y^{\mathfrak  p}_{0}=Y^{ \mathfrak p}, \nonumber 
\end{align}
so that
\begin{equation}\label{eq:jar7}
x_{\cdot}=\pi g_{\cdot}.
\end{equation}
\begin{prop}\label{Prep1}
If $F\in C^{\infty, c}\left(\mathcal{X},\R\right)$, for $t>0$,
then
\begin{align}\label{eq:tch2}
 &\exp\left(-t\mathcal{A}^{X}_{b}\right)F\left(x_{0},Y^{TX}_{0}\right)=\exp\left(-\left\vert  
Y^{TX}\right\vert^{2}/2\right) \nonumber \\
&\qquad\qquad E^{Q}\left[\exp\left(\left\vert  
Y^{TX}_{t}\right\vert^{2}/2\right)F\left(x_{t},Y^{TX}_{t}\right)\right],\\   
&\exp\left(-t\mathcal{B}^{X}_{b}\right)F\left(x_{0},Y^{TX}_{0}\right)=E^{Q}\left[F\left(x_{t},Y^{TX}_{t}\right)\right].\nonumber 
\end{align}
\end{prop}
\begin{proof}
The second identity in (\ref{eq:tch2}) was established in 
 \cite[Theorem 12.2.1]{Bismut08b} using the It\^{o} calculus. Using (\ref{eq:glab6}), we also 
obtain the first identity. The proof of our proposition  is completed. 
\end{proof}

We give another construction of the semigroup 
$\exp\left(-t\mathcal{A}^{X}_{b}\right)$. 
Instead of (\ref{eq:glab9a6a}), we consider the stochastic differential 
equation
 \begin{align}\label{eq:glab9a6ag}
    &\dot x=\frac{Y^{TX}}{b},
    &\dot Y^{TX}=\frac{\dot w^{TX}}{b},\\
    &x_{0}=p1,  &Y^{TX}_{0}=Y^{ \mathfrak p}. \nonumber 
    \end{align}
    By (\ref{eq:glab9a6ag}), we get
    \begin{equation}\label{eq:gagl1}
b^{2}\ddot x=\dot w^{TX}.
\end{equation}
Instead of (\ref{eq:glab9a6ay1}), we consider the system
\begin{align}\label{eq:flab9a6ay1a}
&\dot g=\frac{Y^{\mathfrak p}}{b},&\dot Y^{\mathfrak p}=\frac{\dot 
w^{\mathfrak p}}{b},\\
&g_{0}=1,&Y_{0}^{\mathfrak p}=Y^{\mathfrak p}, \nonumber 
\end{align}
so that (\ref{eq:jar7}) still holds.
To distinguish the systems (\ref{eq:glab9a6a}), 
(\ref{eq:glab9a6ay1}), and (\ref{eq:glab9a6ag}),
(\ref{eq:flab9a6ay1a}), for the last two equations, we denote the 
associated expectation operator by 
\index{EP@$E^{P}$}%
$E^{P}$. 
\begin{prop}\label{Prep2}
If $F\in C^{\infty, c}\left(\mathcal{X},\R\right)$, 
then
\begin{equation}\label{eq:gagl2}
\exp\left(-t\mathcal{A}^{X}_{b}\right)F\left(x_{0},Y^{TX}_{0}\right)=
E^{P}\left[\exp\left(\frac{mt}{2b^{2}}-\frac{1}{2b^{2}}\int_{0}^{t}\left\vert  Y^{TX}\right\vert^{2}ds\right)
F\left(x_{t},Y^{TX}_{t}\right)\right].
\end{equation}
\end{prop}
\begin{proof}
Equation (\ref{eq:gagl2}) was established in \cite[eq. 
(13.2.12)]{Bismut08b} using the It\^{o} calculus. 
\end{proof}
\begin{remk}\label{Req}
As  explained in Remark \ref{RGirs}, equations (\ref{eq:tch2}) and
(\ref{eq:gagl2}) can be deduced from each other. The first 
equation will be useful when $b\to 0$, the second equation when $b\to + 
\infty $.
\end{remk}
\subsection{A generalized It\^{o} formula}%
\label{subsec:geito}
From now on, we assume that (\ref{eq:glab9a6a}), 
(\ref{eq:glab9a6ay1}) hold.

In the sequel, 
\index{dwTX@$dw^{TX}$}%
$dw^{TX}$ denotes the Stratonovitch differential of 
$w^{TX}$. As before, 
\index{dwTX@$\delta w^{TX}$}%
$\delta w^{TX}$ is our notation for  its It\^{o} differential. First, we give  formula established in \cite[eq. 
(12.3.19)]{Bismut08b}.
\begin{prop}\label{Pitw}
Let $f: X\to \R$ be a smooth function. Then
\begin{multline}\label{eq:glab40}
f\left(x_{t}\right)+b\n_{Y^{TX}_{t}}f\left(x_{t}\right)=f\left(x_{0}\right)
    +b\n_{Y^{ TX}_{0}}f\left(x_{0}\right) \\
    +\int_{0}^{t}\n^{TX}_{Y^{TX}_{s}}\n_{Y^{TX}_{s}}
    f\left(x_{s}\right)ds+\int_{0}^{t} \n_{\delta w^{TX}_{s}}f\left(x_{s}\right) .
\end{multline}
\end{prop}
\begin{proof}
Equation (\ref{eq:glab40}) follows from an easy application of 
It\^{o}'s formula to the process 
$f\left(x_{t}\right)+b\n_{Y^{TX}_{t}}f\left(x_{t}\right)$.
\end{proof}

Now we establish our generalized It\^{o} formula.
\begin{thm}\label{Tfueq}
    Set
    \begin{equation}\label{eq:glab40a-1}
A^{f}_{t}=\int_{0}^{t}\n^{TX}_{Y^{TX}_{s}}\n_{Y^{TX}_{s}}
    f\left(x_{s}\right)ds+\int_{0}^{t} \n_{\delta w^{TX}_{s}}f\left(x_{s}\right) .
\end{equation}
The following identity holds:
\begin{equation}\label{eq:fex1}
\left( b^{2}\frac{d}{dt}+1 
\right)f\left(x_{t}\right)=\left(b^{2}\frac{d}{dt}+1\right)f\left(x_{t}\right)\vert_{t=0}+ A_{t}^{f} .
\end{equation}
Moreover, 
\begin{equation}\label{eq:fex2}
f\left(x_{t}\right)=f\left(x_{0}\right) +
b^{2}\frac{d}{dt}f\left(x_{t}\right)\vert_{t=0}\left(1-e^{-t/b^{2}}\right)
+\int_{0}^{t}\frac{e^{-\left(t-s\right)/b^{2}}}{b^{2}}A^{f}_{s}ds.
\end{equation}
\end{thm}
\begin{proof}
From (\ref{eq:glab9a6a}), (\ref{eq:glab40}), we get (\ref{eq:fex1}). By 
integrating the differential equation (\ref{eq:fex1}), we get 
(\ref{eq:fex2}). The proof of our theorem is completed. 
\end{proof}
\begin{remk}\label{Rcomp}
The standard form of It\^{o}'s formula (\ref{eq:anfs-2}) should be 
compared with its approximate version (\ref{eq:fex2}). As  was 
shown in \cite[Theorem 12.8.1]{Bismut08b}, as $b\to 0$, the 
probability law of $x_{\cdot}$ in (\ref{eq:glab9a6a}) converges to 
the probability law of $x_{\cdot}$ in (\ref{eq:anst-1}). Equation 
(\ref{eq:glab40}) plays a key role  in the proof 
of this result given in \cite{Bismut08b}.
\end{remk}
\subsection{A uniform estimate on the rate of escape of the process 
$x_{\cdot}$}%
\label{subsec:esc}
We still consider the probability measure $Q$ in subsection 
\ref{subsec:probhea}, and the corresponding stochastic differential 
equation  in (\ref{eq:glab9a6a}). The corresponding process 
$x_{\cdot}$ depends on the parameter $b>0$.  We will improve on the results obtained in \cite[section 
13.2]{Bismut08b}.
\begin{thm}\label{Tesc}
    Given $M>0 $, there exist $C>0,C'>0$ such that for 
 $b>0,\epsilon\le t\le M, r>0$, then
\begin{equation}\label{eq:led19}
Q\left[\sup_{0\le s\le t}d  \left(x_{0},x_{s}\right)\ge r\right]\le 
C\exp\left(-C'r^{2}/t+\frac{1}{2}\left\vert  
Y_{0}^{TX}\right\vert^{2}\right).
\end{equation}
Given $ M>0 $, there exist $c>0,C>0,C'>0$ such that for 
$b>0, 0<t\le M,r>0$, then 
\begin{multline}\label{eq:led19sa1}
\exp\left(-\left\vert  Y_{0}^{TX}\right\vert^{2}/2\right)E^{Q}\left[
1_{\sup_{0\le s\le t}d\left(x_{0},x_{s}\right)\ge 
r}\exp\left(\left\vert  Y^{TX}_{t}\right\vert^{2}/2\right)\right]\\
\le
C\exp\left(-C' \left( r^{2}/t+
\left(1-e^{-ct/b^{2}}\right)\left\vert  
Y^{TX}_{0}\right\vert^{2} \right) \right).
\end{multline}
Given $M>0,\tau>0$, there exist $C>0,C'>0$ such that for $b>0,\tau 
b^{2}\le t\le M,r>0$, then
\begin{multline}\label{eq:led19s1}
\exp\left(-\left\vert  Y_{0}^{TX}\right\vert^{2}/2\right)E^{Q}\left[
1_{\sup_{0\le s\le t}d\left(x_{0},x_{s}\right)\ge 
r}\exp\left(\left\vert  Y^{TX}_{t}\right\vert^{2}/2\right)\right]\\
\le
C\exp\left(-C' \left( r^{2}/t+\left\vert  
Y^{TX}_{0}\right\vert^{2} \right) \right).
\end{multline}
\end{thm}
\begin{proof}
    We temporarily assume that equation (\ref{eq:glab9a6a}) is replaced by
(\ref{eq:glab9a6ag}). 
Let $P$ denote the corresponding probability measure. Using 
Girsanov's transformation as in Remark \ref{RGirs}, we get
\begin{multline}\label{eq:sto2}
Q\left[\sup_{0\le s\le t}d\left(x_{0},x_{s}\right)\ge r\right] \\
=
E^{P}\left[\exp\left(-\frac{1}{b}\int_{0}^{t}\left\langle  Y^{TX}_{s},\delta 
w^{TX}_{s}\right\rangle-\frac{1}{2b^{2}}\int_{0}^{t}\left\vert  
Y^{TX}_{s}\right\vert^{2}ds\right)1_{\sup_{0\le s\le 
t}d\left(x_{0},x_{s}\right)\ge r}\right].
\end{multline}
By It\^{o}'s formula,  as in (\ref{eq:phan-9}), we get
\begin{equation}\label{eq:ito1}
\frac{1}{2} \left( \left\vert  Y^{TX}_{t}\right\vert^{2}-\left\vert  Y^{TX}_{0}\right\vert^{2}
\right) =\frac{mt}{2b^{2}}+\frac{1}{b}\int_{0}^{t}\left\langle  
Y^{TX}_{s},\delta w^{TX}_{s}\right\rangle.
\end{equation}
Using (\ref{eq:ito1}), equation (\ref{eq:sto2}) can be rewritten in 
the form
\begin{multline}\label{eq:ito2}
Q\left[\sup_{0\le s\le t}d\left(x_{0},x_{s}\right)\ge r\right] 
=
E^{P} \Biggl[
\exp\Biggl(-\frac{1}{2b^{2}}\int_{0}^{t}\left\vert  
Y^{TX}_{s}\right\vert^{2}ds \\
+\frac{mt}{2b^{2}}-\frac{1}{2}\left(\left\vert  
Y^{TX}_{t}\right\vert^{2}-\left\vert  
Y^{TX}_{0}\right\vert^{2}\right)\Biggr)1_{\sup_{0\le s\le 
t}d\left(x_{0},x_{s}\right)\ge r}\Biggr].
\end{multline}

By equation (\ref{eq:glab9a6ag}),  we get
\begin{equation}\label{eq:ito3}
\frac{1}{2b^{2}}\int_{0}^{t}\left\vert  Y^{TX}_{s}\right\vert^{2}ds\ge 
\frac{1}{2}\sup_{0<s\le t}\frac{d^{2}\left(x_{0},x_{s}\right)}{s}.
\end{equation}
From (\ref{eq:ito3}), we obtain
\begin{equation}\label{eq:ito4}
\frac{1}{2b^{2}}\int_{0}^{t}\left\vert  Y^{TX}_{s}\right\vert^{2}ds\ge 
\frac{1}{2t}\sup_{0<s\le t}d^{2}\left(x_{0},x_{s}\right).
\end{equation}
By (\ref{eq:ito2}), (\ref{eq:ito4}), we get
\begin{equation}\label{eq:ito5}
Q\left[\sup_{0\le s\le t}d\left(x_{0},x_{s}\right)\ge r\right]
\le \exp\left(\frac{1}{2}\left\vert  Y^{TX}_{0}\right\vert^{2}\right)\exp
\left(-\frac{r^{2}}{2t}+\frac{mt}{2b^{2}}\right).
\end{equation}
By (\ref{eq:ito5}), if $\tau>0$ is fixed, and if $0<t\le 
\tau b^{2}$, equation (\ref{eq:led19}) holds. This is the case if $b>0$ is bounded 
away from $0$, and if $0<t\le M$.

In the sequel, we may as well 
assume that there is $b_{0}>0$ small enough so that $0<b\le b_{0}$ 
and that  $ b^{2}\le t\le M$.  Even in this case, the estimate 
(\ref{eq:ito5})  will still be used.

We proceed as in \cite[proof of Theorem 13.2.2]{Bismut08b}.  We 
fix $\epsilon>0$. We take
$f:X\to \R$ as in (\ref{eq:glab42}).
We define $A_{t}^{f}$ as in (\ref{eq:glab40a-1}).
Since 
$f\left(x_{0}\right)=0,\frac{d}{dt}f\left(x_{t}\right)\vert_{t=0}=0$, by equation (\ref{eq:fex2}) in Theorem 
\ref{Tfueq}, we get
\begin{equation}\label{eq:add3}
f\left(x_{t}\right)=\frac{1}{b^{2}}\int_{0}^{t}e^{-\left(t-s\right)/b^{2}}A^{f}_{s}ds.
\end{equation}
By (\ref{eq:add3}), we deduce that
\begin{equation}\label{eq:add4}
\sup_{0\le s\le t}f\left(x_{s}\right)\le \sup_{0\le s\le t}A^{f}_{s}.
\end{equation}

For $r\ge \epsilon$,  $\sup_{0\le s\le t}d\left(x_{0},x_{s}\right)\ge r$ if 
and only if $\sup_{0\le s\le t}f\left(x_{s}\right)\ge r$. By 
(\ref{eq:add4}), this is the case only if $\sup_{0\le s\le t}A^{f}_{s}\ge 
r$.
If $\sup_{0\le s\le t}A^{f}_{s}\ge r$,  at least one of the 
$\sup$ of  the two terms in 
the right-hand side of (\ref{eq:glab40a-1}) is larger than 
$r/2$.  

By \cite[eq. (13.1.17)]{Bismut08b}, there is $C>0$ such that
\begin{equation}
    \n^{TX}_{\cdot}\n_{\cdot}f\le \frac{C}{2}.
    \label{eq:glab42m-1}
\end{equation}
By (\ref{eq:glab42m-1}), for $0\le s\le t$, we get
\begin{equation}
    \int_{0}^{s}\n^{TX}_{Y^{TX}_{u}}\n_{Y^{TX}_{u}}f\left(x_{u}\right)du\le 
    \frac{C}{2}\int_{0}^{t}\left\vert  Y^{TX}_{u}\right\vert^{2}du.
    \label{eq:glab42m}
\end{equation}
By (\ref{eq:glab42m}), using Chebyshev's inequality,  for any $\alpha>0$, we get
\begin{multline}
    Q\left[ 
    \int_{0}^{t}\n^{TX}_{Y^{TX}_{u}}\n_{Y^{TX}_{u}}f\left(x_{u}\right)du\ge 
    r/2\right]\le
    \exp\left(-\alpha  
    r/2b^{2}\right)\\
    E^{Q}\left[\exp\left(\frac{\alpha C}{2b^{2}}\int_{0}^{t}
    \left\vert  Y^{TX}_{u}\right\vert^{2}du\right)\right].
    \label{eq:glab42n}
\end{multline}
In the sequel,  we choose $\alpha$ given 
by 
\begin{equation}
    \alpha=\frac{1}{C}.
    \label{eq:glab42ha}
\end{equation}

Using equation (\ref{eq:rot7a}) in Proposition \ref{PMehla} with 
$\beta=1$, and $t$ replaced by $t/b^{2}$,   we get
\begin{multline}\label{eq:glab42u1}
 E^{Q}\left[\exp\left(\frac{1}{2b^{2}}\int_{0}^{t}
		   \left\vert  Y^{TX}_{u}\right\vert^{2}du\right)\right]
		   =\exp\left(mt/2b^{2}\right) \\
\left[\frac{1}{
1+\frac{t}{b^{2}}}\right]
	^{m/2}
\exp\left(\frac{1}{2}\frac{t/b^{2}}{1+t/b^{2}}\left\vert  
Y^{TX}_{0}\right\vert^{2}\right).
\end{multline}
By (\ref{eq:glab42n})--(\ref{eq:glab42u1}), we obtain
\begin{equation}\label{eq:beri1}
 Q\left[ 
    \int_{0}^{t}\n^{TX}_{Y^{TX}_{u}}\n_{Y^{TX}_{u}}f\left(x_{u}\right)du\ge 
    r/2\right]\le 
    \exp\left(-\frac{1}{b^{2}}\left(\frac{r}{2C}-\frac{mt}{2}\right)+\frac{1}{2}\left\vert  
    Y_{0}^{TX}\right\vert^{2}\right).
\end{equation}

Since $\n_{\cdot}f$ is uniformly bounded, by \cite[eq. (2.1) in Theorem 
 4.2.1]{StroockVaradhan79}, we get
\begin{equation}\label{eq:led21}
Q\left[\sup_{0\le s\le t}\int_{0}^{s}\n_{\delta 
w_{u}}f\left(x_{u}\right)\ge r/2\right]\le 
C\exp\left(-C'r^{2}/t\right).
\end{equation}

By 
 (\ref{eq:beri1}),  (\ref{eq:led21}),   if 
 $r\ge \epsilon$, then
\begin{multline}\label{eq:flab42u6}
Q\left[\sup_{0\le s\le t}d\left(x_{0},x_{s}\right)\ge r\right]\\
\le \exp\left(-
    \frac{1}{b^{2}}\left(\frac{r}{2C}-\frac{mt}{2}\right)
   +\frac{1}{2}\left\vert  Y^{TX}_{0}\right\vert^{2}\right)+C\exp\left(-C'r^{2}/t\right).
\end{multline}
Combining  (\ref{eq:ito5}) and (\ref{eq:flab42u6}), if $r\ge 
\epsilon$, then
\begin{multline}\label{eq:ito6}
Q\left[\sup_{0\le s\le t}d\left(x_{0},x_{s}\right)\ge r\right]
 \\ \le
\exp\left(-
    \frac{1}{b^{2}}\left(\frac{r}{4C}-\frac{mt}{2}\right)-\frac{r^{2}}{4t}
   +\frac{1}{2}\left\vert  Y^{TX}_{0}\right\vert^{2}\right)+C\exp\left(-C'r^{2}/t\right).
\end{multline}
If $t\le r/2mC$, by (\ref{eq:ito6}), we get
\begin{equation}\label{eq:ito7}
Q\left[\sup_{0\le s\le t}d\left(x_{0},x_{s}\right)\ge r\right]
\le \exp\left(-r^{2}/4t
   +\frac{1}{2}\left\vert  Y^{TX}_{0}\right\vert^{2}\right)
   +C\exp\left(-C'r^{2}/t\right),
\end{equation}
which is compatible with (\ref{eq:led19}). If $t>r/2mC$, then
\begin{equation}\label{eq:ito8}
\frac{r^{2}}{t}\le 4m^{2}C^{2}t.
\end{equation}
If $M>0$ is given, and if $\frac{r}{2mC}< t\le M$, by 
(\ref{eq:ito8}), we get
\begin{equation}\label{eq:ito9}
1\le \exp\left(4m^{2}C^{2}M-r^{2}/t\right).
\end{equation}
By (\ref{eq:ito9}),  equation (\ref{eq:led19}) still holds, which 
completes the proof of (\ref{eq:led19}) when $r\ge \epsilon$.

We will now establish (\ref{eq:led19}) for $r>0$ is small. As before, 
we will assume that $b^{2}\le t\le M$. We take 
$f_{1},\ldots, f_{m}$ as in the proof of Proposition \ref{Pesbr}. We 
still consider equation (\ref{eq:fex2}) with $f=\pm f_{i}, 1\le i\le 
m$. With respect to what we did before,  we have the extra term 
$b^{2}\frac{d}{dt}f\left(x_{t}\right)\vert_{t=0}\left(1-e^{-t/b^{2}}\right)$. Note that
\begin{equation}\label{eq:extr1-a}
\left\vert 
b^{2}\frac{d}{dt}f\left(x_{t}\right)\vert_{t=0}\left(1-e^{-t/b^{2}}\right) \right\vert\le C''b
\left(1-e^{-t/b^{2}}\right)\left\vert  Y^{TX}_{0}\right\vert\le 
C''b\left\vert  Y^{TX}_{0}\right\vert.
\end{equation}
By proceeding as before, for $r>0$ small, we get the obvious analogue 
of (\ref{eq:flab42u6}), where $r$ is replaced by $\left(r-C''b
\left\vert  Y^{TX}_{0}\right\vert\right) 
_{+}$. More precisely, for $r>0$ small, we obtain
\begin{multline}\label{eq:flab42u6bis}
Q\left[\sup_{0\le s\le t}d\left(x_{0},x_{s}\right)\ge r\right]
\le \exp\left(-
    \frac{1}{b^{2}}\left(\frac{\left(r-C''b
\left\vert  Y^{TX}_{0}\right\vert\right) 
_{+}}{2C}-\frac{mt}{2}\right)
   +\frac{1}{2}\left\vert  Y^{TX}_{0}\right\vert^{2}\right)\\
   +C\exp\left(-C'\left(r-C''b
\left\vert  Y^{TX}_{0}\right\vert\right) 
_{+}^{2}/t\right).
\end{multline}
By combining (\ref{eq:ito5}) and (\ref{eq:flab42u6bis}), we get
\begin{multline}\label{eq:ito6bis}
Q\left[\sup_{0\le s\le t}d\left(x_{0},x_{s}\right)\ge r\right]
 \le
\exp\Biggl(-
    \frac{1}{b^{2}}\left(\frac{\left(r-C''b\left\vert  Y^{TX}_{0}\right\vert\right)_{+}}{4C}
    -\frac{mt}{2}\right) \\
    -\frac{r^{2}}{4t}
   +\frac{1}{2}\left\vert  Y^{TX}_{0}\right\vert^{2}\Biggr)
   +C\exp\left(-C'\left(r-C''b\left\vert  Y^{TX}_{0}\right\vert\right)^{2}_{+}/t\right).
\end{multline}

We claim that  if $t\ge  b^{2}$, then
\begin{equation}\label{eq:thom1}
C'\left(r-C''b\left\vert  Y^{TX}_{0}\right\vert\right)^{2}_{+}/t+\frac{1}{2}
\left\vert  Y^{TX}_{0}\right\vert^{2}\ge C'''r^{2}/t.
\end{equation}
Indeed this is the case if $C''b\left\vert  Y^{TX}_{0}\right\vert\le
r/2$. If $C''b\left\vert  Y^{TX}_{0}\right\vert>r/2$, then
\begin{equation}\label{eq:thom2}
\frac{1}{2}\left\vert  Y^{TX}_{0}\right\vert^{2}\ge 
r^{2}/8C^{ \prime \prime 2}b^{2}\ge r^{2}/8C^{\prime \prime 2}t.
\end{equation}
By (\ref{eq:thom1}), we conclude that the second term in 
(\ref{eq:flab42u6bis}) can be dominated by the right-hand side of 
(\ref{eq:led19}). 
If $t\le\left(r-C''b\left\vert  
Y^{TX}_{0}\right\vert\right)_{+}/2mC$, this also the case for the 
first term. 

If  $t>\left(r-C''b\left\vert  
Y^{TX}_{0}\right\vert\right)_{+}/2mC$, then
\begin{equation}\label{eq:thom3}
\left(r-C''b\left\vert  
Y^{TX}_{0}\right\vert\right)_{+}^{2}/t\le 4m^{2}C^{2}t.
\end{equation}
If $M>0$ is given and
$\left(r-C''b\left\vert  
Y^{TX}_{0}\right\vert\right)_{+}/2mC<t\le M$, by (\ref{eq:thom3}), 
 we 
obtain
\begin{equation}\label{eq:thom4}
1\le \exp\left(\left(4m^{2}C^{2}M-\left(r-C''b\left\vert  
Y^{TX}_{0}\right\vert\right)^{2}_{+}/t\right)\right).
\end{equation}
By (\ref{eq:thom1}), (\ref{eq:thom4}),  we obtain
\begin{equation}\label{eq:thom5}
1\le \exp\left(4m^{2}C^{2}M-C'''r^{2}/t+\frac{1}{2}\left\vert  
Y^{TX}_{0}\right\vert^{2}\right), 
\end{equation}
 so that (\ref{eq:led19}) also holds in this 
case. This concludes the proof of (\ref{eq:led19}).

Now we establish (\ref{eq:led19sa1}). By  equation
(\ref{eq:glab17a}) with $t$ replaced by $t/b^{2}$ and by 
(\ref{eq:phan-4}), we get
\begin{multline}\label{eq:rito0}
\exp\left(-\left\vert  Y_{0}^{TX}\right\vert^{2}\right)E^{Q}
\left[\exp\left(\left\vert  
Y^{TX}_{t}\right\vert^{2}/2\right)\right]\\
=
\left(\frac{e^{t/b^{2}}}{\cosh\left(t/b^{2}\right)}\right)^{m/2}\exp\left(
-\frac{1}{2}\tanh\left(t/b^{2}\right)\left\vert  
Y_{0}^{TX}\right\vert^{2}\right).
\end{multline}
Therefore  when $r^{2}/t$ remains 
uniformly bounded, (\ref{eq:led19sa1}) holds. In the sequel, we may 
take $r^{2}/t$ as large as needed.

The same argument as in (\ref{eq:ito2})  shows that
\begin{multline}\label{eq:ito2bi}
\exp\left(-\left\vert  Y^{TX}_{0}\right\vert^{2}/2\right)E^{Q}\left[
1_{\sup_{0\le s\le t}d\left(x_{0},x_{s}\right)\ge r}\exp\left(
\left\vert  Y^{TX}_{t}\right\vert^{2}/2\right)\right] \\
=
E^{P}\left[\exp\left(-\frac{1}{2b^{2}}\int_{0}^{t}\left\vert  
Y^{TX}_{s}\right\vert^{2}ds+\frac{mt}{2b^{2}}\right)1_{\sup_{0\le s\le 
t}d\left(x_{0},x_{s}\right)\ge r}\right].
\end{multline}
By (\ref{eq:ito4}), (\ref{eq:ito2bi}), we get
\begin{multline}\label{eq:ito2bi1}
\exp\left(-\left\vert  Y^{TX}_{0}\right\vert^{2}/2\right)E^{Q}\left[
1_{\sup_{0\le s\le t}d\left(x_{0},x_{s}\right)\ge r}\exp\left(
\left\vert  Y^{TX}_{t}\right\vert^{2}/2\right)\right] \\
\le \exp\left(-r^{2}/4t+\frac{mt}{2b^{2}}\right)
E^{P}\left[\exp\left(-\frac{1}{4b^{2}}\int_{0}^{t}\left\vert  
Y^{TX}_{s}\right\vert^{2}ds\right)\right].
\end{multline}

By equation (\ref{eq:phan-11b}) in Proposition \ref{Pidpha}, we get 
\begin{multline}\label{eq:ito2bi2}
E^{P}\left[\exp\left(-\frac{1}{4b^{2}}\int_{0}^{t}\left\vert  
Y^{TX}_{s}\right\vert^{2}ds\right)\right]
=\left(\cosh\left(t/\sqrt{2}b^{2}\right)\right)^{-m/2} \\
\exp\left(-\frac{1}{2\sqrt{2}}
\tanh\left(t/\sqrt{2}b^{2} \right) \left\vert  
Y_{0}^{TX}\right\vert^{2}\right).
\end{multline}
By (\ref{eq:ito2bi1}), (\ref{eq:ito2bi2}), we get
\begin{multline}\label{eq:rito-1}
\exp\left(-\left\vert  Y^{TX}_{0}\right\vert^{2}/2\right)E^{Q}\left[
1_{\sup_{0\le s\le t}d\left(x_{0},x_{s}\right)\ge r}\exp\left(
\left\vert  Y^{TX}_{t}\right\vert^{2}/2\right)\right] \\
\le\exp\left(-r^{2}/4t+mt/2b^{2}-\frac{1}{2\sqrt{2}}\tanh\left(t/\sqrt{2}b^{2}\right)\left\vert  
Y_{0}^{TX}\right\vert^{2}\right).
\end{multline}
By (\ref{eq:rito-1}), if $\tau>0$ is fixed, if 
$0<t\le\tau b^{2}$, equation (\ref{eq:rito-1})  is compatible with 
 (\ref{eq:led19sa1}). 
In the sequel, we may as well assume that $b^{2}\le t$ and that 
$r^{2}/t$ is large. 

In the sequel,  H\"older norms are calculated with respect to the 
probability
measure Q. By \cite[eq. (10.7.1)]{Bismut08b}, for $1\le \theta<2$, we get
\begin{equation}\label{eq:rito2}
\left\Vert  \exp\left(\left\vert  
Y^{TX}_{t}\right\vert^{2}/2\right)\right\Vert_{\theta}
\le \frac{1}{\left(1-\theta/2\right)^{m/2\theta}}\exp
\left(\frac{1}{2}e^{-2t/b^{2}}\frac{\left\vert  Y_{0}^{TX}\right\vert^{2}}
{1-\theta e^{-t/b^{2}}\sinh\left(t/b^{2}\right)} \right) .
\end{equation}
By (\ref{eq:rito2}), we deduce that
\begin{multline}\label{eq:rito2x1}
\exp\left(-\frac{1}{2}\left\vert  Y^{TX}_{0}\right\vert^{2}\right)
\left\Vert  \exp\left(\left\vert  
Y^{TX}_{t}\right\vert^{2}/2\right)\right\Vert_{\theta} \\
\le\frac{1}{ \left( 1-\theta/2\right)^{m/2\theta}}
\exp\left(-\frac{\left( 1-e^{-2t/b^{2}}\right)\left(1-\theta/2\right)}
{1-\theta e^{-t/b^{2}}
\sinh\left(t/b^{2}\right)}\frac{\left\vert  
Y^{TX}_{0}\right\vert^{2}}{2}\right).
\end{multline}
By (\ref{eq:rito2x1}), we  get
\begin{multline}\label{eq:rito2x2}
\exp\left(-\frac{1}{2}\left\vert  Y^{TX}_{0}\right\vert^{2}\right)
\left\Vert  \exp\left(\left\vert  
Y^{TX}_{t}\right\vert^{2}/2\right)\right\Vert_{\theta} \\
\le 
\frac{1}{\left(1-\theta/2\right)^{m/2\theta}}\exp\left(-\left(1-e^{-2t/b^{2}}\right)
\left(1-\theta/2\right)\left\vert  Y^{TX}_{0}
\right\vert^{2}/2\right).
\end{multline}

 We take $r\ge \epsilon$. Let $f$ be as in (\ref{eq:glab42}). We 
 still use (\ref{eq:add4}) and the arguments that follow.
 Using H\"{o}lder's inequality, we get
\begin{multline}\label{eq:rito4}
E^{Q}\left[1_{\sup_{0\le s\le t}\int_{0}^{s}\n^{TX}_{Y^{TX}}\n_{Y^{TX}} 
f\left(x_{u}\right)du\ge r/2}\exp\left(\left\vert  
Y^{TX}_{t}\right\vert^{2}/2\right)\right]\\
\le \left\Vert  \exp\left(\left\vert  Y^{TX}_{t}\right\vert^{2}/2\right)\right\Vert_{\theta}
\left(Q\left[\sup_{0\le s\le t}\int_{0}^{s}\n^{TX}_{Y^{TX}}\n_{Y^{TX}} 
f\left(x_{u}\right)du\ge 
r/2\right]\right)^{\left(\theta-1\right)/\theta}.
\end{multline}
Let $\beta, 0<\beta<1$. We will use (\ref{eq:glab42n}) with  $\alpha=\beta^{2}/C$. By 
equation (\ref{eq:rot7a}) in Proposition \ref{PMehla} and by 
(\ref{eq:rito2x2}), (\ref{eq:rito4}),
we obtain
\begin{multline}\label{eq:rito5}
\exp\left(-\left\vert  Y_{0}^{TX}\right\vert^{2}/2\right)E^{Q}\left[1_{\sup_{0\le s\le t}\int_{0}^{s}\n^{TX}_{Y^{TX}}\n_{Y^{TX}} 
f\left(x_{u}\right)du\ge r/2}\exp\left(\left\vert  
Y^{TX}_{t}\right\vert^{2}/2\right)\right]\\
\le 
\frac{1}{\left(1-\theta/2\right)^{m/2}}\exp\left(-\left(1-e^{-2t/b^{2}}\right)\left(1-\theta/2\right)\left\vert  Y_{0}^{TX}
\right\vert^{2}/2\right)\\
\Biggl[
\exp\left(-\beta^{2}r/2Cb^{2}+m\left(1-\rho_{\beta}\right)t/2b^{2}\right)\\
\exp\left(\beta^{2}\frac{\tanh\left(\rho_{\beta}t/b^{2}\right)}{
\rho_{\beta}+\tanh\left(\rho_{\beta}t/b^{2}\right)}\left\vert  
Y_{0}^{TX}\right\vert^{2}/2\right)\Biggr]^{\left(\theta-1\right)/\theta}.
\end{multline}

Observe that since $\rho_{\beta}<1$, we have
\begin{multline}\label{eq:rito5y1}
\left(1-e^{-2t/b^{2}}\right)\left(1-\theta/2\right)-\frac{\theta-1}{\theta}
\beta^{2}\frac{\tanh\left(\rho_{\beta}t/b^{2}\right)}{\rho_{\beta}+
\tanh\left(\rho_{\beta}t/b^{2}\right)} \\
\ge
\left( 1-\theta/2-\frac{\theta-1}{\theta}
\frac{\beta^{2}}{\rho_{\beta}} \right) \left(1-e^{-2\rho_{\beta} 
t/b^{2}}\right).
\end{multline}
Given $\beta$, by taking $\theta$ close enough to $1$ in  
(\ref{eq:rito5y1}), there is $c>0$ such that
\begin{equation}\label{eq:rito5y2}
\left(1-e^{-2t/b^{2}}\right)\left(1-\theta/2\right)-\frac{\theta-1}{\theta}
\beta^{2}\frac{\tanh\left(\rho_{\beta}t/b^{2}\right)}{\rho_{\beta}+
\tanh\left(\rho_{\beta}t/b^{2}\right)} \\
\ge c\left(1-e^{-2\rho_{\beta} t/b^{2}}\right).
\end{equation}

By (\ref{eq:rito5}), (\ref{eq:rito5y2}), we get
\begin{multline}\label{eq:rito8}
\exp\left(-\left\vert  Y_{0}^{TX}\right\vert^{2}/2\right)E^{Q}\left[1_{\sup_{0\le s\le t}\int_{0}^{s}\n^{TX}_{Y^{TX}}\n_{Y^{TX}} 
f\left(x_{u}\right)du\ge r/2}\exp\left(\left\vert  
Y^{TX}_{t}\right\vert^{2}/2\right)\right]\\
\le C'\Biggl[
\exp\left(-\beta^{2}r/2Cb^{2}+m\left(1-\rho_{\beta}\right)t/2b^{2}\right)
\Biggr]^{\left(\theta-1\right)/\theta}\\
\exp\left(-c\left(1-e^{-2\rho
_{\beta}t/b^{2}}\right)\left\vert  Y^{TX}_{0}\right\vert^{2}/2\right).
\end{multline}
Observe that
\begin{equation}\label{eq:rito8a1}
\beta^{2}r/4C-m\left(1-\rho_{\beta}\right)t/2=\sqrt{t}\left( 
\beta^{2}\frac{r}{4C\sqrt{t}}-m\left(1-\rho_{\beta}\right)\sqrt{t}/2 \right) .
\end{equation}
Given $M>0$, from (\ref{eq:rito8a1}), we deduce that if $r^{2}/t$ is 
large enough and $0<t\le M$, (\ref{eq:rito8a1}) is nonnegative. By 
(\ref{eq:rito8}),  we get
\begin{multline}\label{eq:rito8a2}
\exp\left(-\left\vert  Y_{0}^{TX}\right\vert^{2}/2\right)E^{Q}\left[1_{\sup_{0\le s\le t}\int_{0}^{s}\n^{TX}_{Y^{TX}}\n_{Y^{TX}} 
f\left(x_{u}\right)du\ge r/2}\exp\left(\left\vert  
Y^{TX}_{t}\right\vert^{2}/2\right)\right]\\
\le C\exp\left(-C'\left(r/b^{2}+\left(1-e^{-2\rho_{\beta}t/b^{2}}\right)\left\vert  
Y_{0}^{TX}\right\vert^{2}\right)\right).
\end{multline}

By H\"{o}lder's inequality, for 
$\theta\in \left]1,+ \infty \right[$,  we get
\begin{multline}\label{eq:rito1}
E^{Q}\left[
1_{\sup_{0\le s\le t}\int_{0}^{s}\n_{\delta w_{u}}f\left(x_{u}\right)\ge 
r/2}\exp\left(\left\vert  Y^{TX}_{t}\right\vert^{2}/2\right)\right]\\
\le \left\Vert  \exp\left(\left\vert  Y^{TX}_{t}\right\vert^{2}/2\right)\right\Vert_{\theta}
\left(Q\left[\sup_{0\le s\le t}\int_{0}^{s}\n_{\delta w^{TX}_{u}}f\left(x_{u}\right)\ge 
r\right]\right)^{\left(\theta-1\right)/\theta}.
\end{multline}

By (\ref{eq:led21}), (\ref{eq:rito2x2}), and (\ref{eq:rito1}),   we get
\begin{multline}\label{eq:rito3}
\exp\left(-\left\vert  Y^{TX}_{0}\right\vert^{2}/2\right)E^{Q}\left[
1_{\sup_{0\le s\le t}\int_{0}^{s}\n_{\delta w^{TX}_{u}}f\left(x_{u}\right)\ge 
r/2}\exp\left(\left\vert  Y^{TX}_{t}\right\vert^{2}/2\right)\right]\\
\le C\exp\left(-C'\left( r^{2}/t+\left(1-e^{-2t/b^{2}}\right)\left\vert  
Y^{TX}_{0}\right\vert^{2} \right) \right).
\end{multline}
Equation (\ref{eq:rito3}) is compatible with (\ref{eq:led19sa1}).

By (\ref{eq:rito8a2}), (\ref{eq:rito3}),  we find that under the above 
conditions, 
\begin{multline}\label{eq:rito10}
\exp\left(-\left\vert  Y^{TX}_{0}\right\vert^{2}/2\right)E^{Q}\left[
1_{\sup_{0\le s\le t}d\left(x_{0},x_{s}\right)\ge r}\exp\left(
\left\vert  Y^{TX}_{t}\right\vert^{2}/2\right)\right] \\
\le C\exp\left(- C'\left( 
r/b^{2}
+\left(1-e^{-2\rho_{\beta}t/b^{2}}\right)\left\vert  
Y_{0}^{TX}\right\vert^{2} \right) \right) \\
+C\exp\left(-C'\left( r^{2}/t+\left(1-e^{-2t/b^{2}}\right)\left\vert  
Y^{TX}_{0}\right\vert^{2} \right) \right).
\end{multline}

We will now combine the estimates (\ref{eq:rito-1}) and 
(\ref{eq:rito10}). We proceed as in (\ref{eq:rito8a1}). Observe that 
\begin{equation}\label{eq:anfs1}
C'r-mt/2=\sqrt{t}\left(C'r/\sqrt{t}-m\sqrt{t}/2\right).
\end{equation}
Given $M>0$, we deduce that if $0<t\le M$, if $r/\sqrt{t}$ is large enough, 
(\ref{eq:anfs1}) is positive. By (\ref{eq:rito-1}), 
(\ref{eq:rito10}), for $r/\sqrt{t}$ large enough, 
we get the full estimate (\ref{eq:led19sa1}). This completes the 
proof of (\ref{eq:led19sa1}) when $r\ge \epsilon$. 

 Let us now extend (\ref{eq:rito10}) with $r>0$ 
small, while $r^{2}/t $ can be taken to be large and $t\ge b^{2}$. We 
take again
$f_{1},\ldots, f_{m}$ as in the proof of Proposition \ref{Pesbr}. We 
still consider equation (\ref{eq:fex2}) with $f=\pm f_{i}, 1\le i\le 
m$. With respect to what we did before,  we have the extra term 
$b^{2}\frac{d}{dt}f\left(x_{t}\right)\vert_{t=0}\left(1-e^{-t/b^{2}}\right)$. We still use the estimate
(\ref{eq:extr1-a}).
By proceeding the way we did before, the 
estimate (\ref{eq:rito10}) is still valid 
when replacing $r$ by $\left(r-Cb\left\vert  Y^{TX}_{0}\right\vert
\right)_{+}$ in the right-hand side. In 
particular, for $b>0,0<t\le M,r>0$, we have
\begin{multline}\label{eq:extr3}
\exp\left(-\left\vert  Y^{TX}_{0}\right\vert^{2}/2\right)E^{Q}\left[
1_{\sup_{0\le s\le t}d\left(x_{0},x_{s}\right)\ge r}\exp\left(
\left\vert  Y^{TX}_{t}\right\vert^{2}/2\right)\right] \\
\le \exp\left(- C'\left( 
\left(r-Cb\left\vert  
Y^{TX}_{0}\right\vert\right)_{+}/b^{2}
+\left\vert  Y_{0}^{TX}\right\vert^{2} \right) \right)  \\
+C\exp\left(-C'\left( 
\left(r-Cb\left\vert  
Y^{TX}_{0}\right\vert\right)_{+}^{2}/t+\left\vert  
Y^{TX}_{0}\right\vert^{2} \right) \right).
\end{multline}

We claim that if $t\ge b^{2}$, 
\begin{equation}\label{eq:extr4}
\left(r-Cb\left\vert  
Y^{TX}_{0}\right\vert\right)_{+}^{2}/t+\left\vert  
Y^{TX}_{0}\right\vert^{2} \ge C''\left(r^{2}/t+\left\vert  
Y_{0}^{TX}\right\vert^{2}\right).
\end{equation}
Indeed (\ref{eq:extr4}) holds if $Cb\left\vert  
Y^{TX}_{0}\right\vert\le r/2$. If $Cb\left\vert  
Y^{TX}_{0}\right\vert>r/2$, since $t\ge b^{2}$, then 
\begin{equation}\label{eq:extr5}
\left\vert  Y^{TX}_{0}\right\vert^{2}\ge r^{2}/4C^{2}b^{2}\ge 
r^{2}/4C^{2}t,
\end{equation}
so that we still get (\ref{eq:extr4}). By (\ref{eq:extr4}), 
the second term in the right-hand side of (\ref{eq:extr3}) is 
compatible with (\ref{eq:led19sa1}).
Consider the first term in the right-hand side of (\ref{eq:extr3}). 
By (\ref{eq:extr5}), only the case $Cb\left\vert  
Y^{TX}_{0}\right\vert\le r/2$  should be considered. In this case, we 
get
\begin{equation}\label{eq:extr6}
\exp\left(- C'\left( 
\left(r-Cb\left\vert  
Y^{TX}_{0}\right\vert\right)_{+}/b^{2}
+\left\vert  Y_{0}^{TX}\right\vert^{2} \right) \right)\le
\exp\left(-C'\left(r/2b^{2}+\left\vert  
Y_{0}^{TX}\right\vert^{2}\right)\right).
\end{equation}
Using (\ref{eq:rito-1}), (\ref{eq:anfs1}),  and (\ref{eq:extr6}), we 
find that the first 
term in the right-hand side of (\ref{eq:extr3}) is also 
compatible with (\ref{eq:led19sa1}), when taking into account the 
fact that $t\ge b^{2}$.

Finally, (\ref{eq:led19s1}) follows from (\ref{eq:led19sa1}).
The proof of our theorem is completed. 
\end{proof}
\subsection{Convergence in probability of $x_{\cdot}$}%
\label{subsec:copro}
Let $\left(\Omega,\mathcal{F},\Pi\right)$ be a probability space, and 
let
$\left(F,\delta\right)$ be a metric space. Let $X_{n}\vert_{n\in\N}$ 
be a family of random variables defined on $\Omega$ with values in 
$F$. If $X_{\infty }$ is another such random variable, we say that as 
$n\to + \infty $, $X_{n}$ converges to $X_{\infty}$ in probability if for any 
$\epsilon>0$, as $n\to + \infty $, 
$\Pi\left[\delta\left(X_{n},X_{\infty}•\right)\ge\epsilon\right]$ tends to $0$.

For greater clarity, we will now write the dependence of 
$x_{\cdot},g_{\cdot}$ 
on $b>0$ in equations (\ref{eq:glab9a6a}), (\ref{eq:glab9a6ay1}) explicitly, i.e., $x_{\cdot},g_{\cdot}$ will now be denoted 
$x_{b,\cdot},g_{b,\cdot}$.
\begin{defin}\label{Dfinb}
Let $x_{0,\cdot}$ be the solution of the stochastic differential 
equation
\begin{align}\label{eq:jar6}
&\dot x_{0,\cdot}=\dot w^{TX},
&x_{0,0}=p1.
\end{align}
Let $g_{0,\cdot}$ denote the horizontal lift of $x_{0,\cdot}$, so that
$g_{0,\cdot}$ is the solution of
\begin{align}\label{eq:jar7a}
&\dot g_{0,\cdot}=\dot w^{\mathfrak p}, &g_{0,0}=1
\end{align}
\end{defin}
Note that
\begin{equation}\label{eq:jar8}
x_{0,\cdot}=pg_{0,\cdot}.
\end{equation}

Observe that $x_{b,\cdot},g_{b,\cdot}$ in (\ref{eq:glab9a6a}) and 
$x_{0,\cdot},g_{0,\cdot}$ 
in (\ref{eq:jar6}) have been constructed on the same probability 
space. 

We will  improve on a result  established 
in \cite[proof of Theorem 12.8.1]{Bismut08b}, where a  weaker result 
of convergence in probability law was established. This stronger form 
of convergence will not be used in the paper.
\begin{thm}\label{Tconci}
For any $M>0$, as $b\to 0$, $\left( x_{b,\cdot},g_{b,\cdot} \right) $ 
converges to 
$ \left( x_{0,\cdot},g_{0,\cdot} \right) $ 
uniformly over $\left[0,M\right]$ in probability.
\end{thm}
\begin{proof}
By \cite[proof of Theorem 12.8.1 and Remark 12.8.2]{Bismut08b},  as $b\to 
0$, the probability law of $x_{b,\cdot}$ on 
$\mathcal{C}\left(\left[0,M\right],X\right)$ converges to the 
probability law of $x_{0,\cdot}$. 

Set
\begin{align}\label{eq:jar5a}
&H_{b,t}=\int_{0}^{t}\frac{Y^{TX}_{s}}{b}ds,
&H_{0,t}=w^{ \mathfrak p}_{t}.
\end{align}
The same argument shows that as $b\to 0$, the probability law of 
$\left(x_{b,\cdot}, H_{b,\cdot}\right) $ on 
$\mathcal{C}\left(\left[0,M\right],X\times \mathfrak p\right)$ converges to the 
probability law of $\left(x_{0,\cdot},H_{0,\cdot}\right)$. It is 
easy to deduce from this that the probability law of 
$\left(x_{b,\cdot},w^{ \mathfrak p}_{\cdot}\right)$ converges to the 
probability law of $\left(x_{0,\cdot},w^{ \mathfrak 
p}_{\cdot}\right)$. An elementary argument detailed in \cite[p. 
48]{Bismut81a} allows us to obtain our theorem for $x_{b,\cdot}$. The 
same argument can be used also for $g_{b,\cdot}$. The proof of our theorem is completed. 
\end{proof}
\subsection{A uniform estimate on the heat kernel for $\mathcal{A}^{X}_{b}$}%
\label{subsec:heatsce}
By the results of \cite[section 11.5]{Bismut08b}, 
	  for  $t>0$, the  operators 
	  $\exp\left(-t\mathcal{A} 
	  ^{X}_{b}\right),\exp\left(-t\mathcal{B}^{X}_{b}\right)$ are well-defined, 
	  and
	there are associated smooth kernels
    \index{rXbt@$r_{b,t}^{X}\left(\left(x,Y^{TX}\right),\left(x',Y^{TX \prime }\right)\right)$}%
    \index{sXbt@$s^{X}_{b,t}\left(\left(x,Y^{TX}\right),\left(x',Y^{TX \prime}\right) \right) $}%
    $$r_{b,t}^{X}\left(\left(x,Y^{TX}\right),\left(x',Y^{TX \prime}\right)\right),s_{b,t}^{X}\left(\left(x,Y^{TX}\right),\left(x',Y^{TX \prime }
    \right)\right)$$
     with respect to the volume $dx'dY^{TX \prime}$. By Proposition 
     \ref{Prep1}, we get
     \begin{align}\label{eq:gougl1}
&r^{X}_{b,t}\ge 0,&s_{\bt}^{X}\ge 0.
\end{align}
Recall that the smooth kernel 
\index{pt@$p_{t}\left(x,x'\right)$}%
$p_{t}\left(x,x'\right)$ was defined in subsection 
\ref{subsec:bro}.
\begin{defin}\label{Dheatlimbis}	 
  For $t>0,\left(x,Y^{TX}\right),\left(x',Y^{TX \prime}\right)\in 
	 \mathcal{X}$, put
\index{r0t@$r_{0,t}^{X}\left(\left(x,Y^{TX}\right),\left(x',Y^{TX \prime}\right)\right)$}%
	  \begin{equation}
	      r_{0,t}^{X}\left(\left(x,Y^{TX}\right),\left(x',Y^{TX\prime }\right)\right)\\
		  =
		  p_{t}\left(x,x'\right)\pi^{-m/2}\exp\left(-\frac{1}{2}\left( \left\vert  
		  Y^{TX}\right\vert^{2}+\left\vert  Y^{TX \prime}\right\vert^{2} \right) \right).
	      \label{eq:paris-7bis}
	  \end{equation}
\end{defin}
We explain results  established in \cite[Theorems 12.8.1 and 
13.2.4]{Bismut08b}, where we  extend the range of parameters.
\begin{thm}\label{Testardbis}
Given $\tau>0,M>0$,  there exist $k\in\N,C>0,C'>0$ such
that if $b>0,\tau b^{2}\le t\le M$,
$\left(x,Y^{TX}\right),\left(x',Y^{TX \prime }\right)\in \mathcal{X}$,
\begin{equation}
     r_{b,t}^{X}\left(\left(x,Y^{TX}\right),\left(x',Y^{TX\prime}\right)\right)\le \frac{C}{t^{k}}
    \exp\left(-C'\left(
    \frac{d^{2}\left(x,x'\right)}{t}+\left\vert  Y^{TX}\right\vert^{2}+\left\vert  
    Y^{TX\prime}\right\vert^{2}\right)\right).
    \label{EQ:BERN0ardbis}
\end{equation}
Given $t>0$, as $b\to 0$, we have the pointwise convergence
\begin{equation}
    r_{b,t}^{X}\left(\left(x,Y^{TX}\right),\left(x',Y^{TX\prime}\right)\right)\to
    r_{0,t}^{X}\left(\left(x,Y^{TX}\right),\left(x',Y^{TX\prime}\right)\right).
    \label{eq:sumex32bisardbis}
\end{equation}
\end{thm}
\begin{proof}
Equation (\ref{eq:sumex32bisardbis}) was proved in \cite[Theorem 
12.8.1]{Bismut08b}. Equation (\ref{EQ:BERN0ardbis}) was established in \cite[Theorem 
13.2.4]{Bismut08b} when $0<b\le b_{0}, \epsilon\le t\le M$.
Also observe that if $X$ is a 
Euclidean vector space $E$, equation (\ref{EQ:BERN0ardbis}) was 
already established in (\ref{eq:phan4}).

We will explain how to extend the arguments of 
\cite{Bismut08b} to the general case. 
We consider the Euclidean vector space  $E= \mathfrak p$. For  $t>0$, let $L_{2,t}$ be the vector space of $L_{2}$ functions 
from $\left[0,t\right]$ with values in $\mathfrak p$ with respect to 
the Lebesgue measure. For $t=1$, we write $L_{2}$ instead of 
$L_{2,1}$.

We proceed as in \cite[section 12.5]{Bismut08b}. For $b>0,t>0,e\in \mathfrak p,f\in 
\mathfrak p ,v\in L_{2,t}$, as in \cite[eq. (12.5.2)]{Bismut08b}, consider the differential equation on the interval 
$\left[0,t\right]$ on the function $J^{TX}_{\cdot}: 
\left[0,t\right]\to \mathfrak p$,
\begin{align}\label{eq:crip2a}
&b^{2}\ddot J^{TX}+\dot J^{TX}=v,
&J^{TX}_{0}=0,\qquad\dot J^{TX}_{0}=0,\\
&J^{TX}_{t}=e,&\dot J^{TX}_{t}=f/b. \nonumber 
\end{align}
For $s\in \left[0,1\right]$, set
\begin{align}\label{eq:stan9a}
&K^{TX}_{s}=J^{TX}_{st}, &w_{s}=tv_{st}.
\end{align}
If $\beta=b/\sqrt{t}$, we have
\begin{align}\label{eq:stan10a}
&\beta^{2}\ddot K^{TX}+\dot K^{TX}=w,
&K^{TX}_{0}=0,\qquad\dot K^{TX}_{0}=0,\\
&K^{TX}_{1}=e,& \dot K^{TX}_{1}=\sqrt{t}f/\beta. \nonumber 
\end{align}

  We take in (\ref{eq:stan10a}) the $w$ 
minimizing $\left\vert  w\right\vert^{2}_{L_{2}}$ so that the given 
boundary conditions are verified.  By the considerations after 
\cite[eq. (10.3.47)]{Bismut08b}, as long as $e,f$ remain uniformly 
bounded, and $\beta>0,t>0$ remain uniformly bounded, $\left\vert  
w\right\vert^{2}_{L_{2}}$ remains uniformly bounded, and the functions 
$K^{TX}_{\cdot},\beta\dot K^{TX}_{\cdot}$ also remain uniformly bounded.

We can now transfer these results to corresponding results on 
equation (\ref{eq:crip2a}).  Given $b,e,f$, let $v\in 
L_{2,t}$ minimize $\left\vert  v\right\vert^{2}_{L_{2,t}}$ so that 
the boundary conditions in (\ref{eq:crip2a}) are verified. By 
(\ref{eq:stan9a}), we get
\begin{equation}\label{eq:stan11a}
\left\vert  v\right\vert^{2}_{L_{2,t}}=\frac{\left\vert  
w\right\vert^{2}_{L_{2}}}{t}.
\end{equation}
Under the same conditions as before, $t\left\vert  v\right\vert^{2}_{L_{2,t}}$ remains 
uniformly bounded, and  the functions $J^{TX}_{\cdot},b\sqrt{t}\dot 
J^{TX}_{\cdot}$ remain 
uniformly bounded.

In general, to estimate  the kernels $r_{b,t}^{X},s_{b,t}^{X}$, we proceed exactly as in 
   \cite[chapter 12]{Bismut08b}, where $t>0$ was fixed.  The main point is to 
   keep track of the dependence on $t$ of the estimates in 
   \cite{Bismut08b}. Let us first consider \cite[Theorems 12.6.1]{Bismut08b}. Then the above uniform estimates on 
   $t\left\vert  v\right\vert^{2}_{t},J_{\cdot},b\sqrt{t}\dot 
   J_{\cdot}$ show that in the estimate \cite[eq. 
   (12.6.2)]{Bismut08b} in   \cite[Theorem 
   12.6.1]{Bismut08b}, we have an extra constant $1/\sqrt{t}$ in the 
   right-hand side. Namely, given $c'>0,M>0,p\in \left[1,+ \infty 
   \right[$, there is $C>0$ such that for $0<b\le M,0<t\le M,\beta\le 
   c'$, \footnote{We refer to \cite{Bismut08b} for more details on $M_{t}$. Let us just 
mention that an integration by parts formula coming from the 
Malliavin calculus was established in \cite[Theorem 12.5.1]{Bismut08b}, in 
which the random tensor $M_{t}$ appears. Estimating $\left\Vert  
M_{t}\right\Vert_{p}$ is used in \cite{Bismut08b} to control the 
kernel $r_{b,t}$ for bounded $b$ and fixed $t>0$.}
   \begin{equation}\label{eq:stan11b}
\left\Vert  M_{t}\right\Vert_{p}\le 
\frac{C}{\sqrt{t}}\left(1+\left\vert  Y^{\mathfrak 
p}\right\vert^{2}\right).
\end{equation}

We can still use the formula of integration by parts in \cite[Theorem 
12.7.1]{Bismut08b}. We claim that given $\tau>0,
M>0$, there exist $\theta, 1<\theta<2$ and $C_{\theta}>0$ such that for $0<b\le 
M,\tau b^{2}\le t\le M$, we have the obvious analogue of \cite[eq. 
(12.7.2)]{Bismut08b},
\begin{equation}\label{eq:stan12a}
\left\Vert  \exp\left(\left\vert  
Y_{t}^{TX}\right\vert^{2}/2\right)\right\Vert_{\theta}\le 
C_{\theta}\exp\left(\left\vert  Y^{TX}_{0}\right\vert^{2}/2\theta\right).
\end{equation}
This is just a consequence of (\ref{eq:rito2}),  of the fact that if 
$1<\theta\le\frac{e^{t}}{\cosh\left(t\right)}$, then
\begin{equation}\label{eq:phan-14a}
\frac{e^{-2t}}{1-\theta e^{-t}\sinh\left(t\right)}\le  
\frac{1}{\theta},
\end{equation}
and of the fact that $\frac{e^{t}}{\cosh\left(t\right)}$ is an 
increasing function of $t$.

By proceeding as in \cite[Theorem 12.7.4]{Bismut08b}, we find that given $c'>0, 
M>0,k\in \N$, there exists $k'\in \N$ such that if 
$0<b\le M,0<t\le M,0<\beta\le c'$, the covariant derivatives of 
$t^{k'}r_{b,t}^{X}\left(\left(x,Y^{TX}\right),\left(x',Y^{TX \prime 
}\right)\right)$ of order $\le k $ are uniformly bounded and 
uniformly rapidly decreasing  when $\left\vert  Y\right\vert^{TX}$ or 
$\left\vert  Y^{TX \prime }\right\vert$ tend to $+ \infty $.

Now we use the semigroup identity
\begin{multline}\label{eq:phan5}
r^{X}_{b,t}\left(\left(x,Y^{TX}\right),\left(x',Y^{TX\prime}\right)\right) \\
=\int_{\mathcal{X}}^{}r^{X}_{b,t/2}\left(\left(x,Y^{TX}\right),\left(z,Z^{TX}\right)\right)
r^{X}_{b,t/2}\left(\left(z,Z^{TX}\right),\left(x',Y^{TX\prime}\right)\right)dzdZ^{TX}.
\end{multline}

We proceed as in \cite[section 13.2]{Bismut08b}.    First, using the 
above uniform bounds, we deduce from (\ref{eq:phan5}) that
\begin{equation}\label{eq:phan5a}
r^{X}_{b,t}\left(\left(x,Y^{TX}\right),\left(x',Y^{TX\prime}\right)\right)\le Ct^{-k'}
\int_{X}^{}r^{X}_{b,t/2}\left(\left(x,Y^{TX}\right),\left(z,Z^{TX}\right)\right)dz dZ^{TX}.
\end{equation}
Recall that the heat 
kernel 
\index{hEY@$h_{t} ^{\mathfrak p}\left(Y ,Y'\right)$}%
$h^{\mathfrak p}_{t}\left(Y,Y'\right)$ is given by 
(\ref{eq:glab16a}). As explained in \cite[eq. (13.2.5)]{Bismut08b}, we have the identity 
\begin{equation}\label{eq:phan5b}
\int_{X}^{}r^{X}_{b,t/2}\left(\left(x,Y^{TX}\right),\left(z,Z^{TX}\right)\right)dz dZ^{TX}
=
\int_{\mathfrak p}^{}h_{t/2b^{2}}^{\mathfrak p}\left(Y^{TX},Z^{TX}\right)dZ^{TX}.
\end{equation}
By (\ref{eq:glab17a}), (\ref{eq:phan5b}), we get
\begin{multline}\label{eq:phan5c}
\int_{X}^{}r^{X}_{b,t/2}\left(\left(x,Y^{TX}\right),\left(z,Z^{TX}\right)\right)dz dZ^{TX} \\
=
\left(\frac{e^{t/2b^{2}}}{\cosh\left(t/2b^{2}\right)}\right)^{m/2}\exp\left(
-\frac{\tanh\left(t/2b^{2}\right)}{2}\left\vert  Y^{TX}\right\vert^{2}\right).
\end{multline}
By combining (\ref{eq:phan5a}) and (\ref{eq:phan5c}), we deduce that 
given $\tau>0$, there exist $C>0,C'>0$ such that if $t\ge\tau 
b^{2}$, then
\begin{equation}\label{eq:phan5d}
r^{X}_{b,t}\left(\left(x,Y^{TX}\right),\left(x',Y^{TX\prime}\right)\right)\le Ct^{-k'}
\exp\left(-C'\left\vert  Y^{TX}\right\vert^{2}\right).
\end{equation}

The $L_{2}$ formal adjoint of $\mathcal{A}^{X}_{b}$ is the operator 
$\mathcal{A}^{X}_{b,-}$ deduced from $\mathcal{A}^{X}_{b}$ by making 
the change of variables $Y^{TX}\to -Y^{TX}$. By interchanging the 
roles of $Y^{TX}$ and $Y^{TX \prime}$ in (\ref{eq:phan5d}), under the 
same conditions, we get
\begin{equation}\label{eq:phan5e}
r^{X}_{b,t}\left(\left(x,Y^{TX}\right),\left(x',Y^{TX\prime}\right)\right)\le Ct^{-k'}
\exp\left(-C'\left\vert  Y^{TX \prime }\right\vert^{2}\right).
\end{equation}

In the integral in (\ref{eq:phan5}), either $d\left(x,z\right)\ge 
d\left(x,x'\right)/2$ or $d\left(x',z\right)\ge 
d\left(x,x'\right)/2$. From the  uniform bounds on $r^{X}_{b,t}$ that 
were described before equation (\ref{eq:phan5}), we deduce 
from (\ref{eq:phan5}) that under the conditions of our theorem,
\begin{multline}\label{eq:phan6}
r^{X}_{b,t}\left(\left(x,Y^{TX}\right),\left(x',Y^{TX\prime}\right)\right) \\
\le 
Ct^{-k'}\int_{\substack{\left(z,Z^{TX}\right)\in \mathcal{X}\\d\left(x,z\right)\ge 
d\left(x,x'\right)/2}}^{}r^{X}_{b,t/2}\left(\left(x,Y^{TX}\right),\left(z,Z^{TX}\right)\right)dzdZ^{TX}\\
+Ct^{-k'}
\int_{\substack{\left(z,Z^{TX}\right)\in \mathcal{X}\\d\left(x',z\right)\ge 
d\left(x,x'\right)/2}}^{}r^{X}_{b,t/2}\left(\left(z,Z^{TX}\right),\left(x',Y^{TX\prime}\right)
\right)dzdZ^{TX}.
\end{multline}

By equation (\ref{eq:tch2}) in Proposition \ref{Prep1}, we get
\begin{multline}\label{eq:phan7}
\int_{\substack{\left(z,Z^{TX}\right)\in \mathcal{X}\\d\left(x,z\right)\ge 
d\left(x,x'\right)/2}}^{}r^{X}_{b,t/2}\left(\left(x,Y^{TX}\right),\left(z,Z^{TX}\right)\right)dzdZ^{TX} \\
=
\exp\left(-\left\vert  Y^{TX}\right\vert^{2}/2\right)
E^{Q}\left[1_{d\left(x,x_{t/2}\right)\ge 
d\left(x,x'\right)/2}\exp\left(\left\vert  Y^{TX}_{t/2}\right\vert^{2}/2\right)\right].
\end{multline}
From (\ref{eq:phan7}), we deduce that
\begin{multline}\label{eq:phan8}
\int_{\substack{\left(z,Z^{TX}\right)\in \mathcal{X}\\d\left(x,z\right)\ge 
d\left(x,x'\right)/2}}^{}r^{X}_{b,t/2}\left(\left(x,Y^{TX}\right),\left(z,Z^{TX}\right)\right)dzdZ^{TX} \\
\le
\exp\left(-\left\vert  Y^{TX}\right\vert^{2}/2\right)
E^{Q}\left[1_{\sup_{0\le s\le t/2}d\left(x,x_{s}\right)\ge 
d\left(x,x'\right)/2}\exp\left(\left\vert  Y^{TX}_{t/2}\right\vert^{2}/2\right)\right].
\end{multline}
By equation (\ref{eq:led19s1}) in Theorem \ref{Tesc} and by 
(\ref{eq:phan8}), if $0<\tau b^{2}\le t\le M$, we get
\begin{multline}\label{eq:phan9}
\int_{\substack{z\in \mathcal{X}\\d\left(x,z\right)\ge 
d\left(x,x'\right)/2}}^{}r^{X}_{b,t}\left(\left(x,Y^{TX}\right),\left(z,Z^{TX}\right)\right)dzdZ^{TX} \\
\le C\exp\left(-C'\left(d^{2}\left(x,x'\right)/t+\left\vert  
Y^{TX}\right\vert^{2}\right)\right).
\end{multline}
Still using the fact that the $L_{2}$ adjoint of $\mathcal{A}^{X}$ is 
just $\mathcal{A}^{X}_{b,-}$, we also get the estimate
\begin{multline}\label{eq:phan10}
\int_{\substack{\left(z,Z^{TX}\right)\in \mathcal{X}\\d\left(x',z\right)\ge 
d\left(x,x'\right)/2}}^{}r^{X}_{b,t}\left(\left(z,Z^{TX}\right),\left(x',Y^{TX\prime}\right)
\right)dzdZ^{TX}\\
\le 
C\exp\left(-C'\left(d^{2}\left(x,x'\right)/t+\left\vert  
Y^{TX \prime}\right\vert^{2}\right)\right).
\end{multline}

By (\ref{eq:phan6}), (\ref{eq:phan9}), and (\ref{eq:phan10}), we get
\begin{multline}\label{eq:phan11}
r^{X}_{b,t}\left(\left(x,Y^{TX},\left(x',Y^{TX \prime }\right)\right)\right)
\le Ct^{-k'}\left(\exp\left(-C'd^{2}\left(x,x'\right)/t+\left\vert  
Y^{TX}\right\vert^{2}\right)\right) \\
+Ct^{-k'}\exp\left(-C'\left(d^{2}\left(x,x'\right)/t+\left\vert  
Y^{TX \prime }\right\vert^{2}\right)\right).
\end{multline}

By combining (\ref{eq:phan5d}), (\ref{eq:phan5e}), and 
(\ref{eq:phan11}), we get (\ref{EQ:BERN0ardbis}). The proof of our theorem is completed. 
\end{proof}

%% file: Eta9.tex
\section{Uniform estimates for $b$ small: the 
scalar case}%
\label{sec:unisca}
The purpose of this section is to establish uniform estimates for the 
 heat kernels of a scalar version $\mathfrak A^{X}_{\bt}$ of the operator 
 $\overline{\mathcal{L}}^{X}_{b,\vartheta}$. These estimates correspond to the estimates 
 stated in Theorem \ref{Test} for the nonscalar heat kernels 
 $\overline{q}^{X}_{b,\vartheta,t}$.  They are established 
 for bounded values of $b>0$. The main difficulty in the proof 
 of these estimates is that when 
 $\vartheta$ gets close to $\frac{\pi}{2}$, our scalar operators 
 become singular. Our estimates will be essentially consequences of 
  results of \cite{Bismut08b} that were properly extended in 
 section \ref{sec:extra}.
 
As will be shown in section \ref{sec:fin} and in dramatic contrast 
 with \cite{Bismut08b}, this analogy is no longer valid for the 
 kernels $\overline{q}^{X}_{b,\vartheta,t}$ themselves, because the matrix terms 
 in $\overline{\mathcal{L}}^{X}_{b,\vartheta}$ are different from the 
 ones in 
 $\mathcal{L}^{X}_{b}$.
 
 This section is organized as follows. In subsection 
 \ref{subsec:scaanab}, we introduce a scalar analogue 
 $\mathcal{A}^{X}_{\bt}$ of 
 $\overline{\mathcal{L}}^{X}_{\bt}$ over $\mathcal{X}$.
 
 In subsection \ref{subsec:unisca}, we show that the estimates on the 
 heat kernel $r^{X}_{\bt,t}$ for $\mathcal{A}^{X}_{\bt}$ are 
 trivial consequences of the results of section
 \ref{sec:extra}.
 
 Finally, in  subsection \ref{subsec:unibis}, we introduce the scalar operator 
 $\mathfrak A^{X}_{\bt}$ on $\widehat{\mathcal{X}}$, and we establish 
 the required estimates on its heat kernel.
 
 We use the same notation as in section \ref{sec:extra}.
\subsection{The scalar analogues of the operator 
  $\mathcal{L}^{X}_{\bt}$ on $\mathcal{X}$}%
  \label{subsec:scaanab}
    By analogy with  equation (\ref{eq:co19x-1}) for 
    $\mathcal{L}^{X}_{\bt}$ and with equation (\ref{eq:glab5}) for 
	$\mathcal{A}^{X}_{b},\mathcal{B}^{X}_{b}$, given 
    $b>0,\vartheta\in\left[0,\frac{\pi}{2}\right[$, 
   let $\mathcal{A}^{X}_{\bt},\mathcal{B}^{X}_{\bt}$ be the scalar 
   differential operators on $\mathcal{X}$,
    \index{AXbt@$\mathcal{A}^{X}_{b,\vartheta}$}%
    \index{BXbt@$\mathcal{B}^{X}_{b,\vartheta}$}%
    \begin{align}\label{eq:scal1}
&\mathcal{A}^{X}_{b,\vartheta}=\frac{1}{2b^{2}}\left(-\Delta^{TX}+\left\vert  Y^{TX}\right\vert^{2}-m\right)-
\frac{\cos\left(\vartheta\right)}{b}\n_{Y^{TX}},\\
&\mathcal{B}^{\mathcal{X}}_{b,\vartheta}=\frac{1}{2b^{2}}\left(-\Delta^{TX}+2\n^{V}_{Y^{TX}}\right)
-\frac{\cos\left(\vartheta\right)}{b}\n_{Y^{TX}}. \nonumber 
\end{align}
Then
\begin{equation}\label{eq:scal3}
    \mathcal{B}^{X}_{b,\vartheta}=\exp\left(\left\vert  
Y^{TX}\right\vert^{2}/2\right)\mathcal{A}_{b,\vartheta}^{X}\exp\left(-\left\vert  Y^{TX}\right\vert^{2}/2\right).
\end{equation}

Comparing with (\ref{eq:glab5}),  we get
\begin{align}\label{eq:scal4}
&\mathcal{A}^{X}_{b,\vartheta}=\cos^{2}\left(\vartheta\right)\mathcal{A}^{X}_{\cos\left(\vartheta\right)b},
&\mathcal{B}^{X}_{b,\vartheta}=\cos^{2}\left(\vartheta\right)\mathcal{B}_{\cos\left(\vartheta\right)b}^{X}.
\end{align}
For $\vartheta=0$, our operators coincide with 
$\mathcal{A}^{X}_{b},\mathcal{B}^{X}_{b}$.
\subsection{A uniform estimate on the heat kernel for 
$\mathcal{A}^{X}_{\bt}$}
  \label{subsec:unisca}
  Recall that the  kernel 
  \index{pt@$p_{t}\left(x,x'\right)$}%
  $p_{t}\left(x,x'\right)$  was defined in subsection 
  \ref{subsec:bro},  and  the  kernel
\index{rXbt@$r_{b,t}^{X}$}%
$r^{X}_{b,t}\left(\left(x,Y^{TX}\right),\left(x',Y^{TX \prime 
}\right)\right)$ was defined in subsection \ref{subsec:heatsce}.
  \begin{defin}\label{Dker}
    For $t>0$, let 
    \index{rXbtt@$r^{X}_{b,\vartheta,t}\left(\left(x,Y^{TX}\right),\left(x',Y^{TX \prime}\right)\right)$}%
    $r^{X}_{b,\vartheta,t}\left(\left(x,Y^{TX}\right),\left(x',Y^{TX 
    \prime}\right)\right)$ be the smooth kernel associated with the operator 
    $\exp\left(-t\mathcal{A}^{X}_{b,\vartheta}\right)$ with respect 
    to  the volume $dx'dY^{TX \prime}$. Set
\index{rX0@$r^{X}_{0,\vartheta,t}\left(\left(x,Y^{TX}\right),\left(x',Y^{TX \prime 
}\right)\right)$}%
\begin{multline}\label{eq:scal5x1}
r^{X}_{0,\vartheta,t}\left(\left(x,Y^{TX}\right),\left(x',Y^{TX \prime 
}\right)\right)=p_{\cos^{2}\left(\vartheta\right)t}\left(x,x'\right) 
\\
\pi^{-m/2}\exp\left(
-\frac{1}{2}\left(\left\vert  Y^{TX}\right\vert^{2}+\left\vert  Y^{TX 
\prime }\right\vert^{2}\right)\right).
\end{multline} 
\end{defin}
 By (\ref{eq:paris-7bis}), (\ref{eq:scal4}), and (\ref{eq:scal5x1}), 
 for $b\ge 0, \vartheta\in \left[0,\frac{\pi}{2}\right[,t>0$,  we get
\begin{equation}\label{eq:scal5}
r^{X}_{b,\vartheta,t}=r^{X}_{\cos\left(\vartheta\right)b,\cos^{2}\left(\vartheta\right)t}.
\end{equation}

Now we establish an analogue of Theorem \ref{Test} for the kernel 
$r^{X}_{b,\vartheta,t}$, which is also an extension of Theorem 
\ref{Testardbis} for the kernel $r^{X}_{b,t}$. 
\begin{thm}\label{Tinsi}
Given $0<\epsilon\le M<+ \infty $, there exist $C>0,C'>0,k\in\N$ 
such that for $0<b\le M,0\le\vartheta<\frac{\pi}{2}, \epsilon\le t\le 
M, \left(x,Y^{TX}\right),\left(x',Y^{TX \prime }\right)\in\mathcal{X}$, 
\begin{multline}\label{eq:scal6}
r^{X}_{b,\vartheta,t}\left(\left(x,Y^{TX}\right),\left(x',Y^{TX \prime 
}\right)\right)\\
\le 
\frac{C}{\cos^{k}\left(\vartheta\right)}
\exp\left(-C'\left(
\frac{d^{2}\left(x,x'\right)}{\cos^{2}\left(\vartheta\right)}+\left\vert  Y^{TX}
\right\vert^{2}+\left\vert  Y^{TX \prime 
}\right\vert^{2}\right)\right).
\end{multline}
Given $\vartheta\in\left[0,\frac{\pi}{2}\right[, t>0$, as $b\to 0$, we have the pointwise convergence,
\begin{equation}\label{eq:scal6a1}
r_{b,\vartheta,t}^{X}\left(\left(x,Y^{TX}\right),\left(x',Y^{TX 
\prime }\right)\right)\to 
r^{X}_{0,\vartheta,t}\left(\left(x,Y^{TX}\right),\left(x',Y^{TX 
\prime }\right)\right).
\end{equation}
\end{thm}
\begin{proof}
   Our theorem is a consequence of   Theorem \ref{Testardbis} and of 
   equation (\ref{eq:scal5}).
\end{proof}
\subsection{A uniform estimate on a scalar heat kernel over 
$\widehat{\mathcal{X}}$}%
\label{subsec:unibis}
Recall that
\index{X@$\widehat{\mathcal{X}}$}%
$\widehat{\pi}:\widehat{\mathcal{X}}\to X$ is the total space of 
$TX \oplus N$, and that $Y=Y^{TX} \oplus Y^{N}$ is the canonical 
section of $\widehat{\pi}^{*}\left(TX \oplus N\right)$. Let 
\index{AXbt@$\mathfrak A^{X}_{b,\vartheta}$}%
\index{BXbt@$\mathfrak B^{X}_{\bt}$}%
$\mathfrak A^{X}_{b,\vartheta}, \mathfrak B^{X}_{\bt}$ be 
the scalar  differential operators   on $\widehat{\mathcal{X}}$,
\begin{align}\label{eq:scal11}
&\mathfrak A^{X}_{b,\vartheta}=\frac{1}{2b^{2}}\left(-\Delta^{TX}+\left\vert  
Y^{TX}\right\vert^{2}-m\right)
+
\frac{\cos\left(\vartheta\right)}{2b^{2}}\left(-\Delta^{N}+\left\vert  Y^{N}\right\vert^{2}-n\right) \nonumber \\
&\qquad\qquad -\frac{\cos\left(\vartheta\right)}{b}\n_{Y^{TX}},\\
&\mathfrak B^{X}_{\bt}=\frac{1}{2b^{2}}\left(-\Delta^{TX}+2\n^{V}_{Y^{TX}}\right)
+
\frac{\cos\left(\vartheta\right)}{2b^{2}}\left(-\Delta^{N}+2\n^{V}_{Y^{N}}\right)
-\frac{\cos\left(\vartheta\right)}{b}\n_{Y^{TX}}. \nonumber 
\end{align}

Then
\begin{equation}\label{eq:scal11y1}
\mathfrak B^{X}_{\bt}=\exp\left(\left\vert  
Y\right\vert^{2}/2\right)\mathfrak A^{X}_{\bt}\exp\left(-\left\vert  
Y\right\vert^{2}/2\right).
\end{equation}
Note 
that
 $\mathfrak A^{X}_{b,0}, \mathfrak B^{X}_{b,0}$ are just the operators $\mathfrak 
A^{X}_{b},\mathfrak B^{X}_{b}$ defined in \cite[eqs. (11.6.1) and 
(11.6.2)]{Bismut08b}.

For 
$t>0$, let
\index{rXbtt@ $\mathfrak r^{X}_{b,\vartheta,t}\left(\left(x,Y\right),\left(x',Y'\right)\right)$}%
\index{sXbtt@ $\mathfrak s^{X}_{b,\vartheta,t}\left(\left(x,Y\right),\left(x',Y'\right)\right)$}%
$\mathfrak 
r^{X}_{b,\vartheta,t}\left(\left(x,Y\right),\left(x',Y'\right)\right) 
,
\mathfrak 
s^{X}_{\bt,t}\left(\left(x,Y\right),\left(x',Y'\right)\right)$ be the 
smooth kernels associated with the
operators $\exp\left(-t \mathfrak  
A^{X}_{b,\vartheta}\right),\exp\left(-t \mathfrak  
B^{X}_{b,\vartheta}\right)$. The existence of such kernels follows 
from \cite[section 11.6]{Bismut08b}.

Set
\index{rX0btt@$r^{X}_{0,\vartheta,t}\left(\left(x,Y\right),\left(x',Y'\right)\right)$}%
\begin{multline}\label{eq:stan16x1}
\mathfrak 
r^{X}_{0,\vartheta,t}\left(\left(x,Y\right),\left(x',Y'\right)\right)=
p_{\cos^{2}\left(\vartheta\right)t}\left(x,x'\right)\pi^{-\left(m+n\right)/2}
\exp\left(-\frac{1}{2}\left(\left\vert  
Y\right\vert^{2}+\left\vert  Y'\right\vert^{2}\right)\right).
\end{multline}

Now we establish an analogue of Theorem \ref{Test} for the kernel 
$ \mathfrak r^{X}_{b,\vartheta,t}$. These results extend corresponding results 
established in \cite[Theorems 12.10.2 and 13.3.1]{Bismut08b}.
\begin{thm}\label{Tinsi2}
Given $0<\epsilon\le M<+ \infty $, there exist $C>0,C'>0,k\in\N$ 
such that for $0<b\le M,\vartheta\in\left[0,\frac{\pi}{2}\right[, \epsilon\le t\le 
M$, $\left(x,Y\right),\left(x',Y'\right)\in\widehat{\mathcal{X}}$,
\begin{multline}\label{eq:scal7a-1}
\mathfrak r^{X}_{b,\vartheta,t}\left(\left(x,Y\right),\left(x',Y^{\prime 
}\right)\right)
\le 
\frac{C}{\cos^{k}\left(\vartheta\right)}
\exp \Biggl( -C'\left( 
\frac{d^{2}\left(x,x'\right)}{\cos^{2}\left(\vartheta\right)}+\left\vert  Y^{TX}
\right\vert^{2}+\left\vert  Y^{TX \prime } \right\vert^{2} \right) \\
-\frac{1}{4}\left(1-e^{-\cos\left(\vartheta\right)\epsilon/b^{2}}\right)\left(\left\vert  
Y^{N}\right\vert^{2}+\left\vert  Y^{N \prime 
}\right\vert^{2}\right)\Biggr).
\end{multline}
 There exist $C>0,C'>0,k\in\N$ such that under the above conditions,
\begin{multline}\label{eq:scal7}
\mathfrak r^{X}_{b,\vartheta,t}\left(\left(x,Y\right),\left(x',Y^{\prime 
}\right)\right)
\le 
\frac{C}{\cos^{k}\left(\vartheta\right)}
\exp \Biggl( -C'\Biggl( 
\frac{d^{2}\left(x,x'\right)}{\cos^{2}\left(\vartheta\right)}+\left\vert  Y^{TX}
\right\vert^{2}+\left\vert  Y^{TX \prime } \right\vert^{2}\\
+\cos\left(\vartheta\right)\left(\left\vert  
Y^{N}\right\vert^{2}+\left\vert  Y^{N \prime 
}\right\vert^{2}\right)\Biggr)\Biggr).
\end{multline}
Given $t>0,\vartheta\in\left[0,\frac{\pi}{2}\right[$, as $b\to 0$, we have the 
pointwise convergence,
 \begin{equation}\label{eq:scal8}
\mathfrak r_{b,\vartheta,t}^{X}\left(\left(x,Y\right),\left(x',Y^{
\prime }\right)\right)
\to \mathfrak 
r^{X}_{0,\vartheta,t}\left(\left(x,Y\right),\left(x',Y'\right)\right).
\end{equation}
\end{thm}
\begin{proof}
  Take $x_{0}\in X$. Let  $w_{\cdot}=\left(w^{TX}_{\cdot},w^{N}_{\cdot}\right)$ be a 
Brownian motion in $\left(TX \oplus N\right)_{x_{0}}$, and let $E$ 
be the corresponding expectation operator. Consider the stochastic 
differential equation
\begin{align}\label{eq:scal12}
&\dot x=\frac{\cos\left(\vartheta\right)}{b}Y^{TX},&\dot Y^{TX}= \frac{1}{b}\dot w^{TX}, \qquad 
&\dot Y^{N}=\frac{\cos^{1/2}\left(\vartheta\right)}{b}\dot w^{N},\\
&\left(x,Y\right)_{0}=\left(x_{0},Y_{0}\right). \nonumber 
\end{align}
In (\ref{eq:scal12}), $\dot Y^{TX},\dot Y^{N}$ denote the covariant 
derivatives along the path $x_{\cdot}$ with respect to the 
connections $\n^{TX},\n^{N}$. As explained in \cite[eq. 
(13.2.12)]{Bismut08b}, an application of Itô and Feynman Kac formulas shows that if 
$F\in C^{\infty ,c}\left(X,\R\right)$, then
\begin{multline}\label{eq:scal13}
\exp\left(-t \mathfrak A^{X}_{b,\vartheta}\right)F\left(x_{0},Y_{0}\right)=
\exp\left(\left(m+\cos\left(\vartheta\right)n\right)\frac{t}{2b^{2}}\right)\\
E\left[\exp\left(-\frac{1}{2b^{2}}\int_{0}^{t}\left(\left\vert  
Y^{TX}\right\vert^{2}+\cos\left(\vartheta\right)\left\vert  
Y^{N}\right\vert^{2}\right)ds\right)F\left(x_{t},Y_{t}\right)\right].
\end{multline}

We will now be more precise on the proper interpretation of 
(\ref{eq:scal12}), (\ref{eq:scal13}). Indeed let 
$\overline{Y}^{N}_{\cdot}\in N_{x_{0}}$ be the path corresponding to 
$Y^{N}_{\cdot}$, so that if $\tau^{0}_{s}$ is the parallel transport 
from $N_{x_{0}}$ to $N_{x_{s}}$ along $x_{\cdot}$,   and if $\tau^{s}_{0}$ denotes its 
inverse, we have
\begin{equation}\label{eq:scal14}
Y^{N}_{s}=\tau^{0}_{s}\overline{Y}^{N}_{s}.
\end{equation}
Then we rewrite (\ref{eq:scal13}) in the form
\begin{multline}\label{eq:scal15}
\exp\left(-t\mathfrak 
A^{X}_{b,\vartheta}\right)F\left(x_{0},Y_{0}\right)=
\exp\left(\left(m+\cos\left(\vartheta\right)n\right)\frac{t}{2b^{2}}\right)\\
E\left[\exp\left(-\frac{1}{2b^{2}}\int_{0}^{t}\left(\left\vert  
Y^{TX}\right\vert^{2}+\cos\left(\vartheta\right)\left\vert  
\overline{Y}^{N}\right\vert^{2}\right)ds\right)F\left(x_{t},\tau^{0}_{t}\overline{Y}_{t}\right)\right].
\end{multline}

We will use the fact that $x_{\cdot}$ and $\overline{Y}^{N}_{\cdot}$ are 
independent processes.  Recall that the heat kernel 
\index{hNt@$h^{N}_{t}\left(Y,Y'\right)$}%
$h^{N}_{t}\left(Y,Y'\right)$ is given by (\ref{eq:glab16a}). By (\ref{eq:scal15}), we get
\begin{multline}\label{eq:scal16}
\exp\left(-t\mathfrak 
A^{X}_{b,\vartheta}\right)F\left(x_{0},Y_{0}\right)=
\exp\left(\frac{mt}{2b^{2}}\right)
E\Biggl[\exp\left(-\frac{1}{2b^{2}}\int_{0}^{t}\left\vert  
Y^{TX}\right\vert^{2}ds\right) \\
\int_{N_{x_{0}}}^{}h^{N_{x_{0}}}_{\cos\left(\vartheta\right)t/b^{2}}
\left(Y^{N}_{0},Y^{N}\right)F\left(x_{t},Y^{TX}_{t},
\tau^{0}_{t}Y^{N}\right)dY^{N}\Biggr].
\end{multline}
Since $\tau^{0}_{s}$ is an isometry, we can rewrite (\ref{eq:scal16}) 
in the form
\begin{multline}\label{eq:scal16z1}
\exp\left(-t\mathfrak 
A^{X}_{b,\vartheta}\right)F\left(x_{0},Y_{0}\right)=
\exp\left(\frac{mt}{2b^{2}}\right)
E\Biggl[\exp\left(-\frac{1}{2b^{2}}\int_{0}^{t}\left\vert  
Y^{TX}\right\vert^{2}ds\right) \\
\int_{N_{x_{t}}}^{}h^{N_{x_{t}}}_{\cos\left(\vartheta\right)t/b^{2}}
\left(\tau^{0}_{t}Y^{N}_{0},Y^{N}\right)F\left(x_{t},Y^{TX}_{t},
Y^{N}\right)dY^{N}\Biggr].\end{multline}

Assume  that $F$ is nonnegative. Using 
(\ref{eq:scal16z1}), we obtain
\begin{multline}\label{eq:scal17}
\exp\left(-t\mathfrak 
A^{X}_{b,\vartheta}\right)F\left(x_{0},Y_{0}\right)\le
\exp\left(\frac{mt}{2b^{2}}\right)\\
E\Biggl[\exp\left(-\frac{1}{2b^{2}}\int_{0}^{t}\left\vert  
Y^{TX}\right\vert^{2}ds\right)
\int_{N_{x_{t}}}^{}\left(\frac{e^{\cos\left(\vartheta\right)t/b^{2}}}
{2\pi\sinh\left(\cos\left(\vartheta\right)t/b^{2}\right)} \right) 
^{n/2} \\
\exp\left(
-\frac{1}{2}\tanh\left(\cos\left(\vartheta\right)t/2b^{2}\right)\left(\left\vert  Y_{0}^{N}\right\vert^{2}+
\left\vert  Y^{N }\right\vert^{2}\right)\right)\\
F\left(x_{t},Y^{TX}_{t},
Y^{N }\right)dY^{N}\Biggr].
\end{multline}
By (\ref{eq:scal17}), we deduce  that 
\begin{multline}\label{eq:scal18}
\mathfrak 
r^{X}_{b,\vartheta}\left(\left(x,Y\right),\left(x',Y'\right)\right)\le r^{X}_{b,\vartheta}
\left(\left(x,Y^{TX}\right),\left(x',Y^{TX \prime }\right)\right) \\
\left( \frac{e^{\cos\left(\vartheta\right)t/b^{2}}}
{2\pi\sinh\left(\cos\left(\vartheta\right)t/b^{2}\right)} \right) 
^{n/2}\exp\left(-\frac{1}{2}\tanh\left(\cos\left(\vartheta\right)t/2b^{2}\right)
\left(\left\vert Y^{N}\right\vert^{2}+\left\vert  Y^{N \prime 
}\right\vert^{2}\right) \right) .
\end{multline}
For $x\ge 0$, we get
\begin{equation}\label{eq:phan14}
\tanh\left(x\right)=\frac{1-e^{-2x}}{1+e^{-2x}}\ge\frac{1}{2}\left(1-e^{-2x}\right).
\end{equation}
For $0<b\le M,t\ge \epsilon$, then $t/b^{2}\ge\epsilon/M^{2}$, and so
\begin{align}\label{eq:scal19}
&\frac{e^{\cos\left(\vartheta\right)t/b^{2}}}{\sinh\left(\cos\left(\vartheta\right)t/b^{2}\right)}\le
\frac{e^{\cos\left(\vartheta\right)\epsilon/M^{2}}}{\sinh\left(\cos\left(\vartheta\right)\epsilon/M^{2}\right)}
\le \frac{C}{\cos\left(\vartheta\right)},\\
&\tanh\left(\cos\left(\vartheta\right)t/2b^{2}\right)\ge\tanh\left(\cos\left(\vartheta\right)\epsilon
/2M^{2}\right)\ge C'\cos\left(\vartheta\right). \nonumber 
\end{align}
By combining equation (\ref{eq:scal6}) in Theorem \ref{Tinsi} and 
(\ref{eq:scal18})--(\ref{eq:scal19}), we get (\ref{eq:scal7a-1}), 
(\ref{eq:scal7}).

By equation (\ref{eq:sumex32bisardbis}) in Theorem \ref{Testardbis}, 
 by  (\ref{eq:stan16x1}) and  (\ref{eq:scal18}), we get
\begin{equation}\label{eq:phan16a}
\limsup_{b\to 0}\mathfrak 
r^{X}_{\bt,t}\left(\left(x,Y\right),\left(x',Y'\right)\right)\le 
\mathfrak 
r_{0,\vartheta,t}\left(\left(x,Y\right),\left(x',Y'\right)\right).
\end{equation}
Also by (\ref{eq:glab17}), (\ref{eq:sumex32bisardbis}), and 
(\ref{eq:scal16}),  as $b\to 0$, 
\begin{equation}\label{eq:phan17}
\exp\left(-t \mathfrak  A^{X}_{\bt}\right)F\left(x,Y\right)\to
\int_{\widehat{\mathcal{X}}}^{}\mathfrak 
r^{X}_{0,\vartheta,t}\left(\left(x,Y\right),\left(x',Y'\right)\right)F\left(x',Y'\right)
dx'dY'.
\end{equation}
The above is not enough to obtain the pointwise convergence in 
(\ref{eq:scal8}). However, given $\vartheta\in 
\left[0,\frac{\pi}{2}\right[$, as explained in \cite[Theorems 12.10.1 
and 12.10.2]{Bismut08b}, we can still use the Malliavin calculus, and 
show that given $\left(x,Y\right)\in \widehat{\mathcal{X}}$, 
$\mathfrak 
r^{X}_{\bt}\left(\left(x,Y\right),\left(x',Y'\right)\right)$ and its 
derivatives in $\left(x',Y'\right)$ of arbitrary order are uniformly 
bounded on any compact set. By combining this with equation 
(\ref{eq:phan17}), we get (\ref{eq:scal8}). The proof of our theorem is completed. 
\end{proof}

%% file: Eta10.tex
\section{The hypoelliptic heat kernel $\overline{q}^{X}_{b,\vartheta,t}$ for small $b$}%
\label{sec:fin}
The purpose of this section is to prove Theorem \ref{Test}. More 
precisely, we establish uniform estimates on the smooth kernel 
$\overline{q}^{X}_{\bt,t}\left(\left(x,Y\right),\left(x',Y'\right)\right)$ when $0<b\le 1$, and we prove that as $b\to 0$, $\overline{q}^{X}_{\bt,t}$
converges to $\overline{q}^{X}_{0,\vartheta,t}$. A scalar version of 
our estimates was established in section \ref{sec:unisca}. While in 
\cite{Bismut08b}, when $\vartheta=0$, the corresponding estimates for 
$\overline{q}^{X}_{\bt,t}$ could be derived relatively easily from 
the estimates on the scalar version of this heat kernel, this is not 
the case here, because of the matrix structure of our operator. 
At a technical level, while  the solutions of certain linear 
stochastic differential equations could be controlled in 
\cite{Bismut08b} using a version of Gronwall's lemma, this is no 
longer possible. This alone explains the length of this section. 

This section is organized as follows. In subsection 
\ref{subsec:probell}, we give a probabilistic expression for the 
semigroup $\exp\left(-tT^{X}\right)$.

In subsection \ref{subsec:diff}, we give a crucial identity on the 
operator $\overline{\mathcal{M}}^{X}_{\bt}$, which is conjugate to 
the operator 
$\overline{\mathcal{L}}^{X}_{\bt}$

In subsection \ref{subsec:sign}, we make the innocuous change of 
coordinates $Y\to -Y$.

In subsection \ref{subsec:prcs}, we give a probabilistic construction 
of the hypoelliptic  heat equation semigroup associated with 
$\overline{\mathcal{L}}^{X}_{\bt}$. The solution
 $U^{0}_{\bt,t}$ of a linear differential stochastic equation appears, whose uniform 
estimate in the proper $L_{p}$ space turns out to be the main 
difficulty 
in this section.

In subsection \ref{subsec:cru}, we give a crude and insufficient  estimate on 
$U^{0}_{\bt,t}$.

In subsection \ref{subsec:unie}, we obtain a uniform $L_{p}$ estimate 
on the solution $E_{b,\vartheta,\cdot}$ of another stochastic 
differential equation.

In subsection \ref{subsec:estu}, we introduce a stochastic process 
$H_{\bt,t}$, which will be used in subsection \ref{subsec:infser} to 
estimate $U_{\bt,t}^{0}$. The process $H_{\bt,t}$ can itself be 
easily estimated.

In subsection \ref{subsec:infser}, we express $U^{0}_{\bt,t}$ as an 
infinite series in which the process $H_{\bt,\cdot}$ appears.

In subsection \ref{subsec:cruest}, we give a uniform estimate on  the 
$L_{p}$ norm of $\left\vert  U^{0}_{\bt,t}\right\vert$ for bounded 
$t$. To obtain this estimate, we estimate the various terms of the 
series obtained in subsection \ref{subsec:infser}. The generalized It\^{o}
formula of subsection \ref{subsec:geito} and the Girsanov 
transformation play an essential role in the proof of the estimate.

In subsection \ref{subsec:limbs},  using the estimate of subsection 
\ref{subsec:cruest}, we obtain a uniform estimate on the $L_{p}$ norm 
of $\sup_{0\le t\le M}\left\vert  U^{0}_{\bt,t}\right\vert$.

In subsection \ref{subsec:limu}, we evaluate the limit as $b\to 0$ of 
$U^{0}_{\bt,\cdot}$. 

In subsection \ref{subsec:tcs}, we also handle the $d\vartheta$ 
component of our operators.

Finally, in subsection \ref{subsec:finst}, we establish Theorem 
\ref{Test}. The uniform estimates on the kernel 
$\overline{q}^{X}_{\bt,t}$ are proved using the Malliavin calculus 
and the estimates of the previous subsections. The convergence of the 
heat kernels as $b\to 0$ in a weak sense is proved using the 
results of the previous subsections. The convergence of the heat 
kernels themselves is proved by combining these results.

We use the conventions and notation of the previous sections.
\subsection{A probabilistic expression for 
$\exp\left(-t\mathcal{L}^{X}_{0,\vartheta}\right)$}%
\label{subsec:probell}
Recall that the elliptic operator 
\index{LXt@$\mathcal{L}^{X}_{0,\vartheta}$}%
$\mathcal{L}^{X}_{0,\vartheta}$ on 
$X$ was defined in Definition \ref{DLXt}. 

We use the same notation as in subsection \ref{subsec:probhea}. Let 
\index{wp@$w_{\cdot}^{\mathfrak p}$}%
$ 
w_{\cdot}^{\mathfrak p}$ be a Brownian motion with 
values in $\mathfrak p$ such that $w^{\mathfrak p}_{0}=0$, and let
\index{wTX@$w^{TX}_{\cdot}$}%
$w^{TX}_{\cdot}$ denote the corresponding 
Brownian motion in $T_{x_{0}}X$. Let $E$ 
denote the corresponding expectation operator. Consider the stochastic 
differential equation in $X$,
\begin{align}\label{eq:stoch1}
&\dot x=\cos\left(\vartheta\right)\dot w^{TX},&x_{0}=p1_{G},
\end{align}
and also its horizontal lift in $G$, 
\begin{align}\label{eq:stoch2}
&\dot g=\cos\left(\vartheta\right)\dot w^{\mathfrak p},&g_{0}=1,
\end{align}
so that
\begin{equation}\label{eq:stoch3}
x_{\cdot}=pg_{\cdot}.
\end{equation}
Equation (\ref{eq:anst-1}) is a special case of (\ref{eq:stoch2}) for 
$\vartheta=0$.

Let $\tau^{0}_{t}$ denote the parallel transport operator from 
$x_{0}$ to $x_{t}$ along the curve $x_{\cdot}$ with respect to the 
connection $\n^{S^{\overline{TX}} \otimes F}$, and let $\tau^{t}_{0}$ 
be its inverse. Since $g_{\cdot}$ is the horizontal lift of 
$x_{\cdot}$, $\tau^{0}_{t}$ can easily be obtained from $g_{t}$.

We give an extension of \cite[Proposition 14.1.1]{Bismut08b}.
\begin{prop}\label{Pformel}
Let $u\in C^{\infty,c}\left(X,S^{\overline{TX}} \otimes F\right)$. 
For $t\ge 0$,  the following identity holds:
\begin{multline}\label{eq:stoch4}
\exp\left(-t\mathcal{L}^{X}_{0,\vartheta}\right)u\left(x_{0}\right)=
\exp\Biggl(-\frac{\cos^{2}\left(\vartheta\right)}{2}
\Biggl(\frac{1}{4}\Tr^{\mathfrak p}\left[C^{\mathfrak k,\mathfrak p}
\right]
+2\sum_{i=m+1}^{m+n}\widehat{c}\left(\ad\left(e_{i}\right)\vert_{\overline{TX}}\right) \\
\rho^{F}\left(e_{i}\right)\Biggr)t-\left(\frac{1}{48}\Tr^{\mathfrak 
k}\left[C^{\mathfrak k, \mathfrak k}\right]+\frac{1}{2}C^{\mathfrak 
k,F}\right)
t\Biggr)E\left[\tau^{t}_{0}u\left(x_{t}\right)\right].
\end{multline}
\end{prop}
\begin{proof}
Our theorem follows from the third expression for 
$\mathcal{L}^{X}_{0,\vartheta}$ in equation (\ref{eq:japo4})  and from It\^{o}'s formula.
\end{proof}

Recall that the elliptic operator
\index{TX@$T^{X}$}%
$T^{X}$ was defined in Definition 
\ref{DTX}. Let us  show how to modify equation (\ref{eq:stoch4}) in order 
to give a formula for the action of $\exp\left(-tT^{X}\right)$. With 
the notation in (\ref{eq:stoch1})--(\ref{eq:stoch2}), 
we introduce the stochastic differential equation
\begin{align}\label{eq:roba8}
&dA= A\left( -\cos^{2}\left(\vartheta\right)\sum_{i=m+1}^{m+n}\widehat{c}\left(\ad\left(e_{i}\right)
\vert_{\overline{TX}}\right)\rho^{F}\left(e_{i}\right)dt-\frac{d\vartheta}{\sqrt{2}}
\widehat{c}\left(d\overline{w}^{TX}\right) \right),\\
&A_{0}=1. \nonumber 
\end{align}
Using equation (\ref{eq:defa5x1c}) for $T^{X}$,  for $t\ge 0$, we get
\begin{multline}\label{eq:stoch4bis}
\exp\left(-tT^{X}\right)u\left(x_{0}\right)=\exp\left(-\frac{\cos^{2}\left(\vartheta\right)}{8}
\Tr^{\mathfrak p}\left[C^{\mathfrak k, \mathfrak p}\right]t
-\left( \frac{1}{48}\Tr^{\mathfrak k}\left[C^{\mathfrak k, \mathfrak 
k}\right]+\frac{1}{2}C^{\mathfrak k, F} \right)t \right)\\
E\left[A_{t}\tau^{t}_{0}u\left(x_{t}\right)\right].
\end{multline}

\subsection{An identity of partial differential operators}%
\label{subsec:diff}
Recall that  
\index{LXbt@$\mathcal{L}_{b,\vartheta}^{X}$}%
$\overline{\mathcal{L}}^{X}_{b,\vartheta}$ 
is given by equation (\ref{eq:co19x-1}), and 
\index{LX@$\overline{L}^{X}$}%
$\overline{L}^{X}$ by 
(\ref{eq:co36}).
 Set
 \index{MXbt@$\overline{\mathcal{M}}^{X}_{b,\vartheta}$}%
 \index{MX@$\overline{M}^{X}$}%
 \begin{align}\label{eq:psic-4}
&\overline{\mathcal{M}}^{X }_{b,\vartheta}=\exp\left(\left\vert  
Y\right\vert^{2}/2\right)\overline{\mathcal{L}}^{X 
}_{b,\vartheta}\exp\left(-\left\vert  Y\right\vert^{2}/2\right),\\
&\overline{M}^{X}=\exp\left(\left\vert  
Y\right\vert^{2}/2\right)\overline{L}^{X}\exp\left(-\left\vert  
Y\right\vert^{2}/2\right). \nonumber 
\end{align}
By (\ref{eq:co19x-1}), we get
 \begin{multline}\label{eq:formid1a1}
   \overline{\mathcal{M}}^{X }_{b,\vartheta}=\frac{\cos\left(\vartheta\right)}{2}\left\vert  \left[Y^{ N},Y^{TX}\right]\right\vert^{2}
     +\frac{1}{2b^{2}} \left( -\Delta^{TX}+2\n_{Y^{TX}}^{V}\right) \\
     +\frac{\cos\left(\vartheta\right)}{2b^{2}}
     \left(-\Delta^{N}+2\n^{V}_{Y^{N}}\right)
    +\frac{N^{\Lambda\ac\left(T^{*}X \oplus N^{*}\right)}_{-\vartheta}}{b^{2}}\\
     +\frac{\cos\left(\vartheta\right)}{b}\Biggl(\n_{Y^{ TX}}^{C^{ \infty }\left(TX \oplus 
 N,\widehat{\pi}^{*} \left( \Lambda\ac\left(T^{*}X \oplus N^{*}\right)\otimes S^{\overline{TX}} \otimes F 
 \right) \right)}
     +
	 \widehat{c}_{\vartheta}\left(\ad\left(Y^{TX}\right)
	 \vert_{TX \oplus N} \right) \\ 
  -c\left(\ad\left(Y^{ 
	  TX}\right)
   \right) \Biggr) 
   -\frac{i\cos^{1/2}\left(\vartheta\right)}{b}
   \left(c\left(\theta\ad\left(Y^{N}\right)\right)+
   \widehat{c}_{\vartheta}\left(\ad\left(Y^{N}\right)\vert_{\overline{TX}}\right)
   +\rho^{F}\left(Y^{N}\right)\right). 
 \end{multline}
 By (\ref{eq:co29x1}), (\ref{eq:co36}),  and (\ref{eq:psic-4}), we 
 have the identities
 \begin{align}\label{eq:mir-3}
     &\overline{L}^{X}\vert_{db=0}=\overline{\mathcal{L}}^{X}_{b,\vartheta} \nonumber  \\
&\qquad\qquad -\frac{d\vartheta}{\sqrt{2}b}\left( 
b
\frac{\sin\left(\vartheta\right)}{\cos^{1/2}\left(\vartheta\right)}
ic\left(\left[Y^{N},Y^{TX}\right]\right) 
+\widehat{c}\left(\overline{Y}^{TX}\right) +
\frac{\sin\left(\vartheta\right)}{\cos^{1/2}\left(\vartheta\right)}i\mathcal{E}^{N}
\right),\\
&\overline{M}^{X}\vert_{db=0}=\overline{\mathcal{M}}^{X}_{b,\vartheta} \nonumber \\
&\qquad \qquad -\frac{d\vartheta}{\sqrt{2}b}\left( 
b
\frac{\sin\left(\vartheta\right)}{\cos^{1/2}\left(\vartheta\right)}
ic\left(\left[Y^{N},Y^{TX}\right]\right) 
+\widehat{c}\left(\overline{Y}^{TX}\right) +
\frac{\sin\left(\vartheta\right)}{\cos^{1/2}\left(\vartheta\right)}i\mathcal{E}^{N}
\right).\nonumber  
\end{align}

Set
\index{RtY@$\mathcal{R}_{\vartheta}\left(Y^{TX}\right)$}%
\index{RtY@$\mathcal{R}_{\vartheta}\left(Y^{N}\right)$}%
\index{RtY@$\mathcal{R}_{\vartheta}\left(Y\right)$}%
\begin{align} \label{eq:formid1}
   &\mathcal{R}_{\vartheta}\left(Y^{TX}\right)=\cos\left(\vartheta\right)\Biggl(\n_{Y^{ TX}}^{C^{ \infty }\left(TX \oplus 
 N,\widehat{\pi}^{*} \left( \Lambda\ac\left(T^{*}X \oplus N^{*}\right)\otimes S^{\overline{TX}} \otimes F 
 \right) \right)}\nonumber\\
  &   +
	 \widehat{c}_{\vartheta}\left(\ad\left(Y^{TX}\right)
	 \vert_{TX \oplus N} \right) 
-c\left(\ad\left(Y^{ 
	  TX}\right)
   \right) \Biggr),\\ 
  & \mathcal{R}_{\vartheta}\left(Y^{N}\right)=-i\cos^{1/2}\left(\vartheta\right)
   \left(c\left(\theta\ad\left(Y^{N}\right)\right)+
   \widehat{c}_{\vartheta}\left(\ad\left(Y^{N}\right)\vert_{\overline{TX}}\right)
   +\rho^{F}\left(Y^{N}\right)\right), \nonumber \\  
   &\mathcal{R}_{\vartheta}\left(Y\right)=\mathcal{R}_{\vartheta}\left(Y^{TX}\right)+\mathcal{R}_{\vartheta}\left(Y^{N}\right). \nonumber 
 \end{align}
 With the notation  in (\ref{eq:qsic4}), (\ref{eq:qsic5}), we 
 have the identities
 \begin{align}\label{eq:formido1}
&\mathcal{R}_{\vartheta}\left(Y^{TX}\right)=\cos\left(\vartheta\right)\n_{Y^{ TX}}^{C^{ \infty }\left(TX \oplus 
 N,\widehat{\pi}^{*} \left( \Lambda\ac\left(T^{*}X \oplus N^{*}\right)\otimes S^{\overline{TX}} \otimes F 
 \right) \right)}+R_{\vartheta}\left(Y^{TX}\right), \nonumber \\
 &\mathcal{R}_{\vartheta}\left(Y^{N}\right)=R_{\vartheta}\left(Y^{N}\right),\\
 &\mathcal{R}_{\vartheta}\left(Y\right)=\cos\left(\vartheta\right)\n_{Y^{ TX}}^{C^{ \infty }\left(TX \oplus 
 N,\widehat{\pi}^{*} \left( \Lambda\ac\left(T^{*}X \oplus N^{*}\right)\otimes S^{\overline{TX}} \otimes F 
 \right) \right)}+R_{\vartheta}\left(Y\right). \nonumber 
\end{align}

 By (\ref{eq:formid1a1}), we get
 \begin{multline}\label{eq:wat1}
\overline{\mathcal{M}}^{X }_{b,\vartheta}=\frac{\cos\left(\vartheta\right)}{2}\left\vert  \left[Y^{ N},Y^{TX}\right]\right\vert^{2}
     +\frac{1}{2b^{2}} \left( -\Delta^{TX}+2\n_{Y^{TX}}^{V}\right) \\
     +\frac{\cos\left(\vartheta\right)}{2b^{2}}
     \left(-\Delta^{N}+2\n^{V}_{Y^{N}}\right)
    +\frac{N^{\Lambda\ac\left(T^{*}X \oplus N^{*}\right)}_{-\vartheta}}{b^{2}}
    +\frac{\mathcal{R}_{\vartheta}\left(Y\right)}{b}.
\end{multline}

Now we establish an extension of \cite[Proposition 14.2.1]{Bismut08b}.
\begin{prop}\label{Pwondidbis}
For $0\le \vartheta<\frac{\pi}{2}$, if  $s\in C ^{ \infty } \left( X,\Lambda\ac\left(T^{*}X \oplus 
N^{*}\right) \otimes S^{\overline{TX}} \otimes  F \right) $, then
\begin{multline}\label{eq:formid2}
   \left(  \overline{\mathcal{M}}^{X}_{b,\vartheta}-\frac{\cos\left(\vartheta\right)}{2}\left\vert  
   \left[Y^{N},Y^{TX}\right]\right\vert^{2} \right) \\
   \Biggl(1-b
  \left(1+N^{\Lambda\ac\left(T^{*}X \oplus 
  N^{*}\right)}_{-\vartheta}\right)^{-1}\mathcal{R}_{\vartheta}\left(Y^{TX}\right)\\
    -b\left(\cos\left(\vartheta\right)+N^{\Lambda\ac\left(T^{*}X 
    \oplus 
    N^{*}\right)}_{-\vartheta}\right)^{-1}\mathcal{R}_{\vartheta}\left(Y^{N}\right)
    \Biggr) 
    \widehat{\pi}^{*}s\\
    =\frac{N^{\Lambda\ac\left(T^{*}X 
    \oplus 
    N^{*}\right)}_{-\vartheta}}{b^{2}}\widehat{\pi}^{*}s 
    -
    \mathcal{R}_{\vartheta}\left(Y\right)\Biggl( \left(1+N^{\Lambda\ac\left(T^{*}X \oplus
    N^{*}\right)}_{-\vartheta}\right)^{-1}\mathcal{R}_{\vartheta}\left(Y^{TX}\right)\\
    +\left(\cos\left(\vartheta\right)+
    N^{\Lambda\ac\left(\TsX \oplus 
    N^{*}\right)}_{-\vartheta}\right)^{-1}\mathcal{R}_{\vartheta} 
    \left(Y^{N}\right)\Biggr) \widehat{\pi}^{*}s.
\end{multline}
\end{prop}
\begin{proof}
We  proceed as in \cite{Bismut08b}. Note that the eigenspaces of the 
fibrewise operators  
$\frac{1}{2}\left(-\Delta^{TX}+2\n^{V}_{Y^{TX}}\right)$, $\frac{1}{2}\left(-\Delta^{N}+2\n^{V}_{Y^{N}}\right)$
that are associated with the eigenvalue $1$ are spanned by sections 
that are linear in $Y^{TX},Y^{N}$. By 
(\ref{eq:formid1a1})--(\ref{eq:wat1}), we get (\ref{eq:formid2}). 
\end{proof}
\subsection{Changing $Y$ in $-Y$}%
\label{subsec:sign}
We follow \cite[section 14.3]{Bismut08b}.
\begin{defin}\label{Dsig}
Let 
\index{I@$I$}%
$I$ be the map $s\left(x,Y\right)\to s\left(x,-Y\right)$. Set
\index{LXbt@$\overline{\mathcal{L}}^{X}_{b,\vartheta,-}$}%
\index{MXbt@$\overline{\mathcal{M}}^{X}_{b,\vartheta,-}$. }%
\index{LX@$\overline{L}^{X}_{-}$}%
\index{MX@$\overline{M}^{X}_{-}$}%
\begin{align}\label{eq:sig1}
&\overline{\mathcal{L}}^{X}_{b,\vartheta,-}=I\overline{\mathcal{L}}^{X}_{b,\vartheta}I^{-1},
&\overline{\mathcal{M}}^{X}_{b,\vartheta,-}=I\overline{\mathcal{M}}^{X}_{b,\vartheta}I^{-1},\\
&\overline{L}^{X}_{-}=I\overline{L}^{X}I^{-1},&\overline{M}^{X}_{-}=
I\overline{M}^{X}I^{-1}. \nonumber 
\end{align}
\end{defin}
\subsection{A probabilistic construction of the hypoelliptic
semigroups}%
\label{subsec:prcs}
As in \cite[section 
14.6]{Bismut08b}, we denote by 
\index{KC@$K_{\C}$}%
$K_{\C}$ the complexification of the 
compact Lie group  $K$. The splitting of the Lie algebra $\mathfrak 
k_{\C}$ that corresponds to the splitting $\mathfrak g=\mathfrak p 
\oplus \mathfrak k$ is just $\mathfrak k_{\C}=i \mathfrak  k \oplus \mathfrak 
k$. Let 
\index{XKC@$X_{K_{\C}}$}%
$X_{K_{\C}}=K_{\C}/K$ denote the corresponding 
symmetric space. We use the same notation on $X_{K_{\C}}$ as we used 
before for the symmetric space $X$. In particular $p$ is the 
projection $K_{\C}\to X_{K_{\C}}$, and $d$ still denotes 
the Riemannian distance on $X_{K_{\C}}$. Also the action of $K$
on $S^{\overline{\mathfrak p}} $ extends to $K_{\C}$.

Let 
\index{wp@$w_{\cdot}^{\mathfrak p}$}%
\index{wk@$w^{\mathfrak k}_{\cdot}$}%
$w_{\cdot}^{\mathfrak p},w_{\cdot}^{\mathfrak 
k}$ be Brownian motions in $\mathfrak p,\mathfrak k$ such that 
$w_{0}^{\mathfrak p}=0,w_{0}^{\mathfrak k}=0$. We denote by
\index{wTX@$w^{TX}_{\cdot}$}%
\index{wN@$w^{N}_{\cdot}$}%
$w^{TX}_{\cdot},w^{N}_{\cdot}$ the corresponding processes with 
values in $T_{x_{0}}X,N_{x_{0}}$. Let
\index{EP@$E^{P}$}%
$E^{P}$ be the 
corresponding expectation operator. 

Instead of \cite[eq. 
(14.4.1)]{Bismut08b}, we consider the differential equation for 
$\left(x_{\cdot},y_{\cdot}\right)\in X\times X_{K_{\C}}, 
Y_{\cdot}=\left(Y^{TX}_{\cdot},Y^{N}_{\cdot}\right)\in \left( TX 
\oplus N \right) _{x_{0}}$,
\begin{align}\label{eq:wat2bav}
&\dot x=\frac{\cos\left(\vartheta\right)}{b}Y^{TX},
&\dot y=-i\frac{\cos^{1/2}\left(\vartheta\right)}{b}Y^{N},\nonumber \\
&\dot 
Y^{TX}=\frac{\dot w^{TX}}{b},& \dot 
Y^{N}=\cos^{1/2}\left(\vartheta\right)\frac{\dot 
w^{N}}{b},\\
&x_{0}=p1_{G},&y_{0}=p1_{K_{\C}},\qquad  Y_{0}=Y. \nonumber 
\end{align}
By (\ref{eq:wat2bav}), we get
\begin{align}\label{eq:wat2bava1}
&b^{2}\ddot x=\ct\dot w^{TX}, &b^{2}\ddot y=-i\ct\dot w^{N}.
\end{align}

Instead of \cite[eqs. (14.6.4) and (14.6.5)]{Bismut08b}, we also consider the 
associated equations on $\left(g_{\cdot},h_{\cdot}\right)\in G\times 
K_{\C}$,
\begin{align}\label{eq:wat2y1}
&\dot g=\frac{\cos\left(\vartheta\right)}{b}Y^{\mathfrak p},
&\dot h=-i\frac{\cos^{1/2}\left(\vartheta\right)}{b}Y^{\mathfrak k}, 
\nonumber \\
&\dot Y^{ \mathfrak p}=\frac{\dot w^{\mathfrak p}}{b}, 
&\dot Y^{\mathfrak k}=\cos^{1/2}\left(\vartheta\right)\frac{\dot 
w^{\mathfrak k}}{b},\\
&g_{0}=1_{G},&h_{0}=1_{K_{\C}}.,\, Y_{0}=Y. \nonumber 
\end{align}
In (\ref{eq:wat2y1}), $\left(Y^{\mathfrak p},Y^{\mathfrak 
k}\right)\in \mathfrak p \oplus \mathfrak k$ 
corresponds to $\left(Y^{TX},Y^{N}\right)\in TX \oplus N$.
Then
\begin{align}\label{eq:wat2y2}
&x_{\cdot}=pg_{\cdot},&y_{\cdot}=ph_{\cdot}.
\end{align}

We still denote by $\tau^{0}_{t}$ the parallel transport along 
$x_{\cdot}$ from $x_{0}$ to $x_{t}$ and by $\tau^{t}_{0}$ its inverse.

Instead of \cite[eq. (14.4.2)]{Bismut08b}, we consider the 
differential equation,
\index{Ubt@$U_{b,\vartheta}$}%
\begin{align}\label{eq:glab60}
    &\frac{dU_{b,\vartheta}}{dt}=U_{b,\vartheta}\left[-\frac{N^{
    \Lambda\ac\left(T^{*}X \oplus N^{*}\right)}_{-\vartheta}}{b^{2}}+\frac{R_{\vartheta}\left(Y\right)}{b}   
    \right], 
&U_{b,\vartheta,0}=1. 
\end{align}

Then we have the extension of \cite[Theorem 14.4.1]{Bismut08b}.
\begin{thm}\label{Thesga}
If $s\in C^{\infty,c}\left(\widehat{\mathcal{X}},\widehat{\pi}^{*}\left(\Lambda\ac\left(T^{*}X \oplus N
^{*}\right) \otimes S^{\overline{TX}} \otimes F\right)\right)$, 
then 
\begin{multline}\label{eq:sen1}
\exp\left(-t\overline{\mathcal{L}}^{X}_{b,\vartheta,-}\right)s\left(x_{0},Y\right)
=E^{P}\Biggl[\exp\Biggl(\frac{m+\cos\left(\vartheta\right)n}{2b^{2}}t 
\\
-\frac{\cos\left(\vartheta\right)}{2}\int_{0}^{t}\left\vert  \left[Y^{N},Y^{TX}\right]\right\vert
^{2}ds
-\frac{1}{2b^{2}}\int_{0}^{t}\left(\left\vert  
Y^{TX}\right\vert^{2}+\cos\left(\vartheta\right)\left\vert  
Y^{N}\right\vert^{2}\right)ds\Biggr) \\
U_{b,\vartheta,t}\tau^{t}_{0}s\left(x_{t},Y_{t}\right)\Biggr].
\end{multline}
\end{thm}
\begin{proof}
Equation (\ref{eq:sen1}) 
is a  consequence of (\ref{eq:co19x-1}), (\ref{eq:qsic4}),   (\ref{eq:sig1}),  and of It\^{o}'s formula.
\end{proof}

Instead of \cite[eq. (14.8.8)]{Bismut08b}, given $Y=\left(Y^{TX},Y^{ N}\right)
\in \left(TX \oplus N\right)_{x_{0}}$, 
we consider now the 
differential equation,
\begin{align}\label{eq:wat2}
&\dot x=\frac{\cos\left(\vartheta\right)}{b}Y^{TX},
&\dot y=-i\frac{\cos^{1/2}\left(\vartheta\right)}{b}Y^{N},\nonumber \\
&\dot 
Y^{TX}=-\frac{Y^{TX}}{b^{2}}+\frac{\dot w^{TX
}}{b},& \dot 
Y^{N}=-\frac{\cos\left(\vartheta\right)}{b^{2}}Y^{N}+
\frac{\cos^{1/2}\left(\vartheta\right)}{b}\dot 
w^{N},\\
&x_{0}=p1_{G},& Y_{0}=Y. \nonumber 
\end{align}
Also we  define $\left(g_{\cdot},h_{\cdot}\right)$ as in 
(\ref{eq:wat2y1}), so that (\ref{eq:wat2y2})  holds. We still 
define $U_{b,\vartheta,\cdot}$ as in (\ref{eq:glab60}). By 
(\ref{eq:wat2}), we get
\begin{align}\label{eq:sig3}
&b^{2}\ddot x+\dot x=\cos\left(\vartheta\right)\dot w^{TX},
&b^{2}\ddot y+\cos\left(\vartheta\right)\dot y=-i\ct\dot w^{N}.
\end{align}

We now denote by
\index{EQ@$E^{Q}$}%
$E^{Q}$ the expectation operator.
Then we have the 
following extension of \cite[eq. (14.9.2)]{Bismut08b}.
\begin{thm}\label{Thesgab}
If $s\in C^{\infty,c}\left(\widehat{\mathcal{X}},\widehat{\pi}^{*}\left(\Lambda\ac\left(T^{*}X \oplus N
^{*}\right) \otimes S^{\overline{TX}} \otimes F\right)\right)$,  then
\begin{multline}\label{eq:sen1b}
\exp\left(-t\overline{\mathcal{M}}^{X}_{b,\vartheta,-}\right)s\left(x_{0},Y\right)\\
=E^{Q}\left[\exp\left(-\frac{\cos\left(\vartheta\right)}{2}\int_{0}^{t}\left\vert  \left[Y^{N},Y^{TX}\right]\right\vert
^{2}ds\right)U_{b,\vartheta,t}\tau^{t}_{0}s\left(x_{t},Y_{t}\right)\right].
\end{multline}
\end{thm}
\begin{proof}
This follows from (\ref{eq:qsic4}), (\ref{eq:formid1a1}),  (\ref{eq:sig1}), and from 
It\^{o}'s formula.
\end{proof}

We will show how to modify (\ref{eq:sen1}), (\ref{eq:sen1b}) when replacing 
$\overline{\mathcal{L}}^{X}_{\bt,-},\overline{\mathcal{M}}^{X}_{\bt,-}$ by 
$\overline{L}^{X}_{-}\vert_{db=0},\overline{M}^{X}_{-}\vert_{db=0}$. 
Set
\index{Rt@$R_{\vartheta}'\left(Y\right)$}%
\begin{equation}\label{eq:bobi1}
R_{\vartheta}'\left(Y\right)=R_{\vartheta}\left(Y\right)-
\frac{d\vartheta}{\sqrt{2}}\left(\widehat{c}\left(\overline{Y}^{TX}\right)
+\frac{\sin\left(\vartheta\right)}{\cos^{1/2}\left(\vartheta\right)}i\mathcal{E}^{N}\right).
\end{equation}
Let 
\index{Ubt@$U'_{\bt,\cdot}$}%
$U'_{\bt,\cdot}$ be the solution of the differential equation,
\begin{align}\label{eq:bobi2}
&\frac{dU'_{\bt}}{dt}=U'_{\bt}\left[-\frac{N_{-\vartheta}^{\LXN}}{b^{2}}
+\frac{R'_{\vartheta}\left(Y\right)}{b}+\frac{d\vartheta}{\sqrt{2}}
\frac{\sin\left(\vartheta\right)}{\cos^{1/2}\left(\vartheta\right)}ic\left(\left[Y^{N},Y^{TX}\right]
\right)\right],\\
&U'_{\bt,0}=1. \nonumber 
\end{align}
By (\ref{eq:mir-3}), when replacing $\overline{\mathcal{L}}^{X}_{\bt,-}$ by 
$\overline{L}^{X}_{-}\vert_{db=0}$, or  $\overline{\mathcal{M}}^{X}_{\bt,-}$ by 
$\overline{M}^{X}_{-}\vert_{db=0}$, in (\ref{eq:sen1}), (\ref{eq:sen1b}), one should simply 
replace $U_{\bt,t}$ by $U'_{\bt,t}$.

Recall that 
\index{RtY@$R_{\vartheta}\left(Y\right)$}%
$R_{\vartheta}\left(Y\right)$ is given by 
(\ref{eq:qsic4}). In Definition \ref{Dst}, 
\index{R0Y@$R^{0}_{\vartheta}\left(Y\right)$}%
$R_{\vartheta}^{0}\left(Y\right)$ was defined to be 
$R_{\vartheta}\left(Y\right)$ when $E=\C$ is the trivial 
representation of $K$.  By (\ref{eq:qsic4}), we get
\begin{multline}\label{eq:qsic4b}
R^{0}_{\vartheta}\left(Y\right)=\cos\left(\vartheta\right)\left(
	 \widehat{c}_{\vartheta}\left(\ad\left(Y^{TX}\right)
	 \vert_{TX \oplus N} \right)
  -c\left(\ad\left(Y^{ 
	  TX}\right)
   \right) \right) \\
   -i\cos^{1/2}\left(\vartheta\right)
   \left(c\left(\theta\ad\left(Y^{N}\right)\right)+
   \widehat{c}_{\vartheta}\left(\ad\left(Y^{N}\right)\vert_{\overline{TX}}\right)\right). 
 \end{multline}
 By (\ref{eq:qsic4b}), we get
 \begin{equation}\label{eq:est0}
\left\vert  
R^{0}_{\vartheta}\left(Y\right)\right\vert\le 
C' \left( \cos\left(\vartheta\right)\left\vert  Y^{TX}\right\vert
+\cos^{1/2}\left(\vartheta\right)\left\vert  Y^{N}\right\vert \right) .
\end{equation}

Let 
\index{U0bt@$U^{0}_{b,\vartheta,\cdot}$}%
$U^{0}_{b,\vartheta,\cdot}$ be 
$U_{b,\vartheta,\cdot}$ in this special case. More precisely, 
$U^{0}_{b,\vartheta,\cdot}$ is the solution of the differential 
equation,
\begin{align}\label{eq:diff0}
&\frac{dU^{0}_{b,\vartheta}}{ds}=U^{0}_{b,\vartheta}
\left[-\frac{N^{\Lambda\ac\left(T^{*}X \oplus N^{*}\right)}_{-\vartheta}}{b^{2}}+\frac{R^{0}_{\vartheta}\left(Y\right)}{b}
\right],
&U^{0}_{b,\vartheta,0}=1.
\end{align}
Let 
\index{Ebt@$E_{b,\vartheta,\cdot}$}%
$E_{b,\vartheta,\cdot}$ be the 
solution of the differential equation,
\begin{align}\label{eq:diff1}
&\frac{dE_{b,\vartheta}}{ds}=E_{b,\vartheta}\left
[-i\frac{\cos^{1/2}\left(\vartheta\right)}{b}\rho^{E}\left(Y^{N}\right)\right],
&E_{b,\vartheta,0}=1.
\end{align}
Then we have the obvious identity 
\begin{equation}\label{eq:diff2}
U_{b,\vartheta,\cdot}=U^{0}_{b,\vartheta,\cdot} \otimes E_{b,\vartheta,\cdot}.
\end{equation}

Put
\index{RY@$R^{1}\left(Y\right)$}%
\begin{equation}\label{eq:qsic4d}
R^{1}\left(Y\right)=\widehat{c}\left(\ad\left(Y^{TX}\right)\vert_{TX \oplus N}\right)-c\left(\ad\left(Y^{TX}\right)\right)
-i\left(c\left(\theta\ad\left(Y^{N}\right)\right)\right).
\end{equation}
By (\ref{eq:qsic4b}),  $R^{0}_{0}\left(Y\right)$ splits as the sum of two commuting 
pieces,
\begin{equation}\label{eq:qsic4e}
R^{0}_{0}\left(Y\right)=R^{1}\left(Y\right)-i\widehat{c}\left(\ad\left(Y^{N}\right)\vert_{\overline{TX}}\right).
\end{equation}

We make temporarily $\vartheta=0$ in (\ref{eq:wat2bav}), 
(\ref{eq:wat2}). Let $V_{b,\cdot},W_{b,\cdot}$ be the solutions of the differential 
equations,
\begin{align}\label{eq:qsic4ea}
&\frac{dV_{b}}{ds}=V_{b}\left[-\frac{N^{\Lambda\ac\left(T^{*}X \oplus N^{*}
\right)}}{b^{2}}+\frac{R^{1}\left(Y\right)}{b}\right],&V_{b,0}=1,\\
&\frac{dW_{b}}{ds}=W_{b}\left[-\frac{i}{b}\widehat{c}\left(\ad\left(Y^{N}\right)
\vert_{\overline{TX}}\right)\right],&W_{b,0}=1. \nonumber 
\end{align}
As in (\ref{eq:diff2}), $U^{0}_{b,\cdot}=U^{0}_{b,0,\cdot}$ can be 
written in the form,
\begin{equation}\label{eq:qsic4f}
U^{0}_{b,\cdot}=V_{b,\cdot} \otimes W_{b,\cdot}.
\end{equation}
The factorization (\ref{eq:qsic4f}) was already obtained 
in \cite[eq. (14.4.6)]{Bismut08b}.  
Using the fact that $R^{1}\left(Y\right)$ maps forms of 
degree zero to forms of positive degree, a coarse estimate on 
$V_{b,\cdot}$ could be obtained from the differential equation 
(\ref{eq:qsic4ea}) using Gronwall's lemma. 
This coarse estimate played a critical role in obtaining the 
behaviour of $V_{b,\cdot}$ as $b\to 0$. The behaviour of $W_{b,\cdot}$ as 
$b\to 0$ was obtained via the analysis of the hypoelliptic Laplacian 
on the complexification $K_{\C}$ of $K$.

In the sequel, the main difficulty is that for $\vartheta>0$, there is no 
factorization of $U^{0}_{b,\vartheta,\cdot}$ similar to (\ref{eq:qsic4f}), 
to which the above methods could be applied. 

For $\vartheta=\frac{\pi}{2}$, in (\ref{eq:wat2bav}) or in 
(\ref{eq:wat2}), we have
\begin{equation}\label{eq:prin-1}
Y^{N}_{t}=Y^{N}_{0}.
\end{equation}
By (\ref{eq:gip1}), (\ref{eq:qsic4}),  and (\ref{eq:glab60}),  we get
\begin{equation}\label{eq:prin-2}
U_{b,\frac{\pi}{2},t}=\exp\left(-tN^{\Lambda\ac\left(T^{*}X\right)}/b^{2}\right).
\end{equation}
Let $\mathbf{P}^{\Lambda\ac\left(N^{*}\right)}$ be the orthogonal 
projection from $\Lambda\ac\left(T^{*}X \oplus N^{*}\right)$ on 
$\Lambda\ac\left(N^{*}\right)$. By (\ref{eq:prin-2}), as $b\to 0$, for $t>0$, 
\begin{equation}\label{eq:prin-3}
U^{0}_{b,\frac{\pi}{2},t}\to \mathbf{P}^{\Lambda\ac\left(N^{*}\right)}, 
\end{equation}
and the convergence is uniform when $t$ stays away from $0$.

When $E=\C$, is the trivial representation, let 
\index{Ubt@$U^{\prime 0}_{\bt,\cdot} $}%
$U^{\prime 0}_{\bt,\cdot}$  be $U'_{\bt,\cdot}$ in this special 
case.  Instead of (\ref{eq:diff2}), we have the identity 
\begin{equation}\label{eq:diff2a1}
U'_{\bt,\cdot}=U^{\prime 0}_{\bt,\cdot}\otimes E_{\bt,\cdot}.
\end{equation}
\subsection{A crude estimate on $U^{0}_{b,\vartheta,\cdot}, 
E_{b,\vartheta,\cdot}$}%
\label{subsec:cru}
We will now establish a crude estimate on 
$U^{0}_{\bt,\cdot},E_{\bt,\cdot}$.
\begin{prop}\label{Pcru}
There exists $C>0$ such that for  
$b>0,\vartheta\in\left[0,\frac{\pi}{2}\right],t\ge 0$, then
\begin{align}\label{eq:mel1}
&\left\vert  
U^{0}_{b,\vartheta,t}\right\vert\le\exp\left(\frac{C}{b} \left( 
\cos\left(\vartheta\right)\int_{0}^{t}\left\vert  
Y^{TX}_{s}\right\vert ds+\cos^{1/2}\left(\vartheta\right)
\int_{0}^{t}
\left\vert  Y^{N}_{s}\right\vert ds\right) \right),\\
&\left\vert  
E_{b,\vartheta,t}\right\vert\le\exp\left(C\frac{\cos^{1/2}\left(\vartheta\right)}{b}
\int_{0}^{t}
\left\vert  Y^{N}_{s}\right\vert ds\right). \nonumber 
\end{align}
Moreover, there exists $C>0$ such that for 
$b>0,\vartheta\in\left[0,\frac{\pi}{2}\right],t>0$, 
\begin{equation}\label{eq:mel1x1}
\left\vert  U^{0}_{b,\vartheta,t}\right\vert\le \exp\left(
\frac{C\cos\left(\vartheta\right)}{2}\int_{0}^{t}\left\vert  
Y^{TX}_{s}\right\vert^{2}ds+C\frac{\cos^{1/2}\left(\vartheta\right)}{b}
\int_{0}^{t}\left\vert  Y^{N}_{s}\right\vert ds\right).
\end{equation}
\end{prop}
\begin{proof}
   For the first estimate in (\ref{eq:mel1}), we proceed as in 
   \cite[Theorem 14.5.2]{Bismut08b}.  Let
    $U^{0*}_{b,\vartheta,\cdot}$ be the adjoint of 
    $U^{0}_{b,\vartheta,\cdot}$ with respect to the standard 
    Hermitian product on $\Lambda\ac\left(T^{*}X \oplus 
	N^{*}\right)\otimes S^{\overline{TX}}$. Since $N^{\Lambda\ac\left(T^{*}X \oplus 
    N^{*}\right)}_{-\vartheta}, R^{0}_{\vartheta}\left(Y\right)$ are 
    self-adjoint,  by (\ref{eq:diff0}), we get
    \begin{align}\label{eq:he3}
&\frac{dU^{0*}_{b,\vartheta}}{ds}=\left[-\frac{N^{\Lambda\ac\left(T^{*}X \oplus N^{*}\right)}_{-\vartheta}}{b^{2}}
+\frac{R^{0}_{\vartheta}\left(Y\right)}{b}\right]U^{0*}_{b,\vartheta},
&U^{0*}_{b,\vartheta,0}=1.
\end{align}
By (\ref{eq:he3}), if $f\in\Lambda\ac\left(T^{*}X \oplus 
N^{*}\right)\otimes S^{\overline{TX}}$, then
\begin{equation}\label{eq:he4}
\frac{d}{ds}\left\vert  U^{0*}_{b,\vartheta}f\right\vert^{2}=2\left\langle \left(-\frac{N^{\Lambda\ac\left(T^{*}X \oplus N^{*}\right)}_{-\vartheta}}{b^{2}}
+\frac{R^{0}_{\vartheta}\left(Y\right)}{b}\right)U^{0*}_{b,\vartheta}f,U^{0*}_{b,\vartheta}f \right\rangle.
\end{equation}
By (\ref{eq:est0}), (\ref{eq:he4}), we obtain
\begin{equation}\label{eq:he5}
\frac{d}{ds}\left\vert  U^{0*}_{b,\vartheta}f\right\vert^{2}
\le \frac{C}{b}\left(\cos\left(\vartheta\right)\left\vert  
Y^{TX}\right\vert+\cos^{1/2}\left(\vartheta\right)\left\vert  
Y\right\vert^{N}\right)\left\vert  
U^{0*}_{b,\vartheta}f\right\vert^{2}.
\end{equation}
By (\ref{eq:he5}), using Gronwall's lemma, we get the first 
inequality in (\ref{eq:mel1}). The proof of the second inequality in (\ref{eq:mel1}) is
similar.

By equation (\ref{eq:nea1}) in Proposition \ref{Pproj}, we get
\begin{equation}\label{eq:he6}
\mathbf{P}R^{0}_{\vartheta}\left(Y^{TX}\right)\mathbf{P}=0.
\end{equation}
By (\ref{eq:gip1}),  (\ref{eq:est0}), and (\ref{eq:he6}), we get
\begin{equation}\label{eq:he7}
\left\vert  
\left\langle  
R^{0}_{\vartheta}\left(Y^{TX}\right)f,f\right\rangle\right\vert\le C\cos^{1/2}\left(\vartheta\right)\left\vert  
Y^{TX}\right\vert\left\vert  \sqrt{N^{\Lambda\ac\left(T^{*}X 
\oplus N^{*}\right)}_{-\vartheta}}f\right\vert\left\vert f \right\vert.
\end{equation}
By  (\ref{eq:he7}), we deduce that 
\begin{equation}\label{eq:he8}
\left\langle  \left(-\frac{N^{\Lambda\ac\left(T^{*}X \oplus N^{*}\right)}_{-\vartheta}}{b^{2}}
+\frac{R^{0}_{\vartheta}\left(Y^{TX}\right)}{b}\right)f,f\right\rangle\le C\cos\left(\vartheta\right)\left\vert  Y^{TX}\right\vert^{2}
\left\vert  f\right\vert^{2}.
\end{equation}
Using  (\ref{eq:est0}),  
(\ref{eq:he4}), (\ref{eq:he8}), and proceeding as before, we get 
(\ref{eq:mel1x1}). The proof of our proposition  is completed. 
\end{proof}
\begin{remk}\label{Rintel}
  Under $Q$, our $Y^{TX}_{\cdot}$ has the same probability 
law as $Y_{\cdot/b^{2}}$ where $Y_{\cdot}$ is taken as in  
(\ref{eq:phan-2}) with 
$E=TX$, and our $Y^{N}_{\cdot}$ has the same probability law as 
$Y_{\ct \cdot/b^{2}}$, and $Y_{\cdot}$ is  taken as in 
(\ref{eq:phan-2}) with $E=N$.   Using (\ref{eq:af1}) and (\ref{eq:mel1}),  given  $p\ge 1$, there 
    is a constant $C>0$ such that  for $0\le\beta\le 1$, 
     we get
    \begin{multline}\label{eq:prin10}
E^{Q}\left[\sup_{0\le s\le t}\left\vert  
U^{0}_{b,\vartheta,s}\right\vert^{p}\right]\le
\exp\Biggl(C^{2}t/\beta^{2}+\frac{1}{2}\frac{\left(m+\cos\left(\vartheta\right)n\right)\beta^{2}}{b^{2}}t\\
+Cb\left(1-e^{-t/b^{2}}\right)\cos\left(\vartheta\right)\left\vert  
Y_{0}^{TX}\right\vert+C\frac{b}{\cos^{1/2}\left(\vartheta\right)}
\left(1-e^{-t\cos\left(\vartheta\right)/b^{2}}\right)\left\vert  
Y_{0}^{N}\right\vert\Biggr).
\end{multline}
Since $1-e^{-x}\le x$, from (\ref{eq:prin10}), we get
\begin{multline}\label{eq:prin10x1}
E^{Q}\left[\sup_{0\le s\le t}\left\vert  
U^{0}_{b,\vartheta,s}\right\vert^{p}\right]\le
\exp\Biggl(C^{2}t/\beta^{2}+\frac{1}{2}\frac{\left(m+\cos\left(\vartheta\right)n\right)\beta^{2}}{b^{2}}t\\
+C\frac{t}{b}\cos\left(\vartheta\right)\left\vert  Y_{0}^{TX}\right\vert
+Ct\frac{\cos^{1/2}\left(\vartheta\right) }{b}\left\vert  
Y_{0}^{N}\right\vert\Biggr).
\end{multline}
In particular, when $b$ stays away from $0$, when $\left\vert  
Y_{0}\right\vert,t\ge 0$ are  uniformly bounded,  (\ref{eq:prin10x1}) remains uniformly bounded. 

Also by (\ref{eq:rot7ay1}), (\ref{eq:af1}), and (\ref{eq:mel1x1}), 
given $p\ge 1$, there is $C>0$ such that for
$b\le C^{-1/2}$, we get
\begin{multline}\label{eq:prin10y1}
E^{Q}\left[\sup_{0\le s\le t}\left\vert  
U^{0}_{b,\vartheta,s}\right\vert^{p}\right]\le\exp\Biggl(\frac{C}{2}\cos\left(\vartheta\right)mt
+\frac{1}{2}\cos\left(\vartheta\right)Cb^{2}\left\vert  
Y_{0}^{TX}\right\vert^{2}+\frac{C^{2}t}{2}\\
+\frac{1}{2}
\frac{\cos\left(\vartheta\right)}{b^{2}}nt
+C\frac{b}{\cos^{1/2}\left(\vartheta\right)}\left(1-e^{-t\cos\left(\vartheta\right)
/b^{2}}\right) \left\vert  Y_{0}^{N}\right\vert\Biggr).
\end{multline}
Using again the fact that $1-e^{-x}\le x$, from (\ref{eq:prin10y1}), 
we deduce that if $b\le C^{-1/2}$, then
\begin{multline}\label{eq:prin10y2}
E^{Q}\left[\sup_{0\le s\le t}\left\vert  
U^{0}_{b,\vartheta,s}\right\vert^{p}\right]\le\exp\Biggl(\frac{C}{2}\cos\left(\vartheta\right)mt
+\frac{1}{2}\cos\left(\vartheta\right)\left\vert  
Y_{0}^{TX}\right\vert^{2}+\frac{C^{2}t}{2}\\
+\frac{1}{2}
\frac{\cos\left(\vartheta\right)}{b^{2}}nt
+C\frac{\cos^{1/2}\left(\vartheta\right)}{b}t \left\vert  Y_{0}^{N}\right\vert\Biggr).
\end{multline}

The estimates in (\ref{eq:mel1}), (\ref{eq:mel1x1}), 
(\ref{eq:prin10})--(\ref{eq:prin10y2}) deteriorate as $b\to 0$. As we 
will see in subsection \ref{subsec:unie}, a  stronger estimate 
can be found for $E_{b,\vartheta,\cdot}$. The problem is much more 
serious for $U^{0}_{\bt,\cdot}$.
\end{remk}
\subsection{A uniform estimate on $E_{b,\vartheta,\cdot}$}%
\label{subsec:unie}
We assume equation (\ref{eq:wat2}) to be in force. 

In equation (\ref{eq:diff1}) for $E_{b,\vartheta,\cdot}$, 
$Y^{N}_{\cdot}$  in (\ref{eq:wat2}) also depends on $b,\vartheta$.   In the sequel, we use the notation
\index{Eb@$E_{b,\cdot}$}%
\begin{equation}\label{eq:sig4}
E_{b,\cdot}=E_{b,0,\cdot}.
\end{equation}
In the definition of $E_{b,0}$, $\vartheta$ is made equal 
to $0$ also in the definition of $Y^{N}_{\cdot}$.

Note that
\begin{equation}\label{eq:he9}
E_{b,\vartheta,\cdot}=E_{b/\cos^{1/2}\left(\vartheta\right),\cdot}.
\end{equation}

 Let 
 \index{E0@$E_{0,\cdot}$ }%
 $E_{0,\cdot}$ 
be the solution of the stochastic differential equation
\begin{align}\label{eq:roba2}
&dE_{0}=-E_{0}\rho^{E}\left(idw^{\mathfrak k}\right),
&E_{0,0}=1.
\end{align}
Observe that both $E_{b,\cdot}$ and $E_{0,\cdot}$ are constructed on 
the same probability space for  $w^{\mathfrak k}$. 
\begin{thm}\label{Teste}
Given $M\ge 0,p\ge 1$, there exist $C_{p}>0,C'>0$ such that for 
$b>0,\vartheta\in\left[0,\frac{\pi}{2}\right[$,
\begin{equation}\label{eq:prin-6}
E^{Q}\left[\sup_{0\le t\le M}\left\vert  E_{b,\vartheta,t}\right\vert^{p}\right]\le 
C_{p}\exp\left(\frac{1}{2}\left\vert  Y_{0}^{N}\right\vert^{2}\right).
\end{equation}
Given $b_{0}>0,0<\epsilon\le M<+ \infty ,p\ge 1$, there exist 
$C_{b_{0},M,p}>0,C'_{b_{0},\epsilon,M}>0$ 
such that for $0<b\le b_{0}, 
\vartheta\in\left[0,\frac{\pi}{2}\right[,\epsilon\le t\le M$, then
\begin{multline}\label{eq:prin-6z1}
\exp\left(-\left\vert  Y_{0}^{N}\right\vert^{2}/2\right)
E^{Q}\left[\sup_{0\le t\le M}\left\vert  
E_{b,\vartheta,t}\right\vert^{p}\exp\left(\left\vert  
Y^{N}_{t}\right\vert^{2}/2\right)\right] \\
\le
C_{b_{0},M,p}\exp\left(-C'_{b_{0},\epsilon,M}\cos\left(\vartheta\right)\left\vert  Y_{0}^{N}\right\vert^{2}\right).
\end{multline}
Finally, given $M>0$, as $b\to 0$, with respect to $Q$, $E_{b,\cdot}$ converges 
uniformly in 
probability  to $E_{0,\cdot}$ on $\left[0,M\right]$.
\end{thm}
\begin{proof}
     As in \cite[eq. (14.6.3)]{Bismut08b}, if $h\in 
    K_{\C}$ then
    \begin{align}\label{eq:priun-7}
&\left\vert  \rho^{E}\left(h\right)\right\vert\le e^{Cd\left(p 1,p 
h\right)},
&\left\vert  \rho^{E}\left(h^{-1}\right)\right\vert\le e^{Cd\left(p 1,p 
h\right)}
\end{align}
Let $Y^{\mathfrak k}_{\cdot} \simeq Y^{N}_{\cdot}$ be as in (\ref{eq:wat2}) 
with $b>0, \vartheta=0$.
Let $k_{b,\cdot}\in K_{\C}$ be the  solution of the 
differential equation,
\begin{align}\label{eq:jar8a}
&\dot k_{b}=-i\frac{Y^{\mathfrak k}}{b},
&k_{b,0}=1.
\end{align}
Let $k_{0,\cdot}\in K_{\C}$ be the solution of the stochastic differential 
equation,
\begin{align}\label{eq:jar9}
&\dot k_{0,\cdot}=-i\dot w^{\mathfrak k},&k_{0,0}=1.
\end{align}
By (\ref{eq:diff1}),    (\ref{eq:roba2}), (\ref{eq:jar8a}), and
(\ref{eq:jar9}), we get
\begin{align}\label{eq:jar9a}
&E_{b,\cdot}=\rho^{E}k_{b,\cdot},
&E_{0,\cdot}=\rho^{E}k_{0,\cdot}.
\end{align}
Equation (\ref{eq:prin-6})  follows from equations 
(\ref{eq:led19})
in Theorem \ref{Tesc}, and from (\ref{eq:he9}),  
(\ref{eq:priun-7}),  and (\ref{eq:jar9a}). 

Given $b_{0}>0,\epsilon>0 $, there exists $C'_{b_{0},\epsilon}>0$ 
such that if $0<b\le b_{0},
\vartheta\in\left[0,\frac{\pi}{2}\right[, t\ge\epsilon$, then
\begin{equation}\label{eq:jar9b}
1-e^{-c\cos\left(\vartheta\right)t/b^{2}}\ge 
C'_{b_{0},\epsilon}\cos\left(\vartheta\right).  
\end{equation}
By equation (\ref{eq:led19sa1}) in Theorem \ref{Tesc}, by 
(\ref{eq:he9}), (\ref{eq:priun-7}),  (\ref{eq:jar9a}), and 
(\ref{eq:jar9b}),  we get 
(\ref{eq:prin-6z1}). By Theorem \ref{Tconci} and by (\ref{eq:jar9a}), 
we obtain the last statement in our theorem, whose  proof is completed. 
\end{proof}
 \subsection{The process $H_{\bt,\cdot}$}%
\label{subsec:estu}
When $\vartheta> 0$, even when $F$ is trivial, it is no 
longer true that $R^{0}_{\vartheta}\left(Y\right)$ maps forms of 
degree  $0$
into forms of positive degree. More precisely, the spin representation 
$S^{\overline{\mathfrak p}}$ is now coupled to the forms in 
$\Lambda\ac\left(\mathfrak p^{*} \oplus \mathfrak k^{*}\right)$, so that the product 
technique used in \cite{Bismut08b} does not apply any more. We will 
be forced to suitably combine the two sorts of techniques used in 
\cite{Bismut08b}.

We use the same notation as in Definition \ref{RtY}. Here
\index{P@$\mathbf{P}$}%
$\mathbf{P}$ denotes the orthogonal projection from 
$\Lambda\ac\left(\TsX \oplus N^{*}\right) \otimes 
S^{\overline{TX}}$ on $S^{\overline{TX}} $. Set
\index{Pp@$\mathbf{P}^{\perp}$}%
\begin{equation}\label{eq:depp}
\mathbf{P}^{\perp}=1-\mathbf{P}.
\end{equation}
\begin{defin}\label{DRp}
Put
\index{Rt@$R_{\vartheta}^{0,\perp}\left(Y\right)$}%
\begin{equation}\label{eq:nea2}
R_{\vartheta}^{0,\perp}\left(Y\right)=R^{0}_{\vartheta}\left(Y\right)+\cos^{5/2}\left(\vartheta\right)
\widehat{c} \left(i \ad\left(Y^{N}\right) \vert_{\overline{TX}}\right) .
\end{equation}
\end{defin}

By Proposition \ref{Pproj}, we get
\begin{equation}\label{eq:nea2x1}
\mathbf{P}R_{\vartheta}^{0,\perp}\left(Y\right)\mathbf{P}=0.
\end{equation}

We still consider the stochastic differential equation 
(\ref{eq:wat2}), i.e., we work under the probability measure $Q$. Let 
\index{Hbt@$H_{b,\vartheta,\cdot}$}%
$H_{b,\vartheta,\cdot}$ be the solution of the differential equation,
\begin{align}\label{eq:nea3}
&\frac{dH_{b,\vartheta}}{ds}=H_{b,\vartheta}\frac{\cos^{5/2}\left(\vartheta\right)}{b}\widehat{c}\left(-i\ad
\left(Y^{N}\right)\vert_{\overline{TX}}\right),
&H_{b,\vartheta,0}=1.
\end{align}
By proceeding as in  the proof of Proposition \ref{Pcru}, we have the 
trivial bound,
\begin{equation}\label{eq:nea3pr1}
\left\vert  
H_{b,\vartheta,t}\right\vert\le\exp\left(C\frac{\cos^{5/2}\left(\vartheta\right)}{b}
\int_{0}^{t}\left\vert  Y^{N}_{s}\right\vert ds\right).
\end{equation}

\begin{defin}\label{DHsp}
    Let 
   \index{hbt@$h_{b,\vartheta,\cdot}$}%
    $h_{b,\vartheta,\cdot}\in K_{\C}$ be the solution of the 
	differential equation,
\begin{align}\label{eq:nea3y-1}
&\dot 
h_{b,\vartheta}=-\frac{\cos^{5/2}\left(\vartheta\right)}{b}iY^{\mathfrak k},
&h_{b,\vartheta,0}=1.
\end{align}
Set
\index{zbt@$z_{b,\vartheta,\cdot}$}%
\begin{equation}\label{eq:hope4}
z_{b,\vartheta,\cdot}=p h_{b,\vartheta,\cdot}.
\end{equation}
\end{defin}

Then $H_{b,\vartheta,\cdot}$ is the image of $h_{b,\vartheta,\cdot}$
by the spin representation.
By \cite[eq. (14.6.3)]{Bismut08b}, as in (\ref{eq:priun-7}), we get
\begin{align}\label{eq:mai1}
&\left\vert  H_{\bt,\cdot}\right\vert\le e^{Cd\left(p
1,z_{\bt,\cdot}\right)},
&\left\vert  H_{\bt,\cdot}^{-1}\right\vert\le e^{Cd\left(p
1,z_{\bt,\cdot}\right)}.
\end{align}

 By 
(\ref{eq:nea3y-1}),  we get
\begin{align}\label{eq:hope5x1}
&\dot 
z_{b,\vartheta}=-\frac{\cos^{5/2}\left(\vartheta\right)}{b}iY^{N},&z_{b,\vartheta,0}
=p 1_{K_{\C}}.
\end{align}
By (\ref{eq:wat2}), (\ref{eq:hope5x1}), $z_{b,\vartheta,\cdot}$ is a solution of the differential equation,
\begin{align}\label{eq:hope5xy1}
&b^{2}\ddot z_{b,\vartheta}+\cos\left(\vartheta\right)\dot 
z_{b,\vartheta}=-\cos^{3}\left(\vartheta\right)i\dot w^{N},\\
&z_{b,\vartheta,0}=p 1_{K_{\C}}, \qquad\dot 
z_{b,\vartheta,0}=-\frac{\cos^{5/2}\left(\vartheta\right)}{b}iY^{N}_{0}. \nonumber 
 \end{align}

We denote by 
\index{XKC@$\mathcal{X}_{K_{\C}}$}%
$\mathcal{X}_{K_{\C}}$ the analogue of 
$\mathcal{X}$ for $K_{\C}$. Namely, $\mathcal{X}_{K_{\C}}$ is the 
total space of $TX_{K_{\C}}$. We define the scalar differential operators 
\index{AXKC@$\mathcal{A}^{X_{K_{\C}}}_{b}$}%
\index{BXKC@$\mathcal{B}^{X_{K_{\C}}}_{b}$}%
$\mathcal{A}^{X_{K_{\C}}}_{b},\mathcal{B}^{X_{K_{\C}}}_{b}$ acting on 
 $\mathcal{X}_{K_{\C}}$ as in (\ref{eq:glab5}). One verifies 
easily that the operator on $\mathcal{X}_{K_{\C}} $ associated with the process 
$\left(z_{\bt,\cdot},Y^{N}_{\cdot}\right)$ is given by 
$\cos^{4}\left(\vartheta\right)\mathcal{B}^{X_{K_{\C}}}_{\cos^{3/2}\left(\vartheta\right)b}$.

To make our notation more transparent, we will briefly note the dependence of 
$Y^{N}_{\cdot}$ on $\bt$ explicitly. Put
\begin{align}\label{eq:sup1}
&z_{b,\cdot}=z_{b,0,\cdot},&Y^{N}_{b,\cdot}=Y^{N}_{b,0,\cdot}.
\end{align}
 By the above 
considerations, it follows that the probability law of 
$\left(z_{b,\vartheta,\cdot},Y^{N}_{\bt,\cdot}\right)$ is the same as the 
probability law of 
$\left(z_{\cos^{3/2}\left(\vartheta\right)b,\cos^{4}\left(\vartheta\right)\cdot},Y^{N}_{\cos^{3/2}\left(\vartheta\right)b,\cos^{4}\left(\vartheta\right)\cdot}\right)$.

We omit again the subscript $\bt$ in $Y^{N}$. We  proceed as in Theorem \ref{Tfueq}.
Let $f:X_{K_{\C}}\to \R$ be a smooth function. 
Set
\begin{equation}\label{eq:aca1x1}
B_{t}^{f}=\cos^{2}\left(\vartheta\right)\int_{0}^{t}
\n^{TX_{K_{\C}}}_{iY^{N}_{s}}\n_{iY^{N}_{s}}
f\left(z_{b,\vartheta,s}\right) ds +\int_{0}^{t}
\n_{-i\delta w^{N}_{s}} f\left(z_{b,\vartheta,s}\right).
\end{equation}

By (\ref{eq:wat2}), (\ref{eq:hope5x1}), we obtain an analogue of 
equation (\ref{eq:fex1}),
\begin{multline}\label{eq:aca3}
\left( b^{2}\frac{d}{dt}+\cos\left(\vartheta\right) \right) f\left(z_{b,\vartheta,t}\right)
=\left( b^{2}\frac{d}{dt}+\ct \right) 
f\left(z_{b,\vartheta,t}\right)\vert_{t=0}\\
+\cos^{3}\left(\vartheta\right)B^{f}_{t}.
\end{multline}
From (\ref{eq:aca3}), we get the  analogue of (\ref{eq:fex2}), 
\begin{multline}\label{eq:aca3x1}
f\left(z_{b,\vartheta,t}\right)=f\left(z_{b,\vartheta,0}\right)+
\frac{b^{2}}{\cos\left(\vartheta\right)}\frac{d}{dt}f\left(z_{b,\vartheta,t}\right)\vert_{t=0}
\left(1-e^{-\cos\left(\vartheta\right)t/b^{2}}\right)\\
+\frac{1}{b^{2}}\int_{0}^{t}e^{-\cos\left(\vartheta\right)\left(t-s\right)/b^{2}}\cos^{3}\left(\vartheta\right)B^{f}_{s}ds.
\end{multline}

Let 
\index{H0t@ $H_{0,\vartheta,\cdot}$}%
$H_{0,\vartheta,\cdot}$ be the solution of the stochastic differential 
equation
\begin{align}\label{eq:mir0}
&dH_{0,\vartheta}=H_{0,\vartheta}\cos^{2}\left(\vartheta\right)\widehat{c}\left(-i\ad\left(
dw^{\mathfrak k}\right)\vert_{\overline{TX}}\right),
&H_{0,\vartheta,0}=1.
\end{align}
Let $h_{0,\vartheta,\cdot}\in K_{\C}$ be the solution of
\begin{align}\label{eq:miro1}
&\dot h_{0,\cdot}=-\cos^{2}\left(\vartheta\right)i\dot w^{\mathfrak 
k},&h_{0,\vartheta,0}=p1_{K_{\C}}.
\end{align}
Then $H_{0,\cdot}$ is the image of $h_{0,\vartheta,\cdot}$ via the spin 
representation. If $z_{0,\vartheta\cdot}=ph_{0,\vartheta,\cdot}$, then
\begin{align}\label{eq:miro2}
&\dot z_{0,\vartheta}=-\cos^{2}\left(\vartheta\right)i\dot w^{\mathfrak k},
&z_{0,\vartheta,0}=p1_{K_{\C}}.
\end{align}

\begin{thm}\label{Tpou}
Given $M>0,p\ge 1$, there exist $C_{p}>0, C'$ such that for 
$b>0,\vartheta\in\left[0,\frac{\pi}{2}\right[$,
\begin{align}\label{eq:mai2}
&E^{Q}\left[\sup_{0\le t\le M}\left\vert  
H_{\bt,t}\right\vert^{p}\right]\le C_{p}\exp\left(\left\vert  
Y_{0}^{N}\right\vert^{2}/2\right),\\
&E^{Q}\left[\sup_{0\le t\le M}\left\vert  
H_{\bt,t}^{-1}\right\vert^{p}\right]\le C_{p}\exp\left(\left\vert  
Y_{0}^{N}\right\vert^{2}/2\right). \nonumber 
\end{align}
Given $b_{0}>0,0<\epsilon\le M,p\ge 1$, there exist 
$C_{p,M}>0,C'_{b_{0},\epsilon,M}>0$ such that for $0<b\le 
b_{0},\vartheta\in\left[0,\frac{\pi}{2}\right[,\epsilon\le t\le M$, 
then
\begin{align}\label{eq:mai3}
&\exp\left(-\left\vert  
Y^{N}_{0}\right\vert^{2}/2\right)E^{Q}\left[\sup_{0\le t\le 
M}\left\vert  H_{\bt,t}\right\vert^{p}\exp\left(\left\vert  
Y^{N}_{t}\right\vert^{2}/2\right)\right] \nonumber \\
&\le 
C_{p,M}\exp\left(-C'_{b_{0},\epsilon,M}\cos\left(\vartheta\right)\left\vert  Y_{0}^{N}\right\vert^{2}\right),\\
&\exp\left(-\left\vert  
Y^{N}_{0}\right\vert^{2}/2\right)E^{Q}\left[\sup_{0\le t\le 
M}\left\vert  H^{-1}_{\bt,t}\right\vert^{p}\exp\left(\left\vert  
Y^{N}_{t}\right\vert^{2}/2\right)\right] \nonumber \\
&\le 
C_{b_{0},p,M}\exp\left(-C'_{b_{0},\epsilon,M}\cos\left(\vartheta\right)\left\vert  Y_{0}^{N}\right\vert^{2}\right). \nonumber 
\end{align}
Given $\vartheta\in\left[0,\frac{\pi}{2}\right[$, as $b\to 0$, 
$H_{\bt,\cdot}$ converges to $H_{0,\vartheta,\cdot}$ uniformly on 
$\left[0,M\right]$ in probability.
\end{thm}
\begin{proof}
Using the considerations we made after (\ref{eq:hope5xy1}), 
(\ref{eq:sup1}), equation
  (\ref{eq:mai2}) follows from equation (\ref{eq:led19}) in Theorem 
\ref{Tesc} and from (\ref{eq:mai1}). Also note that
\begin{equation}\label{eq:mai4}
\cos^{4}\left(\vartheta\right)/\left(\cos^{3/2}\left(\vartheta\right)b\right)^{2}=\cos\left(\vartheta\right)/b^{2}.
\end{equation}
Equation (\ref{eq:mai3}) now follows from equation (\ref{eq:led19sa1}) in 
Theorem \ref{Tesc}  and from (\ref{eq:jar9b}), (\ref{eq:mai1}) and (\ref{eq:mai4}). The last part of our 
theorem is a consequence of Theorem \ref{Tconci}.
\end{proof}
\subsection{The process $U^{0}_{\bt,\cdot}$ as an infinite series}%
\label{subsec:infser}
Set
\begin{equation}\label{eq:nea3x2}
K_{b,\vartheta,s}=H_{b,\vartheta,s}\exp\left(-sN^{\Lambda\ac\left(T^{*}X \oplus N^{*}\right)}_{-\vartheta}/b^{2}\right).
\end{equation}
Then 
\index{Kbt@$K_{b,\vartheta,\cdot}$}%
$K_{b,\vartheta,\cdot}$ commutes with $\mathbf{P}$ and $\mathbf{P}^{\perp}$.
\begin{prop}\label{Pexp}
The following identity holds:
\begin{multline}\label{eq:nea5y1}
U^{0}_{b,\vartheta,t}=\sum_{k=0}^{+ \infty }\int_{0\le s_{1}\le s_{2}\ldots 
\le s_{k}\le t}^{}
K_{b,\vartheta,s_{1}}\frac{R^{0,\perp}_{\vartheta}\left(Y_{s_{1}}\right)}{b}K_{b,\vartheta,s_{1}}^{-1}
K_{b,\vartheta,s_{2}}\frac{R^{0,\perp}_{\vartheta}\left(Y_{s_{2}}\right)}{b}K_{b,\vartheta,s_{2}}^{-1} \\
K_{b,\vartheta,s_{3}}
\ldots 
K_{b,\vartheta,s_{k-1}}^{-1}K_{b,\theta,s_{k}}\frac{R^{0,\perp}_{\vartheta}\left(Y_{s_{k}}\right)}{b}
K_{b,\vartheta,s_{k}}^{-1}K_{b,\vartheta,t}
ds_{1}\ldots ds_{k},
\end{multline}
and the series in (\ref{eq:nea5y1}) is uniformly convergent over 
compact sets in $\R_{+}$.
\end{prop}
\begin{proof}
Put
\begin{equation}\label{eq:nea4}
L_{b,\vartheta,\cdot}=U^{0}_{b,\vartheta,\cdot}K_{b,\vartheta,\cdot}^{-1}.
\end{equation}
By (\ref{eq:diff0}), (\ref{eq:nea2}), (\ref{eq:nea3}), and 
(\ref{eq:nea4}), 
 $L_{b,\vartheta,\cdot}$ is the solution of the differential 
equation,
\begin{align}\label{eq:nea5}
&\frac{dL_{b,\vartheta}}{ds}=L_{b,\vartheta}
\frac{K_{b,\vartheta}R^{0,\perp}_{\vartheta}\left(Y\right)K_{b,\vartheta}^{-1}}{b},
&L_{b,\vartheta,0}=1.
\end{align}
We can express $L_{b,\vartheta,\cdot}$ as the sum of iterated integrals
which converges uniformly over compact subsets of $\R_{+}$. 
Using (\ref{eq:nea4}),  we get (\ref{eq:nea5y1}).
\end{proof}
\subsection{A uniform estimate on $\sup_{0\le s\le 
t}E^{Q}\left[\left\vert  U^{0}_{\bt,s}\right\vert^{p}\right]$}%
\label{subsec:cruest}
Now we establish the following key estimate.
\begin{thm}\label{Tfuest}
    Given $p\ge 1, M>0$, there exist $C>0,C'>0$ such that for $0<b\le 
1,\vartheta\in\left[0,\frac{\pi}{2}\right[$, then
\begin{equation}\label{eq:couac-1}
\sup_{0\le t\le M}E^{Q}\left[\left\vert  
U^{0}_{b,\vartheta,t}\right\vert^{p}\right]\le C\exp\left(C' \left(\cos\left(\vartheta\right) \left\vert  
Y_{0}^{TX}\right\vert^{2}+\left\vert  Y_{0}^{N}\right\vert^{2} \right) \right).
\end{equation}
Given $p\ge 1,0<\epsilon\le M<+ \infty$, there exist $C>0, C'>0$ such that for 
$0<b\le 1, \vartheta\in\left[0,\frac{\pi}{2}\right[, \epsilon\le t\le M$,
\begin{multline}\label{eq:couac-1y1}
\exp\left(-\left\vert  Y_{0}\right\vert^{2}/2\right)
E^{Q}\left[\left\vert  
U^{0}_{b,\vartheta,t}\right\vert^{p}\exp\left(\left\vert  Y_{t}\right\vert
^{2}/2\right)\right] \\
\le C\exp\left(-C' \left( \left\vert  
Y_{0}^{TX}\right\vert^{2}+\cos\left(\vartheta\right)\left\vert  
Y^{N}_{0}\right\vert^{2} \right) \right).
\end{multline}
\end{thm}
\begin{proof}
    As we saw in Remark \ref{Rintel}, equation (\ref{eq:couac-1}) 
    is trivial when $b$ stays away from $0$. In the sequel, we may as 
    well assume that $b_{0}>0$ is small enough and that $0<b\le 
    b_{0}$.
    
Note that
\begin{equation}\label{eq:nea6}
R_{\vartheta}^{0,\perp}\left(Y\right)=\left(\mathbf{P} + 
\mathbf{P}^{\perp}\right)R^{0,\perp}_{\vartheta}\left(Y\right)\left(\mathbf{P}+\mathbf{P}^{\perp}\right).
\end{equation}
By (\ref{eq:nea2x1}), (\ref{eq:nea6}), we get
\begin{equation}\label{eq:nea6y1}
R_{\vartheta}^{0,\perp}\left(Y\right)=\mathbf{P} 
R_{\vartheta}^{0,\perp}\left(Y\right)\mathbf{P}^{\perp}+\mathbf{P}^{\perp}R_{\vartheta}^{0,\perp}\left(Y\right) \mathbf{P}+
\mathbf{P}^{\perp}R_{\vartheta}^{0,\perp}\left(Y\right) \mathbf{P}^{\perp}.
\end{equation}

Consider the splitting
\begin{equation}\label{eq:nea6ya1}
\Lambda\ac\left( TX^{*} \oplus N^{*}\right)\ho 
\widehat{c}\left(\overline{TX}\right)=\widehat{c}\left(\overline{TX}\right)
\oplus \left( \Lambda^{(>0)}
\left(T^{*}X \oplus 
N^{*}\right)\ho\widehat{c}\left(\overline{TX}\right) \right) .
\end{equation}
As a subalgebra of 
$\End\left(\Lambda\ac\left(\overline{T^{*}X}\right)\right)$, 
$\widehat{c}\left(\overline{TX}\right)$ inherits a corresponding 
norm, and $\End\left(\Lambda\ac\left(T^{*}X \oplus 
N^{*}\right)\right)$ inherits a  norm from the Euclidean norm of 
$\Lambda\ac\left(T^{*}X \oplus N^{*}\right)$.

We will write elements of $\End \left( \Lambda\ac\left(T^{*}X \oplus 
N^{*}\right) \right) \ho\widehat{c}\left(\overline{TX}\right)$ as 
$\left(2,2\right)$ matrices with respect to the splitting 
(\ref{eq:nea6ya1}).  In particular the norm of such an element is not 
a nonnegative number but a $\left(2,2\right)$ matrix of nonnegative 
numbers. Such a norm is still such that 
\begin{equation}\label{eq:nea6ya2}
\left\vert  AB\right\vert\le \left\vert  A\right\vert\left\vert  
B\right\vert.
\end{equation}
The inequality just means that each of the $4$ coefficients in the 
left-hand side is dominated by the corresponding coefficient in the 
right-hand side.

For $0\le s_{1}\le\ldots\le s_{k}\le t$, set
\begin{multline}\label{eq:hope1}
R_{s_{1},\ldots,s_{k},t}=
\left\vert  H_{b,\vartheta,s_{1}}\right\vert \left\vert  
\exp\left(-s_{1}N^{\Lambda\ac\left(T^{*}X \oplus N^{*}\right)}_{-\vartheta}/b^{2}\right)\right\vert
\left\vert
\frac{R^{0,\perp}_{\vartheta}\left(Y_{s_{1}}\right)}{b}\right\vert 
\\
\left\vert  H_{b,\vartheta,s_{1}}^{-1}
H_{b,\vartheta,s_{2}}\right\vert  
\left\vert  
\exp\left(-\left(s_{2}-s_{1}\right)N^{\Lambda\ac\left(T^{*}X \oplus N^{*}\right)}_{-\vartheta}/b^{2}\right)
\right\vert 
\left\vert  
\frac{R^{0,\perp}_{\vartheta}\left(Y_{s_{2}}\right)}{b}\right\vert 
\ldots\\
\ldots 
\left\vert  H_{b,\vartheta,s_{k-1}}^{-1}
H_{b,\vartheta,s_{k}}\right\vert  
\left\vert  
\exp\left(-\left(s_{k}-s_{k-1}\right)N^{\Lambda\ac\left(T^{*}X \oplus N^{*}\right)}_{-\vartheta}/b^{2}\right)\right\vert 
\left\vert  
\frac{R^{0,\perp}_{\vartheta}\left(Y_{s_{k}}\right)}{b}\right\vert\\
\left\vert  H_{b,\vartheta,s_{k}}^{-1}H_{b,\vartheta,t}\right\vertﬁ
\left\vert  
\exp\left(-\left(t-s_{k}\right)N^{\Lambda\ac\left(T^{*}X \oplus N^{*}\right)}_{-\vartheta}/b^{2}\right)\right\vert.
\end{multline}
By (\ref{eq:nea3x2}),
(\ref{eq:nea5y1}), and (\ref{eq:hope1}), we get
\begin{equation}\label{eq:hope1y1}
\left\vert  U^{0}_{b,\vartheta,t}\right\vert\le 
\sum_{k=0}^{+ \infty }\int_{0\le s_{1}\le\ldots\le s_{k}\le 
t}^{}R_{s_{1},\ldots,s_{k},t}ds_{1}ds_{2}\ldots ds_{k}.
\end{equation}

We fix temporarily $s_{1}\ge 0$. By \cite[eq. (14.6.3)]{Bismut08b} or 
by (\ref{eq:mai1}), there exists 
$C>0$ such that
for 
$s_{1}\le s_{2}$,  
\begin{equation}\label{eq:rot1}
\left\vert  H_{b,\vartheta,s_{1}}^{-1}H_{b,\vartheta,s_{2}}\right\vert\le \exp\left(
Cd\left(z_{b,\vartheta,s_{1}},z_{b,\vartheta,s_{2}}\right)\right).
\end{equation}

Let $k:\R_{+}\to \R_{+}$ be a smooth increasing  function such that 
\begin{align}\label{eq:glab41a}
    k\left(u\right)=&\ 0\ \mathrm{for} \ u\le 
1/2,\\
=&\ u\ \mathrm{for}\ u\ \ge 1. \nonumber 
\end{align}
For $y,z\in X_{K_{\C}}$, set
\begin{equation}
    f\left(y,z\right)=k\left(d\left(y,z\right)\right).
    \label{eq:glab42a}
\end{equation}
Then $f\left(y,z\right)$ is a smooth function such that
\begin{equation}\label{eq:glab42z1}
d\left(y,z\right)\le f\left(y,z\right)+1.
\end{equation}
Moreover, if $d\left(y,z\right)\le 1/2$, then $f\left(y,z\right)=0$. In the sequel $\n f\left(y,z\right) $ denotes the gradient of $f$ with 
respect to the second variable $z$. Note that $f$ and $\n f$ vanish on the 
diagonal. 

By (\ref{eq:rot1}), (\ref{eq:glab42z1}), for $0\le s_{1}\le s¨_2$, we get
\begin{equation}\label{eq:rot2}
\left\vert  H_{b,\vartheta,s_{1}}^{-1}H_{b,\vartheta,s_{2}}\right\vert\le
\exp\left(C+Cf\left(z_{b,\vartheta,s_{1}},z_{b,\vartheta,s_{2}}\right)\right).
\end{equation}

Now we use the formulas in (\ref{eq:aca1x1})--(\ref{eq:aca3x1}). As in 
(\ref{eq:aca1x1}), for $s\ge s_{1}$, set
\begin{equation}\label{eq:aca7}
B^{f}_{s_{1},s}=\cos^{2}\left(\vartheta\right)\int_{s_{1}}^{s}
\n_{iY^{N}_{u}}^{TX_{K_{\C}}}\n_{iY^{N}_{u}}
f\left(z_{b,\vartheta,s_{1}},z_{b,\vartheta,u}\right)du+\int_{s_{1}}^{s}\n_{-i\delta 
w^{N}_{u}}f\left(z_{b,\vartheta,s_{1}},z_{b,\vartheta,u}\right).
\end{equation}
By (\ref{eq:aca3x1}), since $f\left(z_{b,\vartheta,s_{1}},\cdot\right)$ vanishes 
near $z_{b,\vartheta,s_{1}}$,  we get
\begin{equation}\label{eq:aca8}
f\left(z_{b,\vartheta,s_{1}},z_{b,\vartheta,s_{2}}\right)=
\frac{1}{b^{2}}\int_{s_{1}}^{s_{2}}e^{-\cos\left(\vartheta\right)
\left(s_{2}-s\right)/b^{2}}\cos^{3}\left(\vartheta\right)B^{f}_{s_{1},s}ds.
\end{equation}

By \cite[Proposition 13.1.2]{Bismut08b}, $f\left(z_{1},z_{2}\right)$ 
and its first and second covariant derivatives in the variable 
$z_{2}$ are uniformly bounded. In particular,
\begin{align}\label{eq:aca8y}
    &\left\vert  \n_{-iY^{N}}f\left(z_{1},z_{2}\right)\right\vert\le 
    c\left\vert  Y^{N}\right\vert,
&\n^{TX_{K_{\C}}}_{iY^{N}}\n_{iY^{N}} f\left(z_{1},z_{2}\right)\le 
c\left\vert  Y^{N}\right\vert^{2}.
\end{align}

For $s\ge s_{1}$, put
\begin{equation}\label{eq:aca8z}
M^{f}_{s_{1},s}=\int_{s_{1}}^{s}\n_{-i\delta 
w^{N}_{u}}f\left(z_{\bt,s_{1}},z_{\bt,u}\right).
\end{equation}
By (\ref{eq:aca7}), (\ref{eq:aca8y}), and (\ref{eq:aca8z}), we get
\begin{equation}\label{eq:aca8za1}
B^{f}_{s_{1},s}\le c\cos^{2}\left(\vartheta\right)\int_{s_{1}}^{s}\left\vert  
Y^{N}_{s}\right\vert^{2}ds+M^{f}_{s_{1},s}.
\end{equation}
By (\ref{eq:aca8}), (\ref{eq:aca8za1}), we get
\begin{multline}\label{eq:aca8za2}
f\left(z_{b,\vartheta,s_{1}},z_{b,\vartheta,s_{2}}\right) \\
\le 
c\cos^{4}\left(\vartheta\right)\int_{s_{1}}^{s_{2}}\left\vert  
Y^{N}_{s}\right\vert^{2}ds+
\frac{1}{b^{2}}\int_{s_{1}}^{s_{2}}e^{-\cos\left(\vartheta\right)
\left(s_{2}-s\right)/b^{2}}\cos^{3}\left(\vartheta\right)M^{f}_{s_{1},s}ds.
\end{multline}
Also using (\ref{eq:aca8z}) and integration by parts, we obtain
\begin{multline}\label{eq:nab0}
\frac{1}{b^{2}}\int_{s_{1}}^{s_{2}}e^{-\cos\left(\vartheta\right)
\left(s_{2}-s\right)/b^{2}}\cos^{3}\left(\vartheta\right)M^{f}_{s_{1},s}ds\\
=
\int_{s_{1}}^{s_{2}}\cos^{2}\left(\vartheta\right)\left(1-e^{-\cos\left(\vartheta\right)\left(s_{2}-s\right)/b^{2}}\right)
\n_{-i\delta w^{N}_{s}}f\left(z_{b,\vartheta,s_{1}},z_{b,\vartheta,s}\right).
\end{multline}

By (\ref{eq:rot2}), (\ref{eq:aca8za2}), and (\ref{eq:nab0}),  we obtain
\begin{multline}\label{eq:aca10}
\left\vert  H_{b,\vartheta,s_{1}}^{-1}H_{b,\vartheta,s_{2}}\right\vert\le 
\exp\left(C+cC\cos^{4}\left(\vartheta\right)\int_{s_{1}}^{s_{2}}\left\vert  Y^{N}_{s}
\right\vert^{2}ds \right) \\
\exp \left( C
\int_{s_{1}}^{s_{2}}\cos^{2}\left(\vartheta\right)\left(1-e^{-\cos\left(\vartheta\right)\left(s_{2}-s\right)/b^{2}}\right)
\n_{-i\delta w^{N}_{s}}f\left(z_{b,\vartheta,s_{1}},z_{b,\vartheta,s}\right)\right).
\end{multline}

We fix $s_{1},s_{2}$ with $0\le s_{1}\le s_{2}$. For $s\ge s_{1}$, set
\begin{multline}\label{eq:nab1}
N^{f}_{s_{1},s_{2},s}=\exp\Biggl(C
\int_{s_{1}}^{s}\cos^{2}\left(\vartheta\right)\left(1-e^{-\cos\left(\vartheta\right)\left(s_{2}-u\right)/b^{2}}\right)
\n_{-i\delta 
w^{N}_{u}}f\left(z_{b,\vartheta,s_{1}},z_{b,\vartheta,u}\right)\\
-\frac{C^{2}}{2}\int_{s_{1}}^{s}\cos^{4}\left(\vartheta\right)
\left(1-e^{-\cos\left(\vartheta\right)\left(s_{2}-u\right)/b^{2}}\right)^{2}
\left\vert  
\n_{\cdot}f\left(z_{b,\vartheta,s_{1}},z_{b,\vartheta,u}\right)\right\vert^{2}du\Biggr).
\end{multline}
Then $N^{f}_{s_{1},s_{2},\cdot}$ is the unique solution of the It\^{o} stochastic 
differential equation,
\begin{multline}\label{eq:nab2}
N^{f}_{s_{1},s_{2},s}=1+\int_{s_{1}}^{s}N^{f}_{s_{1},s_{2},u}C\cos^{2}\left(\vartheta\right)\left(1-e^{-\cos\left(\vartheta\right)\left(s_{2}-u\right)/b^{2}}
\right)\\
\n_{-i\delta 
w^{N}_{u}}f\left(z_{b,\vartheta,s_{1}},z_{b,\vartheta,u}\right).
\end{multline}
In particular $N^{f}_{s_{1},s_{2},\cdot}$ is a martingale. In the sequel, we 
will write $N^{f}_{s_{1},s_{2}}$ instead of $N^{f}_{s_{1},s_{2},s_{2}}$. 
By (\ref{eq:aca8y}), (\ref{eq:aca10}), and (\ref{eq:nab1}), we get
\begin{multline}\label{eq:nab3}
\left\vert  
H_{b,\vartheta,s_{1}}^{-1}H_{b,\vartheta,s_{2}}\right\vert \\
\le
\exp\left(C+cC\cos^{4}\left(\vartheta\right)\int_{s_{1}}^{s_{2}}\left\vert  Y^{N}_{s}\right\vert^{2}ds+\frac{1}{2}
C^{2}c^{2}\left(s_{2}-s_{1}\right)\right)N^{f}_{s_{1},s_{2}}.
\end{multline}

For $s_{1}\le s\le s_{2}$, set
\begin{equation}\label{eq:nab4}
\overline{w}^{N}_{s}=w^{N}_{s}-\int_{s_{1}}^{s}C\cos^{2}\left(\vartheta\right)
\left(1-e^{-\cos\left(\vartheta\right)\left(s_{2}-u\right)/b^{2}}\right)\n_{-i\cdot}
f\left(z_{b,\vartheta,s_{1}},z_{b,\vartheta,u}\right)du.
\end{equation}
Using the properties of the Girsanov transformation, we know that 
with respect to  the probability measure $dQ'=N^{f}_{s_{1},s_{2}}dQ$,  $\overline{w}^{N}_{\cdot}$ is a 
Brownian motion on $\left[s_{1},s_{2}\right]$. By (\ref{eq:wat2}),  
for $s_{1}\le s\le s_{2}$, 
we get
\begin{equation}\label{eq:nab5}
Y^{N}_{s}=e^{-\cos\left(\vartheta\right)\left(s-s_{1} \right)   
/b^{2}}Y^{N}_{s_{1}}+\frac{\cos^{1/2}\left(\vartheta\right)}{b}\int_{s_{1}}^{s}e^{-\cos \left( \vartheta\right)\left(s-
u\right)/b^{2}}dw^{N}_{u}.
\end{equation}
By (\ref{eq:nab4}), we can rewrite (\ref{eq:nab5}) in the form,
\begin{multline}\label{eq:nab6}
Y^{N}_{s}=e^{-\cos\left(\vartheta\right)\left(s-s_{1} \right)   
/b^{2}}Y^{N}_{s_{1}}+\frac{\cos^{1/2}\left(\vartheta\right)}{b}
\int_{s_{1}}^{s}e^{-\cos \left( \vartheta\right)\left(s-u\right)/b^{2}}
d\overline{w}^{N}_{u}\\
+\frac{\cos^{1/2}\left(\vartheta\right)}{b}\int_{s_{1}}^{s}e^{-\cos\left(\vartheta\right)\left(s-u\right)/b^{2}}C
\cos^{2}\left(\vartheta\right)\left(1-e^{-\cos\left(\vartheta\right)\left(s_{2}-u\right)/b^{2}}\right) \\
\n_{-i\cdot}f\left(z_{b,\vartheta,s_{1}},z_{b,\vartheta,u}\right)du.
\end{multline}

By (\ref{eq:aca8y}), for $s_{1}\le s\le s_{2}$,  we get
\begin{multline}\label{eq:nab7}
\Biggl\vert\frac{\cos^{1/2}\left(\vartheta\right)}{b}\int_{s_{1}}^{s}e^{-\cos\left(\vartheta\right)\left(s-u\right)/b^{2}}C
\cos^{2}\left(\vartheta\right)\left(1-e^{-\cos\left(\vartheta\right)\left(s_{2}-u\right)/b^{2}}\right)\\
\n_{-i\cdot}f\left(z_{b,\vartheta,s_{1}},z_{b,\vartheta,u}\right)du\Biggr\vert
\le cC\cos^{3/2}\left(\vartheta\right)b.
\end{multline}

Set
\begin{equation}\label{eq:nab8}
\overline{Y}^{N}_{s}=e^{-\cos\left(\vartheta\right)\left(s-s_{1}\right)/b^{2}}Y^{N}_{s_{1}}+
\frac{\cos^{1/2}\left(\vartheta\right)}{b}
\int_{s_{1}}^{s}e^{-\cos\left(\vartheta\right)\left(s-u\right)/b^{2}}
d\overline{w}^{N}_{u}.
\end{equation}
Under $Q'$, the process $\overline{Y}^{N}_{\cdot}$ on $\left[s_{1},s_{2}\right]$ has the same probability law as 
$Y^{N}_{\cdot}$ under $Q$. By (\ref{eq:nab6}), (\ref{eq:nab8}), we get
\begin{equation}\label{eq:nab9}
\left\vert  Y^{N}_{s}\right\vert\le \left\vert  
\overline{Y}^{N}_{s}\right\vert+cC\cos^{3/2}\left(\vartheta\right)b.
\end{equation}

By (\ref{eq:hope1}),  
(\ref{eq:nab3}), given $k\in \N$, with the convention that $s_{0}=0, s_{k+1}=t$, we get
\begin{multline}\label{eq:rot4}
R_{s_{1},\ldots,s_{k},t}\le 
\exp \left( 
\left(k+1\right)C+cC\cos^{4}\left(\vartheta\right)\int_{0}^{t}\left\vert  
Y^{N}_{s}\right\vert^{2}ds+\frac{1}{2}C^{2}c^{2}t\right)
\prod_{i=0}^{k}N^{f}_{s_{i},s_{i+1}}
\\
\left\vert  \exp\left(-s_{1}N^{\Lambda\ac\left(T^{*}X \oplus N^{*}\right)}_{-\vartheta}/b^{2}\right)\right\vert
\left\vert  
\frac{R^{0,\perp}_{\vartheta}\left(Y_{s_{1}}\right)}{b}\right\vert 
\left\vert 
\exp\left(-\left(s_{2}-s_{1}\right)N^{\Lambda\ac\left(T^{*}X \oplus 
N^{*}\right)}_{-\vartheta}/b^{2}\right)\right\vert \\
\left\vert  
\frac{R^{0,\perp}_{\vartheta}\left(Y_{s_{2}}\right)}{b}\right\vert 
\ldots \left\vert  
\frac{R^{0,\perp}_{\vartheta}\left(Y_{s_{k}}\right)}{b}\right\vert
\left\vert  \exp\left(-
\left(t-s_{k}\right)N^{\Lambda\ac\left(T^{*}X \oplus N^{*}\right)}_{-\vartheta}/b^{2}\right)\right\vert.
\end{multline}

Again $\prod_{i=0}^{k}N^{f}_{s_{i},s_{i+1}}$ is a Girsanov exponential. 
Instead of (\ref{eq:nab4}), on $\left[0,t\right]$, we define $\overline{w}^{N}_{\cdot}$ by 
the formula
\begin{multline}\label{eq:nab10}
\overline{w}^{N}_{s}=w^{N}_{s}-
\int_{0}^{s}\sum_{i=0}^{k}1_{s_{i}\le u\le s_{i+1}}C\cos^{2}\left(\vartheta\right)
\left(1-e^{-\cos\left(\vartheta\right)\left(s_{i+1}-u\right)/b^{2}}\right) \\
\n_{-i\cdot}
f\left(z_{b,\vartheta,s_{i}},z_{b,\vartheta,u}\right).
\end{multline}
Under the probability law 
$dQ_{s_{1},\ldots,s_{k},t}=\prod_{i=0}^{k}N^{f}_{s_{i},s_{i+1}}dQ$, on 
$\left[0,t\right]$, $\overline{w}^{N}_{\cdot}$ is a Brownian motion.
Instead of (\ref{eq:nab6}), we now have 
\begin{multline}\label{eq:nab11}
Y^{N}_{s}=e^{-\cos\left(\vartheta\right)s/b^{2}}Y^{N}_{0}+\frac{\cos^{1/2}\left(\vartheta\right)}{b}\int_{0}^{s}e^{-\cos\left(\vartheta
\right)\left(s-u\right)}d\overline{w}^{N}_{u}\\
+\frac{\cos^{1/2}\left(\vartheta\right)}{b}\int_{0}^{s}e^{-\cos\left(\vartheta\right)
\left(s-u\right)/b^{2}}\\
\sum_{i=0}^{k}1_{s_{i}\le u\le 
s_{i+1}}C\cos^{2}\left(\vartheta\right)
\left(1-e^{-\cos\left(\vartheta\right)\left(s_{i+1}-u\right)/b^{2}}\right)
\n_{-i\cdot}f\left(z_{b,\vartheta,s_{i}},z_{b,\vartheta,u}\right).
\end{multline}

Instead of (\ref{eq:nab8}), set
\begin{equation}\label{eq:nab12}
\overline{Y}^{N}_{s}=e^{-\cos\left(\vartheta\right)s/b^{2}}Y^{N}_{0}+\frac{\cos^{1/2}\left(\vartheta\right)}{b}
\int_{0}^{s}e^{-\cos\left(\vartheta\right)\left(s-u\right)/b^{2}}d\overline{w}^{N}_{s}.
\end{equation}
Under $Q_{s_{1},\ldots,s_{k},t}$, on $\left[0,t\right]$, $\overline{Y}^{N}_{\cdot}$ as the 
same probability law as $Y^{N}_{\cdot}$ under $Q$.

In the sequel, we denote by $\overline{Y}_{\cdot}$ the process whose 
$TX$ component coincides with 
$Y^{TX}_{\cdot}$, and whose $N$ component coincides with 
$\overline{Y}^{N}_{\cdot}$. It is still true that under 
$Q_{s_{1},\ldots,s_{k},t}$, the 
probability law of $\overline{Y}_{\cdot}$  on $\left[0,t\right]$ is the same as the 
probability law of $Y_{\cdot}$ under $Q$.

Instead of (\ref{eq:nab7}), for $0\le s\le t$, we obtain
\begin{multline}\label{eq:nab13}
\Biggl\vert  
\frac{\cos^{1/2}\left(\vartheta\right)}{b}\int_{0}^{s}e^{-\cos\left(\vartheta\right)\left(s-u\right)/b^{2}}
\Biggl(\sum_{i=0}^{k}1_{s_{i}\le u\le 
s_{i+1}}C\cos^{2}\left(\vartheta\right)\\
\left(1-e^{-\cos\left(\vartheta\right)\left(s_{i+1}-u\right)/b^{2}}\right)
\n_{-i\cdot}f\left(z_{b,\vartheta,s_{i}},z_{b,\vartheta,u}\right)\Biggr)du\Biggr\vert
\le cC\cos^{3/2}\left(\vartheta\right)b.
\end{multline}
By (\ref{eq:nab11})--(\ref{eq:nab13}), as in (\ref{eq:nab9}), we get
\begin{equation}\label{eq:nab14}
\left\vert  Y^{N}_{s}\right\vert\le\left\vert  
\overline{Y}^{N}_{s}\right\vert+cC\cos^{3/2}\left(\vartheta\right)b.
\end{equation}
By (\ref{eq:nab14}), we get
\begin{equation}\label{eq:nab16}
\int_{0}^{t}\left\vert  Y^{N}_{s}\right\vert^{2}ds\le 
2\int_{0}^{t}\left\vert \overline{Y}^{N}_{s} 
\right\vert^{2}ds+2c^{2}C^{2}\cos^{3}\left(\vartheta\right)b^{2}t.
\end{equation}

Let $E^{Q_{s_{1},\ldots,s_{k},t}}$ be the expectation operator with 
respect to $Q_{s_{1},\ldots,s_{k},t}$. By (\ref{eq:rot4}), we obtain
\begin{multline}\label{eq:nab15}
E^{Q}\left[R_{s_{1},\ldots,s_{k},t}\right]\le 
\exp\left(\left(k+1\right)C+\frac{1}{2}C^{2}c^{2}t\right)\\
E^{Q_{s_{1},\ldots,s_{k},t}}\Biggl[\exp\left(cC\cos^{4}\left(\vartheta\right)\int_{0}^{t}\left\vert  Y^{N}_{s}
\right\vert^{2}ds\right) \\
\left\vert  
\exp\left(-s_{1}N^{\Lambda\ac\left(T^{*}X \oplus N^{*}\right)}_{-\vartheta}/b^{2}\right)\right\vert 
\left\vert  
\frac{R^{0,\perp}_{\vartheta}\left(Y_{s_{1}}\right)}{b}\right\vert 
\left\vert  
\exp\left(-\left(s_{2}-s_{1}\right)N^{\Lambda\ac\left(T^{*}X \oplus 
N^{*}\right)}_{-\vartheta}/b^{2}\right)\right\vert \\
\left\vert  
\frac{R^{0,\perp}_{\vartheta}\left(Y_{s_{2}}\right)}{b}\right\vert 
\ldots \left\vert\frac{R^{0,\perp}_{\vartheta}\left(Y_{s_{k}}\right)}{b}\right\vert
\left\vert  \exp\left(-
\left(t-s_{k}\right)N^{\Lambda\ac\left(T^{*}X \oplus 
N^{*}\right)}_{-\vartheta}/b^{2}\right)\right\vert\Biggr].
\end{multline}

Let $K$ be the $\left(2,2\right)$ matrix
\begin{equation}\label{eq:nab17}
K=
\begin{bmatrix}
    0 & 1 \\
    1 & 1
\end{bmatrix}.
\end{equation}
For simplicity, we will write $\R^{2}$ in the form
\begin{equation}\label{eq:nab18ax1}
\R^{2}=\R \oplus \R^{0,\perp},
\end{equation}
so that $K$ acts on $\R \oplus \R^{0,\perp}$.

By (\ref{eq:est0}), (\ref{eq:nea2x1}), and (\ref{eq:nab14}), we get
\begin{equation}\label{eq:nab18}
\left\vert  
R_{\vartheta}^{0,\perp}\left(Y_{s}\right)\right\vert\le\left\vert  
R_{\vartheta}^{0,\perp}\left(\overline{Y}_{s}\right)\right\vert+cCC'\cos^{2}\left(\vartheta\right)bK.
\end{equation}
Using  (\ref{eq:nab16}), (\ref{eq:nab15}),    (\ref{eq:nab18}), 
and the fact that under $Q_{s_{1},\ldots,s_{k},t}$, the probability 
law of $\overline{Y}_{\cdot}$ is the same as the probability law of 
$Y_{\cdot}$ under $Q$,  we get
\begin{multline}\label{eq:nab17a}
E^{Q}\left[R_{s_{1},\ldots,s_{k},t}\right]\le
\exp\left(\left(k+1\right)C+\frac{1}{2}C^{2}c^{2}t+2c^{3}C^{3}\cos^{7}\left(\vartheta\right)b^{2}t\right)\\
E^{Q}\Biggl[\exp\left(2cC\cos^{4}\left(\vartheta\right)\int_{0}^{t}\left\vert 
Y^{N}_{s}\right\vert^{2}ds\right)
\left\vert  
\exp\left(-s_{1}N^{\Lambda\ac\left(T^{*}X \oplus N^{*}\right)}_{-\vartheta}/b^{2}\right)\right\vert \\
\left( \left\vert  
\frac{R^{0,\perp}_{\vartheta}\left(Y_{s_{1}}\right)}{b}\right\vert+cCC'\cos^{2}\left(\vartheta\right)K \right) 
\left\vert  
\exp\left(-\left(s_{2}-s_{1}\right)N^{\Lambda\ac\left(T^{*}X \oplus N^{*}\right)}_{-\vartheta}/b^{2}\right)\right\vert\\
\left( \left\vert  
\frac{R^{0,\perp}_{\vartheta}\left(Y_{s_{2}}\right)}{b}\right\vert 
+cCC'\cos^{2}\left(\vartheta\right)K \right) 
\ldots \left( \left\vert  
\frac{R^{0,\perp}_{\vartheta}\left(Y_{s_{k}}\right)}{b}\right\vert 
+cCC'\cos^{2}\left(\vartheta\right)K \right) \\
\left\vert  \exp\left(-
\left(t-s_{k}\right)N^{\Lambda\ac\left(T^{*}X \oplus N^{*}\right)}_{-\vartheta}/b^{2}\right)\right\vert\Biggr].
\end{multline}

By (\ref{eq:hope1y1}), (\ref{eq:nab17a}), we get
\begin{multline}\label{eq:nab18a}
E^{Q}\left[\left\vert  U^{0}_{b,\vartheta,t}\right\vert\right]\le
\sum_{k=0}^{+ 
\infty}\exp\left(\left(k+1\right)C+\frac{1}{2}C^{2}c^{2}t+2c^{3}C^{3}\cos^{7}\left(\vartheta\right)b^{2}t\right)\\
E^{Q}\Biggl[\int_{0\le s_{1}\ldots s_{k}\le t}^{}\Biggl[\exp\left(2cC\cos^{4}\left(\vartheta\right)\int_{0}^{t}\left\vert  Y^{N}_{s}\right\vert^{2}ds\right)
\left\vert  
\exp\left(-s_{1}N^{\Lambda\ac\left(T^{*}X \oplus N^{*}\right)}_{-\vartheta}/b^{2}\right)\right\vert \\
\left( \left\vert  
\frac{R^{0,\perp}_{\vartheta}\left(Y_{s_{1}}\right)}{b}\right\vert+cCC'\cos^{2}\left(\vartheta\right)K \right) 
\left\vert  
\exp\left(-\left(s_{2}-s_{1}\right)N^{\Lambda\ac\left(T^{*}X \oplus N^{*}\right)}_{-\vartheta}/b^{2}\right)\right\vert\\
\left( \left\vert  
\frac{R^{0,\perp}_{\vartheta}\left(Y_{s_{2}}\right)}{b}\right\vert 
+cCC'\cos^{2}\left(\vartheta\right)K \right) 
\ldots \left( \left\vert  
\frac{R^{0,\perp}_{\vartheta}\left(Y_{s_{k}}\right)}{b}\right\vert 
+cCC'\cos^{2}\left(\vartheta\right)K \right) \\
\left\vert  \exp\left(-
\left(t-s_{k}\right)N^{\Lambda\ac\left(T^{*}X \oplus N^{*}\right)}_{-\vartheta}/b^{2}\right)\right\vert\Biggr]ds_{1}\ldots ds_{k}\Biggr].
\end{multline}

Let 
\index{Vbt@$V_{b,\vartheta,\cdot}$}%
$V_{b,\vartheta,\cdot}$ be the solution of the differential equation,
\begin{align}\label{eq:nab19}
&\frac{dV_{b,\vartheta}}{ds}=V_{b,\vartheta}\left[-
\frac{\left\vert  N^{\Lambda\ac\left(T^{*}X \oplus N^{*}\right)}_{-\vartheta}\right\vert}{b^{2}}+\left\vert e^{C} 
\frac{R^{0,\perp}_{\vartheta}\left(Y_{s}\right)}{b}\right\vert+cCC'e^{C}\cos^{2}\left(\vartheta\right)K\right],\\
&V_{b,\vartheta,0}=1. \nonumber 
\end{align}
In (\ref{eq:nab19}), $V_{b,\vartheta,\cdot}$ is a process of 
$\left(2,2\right)$ matrices.
One verifies easily that equation (\ref{eq:nab18a}) can be rewritten 
in the form,
\begin{multline}\label{eq:nab20}
E^{Q}\left[\left\vert  U^{0}_{b,\vartheta,t}\right\vert\right]\le 
\exp\left(C+\frac{1}{2}C^{2}c^{2}t+2c^{3}C^{3}\cos^{7}\left(\vartheta\right)t\right) \\
E^{Q}\left[
\exp\left(2cC\cos^{4}\left(\vartheta\right)\int_{0}^{t}\left\vert  
Y_{s}^{N}\right\vert^{2}ds\right)V_{b,\vartheta,t}\right].
\end{multline}

We   will estimate $V_{b,\vartheta,\cdot}$ by the method used in 
\cite[Theorem 14.5.2]{Bismut08b} and in the proof of Proposition 
\ref{Pcru}.  Put
\begin{equation}\label{eq:nab21}
M_{b,\vartheta}=-
\frac{\left\vert  N^{\Lambda\ac\left(T^{*}X \oplus N^{*}\right)}_{-\vartheta}\right\vert}{b^{2}}+\left\vert e^{C} 
\frac{R^{0,\perp}_{\vartheta}\left(Y\right)}{b}\right\vert+cCC'e^{C}\cos^{2}\left(\vartheta\right)K.
\end{equation}
Then $M_{b,\vartheta}$ is a self-adjoint matrix.  Equation 
(\ref{eq:nab19}) can be written in the form,
\begin{align}\label{eq:nab20x1}
&\frac{dV_{b,\vartheta}}{ds}=V_{b,\vartheta,s}M_{b,\vartheta},&V_{b,\vartheta,0}
=1.
\end{align}

Since the component 
of $\left\vert  R_{\vartheta}^{0,\perp}\left(Y\right)\right\vert$  
mapping $\R$ into $\R$ vanishes identically, using (\ref{eq:est0}), if 
$f\in\R^{2}$, then
\begin{equation}\label{eq:nab22}
\left\vert\left\langle  \left\vert  
R_{\vartheta}^{0,\perp}\left(Y\right)\right\vert 
f,f\right\rangle\right\vert\le C'\left( \cos^{1/2}\left(\vartheta\right)\left\vert  
Y^{TX}\right\vert +\left\vert  Y^{N}\right\vert\right) \left\vert  f\right\vert\left\vert \sqrt{ \left\vert  
N^{\Lambda\ac\left(T^{*}X \oplus N^{*}\right)}_{-\vartheta}\right\vert} f\right\vert.
\end{equation}
For  $\eta>0$, we get
\begin{multline}\label{eq:nab23}
\frac{1}{b}\left(\cos^{1/2}\left(\vartheta\right)\left\vert  
Y^{TX}\right\vert+\left\vert  Y^{N}\right\vert\right)\left\vert  f\right\vert\left\vert \sqrt{ \left\vert  
N^{\Lambda\ac\left(T^{*}X \oplus N^{*}\right)}_{-\vartheta}\right\vert} 
f\right\vert \\
\le\frac{1}{2}\Biggl(\frac{\eta}{b^{2}}\left\langle  
\left\vert N^{\Lambda\ac\left(T^{*}X \oplus N^{*}\right)}_{-\vartheta}\right\vert f,f\right\rangle
+
\frac{2}{\eta} \left( \cos\left(\vartheta\right)
\left\vert  Y^{TX}\right\vert^{2}+\left\vert  Y^{N}\right\vert^{2} \right) 
\left\vert  f\right\vert^{2}\Biggr).
\end{multline}
By (\ref{eq:nab21})--(\ref{eq:nab23}), we obtain
\begin{multline}\label{eq:nab24}
\left\langle  M_{b,\vartheta}f,f\right\rangle\le 
-\frac{1}{b^{2}}\left\langle  \left\vert  
N^{\Lambda\ac\left(T^{*}X \oplus N^{*}\right)}_{-\vartheta}f,f\right\vert\right\rangle\\
+e^{C}C'\frac{1}{2}\left(\frac{\eta}{b^{2}}\left(\langle  
N^{\Lambda\ac\left(T^{*}X \oplus N^{*}\right)}_{-\vartheta}f,f\right\rangle
+
\frac{2}{\eta}\left(\cos\left(\vartheta\right)\left\vert  
Y^{TX}\right\vert^{2}+\left\vert  Y^{N}\right\vert^{2}\right) 
\right)\left\vert  f\right\vert^{2}\\
+cCC'e^{C}\cos^{2}\left(\vartheta\right)\left\vert  f\right\vert^{2}.
\end{multline}
By taking $\eta>0$ small enough, we deduce from (\ref{eq:nab24}) that 
there exists $C''>0$ such that
\begin{equation}\label{eq:nab25}
\left\langle  M_{b,\vartheta}f,f\right\rangle\le \frac{C''}{2}\left( \cos\left(\vartheta\right)
\left\vert  Y^{TX}\right\vert^{2}+\left\vert  
Y^{N}\right\vert^{2}+\cos^{2}\left(\vartheta\right) \right)  \left\vert f\right\vert^{2}.
\end{equation}

By (\ref{eq:nab20x1}), we get
\begin{equation}\label{eq:nab25x2}
\frac{d}{ds}\left\vert  V^{*}_{b,\vartheta,s}f\right\vert^{2}=2
\left\langle  
M_{b,\vartheta,s}V^{*}_{b,\vartheta,s},V^{*}_{b,\vartheta,s}f\right\rangle.
\end{equation}
Using (\ref{eq:nab25}), (\ref{eq:nab25x2}), we get
\begin{equation}\label{eq:nab26}
\frac{d}{ds}\left\vert  V^{*}_{b,\vartheta,s}f\right\vert^{2}\le
C''\left( \cos\left(\vartheta\right)
\left\vert  Y_{s}^{TX}\right\vert^{2}+\left\vert  
Y^{N}_{s}\right\vert^{2}+\cos^{2}\left(\vartheta\right) \right) \left\vert  V^{*}_{\vartheta,s}f\right\vert^{2}.
\end{equation}
By Gronwall's lemma, from (\ref{eq:nab26}), we obtain
\begin{equation}\label{eq:nab27}
\left\vert  V^{*}_{b,\vartheta,t}f\right\vert^{2}\le 
\exp\left(C''\int_{0}^{t}\left( \cos\left(\vartheta\right)\left\vert  
Y^{TX}_{s}\right\vert^{2}+\left\vert  Y^{N}_{s}\right\vert^{2} \right) ds+C''
\cos^{2}\left(\vartheta\right)t\right).
\end{equation}
Equation (\ref{eq:nab27}), is equivalent to
\begin{equation}\label{eq:nab28}
\left\vert  V_{b,\vartheta,t}\right\vert\le 
\exp\left(\frac{C''}{2}\int_{0}^{t}\left( \cos\left(\vartheta\right)\left\vert  
Y^{TX}_{s}\right\vert^{2}+\left\vert  Y^{N}_{s}\right\vert^{2} 
\right) ds+\frac{C''}{2}
\cos^{2}\left(\vartheta\right)t\right).
\end{equation}

By (\ref{eq:nab20}), (\ref{eq:nab28}), we get
\begin{equation}\label{eq:nab29}
E^{Q}\left[\left\vert  U^{0}_{b,\vartheta,t}\right\vert\right]\le
e^{C}E^{Q}\left[\exp\left(\frac{C'''}{2}\int_{0}^{t}\left(\cos\left(\vartheta\right)\left\vert  Y^{TX}_{s}\right\vert^{2}+
\left\vert  Y^{N}_{s}\right\vert^{2}\right)ds+C'''t\right)\right].
\end{equation}
 By (\ref{eq:rot7ay1}), (\ref{eq:nab29}), and using the same 
 arguments as in Remark \ref{Rintel},  given $c_{0}, 0<c_{0}\le 1$, if $b>0$ is such that 
$C'''b^{2}/\cos\left(\vartheta\right)\le c_{0}$, then
\begin{multline}\label{eq:nab30}
E^{Q}\left[\left\vert  U^{0}_{b,\vartheta,t}\right\vert\right]\le
e^{C}\exp\Biggl(\frac{C'''}{2}\cos\left(\vartheta\right)\left(mt+
b^{2}\left\vert  Y_{0}^{TX}
\right\vert^{2}\right) \\
+C'''\left(\frac{n}{2}+1\right)t+\frac{c_{0}}{2}\left\vert  
Y_{0}^{N}\right\vert^{2}\Biggr).
\end{multline}
Also by (\ref{eq:prin10y2}), if $b\le C^{-1/2}, C'''b^{2}/\cos\left(\vartheta\right)>1$, then
\begin{multline}\label{eq:nab31}
E^{Q}\left[\left\vert  
U^{0}_{b,\vartheta,t}\right\vert\right]\le\exp\Biggl(\frac{1}{2}C\cos\left(\vartheta\right)mt+
\frac{1}{2}\cos\left(\vartheta\right)\left\vert  
Y_{0}^{TX}\right\vert^{2}+\frac{C^{2}t}{2}\\
+\frac{1}{2}C'''nt+CC^{\prime \prime \prime 
1/2}t\left\vert  Y^{N}_{0}\right\vert\Biggr).
\end{multline}
By (\ref{eq:nab30}), (\ref{eq:nab31}), we get (\ref{eq:couac-1}) when 
$p=1$. 

To obtain (\ref{eq:couac-1}) for arbitrary $p\ge 1$, we may 
limit ourselves to the case where $p\in \N^{*}$. However, it is 
easy to verify that the $p^{\mathrm{th}}$ power of the right-hand 
side of (\ref{eq:hope1y1}) has a similar expansion. The arguments that were given before 
then extend easily to the case of a general $p$. 

Let us now establish (\ref{eq:couac-1y1}). By H\"{o}lder's 
inequality\footnote{As before, the H\"older norms are calculated with 
respect to $Q$.}, for $1<\theta<2$, we get
\begin{multline}\label{eq:rito11}
\exp\left(-\left\vert  Y_{0}\right\vert^{2}/2\right)
E^{Q}\left[\left\vert  U^{0}_{b,\vartheta,t}\right\vert^{p}\exp\left(
\left\vert  Y_{t}\right\vert^{2}/2\right)\right]\\
\le \left\Vert  \left\vert  
U^{0}_{b,\theta,t}\right\vert^{p}\right\Vert_{\theta/\left(\theta-1\right)}
\exp\left(-\left\vert  Y_{0}\right\vert^{2}/2\right)\left\Vert  \exp\left(\left\vert  
Y_{t}\right\vert^{2}/2\right)\right\Vert_{\theta}.
\end{multline}
 By 
(\ref{eq:rito2x2}),  (\ref{eq:prin10}), (\ref{eq:jar9b}), and (\ref{eq:rito11}),   for $0<b\le 1,t\ge \epsilon$
\begin{multline}\label{eq:rito12-1}
\exp\left(-\left\vert  Y_{0}\right\vert^{2}/2\right)
E^{Q}\left[\left\vert  U^{0}_{b,\vartheta,t}\right\vert^{p}\exp\left(
\left\vert  Y_{t}\right\vert^{2}/2\right)\right]
\le 
C\exp\Biggl(C^{2}t+\frac{1}{2}\frac{\left(m+\cos\left(\vartheta\right)n\right)}{b^{2}}
t\\
+C\cos\left(\vartheta\right)b\left\vert  
Y_{0}^{TX}\right\vert
+C\frac{b}{\cos^{1/2}\left(\vartheta\right)}
\left(1-e^{-t\cos\left(\vartheta\right)/b^{2}}\right)\left\vert  
Y_{0}^{N}\right\vert\Biggr)
\\
\exp\left(-C' \left(  \left\vert  Y_{0}^{TX}\right\vert^{2}+\cos\left(\vartheta\right)
\left\vert  Y_{0}^{N}\right\vert^{2}\right) \right).
\end{multline}
By (\ref{eq:rito12-1}), if $b_{0}$ such that $0<b_{0}\le 1$ is given, 
if $b_{0}\le b\le 1,\epsilon\le t\le M$,  then
\begin{multline}\label{eq:rito12-2}
\exp\left(-\left\vert  Y_{0}\right\vert^{2}/2\right)
E^{Q}\left[\left\vert  U^{0}_{b,\vartheta,t}\right\vert^{p}\exp\left(
\left\vert  Y_{t}\right\vert^{2}/2\right)\right]
\le \\
C\exp\left(C\cos\left(\vartheta\right)b\left\vert  
Y_{0}^{TX}\right\vert
+C\frac{\cos^{1/2}\left(\vartheta\right)}{b}\left\vert  
Y_{0}^{N}\right\vert \right) \\
\exp\left(-C' \left(  \left\vert  Y_{0}^{TX}\right\vert^{2}+\cos\left(\vartheta\right)
\left\vert  Y_{0}^{N}\right\vert^{2}\right) \right).
\end{multline}
By (\ref{eq:rito12-2}), we deduce that (\ref{eq:couac-1y1}) holds in 
the above range of parameters. In the sequel, we may as well take $b$ 
to be arbitrarily small.

By  (\ref{eq:rito2x2}),  (\ref{eq:prin10y1}), and (\ref{eq:rito11}),   
for $b>0$ small enough, and $\epsilon\le t\le M$, we get
\begin{multline}\label{eq:iga1}
\exp\left(-\left\vert  Y_{0}\right\vert^{2}/2\right)
E^{Q}\left[\left\vert  U^{0}_{b,\vartheta,t}\right\vert^{p}\exp\left(
\left\vert  Y_{t}\right\vert^{2}/2\right)\right]\\
\le C\exp\left(\frac{1}{2}\frac{\cos\left(\vartheta\right)}{b^{2}}nM+
C\frac{b}{\cos^{1/2}\left(\vartheta\right)}\left(1-
e^{-t\cos\left(\vartheta\right)/b^{2}}\right)\left\vert  
Y_{0}^{N}\right\vert\right)\\
\exp\left(-C'\left(\left\vert  Y_{0}^{TX}\right\vert^{2}
+\left(1-e^{-2t\cos\left(\vartheta\right)/b^{2}}\right)\left\vert  
Y_{0}^{N}\right\vert^{2}\right)\right).
\end{multline}
Take $c_{0}>0$. If $b>0$ is small enough, if $b^{2}/\cos\left(\vartheta\right)\ge c_{0},\epsilon\le 
t\le M$, then
\begin{multline}\label{eq:iga2}
C'\left(1-e^{-2t\cos\left(\vartheta\right)/b^{2}}\right)\left\vert  
Y_{0}^{N}\right\vert^{2}-C\frac{b}{\cos^{1/2}\left(\vartheta\right)}
\left(1-e^{-t\cos\left(\vartheta\right)/b^{2}}\right)\left\vert  Y_{0}^{N}\right\vert
\\
\ge c'\frac{\cos\left(\vartheta\right)}{b^{2}}\left\vert  
Y_{0}^{N}\right\vert^{2}-C\frac{\cos^{1/2}\left(\vartheta\right)}{b}
\left\vert  Y_{0}^{N}\right\vert\ge 
\frac{c'}{2}\frac{\cos\left(\vartheta\right)}{b^{2}}\left\vert  
Y^{N}_{0}\right\vert^{2}-C'.
\end{multline}
By (\ref{eq:iga2}), if we have also $0<b\le 1$, we get
\begin{multline}\label{eq:iga3}
C'\left(1-e^{-2t\cos\left(\vartheta\right)/b^{2}}\right)\left\vert  
Y_{0}^{N}\right\vert^{2}-C\frac{b}{\cos^{1/2}\left(\vartheta\right)}
\left(1-e^{-t\cos\left(\vartheta\right)/b^{2}}\right)\left\vert  
Y_{0}^{N}\right\vert \\
\ge c'\cos\left(\vartheta\right)\left\vert  
Y_{0}^{N}\right\vert^{2}-C'.
\end{multline}
By (\ref{eq:iga1}), (\ref{eq:iga3}), we deduce that for $0<b\le 1$, 
for $b$ small enough, $b^{2}/\cos\left(\vartheta\right)\ge c_{0}, 
\epsilon\le t\le M$, then (\ref{eq:couac-1y1}) still holds.

By the above, in the sequel, we can take $b>0$ and $b^{2}/\cos\left(\vartheta\right)$
as small as needed.

By (\ref{eq:rito2x2}), by (\ref{eq:nab30}) which 
we use with an arbitrary power instead of with the power $1$, and (\ref{eq:rito11}),  given 
$c_{0}\in \left]0,1\right]$,  for 
$0<b\le 1,\epsilon\le t\le M,C'''b^{2}/\cos\left(\vartheta\right)\le c_{0}$, we get
\begin{multline}\label{eq:rito12}
\exp\left(-\left\vert  Y_{0}\right\vert^{2}/2\right)
E^{Q}\left[\left\vert  U^{0}_{b,\vartheta,t}\right\vert^{p}\exp\left(
\left\vert  Y_{t}\right\vert^{2}/2\right)\right]\\
\le C
\exp\left(- 
C\left( 
\left\vert  Y_{0}^{TX}\right\vert^{2}+\left\vert  
Y_{0}^{N}\right\vert^{2}\right) \right) \\
\exp\left(\frac{C'''}{2}\cos\left(\vartheta\right)b^{2}\left\vert  
Y_{0}^{TX}\right\vert^{2}+\frac{c_{0}}{2}\left\vert  Y_{0}^{N}\right\vert^{2}\right).
\end{multline}
Since we can take $b,c_{0}$ to be arbitrarily small, from 
(\ref{eq:rito12}), we still get (\ref{eq:couac-1y1}).

The proof of our theorem is completed. 
\end{proof}

\subsection{A uniform estimate on $\left\Vert  \sup_{0\le t\le 
M}\left\vert  U^{0}_{b,\vartheta,t}\right\vert\right\Vert_{p}$}%
\label{subsec:limbs}
\begin{thm}\label{Tgres1}
Given $p>2 , M>0$, there exist $C_{p,M}>0, C'>0$ such that 
for $0<b\le 1,\vartheta\in\left[0,\frac{\pi}{2}\right[$, 
\begin{multline}
   \left\Vert  \sup_{0\le t\le M} \left\vert \left( U^{0}_{b,\vartheta,t}-
   \exp\left(-tN^{\Lambda\ac\left(T^{*}X \oplus N^{*}\right)}_{-\vartheta}/b^{2}\right) \right)  
       \mathbf{P}^{\perp}\right\vert\right\Vert_{p}\\
       \le
       C_{p,M}\exp\left(C'\left(\cos\left(\vartheta\right)\left\vert  
    Y_{0}^{TX}\right\vert^{2}+\left\vert  
    Y_{0}^{N}\right\vert^{2}\right)\right)\\
    \inf\left(\left(b/\cos^{1/2}\left(\vartheta\right)
    \right)^{\left(p-2\right)/p},\cos^{1/2}\left(\vartheta\right)/b\right).
    \label{eq:flimsw5}
\end{multline}
In particular, given $Y_{0}$,  for $0<b\le 1,
\vartheta\in\left[0,\frac{\pi}{2}\right[, 0\le t\le M$, the left-hand 
side of (\ref{eq:flimsw5}) is uniformly bounded.

Given $M>0, \vartheta\in \left[0,\frac{\pi}{2}\right[$, as $b\to 0$, 
$\left( 
U^{0}_{b,\vartheta,\cdot}-\exp\left(-tN^{\Lambda\ac\left(T^{*}X \oplus N^{*}\right)}_{-\vartheta}/b^{2}\right)
\right) \mathbf{P}^{\perp}$
converges uniformly to $0$ on $\left[0,M\right]$ in probability. In 
particular, given
$0<\epsilon\le M<+ \infty ,\vartheta\in\left[0,\frac{\pi}{2}\right[$, as $b\to 0$, 
$U^{0}_{b,\vartheta,\cdot}\mathbf{P^{\perp}}$ converges uniformly to 
$0$ on $\left[\epsilon,M\right]$ in probability.
\end{thm}
\begin{proof}
We proceed as in the proof of \cite[Proposition 14.10.2]{Bismut08b}.
By (\ref{eq:diff0}), we get
\begin{multline}
    U^{0}_{b,\vartheta,t}=
    \exp\left(-tN^{\Lambda\ac\left(T^{*}X \oplus 
    N^{*}\right)}_{-\vartheta}/b^{2}\right) \\
    +\int_{0}^{t}U^{0}_{b,\vartheta,
    s}\frac{R^{0}_{\vartheta}\left(Y_{s}\right)}{b}
    \exp\left(-\left(t-s\right)N^{\Lambda\ac\left(T^{*}X \oplus N^{*}\right)}_{-\vartheta}/b^{2}\right)ds.
    \label{eq:flimsw2}
\end{multline}
Clearly,
\begin{multline}\label{eq:prin15}
\left\vert  \int_{0}^{t}U^{0}_{b,\vartheta,
    s}\frac{R^{0}_{\vartheta}\left(Y_{s}\right)}{b}
    \exp\left(-\left(t-s\right)N^{\Lambda\ac\left(T^{*}X \oplus N^{*}\right)}_{-\vartheta}/b^{2}\right)ds\mathbf{P}^{\perp}\right\vert \\
    \le\int_{0}^{t}\left\vert  U^{0}_{b,\vartheta,
    s}\frac{R^{0}_{\vartheta}\left(Y_{s}\right)}{b}
    \exp\left(-\left(t-s\right)N^{\Lambda\ac\left(T^{*}X \oplus N^{*}\right)}_{-\vartheta}/b^{2}\right)
    \mathbf{P}^{\perp}\right\vert ds.
\end{multline}
Take $p>2$, and let $q$ be such that $\frac{1}{p}+\frac{1}{q}=1$. By 
(\ref{eq:est0}) and by
H\"{o}lder's inequality, for $0\le t\le M$, we get
\begin{multline}\label{eq:prin16}
\int_{0}^{t}\left\vert  U^{0}_{b,\vartheta,
    s}\frac{R^{0}_{\vartheta}\left(Y_{s}\right)}{b}
    \exp\left(-\left(t-s\right)N^{\Lambda\ac\left(T^{*}X \oplus N^{*}\right)}_{-\vartheta}/b^{2}\right)
    \mathbf{P}^{\perp}\right\vert ds\\
    \le Cb^{\left(p-2\right)/p}\left[\int_{0}^{M}\left\vert  
    U_{b,\vartheta,s}^{0}\right\vert^{2p}ds\right]^{1/2p}\\
   \left( \cos\left(\vartheta\right) 
   \left[\int_{0}^{M}\left\vert  
   Y^{TX}_{s}\right\vert^{2p}ds\right]^{1/2p}+
   \cos^{1/2}\left(\vartheta\right)\left[\int_{0}^{M}\left\vert  
   Y^{N}_{s}\right\vert^{2p}ds\right]^{1/2p} \right) \\
  \left\vert  \frac{1-\exp\left(-qMN^{\Lambda\ac\left(T^{*}X \oplus N^{*}\right)}_{-\vartheta}/b^{2}\right)}
  {qN^{\Lambda\ac\left(T^{*}X \oplus N^{*}\right)}_{-\vartheta}}\mathbf{P}^{\perp}\right\vert
  ^{1/q}.
\end{multline}

By Cauchy-Schwarz, we have
\begin{multline}\label{eq:prin17x1}
E^{Q}\left[\left[\int_{0}^{M}\left\vert  
U^{0}_{b,\vartheta,s}\right\vert^{2p}ds\right]^{1/2}
\left[\int_{0}^{M}\left\vert  
Y_{s}^{TX}\right\vert^{2p}ds\right]^{1/2}\right]\\
\le \left\{E^{Q}\left[\int_{0}^{M}\left\vert  
U^{0}_{b,\vartheta,s}\right\vert^{2p}ds\right]\right\}^{1/2}
\left\{E^{Q}\left[\int_{0}^{M}\left\vert  Y_{s}^{TX}\right\vert^{2p}ds\right]\right\}
^{1/2}.
\end{multline}
In (\ref{eq:prin17x1}), we can replace $Y^{TX}_{\cdot}$ by 
$Y^{N}_{\cdot}$.

By (\ref{eq:flimsw2})--(\ref{eq:prin17x1}), we obtain
\begin{multline}\label{eq:mrin17x2}
 \left\Vert  \sup_{0\le t\le M} \left\vert \left( U^{0}_{b,\vartheta,t}-
   \exp\left(-tN^{\Lambda\ac\left(T^{*}X \oplus N^{*}\right)}_{-\vartheta}/b^{2}\right) \right)  
       \mathbf{P}^{\perp}\right\vert\right\Vert_{p}\\
       \le Cb^{\left(p-2\right)/p}
       \left\{E^{Q}\int_{0}^{M}\left\vert  
       U^{0}_{b,\vartheta,s}\right\vert^{2p}ds\right\}^{1/2p}
       \cos^{1/2}\left(\vartheta\right) \\
       \left( \cos^{1/2}\left(\vartheta\right)\left\{E^{Q}\left[\int_{0}^{M}
       \left\vert  Y_{s}^{TX}\right\vert^{2p}ds\right]\right\}^{1/2p}+
      \left\{E^{Q}\left[\int_{0}^{M}\left\vert  
       Y^{N}_{s}\right\vert^{2p}ds\right]\right\}^{1/2p}\right)\\ 
       \left\vert  \frac{1-\exp\left(-qMN^{\Lambda\ac\left(T^{*}X \oplus N^{*}\right)}_{-\vartheta}/b^{2}\right)}
  {qN^{\Lambda\ac\left(T^{*}X \oplus N^{*}\right)}_{-\vartheta}}\mathbf{P}^{\perp}\right\vert
  ^{1/q}.
\end{multline}

By \cite[Proposition 10.8.2]{Bismut08b}, we get
\begin{align}\label{eq:prin17}
&E^{Q}\left[\left\vert  Y_{t}^{TX}\right\vert^{2p}\right]\le 
C_{p}\left(1+\left\vert  Y_{0}^{TX}\right\vert^{2p}\right),
&E^{Q}\left[\left\vert  Y_{t}^{N}\right\vert^{2p}\right]\le 
C_{p}\left(1+\left\vert  Y_{0}^{N}\right\vert^{2p}\right).
\end{align}

For $a>0,x\ge 0$, we have the  inequality,
\begin{equation}\label{eq:fou1}
\frac{1-e^{-x/a}}{x}\le \inf\left(1/a,1/x\right).
\end{equation}
From (\ref{eq:fou1}), we obtain
\begin{equation}\label{eq:prin18x1}
\left\vert  \frac{1-\exp\left(-qMN^{\Lambda\ac\left(T^{*}X \oplus N^{*}\right)}_{-\vartheta}/b^{2}\right)}
  {qN^{\Lambda\ac\left(T^{*}X \oplus N^{*}\right)}_{-\vartheta}}\mathbf{P}^{\perp}\right\vert\le
  C\inf \left( \frac{1}{\cos\left(\vartheta\right)}, 
 \frac{1}{b^{2}}\right).
\end{equation}
Using equation (\ref{eq:couac-1}) in Theorem \ref{Tfuest} and 
(\ref{eq:mrin17x2})--(\ref{eq:prin18x1}), for $0\le t\le M$, we get
(\ref{eq:flimsw5}).  The last part of our theorem follows from 
(\ref{eq:flimsw5}). The proof of our theorem is completed. 
\end{proof}

Now we get an extension  of
\cite[Proposition 14.10.3]{Bismut08b}, which is also a path integral 
version of Proposition \ref{Pwondidbis}.
\begin{prop}\label{Pformid}
For $0\le \vartheta<\frac{\pi}{2}$, the following identity holds:
\begin{multline}
    U^{0}_{b,\vartheta,t}\Biggl( 
   1+b\left(1+N^{\Lambda\ac\left(T^{*}X \oplus 
   N^{*}\right)}_{-\vartheta}\right)^{-1}R^{0}_{\vartheta}\left(Y_{t}^{TX}\right)\\
   +b
   \left(\cos\left(\vartheta\right)+N^{\Lambda\ac\left(T^{*}X \oplus N^{*}\right)}_{-\vartheta}\right)^{-1}R^{0}_{\vartheta}\left(Y^{N}_{t}\right)\Biggr)\\
   =1+b\left(1+N^{\Lambda\ac\left(T^{*}X \oplus N^{*}\right)}_{-\vartheta}\right)^{-1}R^{0}_{\vartheta}\left(Y_{0}^{TX}\right)+b
   \left(\cos\left(\vartheta\right)+N^{\Lambda\ac\left(T^{*}X \oplus N^{*}\right)}_{-\vartheta}\right)^{-1}R^{0}_{\vartheta}\left(Y^{N}_{0}\right)
    \\
   +
   \int_{0}^{t}U^{0}_{b,\vartheta,s}\Biggl(-\frac{N^{\Lambda\ac\left(T^{*}X \oplus N^{*}\right)}_{-\vartheta}}{b^{2}}+
  R^{0}_{\vartheta}\left(Y_{s}\right)
   \Biggl [\left(1+N^{\Lambda\ac\left(T^{*}X \oplus N^{*}\right)}_{-\vartheta}\right)^{-1}R^{0}_{\vartheta}\left(Y^{TX}\right)\\
   +\left(\cos\left(\vartheta\right)+
   N^{\Lambda\ac\left(T^{*}X \oplus N^{*}\right)}_{-\vartheta}\right)^{-1}R^{0}_{\vartheta}\left(Y^{N}\right)\Biggr]\Bigg)ds\\
   +\int_{0}^{t}U^{0}_{b,\vartheta,s}\Bigg( \left(1+N^{\Lambda\ac\left(T^{*}X \oplus N^{*}\right)}_{-\vartheta}\right)  
   ^{-1}R^{0}_{\vartheta}\left(\delta 
   w^{TX}\right)\\
   +\left(\cos\left(\vartheta\right)+N^{\Lambda\ac\left(T^{*}X \oplus N^{*}\right)}_{-\vartheta}\right)^{-1}
   \cos^{1/2}\left(\vartheta\right)R^{0}_{\vartheta}\left( \delta 
   w^{N}\right) \Biggr).
    \label{eq:formid5}
\end{multline}
\end{prop}
\begin{proof}
 By (\ref{eq:wat2}), (\ref{eq:diff0}), we get 
(\ref{eq:formid5}).
\end{proof}
\begin{remark}\label{Rito}
	In the right hand-side of (\ref{eq:formid5}), we could have replaced 
	the It\^o integrals by more classical Stratonovitch integrals. 
	However, in the proof of Theorem \ref{Tunif}, the estimates will 
	be established using the fact that It\^o integrals
	appear in (\ref{eq:formid5}).
\end{remark}

Now we will establish a crucial estimate.
\begin{thm}\label{Tunif}
For any $p>2 , M>0$, there exist $ C_{p,M}>0, C'>0$ such that 
for $0<b\le 1,\vartheta\in \left[0,\frac{\pi}{2}\right[$, then
\begin{equation}
   \left\Vert  \sup_{0\le t\le M} \left\vert U^{0}_{b,\vartheta,t}
   \right\vert
      \right \Vert_{p}
       \le
       C_{p,M}\exp\left(C'\left(\cos\left(\vartheta\right)\left\vert  
    Y_{0}^{TX}\right\vert^{2}+\left\vert  
    Y_{0}^{N}\right\vert^{2}\right)\right).
    \label{eq:prin18}
\end{equation}
\end{thm}
\begin{proof}
    If we just consider $U^{0}_{b,\vartheta,\cdot}\mathbf{P}^{\perp}$, 
    our estimate follows from Theorem \ref{Tgres1}. We 
    should then prove the corresponding estimate for 
    $U^{0}_{b,\vartheta,\cdot}\mathbf{P}$.
    
    By equation (\ref{eq:formid5}) in Proposition \ref{Pformid}, we get 
an analogue of \cite[eq. (14.10.32)]{Bismut08b},
\begin{multline}
   U^{0}_{b,\vartheta,t}\Bigg( 
   1+b\left(1+N^{\Lambda\ac\left(T^{*}X \oplus 
   N^{*}\right)}_{-\vartheta}\right)^{-1}R^{0}_{\vartheta}\left(Y_{t}^{TX}\right)\\
   +b
   \left(\cos\left(\vartheta\right)+N^{\Lambda\ac\left(T^{*}X \oplus N^{*}\right)}_{-\vartheta}
   \right)^{-1}R^{0}_{\vartheta}\left(Y^{N}_{t}\right)\Biggr)\mathbf{P}\\
   = \Biggl( 1+b\left(1+N^{\Lambda\ac\left(T^{*}X \oplus N^{*}
   \right)}_{-\vartheta}\right)^{-1}R^{0}_{\vartheta}
   \left(Y_{0}^{TX}\right) \\
   +b
   \left(\cos\left(\vartheta\right)+N^{\Lambda\ac\left(T^{*}X \oplus N^{*}\right)}_{-\vartheta}\right)^{-1}R^{0}_{\vartheta}
   \left(Y^{N}_{0}\right)\Biggr) \mathbf{P}
    \\
   +
   \int_{0}^{t}U^{0}_{b,\vartheta,s}R^{0}_{\vartheta}\left(Y_{s}\right)
   \Biggl[\left(1+N^{\Lambda\ac\left(T^{*}X \oplus 
   N^{*}\right)}_{-\vartheta}\right)^{-1}R^{0}_{\vartheta}\left(Y^{TX}_{s}\right) \\
   +\left(\cos\left(\vartheta\right)+
   N^{\Lambda\ac\left(T^{*}X \oplus 
   N^{*}\right)}_{-\vartheta}\right)^{-1}R^{0}_{\vartheta}\left(Y^{N}_{s}\right)\Biggr]\mathbf{P}ds\\
   +\int_{0}^{t}U^{0}_{b,\vartheta,s}\Biggl[ \left(1+N^{\Lambda\ac\left(T^{*}X \oplus N^{*}\right)}_{-\vartheta}\right)  
   ^{-1}R^{0}_{\vartheta}\left(\delta 
   w^{TX}\right) \\
   +\left(\cos\left(\vartheta\right)+N^{\Lambda\ac\left(T^{*}X \oplus N^{*}\right)}_{-\vartheta}\right)^{-1}\cos^{1/2}
   \left(\vartheta\right)R^{0}_{\vartheta}\left( \delta w^{N}\right) 
   \Biggr] \mathbf{P}.
    \label{eq:formid7}
\end{multline}

By (\ref{eq:est0}) and by H\"{o}lder's inequality, for $p>1$, we get
\begin{multline}\label{eq:afb1}
 \int_{0}^{M}\Biggl\vert U^{0}_{b,\vartheta,s}R^{0}_{\vartheta}\left(Y_{s}\right)
   \Biggl[\left(1+N^{\Lambda\ac\left(T^{*}X \oplus 
   N^{*}\right)}_{-\vartheta}\right)^{-1}R^{0}_{\vartheta}\left(Y^{TX}_{s}\right) \\
   +\left(\cos\left(\vartheta\right)+
   N^{\Lambda\ac\left(T^{*}X \oplus 
   N^{*}\right)}_{-\vartheta}\right)^{-1}R^{0}_{\vartheta}\left(Y^{N}_{s}\right)\Biggr]\mathbf{P}\Biggr\vert ds \\
   \le C_{M,p}\Biggl[\int_{0}^{M}\left\vert  
   U^{0}_{b,\vartheta,s}\right\vert^{p}\left(\cos^{p}\left(\vartheta\right)\left\vert  
   Y^{TX}\right\vert^{2p}+\left\vert  
   Y^{N}\right\vert^{2p}\right)ds\Biggr]^{1/p}.
\end{multline}
In (\ref{eq:afb1}), equation (\ref{eq:est0}) has been used in particular to 
overcome the potentially singular term 
$\left(\cos\left(\vartheta\right)+N^{\Lambda\ac\left(T^{*}X 
\oplus N^{*}\right)}_{-\vartheta}\right)^{-1}$.
By (\ref{eq:afb1}), and by Cauchy-Schwarz, we obtain
\begin{multline}\label{eq:afb2}
E^{Q}\Biggl[\Biggl\{\int_{0}^{M}\Biggl\vert U^{0}_{b,\vartheta,s}R^{0}_{\vartheta}\left(Y_{s}\right)
   \Biggl[\left(1+N^{\Lambda\ac\left(T^{*}X \oplus 
   N^{*}\right)}_{-\vartheta}\right)^{-1}R^{0}_{\vartheta}\left(Y^{TX}_{s}\right) \\
   +\left(\cos\left(\vartheta\right)+
   N^{\Lambda\ac\left(T^{*}X \oplus 
   N^{*}\right)}_{-\vartheta}\right)^{-1}R^{0}_{\vartheta}\left(Y^{N}_{s}\right)\Biggr]\mathbf{P}\Biggr\vert ds
   \Biggr\}^{p}\Biggr]\\
   \le C_{M,p}\left\{E^{Q}\left[\int_{0}^{M}\left\vert  
   U^{0}_{b,\vartheta,s}\right\vert^{2p}ds\right]\right\}^{1/2}\\
   \left\{E^{Q}\left[\int_{0}^{M}\left(\cos^{2p}\left(\vartheta\right)\left\vert  
   Y^{TX}\right\vert^{4p}+\left\vert  
   Y^{N}\right\vert^{4p}\right)ds\right]\right\}^{1/2}.
\end{multline}
Using equation (\ref{eq:couac-1}) in Theorem \ref{Tfuest},  
(\ref{eq:prin17}),  and (\ref{eq:afb2}), we find that the estimate in 
(\ref{eq:afb2}) is compatible with (\ref{eq:prin18}).

The contribution of the last stochastic integral $I_{t}$ in the right-hand 
side of (\ref{eq:formid7}) can be estimated using an inequality by
Burkholder-Davis-Gundy \cite{BurkholderGundy70}, which asserts 
that if $f_{\cdot}$ is an adapted stochastic process such that 
$\int_{0}^{M}\left\vert  f\right\vert^{2}ds< + \infty $ a.s., for $p>1$
\begin{equation}\label{eq:prin18a1}
\left[E^{Q}\left[\sup_{0\le t\le M}\left\vert  \int_{0}^{t}\left\langle  f_{s},\delta 
w_{s}\right\rangle\right\vert^{p}\right]\right]^{1/p}\le
C_{p}\left[E^{Q}\left[\left\{\int_{0}^{M}\left\vert  
f_{s}\right\vert^{2}ds\right\}^{p/2}\right]\right]^{1/p}.
\end{equation}

By  (\ref{eq:est0}),  (\ref{eq:prin18a1}),  for $p>1$, we get
\begin{equation}\label{eq:prin19}
\left[E^{Q}\left[\sup_{0\le t\le M}\left\vert  
I_{t}\right\vert^{p}\right]\right]^{1/p}\le
C_{p}E^{Q}\left[\left\vert  \int_{0}^{M}\left\vert  
U^{0}_{b,\vartheta,s}\right\vert^{2}ds\right\vert^{p/2}\right]^{1/p}.
\end{equation}
By H\"{o}lder's inequality, for $p>2$, we get
\begin{equation}\label{eq:prin20}
\left\vert  \int_{0}^{M}\left\vert  
U^{0}_{b,\vartheta,s}\right\vert^{2}ds\right\vert^{p/2}\le M^{\left(p-2\right)/2}
\int_{0}^{M}\left\vert  U_{b,\vartheta,s}\right\vert^{p}ds. 
\end{equation}
Using equation (\ref{eq:couac-1}) in Theorem \ref{Tfuest}, (\ref{eq:prin19}), and (\ref{eq:prin20}), 
the contribution of $I_{\cdot}$ to the estimation of the right-hand 
side of (\ref{eq:formid7}) is still compatible with (\ref{eq:prin18}).

The most annoying term in (\ref{eq:formid7}) is the contribution of 
\begin{multline}\label{eq:prin21}
A_{b,\vartheta,t}=U^{0}_{b,\vartheta,t}b \Biggl( \left(1+N^{\Lambda\ac\left(T^{*}X \oplus N^{*}\right)}_{-\vartheta}
\right)^{-1}R^{0}_{\vartheta}\left(Y_{t}^{TX}\right) \\
+
   \left(\cos\left(\vartheta\right)+N^{\Lambda\ac\left(T^{*}X \oplus N^{*}\right)}_{-\vartheta}
   \right)^{-1}R^{0}_{\vartheta}\left(Y^{N}_{t}\right)\Biggr)-b\Biggl( \left(1+N^{\Lambda\ac\left(T^{*}X \oplus N^{*}\right)}_{-\vartheta}
\right)^{-1}R^{0}_{\vartheta}\left(Y_{0}^{TX}\right) \\
+
   \left(\cos\left(\vartheta\right)+N^{\Lambda\ac\left(T^{*}X \oplus N^{*}\right)}_{-\vartheta}
   \right)^{-1}R^{0}_{\vartheta}\left(Y^{N}_{0}\right)\Biggr)\mathbf{P}.
\end{multline}
Set
\begin{align}\label{eq:prin22}
    B_{b,\vartheta,t}= & \,U^{0}_{b,\vartheta,t}\mathbf{P}b \Biggl( 
\left(1+N^{\Lambda\ac\left(T^{*}X \oplus N^{*}\right)}
_{-\vartheta}\right)^{-1}
R^{0}_{\vartheta}\left(Y_{t}^{TX}\right) \nonumber  \\
&+\left(\cos\left(\vartheta\right) 
    +N^{\Lambda\ac\left(T^{*}X \oplus N^{*}\right)}_{-\vartheta}
   \right)^{-1}R^{0}_{\vartheta}\left(Y^{N}_{t}\right)\Biggr)\mathbf{P}\nonumber \\
   &-\mathbf{P}b \Biggl( 
\left(1+N^{\Lambda\ac\left(T^{*}X \oplus N^{*}\right)}_{-\vartheta}\right)^{-1}
R^{0}_{\vartheta}\left(Y_{0}^{TX}\right)  \nonumber \\
&+\left(\cos\left(\vartheta\right) 
    +N^{\Lambda\ac\left(T^{*}X \oplus N^{*}\right)}_{-\vartheta}
   \right)^{-1}R^{0}_{\vartheta}\left(Y^{N}_{0}\right)\Biggr)\mathbf{P},\\
B^{\perp}_{b,\vartheta,t}=& \, U^{0}_{b,\vartheta,t}\mathbf{P}^{\perp}b 
\Biggl( \left(1+N^{\Lambda\ac\left(T^{*}X \oplus N^{*}\right)}
_{-\vartheta}\right)^{-1}
R^{0}_{\vartheta}\left(Y_{t}^{TX}\right) \nonumber 
\\
&+\left(\cos\left(\vartheta\right)+N^{\Lambda\ac\left(T^{*}X \oplus N^{*}\right)}_{-\vartheta}
   \right)^{-1}R^{0}_{\vartheta}\left(Y^{N}_{t}\right)\Biggr)\mathbf{P} \nonumber \\
   &-\mathbf{P}^{\perp}
 b 
\Biggl( \left(1+N^{\Lambda\ac\left(T^{*}X \oplus N^{*}\right)}_{-\vartheta}\right)^{-1}
R^{0}_{\vartheta}\left(Y_{0}^{TX}\right) \nonumber 
\\
&+\left(\cos\left(\vartheta\right)+N^{\Lambda\ac\left(T^{*}X \oplus N^{*}
\right)}_{-\vartheta}
   \right)^{-1}R^{0}_{\vartheta}\left(Y^{N}_{0}\right)\Biggr)\mathbf{P}. \nonumber 
\end{align}
Then
\begin{equation}\label{eq:orin22x1}
A_{b,\vartheta,t}=B_{b,\vartheta,t}+B^{\perp}_{b,\vartheta,t}.
\end{equation}
First we will control $B^{\perp}_{\bt,\cdot}$, and later 
$B_{\bt,\cdot}$.\\
\noindent $\bullet$ \underline{\textit{Control of} $B^{\perp}_{\bt,\cdot}$}\\
By (\ref{eq:est0}),  we get
\begin{multline}\label{eq:afb4}
b\left\vert \left(1+N^{\Lambda\ac\left(T^{*}X \oplus N^{*}\right)}_{-\vartheta}\right)^{-1}R^{0}_{\vartheta}
\left(Y^{TX}\right)
+\left(\cos\left(\vartheta\right)+N^{\Lambda\ac\left(T^{*}X \oplus N^{*}\right)}_{-\vartheta}
   \right)^{-1}R^{0}_{\vartheta}\left(Y^{N}\right) \right\vert\\
   \le C\left(b\cos\left(\vartheta\right)\left\vert  
   Y^{TX}\right\vert+\frac{b}{\cos^{1/2}\left(\vartheta\right)}\left\vert  Y^{N}\right\vert\right).
\end{multline}
By Theorem \ref{Tgres1} and by
(\ref{eq:afb4}), for $p>2$,  we get
\begin{multline}\label{eq:afb5}
\Biggl\Vert  \sup_{0\le t\le 
M}\Biggl\vert\left( U^{0}_{b,\vartheta,t}-\exp\left(-tN_{-\vartheta}^{\Lambda\ac\left(T^{*}X \oplus N\right)}/b^{2}\right) \right) 
\mathbf{P}^{\perp} \\
b \Biggl( \left(1+N^{\Lambda\ac\left(T^{*}X \oplus N^{*}\right)}_{-\vartheta}
\right)^{-1}R^{0}_{\vartheta}\left(Y_{t}^{TX}\right) \\
+
   \left(\cos\left(\vartheta\right)+N^{\Lambda\ac\left(T^{*}X \oplus N^{*}\right)}_{-\vartheta}
   \right)^{-1}R^{0}_{\vartheta}\left(Y^{N}_{t}\right)\Biggr)\mathbf{P}\Biggr\vert\Biggr\Vert_{p}\\
   \le 
   C_{2p,M}\exp\left(C'\left(\cos\left(\vartheta\right)\left\vert  
   Y_{0}^{TX}\right\vert^{2}+\left\vert  
   Y^{N}_{0}\right\vert^{2}\right)\right)\\
   \Biggl(b\cos\left(\vartheta\right)\left\Vert  \sup_{0\le t\le 
   M}\left\vert  
   Y^{TX}_{t}\right\vert\right\Vert_{2p}\\
   +\inf\left(\left(b/\cos^{1/2}
   \left(\vartheta\right)\right)^{\left(2p-1\right)/p},1\right)\left\Vert  
   \sup_{0\le t\le M}\left\vert  
   Y^{N}_{t}\right\vert\right\Vert_{2p}\Biggr).
   \end{multline}
   By  (\ref{eq:lom7}), for $0<b\le 1$, we get
   \begin{equation}\label{eq:afb6}
\left\Vert  \sup_{0\le t\le M}\left\vert  
bY^{TX}_{t}\right\vert\right\Vert_{2p}\le C\left(1+\left\vert  
Y_{0}^{TX}\right\vert\right).
\end{equation}
The same argument shows that if $b/\cos^{1/2}\left(\vartheta\right)\le 1$,  then
\begin{equation}\label{eq:afb7}
\frac{b}{\cos^{1/2}\left(\vartheta\right)}
\left\Vert  \sup_{0\le t\le M}\left\vert  
Y_{t}^{N}\right\vert\right\Vert_{2p}\le
C\left(1+\left\vert  Y_{0}^{N}\right\vert\right).
\end{equation}
Since $p>2$, then $\left(2p-2\right)/p>1$, so that in (\ref{eq:afb7}), 
$\frac{b}{\cos^{1/2}\left(\vartheta\right)}$ can be replaced by 
$\left(\frac{b}{\cos^{1/2}\left(\vartheta\right)}\right)^{\left(2p-1\right)/p}$.
If $b/\cos^{1/2}\left(\vartheta\right)>1$, by (\ref{eq:lom7}), we 
get
\begin{equation}\label{eq:afb8}
\left\Vert  \sup_{0\le t\le M}\left\vert  
Y^{N}_{t}\right\vert\right\Vert_{2p}\le C\left(1+\left\vert  
Y_{0}^{N}\right\vert\right).
\end{equation}
By (\ref{eq:afb6})--(\ref{eq:afb8}),  we can give an upper bound for 
the left-hand side of (\ref{eq:afb5}) which is compatible with 
(\ref{eq:prin18}).

 By (\ref{eq:afb4}), we get
 \begin{multline}\label{eq:afb9}
\Biggl\Vert  \sup_{0\le t\le 
M}\Biggl\vert \left( \exp\left(-tN_{-\vartheta}^{\Lambda\ac\left(T^{*}X \oplus N\right)}/b^{2}\right)
-1 \right) 
\mathbf{P}^{\perp} \\
b \Biggl( \left(1+N^{\Lambda\ac\left(T^{*}X \oplus N^{*}\right)}_{-\vartheta}
\right)^{-1}R^{0}_{\vartheta}\left(Y_{t}^{TX}\right) \\
+
   \left(\cos\left(\vartheta\right)+N^{\Lambda\ac\left(T^{*}X \oplus N^{*}\right)}_{-\vartheta}
   \right)^{-1}R^{0}_{\vartheta}\left(Y^{N}_{t}\right)\Biggr)\mathbf{P}\Biggr\vert\Biggr\Vert_{p}\\
   \le 
   C
   \Biggl(b\cos\left(\vartheta\right)\left\Vert  \sup_{0\le t\le 
   M}\left\vert  
   Y^{TX}_{t}\right\vert\right\Vert_{p}+\Biggl( \left( 1-
   \exp\left(-\cos\left(\vartheta\right)nM/b^{2}\right) \right) 
   \frac{b}{\cos^{1/2}\left(\vartheta\right)}\\
   +b\cos^{1/2}\left(\vartheta\right) \Biggr) 
   \left\Vert  
   \sup_{0\le t\le M}\left\vert  
   Y^{N}_{t}\right\vert\right\Vert_{p}\Biggr)\left\vert  f\right\vert.
\end{multline}
The last two terms in the right-hand side of (\ref{eq:afb9}) 
correspond to the cases where $N^{\Lambda\ac\left(T^{*}X\right)}=0$ 
and $N^{\Lambda\ac\left(T^{*}X\right)}>0$.
If $b/\cos^{1/2}\left(\vartheta\right)\le 1$, by (\ref{eq:afb6}), 
(\ref{eq:afb7}), we still get an upper bound for the left-hand side 
of (\ref{eq:afb9}) that is compatible with (\ref{eq:prin18}). If 
$0<b\le 1,\frac{b}{\cos^{1/2}\left(\vartheta\right)}>1$, we get
\begin{equation}\label{eq:afb10}
\left( 1-
   \exp\left(-\cos\left(\vartheta\right)nM/b^{2}\right) \right) 
   \frac{b}{\cos^{1/2}\left(\vartheta\right)}  
   \le nM\frac{\cos^{1/2}\left(\vartheta\right)}{b}\le nM.
\end{equation}
Using (\ref{eq:afb6}), (\ref{eq:afb8}), and (\ref{eq:afb10}), we still obtain a good upper 
bound for the left-hand side of (\ref{eq:afb9}). 

By (\ref{eq:est0}), we get 
\begin{multline}\label{eq:afb11}
\Biggl\vert  
\mathbf{P^{\perp}}\Biggl( b\left(1+N^{\Lambda\ac\left(T^{*}X \oplus 
N^{*}\right)}_{-\vartheta}\right)^{-1}(R^{0}_{\vartheta}\left(Y^{TX}_{t}-Y^{TX}_{0}\right)
\\
+b\left(\cos\left(\vartheta\right)+N^{\Lambda\ac\left(T^{*}X 
\oplus 
N^{*}\right)}_{-\vartheta}\right)^{-1}R^{0}_{\vartheta}\left(Y^{N}_{t}-Y^{N}_{0}\right)
\Biggr)\mathbf{P}\Biggr\vert\\
\le C b\cos\left(\vartheta\right)\left\vert Y^{TX}_{t}-Y_{0}^{TX}\right\vert+C\frac{b}{\cos^{1/2}
\left(\vartheta\right)}
\left\vert   
Y^{N}_{t}-Y^{N}_{0} \right\vert.
\end{multline}

By (\ref{eq:afb6}), the first term in the 
right-hand side of (\ref{eq:afb11}) can be dominated by the 
right-hand side of (\ref{eq:prin18}).
If $ b/\cos^{1/2}\left(\vartheta\right)\le1$, using (\ref{eq:afb7}), 
we can dominate 
properly the 
contribution of the second term in the right-hand side of (\ref{eq:afb11}). 
Let $Y^{N,0}_{\cdot}$ be 
the solution of (\ref{eq:wat2}) with $Y^{N,0}_{0}=0$.  Then
\begin{equation}\label{eq:afb12}
Y_{t}^{N}-Y^{N}_{0}=\left(e^{-\cos\left(\vartheta\right)t/b^{2}}-1\right)Y^{N}_{0}+Y^{N,0}_{t}.
\end{equation}
Then
\begin{multline}\label{eq:afb13}
\frac{b}{\cos^{1/2}\left(\vartheta\right)} \left\Vert  \sup_{0\le t\le M}\left\vert  
\left( Y^{N}_{t}-Y^{N}_{0} 
\right) \right\vert\right\Vert_{p}\le
\frac{b}{\cos^{1/2}\left(\vartheta\right)}\left(1-e^{-\cos\left(\vartheta\right)M/b^{2}}\right)\left\vert  Y^{N}_{0}\right\vert\\
+ \frac{b}{\cos^{1/2}\left(\vartheta\right)} 
\left\Vert  \sup_{0\le t\le M}\left\vert  
Y^{N,0}_{t}\right\vert\right\Vert_{p}.
\end{multline}
Also
\begin{equation}\label{eq:afb14}
\frac{b}{\cos^{1/2}\left(\vartheta\right)}\left(1-e^{-\cos\left(\vartheta\right)M/b^{2}}\right)\le 
\frac{\cos^{1/2}\left(\vartheta\right)}{b}M.
\end{equation}
If $b/\cos^{1/2}\left(\vartheta\right)>1$, the right-hand side of 
(\ref{eq:afb14}) is dominated by $M$.
By (\ref{eq:lom7}), when $b/\cos^{1/2}\left(\vartheta\right)>1$, we get
\begin{equation}\label{eq:afb15}
\frac{b}{\cos^{1/2}\left(\vartheta\right)} \left\Vert  \sup_{0\le t\le 
M}\left\vert  Y^{N,0}_{t}\right\vert\right\Vert_{p}\le C.
\end{equation}
By (\ref{eq:afb11})--(\ref{eq:afb15}), we obtain an upper bound for 
(\ref{eq:afb11}) that is compatible with (\ref{eq:prin18}).

By (\ref{eq:afb5})--(\ref{eq:afb15}),  we find that $\left\Vert\sup_{0\le t\le 
M}\left\vert  B^{\perp}_{\bt,t}\right\vert\right\Vert_{p}$ is dominated by 
an expression like the right-hand side of (\ref{eq:prin18}). \\
\noindent $\bullet$ \underline{\textit{Control of} $B_{\bt,\cdot}$}\\
Set
\begin{equation}\label{eq:afb16}
C_{b,\vartheta,t}=U^{0}_{b,\vartheta,t}
\mathbf{P}+B_{b,\vartheta,t}.
\end{equation}
It follows from equation (\ref{eq:formid7}) and from the above that
\begin{equation}\label{eq:afb16x1}
\left\Vert  \sup_{0\le t\le M}\left\vert  
C_{b,\vartheta,t}\right\vert\right\Vert_{p}\le
C_{p,M}\exp\left(C'\left(\cos\left(\vartheta\right)\left\vert  
Y_{0}^{TX}\right\vert^{2}+\left\vert  
Y_{0}^{N}\right\vert^{2}\right)\right).
\end{equation}

Note that
\begin{multline}\label{eq:prin24}
B_{b,\vartheta,t}=U^{0}_{b,\vartheta,t}\mathbf{P}b \left( 
R^{0}_{\vartheta}\left(Y_{t}^{TX}\right)+
   \frac{R^{0}_{\vartheta}\left(Y^{N}_{t}\right)}{\cos\left(\vartheta\right)}\right)\mathbf{P}\\
   -\mathbf{P}b \left( 
R^{0}_{\vartheta}\left(Y_{0}^{TX}\right)+
   \frac{R^{0}_{\vartheta}\left(Y^{N}_{0}\right)}{\cos\left(\vartheta\right)}\right)\mathbf{P}.
\end{multline}
Using equation (\ref{eq:nea1}) in Proposition \ref{Pproj}, we can 
rewrite (\ref{eq:prin24}) in the form,
\begin{equation}\label{eq:afb17}
B_{b,\vartheta,t}=U^{0}_{b,\vartheta,t}\frac{b}{\cos\left(\vartheta\right)}\mathbf{P}R^{0}_{\vartheta}\left(Y_{t}\right)\mathbf{P}
-\frac{b}{\cos\left(\vartheta\right)}\mathbf{P}R^{0}_{\vartheta}\left(Y_{0}\right)\mathbf{P}.
\end{equation}
We rewrite (\ref{eq:afb17}) in the form
\begin{equation}\label{eq:afb18}
B_{b,\vartheta,t}=\frac{b}{\cos\left(\vartheta\right)}U^{0}_{b,\vartheta,t}R^{0}_{\vartheta}\left(Y_{t}\right)\mathbf{P}-
U^{0}_{b,\vartheta,t}\mathbf{P}^{\perp}\frac{b}{\cos\left(\vartheta\right)}R^{0}_{\vartheta}\left(Y_{t}\right)\mathbf{P}-
\frac{b}{\cos\left(\vartheta\right)}\mathbf{P}R^{0}_{\vartheta}\left(Y_{0}\right)\mathbf{P}.
\end{equation}
By combining (\ref{eq:diff0}) and (\ref{eq:afb18}), we obtain
\begin{equation}\label{eq:afb19}
B_{b,\vartheta,t}=\frac{b^{2}}{\cos\left(\vartheta\right)}\frac{d}{dt}U^{0}_{b,\vartheta,t}\mathbf{P}
-U^{0}_{b,\vartheta,t}\mathbf{P}^{\perp}\frac{b}{\cos\left(\vartheta\right)}R^{0}_{\vartheta}\left(Y_{t}\right)
\mathbf{P}-\frac{b}{\cos\left(\vartheta\right)}\mathbf{P}R_{\vartheta}^{0}\left(Y_{0}\right)\mathbf{P}.
\end{equation}

By (\ref{eq:afb16}), (\ref{eq:afb19}), we get
\begin{equation}\label{eq:afb20}
\left( \frac{b^{2}}{\cos\left(\vartheta\right)}\frac{d}{dt}+1 \right) 
U_{b,\vartheta,t}^{0}\mathbf{P}=C_{b,\vartheta,t}+U^{0}_{b,\vartheta,t}\mathbf{P}^{\perp}\frac{b}{\cos\left(\vartheta\right)}
R^{0}_{\vartheta}\left(Y_{t}\right)\mathbf{P}+\frac{b}{\cos\left(\vartheta\right)}\mathbf{P}R^{0}_{\vartheta}\left(Y_{0}\right)\mathbf{P}.
\end{equation}
By (\ref{eq:afb20}), we obtain
\begin{multline}\label{eq:afb21}
U^{0}_{b,\vartheta,t}\mathbf{P}=\exp\left(-\cos\left(\vartheta\right)t/b^{2}\right)\mathbf{P}
+\int_{0}^{t}\exp\left(-\left(t-s\right)\cos\left(\vartheta\right)/b^{2}\right) \\
\frac{\cos\left(\vartheta\right)}{b^{2}}\left(C_{b,\vartheta,s}
+U^{0}_{b,\vartheta,s}\mathbf{P}^{\perp}\frac{b}{\cos\left(\vartheta\right)}
R^{0}_{\vartheta}\left(Y_{s}\right)\mathbf{P}
+\frac{b}{\cos\left(\vartheta\right)}\mathbf{P}R^{0}_{\vartheta}\left(Y_{0}\right)\mathbf{P}\right)
ds.
\end{multline}
For $0\le t\le M$, we have the obvious inequalities,
\begin{equation}\label{eq:afb23}
\left\vert  
\int_{0}^{t}\exp\left(-\left(t-s\right)\cos\left(\vartheta\right)/b^{2}\right)\frac{\cos\left(\vartheta\right)}{b^{2}}
C_{b,\vartheta,s}ds\right\vert\le \sup_{0\le t\le M}\left\vert  
C_{b,\vartheta,t}\right\vert.
\end{equation}
By (\ref{eq:afb16x1}), (\ref{eq:afb23}), we get
\begin{multline}\label{eq:afb23x1}
\left\Vert \sup_{0\le t\le M}\left\vert  \int_{0}^{t}\exp\left(-\left(t-s\right)\cos\left(\vartheta\right)/b^{2}\right)\frac{\cos\left(\vartheta\right)}{b^{2}}
C_{b,\vartheta,s}ds\right\vert \right\Vert_{p}\\
\le C_{p,M}\exp\left(C'\left(\cos\left(\vartheta\right)\left\vert  
Y_{0}^{TX}\right\vert^{2}+\left\vert  
Y_{0}^{N}\right\vert^{2}\right)\right).
\end{multline}

For $0\le t\le M$, by (\ref{eq:est0}), we have the inequality,
\begin{multline}\label{eq:afb24}
\left\vert  
\int_{0}^{t}\exp\left(-\left(t-s\right)\cos\left(\vartheta\right)/b^{2}\right)
\frac{\cos\left(\vartheta\right)}{b^{2}}
U^{0}_{b,\vartheta,s}\mathbf{P}^{\perp}\frac{b}{\cos\left(\vartheta\right)}
R^{0}_{\vartheta}\left(Y_{s}\right)\mathbf{P}ds\right\vert \\
\le \sup_{0\le s\le M}\left\vert  
U^{0}_{b,\vartheta,s}\mathbf{P}^{\perp}\right\vert 
\left(1-e^{-M\cos\left(\vartheta\right)/b^{2}}\right) \\
\left( 
\sup_{0\le s\le M}\left\vert bY^{TX}_{s} \right\vert+
\frac{b}{\cos^{1/2}\left(\vartheta\right)}\sup_{0\le s\le M}\left\vert  
Y^{N}_{s}\right\vert \right) .
\end{multline}
Using the notation in (\ref{eq:lom7}), we get
\begin{align}\label{eq:afb24x1}
&\sup_{0\le s\le M}\left\vert  bY^{TX}_{s}\right\vert\le \left\vert  
bY^{TX}_{0}\right\vert+2\sup_{0\le s\le M}\left\vert  
w^{TX}_{s}\right\vert,\\
&\sup_{0\le s\le M}\left\vert  \frac{b}{\cos^{1/2}(\vartheta)}Y^{N}_{s}\right\vert\le \left\vert  
\frac{b}{\cos^{1/2}(\vartheta)}Y^{N}_{0}\right\vert+2\sup_{0\le s\le M}\left\vert  
w^{N}_{s}\right\vert. \nonumber 
\end{align}
By (\ref{eq:afb24x1}), we obtain
\begin{align}\label{eq:afb24x2}
&\left(1-e^{-M\ct/b^{2}}\right)\sup_{0\le s\le M}\left\vert  bY^{TX}_{s}\right\vert\le 
C\cos^{1/2}\left(\vartheta\right)\left\vert  
Y^{TX}_{0}\right\vert+2\sup_{0\le s\le M}\left\vert  
w^{TX}_{s}\right\vert,\\
&\left(1-e^{-M\ct/b^{2}}\right)\sup_{0\le s\le M}\left\vert \frac{ 
b}{\cos^{1/2}(\vartheta)}Y^{N}_{s}\right\vert\le 
C\left\vert  
Y^{N}_{0}\right\vert+2\sup_{0\le s\le M}\left\vert  
w^{N}_{s}\right\vert. \nonumber 
\end{align}
Using Theorem \ref{Tgres1},  and equations (\ref{eq:afb24}), 
(\ref{eq:afb24x2}),  we obtain\begin{multline}\label{eq:afb25}
\left\Vert\sup_{0\le t\le M}\left\vert  
\int_{0}^{t}\exp\left(-\left(t-s\right)\cos\left(\vartheta\right)/b^{2}\right)
\frac{\cos\left(\vartheta\right)}{b^{2}}
U^{0}_{b,\vartheta,s}\mathbf{P}^{\perp}\frac{b}{\cos\left(\vartheta\right)}
R^{0}_{\vartheta}\left(Y_{s}\right)\mathbf{P}ds\right\vert   
\right\Vert_{p}\\
\le 
C_{p,M}\exp\left(C'\left(\cos\left(\vartheta\right)\left\vert  
Y^{TX}\right\vert^{2}+\left\vert  
Y_{0}^{N}\right\vert^{2}\right)\right).
\end{multline}

By Proposition \ref{Pproj}, we get
\begin{multline}\label{eq:afb26}
\int_{0}^{t}\exp\left(-\left(t-s\right)\cos\left(\vartheta\right)/b^{2}\right)\frac{\cos\left(\vartheta\right)}{b^{2}}
\frac{b}{\cos\left(\vartheta\right)}\mathbf{P}
R^{0}_{\vartheta}\left(Y_{0}\right)\mathbf{P}ds\\
=\left(1-e^{-t\cos\left(\vartheta\right)/b^{2}}\right)b\cos^{3/2}\left(\vartheta\right)
\widehat{c}\left(-i\ad\left(Y^{N}_{0}\right)\vert_{\overline{TX}}\right)\mathbf{P}.
\end{multline}
By (\ref{eq:afb26}), we deduce that for $0<b\le 1$, 
\begin{equation}\label{eq:afb27}
\left\vert  \int_{0}^{t}\exp\left(-\left(t-s\right)\cos\left(\vartheta\right/b^{2}\right)\frac{\cos\left(\vartheta\right)}{b^{2}}
\frac{b}{\cos\left(\vartheta\right)}\mathbf{P}
R^{0}_{\vartheta}\left(Y_{0}\right)\mathbf{P}ds\right\vert 
\le C\left\vert  Y_{0}^{N}\right\vert.
\end{equation}

By (\ref{eq:afb21}), (\ref{eq:afb23x1}), (\ref{eq:afb25}), and 
(\ref{eq:afb27}), we deduce that
\begin{equation}\label{eq:afb28}
\left\Vert \sup_{0\le t\le M}\left\vert  
U^{0}_{b,\vartheta,t}\mathbf{P}\right\vert \right\Vert_{p}\le
C_{p,M}\exp\left(C'\left(\cos\left(\vartheta\right)\left\vert  
Y_{0}^{TX}\right\vert^{2}+\left\vert  
Y_{0}^{N}\right\vert^{2}\right)\right).
\end{equation}

By (\ref{eq:afb16}), (\ref{eq:afb16x1}), and (\ref{eq:afb28}), $\left\Vert\sup_{0\le t\le 
M}\left\vert  B_{\bt,t}\right\vert\right\Vert_{p}$ is dominated by 
an expression like the right-hand side of (\ref{eq:prin18}).

This concludes the proof of equation (\ref{eq:prin18}), and of our 
theorem.
\end{proof}
\subsection{The limit of $U^{0}_{b,\vartheta,\cdot}$ as $b\to 0$}
\label{subsec:limu}%
Recall that 
\index{StY@$\overline{S}^{0}_{\vartheta}\left(Y\right)$}%
$\overline{S}^{0}_{\vartheta}\left(Y\right)$ was defined in 
Definition \ref{DSbis}.
By equation (\ref{eq:qsic7}) in Proposition \ref{Pidnew}, we get
\begin{equation}\label{eq:bebe1}
\left\vert  \overline{S}^{0}_{\vartheta}\left(Y\right)\right\vert\le 
C\left(\cos^{2}\left(\vartheta\right)\left\vert  
Y^{TX}\right\vert^{2}+\left\vert  Y^{N}\right\vert^{2}\right).
\end{equation}

Recall that 
\index{Hbt@$H_{b,\vartheta,\cdot}$}%
$H_{b,\vartheta,\cdot}$ was defined in equation 
(\ref{eq:nea3}). Using (\ref{eq:wat2}), (\ref{eq:nea3}), we conclude that
\begin{multline}\label{eq:spi1}
\frac{b^{2}}{\cos\left(\vartheta\right)}\ddot H_{b,\vartheta}+\dot 
H_{b,\vartheta} \\
=H_{b,\vartheta}\left(-\cos^{4}\left(\vartheta\right)
\widehat{c}\left( \ad\left(Y^{N}\right)\vert_{\overline{TX}} \right) 
^{2}
+\cos^{2}\left(\vartheta\right)\widehat{c}\left(\ad\left(-i\dot w^{N}\right)\vert_{\overline{TX}}\right)
\right).
\end{multline}
Equations (\ref{eq:hope5xy1}) and (\ref{eq:spi1}) are intimately 
related.

By (\ref{eq:nea3}), we get
\begin{align}\label{eq:spi2}
&\frac{d H_{b,\vartheta}^{-1}}{ds}=\frac{\cos^{5/2}\left(\vartheta\right)}{b}\widehat{c}\left(i\ad\left(Y^{N}\right)
\vert_{\overline{TX}}\right)H^{-1}_{b,\vartheta},
&H^{-1}_{b,\vartheta,0}=1.
\end{align}
By (\ref{eq:wat2}), (\ref{eq:spi2}), we conclude that
\begin{multline}\label{eq:spi3}
\frac{b^{2}}{\ct}\ddot H^{-1}_{b,\vartheta}+\dot 
H^{-1}_{b,\vartheta} \\
=\left( -\cos^{4}\left(\vartheta\right)
\widehat{c}\left( \ad\left(Y^{N}\right)\vert_{\overline{TX}} \right) 
^{2}
-\cos^{2}\left(\vartheta\right)\widehat{c}\left(\ad\left(-i\dot 
w^{N}\right)\vert_{\overline{TX}}\right)\right) 
H^{-1}_{b,\vartheta}.
\end{multline}
\begin{defin}\label{Dkbt}
Put
\index{Lbt@$L_{\bt,t}$}%
   \begin{equation}\label{eq:mira0}
L_{\bt,t}=\exp\left(\int_{0}^{t}\overline{S}^{0}_{\vartheta}\left(Y\right)ds\right)H_{\bt,t}.
\end{equation}
\end{defin}

By  equation (\ref{eq:rot7ay1}) in Proposition  \ref{PMehla}, in which $\beta$ is replaced by $\beta 
b$, and $t$ by $t/b^{2}$, given 
$\beta>0$, if $b>0$ is small enough so that $\beta b\le 1$, then
\begin{equation}\label{eq:mira0zz1}
E^{Q}\left[\exp\left(\frac{\beta^{2}}{2}\int_{0}^{t}\left\vert  
Y^{TX}_{s}\right\vert^{2}ds\right)\right]\le 
\exp\left(\frac{\beta^{2}}{2}\left(mt+b^{2}\left\vert  
Y_{0}^{TX}\right\vert^{2}\right)\right).
\end{equation}
Under the above conditions, by (\ref{eq:mira0zz1}), we get
\begin{equation}\label{eq:mira0zzz1}
E^{Q}\left[\exp\left(\frac{\beta^{2}}{2}\int_{0}^{t}\left\vert  
Y^{TX}_{s}\right\vert^{2}ds\right)\right]\le 
\exp\left(\frac{1}{2}\left(\beta^{2}mt+\left\vert  
Y_{0}^{TX}\right\vert^{2}\right)\right).
\end{equation}
Similarly, given $\vartheta\in \left[0,\frac{\pi}{2}\right[, 
\beta>0$,  for $b>0$ small enough so that $\beta 
b/\cos^{1/2}\left(\vartheta\right)\le 1$, then
\begin{equation}\label{eq:mira0zz2}
E^{Q}\left[\exp\left(\frac{\beta^{2}}{2}\int_{0}^{t}\left\vert  
Y^{N}_{s}\right\vert^{2}ds\right)\right]\le 
\exp\left(\frac{\beta^{2}}{2}\left(mt+\frac{b^{2}}{\ct}\left\vert  
Y_{0}^{N}\right\vert^{2}\right)\right).
\end{equation}
Under the above conditions, by (\ref{eq:mira0zz2}), we deduce that
\begin{equation}\label{eq:mira0zz3}
E^{Q}\left[\exp\left(\frac{\beta^{2}}{2}\int_{0}^{t}\left\vert  
Y^{N}_{s}\right\vert^{2}ds\right)\right]\le 
\exp\left(\frac{1}{2}\left( \beta^{2}mt+\left\vert  
Y_{0}^{N}\right\vert^{2}\right)\right).
\end{equation}

We fix $\vartheta\in \left[0,\frac{\pi}{2}\right[, 
Y_{0}^{TX},Y_{0}^{N}$. 
By Theorem  \ref{Tpou}, and by (\ref{eq:bebe1}),  (\ref{eq:mira0}),  
(\ref{eq:mira0zzz1}),  and
(\ref{eq:mira0zz3}),   we conclude that given $p\ge 
1$, for $b>0$ small enough, $\sup_{0\le t\le M}\left\vert  
L_{\bt,t}\right\vert$ and $\sup_{0\le t\le M}\left\vert  
L_{\bt,t}^{-1}\right\vert$ are uniformly bounded in $L_{p}$.

Recall that the constant
\index{d0t@$\pmb \delta^{0}_{\vartheta}$}%
$\pmb \delta^{0}_{\vartheta}$ was defined in (\ref{eq:mir-1}), and is 
given by (\ref{eq:rob3}). Also
\index{H0t@ $H_{0,\vartheta,\cdot}$}%
 $H_{0,\vartheta,\cdot}$
 was defined in equation 
(\ref{eq:mir0}).

Now we establish an extension of \cite[Theorem 14.10.6]{Bismut08b}.
\begin{thm}\label{TlimV}
Given $\vartheta\in\left[0,\frac{\pi}{2}\right[, M>0$, as $b\to 0$, 
$U^{0}_{b,\vartheta,\cdot}\mathbf{P}$ converges uniformly  on 
$\left[0,M\right]$ to $\exp\left(\pmb 
\delta^{0}_{\vartheta}\cdot\right)H_{0,\vartheta,\cdot}\mathbf{P}$ in probability. 
\end{thm}
\begin{proof}
   In the proof, $\vartheta\in \left[0,\frac{\pi}{2}\right[$ will be 
   fixed. 
   We start from equation (\ref{eq:formid7}), that will be combined 
   with equations (\ref{eq:spi2}) and (\ref{eq:spi3}). 
Let 
   $\mathsf{A}_{b,\vartheta,\cdot}$ be the stochastic process,
   \begin{multline}\label{eq:mir1}
\mathsf{A}_{b,\vartheta,\cdot}=U^{0}_{b,\vartheta,\cdot}
\Biggl(1+b\left(1+N^{\LXN}_{-\vartheta}\right)R_{\vartheta}^{0}\left(Y^{TX}_{\cdot}\right)\\
+b\left(\ct+N^{\LXN}_{-\vartheta}\right)^{-1}R^{0}_{\vartheta}\left(Y^{N}_{\cdot}\right)
+b\cos^{3/2}\left(\vartheta\right)\widehat{c}\left(i\ad\left(Y^{N}_{\cdot}\right)
\vert_{\overline{TX}}\right) \Biggr)L^{-1}_{\bt,\cdot}\mathbf{P}.
\end{multline}
Using (\ref{eq:wat2}),  (\ref{eq:diff0}), (\ref{eq:formid7}), and (\ref{eq:spi2}), we 
get
\begin{multline}\label{eq:mir2}
\mathsf{A}_{\bt,t}=\mathsf{A}_{\bt,0}+
\int_{0}^{t}U^{0}_{\bt}\Biggl[-\overline{S}^{0}_{\vartheta}\left(Y\right)+R^{0}_{\vartheta}\left(Y\right)
\Biggl[\left(1+N^{\LXN}_{-\vartheta}\right)^{-1}R^{0}_{\vartheta}\left(Y^{TX}\right)\\
+\left(\ct+N^{\LXN}_{-\vartheta}\right)^{-1}R^{0}_{\vartheta}\left(Y^{N}\right)\Biggr]ds+
\left(1+N^{\LXN}_{-\vartheta}\right)^{-1}R^{0}_{\vartheta}\left(dw^{TX}\right)\\
+\left(\ct+\Nto\right)^{-1}\cos^{1/2}\left(\vartheta\right)R^{0}_{\vartheta}
\left(dw^{N}\right)\\
+R^{0}_{\vartheta}\left(Y\right)\cos^{3/2}\left(\vartheta\right)\widehat{c}\left(i\ad\left(Y^{N}\right)
\vert_{\overline{TX}}\right)ds\\
+
 \Biggl(  \left(1+\Nto\right)^{-1}R_{\vartheta}^{0}\left(Y^{TX}\right)
+\left(\ct+\Nto\right)^{-1}R^{0}_{\vartheta}\left(Y^{N}\right) 
 \\
+\cos^{3/2}\left(\vartheta\right)\widehat{c}\left(i\ad\left(Y^{N}\right)
\vert_{\overline{TX}}\right)\Biggr)\cos^{5/2}\left(\vartheta\right)
\widehat{c}\left(i\ad\left(Y^{N}\right)
\vert_{\overline{TX}}\right)ds\\
+\cos^{2}\left(\vartheta\right)\widehat{c}\left(i\ad\left(dw^{N}\right)
\vert_{\overline{TX}}\right)\Biggr]L^{-1}_{\bt}\mathbf{P}.
\end{multline}
In (\ref{eq:mir2}), the stochastic integrals with respect to $w^{TX},w^{N}$ are 
 standard It\^{o} integrals, i.e., $dw^{TX},dw^{N}$ can be 
replaced by $\delta w^{TX},\delta w^{N}$.

We denote by $\mathsf{B}_{\bt,t},\mathsf{B}^{\perp}_{\bt,t}$ be
 obtained from $\mathsf{A}_{\bt,t}-\mathsf{A}_{\bt,0}$ in (\ref{eq:mir2}) by 
replacing $U^{0}_{\bt,\cdot}$ by 
$U^{0}_{\bt}\mathbf{P},U^{0}_{\bt}\mathbf{P}^{\perp}$, so that
\begin{equation}\label{eq:mir3}
\mathsf{A}_{\bt,t}=\mathsf{A}_{\bt,0}+\mathsf{B}_{\bt,t}+\mathsf{B}_{\bt,t}^{\perp}.
\end{equation}

By equation (\ref{eq:nea1}) in Proposition \ref{Pproj}, by equation 
(\ref{eq:psica1}) in Proposition \ref{Peqel}, by 
(\ref{eq:psic27y1}), and (\ref{eq:mir2}), we get
\begin{multline}\label{eq:mir4}
\mathsf{B}_{\bt,t}=\int_{0}^{t}U^{0}_{\bt}\mathbf{P}\Biggl[-\cos^{4}\left(\vartheta\right)
\widehat{c}\left(\ad\left(Y^{N}\right)\vert_{\overline{TX}}
\right)^{2}ds-\cos^{2}\left(\vartheta\right)\widehat{c}\left(i\ad\left(dw^{N}\right)\vert_{\overline{TX}}\right)\\
+\cos^{4}\left(\vartheta\right)\widehat{c}\left(\ad\left(Y^{N}\right)\vert_{\overline{TX}}\right)^{2}ds
+\cos^{2}\left(\vartheta\right)\widehat{c}\left(i\ad\left(dw^{N}\right)\vert_{\overline{TX}}\right)
\Biggr]L^{-1}_{\bt}.
\end{multline}
By (\ref{eq:mir4}), we deduce that
\begin{equation}\label{eq:mir5}
\mathsf{B}_{\bt,t}=0.
\end{equation}

By equation (\ref{eq:flimsw5})
   in Theorem \ref{Tgres1}, and by the considerations following 
 equation (\ref{eq:mira0zz3}), by proceeding as in \cite[Proposition 
  14.10.4]{Bismut08b}, as $b\to 0$, we have the uniform convergence 
   over $\left[0,M\right]$ in probability,
  \begin{equation}\label{eq:mir6}
\mathsf{B}_{\bt,\cdot}^{\perp}\to 0.
\end{equation}

By (\ref{eq:mir3}), (\ref{eq:mir5}), and (\ref{eq:mir6}), we get the 
uniform convergence  over $\left[0,M\right]$ in probability,
\begin{equation}\label{eq:mir7}
\mathsf{A}_{\bt,\cdot}-\mathsf{A}_{\bt,0}\to 0.
\end{equation}

By \cite[Proposition 14.10.1]{Bismut08b},
 as $b\to 0$, $bY_{\cdot}$ 
converges uniformly on $\left[0,M\right]$ to $0$ in probability. By 
Theorem \ref{Tunif} and by the considerations we made after 
(\ref{eq:mira0zz3}) on the process $L_{\bt,\cdot}$, we deduce from (\ref{eq:mir1}), (\ref{eq:mir7}) that  we have the 
uniform convergence  over $\left[0,M\right]$ in probability, 
\begin{equation}\label{eq:mir8}
U^{0}_{\bt}L^{-1}_{\bt}\mathbf{P}\to \mathbf{P}.
\end{equation}
By the considerations we made after (\ref{eq:mira0zz3}), we deduce from 
(\ref{eq:mir8}) that as $b\to 0$, we have the uniform convergence in 
probability over $\left[0,M\right]$, 
\begin{equation}\label{eq:mir9}
U^{0}_{\bt,\cdot}\mathbf{P}-L_{\bt,\cdot}\mathbf{P}\to 0.
\end{equation}

By \cite[Proposition 14.10.1]{Bismut08b} and by (\ref{eq:mir9a}),  as 
$b\to 0$, the process 
$\int_{0}^{t}\overline{S}^{0}_{\vartheta}\left(Y\right)ds$ converges 
uniformly on $\left[0,M\right]$ to the process $\pmb 
\delta^{0}_{\vartheta}t$ in probability. Using  Theorem 
\ref{Tpou}, we conclude that as $b\to 0$, the process  $L_{\bt,\cdot}$ converges 
to $\exp\left(\pmb 
\delta^{0}_{\vartheta}\cdot\right)H_{0,\vartheta,\cdot}$ uniformly on 
$\left[0,M\right]$ in probability. Our theorem follows from 
(\ref{eq:mir9}) and from this last result.
\end{proof}

We now give an extension of \cite[Theorem 14.10.7]{Bismut08b}.
\begin{thm}\label{Tfiut}
Given $\vartheta\in\left[0,\frac{\pi}{2}\right[$, $0<\epsilon\le M$, 
as $b\to 0$, $U^{0}_{\bt,\cdot}$ converges uniformly on 
$\left[\epsilon,M\right]$ to 
$\exp\left(\pmb\delta^{0}_{\vartheta}\cdot\right)H_{0,\vartheta,\cdot}\mathbf{P}$ in probability.
\end{thm}
\begin{proof}
This is a consequence of Theorems \ref{Tgres1} and \ref{TlimV}.
\end{proof}
\subsection{The case of $U^{\prime 0}_{\bt,\cdot}$}%
\label{subsec:tcs}
    There is an obvious analogue of Proposition \ref{Pwondidbis} when 
    replacing $\overline{\mathcal{M}}^{X}_{\bt}$ by 
    $\overline{M}^{X}\vert_{db=0}$.
Recall that $U'_{\bt,\cdot}$ was defined in (\ref{eq:bobi2}). Let 
$U^{\prime 0 }_{\bt,\cdot}$ be $U'_{\bt,\cdot}$ when $\rho^{E}$ is 
the trivial representation.
Proposition \ref{Pcru} has a trivial extension to $U^{\prime 0 
}_{\bt,\cdot}$. The presence of 
$\frac{d\vartheta}{\cos^{1/2}\left(\vartheta\right)}$ introduces an 
extra factor $1/\cos\left(\vartheta\right)$ in whatever estimate we 
make of $U^{ \prime 0}_{\bt,\cdot}$ with respect to the estimates 
we made for $U^{0}_{\bt,\cdot}$. Indeed 
the term 
$-\frac{d\vartheta}{\sqrt{2}}\left(\widehat{c}\left(Y^{TX}\right)+\frac{\sin\left(\vartheta\right)}{\cos^{1/2}\left(\vartheta\right)}
i\mathcal{E}^{N}\right)$ just introduces a factor 
$1/\cos^{1/2}\left(\vartheta\right)$. The presence of 
$\exp\left(-\frac{\ct}{2}\int_{0}^{t}\left\vert  
\left[Y^{N},Y^{TX}\right]\right\vert^{2}ds \right) $ in the 
right-hand side of (\ref{eq:sen1b}) allows us to dominate the 
contribution of  $\frac{d\vartheta}{\sqrt{2}}
\frac{\sin\left(\vartheta\right)}{\cos^{1/2}\left(\vartheta\right)}ic\left(\left[Y^{N},Y^{TX}\right]
\right)$ with an extra factor $1/\ct$.  Equation 
(\ref{eq:nea3}) for $H_{\bt,\cdot}$ should be replaced by 
\begin{align}\label{eq:nea3bis}
&\frac{dH'_{\bt}}{ds}=H'_{\bt}\frac{1}{b}\left(\cos^{5/2}\left(\vartheta\right)\widehat{c}\left(-i\ad\left(Y^{N}\right)
\vert_{\overline{TX}}\right)
-\frac{d\vartheta}{\sqrt{2}}\widehat{c}\left(\overline{Y}^{TX}\right)\right),&H'_{\bt,0}=1.
\end{align}
Equation (\ref{eq:nea3bis}) can be easily integrated by the formula
\begin{equation}\label{eq:fufu4}
H'_{\bt,s}=H_{\bt,s}-\frac{1}{b}\int_{0}^{s}H_{\bt,u}\frac{d\vartheta}{\sqrt{2}}\widehat{c}\left(\overline{Y}^{TX}_{u}\right)
H_{\bt,u}^{-1}duH_{\bt,s}.
\end{equation}
Using (\ref{eq:wat2}), (\ref{eq:nea3}), equation (\ref{eq:fufu4}) can 
be rewritten in the form,
\begin{multline}\label{eq:fufu5}
H'_{\bt,s}=H_{\bt,s}+d\vartheta 
bH_{\bt,s}\widehat{c}\left(\overline{Y}^{TX}_{s}\right)-d\vartheta  b\widehat{c}\left(\overline{Y}^{TX}_{0}\right)H_{\bt, s}\\
+d\vartheta\int_{0}^{s}H_{\bt,u}\cos^{5/2}\left(\vartheta\right)\widehat{c}\left(i\overline{\left[Y^{N}_{u},Y^{TX}_{u}\right]}\right)
H_{\bt,u}^{-1}duH_{\bt,s}\\
-d\vartheta\int_{0}^{t}H_{\bt,u}\widehat{c}\left(\delta 
w^{TX}\right)H_{\bt,u}^{-1}H_{\bt,s}.
\end{multline}
By equations (\ref{eq:lom7}) and  (\ref{eq:fufu5}), we can estimate 
$H'_{\bt,\cdot}$ using the previous estimates on $H_{\bt,\cdot}$. 

Let 
\index{H0t@ $H'_{0,\vartheta,\cdot}$}%
$H'_{0,\vartheta,\cdot}$ be the solution of the stochastic 
differential equation,
\begin{align}\label{eq:mir0bis}
&dH'_{0,\vartheta}=H'_{0,\vartheta}\left(\cos^{2}\left(\vartheta\right)\widehat{c}\left(-i\ad\left(dw^{\mathfrak k}
\right)\vert_{\overline{TX}}\right)-\frac{d\vartheta}{\sqrt{2}}\widehat{c}\left(\overline{dw}^{\mathfrak p}\right)\right),\\
&H'_{0,\vartheta,0}=1. \nonumber 
\end{align}

The  obvious extensions of Theorem \ref{Tpou},  of 
Proposition \ref{Pexp}, of Theorems \ref{Tfuest} and \ref{Tgres1}, of Proposition  \ref{Pformid}, and of Theorem \ref{Tunif} still hold.

Now we state an extension of Theorems \ref{TlimV} and \ref{Tfiut}.
\begin{thm}\label{TlimVpr}
Given $\vartheta\in\left[0,\frac{\pi}{2}\right[,M>0$, as $b\to 0$, $U^{\prime 0 
}_{\bt,\cdot}\mathbf{P}$ converges uniformly on  $\left[0,M\right]$ to 
$\exp\left(\bm{\delta}^{0}_{\vartheta}\cdot\right)H'_{0,\vartheta,\cdot}\mathbf{P}$ in 
probability.

Given $\vartheta\in\left[0,\frac{\pi}{2}\right[$, $0<\epsilon\le M$, 
as $b\to 0$, $U^{\prime 0}_{\bt,\cdot}$ converges uniformly on 
$\left[\epsilon,M\right]$ to 
$\exp\left(\pmb\delta^{0}_{\vartheta}\cdot\right)H'_{0,\vartheta,\cdot}\mathbf{P}$ in probability.
\end{thm}
\begin{proof}
The proof is essentially the same as the proof of Theorems 
\ref{TlimV} and \ref{Tfiut}.
\end{proof}
\subsection{The final steps in the proof of Theorem \ref{Test}}%
\label{subsec:finst}
Recall that the smooth kernels 
$\overline{q}^{X}_{\bt,t},\overline{q}^{X}_{0,\vartheta,t}$ were defined in 
Definition \ref{Dplim}. We will give an extension of \cite[Theorem 
14.11.2]{Bismut08b}.
   \begin{thm}\label{Tfin}  
Let $s\in C^{ \infty,c 
}\left(\widehat{\mathcal{X}},\widehat{\pi}^{*}\left(\Lambda\left(T^{*}X \oplus N^{*}\right)
\otimes S^{\overline{TX}} \otimes F\right)\right)$. Given 
$\vartheta\in\left[0,\frac{\pi}{2}\right[,t>0$, as $b\to 0$, 
\begin{multline}\label{eq:rob4}
\int_{\widehat{\mathcal{X}}}^{}\overline{q}^{X}_{b,\vartheta,t}\left(\left(x,Y\right),\left(x',Y'\right)\right)s\left(x',Y'\right)dx'dY'\\
\to 
\int_{\widehat{\mathcal{X}}}^{}\overline{q}^{X}_{0,\vartheta,t}\left(\left(x,Y\right),
\left(x',Y'\right)\right)s\left(x',Y'\right)dx'dY'.
\end{multline}
\end{thm}
\begin{proof}
 First, we prove equation (\ref{eq:rob4}) when making $d\vartheta=0$. 
 In the sequel, we fix $\vartheta\in 
\left[0,\frac{\pi}{2}\right[, t>0$. Set 
$x=\cos\left(\vartheta\right)$.
By Proposition \ref{Pformel}, by equations    (\ref{eq:psic-4}), (\ref{eq:sig1}), by 
Theorem \ref{Thesgab}, and by  (\ref{eq:diff2}), (\ref{eq:he9}),  
   we only 
need to prove that given 
$\vartheta\in\left[0,\frac{\pi}{2}\right[, t>0$,  as $b\to 0$, 
\begin{multline}\label{eq:roba1}
E^{Q}\left[\exp\left(-\frac{x}{2}\int_{0}^{t}\left\vert  
\left[Y^{N}_{b},Y^{TX}_{b}\right]\right\vert
^{2}ds\right)U^{0}_{b,\vartheta,t} \otimes E_{b/\cos^{1/2}\left(\vartheta\right),
t}\tau^{t}_{0}
s\left(x_{b,t},Y_{b,t}\right)\right]\\
\to 
\exp\Biggl(-\frac{x^{2}}{2}\left(\frac{1}{4}\Tr^{\mathfrak p}
\left[C^{\mathfrak k,\mathfrak p}
\right]
+2\sum_{i=m+1}^{m+n}\widehat{c}\left(\ad\left(e_{i}\right)\vert_{\overline{TX}}\right) 
\rho^{F}\left(e_{i}\right)\right)t \\
-\left(\frac{1}{48}\Tr^{\mathfrak 
k}\left[C^{\mathfrak k, \mathfrak k}\right]+\frac{1}{2}C^{\mathfrak 
k,F}\right)t
\Biggr)\\
E^{Q}\left[\tau^{t}_{0}\int_{\left(TX \oplus 
N\right)_{x_{0,t}}}^{}\mathbf{P}s\left(x_{0,t},Y\right)\exp\left(-
\left\vert  Y\right\vert^{2}\right)\frac{dY}{\pi^{\left(m+n\right)/2}}\right],
\end{multline}
where in the left-hand side, $\left(x_{b,\cdot},Y_{b,\cdot}\right)$ 
verifies (\ref{eq:wat2}), and where in the right-hand side, 
$x_{0,\cdot}$ verifies (\ref{eq:stoch1}). The dependence of these 
processes on $\vartheta$ will not be denoted explicitly.

By the ergodic theory arguments explained in \cite[Proposition 14.10.1]{Bismut08b}, 
for any $M>0$, the process 
$t\in\R_{+}\to\int_{0}^{t}\left\vert  \left[Y^{N}_{b,s},Y^{TX}_{b,s}\right]\right\vert^{2}ds$ converges uniformly on 
$\left[0,M\right]$ to $t\in\R_{+}\to -\frac{1}{4}\Tr^{\mathfrak 
p}\left[C^{\mathfrak k,\mathfrak p}\right]t$  in probability.  By Theorem 
\ref{Tunif}, given $p>2$, for 
 $b>0$ small enough, 
$\sup_{0\le t\le M}\left\vert  U^{0}_{b,\vartheta,t}\right\vert$ 
remains uniformly bounded in all the $L_{p},1\le p < + \infty $ . 
By  
Theorem \ref{Teste},   $\sup_{0\le 
t\le M}\left\vert  E_{b/\cos^{1/2}\left(\vartheta\right),t}\right\vert$ remains also uniformly 
bounded in all the $L_{p},1\le p<+ \infty$.  

Combining Theorems \ref{Teste}, \ref{Tpou},  and   \ref{Tfiut} with the above considerations, we find 
that   as $b\to 0$,  then 
\begin{multline}\label{eq:jar5}
E^{Q}\left[\exp\left(-\frac{x}{2}\int_{0}^{t}\left\vert  
\left[Y_{b}^{N},Y_{b}^{TX}\right]\right\vert
^{2}ds\right)U^{0}_{b,\vartheta,t} \otimes E_{b/\cos^{1/2}\left(\vartheta\right),t}\tau^{t}_{0}
s\left(x_{b,t},Y_{t}\right)\right]\\
-\exp\left(\left( \frac{x}{8}\Tr^{\mathfrak p}\left[C^{\mathfrak k, 
\mathfrak p}\right] +\pmb{\delta}^{0}_{\vartheta}\right)t 
\right)E^{Q}\left[H_{b,\vartheta,t}\otimes 
E_{b/\cos^{1/2}\left(\vartheta\right),t}\tau^{t}_{0}\mathbf{P} s\left(x_{b,t},Y_{b,t}\right)\right]\to 0.
\end{multline}

We are now ready to use the full strength of \cite[Theorem 12.8.1 and 
Remark 12.8.2]{Bismut08b}, and more specifically \cite[eq. 
(12.8.47)]{Bismut08b}, from which we deduce that as $b\to 0$, 
\begin{multline}\label{eq:jar10}
E^{Q}\left[H_{b,\vartheta,t}\otimes 
E_{b/\cos^{1/2}\left(\vartheta\right),t}\tau^{t}_{0}\mathbf{P} s\left(x_{b,t},Y_{b,t}\right)\right]\\
\to E^{Q}\left[H_{0,\vartheta,t}\otimes 
E_{0,t}\tau^{t}_{0}\int_{\left(TX \oplus N\right)_{x_{0,t}}}^{}\mathbf{P}s\left(x_{0,t},Y\right)\exp\left(-\left\vert  Y\right\vert^{2}\right)
\frac{dY}{\pi^{\left( m+n \right) /2}}\right].
\end{multline}
Note that $x_{0,\cdot}$, which only depends on $w^{TX
}_{\cdot}$, and $H_{0,\vartheta,\cdot} \otimes E_{0,\cdot}$, 
which only depends on $w^{N}_{\cdot}$, are independent, so 
that 
\begin{multline}\label{eq:roba4}
E^{Q}
\left[H_{0,\vartheta,t} \otimes E_{0,t}\int_{\left(TX \oplus 
N\right)_{x_{0,t}}}^{}\mathbf{P}\tau^{t}_{0}s\left(x_{0,t},Y\right)
\exp\left(-\left\vert  Y\right\vert^{2}\right)dY\right]\\
=E^{Q}
\left[H_{0,\vartheta,t} \otimes E_{0,t}\right]E^{Q}\left[\int_{\left(TX \oplus 
N\right)_{x_{0,t}}}^{}\mathbf{P}\tau^{t}_{0}s\left(x_{0,t},Y\right)
\exp\left(-\left\vert  Y\right\vert^{2}\right)dY\right].
\end{multline}

From the stochastic differential equations 
(\ref{eq:roba2}), (\ref{eq:mir0}), 
 using It\^{o}'s formula, we get
\begin{multline}\label{eq:roba5}
d \left( H_{0,\vartheta} \otimes E_{0}\right) =\left( H_{0,\vartheta} \otimes 
E_{0}\right)\frac{1}{2} \Biggl(-x^{4}\sum_{m+1}^{m+n}\widehat{c}\left(\ad\left(e_{i}\right)
\vert_{\overline{\mathfrak 
p}}\right)^{2}-\sum_{m+1}^{m+n}\rho^{E}\left(e_{i}\right)^{2} \\
-
2x^{2}\sum_{m+1}^{m+n}\widehat{c}\left(\ad\left(e_{i}\right)\vert_{\overline{\mathfrak p}}\right)\rho^{E}\left(e_{i}\right)\Biggr)
+\left(H_{0,\vartheta} \otimes  
E_{0}\right)\left(-x^{2}\widehat{c}\left(i\ad\left(\delta 
w^{\mathfrak k}\right)\right)-\rho^{E}\left(i\delta w^{\mathfrak k}\right)\right).
\end{multline}
By (\ref{eq:roba5}), we deduce that
\begin{multline}\label{eq:roba6}
d E^{Q}\left[\left( H_{0,\vartheta} \otimes E_{0}\right) 
\right]=E^{Q}\left[\left( H_{0,\vartheta} \otimes 
E_{0}\right)\right]
\frac{1}{2} \Biggl(-x^{4}\sum_{m+1}^{m+n}\widehat{c}\left(\ad\left(e_{i}\right)
\vert_{\overline{\mathfrak 
p}}\right)^{2}-\sum_{m+1}^{m+n}\rho^{E}\left(e_{i}\right)^{2} \\
-
2x^{2}\sum_{m+1}^{m+n}\widehat{c}\left(\ad\left(e_{i}\right)\vert_{\overline{\mathfrak p}}\right)\rho^{E}\left(e_{i}\right)\Biggr).
\end{multline}
By (\ref{eq:batt2}), (\ref{eq:Lie11x1}), and (\ref{eq:roba6}), we conclude that
\begin{multline}\label{eq:roba7}
E^{Q}\left[\left( H_{0,\vartheta,t} \otimes 
E_{0,t}\right)\right]\\
=\exp\left(-\frac{1}{2}\left(\frac{x^{4}}{8}\Tr^{\mathfrak p}\left[C^{\mathfrak k, \mathfrak p}\right]
+C^{\mathfrak 
k,E}+2x^{2}\sum_{m+1}^{m+n}\widehat{c}\left(\ad\left(e_{i}\right)\vert_{\overline{\mathfrak p}}\right)\rho^{E}\left(e_{i}\right)
\right)t\right).
\end{multline}
By  (\ref{eq:rob3}),  (\ref{eq:jar5})--(\ref{eq:roba4}), and 
(\ref{eq:roba7}), we get (\ref{eq:roba1}). 

Now we will consider the $d\vartheta$ component of (\ref{eq:rob4}). Using 
(\ref{eq:stoch4bis}), (\ref{eq:sen1b}),  the considerations that 
follow (\ref{eq:bobi2}), and (\ref{eq:diff2a1}), to establish
(\ref{eq:roba1}), we have to prove that as $b\to 0$, 
\begin{multline}\label{eq:roba1ter}
E^{Q}\left[\exp\left(-\frac{x}{2}\int_{0}^{t}\left\vert  
\left[Y^{N}_{b},Y^{TX}_{b}\right]\right\vert
^{2}ds\right)U^{\prime 0}_{b,\vartheta,t} \otimes E_{b/\cos^{1/2}\left(\vartheta\right),
t}\tau^{t}_{0}
s\left(x_{b,t},Y_{b,t}\right)\right]\\
\to 
\exp \left( -\left( \frac{x^{2}}{8}\Tr^{\mathfrak p}\left[C^{\mathfrak k,\mathfrak p}
\right]
+\frac{1}{48}\Tr^{\mathfrak 
k}\left[C^{\mathfrak k, \mathfrak k}\right]+\frac{1}{2}C^{\mathfrak 
k,F}\right) t
\right) \\
E^{Q}\left[A_{t}\tau^{t}_{0}\int_{\left(TX \oplus 
N\right)_{x_{0,t}}}^{}\mathbf{P}s\left(x_{0,t},Y\right)\exp\left(-
\left\vert  
Y\right\vert^{2}\right)\frac{dY}{\pi^{\left(m+n\right)/2}}\right].
\end{multline}

 By Theorem 
\ref{TlimVpr}, in the extension of (\ref{eq:jar5}), we have to 
replace $H_{\bt,t}$ by $H'_{\bt,t}$. In (\ref{eq:jar10}) we should 
replace $H_{0,\vartheta,t}$ by $H'_{0,\vartheta,t}$. 

Let $E^{Q,w^{\mathfrak p}}$ denote conditional expectation with 
respect to $w^{\mathfrak p}$. Set
\begin{equation}\label{eq:fufu5a}
D_{\vartheta,\cdot}=E^{Q,w^{\mathfrak p}}\left[H'_{0,\vartheta,\cdot} 
\otimes E_{0,\cdot}\right].
\end{equation}
Since $x_{0,\cdot}$ depends only on $w^{\mathfrak p}$, instead of 
(\ref{eq:roba4}), we get
\begin{multline}\label{eq:roba4bis}
E^{Q}
\left[H'_{0,\vartheta,t} \otimes E_{0,t}\int_{\left(TX \oplus 
N\right)_{x_{0,t}}}^{}\mathbf{P}\tau^{t}_{0}s\left(x_{0,t},Y\right)
\exp\left(-\left\vert  Y\right\vert^{2}\right)dY\right]\\
=E^{Q}\left[D_{\vartheta,t}\int_{\left(TX \oplus 
N\right)_{x_{0,t}}}^{}\mathbf{P}\tau^{t}_{0}s\left(x_{0,t},Y\right)
\exp\left(-\left\vert  Y\right\vert^{2}\right)dY\right].
\end{multline}

 By (\ref{eq:roba2}), (\ref{eq:mir0bis}), instead of (\ref{eq:roba5}), we get
\begin{multline}\label{eq:roba5bis}
d \left( H'_{0,\vartheta} \otimes E_{0}\right) =\left( H'_{0,\vartheta} \otimes 
E_{0}\right)\frac{1}{2} \Biggl(-x^{4}\sum_{m+1}^{m+n}\widehat{c}\left(\ad\left(e_{i}\right)
\vert_{\overline{\mathfrak 
p}}\right)^{2}-\sum_{m+1}^{m+n}\rho^{E}\left(e_{i}\right)^{2} \\
-
2x^{2}\sum_{m+1}^{m+n}\widehat{c}\left(\ad\left(e_{i}\right)\vert_{\overline{\mathfrak p}}\right)\rho^{E}\left(e_{i}\right)\Biggr)\\
+\left(H'_{0,\vartheta} \otimes  
E_{0}\right)\left(-x^{2}\widehat{c}\left(\ad\left(i\delta 
w^{\mathfrak k}\right)\right)-\rho^{E}\left(i\delta w^{\mathfrak 
k}\right)-\frac{d\vartheta}{\sqrt{2}}\widehat{c}\left(\overline{dw}^{\mathfrak p}\right){}\right).
\end{multline}

By (\ref{eq:fufu5a}), we get
\begin{multline}\label{eq:fufu6}
    dD_{\vartheta}=D_{\vartheta}\frac{1}{2} \Biggl(-x^{4}\sum_{m+1}^{m+n}\widehat{c}\left(\ad\left(e_{i}\right)
\vert_{\overline{\mathfrak 
p}}\right)^{2}-\sum_{m+1}^{m+n}\rho^{E}\left(e_{i}\right)^{2} \\
-
2x^{2}\sum_{m+1}^{m+n}\widehat{c}\left(\ad\left(e_{i}\right)\vert_{\overline{\mathfrak p}}\right)\rho^{E}\left(e_{i}\right)
-\frac{d\vartheta}{\sqrt{2}}\widehat{c}\left(\overline{dw}^{\mathfrak p}\right) \Biggr).
\end{multline}
Using  the same arguments as in (\ref{eq:roba7}), and comparing with 
(\ref{eq:roba8}), we deduce from (\ref{eq:fufu6}) that
\begin{equation}\label{eq:fufu7}
D_{\vartheta,t}=\exp\left(-\frac{t}{2}\left(\frac{x^{4}}{8}\Tr^{\mathfrak p}\left[C^{\mathfrak k, \mathfrak p}\right]
+C^{\mathfrak k,E}\right)\right)A_{t}.
\end{equation}
By (\ref{eq:rob3}), by the analogue of (\ref{eq:jar5}), 
(\ref{eq:jar10}), and using  (\ref{eq:roba4bis}),  (\ref{eq:fufu7}),  we 
get (\ref{eq:roba1ter}).
The proof of our theorem is completed. 
\end{proof}

Given what we did before, we are now ready to indicate the main steps 
of the proof of Theorem \ref{Test}, by properly referring to 
\cite[chapters 12--14]{Bismut08b} when necessary.

In Theorem \ref{Tinsi2}, we gave a complete proof of 
the corresponding estimates for the scalar heat kernel $\mathfrak 
r_{b,\vartheta,t}^{X}$ over $\widehat{\mathcal{X}}$ and of the 
corresponding convergence result as $b\to 0$. 

For the moment, we make $d\vartheta=0$.

As in \cite{Bismut08b},  the idea is to proceed in two steps:
\begin{enumerate}
    \item  In a first step,  we obtain rough estimates on 
    $\overline{q}^{X}_{b,\vartheta,t}$ and its derivatives  in the considered range of 
    parameters, by using the proper version of the Malliavin 
    calculus, and we prove the convergence in (\ref{eq:sumex32bis}).

    \item  In a second step, using 
   the uniform bounds (\ref{eq:prin-6z1}) in Theorem \ref{Teste}, (\ref{eq:couac-1y1}) in Theorem 
    \ref{Tfuest},  and   the  uniform bounds of the   first step, we obtain the 
    uniform estimates in (\ref{EQ:BERN0}).
\end{enumerate}
The main difficulty is the presence of 
$U_{\bt,t}=U^{0}_{\bt,t} \otimes E_{\bt,t}$ in the right-hand side of 
(\ref{eq:sen1}), (\ref{eq:sen1b}). 

By (\ref{eq:psic-4}), 
(\ref{eq:sig1}), we get
\begin{equation}\label{eq:fufu1}
\overline{\mathcal{L}}^{X}_{\bt,-}=\exp\left(-\left\vert  
Y\right\vert^{2}/2\right)\overline{\mathcal{M}}^{X}_{\bt,-}\exp\left(\left\vert  
Y\right\vert^{2}/2\right).
\end{equation}

If $s\in C^{\infty,c}\left(\widehat{\mathcal{X}}, 
\widehat{\pi}^{*}\left(\Lambda\ac\left(T^{*}X \oplus N^{*}\right) 
\otimes S^{\overline{TX}} \otimes F\right)\right)$,  using equation  (\ref{eq:sen1b}) in Theorem \ref{Thesgab} and  
(\ref{eq:fufu1}),  we get
\begin{multline}\label{eq:robx1}
\exp\left(-t\overline{\mathcal{L}}^{X}_{b,\vartheta,-}\right)s\left(x_{0},Y_{0}\right)
=\exp\left(-\left\vert  Y\right\vert^{2}/2\right) \\
E^{Q}\left[\exp\left(-\frac{\cos\left(\vartheta\right)}{2}\int_{0}^{t}\left\vert  \left[Y^{N},Y^{TX}\right]\right\vert
^{2}ds\right)U_{b,\vartheta,t}\tau^{t}_{0}s\left(x_{t},Y_{t}\right)\exp\left(\left\vert  Y_{t}\right\vert^{2}
/2\right)\right].
\end{multline}

An important observation is that contrary to what we did in 
\cite[sections 14.7 and 14.8]{Bismut08b} for the kernel 
$\overline{q}^{X}_{b,t}$, it is not possible to deduce the uniform 
pointwise estimate (\ref{EQ:BERN0}) from the corresponding estimate 
(\ref{eq:scal7}) for the scalar kernel $ \mathfrak r^{X}_{\bt,t}$. In 
\cite{Bismut08b}, a simple proof was possible because one could 
obtain a simple pointwise estimate for $U^{0}_{b,0,t}$, which is not 
the case here  when $\vartheta>0$. So 
we start from (\ref{eq:robx1}), apply the Malliavin calculus, and use 
also the  uniform bound (\ref{eq:prin-6z1}) in Theorem \ref{Teste} 
for $E_{\bt,\cdot}$, 
and the estimate (\ref{eq:couac-1y1}) in Theorem \ref{Tfuest} for  
$U^{0}_{\bt,\cdot}$. This last estimate, that provides the required 
uniformity in $\vartheta\in \left[0,\frac{\pi}{2}\right[$ plays a 
critical role. Ultimately, given $0<\epsilon\le M<+ \infty $,  we obtain the existence of $C>0,k\in\N$ 
such that for $0<b\le 
1,\vartheta\in\left[0,\frac{\pi}{2}\right[,\epsilon\le t\le M$,
\begin{equation}\label{eq:fufu2}
\left\vert  
\overline{q}^{X}_{\bt,t}\left(\left(x,Y\right),\left(x',Y'\right)\right)\right\vert\le C\cos^{-k}\left(\vartheta\right),
\end{equation}
Also,  we obtain corresponding bounds for the derivatives of arbitrary order of 
$\overline{q}^{X}_{\bt,t}\left(\left(x,Y\right),\left(x',Y'\right)\right)$ in $\left(x',Y'\right)$.

In a second step, we use the semigroup property of 
$\exp\left(-t\overline{\mathcal{L}}^{X}_{\bt,t}\right)$, and we get 
an   analogue of (\ref{eq:phan5}). From 
this analogue, we obtain 
\begin{multline}\label{eq:robx1y1}
\left\vert  
\overline{q}^{X}_{\bt,t}\left(\left(x,Y\right),\left(x',Y'\right)\right)\right\vert\le
\int_{\widehat{\mathcal{X}}}^{}\left\vert  
\overline{q}^{X}_{\bt,t/2}\left(\left(x,Y\right),\left(z,Z\right)\right)\right\vert \\ 
\left\vert  \overline{q}^{X}_{\bt,t/2}\left(\left(z,Z\right),\left(x',Y'\right)\right)\right\vert dzdZ.
\end{multline}
By (\ref{eq:fufu2}), (\ref{eq:robx1y1}),  we 
get the obvious extension of (\ref{eq:phan5a}),
\begin{equation}\label{eq:robx2}
\left\vert  \overline{q}^{X}_{\bt,t}\left(\left(x,Y\right),\left(x',Y'\right)\right)\right\vert
\le C\cos^{-k}\left(\vartheta\right)\int_{\widehat{\mathcal{X}}}^{}
\left\vert  
\overline{q}^{X}_{\bt,t/2}\left(\left(x,Y\right),\left(z,Z\right)\right)\right\vert dzdZ.
\end{equation}

By (\ref{eq:robx1}), we get
\begin{equation}\label{eq:robx3}
\int_{\widehat{\mathcal{X}}}^{}
\left\vert  
\overline{q}^{X}_{\bt,t/2}\left(\left(x,Y\right),\left(z,Z\right)\right)\right\vert dzdZ
\le C\exp\left(-\left\vert  
Y\right\vert^{2}/2\right)E^{Q}\left[\left\vert  
U_{\bt,t}\right\vert\exp\left(\left\vert  
Y_{t}\right\vert^{2}/2\right)\right].
\end{equation}
Using again the uniform bounds (\ref{eq:prin-6z1}),  
(\ref{eq:couac-1y1}), we deduce 
from (\ref{eq:robx2}),  (\ref{eq:robx3})  that
\begin{equation}\label{eq:robx4}
\left\vert  \overline{q}^{X}_{\bt,t}\left(\left(x,Y\right),\left(x',Y'\right)\right)\right\vert
\le 
C\cos^{-k}\left(\vartheta\right)\exp\left(-C'\left(\left\vert  
Y^{TX}\right\vert^{2}+\cos\left(\vartheta\right)\left\vert  
Y^{N}\right\vert^{2}\right)\right).
\end{equation}

The $L_{2}$ formal adjoint of $\overline{\mathcal{L}}^{X}_{\bt,-}$ 
has  the same structure as 
$\overline{\mathcal{L}}^{X}_{\bt,-}$. By exchanging the roles of 
$\left(x,Y\right)$ and $\left(x',Y'\right)$, we obtain the analogue 
of (\ref{eq:robx4}), 
\begin{equation}\label{eq:robx4bis}
\left\vert  \overline{q}^{X}_{\bt,t}\left(\left(x,Y\right),\left(x',Y'\right)\right)\right\vert
\le 
C\cos^{-k}\left(\vartheta\right)\exp\left(-C'\left(\left\vert  
Y^{TX \prime }\right\vert^{2}+\cos\left(\vartheta\right)\left\vert  
Y^{N \prime }\right\vert^{2}\right)\right).
\end{equation}

From  (\ref{eq:fufu2}), (\ref{eq:robx1y1}), we get the obvious analogue of 
 (\ref{eq:phan6}),
 \begin{multline}\label{eq:phan6bis}
\left\vert  \overline{q}^{X}_{\bt,t}\left(\left(x,Y\right),\left(x',Y'\right)\right)\right\vert \\
\le 
C\cos^{-k}\left(\vartheta\right)\int_{\substack{\left(z,Z\right)\in \widehat{\mathcal{X}}\\d\left(x,z\right)\ge 
d\left(x,x'\right)/2}}^{}\left\vert  \overline{q}^{X}_{\bt,t/2}\left(\left(x,Y\right),\left(z,Z\right)\right)\right\vert 
dzdZ\\
+C\cos^{-k}\left(\vartheta\right)
\int_{\substack{\left(z,Z\right)\in 
\widehat{\mathcal{X}}\\ d\left(x',z\right)\ge 
d\left(x,x'\right)/2}}^{}\left\vert  \overline{q}^{X}_{\bt,t/2}\left(\left(z,Z\right),\left(x',Y^{TX\prime}\right)
\right)\right\vert dzdZ.
\end{multline}

 By (\ref{eq:robx1}), instead of 
(\ref{eq:phan7}), we get
\begin{multline}\label{eq:robx4y1}
\int_{\substack{\left(z,Z\right)\in\widehat{\mathcal{X}}\\
d\left(x,z\right)\ge d\left(x,x'\right)/2}}\left\vert  
\overline{q}^{X}_{\bt,t/2}\left(\left(x,Y\right),\left(z,Z\right)\right)\right\vert dzdZ \\
\le
\exp\left(-\left\vert  Y\right\vert^{2}/2\right)
E^{Q}\left[1_{\sup_{0\le s\le t/2}d\left(x,x_{s}\right)\ge 
d\left(x,x'\right)/2}\left\vert  
U_{\bt,t}\right\vert\exp\left(\left\vert  Y_{t/2}\right\vert^{2}/2\right)\right].
\end{multline}
It is now crucial to observe that in (\ref{eq:wat2}), 
$\left(x_{\cdot},Y^{TX}_{\cdot}\right)$ has 
the same probability law 
as $\left( 
x_{\cos^{2}\left(\vartheta\right)\cdot},Y^{TX}_{\cos^{2}\left(\vartheta\right)\cdot} \right) $ in 
(\ref{eq:glab9a6a}), when replacing $b$ by $b\ct$. Similarly,  
$Y^{N}_{\cdot}$ 
in (\ref{eq:wat2}) has 
the same probability law as 
$Y_{\ct \cdot/b^{2}}$ 
in (\ref{eq:phan-2}) (with $E=N$). In equations 
(\ref{eq:scal4}) and (\ref{eq:scal16}), we already used a similar 
argument. By  equation (\ref{eq:glab17a}) (used with $t$ replaced by 
$\ct t/b^{2}$),  and by  equation 
(\ref{eq:led19s1}) in Theorem \ref{Tesc},  if $0<b\le 
1,\vartheta\in\left[0,\frac{\pi}{2}\right[, \epsilon\le t\le M$, we 
get
\begin{multline}\label{eq:fufu3}
\exp\left(-\left\vert  Y\right\vert^{2}/2\right)E^{Q}\left[1_{\sup_{0\le s\le t/2}d\left(x,x_{s}\right)\ge 
d\left(x,x'\right)/2}\exp\left(\left\vert  
Y_{t/2}\right\vert^{2}\right)\right]\\
\le 
C\exp\left(-C'\left(\frac{d^{2}\left(x,x'\right)}{\cos^{2}\left(\vartheta\right)}+\left\vert  Y^{TX}\right\vert^{2}+\cos\left(\vartheta\right)
\left\vert  Y^{N}\right\vert^{2}\right)\right).
\end{multline}

By using equation (\ref{eq:couac-1y1}) in Theorem \ref{Tfuest}, 
(\ref{eq:fufu3}), and H\"older's inequality, in the considered range 
of parameters, we get
\begin{multline}\label{eq:fufu3a1}
\exp\left(-\left\vert  Y\right\vert^{2}/2\right)E^{Q}\left[1_{\sup_{0\le s\le t/2}d\left(x,x_{s}\right)\ge 
d\left(x,x'\right)/2}\left\vert  U_{\bt,t}\right\vert\exp\left(\left\vert  
Y_{t/2}\right\vert^{2}\right)\right]\\
\le 
C\exp\left(-C'\left(\frac{d^{2}\left(x,x'\right)}{\cos^{2}\left(\vartheta\right)}+\left\vert  Y^{TX}\right\vert^{2}+\cos\left(\vartheta\right)
\left\vert  Y^{N}\right\vert^{2}\right)\right).
\end{multline}
 
By  (\ref{eq:robx4y1}),   (\ref{eq:fufu3a1}), we 
obtain
\begin{multline}\label{eq:robx5}
\int_{\substack{\left(z,Z\right)\in\widehat{\mathcal{X}}\\
d\left(x,z\right)\ge d\left(x,x'\right)/2}}\left\vert  
\overline{q}^{X}_{\bt,t/2}\left(\left(x,Y\right),\left(z,Z\right)\right)\right\vert dzdZ \\
\le C\exp\left(-C'\left(\frac{d^{2}\left(x,x'\right)}{\cos^{2}\left(\vartheta\right)}+\left\vert  
Y^{TX}\right\vert^{2}+\ct\left\vert  
Y^{N}\right\vert^{2}\right)\right).
\end{multline}
Still using the properties of the formal adjoint of 
$\overline{\mathcal{L}}^{X}_{\bt,-}$, the same arguments as before 
give 
\begin{multline}\label{eq:robx6}
\int_{\substack{\left(z,Z\right)\in\widehat{\mathcal{X}}\\
d\left(x',z\right)\ge d\left(x,x'\right)/2}}\left\vert  
\overline{q}^{X}_{\bt,t/2}\left(\left(z,Z\right),\left(x',Y'\right)\right)\right\vert dzdZ \\
\le C\exp\left(-C'\left(\frac{d^{2}\left(x,x'\right)}{\cos^{2}\left(\vartheta\right)}+\left\vert  
Y^{TX \prime}\right\vert^{2}+\ct\left\vert  
Y^{N \prime}\right\vert^{2}\right)\right).
\end{multline}

By (\ref{eq:phan6bis}), (\ref{eq:robx5}), and (\ref{eq:robx6}), we get
\begin{equation}\label{eq:robx7}
\left\vert  
\overline{q}^{X}_{\bt,t}\left(\left(x,Y\right)\left(x',Y'\right)\right)\right\vert\le
C\cos^{-k}\left(\vartheta\right)\exp\left(-C'\frac{d^{2}\left(x,x'\right)}{\cos^{2}\left(\vartheta\right)}\right).
\end{equation}

By combining (\ref{eq:robx4}), (\ref{eq:robx4bis}), and 
(\ref{eq:robx7}), we get the estimate (\ref{EQ:BERN0}), at least when 
making $d\vartheta=0$. The proof of this estimate when $d\vartheta$ 
is not equal to $0$ is exactly the same, given the estimates we 
obtained on $U^{\prime 0}_{\bt,\cdot}$.

Let us now establish the convergence result in equation 
(\ref{eq:sumex32bis}).  As was explained before,  the methods of 
the Malliavin calculus show that given $0<b\le 1,\vartheta\in 
\left[0,\frac{\pi}{2}\right[, \left(x,Y\right)\in \widehat{\mathcal{X}}$, 
$\overline{q}^{X}_{b,\vartheta,t}\left(\left(x,Y\right),\left(x',Y'\right)\right)$ and its derivatives of arbitrary order
in $\left(x',Y'\right)$ are uniformly bounded over compact subsets of 
$\widehat{\mathcal{X}}$. Note here that we can also obtain 
corresponding uniform bounds for these derivatives that are similar 
to (\ref{EQ:BERN0}), but this plays no role in the sequel. Using  
Theorem \ref{Tfin}, we get the convergence in (\ref{eq:sumex32bis}). 
More precisely, the above convergence is uniform over compact 
subsets, and derivatives of arbitrary order converge as well. 

This completes the proof of Theorem \ref{Test}.

%% file: Eta11.tex
\section{Estimates on the hypoelliptic heat kernel for 
$b$ large}%
\label{sec:unilar}
The purpose of this section is to establish Theorems \ref{TLARGE1QU} 
and \ref{TBEAUTY}, 
i.e., to obtain uniform estimates on the hypoelliptic heat kernels 
$\underline{q}^{X \prime}_{b,\vartheta,t}$ and $\mathfrak 
q^{X \prime }_{\bt}$ for large values of 
$b$.

When $\vartheta=0$,  Theorem \ref{TLARGE1QU} coincides with \cite[Theorem 
9.1.1]{Bismut08b}. 
 We will briefly explain how to adapt the techniques used in the 
 proof of
 \cite[Theorem 9.1.1]{Bismut08b} in order to establish
 Theorem \ref{TLARGE1QU}.  
 
The operator 
\index{LXbt@$\underline{\mathcal{L}}^{X \prime }_{b,\vartheta}$}%
$\underline{\mathcal{L}}^{X \prime 
}_{b,\vartheta}$ is given by (\ref{eq:glub2bis1}). 
The operators 
 $\underline{\mathcal{L}}^{X \prime 
}_{b,\vartheta}$ and $\underline{\mathcal{L}}^{X \prime 
}_{b/\cos^{1/2}\left(\vartheta\right),0}$ can be easily compared. 
The only substantial difference is that in the first operator, we 
have the term $\frac{1}{2\cos^{2}\left(\vartheta\right)}\left\vert  
Y^{TX}\right\vert^{2}$, while in the second operator, the smaller term 
$\frac{1}{2}\left\vert  
Y^{TX}\right\vert^{2}$ appears. The operator $\underline{\mathcal{L}}^{X 
\prime }_{b,0}$ and its heat kernel were  studied in 
detail in \cite[Theorem 9.1.1 and chapter 15]{Bismut08b} when $b\to + 
\infty $. Here,   the  difference 
between $\underline{\mathcal{L}}^{X \prime 
}_{b,\vartheta}$ and $\underline{\mathcal{L}}^{X \prime 
}_{b/\cos^{1/2}\left(\vartheta\right),0}$  will 
play in our favour. 

Put
\index{LXbt@$\underline{\mathcal{L}}^{X \prime  \prime}_{b,\vartheta}$}%
 \begin{multline}\label{eq:glub2bis2}
\underline{\mathcal{L}}^{X \prime  \prime}_{b,\vartheta}=\frac{b^{4}}{2\cos^{2}\left(\vartheta\right)}\left\vert  \left[Y^{ N},Y^{TX}\right]\right\vert^{2}
     +\frac{1}{2} \Biggl( 
     -\frac{\cos^{2}\left(\vartheta\right)}{b^{4}}\Delta^{TX\oplus 
     N}+\frac{1}{2}\left\vert  
     Y^{TX}\right\vert^{2}+
     \left\vert  Y^{N}\right\vert^{2}  \\
     - \frac{1}{b^{2}}\left( m+\cos\left(\vartheta\right)n \right) \Biggr) 
    +\frac{N^{\Lambda\ac\left(T^{*}X \oplus N^{*}\right) \prime }_{-\vartheta}}{b^{2}}  \\
     - \Biggl(\n_{Y^{ TX}}^{C^{ \infty }\left(TX \oplus 
 N,\widehat{\pi}^{*} \left( \Lambda\ac\left(T^{*}X \oplus 
 N^{*}\right)\otimes  S^{\overline{TX}} \otimes F
 \right) \right),f*,\widehat{f} }
	    \\
  -c\left(
   i\theta\ad\left(Y^{N}\right) \right) 
   -i\widehat{c}\left(\ad\left(Y^{N}\right)\vert_{\overline{TX}}\right)
   -i\rho^{F}\left(Y^{N}\right)\Biggr).
   \end{multline}
   Up to terms that are irrelevant in the 
 range of parameters $b\ge 1,\vartheta\in\left[0,\frac{\pi}{2}\right[$, the operator $\underline{\mathcal{L}}^{X \prime 
\prime }_{b,\vartheta}$ is now easily comparable to the 
operator $\underline{\mathcal{L}}^{X \prime 
}_{b/\cos^{1/2}\left(\vartheta\right),0}$. 

For $t>0$, let $\underline{q}^{X \prime \prime 
}_{b,\vartheta,t}\left(\left(x,Y\right),\left(x',Y'\right)\right)$ be the smooth 
kernel associated with the operator $\exp\left(-t\underline{\mathcal{L}}^{X \prime \prime  
}_{b,\vartheta}\right)$. We use the 
notation $\underline{q}^{X \prime \prime 
}_{b,\vartheta}=\underline{q}^{X \prime \prime }_{b,\vartheta,1}$. 
It is  easy to see that the estimates in \cite[Theorem 
9.1.1]{Bismut08b} for 
$\underline{q}^{X \prime }_{b/\cos^{1/2}\left(\vartheta\right),t}$ remain 
valid for $\underline{q}^{X \prime \prime }_{b,\vartheta,t}$ in the 
range $b\ge 1,\vartheta\in \left[0,\frac{\pi}{2}\right[$.

   Comparing  (\ref{eq:glub2bis1}) and (\ref{eq:glub2bis2}), we get
   \begin{equation}\label{eq:ors15}
\underline{\mathcal{L}}^{X \prime }_{b,\vartheta}=
\underline{\mathcal{L}}^{X \prime \prime}_{b,\vartheta}+
\frac{1}{2}\left(\frac{1}{\cos^{2}\left(\vartheta\right)}-\frac{1}{2}\right)\left\vert  Y^{TX}\right\vert^{2}.
\end{equation}
Also
\begin{equation}\label{eq:ors16}
\frac{1}{\cos^{2}\left(\vartheta\right)}-\frac{1}{2}\ge 
\frac{1}{2\cos^{2}\left(\vartheta\right)}.
\end{equation}

If $\vartheta\in \left[0,\frac{\pi}{2}\right[$, let $s\in\left[0,1\right]\to x_{s}\in X$ 
be a $C^{1}$ path such that
\begin{equation}\label{eq:ors13a}
\dot x_{s}=Y^{TX}_{s}.
\end{equation}
If $x_{0}=x,x_{t}=x'$, using (\ref{eq:ors16}), we get
\begin{equation}\label{eq:ors14}
\frac{1}{2}\left(\frac{1}{\cos^{2}\left(\vartheta\right)}-\frac{1}{2}\right)
\int_{0}^{t}\left\vert  Y^{TX}\right\vert^{2}ds
\ge 
\frac{d^{2}\left(x,x'\right)}{4\cos^{2}\left(\vartheta\right)t}.
\end{equation}

From the 
probabilistic representation of the heat kernels $\underline{q}^{X 
\prime }_{b,\vartheta,t}$,  using equation (\ref{eq:ors14}),  up to 
fixed constants, in the range of parameters  described before,  an upper bound for $\left\vert
\underline{q}^{X \prime }_{b,\vartheta,t}
\left(\left(x,Y\right),\left(x',Y'\right)\right)\right\vert$ is  the 
product of 
$\exp\left(-d^{2}\left(x,x'\right)/4\cos^{2}\left(\vartheta\right)t\right)$ by the 
upper bounds  for $\left\vert  \underline{q}^{X \prime}_{b/\cos^{1/2}\left(\vartheta\right),t}\left(\left(x,Y\right),\left(x',Y'\right)\right)\right\vert$ that were obtained 
in \cite[Theorem 9.1.1]{Bismut08b}.
Since $d_{\gamma}\left(x\right)\ge\left\vert  a\right\vert$, from 
these estimates, we get Theorem \ref{TLARGE1QU} except for the very 
last statement concerning vertical derivatives. However, the same 
arguments as in \cite[Theorem 12.11.2]{Bismut08b} permits us also to 
properly control the vertical derivatives.

By exactly the same arguments, we can derive Theorem \ref{TBEAUTY} 
from \cite[Theorem 9.5.6]{Bismut08b}.
More precisely, we obtain exactly the same estimates as in 
\cite{Bismut08b} for the rescaled heat kernel associated with 
$\underline{\mathcal{L}}^{X \prime \prime }_{\bt}$. Equations 
(\ref{eq:ors13a}), (\ref{eq:ors14}) are responsible for the 
appearance of $\frac{\left\vert  
a\right\vert^{2}}{\cos^{2}\left(\vartheta\right)}$ in the right-hand 
side of (\ref{eq:mossbrugg3bis}). As to replacing $\mathfrak q^{X 
\prime }_{\bt}$ by $\frac{\ct}{b^{2}}\n^{V} \mathfrak q^{X \prime 
}_{\bt}$ in (\ref{eq:mossbrugg3bis}), this can be done by exactly the 
same arguments.

.